\documentclass[12pt,leqno]{amsart}

\usepackage{amssymb, mathrsfs} 

\usepackage{tikz}
\usepackage{tikz-cd}
\usetikzlibrary{decorations.pathreplacing}
\usetikzlibrary{decorations.markings}
\usetikzlibrary{calc}

\usepackage{hyperref}
\numberwithin{section}{part}
\numberwithin{equation}{section}
\numberwithin{figure}{section}

\newcommand{\mo}{\mathopen}
\newcommand{\mc}{\mathclose}

\newtheorem{theorem}{Theorem}[section]
\newtheorem{proposition}[theorem]{Proposition}
\newtheorem{lemma}[theorem]{Lemma}
\newtheorem{corollary}[theorem]{Corollary}

\theoremstyle{definition}

\newtheorem{definition}[theorem]{Definition}
\newtheorem{notation}[theorem]{Notation}
\newtheorem{example}[theorem]{Example}
\newtheorem{examples}[theorem]{Examples}
\newtheorem{remark}[theorem]{Remark}
\newtheorem{remarks}[theorem]{Remarks}

\newcommand{\C}{{\mathbb{C}}}
\newcommand{\N}{{\mathbb{N}}}
\newcommand{\R}{{\mathbb{R}}}

\newcommand{\Z}{{\mathbb{Z}}}

\newcommand{\shf}{{\mathcal{F}}}

\newcommand{\shl}{{\mathcal{L}}}

\newcommand{\shn}{{\mathcal{N}}}
\newcommand{\shs}{{\mathcal{S}}}

\newcommand{\shv}{{\mathcal{V}}}

\newcommand{\Cinf}{C^\infty}

\newcommand{\cor}{{\bf k}}

\newcommand{\pt}{\{{\rm pt}\}}
\newcommand{\rpos}{\R_{\geq 0}}
\newcommand{\rspos}{\R_{>0}}
\newcommand{\Rm}{\mathbb{R}^\times}
\newcommand{\Rp}{\mathbb{R}^+}

\renewcommand{\to}[1][]{\xrightarrow[]{#1}}
\newcommand{\from}[1][]{\xleftarrow[]{#1}}
\newcommand{\isoto}[1][]{\xrightarrow[#1]%
{{\raisebox{-.6ex}[0ex][-.6ex]{$\mspace{1mu}\sim\mspace{2mu}$}}}}
\newcommand{\isofrom}[1][]{\xleftarrow[#1]%
{{\raisebox{-.6ex}[0ex][-.6ex]{$\mspace{1mu}\sim\mspace{2mu}$}}}}

\newcommand{\xfrom}[2][]{\xleftarrow[#1]{#2}}

\newcommand{\RR}{\mathrm{R}}

\newcommand{\muhom}{\mu hom}
\newcommand{\Hom}{\mathrm{Hom}}
\newcommand{\RHom}{\RR\mathrm{Hom}}
\newcommand{\Ext}{\mathrm{Ext}}
\renewcommand{\hom}{{\mathcal{H}om}}
\newcommand{\rhom}{{\RR\hom}}

\newcommand{\homeps}{{\mathcal{H}om^\varepsilon}}
\newcommand{\rhomeps}{{\RR\hom^\varepsilon}}
\newcommand{\DD}{\mathrm{D}}
\newcommand{\tens}{\otimes}
\newcommand{\ltens}{\mathbin{\overset{\scriptscriptstyle\mathrm{L}}\tens}}
\newcommand{\epstens}{\otimes^\varepsilon}
\newcommand{\etens}{\mathbin{\boxtimes}}
\newcommand{\letens}{\overset{\mathrm{L}}{\etens}}

\newcommand{\End}{{\operatorname{End}}}
\newcommand{\Aut}{{\operatorname{Aut}}}
\newcommand{\Iso}{{\operatorname{Iso}}}

\newcommand{\sect}{\Gamma}
\newcommand{\rsect}{\mathrm{R}\Gamma}

\newcommand{\oim}[1]{{#1}_*}
\newcommand{\eim}[1]{{#1}_!}
\newcommand{\roim}[1]{\RR{#1}_*}
\newcommand{\reim}[1]{\RR{#1}_!}
\newcommand{\opb}[1]{#1^{-1}}
\newcommand{\epb}[1]{#1^{!}}

\newcommand{\eqdot}{\mathbin{:=}}
\newcommand{\cl}{\colon}
\newcommand{\pointdiag}{\makebox[0mm]{\;\;.}}
\newcommand{\virgdiag}{\makebox[0mm]{\;\;,}}
\newcommand{\scbul}{{\,\raise.4ex\hbox{$\scriptscriptstyle\bullet$}\,}}
\newcommand{\ol}{\overline}
\newcommand{\ul}{\underline}

\newcommand{\id}{\mathrm{id}}

\newcommand{\supp}{\operatorname{supp}}

\DeclareMathOperator{\SSi}{SS}
\newcommand{\Int}{\operatorname{Int}}

\newcommand{\CC}{\mathsf{C}}
\newcommand{\CK}{\mathsf{K}}
\newcommand{\Der}{\mathsf{D}}
\newcommand{\Derb}{\Der^{\mathrm{b}}}
\newcommand{\Derlb}{\Der^{\mathrm{lb}}}

\newcommand{\Dersf}{\Der^{s,f}}

\newcommand{\Mod}{\operatorname{Mod}}

\newcommand{\coker}{\operatorname{coker}}

\newcommand{\im}{\operatorname{im}}

\newcommand{\rc}{{\mathrm{c}}}

\newcommand{\dT}{{\dot{T}}}
\newcommand{\Cor}{{\mathbb{K}}}

\newcommand{\gammaof}{\gamma}

\newcommand{\gammaf}{{\ol\gamma}}
\newcommand{\dbl}{{dbl}}
\newcommand{\Dertp}{\Der_{\tau>0}}
\newcommand{\Dertpn}{\Der_{\tau\geq 0}}

\newcommand{\Derra}{\Der^{r_!\mathrm{a}}_{\tau\geq 0}}

\newcommand{\loc}{\mathsf{Loc}}
\newcommand{\Dloc}{\mathsf{DL}}

\newcommand{\kss}{\mu\mathsf{Sh}}
\newcommand{\kssfunc}{\mathfrak{m}}

\newcommand{\mucirc}{\mathbin{\overset{\scriptscriptstyle\mu}{\circ}}}

\newcommand{\lag}{\mathcal{L}}

\newcommand{\demi}{\frac{1}{2}}
\newcommand{\pdemi}{{\textstyle \frac{1}{2}}}

\newcommand{\hplus}{\mathbin{\widehat +}}

\newcommand{\catc}{\mathcal{C}}
\newcommand{\catd}{\mathcal{D}}

\newcommand{\catt}{\mathcal{T}}

\newcommand{\perf}{\mathsf{E}}
\newcommand{\Orb}{{\mathsf{D}_{/[1]}}}
\newcommand{\OrbL}[1]{{\mathsf{D}_{/[1],#1}}}
\newcommand{\orb}{\mathrm{orb}}

\newcommand{\SSo}{\mathrm{SS}^\orb}
\newcommand{\dSSo}{\smash{\dot{\mathrm{SS}}}^\orb}
\newcommand{\supporb}{\mathrm{supp}^\orb}
\newcommand{\Oloc}{\mathsf{OL}}

\newcommand{\bpmat}{\left( \begin{smallmatrix}}
\newcommand{\epmat}{\end{smallmatrix} \right)}

\newcommand{\cer}{{\operatorname{S^1}}}
\newcommand{\sph}{{\mathbb{S}^2}}
\newcommand{\sphere}{\mathbb{S}}
\newcommand{\RP}{\mathbb{RP}}

\newcommand{\sui}{\medskip\noindent}
\newcommand{\suiv}[1]{\smallskip\noindent{\rm #1}}

\newcommand{\Mat}{\operatorname{Mat}}
\newcommand{\GL}{\operatorname{GL}}

\newcommand{\GO}{\operatorname{O}}
\newcommand{\SO}{\operatorname{SO}}
\newcommand{\Sym}{\operatorname{Sym}}

\newcommand{\mon}{\operatorname{m}}
\newcommand{\sgn}{\operatorname{sgn}}

\newcommand{\de}{\operatorname{e}}
\newcommand{\It}{\mathbb{I}}
\newcommand{\Jj}{\mathbb{J}}
\newcommand{\Uu}{\mathcal{U}}
\newcommand{\GammaE}{\Gamma}

\begin{document}


\title{Sheaves and symplectic geometry of cotangent bundles}

\author{St\'ephane Guillermou}

\thanks{Institut Fourier, CNRS and Universit\'e Grenoble Alpes.  The author is also partially
  supported by the ANR project MICROLOCAL (ANR-15CE40-0007-01). Part of this paper was written
  during a stay at UMI 3457 in Montr\'eal (CNRS - CRM - Universit\'e de Montr\'eal)}

\maketitle

\setcounter{tocdepth}{1}
\tableofcontents

\section*{Introduction}

As the title suggests this paper explains some applications of the microlocal
theory of sheaves of Kashiwara and Schapira to the symplectic geometry of
cotangent bundles.  The main notions of the microlocal theory of sheaves are
Sato's microlocalization, introduced in the 70's, and the notion of microsupport
of a sheaf, introduced by Kashiwara and Schapira in the 80's. These notions were
motivated by the study of modules over the ring of (micro-)differential
operators.  The link with the symplectic geometry was noticed (a deep result
of~\cite{KS85} says that the microsupport of any sheaf is coisotropic) but not
used to study global aspects of symplectic geometry until the papers \cite{NZ09}
of Nadler-Zaslow and~\cite{T08} of Tamarkin.  The paper~\cite{NZ09}, together
with~\cite{N09}, show that the dg-category of constructible sheaves on a real
analytic manifold $M$ is equivalent to the triangulated envelope of a version of
the Fukaya category of $T^*M$.  The paper~\cite{T08} proves non-displaceability
results in symplectic geometry using the properties of the microsupports of
sheaves.  Building on the ideas of this paper it is explained in~\cite{GKS12}
how to associate a sheaf with a Hamiltonian isotopy of a cotangent bundle. In
this paper we go on in this direction and use sheaves to recover some classical
results of symplectic geometry (the Gromov non-squeezing theorem and the
Gromov-Eliashberg rigidity theorem) and a more recent result, which says that a
compact exact Lagrangian submanifold of a cotangent bundle is homotopically
equivalent to the base.  We also prove a result about cusps of curves on the
sphere (Arnol'd three cusps conjecture).

\smallskip

Before we give more details we recall some facts about the microsupport.  Let
$M$ be a manifold of class $C^\infty$ and let $\cor$ be a ring.  We denote by
$\Der(\cor_M)$ the derived category of sheaves of $\cor$-modules over $M$.  The
microsupport $\SSi(F)$ of an object $F$ of $\Der(\cor_M)$ is introduced
in~\cite{KS82}. It is a closed subset of the cotangent bundle $T^*M$, conic for
the action of $\Rp$ on $T^*M$.  It is defined as the closure of the set of
singular directions with respect to $F$, where $(x;\xi) \in T^*M$ is said non
singular if the restriction maps from a neighborhood $B$ of $x$ to
$B \cap \{f<f(x)\}$ induce isomorphisms between $H^iF_x$ and
$\varinjlim_{B\ni x} H^i(B \cap \{f<f(x)\}; F)$, for all functions $f$ with
$df(x) = \xi$ and all $i\in \Z$.  The easiest example is
$\SSi(\cor_N) = T^*_NM$, where $\cor_N$ is the constant sheaf on a submanifold
$N$ of $M$.  In general the microsupport can be a very singular set but it is
coisotropic in some sense (see Theorem~\ref{thm:invol}).  If $M$ is real
analytic and $F$ is constructible, then $\SSi(F)$ is Lagrangian.  Any smooth
conic Lagrangian submanifold of $T^*M$ is locally the microsupport of some sheaf
on $M$.

The microsupport is well-behaved with respect to the standard sheaf operations.
An important example is the composition.  Let $M_i$, $i=1,2,3$, be three
manifolds and let $q_{ij}$ be the projection from $M_1 \times M_2 \times M_3$ to
$M_i \times M_j$.  For $K_1 \in \Der(\cor_{M_1 \times M_2})$ and
$K_2 \in \Der(\cor_{M_2 \times M_3})$ we set
$K_1 \circ K_2 = \reim{q_{13}}( \opb{q_{12}}K_1 \ltens \opb{q_{23}} K_2)$.  We
can define a set theoretic analog of the composition where direct and inverse
images are replaced by the same set operations and the tensor product by the
intersection.  Then, under some geometric hypotheses, we have
$\SSi(K_1 \circ K_2) \subset \SSi(K_1) \circ \SSi(K_2)$.

In~\cite{T08} a sheaf version of the Chekanov-Sikorav theorem (see~\cite{C96}
and~\cite{S87}) is given. A more functorial version is given in~\cite{GKS12} as
follows.  Let $\Phi$ be an $\rspos$-homogeneous Hamiltonian isotopy of
$\dT^*M = T^*M \setminus T^*_MM$.  Then there exists a sheaf $K_\Phi$ on $M^2$
which is invertible for the composition (there exists $K'$ such that
$K_\Phi \circ K' = \cor_{\Delta_M}$) and such that $\dot\SSi(K_\Phi)$ is the
graph of $\Phi$. Here we let $\dot\SSi(F)$ be $\SSi(F)$ with the zero section
removed.  We then have $\dot\SSi(K_\Phi \circ F) = \Phi(\dot\SSi(F))$ for any
$F\in \Der(\cor_M)$ and $F \mapsto K_\Phi \circ F$ is an auto-equivalence of
$\Der(\cor_M)$.  We recall this in Part~\ref{chap:HamIsot}.

Since the microsupport is conic it is rather related with the contact geometry
of the sphere cotangent bundle than the symplectic geometry of the cotangent
bundle.  We can also consider a Legendrian submanifold of the $1$-jet space
$J^1(M)$ as a conic Lagrangian submanifold in $T^*(M\times \R)$ contained in
$\{\tau >0\}$, where we use the coordinates $(t;\tau)$ on $T^*\R$.
In~\cite{T08} Tamarkin also remarks that a sheaf $F$ on $M\times\R$ with a
microsupport in $\{\tau \geq 0\}$ comes with natural morphisms
$\tau_c \colon F \to \oim{T_c}(F)$, where $c\geq 0$ and $T_c$ is the vertical
translation in $M\times\R$, $T_c(x,t) = (x,t+c)$.  A useful invariant of $F$ is
then $\de(F) = \sup \{c\geq 0$; $\tau_c(F) \not= 0\}$.  It is introduced
in~\cite{T08} and used in~\cite{AI17} to obtain displacement energy bounds.  We
use it in Part~\ref{part:nonsqueezing} to prove classical nonsqueezing
results. Our proof is a baby case of the proof of Chiu of a contact nonsqueezing
theorem in~\cite{C17}.

In Part~\ref{part:nonsqueezing} we use some operations on sheaves introduced
in~\cite{KS90} (cut-off lemmas) to reduce the size of a microsupport.  These
operations are compositions with the constant sheaf on a cone.  We recall them
in Part~\ref{chap:cutoff}, where we also prove that a sheaf $F$ whose
microsupport can be decomposed into two disjoint (and unknotted) subsets, say
$\dot\SSi(F) = S_1 \sqcup S_2$, can itself be locally decomposed, up to constant
sheaves, as $F_1 \oplus F_2$ with $\dot\SSi(F_i) = S_i$.

We also use the cut-off results in Part~\ref{part:graphsel} to prove that a
Legendrian submanifold of $J^1(M)$ has a graph selector as soon as it is the
microsupport of a sheaf $F$ satisfying some conditions at infinity.  The graph
selector is given by the boundary of the support of a section of $F$.

In Part~\ref{part:GEthm} we prove the Gromov-Eliashberg rigidity theorem as a
consequence of the involutivity theorem of Kashiwara-Schapira.  The starting
point is very simple.  Let $M$ be a manifold and let $\phi_n$ be a sequence of
homogeneous Hamiltonian isotopies of $\dT^*M$ which converges in $C^0$ norm to a
diffeomorphism $\phi_\infty$ of $\dT^*M$. Let $K_n \in \Der(\cor_{M^2})$ be the
sheaf associated with $\phi_n$ as recalled above.  Then we can consider a kind
of limit $K_\infty$ of $K_n$ and the microsupport of $K_\infty$ is contained in
the graph of $\phi_\infty$. We deduce from the involutivity theorem that this
graph is Lagrangian, hence that $\phi_\infty$ is a symplectic map.  This idea
does not work directly to prove a local statement but we can cut-off the
microsupport (using the cut-off results recalled in Part~\ref{chap:cutoff}) and
make it work.

The main result of this paper is a sheaf theoretic proof that a compact exact
Lagrangian submanifold $L$ of a cotangent bundle $T^*M$ is homotopically
equivalent to $M$. This is done in Parts~\ref{part:orbcat}-\ref{chap:ex_Lag}.
This result was previously obtained with Floer homology methods (see the
beginning of Part~\ref{chap:ex_Lag} for references). However we do not recover
the more precise results of Abouzaid and Kragh, who proved in~\cite{AK18} that
the map $L \to M$ is a simple homotopy equivalence and gave some conditions on
the higher Maslov classes in~\cite{AK16} (we only prove the vanishing of the
first two classes; for the other classes we should use sheaves of spectra --
see~\cite{JT17, Jin19}).

An important tool for our proof is the Kashiwara-Schapira stack
$\kss(\cor_\Lambda)$ of a Lagrangian submanifold $\Lambda$ of a cotangent bundle
$T^*M$.  In~\cite{KS90} Kashiwara and Schapira consider the ``microlocal''
category $\Der(\cor_M; \Omega)$ where $\Omega$ is a subset of $T^*M$.  It is
defined as the quotient of $\Der(\cor_M)$ by the subcategory formed by the $F$
such that $\SSi(F) \cap \Omega =\emptyset$.  When $\Omega$ runs over the open
subsets of $T^*M$ this gives a prestack on $T^*M$ and we consider its associated
stack, say $\kss(\cor_{T^*M})$.  In~\cite{KS90} it is proved that the $\hom$
sheaf in $\kss(\cor_{T^*M})$ is $H^0\muhom$ where $\muhom$ is a variant of
Sato's microlocalization (our stack has a very poor structure because the
triangulated structure does not survive in the stackification and we only obtain
$H^0\muhom$, not $\muhom$).  The stack $\kss(\cor_\Lambda)$ is the substack of
$\kss(\cor_{T^*M})$ formed by the objects with microsupport contained in
$\Lambda$.  One step in the proof of the homotopy equivalence $L \simeq M$ is
the construction of a sheaf representing a given global object of
$\kss(\cor_\Lambda)$, where $\Lambda$ is a conic Lagrangian submanifold of
$T^*(M\times \R)$ deduced from $L$ by adding one variable.  This is done in
Part~\ref{part:quant}.  A similar result is obtained by Viterbo in~\cite{V19}
using Floer homology methods.  More precisely, any given
$\shf \in \kss(\cor_\Lambda)(\Lambda)$ is represented by
$F \in \Der(\cor_{M\times\R})$ such that $F_- := F|_{M \times \{t\}}$, $t\ll0$,
vanishes and $F_+ := F|_{M \times \{t\}}$, $t\gg0$, is locally constant.  Then
we prove that $F \mapsto F_+$ gives an equivalence between the category
$\Der_{[\Lambda],+}$ formed by the $F$ such that $\dot\SSi(F) \subset \Lambda$ and
$F_-\simeq 0$ and the subcategory $\Der_{lc}(\cor_M)$ of $\Der(\cor_M)$ of
locally constant sheaves.  We prove another equivalence between
$\Der_{[\Lambda],+}$ and $\Der_{lc}(\cor_\Lambda)$.  Hence $\Der_{lc}(\cor_M)$ is
equivalent to $\Der_{lc}(\cor_\Lambda)$ and it follows that $L \to M$ is a
homotopy equivalence.  Actually the proof is not so straightforward. We first
prove the fully faithfulness of $F \mapsto F_+$ in Part~\ref{part:quant}, then
use this fully faithfulness in the beginning of Part~\ref{chap:ex_Lag} to prove
some result on the fundamental groups and the vanishing of the first Maslov
class.  For this we need to know that $\kss(\cor_\Lambda)(\Lambda)$ has many
objects. But the first Maslov class is an obstruction to the existence of a
global object in $\kss(\cor_\Lambda)$.  To bypass this problem we first work in
the orbit category of sheaves (see Part~\ref{part:orbcat}) where there is no
such obstruction.  The orbit category contains less information than
$\Der(\cor_M)$ but the above argument works well enough in this framework to
obtain the vanishing of the Maslov class.  Then we can go back to the usual
category of sheaves and we can prove
$\Der_{lc}(\cor_M) \simeq \Der_{lc}(\cor_\Lambda)$, as claimed.

\subsection*{Acknowledgments}

The starting point of this paper is a discussion with Claude Viterbo.  He explained
me his construction of a quantization in the sense of this paper using Floer homology
and asked whether it was possible to obtain it with the methods of algebraic
analysis.  Masaki Kashiwara gave me the idea to glue locally defined simple sheaves
under the assumption that the Maslov class vanishes.  Claire Amiot explained me that
we can use the orbit category to work without this vanishing assumption.  The idea of
applying the involutivity theorem to the $C^0$-rigidity emerged after several
discussions with Claude Viterbo, Pierre Schapira and Vincent Humili\`ere.  The work
on the three cusps conjecture was motivated by discussions with Emmanuel Giroux and
Emmanuel Ferrand.  I also thank Sylvain Courte, Pierre Schapira and Nicolas Vichery
for several remarks and many stimulating discussions.  I thank the anonymous referees
for their careful reading of the paper, for pointing out several mistakes and for
their constructive suggestions, making the exposition more clear.

\part{Microlocal theory of sheaves}

We recall here some results of~\cite{KS90} that we will use   very often.  The notion of microsupport of a sheaf is in particular very
important. We recall its definition and its behaviour under sheaf operations.  We
also recall quickly the definition of Sato's microlocalization and the $\muhom$
functor, later introduced by Kashiwara and Schapira; it will be used in particular
in~\S\ref{chap:KSstack} to define a category of sheaves associated with a Lagrangian
submanifold of a cotangent bundle.

When~\cite{KS90} was written, it was not well understood how to deal with
unbounded derived categories. In particular the theory of microsupport is
written for bounded derived categories of sheaves.  Moreover one of the
fundamental lemmas was proved using an induction on the cohomological degree and
its extension to the unbounded case could be unclear.  However this problem has
been solved in~\cite{RS16} and we will state the results on the microsupport for
unbounded categories (although the sheaves we consider later are always locally
bounded).

\section{Notations}
\label{sec:notations}
We mainly follow the notations of~\cite{KS90}.

\subsubsection*{Geometry}

  Unless otherwise
specified the manifolds we consider here are not ``manifolds with boundary''.  (The reason
is that the definition of the microsupport is not so meaningful at the boundary -- if we
want to deal with a sheaf $F$ on a manifold with boundary $M$, we consider an embedding
$i\colon M \to M^+$, with $M^+$ a usual manifold, and look at the microsupport of the
direct image $\roim{i}(F)$.)  When we say that a submanifold $N$ of $M$ is {\em locally
  closed}, {\em closed} or {\em compact}, we mean that it is locally closed (intersection
of a closed subset and an open subset), closed or compact as a subset of $M$.  We denote
by $\pi_M \cl T^*M\to M$ the cotangent bundle of $M$.   
If $N\subset M$ is a submanifold, we denote by $T^*_NM$ its conormal bundle; it is the
subbundle of the restriction of $T^*M$ over $N$ whose fiber over a point $x\in N$ is the
orthogonal space of $T_xN$, that is, $(T^*_NM)_x = \{\theta \in T^*_xM$; $\langle
T_xN,\theta \rangle = 0\}$.  The zero-section of $T^*M$ will usually be denoted by
$T^*_MM$, or $M$, if there is no ambiguity.  We set $\dT^*M = T^*M\setminus T^*_MM$ and we
denote by $\dot\pi_M\cl\dT^*M\to M$ the projection. For any subset $A$ of $T^*M$ we define
its antipodal $A^a = \{(x;\xi) \in T^*M$; $(x;-\xi) \in A\}$.    We usually denote by $\Delta_M$ the diagonal of $M^2$.

  We denote the normal bundle of $N$ by $T_NM$. It is defined as the
quotient bundle $(N\times_M TM) / TN$, where $N\times_M TM$ is the restriction of the
bundle $TM$ to $N$.

The cotangent bundle $T^*M$ carries an exact symplectic structure.  We denote the
Liouville $1$-form by $\alpha_M$. It is given in local coordinates $(x;\xi)$ by
$\alpha_M = \sum_i \xi_i dx_i$.    The symplectic
structure is then given by the non-degenerate $2$-form $d\alpha_M$.  Since it is
non-degenerate, it induces an isomorphism $H_p \cl T^*_pT^*M \isoto T_pT^*M$,
$p\in T^*M$, such that $\langle v,\theta \rangle = (d\alpha_M)_p(v, H_p(\theta))$,
$v\in T_pT^*M$, $\theta \in T^*_pT^*M$.  We obtain in this way an isomorphism
$H\cl T^*T^*M \isoto TT^*M$, the {\em Hamiltonian isomorphism}.  For the sections of
$T^*T^*M$ it gives $H(dx_i) = -\partial/\partial \xi_i$ and
$H(d\xi_i) = \partial/\partial x_i$.  The Hamiltonian vector field of a function
$h\colon T^*M \to \R$ is $X_h = H(dh)$; in local coordinates we have
$$
X_h(x;\xi) = \sum_i \frac{\partial h}{\partial \xi_i} \frac{\partial}{\partial x_i}
- \frac{\partial h}{\partial x_i} \frac{\partial}{\partial \xi_i} .
$$
A submanifold $\Lambda$ of $T^*M$ is {\em Lagrangian} if, for each $p\in\Lambda$, the
tangent space $T_p\Lambda$ is its own orthogonal with respect to the symplectic
structure of $T_pT^*M$.  If $N\subset M$ is a submanifold, then $T^*_NM$ is a
Lagrangian submanifold of $T^*M$.

Let $f\cl M\to N$ be  a morphism of  manifolds. It induces morphisms
on the cotangent bundles:
\begin{equation}\label{eq:def_derivee_morph}
T^*M \from[f_d] M\times_N T^*N \to[f_\pi] T^*N.
\end{equation}
Let $N\subset M$ be a submanifold and $A\subset M$ any subset. We denote by
$C_N(A) \subset T_NM$ the cone of $A$ along $N$.   We recall its definition in \S\ref{sec:Satomicr} using the
normal deformation of $N$ in $M$; in the case where $M$ is a vector space, $x_0\in N$
and $q\cl M \to T_{N,x_0}M$ denotes the natural quotient map, then
\begin{equation}\label{eq:form_cone1}
C_N(A) \cap T_{N,x_0}M 
= \bigcap_{U} \, \ol{\bigcup_{x\in A\cap (U\setminus \{x_0\})} q([x_0,x))},
\end{equation}
where $U$ runs over the neighborhoods of $x_0$ and $[x_0,x)$ denotes the half
line starting at $x_0$ and containing $x$.

If $A,B$ are two subsets of $M$, we set $C(A,B) = C_{\Delta_M}(A\times B)$.
Identifying $T_{\Delta_M}(M\times M)$ with $TM$ through the first projection,
we consider $C(A,B)$ as a subset of $TM$. If $M$ is a vector space and
$x_0\in M$, we have
\begin{equation}\label{eq:form_cone2}
C(A,B) \cap T_{x_0}M 
= \bigcap_{U} \, \ol{\bigcup_{x\in A \cap U,\; y\in B \cap U,\; x\not=y} q([y,x))},
\end{equation}
where $U$ runs over the neighborhoods of $x_0$.

\subsubsection*{Sheaves}

We consider a commutative unital ring $\cor$ of finite global dimension (we will
use $\cor=\Z$ or $\cor = \Z/2\Z$).  We denote by $\Mod(\cor)$ the category of
$\cor$-modules and by $\Mod(\cor_M)$ the category of sheaves of $\cor$-modules
on $M$.  We denote by $\Der(\cor_M)$ (resp.\ $\Derb(\cor_M)$, $\Derlb(\cor_M)$)
the derived category (resp.\ bounded derived category, locally bounded derived
category) of $\Mod(\cor_M)$. (Hence $\Derlb(\cor_M)$ is the subcategory of
$\Der(\cor_M)$ formed by the $F$ such that $F|_C \in \Derb(\cor_C)$, for any
compact subset $C \subset M$.)  We recall that $\Mod(\cor_M)$ has a fully
faithful embedding (that is, which preserves the $\Hom$ sets) into
$\Der(\cor_M)$ by sending a sheaf to a complex of sheaves concentrated in degree
$0$.  We refer to~\cite{KS90} for results about derived categories and sheaves
and also~\cite{KS06} for the unbounded case.

We recall some standard notations (following mainly~\cite{KS90}).  For a morphism of
manifolds $f\colon M \to N$, we denote by
$\roim{f}, \reim{f} \colon \Der(\cor_M) \to \Der(\cor_N)$ the direct image and proper
direct image functors. We denote by
$\opb{f}, \epb{f} \colon \Der(\cor_N) \to \Der(\cor_M)$ their adjoint functors.  We
thus have adjunctions $(\opb{f}, \roim{f})$ and $(\reim{f}, \epb{f})$.

  For $L \in \Der(\cor)$ we denote by $L_M = \opb{a_M}(L)$ the {\em
  constant sheaf with stalk $L$} on $M$, where $a_M\colon M \to \pt$ is the map to a
point.  For the inclusion $j \cl Z \to M$ of a subset of $M$ and $F\in \Der(\cor_M)$
we often write
$$
F|_Z = \opb{j}F .
$$
We   say that $F \in \Der(\cor_M)$ is {\em locally constant} if any point
of $M$ has a neighborhood $U$ such that $F|_U$ is constant.  If $F$ is concentrated
in degree $0$, we often say ``$F$ is a local system'' instead of ``$F$ is a locally
constant sheaf''.  When $Z$ is locally closed we define
$$
F_Z  = \eim{j}\opb{j}F , \qquad \rsect_Z(F) = \roim{j} \epb{j}F.
$$
When $F= L_M$ is the constant sheaf with stalk $L \in \Der(\cor)$, we set for short
$L_Z = (L_M)_Z = \eim{j}(L_Z)$.  (In case of ambiguity we write $L_{M,Z} = (L_M)_Z$,
but in general the ambient manifold $M$ is understood.)  We recall that, if $Z$ is
closed and $U$ is open, then $\sect(U; \cor_Z) = \{f \colon Z\cap U \to \cor$; $f$ is
locally constant$\}$.  If $Z'$ is locally closed and $Z$ is a closed subset of $Z'$,
we have an exact sequence in $\Mod(\cor_M)$
\begin{equation}
  \label{eq:excisionseq1}
0 \to \cor_{Z'\setminus Z} \to \cor_{Z'} \to \cor_Z \to 0.
\end{equation}

We let $\otimes$ and $\hom$ denote the tensor product and internal $\Hom$ in
$\Mod(\cor_M)$.  We recall that, for $F,G \in \Mod(\cor_M)$, their internal $\Hom$ is
  the sheaf $\hom(F,G)$ whose sections over an
open subset $U$ is given by $\sect(U; \hom(F,G)) = \Hom_{\Mod(\cor_U)}(F|_U, G|_U)$.
The derived functors of $\otimes$ and $\hom$ are denoted $\ltens$ and $\rhom$ (we
sometimes write $F \otimes G$ for $F\ltens G$ when $F$ or $G$ has free stalks over
$\cor$, typically $F \otimes \cor_Z$ for $F \ltens \cor_Z$).  We have an adjunction
$(\ltens, \rhom)$.

We have natural isomorphisms
\begin{equation}
  \label{eq:cohom_support}
F_Z  \simeq F \tens \cor_Z, \qquad \rsect_Z(F) \simeq \rhom(\cor_Z,F) 
\end{equation}
and~\eqref{eq:excisionseq1} gives the excision distinguished triangles, for 
$F\in \Der(\cor_M)$,
\begin{gather*}
  F_{Z'\setminus Z} \to F_{Z'} \to F_Z \to[+1] , \\
  \rsect_Z(F) \to   \rsect_{Z'}(F) \to   \rsect_{Z'\setminus Z}(F) \to[+1] .  
\end{gather*}
If $U$ is open, we let $\rsect(U;-)$ be the derived section functor.  We set
$H^i(U;F) = H^i\rsect(U;F)$.  We let $\sect_c(-)$ be the functor of sections
with compact support and $\rsect_c(-)$ its derived functor.  If
$a_M \colon M \to \pt$ is the map to a point, we thus have
$\rsect(M; F) \simeq \roim{a_M}(F)$ and $\rsect_c(M; F) \simeq \reim{a_M}(F)$.  We
also set
\begin{equation}
  \label{eq:def_HiZ}
  \rsect_Z(U;F) = \rsect(U; \rsect_Z(F)), \quad
  H^i_Z(U;F) = H^i\rsect_Z(U;F) .
\end{equation}

We denote by $\omega_M$ the dualizing complex on $M$.  Since $M$ is a manifold,
$\omega_M$ is actually the orientation sheaf shifted by the dimension, that is,
$\omega_M \simeq or_M[d_M]$.  (In general it is defined by
$\omega_M = \epb{a_M}(\cor_{\pt})$.)    We recall that
the sheaf $or_M$ is locally constant with stalks $\cor$ and, for a connected
orientable open subset $U$ of $M$, $or_M(U)$ is the rank $1$ free module  
$\langle o, o'; o = -o' \rangle$ where $o, o'$ are the two possible orientations of
$U$.  The duality functors are defined by
\begin{equation}\label{eq:dualfct}
\DD_M(\scbul)=\rhom(\scbul,\omega_M), \qquad
\DD'_M(\scbul)=\rhom(\scbul,\cor_M).
\end{equation}
For $f\colon M \to N$ we also use the notation $\omega_{M|N} = \epb{f}(\cor_N)$.  We
have  
$\omega_{M|N} \simeq \omega_M\tens\opb{f}(\DD'_N(\omega_N)) \simeq or_M \tens
\opb{f}(or_N)[d_M-d_N]$ (we remark that $\DD'_N(or_N) \simeq or_N$; indeed there
exists a canonical isomorphism $u \colon or_N \otimes or_N \isoto \cor_N$ such that,
for any $x\in N$ and any choice of orientation $o$ around $x$, we have
$u_x(o\otimes o) = 1$).

For two manifolds $M,N$ and $F\in\Der(\cor_M)$, $G\in\Der(\cor_N)$ we define
$F \letens G \in \Der(\cor_{M\times N})$ by
$$
F \letens G = \opb{q_1}F \ltens \opb{q_2}G ,
$$
where $q_i$ ($i=1,2$) is the $i$-th projection defined on $M\times N$.

We recall some useful facts (see~\cite[\S2, \S3]{KS90}).
\begin{proposition}\label{prop:formulaire}
  Let $f\colon M \to N$ be a morphism of manifolds, $F,G,H \in \Der(\cor_M)$,
  $F',G' \in \Der(\cor_N)$.  Then we have for any $i\in \Z$

  \suiv{(a)} $H^i\RHom(F,G) \simeq \Hom(F,G[i])$,  
  
  \suiv{(b)} $\RHom(\cor_U, F) \simeq \rsect(U;F)$, for $U \subset M$ open,

  \suiv{(c)} $\rsect(U; \rhom(F,G) ) \simeq \RHom(F|_U, G|_U)$, for
  $U \subset M$ open,

  \suiv{(d)} $H^iF$ is the sheaf associated with $V\mapsto H^i(V;F)$,

  \suiv{(e)} $H^i\rhom(F,G)$ is the sheaf associated with
  $V\mapsto \Hom(F|_V, G|_V[i])$,

  \suiv{(f)} $\rhom(F\ltens G, H) \simeq \rhom(F, \rhom(G,H))$,

  \suiv{(g)} $\reim{f}(F\ltens \opb{f}F') \simeq (\reim{f}F) \ltens F'$,
  (projection formula),

  \suiv{(h)}  $\epb{f} \, \rhom(F',G') \simeq \rhom(\opb{f}F', \epb{f}G')$,

  \suiv{(i)} if $f$ is an embedding,
  $\roim{f} \rhom(F,G) \simeq \rhom(\reim{f}F, \roim{f}G)$,
 
  \suiv{(j)} for a Cartesian  diagram 
 \begin{tikzcd}[row sep=5mm]
   M \ar[r, "f"] \ar[d, "g"] & N \ar[d, "g'"] \\
   M' \ar[r, "f'"] & N' \end{tikzcd} we have the base change formulas
 $\opb{f'} \reim{g'}(F') \simeq \reim{g} \opb{f}(F')$
 and  $\epb{f'} \roim{g'}(F') \simeq \roim{g} \epb{f}(F')$.
\end{proposition}
The adjunction between $\ltens$ and $\rhom$ together with
$\cor_U \otimes \cor_{\ol U} \simeq \cor_U$ give
$\Hom(\cor_U, \DD'(\cor_{\ol U})) \simeq \Hom(\cor_U,\cor_M) \simeq H^0(U;
\cor_M)$. The canonical section of this last group gives a morphism
$\cor_U \to \DD'(\cor_{\ol U})$. Similarly we have a natural morphism
$\cor_{\ol U} \to \DD'(\cor_U)$. In the following case they are isomorphisms.

If the inclusion $U \subset M$ is locally homeomorphic to the inclusion
$\mo]-\infty,0\mc[ \times \R^{n-1} \subset \R^n$ (for example, if $\partial U$ is
smooth), then we have the first isomorphism in~\eqref{eq:dual_bordlisse} below.
Indeed, to prove that some morphism is an isomorphism we can work locally and thus
assume that $U = \mo]-\infty,0\mc[ \times \R^{n-1} \subset \R^n$.  For a point
$x\in \partial U$ and an open ball $B$ around $x$, we have
\begin{align*}
&\rsect(B; \DD'(\cor_U)) \simeq \RHom( \cor_U|_B, \cor_M|_B) \\
&\hspace{3.3cm}  \simeq  \RHom( \cor_{U\cap B}, \cor_B) 
  \simeq \rsect(U\cap B; \cor_B) \simeq \cor
\end{align*}
using~(b) and (c). It follows that $(\DD'(\cor_U))_x \simeq (\cor_{\ol U})_x$.  This
also holds for $x\in U$ or $x\in M\setminus \ol U$ and we obtain the claim.  The
second isomorphism in~\eqref{eq:dual_bordlisse} follows from the first one applied to
$M\setminus \ol U$ and the exact sequence
$0 \to \cor_{M \setminus \ol U} \to \cor_M \to \cor_{\ol U} \to 0$.
\begin{equation}
  \label{eq:dual_bordlisse}
 \cor_{\ol U} \isoto  \DD'(\cor_U), \qquad \cor_U \isoto \DD'(\cor_{\ol U}).
\end{equation}

\section{Microsupport}
\label{sec:microsupport}

\subsection{Definition and first properties}

We recall the definition of the microsupport (or singular support) $\SSi(F)$ of
a sheaf $F$, introduced by M.~Kashiwara and P.~Schapira in~\cite{KS82}
and~\cite{KS85}.

Let $F\in \Der(\cor_M)$ and $p= (x_0;\xi_0) \in T^*M$ be given.  We choose a
real $C^1$-function $\phi$ on $M$ satisfying $d\phi(x_0) = \xi_0$ and we
consider the restriction morphism ``in the direction $p$'' for a given degree
$i\in \Z$:
\begin{equation}
\label{eq:restr_defSS}
H^i F_{x_0} \simeq \varinjlim_U H^i(U;F)
\to  \varinjlim_U H^i(U \cap \{x;\, \phi(x) < \phi(x_0)\}; F) ,
\end{equation}
where $U$ runs over the open neighborhoods of $x_0$.  We are interested in the
points $p$ where this morphism is not an isomorphism (for some $\phi$ and $i$).
Taking the cone of the restriction morphism we obtain the following definition.

\begin{definition}{\rm (see~\cite[Def.~5.1.2]{KS90})}
\label{def:SSF}
Let $F\in \Der(\cor_M)$. We define $\SSi(F) \subset T^*M$ as the closure of the
set of points $(x_0;\xi_0) \in T^*M$ such that there exists a real
$C^1$-function $\phi$ on $M$ satisfying $d\phi(x_0) = \xi_0$ and
$(\rsect_{\{x;\, \phi(x)\geq \phi(x_0)\}} (F))_{x_0} \not\simeq 0$.

We set $\dot\SSi(F) = \SSi(F) \cap \dT^*M$.
\end{definition}

The following properties are easy consequences of the definition: \\
(a) the microsupport is closed and $\rspos$-conic, that is,
invariant by the action of  $(\rspos,\times)$ on $T^*M$, \\
(b) $\SSi(F)\cap T^*_MM =\pi_M(\SSi(F))=\supp(F)$, \\
(c) the microsupport satisfies the triangular inequality: if $F_1\to F_2\to
F_3\to[+1]$ is a distinguished triangle in $\Der(\cor_M)$, then
$\SSi(F_i)\subset\SSi(F_j)\cup\SSi(F_k)$ for all $i,j,k\in\{1,2,3\}$ with
$j\not=k$.

\begin{notation}\label{not:cat_micsupp_fixe}   
  Following~\cite[\S6.1]{KS90}, for an $\rspos$-conic
  subset $S$ of $T^*M$, we denote by $\Der_S(\cor_M)$ the full subcategory of
  $\Der(\cor_M)$ formed by the $F$ with $\SSi(F) \subset S$.  We also set
  $\Der_{[S]}(\cor_M) = \Der_{S \cup T^*_MM}(\cor_M)$.  We denote by
  $\Der_{(S)}(\cor_M)$ the full subcategory of $\Der(\cor_M)$ formed by the $F$ for
  which there exists a neighborhood $\Omega$ of $S$ in $T^*M$ such that
  $\SSi(F) \cap \Omega\subset S$.  By the property~(c) above, all these subcategories
  are triangulated.  
\end{notation}

\begin{example}\label{ex:microsupport}
  (i) If $F$ is a non-zero local system on a connected manifold $M$, then
  $\SSi(F)=T^*_MM$, the zero-section. Conversely, if $\SSi(F)\subset T^*_MM$,
  then the cohomology sheaves $H^i(F)$ are local systems, for all $i\in \Z$.
  (We say that $F$ is {\em locally constant} for short.)  This follows for
  example from Proposition~\ref{prop:iminvproj} below.

  \smallskip\noindent (ii) If $N$ is a closed submanifold (we only mean $N$ is a closed
  subset of $M$ -- see~\S\ref{sec:notations}) of $M$ and $F=\cor_N$, then
  $\SSi(F)=T^*_NM$.

\smallskip\noindent
(iii) Let $U \subset M$ be an open subset with smooth boundary. Then 
\begin{align*}
\SSi(\cor_U) & = (U\times_MT^*_MM) \cup T^{*,out}_{\partial U}M, \\
\SSi(\cor_{\ol U}) & = (U\times_MT^*_MM) \cup (T^{*,out}_{\partial U}M)^a ,
\end{align*}
where $T^{*,out}_{\partial U}M = \{(x;\lambda df(x))$; $f(x) = 0$, $\lambda\leq 0\}$, if
$U = \{f>0\}$ and $df\not=0$ on $\partial U$ (see Fig.~\ref{fig:SSdisque}).    To remember the directions of the
microsupports we can take $\phi = f$ in Definition~\ref{def:SSF} and use $\rsect_{\{x;\,
  \phi(x)\geq 0\}} (F) \simeq F$ for both $F = \cor_U$ and $F = \cor_{\ol U}$ (since
$\supp(F) \subset \{x;\, \phi(x)\geq 0\}$).  This implies $(\rsect_{\{x;\, \phi(x)\geq
  0\}} (\cor_U))_{x_0} \simeq (\cor_U)_{x_0} \simeq 0$ and $(\rsect_{\{x;\, \phi(x)\geq
  0\}} (\cor_{\ol U}))_{x_0} \simeq (\cor_{\ol U})_{x_0} \simeq \cor$, for $x_0 \in
\partial U$, and both formulas are consistent with the above description of the
microsupport.

\begin{figure}[ht]
\newcommand{\monellipse}{ellipse [x radius=15mm, y radius=8mm]}
\begin{tikzpicture}[decoration={markings,
mark=between positions 0 and 1 step 0.0715
with { \draw[->](0,0) -- (0,-0.2);}} ]

\node at (0,0) {$\SSi(\cor_U)$};
\fill [fill=gray!20] (2.7,0) \monellipse ;
\draw [decorate] (2.7,0) \monellipse ;
\draw [dotted] (2.7,0) \monellipse ;

\end{tikzpicture}
\hspace{1cm}
\begin{tikzpicture}[decoration={markings,
mark=between positions 0 and 1 step 0.0715
with { \draw[->](0,0) -- (0,0.2);}} ]

\node at (0,0) {$\SSi(\cor_{\ol U})$};
\fill [fill=gray!20] (2.7,0) \monellipse ;
\draw [postaction={decorate,draw}] (2.7,0) \monellipse ;

\end{tikzpicture}
\caption{We identify covectors with vectors to draw the microsupport. The microsupport of
  the constant sheaf on an open (resp. closed) subset with smooth boundary points outward
  (resp. inward). }
 \label{fig:SSdisque}
\end{figure}

\smallskip\noindent (iv) Let $\lambda$ be a closed convex cone with vertex at
$0$ in $E=\R^n$.  By~\cite[Prop.~5.3.1]{KS90} we have
$\SSi(\cor_\lambda) \cap T^*_0E = \lambda^\circ$, where $\lambda^\circ$ is the
polar cone of $\lambda$:
\begin{equation}\label{eq:def_polar_cone}
\lambda^\circ = \{\xi\in E^*; \; \langle v,\xi \rangle \geq 0
 \text{ for all $v\in \lambda\}$.}
\end{equation}
For a given $a>0$ we define $h_a \cl E \to E$, $x\mapsto ax$.  We clearly have
$\opb{h_a}(\cor_\lambda) \simeq \cor_\lambda$. Hence $\SSi(\cor_\lambda)$ is invariant by
the map induced by $h_a$ on $T^*E$.  Since the microsupport is closed, we deduce the rough
bound $\SSi(\cor_\lambda) \subset E \times \lambda^\circ$ by letting $a\to 0$ (see
Fig.~\ref{fig:SScone}).

\begin{figure}[ht]
  \begin{tikzpicture}[decoration={markings,
mark=between positions 0 and 1 step 0.082
with { \draw[->](0,0) -- (0,0.2);}} ]

\coordinate (A) at (0,0); \coordinate (B) at (-1,2);
\coordinate (C) at (1,2); \coordinate (D) at (.2,.1); 
\fill [fill=gray!20] (B) -- (A) -- (C) -- cycle;
\draw (B)--(A)--(C);

\draw [postaction={decorate,draw}] (B)--(A)--(C);
\fill [fill=black] (A) -- (D) arc[radius=.22cm, start angle=27, end angle=153];
\end{tikzpicture}
\caption{At the vertex $\SSi(\cor_\lambda)$ is the polar cone of $\lambda$.}
\label{fig:SScone}
\end{figure}

\sui(v) Let $F\in \Der(\cor_\R)$ be such that $\dot\SSi(F) = \{(0;\xi)$; $\xi > 0\}$.  On
one hand, it follows from~(i) that $F|_{\R\setminus \{0\}}$ is locally constant, hence
$F|_{U_\pm}$ is constant, where $U_\pm = \{\pm x >0\}$.  In particular $(\rsect_{\{x;\,
  \phi(x)\geq \phi(x_0)\}} (F))_{x_0} \simeq 0$ for any $x_0\not= 0$ and any function
$\phi$ with $d\phi(x_0)\not=0$.  On the other hand,   by the
definition of the microsupport there exist points $x_1$ arbitrarily close to $0$ (maybe
equal to $0$) together with functions $\phi$ such that $d\phi(x_1) >0$ and
$(\rsect_{\{x;\, \phi(x)\geq \phi(x_1)\}} (F))_{x_1} \not\simeq 0$.  We have seen that the
case $x_1 \not=0$ is excluded.  Hence $(\rsect_{\{x;\, \phi(x)\geq \phi(0)\}} (F))_0
\not\simeq 0$ for some function $\phi$ with $d\phi(0)>0$. This means $(\rsect_ZF)_0 \not=
0$, where $Z = \ol{U_+}$. Let us set $B_r = \mo]-r,r[$.  Since $F|_{U_\pm}$ is constant,
$\rsect(B_r; \rsect_ZF)$ is independent of $r>0$.  Hence we have in fact $(\rsect_ZF)_0
\simeq \rsect(\R; \rsect_ZF)$.  Let us set $E = (\rsect_ZF)_0$. By the adjunction
$(\opb{a},\roim{a})$ where $a\colon \R \to\pt$ is the projection, the isomorphism $E\isoto
\rsect(\R; \rsect_ZF) = \roim{a}\rsect_ZF$ gives a morphism $u \colon E_\R \to \rsect_ZF$.
By the adjunction $(\ltens, \rhom)$ we obtain from $u$ another morphism $v \colon E_Z \to
F$.  We have $(\rsect_Z(E_Z))_0 \simeq E$ and the morphism $v$ induces an isomorphism
$(\rsect_Z(E_Z))_0 \isoto (\rsect_ZF)_0$.  In other words, defining $G$ by a distinguished
triangle $G \to E_Z \to[v] F \to[+1]$, we have $(\rsect_ZG)_0 \simeq 0$.  By the triangle
inequality we deduce that $\dot\SSi(G) = \emptyset$, hence $G$ is a constant sheaf on $\R$
and we can write $G = E'_\R$ for some $E'\in \Der(\cor)$.  In conclusion there exist $E,
E' \in \Der(\cor)$ and a distinguished triangle
$$
E'_\R \to E_Z \to F \to[+1] .
$$
Conversely a sheaf $F$ defined by such a distinguished triangle satisfies
$\dot\SSi(F) = \{(0;\xi)$; $\xi > 0\}$ and $(\rsect_ZF)_0 \simeq E$.
\end{example}

\subsection{Functorial operations}
\label{sec:funct_op_SS}

\begin{proposition}{\rm (See~\cite[Prop.~5.4.4]{KS90}.)}
\label{prop:oim} 
Let $f\cl M\to N$ be a morphism of manifolds and let $F\in\Der(\cor_M)$.
We assume that $f$ is proper on $\supp(F)$. Then $\SSi(\reim{f}F)\subset
f_\pi\opb{f_d}\SSi(F)$, with equality when $f$ is a closed embedding (see the beginning
of~\S\ref{sec:notations} for ``closed'').
\end{proposition}

\begin{example}
\label{ex:loccst_submfd}
With the notations of Proposition~\ref{prop:oim} we assume that $f$ is an
embedding.  Let $F \in \Der(\cor_N)$ be such that $\SSi(F) \subset
T^*_MN$. Then there exists a locally constant $G \in \Der(\cor_M)$ such that $F
\simeq \roim{f}G$.  Indeed, since $\SSi(F) \cap T^*(N\setminus M) = \emptyset$,
we have $F|_{N\setminus M} \simeq 0$, by the property~(b) after
Definition~\ref{def:SSF}.  Hence $F \simeq \roim{f}G$ where $G = \opb{f}F$.  Then
$\SSi(F) = f_\pi\opb{f_d}\SSi(G)$ by Proposition~\ref{prop:oim} and we deduce
$\SSi(G) \subset T^*_NN$.  Now the result follows from
Example~\ref{ex:microsupport}-(i).
\end{example}

We  recall some notations of~\cite[Def.~6.2.3]{KS90}.

Let $M$ and $N$ be two manifolds and $f \cl M \to N$ a morphism.  Let
$\Gamma_f \subset M\times N$ be the graph of $f$. We have
$T^*_{\Gamma_f}(M\times N) \simeq M \times_N T^*N$.
If $X$ is a manifold and $\Lambda$ a Lagrangian submanifold of $T^*X$, the
Hamiltonian isomorphism identifies $T_\Lambda T^*X$ with $T^*\Lambda$. Applying
this to $X = M\times N$ and $\Lambda = T^*_{\Gamma_f}(M\times N)$,
we obtain a natural identification
$$
T_{T^*_{\Gamma_f}(M\times N)} T^*(M\times N)
\simeq T^*(M \times_N T^*N) .
$$
We also remark that $T^*M$ has a natural embedding in $T^*(M \times_N T^*N)$ as
$T^*M = T^*(M \times_N T^*_NN)$.  For a conic subset $A \subset T^*N$ we set
\begin{equation}
\label{eq:def_fsharpA}
f^\sharp(A) = T^*M \cap C_{T^*_{\Gamma_f}(M\times N)}(T^*M \times A) .  
\end{equation}

\begin{example}
\label{ex:fsharp_embed}
  Let $f$ be the embedding of $\R^m$ in $\R^{n+m}$. We take coordinates
  $(x',x'';\xi',\xi'')$ on $T^*\R^{n+m}$ such that $\R^m = \{x'=0\}$. Then
  $(x''_\infty,\xi''_\infty) \in f^\sharp(A)$ if and only if there exists a
  sequence $\{(x'_n,x''_n;\xi'_n,\xi''_n)\}$ in $A$ such that $x'_n \to 0$,
  $x''_n \to x''_\infty$, $\xi''_n \to \xi''_\infty$ and $|x'_n| \, |\xi'_n| \to
  0$.
\end{example}

\begin{example}
\label{ex:fsharp_nonchar}
With the notations~\eqref{eq:def_derivee_morph} we say that $f$ and $A$ are {\em
  non-characteristic}  if
$$
\opb{f_\pi}(A)\cap f_d^{-1}(T^*_MM) \subset M\times_NT^*_NN .
$$
If $A \subset T^*N$ is closed conic and $f$  and $A$ are non-characteristic, then
$f^\sharp(A) = f_d\opb{f_\pi}(A)$.
\end{example}

\begin{theorem}{\rm (See~\cite[Cor.~6.4.4]{KS90}.)}
\label{thm:iminv} 
Let $f \cl M \to N$ be a morphism of manifolds and let $F \in
\Der(\cor_N)$. Then, using the notation $f^\sharp$ of~\eqref{eq:def_fsharpA}
(see also Example~\ref{ex:fsharp_nonchar}),
$$
\SSi(\opb{f}F) \subset f^\sharp(\SSi(F)) 
\qquad \text{and} \qquad
\SSi(\epb{f}F) \subset f^\sharp(\SSi(F)) .
$$
If $f$ is smooth, these inclusions are equalities.  If $f$ is non-characteristic
for $\SSi(F)$, then the natural morphism
$$
\opb{f}F \tens \omega_{M|N} \to \epb{f}(F)
$$
is an isomorphism, where $\omega_{M|N} \simeq or_M \tens \opb{f}(or_N)[d_M-d_N]$
is the relative dualizing complex.
\end{theorem}

\begin{proposition}{\rm (See~\cite[Prop.~5.4.5]{KS90}.)}
\label{prop:iminvproj} 
Let $N, I$ be manifolds. We assume that $I$ is contractible.
Let $f \cl N\times
I \to N$ be the projection.  Let $F\in\Der(\cor_{N\times I})$.  Then
$\SSi(F)\subset T^*N \times T^*_II$ if and only if $\opb{f} \roim{f} (F) \isoto
F$.
\end{proposition}

\begin{example}\label{ex:SS=conormal_hypersurface}
  We set $M = \R^n$, $S = \{x_n = 0\}$, $Z = \{x_n \geq 0\}$ and
  $\Lambda = \{(x_1,\dots,x_{n-1}, 0; \ul 0, \xi_n)$; $\xi_n>0\}$.  Let
  $F \in \Der(\cor_M)$ be such that $\dot\SSi(F) = \Lambda$.  Then there exist
  $E, E' \in \Der(\cor)$ and a distinguished triangle
  $$
  E'_M \to E_Z \to F \to[+1] .
$$
Indeed, we can identify $M$ with a product $I \times N$, where $I = \R^{n-1}$
and $N = \R$.  Then $\Lambda = T^*_II \times \{(0;\xi_n)$; $\xi_n>0\}$.  By
Proposition~\ref{prop:iminvproj} we can write $F \simeq \opb{f} G$ for some
$G \in \Der(\cor_N)$, where   $f \cl N\times I \to N$ is the projection.  By
Theorem~\ref{thm:iminv} we must have $\dot\SSi(G) = \{(0;\xi_n)$; $\xi_n>0\}$
and we conclude with Example~\ref{ex:microsupport}-(v).
\end{example}

\begin{example}[Conic sheaves]\label{ex:conicsheaves}
  (i)   We set $S^{n-1} = (\R^n\setminus\{0\})/\rspos$ and let
  $q \colon \R^n\setminus\{0\} \to S^{n-1}$ be the quotient map.  We say that
  $F\in \Der(\cor_{\R^n})$ is {\em conic} if there exists $G\in \Der(\cor_{S^{n-1}})$
  such that $F|_{\R^n \setminus\{0\}} \simeq \opb{q}G$.  By
  Proposition~\ref{prop:iminvproj} $F$ is conic if and only if, for any
  $x\not=0 \in \R^n$, we have $\SSi(F) \cap T^*_x\R^n \subset (T^*_l\R^n)_x$, where
  $l=\R \cdot x$ is the line generated by $x$.

  \sui(ii) In Example~\ref{ex:microsupport}-(iv) we have given a bound for
  $\SSi(\cor_\lambda)$ when $\lambda$ is a closed convex cone in $\R^n$, namely
  $\SSi(\cor_\lambda) \subset \R^n \times \lambda^\circ$ and
  $\SSi(\cor_\lambda) \cap T^*_0\R^n = \lambda^\circ$, where $\lambda^\circ$ is
  the polar cone of $\lambda$.  Since $\cor_\lambda$ is conic, the above
  discussion also implies
  $\SSi(\cor_\lambda) \cap T^*_x\R^n \subset (T^*_l\R^n)_x$, where
  $l=\R \cdot x$, for any $x\not=0 \in \R^n$.

  We remark that $\Int(\lambda^\circ) = \{\xi\in (\R^n)^*$;
  $\langle v,\xi \rangle > 0$ for all $v\not=0 \in \lambda\}$ (if $\lambda$
  contains a line, then $\Int(\lambda^\circ) = \emptyset$).  In particular, for
  $x\not=0 \in \lambda$, we have
  $\Int(\lambda^\circ) \cap (T^*_l\R^n)_x = \emptyset$ and we deduce the more
  precise bound
  \begin{equation}
    \label{eq:borneSScone_bord}
  \SSi(\cor_\lambda) \cap T^*(\R^n\setminus\{0\}) \subset (\R^n\setminus\{0\})
  \times \partial\lambda^\circ,
  \end{equation}
  where
  $\partial\lambda^\circ = \lambda^\circ \setminus \Int(\lambda^\circ)$.
\end{example}

Let $\delta_M \cl M \to M^2$ be the diagonal embedding.  For conic subsets $A,B
\subset T^*M$ we set
\begin{equation}
\label{eq:def_hplus}
A \hplus B = \delta_M^\sharp(A\times B) .
\end{equation}
In local coordinates $A \hplus B$ is the set of $(x;\xi)$ such that there exist
two sequences $(x_n;\xi_n)$ in $A$ and $(y_n;\eta_n)$ in $B$ such that $x_n, y_n
\to x$, $\xi_n+\eta_n \to \xi$ and $|x_n-y_n| |\xi_n| \to 0$ when $n\to \infty$.

\begin{example}
\label{ex:hplus_noncar}
If $A,B \subset T^*M$ are closed conic subsets such that $A^a \cap B \subset
T^*_MM$, then $A \hplus B = A + B$.
\end{example}

For the definition of \emph{cohomologically constructible} we refer to~\cite[\S
3.4]{KS90}.  An example of a cohomologically constructible sheaf $F$ is given by
the case where $F$ is constructible with respect to a Whitney stratification
(that is, the restriction of $F$ to each stratum is locally constant and of
finite rank).

\begin{theorem}{\rm (See~\cite[Cor.~6.4.5]{KS90}.)}
\label{thm:SSrhom} 
Let $F,G \in\Der(\cor_M)$. Then, using the notation $\hplus$
of~\eqref{eq:def_hplus} (see also Example~\ref{ex:hplus_noncar}),
\begin{align*}
\SSi(F \ltens G) &\subset \SSi(F) \hplus \SSi(G) ,  \\
\SSi(\rhom(F,G)) &\subset \SSi(F)^a \hplus \SSi(G) , \\
\SSi(\DD'F ) &= \SSi(F)^a .
\end{align*}
If   we assume that $\SSi(F)\cap\SSi(G)\subset T^*_MM$ and that
$F$ is cohomologically constructible, then the natural morphism
$\DD'F \ltens G \to \rhom(F,G)$ is an isomorphism.
\end{theorem}

\begin{remark}\label{rem:SSexttens}
  Let $M,N$ be manifolds and let $q_1$, $q_2$ be the projections from
  $M\times N$ to $M$, $N$.  For $F \in\Der(\cor_M)$ and $G \in \Der(\cor_N)$
  Theorem~\ref{thm:SSrhom} implies
  $\SSi(\rhom(\opb{q_1} F, \opb{q_2}G )) \subset \SSi(F)^a \times \SSi(G)$ and
  $\SSi(\opb{q_1} F \ltens \opb{q_2}G ) \subset \SSi(F) \times \SSi(G)$.
\end{remark}

Using the $\hplus$ operation we can give a version of Proposition~\ref{prop:oim}
for an open embedding. We only state the case where the boundary of the open
subset is smooth (see~\cite{KS90} for the general case).

\begin{theorem}{\rm (See~\cite[Thm.~6.3.1]{KS90}.)}
\label{thm:oim_open} 
Let $j\colon U \hookrightarrow M$ be the embedding of an open subset with a
smooth boundary.  Let $F \in \Der(\cor_U)$.  Then
$\SSi(\roim{j}F) \subset \SSi(F) \hplus \SSi(\cor_U)^a$ and
$\SSi(\reim{j}F) \subset \SSi(F) \hplus \SSi(\cor_U)$.
\end{theorem}

The next result follows immediately from Proposition~\ref{prop:oim} and
Example~\ref{ex:microsupport}~(i).   It is a
particular case of a microlocal Morse result (see~\cite[Cor.~5.4.19]{KS90}), the
classical theory corresponding to the constant sheaf $F=\cor_M$.
\begin{corollary}\label{cor:Morse}
Let $F\in\Der(\cor_M)$, let $\phi\cl M\to\R$ be a function of class $C^1$ and
assume that $\phi$ is proper on $\supp(F)$.  Let $a<b$ in $\R$ and assume that
$d\phi(x)\notin\SSi(F)$ for $a\leq \phi(x)<b$. Then the natural morphisms
$\rsect(\opb{\phi}(\mo]-\infty,b[);F) \to \rsect(\opb{\phi}(\mo]-\infty,a[);F)$
and
$\rsect_{\opb{\phi}([b,+\infty[)}(M;F) \to \rsect_{\opb{\phi}([a,+\infty[)}(M;F)$
are isomorphisms.
\end{corollary}

Here is another useful consequence of the properties of the microsupport which
appears in~\cite{N15}.
\begin{corollary}\label{cor:homFtGt_indpdtt}
  Let $M$ be a manifold and $I$ an open interval of $\R$.  Let $F$,
  $G \in \Der(\cor_{M\times I})$. We assume
  \begin{itemize}
  \item [(i)] the projection $\supp(F) \cap \supp(G) \to I$ is proper,
  \item [(ii)] $F$, $G$ are non-characteristic for all maps
    $i_t \colon M \times \{t\} \to M\times I$, $t\in I$, that is,
    $\dot\SSi(A) \cap (T^*_MM \times T^*_tI) = \emptyset$ for $A=F$, $G$,
  \item [(iii)] setting $\Lambda_t = i_t^\sharp(\dot\SSi(F))$ and
    $\Lambda'_t = i_t^\sharp(\dot\SSi(G))$, we have
    $\Lambda_t \cap \Lambda'_t = \emptyset$ for all $t\in I$.
  \end{itemize}
  Then $\RHom(\opb{i_t}F, \opb{i_t}G)$ is independent of $t$.
\end{corollary}
\begin{proof}
  We   recall that in the
  non-characteristic case we have $i_t^\sharp = {i_t}_d {i_t}_\pi^{-1}$ (see
  Example~\ref{ex:fsharp_nonchar}).  We set $H = \rhom(F,G)$.  Since $\Lambda := \SSi(F)$
  and $\Lambda' :=\SSi(G)$ are non-characteristic for $i_t$ and $\Lambda_t \cap \Lambda'_t
  = \emptyset$, we can see that $\Lambda \cap \Lambda' \subset T^*_{M\times I}(M\times I)$
  and then $\SSi(H) \subset \Lambda^a + \Lambda'$.  We can see moreover that $\Lambda^a +
  \Lambda'$ is also non-characteristic for $i_t$.  Hence Theorem~\ref{thm:iminv} gives
  $\opb{i_t}G \simeq \epb{i_t}G [1]$ and $\opb{i_t}H \simeq \epb{i_t}H [1]$. Then it
  follows from Proposition~\ref{prop:formulaire}-(h) that $\rhom(\opb{i_t}F, \opb{i_t}G)
  \simeq \opb{i_t}H$.  By the base change formula we have $\rsect(M; \opb{i_t}H) \simeq
  (\reim{q}H)_t$, where $q \colon M \times I \to I$ is the projection.  Hence it is enough
  to check that $\reim{q}H$ is locally constant, that is, $\SSi(\reim{q}H) \subset
  T^*_II$.  By Proposition~\ref{prop:oim} this follows from the fact that $\SSi(H)$ is
  non-characteristic for each $i_t$.
\end{proof}

\begin{remark}\label{rem:homFtGt_indpdtt}
  The particular case $F = \cor_{M\times I}$ of
  Corollary~\ref{cor:homFtGt_indpdtt} gives the following. Let
  $G \in \Der(\cor_{M\times I})$.  We assume that $G$ is non-characteristic for
  all maps $i_t \colon M \times \{t\} \to M\times I$, $t\in I$, and that the map
  $ \supp(G) \to I$ is proper. Then the sections $\rsect(M; \opb{i_t}G)$ are
  independent of $t$.
\end{remark}

\subsection{Constructibility}
\label{sec:constructibility}

  We say a few words about the notion of constructibility for sheaves
and its relation with the microsupport.

A stratification $\Sigma = \{\Sigma_i\}_{i\in I}$ of a manifold $M$ is a partition
$M = \bigsqcup_{i \in I} \Sigma_i$ by locally closed subsets such that the closure of
each stratum is a union of strata.  We will always assume that our stratifications
are locally finite (any compact subset meets finitely many strata).  We say that
$F\in \Der(\cor_M)$ is {\em weakly constructible} with respect to $\Sigma$ if
$F|_{\Sigma_i}$ is locally constant for each $i\in I$.  If $\cor$ is a Noetherian
ring, we say that $F$ is constructible with respect to $\Sigma$ if
$F \in \Derb(\cor_M)$, $F$ is weakly constructible and, for each $x\in M$ and
$j\in \Z$, the stalk $H^j(F)_x$ is finitely generated over $\cor$.  We recall an
important property of stratifications introduced by Kashiwara and Schapira because it
helps to bound the microsupport of sheaves.
\begin{definition}[Def.~8.3.19 of~\cite{KS90}]
  \label{def:mustrat}
  A stratification $\Sigma = \{\Sigma_i\}_{i\in I}$ of $M$ satisfies the {\em
    $\mu$-condition} if the strata are locally closed submanifolds and
  $\Lambda_\Sigma \hplus \Lambda_\Sigma = \Lambda_\Sigma$, where
  $\Lambda_\Sigma = \bigsqcup_{i\in I} T^*_{\Sigma_i}M$.
\end{definition}
In a subanalytic framework, the $\mu$-condition implies Whitney's conditions (a) and
(b) for any two strata (see Exercise~VIII.12 of~\cite{KS90}).  In Chapter 8
of~\cite{KS90} it is proved that, if $M$ is real analytic, any stratification by
subanalytic subsets can be refined into a subanalytic stratification satisfying the
$\mu$-condition.  The following proposition appears in the same chapter but the proof
does not require analyticity.

\begin{proposition}[Prop.~8.4.1 in~\cite{KS90}]
  \label{prop:mustratif}
  Let $\Sigma$ be a stratification of $M$ satisfying the $\mu$-condition.  Then
  $F \in \Der(\cor_M)$ is weakly constructible with respect to $\Sigma$ if and
  only if $\SSi(F) \subset \Lambda_\Sigma$.  
\end{proposition}
\begin{proof}
  (i) We first assume that $\SSi(F) \subset \Lambda_\Sigma$.  Let $\Sigma_i$ be a
  stratum and let $U$ be a neighborhood of $\Sigma_i$ such that $\Sigma_i$ is closed
  in $U$.  By Theorem~\ref{thm:SSrhom} we have
  $\SSi(F_{\Sigma_i}) \cap T^*U \subset (\Lambda_\Sigma \hplus T^*_{\Sigma_i}M) \cap
  T^*U$.  Since
  $\Lambda_\Sigma \hplus T^*_{\Sigma_i}M \subset \Lambda_\Sigma \hplus \Lambda_\Sigma
  = \Lambda_\Sigma$ and $F_{\Sigma_i}$ vanishes outside $\Sigma_i$ we deduce
  $\SSi(F_{\Sigma_i}) \cap T^*U \subset T^*_{\Sigma_i}M \cap T^*U$.  It then follows
  from Example~\ref{ex:loccst_submfd} that
  $F|_{\Sigma_i} (= F_{\Sigma_i}|_{\Sigma_i})$ is locally constant.

  \sui(ii) Now we assume that $F$ is weakly constructible with respect to $\Sigma$.
  Since the problem is local on $M$ we can assume that the stratification is finite,
  say $\Sigma = \{\Sigma_1,\dots, \Sigma_N\}$.  We also assume that
  $\dim(\Sigma_i) \geq \dim(\Sigma_{i+1})$. Hence
  $\Sigma^k = \bigsqcup_{i=1}^k \Sigma_i$ is an open subset of $M$.  We prove by
  induction on $k$ that
  $\SSi(F) \cap \pi_M^{-1}(\Sigma^k) \subset \Lambda_\Sigma\cap
  \pi_M^{-1}(\Sigma^k)$.  The case $k=1$ is clear because $F|_{\Sigma_1}$ is locally
  constant.  Let us assume that the claim is proved for $k$.  We have the excision
  distinguished triangle
  $F_{\Sigma^k} \to F|_{\Sigma^{k+1}} \to F_{\Sigma_{k+1}} \to[+1]$ over
  $\Sigma^{k+1}$. We recall that $F_{\Sigma^k} \simeq \eim{j}\oim{j}(F)$, where $j$
  is the inclusion of $\Sigma^k$ in $\Sigma^{k+1}$. By Theorem~\ref{thm:oim_open},
  the induction hypothesis and the $\mu$-condition, we have
  $\SSi(F_{\Sigma^k}) \cap \pi_M^{-1}(\Sigma^{k+1}) \subset \Lambda_\Sigma\cap
  \pi_M^{-1}(\Sigma^{k+1})$.  Since $F|_{\Sigma_{k+1}}$ is locally constant and
  $\Sigma_{k+1}$ is smooth and closed in $\Sigma^{k+1}$ we have
  $\SSi(F_{\Sigma^{k+1}}) \cap \pi_M^{-1}(\Sigma^{k+1}) \subset T^*_{\Sigma^{k+1}}M
  \cap \pi_M^{-1}(\Sigma^{k+1})$.  Now the result follows from the triangular
  inequality for the microsupport.
\end{proof}

\begin{corollary}\label{cor:tensor_prod_construct}
  Let $\Sigma$ be a stratification of $M$ satisfying the $\mu$-condition and let
  $F, G \in \Derb(\cor_M)$ be given. We assume that $F$ and $G$ are weakly
  constructible with respect to $\Sigma$. Then $F\ltens G$ and $\rhom(F,G)$ are
  also weakly constructible with respect to $\Sigma$.
\end{corollary}
\begin{proof}
  This follows from Proposition~\ref{prop:mustratif} and
  Theorem~\ref{thm:SSrhom} (the case of $F\ltens G$ also follows from the fact
  that $\ltens$ commutes with the inverse image and does not need the
  $\mu$-condition).
\end{proof}

\begin{remark}
  \label{rem:tensor_prod_construct}
  (a) If $M$ is of dimension $1$, a (locally finite) stratification $\Sigma$ is
  the data of a discrete set of points, say $\Sigma^0$, and the connected
  components of $M\setminus \Sigma^0$.  Then $F \in \Derb(\cor_M)$ is weakly
  constructible with respect to $\Sigma$ if and only if
  $F|_{M\setminus \Sigma^0}$ is locally constant (indeed the condition of being
  locally constant at the points of $\Sigma^0$ is empty).  Since $\ltens$ and
  $\rhom$ both commute with the restriction to an open subset,
  Corollary~\ref{cor:tensor_prod_construct} is obvious in the case of dimension
  $1$.

  \sui (b) More generally, let us assume that the stratification $\Sigma$
  consists of a discrete set of disjoint closed (not only locally closed)
  submanifolds and the connected components of their complement. We see easily
  that $\Sigma$ satisfies the $\mu$-condition.  Moreover
  Proposition~\ref{prop:mustratif} follows directly from
  Proposition~\ref{prop:iminvproj}.
\end{remark}

If $M$ is a real analytic manifold and $\Lambda \subset \dT^*M$ is an $\rspos$-conic
real analytic Lagrangian submanifold, Corollary~8.3.22 of~\cite{KS90} says that there
exists a stratification $\Sigma$ of $M$ satisfying the $\mu$-condition such that
$\Lambda \subset \Lambda_\Sigma$.  We can deform a $\Cinf$ Lagrangian submanifold
into an analytic one and obtain the next proposition.  We thank Sylvain Courte for
his suggestions about this result.
\begin{proposition}
  \label{prop:deformeversmustrat}
  Let $M$ be a manifold and $\Lambda \subset \dT^*M$ an $\rspos$-conic closed
  Lagrangian submanifold (both $M$ and $\Lambda$ of class $\Cinf$).  Then there
  exists an $\rspos$-homogeneous Hamiltonian isotopy
  $\Psi\cl \dT^*M \times [0,1] \to \dT^*M$ such that $\Psi_0= \id$ and
  $\Psi_1(\Lambda) \subset \Lambda_\Sigma$, for some stratification $\Sigma$ of $M$
  satisfying the $\mu$-condition.  Moreover $\Psi_1(\Lambda)$ can be chosen
  arbitrarily close to $\Lambda$.
\end{proposition}
\begin{proof}
  We reduce to the real analytic case, following the ideas in~\cite{KL00}, and apply
  results of~\cite{KS90}.

  \sui(i) We set $\bar\Lambda = \Lambda/\rspos$ and $S^*M = \dT^*M/\rspos$.  Then
  $\bar\Lambda$ is a Legendrian submanifold of $S^*M$ and the contact version of the
  Weinstein neighborhood theorem gives a contact embedding $j\colon U \to S^*M$ of
  class $\Cinf$, where $U$ is an open neighborhood of $\bar\Lambda$ in
  $J^1(\bar\Lambda) = T^*\bar\Lambda \times\R$.  We denote by $\xi_{\bar\Lambda}$ and
  $\xi_M$ the standard contact structures of $J^1(\bar\Lambda)$ and $S^*M$ and by
  $\alpha_{\bar\Lambda}$ and $\alpha_M$ the standard contact forms.

  We know by~\cite{W36} that there exist real analytic structures on $\bar\Lambda$
  and $M$ compatible with their $\Cinf$ structures.  They induce real analytic
  structures on $J^1(\bar\Lambda)$, $S^*M$ and $\alpha_{\bar\Lambda}$, $\alpha_M$ are
  analytic.  We know also that we can find a $C^r$ map
  $\tilde \jmath \colon U \times [0,1] \to S^*M$, for any integer $r$, such that,
  setting $j^s = \tilde \jmath|_{U\times\{s\}}$, we have: $j^0 = j$, $j^1$ is an
  analytic open embedding and the $j^s$ are as close as required to $j$ in the
  compact open $C^r$-topology.

  On $U$ we thus have a family of contact structures $\xi^s = j^{s*}(\xi_M)$, with
  $\xi^0  = j^*(\xi_M) = \xi_{\bar\Lambda}|_U$ and $\xi^1$ analytic.  For each given
  $s$ we consider a linear interpolation between $\xi^0$ and $\xi^s$:
  $$
  \alpha^s_t = (1-t)\alpha_{\bar\Lambda} + t\,j^{s*}(\alpha_M),
\qquad \xi^s_t = \ker(\alpha^s_t).
$$
We remark that the $\alpha^1_t$ are analytic.  Choosing $j^s$ $C^r$-close to $j$, we
can assume that $j^{s*}(\alpha_M)$ is $C^{r-1}$-close to $\alpha_{\bar\Lambda}$.  By
Gray's theorem, for each $s$ we can find an isotopy $\varphi^s_t$, $t\in [0,1]$, such
that $(\varphi^s_t)^*(\xi^s_t) = \xi^0$ near $\bar\Lambda$ as we recall in~(ii) (see
for example Theorem~2.2.2 of~\cite{G08}).

\sui(ii) For a given $s\in [0,1]$, we define the vector field $X^s_t$, $t\in [0,1]$,
on $U$ uniquely determined by
\begin{equation}\label{eq:deformeversmustrat1}
X^s_t \in \xi^s_t \quad\text{ and }\quad
X^s_t \lrcorner \, d\alpha^s_t =-{\dot\alpha}^s_t|_{\xi^s_t}
\end{equation}
(recall that $d\alpha^s_t |_{\xi^s_t}$ is non-degenerate).  We remark that $X^1_t$ is
analytic.  We also remark that $d\alpha^s_t$ is close to $d\alpha_{\bar\Lambda}$ and
${\dot\alpha}^s_t = j^{s*}(\alpha_M)-\alpha_{\bar\Lambda}$ is close to $0$.  Hence
$X^s_t$ is as close to $0$ as required and the flow $\varphi^s$ of
$\{X^s_t\}_{t\in [0,1]}$ is defined on a set $V \times [0,1]$ where $V$ is some
neighborhood of $\bar\Lambda$ in $U$.  Moreover $\varphi^1$ is analytic.

Now we check that $(\varphi^s_t)^*(\xi^s_t) =\xi^0$, where
$\varphi^s_t = \varphi^s|_{V\times\{t\}}$.  By~\eqref{eq:deformeversmustrat1} we have
$X^s_t \lrcorner d\alpha^s_t+{\dot\alpha}^s_t = f^s_t\alpha^s_t$ for some function
$f^s_t$ and
\begin{align*}
  \frac{d}{dt}((\varphi^s_t)^*(\alpha^s_t)\wedge \alpha_{\bar\Lambda})
  &=(\varphi^s_t)^*(L_{X^s_t}(\alpha^s_t)+ {\dot\alpha}^s_t)\wedge\alpha_{\bar\Lambda}\\
  &=(\varphi^s_t)^*(f^s_t\alpha^s_t)\wedge \alpha_{\bar\Lambda}\\
  &= (f^s_t\circ \varphi^s_t) ((\varphi^s_t)^*(\alpha^s_t)\wedge \alpha_{\bar\Lambda}).
\end{align*}  
Since $(\varphi^s_0)^*(\alpha_{\bar\Lambda})\wedge \alpha_{\bar\Lambda}=0$, we obtain
$(\varphi^s_t)^*(\alpha^s_t)\wedge \alpha_{\bar\Lambda}=0$ for all $t\in[0,1]$, which
implies $(\varphi^s_t)^*(\xi^s_t) =\xi^0$.

\sui(iii) By construction the map $\psi^s = j^s \circ \varphi^s_1 \colon V \to S^*M$
is contact.  Moreover $\psi^1$ is analytic.  Hence
$\bar\Lambda^s = \psi^s(\bar\Lambda)$ defines a Legendrian deformation of
$\bar\Lambda$ with $\bar\Lambda^1$ analytic.  The existence of the required
stratification for $\bar\Lambda^1$ is given by Corollary~8.3.22 of~\cite{KS90}.  Now
a Legendrian deformation can be lifted to an ambient contact isotopy, which is the
same as an $\rspos$-homogeneous Hamiltonian isotopy of $\dT^*M$.
\end{proof}

\section{Sato's microlocalization}
\label{sec:Satomicr}

We quickly review the definition of the specialization and microlocalization functors
as introduced in~\cite{KS90}.  We first recall the notion of normal deformation.  Let
$M$ be a manifold and $N$ a closed submanifold of $N$.  The normal deformation of $N$
in $M$ is a manifold $\widetilde{M}_N$ together with three maps
$$
s \cl T_NM \to \widetilde{M}_N,
\qquad
p \cl \widetilde{M}_N \to M,
\qquad
t \cl \widetilde{M}_N \to \R,
$$
such that
\begin{equation*}
\left\{ \,
\begin{minipage}[c]{10cm}
 $s$ is an embedding and $\im(s) = \opb{t}(0)$,

 \smallskip $p(\im(s)) = N$ and $p\circ s$ is the projection $T_NM \to N$,
 
 \smallskip
 $\opb{p}(M\setminus N ) \simeq (M\setminus N ) \times (\R \setminus \{0\})$,

 \smallskip $\opb{t}(\R \setminus \{0\}) \simeq M \times (\R \setminus \{0\})$,

 \smallskip $p|_{\opb{t}(u)} \cl \opb{t}(u) \to M$ is a diffeomorphism, for all
 $u\not= 0$.
\end{minipage}
\right.
\end{equation*}
We   can define $\widetilde{M}_N$ as an open subset of
some blow-up as follows.  The blow-up $B_{N \times \{0\}}(M\times\R)$ is set
theoretically the union of $U = (M\times \R) \setminus (N \times \{0\})$ and
$P(T_NM \times \R)$, the projectivization of the vector bundle
$T_NM \times \R \to N$.  We consider $(M \setminus N) \times \{0\}$ as a subset of
$U$ and $P(T_NM \times \{0\})$ as a subset of $P(T_NM \times \R)$. We set
$$
\widetilde{M}_N = B_{N \times \{0\}}(M\times\R) \setminus
( ((M \setminus N) \times \{0\}) \cup P(T_NM \times \{0\}) ) .
$$
Now $B_{N \times \{0\}}(M\times\R)$ comes with a map to $M\times \R$ and this map
induces the maps $p$, $t$.  The difference
$P(T_NM \times \R) \setminus P(T_NM \times \{0\})$ is identified with $T_NM$ and this
gives $s$.

We set $\Omega = \opb{t}(]0,+\infty[)$.  We let $j\colon \Omega \to \widetilde{M}_N$
be the inclusion and set $p_+ = p \circ j$.  We have the following formula for the
cone of a subset $A \subset M$ along $N$: $C_N(A) = \opb{s}(\ol{\opb{p_+}(A)})$.  The
sheaf counterpart of the cone construction is the following specialization functor.

\begin{definition}{\rm (See~\cite[Def.~4.2.2]{KS90}.)}
\label{def:specialtion}
With the above notations we define the functor $\nu_N \cl \Der(\cor_M) \to
\Der(\cor_{T_NM})$ by $\nu_N(F) = \opb{s} \roim{j} \opb{p_+}(F)$.
\end{definition}

The sheaf $\nu_N(F)$ is {\em conic}, that is, invariant by the multiplicative
action of $\rspos$ on the fibers of $T_NM$. We can deduce a sheaf over $T^*_NM$
using the {\em Fourier-Sato transform} defined as follows.
\begin{definition}{\rm (See~\cite[Def.~3.7.8]{KS90}.)}
\label{def:FourierSato}
Let $q_i$ be the $i^{th}$ projection from $T_NM \times_N T^*_NM$ and let $P
\subset T_NM \times_N T^*_NM$ be the subset $P = \{(\nu,\xi)$; $\langle \nu, \xi
\rangle \leq 0\}$.  For $F \in \Der(\cor_{T_NM})$ we define $F^\wedge \in
\Der(\cor_{T^*_NM})$ by $F^\wedge = \reim{q_2}(\opb{q_1}F \tens \cor_P)$.
\end{definition}

In~\cite{KS90} the Fourier-Sato transform is actually defined for general vector
bundles. It is proved that it gives an equivalence between conic sheaves on a
vector bundle and conic sheaves on its dual.

\begin{definition}{\rm (See~\cite[Def.~4.3.1]{KS90}.)}
\label{def:microlocalization}
The microlocalization functor $\mu_N \cl \Der(\cor_M) \to \Der(\cor_{T^*_NM})$
is defined by $\mu_N(F) = (\nu_N(F))^\wedge$.
\end{definition}

If $V \subset T^*_NM$ is a convex open cone, we have, using
notation~\eqref{eq:def_HiZ},  
$$
H^i(V; \mu_N(F)) \simeq \varinjlim_{U,Z} H^i_{Z\cap U} (U;F),
$$
where $U$ runs over the open subsets of $M$ containing $\pi_M(V)$ and $Z$ over
the closed subsets of $M$ such that $C_N(Z) \subset V^\circ$ (recall that
$V^\circ$ is the polar cone of $V$).

In~\cite{KS90} we also find a generalization of Sato's microlocalization which
will be important when we consider the Kashiwara-Schapira stack.  Let $\Delta_M$
be the diagonal of $M\times M$. Let $q_1,q_2\cl M\times M \to M$ be the
projections.  We identify $T^*_{\Delta_M}(M\times M)$ with $T^*M$ through the
first projection. 
\begin{definition}{\rm (See~\cite[Def.~4.4.1]{KS90}.)}
\label{def:muhom}
For $F,G\in \Der(\cor_M)$ we define $\muhom (F,G) \in \Der(\cor_{T^*M})$ by
\begin{equation}\label{eq:def_muhom}
\muhom (F,G) = \mu_{\Delta_M}(\rhom(\opb{q_2}F,\epb{q_1}G)) .
\end{equation}
\end{definition}

For a submanifold $N$ of $M$ and $i\colon T^*_NM \to T^*M$ the inclusion, we have
$\oim{i}\mu_N(G) \simeq \muhom(\cor_N,G)$, for any $G\in \Der(\cor_M)$. The functor
$\muhom$ is a refinement of the functor $\rhom$ in view of the following properties:
\begin{align}
\label{eq:proj_muhom_oim}
\roim{\pi_M} \muhom (F,G) &\simeq \rhom (F,G)  , \\
\label{eq:proj_muhom}
\reim{\pi_M} \muhom (F,G) &\simeq 
\opb{\delta_M} \rhom(\opb{q_2}F,\opb{q_1}G) ,
\end{align}
where $\delta_M\cl M \to M\times M$ is the diagonal embedding.  For a conic
sheaf $H$ on $T^*M$ (or on any vector bundle) we have a natural isomorphism
$\reim{\pi_M}(H) \simeq \roim{\pi_M}\rsect_M(H)$ (where $M$ is here the zero
section of $T^*M$) and the natural morphism
$\reim{\pi_M}(H) \to \roim{\pi_M}(H)$ coincides with the morphism deduced from
$\rsect_M(H) \to H$.  The excision distinguished triangle associated with the
inclusion $M \subset T^*M$ then gives a version of {\em Sato's distinguished
  triangle}:
\begin{equation}\label{eq:SatoDTmuhom0}
  \begin{split}
  \opb{\delta_M} \rhom(\opb{q_2}F & , \opb{q_1}G) \to \rhom (F,G) \\
&\to \roim{\dot\pi_M{}} (\muhom (F,G) |_{\dT^*M}) \to[+1].
\end{split}
\end{equation}
If $F$ is cohomologically constructible, then the first term
of~\eqref{eq:SatoDTmuhom0} is isomorphic to $\DD'(F) \ltens G$ by
Theorem~\ref{thm:SSrhom} and we obtain
\begin{equation}\label{eq:SatoDTmuhom1}
\DD'(F) \ltens G  \to \rhom (F,G) 
\to \roim{\dot\pi_M{}} (\muhom (F,G) |_{\dT^*M}) \to[+1].
\end{equation}

\begin{proposition}{\rm(Cor.~6.4.3 of~\cite{KS90}.)}
\label{prop:SSmuhom}
Let $F,G\in \Der(\cor_M)$. Then
\begin{gather}
\label{eq:suppmuhom}
\supp(\muhom (F,G)) \subset \SSi(F)\cap \SSi(G), \\
\label{eq:SSmuhom}
\SSi(\muhom(F,G) ) \subset (H^{-1}(C(\SSi(G) , \SSi(F) )))^a ,
\end{gather}
where $H$ is the Hamiltonian isomorphism.
\end{proposition}

When $F=G$, the inclusion~\eqref{eq:suppmuhom} is an equality. More precisely,
by~\eqref{eq:proj_muhom_oim}, $\id_F \in \Hom(F,F)$ gives a global section of
$\muhom$, say
\begin{equation}
\label{eq:def_idmuF}
\id^\mu_F \in H^0(T^*M; \muhom(F,F)) 
\end{equation}
and~\cite[Cor.~6.1.3]{KS90} says that  
\begin{equation}
\label{eq:suppmuhom=SS}
\supp(\id^\mu_F) = \supp \muhom (F,F) = \SSi(F).   
\end{equation}
An important consequence of~\eqref{eq:SSmuhom} and~\eqref{eq:suppmuhom=SS} is
the following involutivity theorem.

\begin{theorem}[Thm.~6.5.4 of~\cite{KS90}]
\label{thm:invol}
Let $M$ be a manifold and $F\in \Der(\cor_M)$.  Then $S = \SSi(F)$ is a
coisotropic subset of $T^*M$ in the sense that $C_p(S)$ contains the symplectic
orthogonal of $C_p(S,S)$, for all $p\in S$.
\end{theorem}

When $\Lambda \subset T^*M$ is a Lagrangian submanifold we have
$H^{-1}(T\Lambda) = T^*_\Lambda T^*M$.  Hence~\eqref{eq:suppmuhom},
\eqref{eq:SSmuhom} and Example~\ref{ex:loccst_submfd} give the following result.

\begin{corollary}
\label{cor:muhom_loccst}
Let $\Lambda$ be a conic Lagrangian submanifold of $\dT^*M$.  Let
$F,G\in \Der(\cor_M)$. We assume that there exists a neighborhood $\Omega$ of
$\Lambda$ such that $\SSi(F)\cap \Omega \subset \Lambda$ and
$\SSi(G)\cap \Omega \subset \Lambda$.  Then $\muhom(F,G)|_\Omega$ is supported
on $\Lambda$ and is locally constant on $\Lambda$.
\end{corollary}

By the following result we can see $\muhom$ as a microlocal version of
$\rhom$. Let $p\in T^*M$ be a given point. By the triangular inequality the full
subcategory $N_p$ of $\Der(\cor_M)$ formed by the $F$ such that $p\not\in
\SSi(F)$ is triangulated and we can set $\Der(\cor_M;p) = \Der(\cor_M)/N_p$
(see the more general Definition~6.1.1 of~\cite{KS90}).  The functor $\mu
hom(\cdot,\cdot)_p$ induces a bifunctor on $\Der(\cor_M;p)$ and we have

\begin{theorem}[Theorem~6.1.2 of~\cite{KS90}]
\label{thm:germemuhom}
  For all $F,G \in \Der(\cor_M)$, the morphism $\Hom_{\Der(\cor_M;p)}(F,G) \to
  H^0(\muhom(F,G))_p$ is an isomorphism.
\end{theorem}

We also give the following useful consequence of Theorem~\ref{thm:invol}.
\begin{corollary}
\label{cor:micsuppLagrlisse}
Let $\Lambda$ be a conic connected Lagrangian submanifold of $\dT^*M$.  Let $F
\in \Der(\cor_M)$ be such that $\dot\SSi(F) \not= \emptyset$ and $\dot\SSi(F)
\subset \Lambda$.  Then $\dot\SSi(F) = \Lambda$.
\end{corollary}
\begin{proof}
  Arguing by contradiction we assume that $U = \Lambda \setminus \dot\SSi(F)$
  is non empty.  The set $U$ is open in $\Lambda$ with a non empty boundary
  $\partial U$. We choose a chart $V$ in $\Lambda$ around a point of $\partial
  U$ and we choose a point $p_0 \in V \cap U$.  Let $B$ be the open ball in $V$
  with center $p_0$ and maximal radius such that $B \cap \dot\SSi(F) =
  \emptyset$. Then $\partial B \cap \dot\SSi(F)$ is non empty and we let $p$ be
  any of its points.

  Since $\dot\SSi(F) \subset \Lambda$ and $\Lambda$ is smooth, we have
  $C_p(\SSi(F), \SSi(F)) \subset C_p(\Lambda, \Lambda) = T_p\Lambda$.  Since
  $\Lambda$ is Lagrangian, it follows that the symplectic orthogonal of
  $C_p(\SSi(F), \SSi(F))$ contains $T_p\Lambda$.  On the other hand
  $C_p(\SSi(F))$ is contained in $C_p(V \setminus B)$ which is a half space of
  $T_p\Lambda$. Hence $C_p(\SSi(F))$ does not contain $T_p\Lambda$ and this
  contradicts Theorem~\ref{thm:invol}.
\end{proof}

\section{Simple sheaves}
\label{sec:simplesheaves}

Let $\Lambda$ be a closed conic Lagrangian submanifold of $\dT^*M$.  We recall
the definition of simple and pure sheaves along $\Lambda$ and give some of their
properties.  We first recall some notations from~\cite{KS90}.  For a function
$\varphi\cl M\to \R$ of class $C^\infty$ we define
\begin{equation}
\label{eq:def_Lambdaphi}
  \Lambda_\varphi = \{(x; d\varphi(x)); \; x\in M\}  .
\end{equation}
We notice that $\Lambda_\varphi$ is a closed Lagrangian submanifold of $T^*M$.
For a given point $p=(x;\xi) \in \Lambda \cap \Lambda_\varphi$ we have the
following Lagrangian subspaces of $T_p(T^*M)$
\begin{equation}\label{eq:def_lambdas_p}
\lambda_0(p) = T_p(T_x^*M), \qquad 
\lambda_\Lambda(p) = T_p\Lambda, \qquad
\lambda_\varphi(p) = T_p\Lambda_\varphi.
\end{equation}
We recall the definition of the inertia index (see for example \S A.3
in~\cite{KS90}).  Let $(E,\sigma)$ be a symplectic vector space and let
$\lambda_1, \lambda_2, \lambda_3$ be three Lagrangian subspaces of $E$. We
define a quadratic form $q$ on $\lambda_1\oplus \lambda_2\oplus \lambda_3$ by
$q(x_1,x_2,x_3) = \sigma(x_1,x_2) + \sigma(x_2,x_3) + \sigma(x_3,x_1)$ and
\begin{equation}\label{eq:def_inertia_ind}
  \tau_E(\lambda_1, \lambda_2, \lambda_3) = \sgn(q)
\end{equation}
where $\sgn(q)$ is the signature of $q$, that is, $p_+-p_-$, where
$p_\pm$ is the number of $\pm 1$ in a diagonal form of $q$.  We set
\begin{equation*}
\tau_{\varphi} = \tau_{p,\varphi} 
= \tau_{T_pT^*M}(\lambda_0(p), \lambda_\Lambda(p), \lambda_\varphi(p)).
\end{equation*}

\begin{proposition}[Proposition~7.5.3 of~\cite{KS90}]
\label{prop:inv_microgerm}
Let $\varphi_0, \varphi_1\cl M\to \R$ be functions of class $C^\infty$, let
$p=(x;\xi) \in \Lambda$ and let $F\in \Der_{(\Lambda)}(\cor_M)$ (see
Notation~\ref{not:cat_micsupp_fixe}).  We assume that $\Lambda$ and
$\Lambda_{\varphi_i}$ intersect transversely at $p$, for $i=0,1$.  Then there exists
an isomorphism
$$
(\rsect_{\{ \varphi_1 \geq \varphi_1(x) \}}(F))_x
\simeq  (\rsect_{\{ \varphi_0 \geq \varphi_0(x) \}}(F))_x 
[\pdemi (\tau_{\varphi_0} - \tau_{\varphi_1})] .
$$
\end{proposition}

\begin{definition}[Definition~7.5.4 of~\cite{KS90}]
\label{def:simple_pure}
In the situation of Proposition~\ref{prop:inv_microgerm} we say that $F$ is pure
at $p$ if $(\rsect_{\{ \varphi_0 \geq \varphi_0(x)\}}(F))_x$ is concentrated in
a single degree and free, that is,
$(\rsect_{\{ \varphi_0 \geq \varphi_0(x)\}}(F))_x \simeq L[d]$, for some free
module $L\in \Mod(\cor)$ and $d\in \Z$. The half integer
\begin{equation}
  \label{eq:def_shift_simple}
  d +\pdemi d_M +\pdemi \tau_{\varphi_0}
\end{equation}
is called the shift of $F$. If $L\simeq \cor$, we say that $F$ is simple at $p$.

If $F$ is pure (resp. simple) at all points of $\Lambda$ we say that it is pure
(resp. simple) along $\Lambda$.

We denote by $\Dersf_{[\Lambda]}(\cor_M)$ the full subcategory of
$\Der_{[\Lambda]}(\cor_M)$ formed by the $F$ such that $F$ is simple along
$\dot\SSi(F)$ and the stalks of $F$ at the points of $M \setminus \dot\pi_M(\Lambda)$
are finitely generated.  
\end{definition}

\begin{proposition}[Cor.~7.5.7 in~\cite{KS90}]
\label{prop:pure_conn_comp}
We assume that $\Lambda$ is connected and $F\in \Der_{(\Lambda)}(\cor_M)$ is pure
at some $p\in \Lambda$. Then $F$ is pure along $\Lambda$. Moreover the
$L\in \Mod(\cor)$ in the above definition is the same at every point.
\end{proposition}

\begin{remark}
  In Proposition~\ref{prop:SSmicrogerm} below we will give a more precise result
  than Propositions~\ref{prop:inv_microgerm} and~\ref{prop:pure_conn_comp}.
\end{remark}

If $\cor$ is a field, we know that $F$ is pure along $\Lambda$ if and only if
$\muhom(F,F)|_{\dT^*M}$ is concentrated in degree $0$ and $F$ is simple along
$\Lambda$ if and only if the natural morphism
$\cor_\Lambda \to \mu hom(F,F)|_{\dT^*M}$ induced by~\eqref{eq:proj_muhom_oim}
and the section $\id_F$ of $\rhom(F,F)$ is an isomorphism:
\begin{equation}
\label{eq:carc_Fsimple}
\cor_\Lambda  \isoto \muhom(F,F) .
\end{equation}
For coefficients in a field the property~\eqref{eq:carc_Fsimple} could be a
definition of a simple sheaf.

\begin{example}
\label{ex:shift}
(1)   Let $N$ be a submanifold of $M$ of codimension $d$.
Then $\cor_N$ is simple with shift $\frac12 d$. To see that $\cor_N$ is simple we can
assume $M= \R^n$, $N = \R^{n-d}\times \{0\}$ and $p = (x_0;\xi_0)$ with $x_0 = 0$,
$\xi_0 = (0,\dots,0,1)$.  We choose the function
$\varphi_0(x) = x_n +\sum_{i=1}^{n-1} x_i^2$. Then
$\rsect_{\{ \varphi_0 \geq 0\}}(\cor_N) \simeq \cor_N$ and
$(\rsect_{\{ \varphi_0 \geq 0\}}(\cor_N))_{x_0} \simeq \cor$.  The computation of the
shift is not difficult but lengthy. We refer to Example~7.5.5 of~\cite{KS90}.

\sui(2) In Example~\ref{ex:microsupport}~(iii) the sheaves $\cor_U$ and
$\cor_{\ol U}$ are simple; $\cor_U$ has shift $-1/2$ and $\cor_{\ol U}$ has shift
$1/2$. In fact both compare to $\cor_{\partial U}$, which has shift $1/2$ by (1),
through the distinguished triangle
$\cor_U \to \cor_{\ol U} \to \cor_{\partial U} \to[+1]$.  Choosing $\varphi_0$ with
$d\varphi_0 \in \dot\SSi(\cor_U)$ we deduce from the triangle
$\rsect_{\{ \varphi_0 \geq \varphi_0(0)\}}(\cor_U) \simeq \rsect_{\{ \varphi_0 \geq
  \varphi_0(0)\}}(\cor_{\partial U})[-1]$.  The argument is similar for
$\cor_{\ol U}$. We also refer to Example~7.5.5 of~\cite{KS90}.   

\sui(3)
For $i \in \N$ we let $\Lambda_i \subset \Lambda$ be the set of points such that
the rank of $d\pi_M|_\Lambda$ is $(\dim M -1 -i)$.  For a generic closed conic
Lagrangian connected submanifold $\Lambda$ in $\dT^*M$, $\Lambda_0$ is an open
dense subset of $\Lambda$ and, for a given simple sheaf $F \in
\Der_{(\Lambda)}(\cor_M)$, the shift of $F$ at $p$ is locally constant on
$\Lambda_0$ and changes by $1$ when $p$ crosses $\Lambda_1$.

  These properties correspond exactly to the
definition of a {\em Maslov potential} for $\Lambda$.  We recall that the {\em Maslov
  class} of $\Lambda$ (an element of $H^1(\Lambda;\Z_\Lambda)$ is the obstruction to
the existence of a Maslov potential for $\Lambda$.  Hence, if there exists a simple
sheaf $F \in \Der_{(\Lambda)}(\cor_M)$, then the Maslov class of $\Lambda$ vanishes and a
Maslov potential is given by the shift of $F$.
\end{example}

We can also compute the stalks of $\muhom$ for sheaves in
$\Der_{(\Lambda)}(\cor_M)$.  Let $p=(x;\xi) \in \Lambda$ and
$\varphi_0\cl M\to \R$ be as in Proposition~\ref{prop:inv_microgerm}.  For
$F,G\in \Der_{(\Lambda)}(\cor_M)$, we have
\begin{equation}\label{eq:stalk_muhom}
\muhom(F,G)_p \simeq
\RHom( (\rsect_{Z_0}(F))_x , (\rsect_{Z_0}(G))_x ) ,
\end{equation}
where $Z_0 = \{ \varphi_0 \geq \varphi_0(x) \}$.

Now we prove that a simple sheaf belongs to $\Dersf_{[\Lambda]}(\cor_M)$ as soon as
its stalk at some given point is finitely generated.  We set
\begin{align*}
Z_\Lambda = \{ &x\in \dot\pi_M(\Lambda);\;
\text{there exist a neighborhood $W$ of $x$ and a} \\
&\text{smooth hypersurface $S \subset W$ such that
$\Lambda \cap T^*W \subset T^*_SW$} \} .
\end{align*}
The transversality theorem implies the following result.

\begin{lemma}\label{lem:def_chemin}
Let $x,y\in M \setminus \dot\pi_M(\Lambda)$.
Let $I$ be an open interval containing $0$ and $1$.
Then there exists a $C^\infty$ embedding $c \cl I \to M$ such that $c(0) = x$,
$c(1) = y$ and $c([0,1])$ only meets $\dot\pi_M(\Lambda)$ at points of
$Z_\Lambda$, with a transverse intersection.
\end{lemma}

\begin{lemma}\label{lem:stalk_simple_sh}
    Let
  $F\in \Der_{[\Lambda]}(\cor_M)$ be a simple sheaf along $\Lambda$.  We set
  $U = M\setminus \dot\pi_M(\Lambda)$.  We assume that $M$ is connected and that
  there exists $x_0\in U$ such that the $\cor$-module $\bigoplus_{i\in \Z}H^iF_{x_0}$
  is finitely generated.  Then, for any $x\in U$, the $\cor$-module
  $\bigoplus_{i\in \Z}H^iF_x$ is finitely generated.  In other words $F$ belongs to
  $\Dersf_{[\Lambda]}(\cor_M)$.
\end{lemma}
\begin{proof}
Let $x\in U$ and let $I$ be an open interval containing $0$ and $1$.  
By Lemma~\ref{lem:def_chemin} we can choose a $C^\infty$ path
$\gamma\cl I \to M$ such that $\gamma(0) = x_0$, $\gamma(1) = x$ and
$\gamma([0,1])$ meets $\dot\pi_M(\Lambda)$ at finitely many points, all
contained in $Z_\Lambda$ and with a transverse intersection.
We denote these points by $\gamma(t_i)$, where $0< t_1< \cdots < t_k <1$.

Since $F$ is locally constant on $U$, the stalk $F_{\gamma(t)}$ is constant for $t\in
\mo]t_i,t_{i+1}[$.   Near a point $x_i = \gamma(t_i)$ we
have a hypersurface $S$ of $M$ such that $\Lambda \subset T^*_SM$.  Let us first assume
that $\Lambda$ is one half of $T^*_SM$.  Using Example~\ref{ex:SS=conormal_hypersurface}
and the fact that $F$ is simple, there exist $d\in \Z$, $E' \in \Der(\cor)$ and a
distinguished triangle, in some neighborhood of $x_i$,
$$
E'_M \to \cor_Z[d] \to F \to[+1] ,
$$
where   $Z$ is one of the closed half-spaces bounded by $S$.  It
follows that the stalks $F_{\gamma(t_i-\varepsilon)}$ and $F_{\gamma(t_i+\varepsilon)}$
differ by $\cor[d]$ or $\cor[d+1]$ for $\varepsilon>0$ small enough.  Now we assume that
$\Lambda = T^*_SM$.  The argument in Example~\ref{ex:SS=conormal_hypersurface} works in
the same way to reduce the situation to Example~\ref{ex:microsupport}-(v), which also
works in the same way and gives a distinguished triangle $F' \to \cor_Z[d] \to F
\to[+1]$. Now $F'$ is no longer constant, but $\dot\SSi(F') = \dot\SSi(F) \setminus
\dot\SSi(\cor_Z[d])$ is half of $T^*_SM$. So we are back to the previous case.
\end{proof}

\section{Composition of sheaves}

We will use several times a usual operation associated with sheaves called
``composition'' or ``convolution''. We refer to~\cite[\S 3.6]{KS90} for more
details, or~\cite[\S 1.6, \S 1.10]{GKS12}.

Let $M_i$, $i=1,2,3$, be three manifolds.  We denote by $q_{ij}$ the projection
from $M_1 \times M_2 \times M_3$ to $M_i \times M_j$.  For
$K_1 \in \Der(\cor_{M_1 \times M_2})$ and $K_2 \in \Der(\cor_{M_2 \times M_3})$
we denote by $K_1 \circ K_2 \in \Der(\cor_{M_1 \times M_3})$ the {\em
  composition} of $K_1$ and $K_2$:
\begin{equation}
  \label{eq:def_compo_faisceaux}
K_1 \circ K_2 =  \reim{q_{13}}( \opb{q_{12}}K_1 \ltens \opb{q_{23}} K_2) .
\end{equation}
The Fourier-Sato transform of Definition~\ref{def:FourierSato} is an example of
composition of sheaves.

Using the base change formula we see that the composition product is associative
in the sense that, for another manifold $M_4$ and
$K_3 \in \Der(\cor_{M_3 \times M_4})$, we have a natural isomorphism
$(K_1 \circ K_2) \circ K_3 \simeq K_1 \circ (K_2 \circ K_3)$.  If $M_1=M_2$, the
constant sheaf on the diagonal $K_1 = \cor_{\Delta_{M_1}}$ is a left unit for
this product: we have $\cor_{\Delta_{M_1}} \circ K_2 \simeq K_2$ for any
$K_2 \in \Der(\cor_{M_2 \times M_3})$. Similarly, if $M_2=M_3$, the sheaf
$\cor_{\Delta_{M_2}}$ is a right unit.

The base change formula gives a useful expression for the stalks of the
composition. For $(x,z) \in M_1 \times M_3$ we have
\begin{equation}
  \label{eq:germ_compo}
  (K_1 \circ K_2)_{(x,z)}
  \simeq \rsect_c(M_2; (K_1|_{\{x\} \times M_2}) \ltens (K_2|_{ M_2 \times \{z\}})) .
\end{equation}
For $K_1 \in \Der(\cor_{M_1 \times M_2})$ we have a natural candidate for an
inverse, denoted $K_1^{-1}$ and defined as follows.  Let
$q_2 \cl M_1 \times M_2 \to M_2$ be the projection and
$v \cl M_1 \times M_2 \isoto M_2 \times M_1$ the swap isomorphism.  We recall
that $\epb{q_2}(\cor_{M_2}) \simeq \omega_{M_1}\etens \cor_{M_2}$.  We define
\begin{equation}
  \label{eq:def_inv_compo}
  K_1^{-1} = \opb{v} \rhom(K_1 , \epb{q_2}(\cor_{M_2}))
  \in \Der(\cor_{M_2 \times M_1}) .
\end{equation}

Let $\delta_2 \cl M_2 \to M_2 \times M_2$ and
$\delta'_2 \cl M_1 \times M_2 \to M_2 \times M_1 \times M_2$ be the diagonal
embeddings. The base change formula
$\opb{\delta_2} \circ \reim{q_{23}} \simeq \reim{q_2} \circ \opb{{\delta'_2}}$
implies
$$
\opb{\delta_2} (K_1^{-1} \circ K_1) \simeq
\reim{q_2} (K_1 \ltens \rhom(K_1 , \epb{q_2}(\cor_{M_2})))  .
$$
Using the contraction $K_1 \ltens \rhom(K_1 , L) \to L$ and the adjunction
morphisms for $(\reim{q_2}, \epb{q_2})$ and $(\opb{\delta_2}, \roim{\delta_2})$
we deduce the first morphism in~\eqref{eq:Kinv_compo_K} below; the second
morphism is obtained in the same way.
\begin{equation}
  \label{eq:Kinv_compo_K}
  K_1^{-1} \circ K_1 \to \cor_{\Delta_{M_2}}, \qquad
 K_1 \circ  K_1^{-1} \to \cor_{\Delta_{M_1}} .
\end{equation}

Using the bounds given by Proposition~\ref{prop:oim} and
Theorem~\ref{thm:iminv}, \ref{thm:SSrhom} for the behaviour of the microsupport
under sheaves operations we obtain the following result.  We denote by $p_{ij}$
the projections from $T^*(M_1 \times M_2 \times M_3)$ similar to the
$q_{ij}$. We also define $a_2$ on $T^*(M_1 \times M_2)$ by
$a_2(x,y;\xi,\eta) = (x,y;\xi,-\eta)$.
\begin{lemma}
  \label{lem:SScompo}
  Let $K_1 \in \Der(\cor_{M_1 \times M_2})$ and
  $K_2 \in \Der(\cor_{M_2 \times M_3})$ be given. We assume that $q_{13}$ is
  proper on $\opb{q_{12}}\supp(K_1)\cap\opb{q_{23}}\supp(K_2)$ and
  $\opb{p_{12}}\opb{a_2}\SSi(K_1) \cap \opb{p_{23}}\SSi(K_2) \cap (T^*_{M_1}M_1\times
  T^*M_2\times T^*_{M_3}M_3)$ is contained in the zero-section of
  $T^*(M_1 \times M_2 \times M_3)$.  Then
  \begin{equation}
  \label{eq:def_compo_ensembles}
  \SSi(K_1 \circ K_2) \subset \SSi(K_1) \circ^a \SSi(K_2) ,
  \end{equation}
  where the operation $\circ^a$ for $A_1 \subset T^*(M_1 \times M_2)$ and
  $A_2 \subset T^*(M_2 \times M_3)$ is defined by
  $A_1 \circ^a A_2 = p_{13}(\opb{p_{12}}\opb{a_2}(A_1) \cap \opb{p_{23}}(A_2)
  )$.
\end{lemma}

We will also use a relative version of the composition.  For a manifold $I$ we
denote by $q_{ijI}$ the projections from $M_1 \times M_2 \times M_3 \times I$ to
$M_i \times M_j \times I$ similar to the $q_{ij}$.  For
$K_1 \in \Der(\cor_{M_1 \times M_2 \times I})$ and
$K_2 \in\Der(\cor_{M_2 \times M_3 \times I})$ we set
\begin{equation}
  \label{eq:def_compo_faisceaux_rel}
K_1\circ|_I K_2 =
\reim{q_{13I}}(\opb{q_{12I}}K_1\ltens\opb{q_{23I}}K_2) .
\end{equation}
The definition is chosen so that  
$(K_1\circ|_I K_2)|_{M_1 \times M_3 \times \{t\}} \simeq K_{1,t} \circ K_{2,t}$
for all $t\in I$, where $K_{1,t} = K_1|_{M_1 \times M_2 \times \{t\}}$,
$K_{2,t} = K_2|_{M_2 \times M_3 \times \{t\}}$.

The previous results generalize to the relative setting (see~\cite{GKS12}).  In
particular we can define
$K_1^{-1} = \opb{v} \rhom(K_1 , \epb{q_2}(\cor_{M_2 \times I}))$ (an object of
$\Der(\cor_{M_2 \times M_1 \times I})$) and we have natural morphisms
\begin{equation}
  \label{eq:Kinv_compo_K-rel}
  K_1^{-1} \circ|_I K_1 \to \cor_{\Delta_{M_2} \times I}, \qquad
 K_1 \circ|_I  K_1^{-1} \to \cor_{\Delta_{M_1} \times I} .
\end{equation}
We also have $K_1^{-1}|_{M_1 \times M_2 \times \{t\}} \simeq K_{1,t}^{-1}$.

\begin{remark}\label{rem:comp=foncteur}
  For $K_1 \in \Der(\cor_{M_1 \times M_2})$ the composition
  $\Phi_{K_1} \colon F \mapsto K \circ F$ is a functor
  $\Der(\cor_{M_2}) \to \Der(\cor_{M_1})$ and we have
  $\Phi_{K_1 \circ K_2} \simeq \Phi_{K_1} \circ \Phi_{K_2}$.  When
  $M_1 = M_2 = M$, $\Phi_{\cor_{\Delta_M}}$ is the identity functor.  When
  $M_1= M_2 =M_3 =M$, if $K_1 \circ K_2 \simeq \cor_{\Delta_M}$, then
  $K_2 \circ K_1 \simeq \cor_{\Delta_M}$ and $\Phi_{K_1}$, $\Phi_{K_2}$ are
  mutually inverse equivalences of categories.
\end{remark}

\part{Sheaves associated with Hamiltonian isotopies}
\label{chap:HamIsot}

In this part we recall the main result of~\cite{GKS12}, which, following ideas of
Tamarkin in~\cite{T08}, gives a sheaf version of the Chekanov-Sikorav theorem about
generating functions (see~\cite{C96} and~\cite{S87}).  Let
$\Psi\cl J^1(N) \times I \to J^1(N)$ be a contact isotopy of the $1$-jet bundle of
some manifold $N$.  The Chekanov-Sikorav theorem says that, if a Legendrian
submanifold $L$ of $J^1(N)$ has a generating function, so does $\Psi_s(L)$, for any
$s\in I$.  The sheaf version is more functorial. We can associate a sheaf $K_\Psi$ on
$N^2 \times I$ with $\Psi$.  This sheaf acts by composition on $\Der(\cor_N)$ and
induces equivalences of categories, $F \mapsto (K_\Psi|_{N^2 \times \{s\}}) \circ F$.
We state this result with homogeneous Hamiltonian isotopies of $\dT^*N$ (which is the
same thing as contact isotopies of the sphere bundle of $T^*N$). We recall how it
implies the non homogeneous case and give some complementary remarks.

\section{Homogeneous case}

Let $N$ be a manifold and $I$ an open ball of $\R^d$ containing $0$ (in general
$I$ will be an open interval of $\R$ containing $0$).  We consider a homogeneous
Hamiltonian isotopy $\Psi\cl \dT^*N \times I \to \dT^*N$ of class $\Cinf$. For
$s\in I$, $p\in \dT^*N$ we set $\Psi_s(p) = \Psi(p,s)$.  Hence
$\Psi_0 = \id_{\dT^*N}$ and, for each $s\in I$, $\Psi_s$ is a symplectic
diffeomorphism such that $\Psi_s(x;\lambda\xi) = \lambda\cdot \Psi_s(x;\xi)$,
for all $(x;\xi) \in \dT^*N$ and $\lambda>0$.  We let
$\Lambda_{\Psi_s} \subset \dT^*N^2$ be the twisted graph of $\Psi_s$, that is,
\begin{equation}
\label{eq:def_twistedgraph}
\Lambda_{\Psi_s} = \{(\Psi(x,\xi,s), (x;-\xi) ); \;(x;\xi) \in \dT^*N \} .
\end{equation}
We can see that there exists a unique conic Lagrangian submanifold
$\Lambda_\Psi \subset \dT^*(N^2 \times I)$, described
in~\eqref{eq:def_twistedgraph_global} below, such that
$\Lambda_{\Psi_s} = i_s^\sharp(\Lambda_\Psi)$, for all $s\in I$, where $i_s$ is
the embedding $N^2\times \{s\} \to N^2 \times I$.  Our homogeneous Hamiltonian
isotopy $\Psi$ is the flow of some $h \cl \dT^*N \times I \to \R$ and there is a
unique such $h$ which is homogeneous of degree $1$ in the variable $\xi$.  Then
we have
\begin{equation}
\label{eq:def_twistedgraph_global}
\begin{split}
\Lambda_{\Psi} = \{(\Psi(x,\xi,s), (x;-\xi),
 (s; -h(\Psi(x&,\xi,s),s))); \\
 &(x;\xi) \in \dT^*N, \; s\in I \} .
\end{split}
\end{equation}
We also remark that $\Lambda_{\Psi}$ is non-characteristic for the inclusion
$i_{s}$, for any $s\in I$.  We recall that $\Derlb(\cor_M)$ is the subcategory
of $\Der(\cor_M)$ of locally bounded complexes.

\begin{theorem}[Theorem.~3.7 and Remark~3.9 of~\cite{GKS12} --
  see also Proposition~ 3.16 of~\cite{T08}]
\label{thm:GKS}
There exists a unique $K_\Psi \in\Derlb(\cor_{N^2\times I})$ such that
$\dot\SSi(K_\Psi) \subset \Lambda_\Psi$ and
$K_\Psi|_{N^2\times \{0\}} \simeq \cor_{\Delta_N}$.  Moreover we have
$\dot\SSi(K_\Psi) = \Lambda_\Psi$, $K_\Psi$ is simple along $\Lambda_\Psi$, both
projections $\supp(K_\Psi) \to N\times I$ are proper and the
morphisms~\eqref{eq:Kinv_compo_K-rel} are isomorphisms:
$$
  K_\Psi^{-1} \circ|_I K_\Psi \isoto \cor_{\Delta_N \times I}, \qquad
 K_\Psi \circ|_I  K_\Psi^{-1} \isoto \cor_{\Delta_N \times I} .
$$
\end{theorem}

\begin{remarks}\label{rem:thm_isot}
  (1) The equivalence between $\dot\SSi(K_\Psi) \subset \Lambda_\Psi$ and
  $\dot\SSi(K_\Psi) = \Lambda_\Psi$ follows from
  Corollary~\ref{cor:micsuppLagrlisse}.

\sui(2-a)
The fact that $K_\Psi$ is simple along $\Lambda_\Psi$ is not explicitly stated
in~\cite{GKS12} (although it is used in the proof of the Arnol'd conjecture
about the intersection of the zero section of a cotangent bundle which its
image under a Hamiltonian isotopy).  However it is easily deduced from the
construction of $K_\Psi$ which is obtained as a composition of constant sheaves
over open subsets with smooth boundaries. Such sheaves are simple and a
composition of simple sheaves is simple by~\cite[Thm.~7.5.11]{KS90}.

\sui(2-b) We can also check the simplicity without going back to the
construction of $K_\Psi$ and using only the properties
$\dot\SSi(K_\Psi) = \Lambda_\Psi$ and
$K_\Psi|_{N^2\times \{0\}}\simeq \cor_{\Delta_N}$.  Indeed, by
Proposition~\ref{prop:pure_conn_comp} it is enough to check the simplicity at
one point of $\Lambda_\Psi$.  The map
$i_{s,\pi} \cl T^*N^2 \times T^*_s I \to T^*(N^2 \times I)$ is transverse to
$\Lambda_\Psi$, for any $s\in I$, and the projection
$i_{s,d} \cl \Lambda_\Psi \cap (T^*N^2 \times T^*_s I ) \to \Lambda_{\Psi_s}$ is
a bijection.  By~\cite[Cor.~7.5.13]{KS90} it follows that the type of
$K_\Psi|_{N^2\times \{s\}}$ at a point $p \in \Lambda_{\Psi_s}$ is the same as
the type of $K_\Psi$ at the point $p' \in \Lambda_{\Psi}$ such that
$i_{s,d}(p') = p$.  For $s=0$ we deduce from
$K_\Psi|_{N^2\times \{0\}}\simeq \cor_{\Delta_N}$ that $K_\Psi$ is simple at any
point of $\Lambda_\Psi \cap (T^*N^2 \times \{0\})$.

\sui(3) We can define the isotopy $\Psi'\cl \dT^*N \times I \to \dT^*N$ by
$\Psi'_s = \Psi_s^{-1}$ for all $s\in I$.  Using
Lemma~\ref{lem:SScompo} we see that
$\dot\SSi(K_\Psi \circ|_I K_{\Psi'}) \subset T^*_{\Delta_N \times I}(N^2\times
I)$.  Since
$(K_\Psi \circ|_I K_{\Psi'})|_{N^2\times \{0\}} \simeq \cor_{\Delta_N}$,
Proposition~\ref{prop:iminvproj} gives
$K_\Psi \circ|_I K_{\Psi'} \simeq \cor_{\Delta_N \times I}$.  The uniqueness of the
inverse then implies that $K_\Psi^{-1} \simeq K_{\Psi'}$.

\sui(4) By Remark~\ref{rem:comp=foncteur} the composition with
$K_{\Psi,s} = K_\Psi|_{N \times \{s\}}$ gives an equivalence of categories
$\Der(\cor_N) \to \Der(\cor_N)$, $F \mapsto K_{\Psi,s} \circ F$.  Its inverse is
given by $G \mapsto (K_{\Psi,s})^{-1} \circ G$; with the notations of~(3) we
have $(K_{\Psi,s})^{-1} \simeq K_{\Psi',s}$.
\end{remarks}

\begin{remark}\label{rem:thm_isot2}
  The link with the Chekanov-Sikorav theorem is as follows.  Let $M$ be a
  manifold and let $\Lambda \subset J^1(M)$ be a Legendrian submanifold which
  admits a generating function $f\colon M\times \R^d \to \R$ quadratic at
  infinity.  We define the epigraph of $f$,
  $\Gamma_f^+ \subset M\times\R^d\times\R$, by $\Gamma_f^+ = \{(x,v,t)$;
  $t\geq f(x,v)\}$. We let $q\colon M\times\R^d\times\R \to M\times\R$ be the
  projection and set $F_f = \reim{q}(\cor_{\Gamma_f^+})$.  Then we can check
  that $\dot\SSi(F_f) = \Lambda$, where we identify $\Lambda$ with a conic
  Lagrangian submanifold of $\dT^*(M\times \R)$ (this follows directly from
  Proposition~\ref{prop:oim} if $q$ is proper on $\Gamma_f^+$; in general we
  have to check that $q|_{\Gamma_f^+}$ has a good behaviour at infinity and we
  need some hypotheses on $f$ -- quadratic at infinity for example).
  
  Now we assume to be given a Legendrian isotopy $\Psi$ of $J^1(M)$, that we
  identify with a homogeneous Hamiltonian isotopy of $\dT^*(M\times \R)$.  The
  Chekanov-Sikorav theorem says that, up to adding extra variables, we can
  modify $f$ into $f_s$, such that $f_s$ is a generating function of
  $\Psi_s(\Lambda)$.  Then we have $F_{f_s} \simeq K_{\Psi,s} \circ F_f$.
\end{remark}

\begin{example}\label{ex:faisceau_flotgeod}
  The easiest illustration of Theorem~\ref{thm:GKS} is the sheaf associated with
  the (normalized) geodesic flow of $\dT^*\R^n$.  It is Example~3.10
  of~\cite{GKS12} and we recall it briefly.  We set $I = \R$ and define
  $\Psi\colon \dT^*\R^n \times I \to \dT^*\R^n$, by
  $\Psi_s(x;\xi) = (x + s \frac{\xi}{||\xi||}; \xi)$.  It is the Hamiltonian
  flow of $h(x;\xi) = ||\xi||$  
    (it is $-||\xi||$ in~\cite{GKS12}, so our example
  is the same up to inverting time). For $s>0$ we define
  $U_s = \{(x,y) \in \R^{2n}$; $||x-y|| < s\}$.  Using
  Example~\ref{ex:microsupport}~(iii) we see directly that
  $\Lambda_{\Psi_s} = \dot\SSi(\cor_{U_s})$.  The formula~\eqref{eq:germ_compo}
  gives
  $$
  (\cor_{U_s} \circ \cor_{U_t})_{(x,z)} \simeq \rsect_c(\R^n; \cor_{B(x,s)}
  \otimes \cor_{B(z,t)}) \simeq \rsect_c(\R^n; \cor_{B(x,s) \cap B(z,t)}) ,
  $$
  where $B(x,s)$ is the open ball of radius $s$ centered at $x$.  The intersection
  $B(x,s) \cap B(z,t)$ is empty or homeomorphic to an $n$-ball and we deduce
  $(\cor_{U_s} \circ \cor_{U_t})_{(x,z)} \simeq \cor[-n]$ if $||x-z|| < t+s$.  We
  cannot deduce a sheaf from its stalks.  However we can make the same computation
  for a small ball instead of the stalks (it is a similar computation and the sets
  involved are cylinders instead of balls).  Alternatively we have a bound for the
  microsupport of a composition (Lemma~\ref{lem:SScompo}) which says in our case that
  $\cor_{U_s} \circ \cor_{U_t}$ is locally constant on $U_{s+t}$, hence constant
  since $U_{s+t}$ is contractible.  It follows
  $\cor_{U_s} \circ \cor_{U_t} \simeq \cor_{U_{s+t}}[-n]$.  Defining
  $K_s = \cor_{U_s}[n]$ we thus have $K_s \circ K_t \simeq K_{s+t}$ and we can expect
  that the sheaf associated with $\Psi$ is given on $\{s>0\}$ by $K = \cor_U[n]$,
  where $U = \{(x,y,s) \in \R^{2n+1}$; $s>0$, $||x-y|| < s\}$.

  What about negative times?  For $s>0$ we have
  $$
  K_s^{-1} \simeq \opb{v} \rhom(K_s , \epb{q_2}(\cor_{\R^n}))
  \simeq \opb{v} \rhom(K_s , \cor_{\R^{2n}}[n] )
  \simeq  \cor_{\ol{U_s}}
  $$
  (see~\eqref{eq:def_inv_compo}   for the definition of
  $K_s^{-1}$ and use~\eqref{eq:dual_bordlisse}).    The
  composition $K_s^{-1} \circ K$ is a sheaf on $\R^{2n} \times \rspos$. We shift
  it by $s$ to the left and get a sheaf, say $L(s)$, on
  $\R^{2n} \times \mo]-s,+\infty[$, whose restriction to $\R^{2n} \times \rspos$
  is $K$. We can compute its restriction to $\R^{2n} \times \mo]-s,0]$ and find
  $L(s)|_{\R^{2n} \times \mo]-s,0]} \simeq \cor_{i(\ol{U_s})}$ where
  $i(x,y,t) = (x,y,-t)$.  The microsupport of $L(s)$ is again given by
  Lemma~\ref{lem:SScompo} and thus is the graph of $\Psi$. Since
  $L(s)|_{\R^{2n} \times \{0\}} \simeq \cor_{\Delta_{\R^n}}$ we deduce that
  $L(s) \simeq K_\Psi$. We can take $s$ arbitrarily big and obtain
  \begin{equation}
    \label{eq:faisceau_flotgeod}
    \left\{
      \begin{aligned}
        K_\Psi|_{\R^{2n} \times \mo]0,+\infty[} &\simeq \cor_U[n] , \\
       K_\Psi|_{\R^{2n} \times \mo]-\infty,0]} &\simeq \cor_Z ,
      \end{aligned} \right.
  \end{equation}
  where $U = \{(x,y,s) \in \R^{2n+1}$; $s>0$, $||x-y|| < s\}$ and
  $Z = \{(x,y,s) \in \R^{2n+1}$; $s\leq 0$, $||x-y|| \leq -s\}$.  We remark that
  we have a distinguished triangle
  \begin{equation}
    \label{eq:dt_faisceau_flotgeod}
    \cor_U[n] \to K_\Psi \to \cor_Z \to[u] \cor_U[n+1]
   \end{equation} 
   which is not split, that is, $K_\Psi \not\simeq \cor_Z \oplus \cor_U[n]$.
   Indeed this would imply $\SSi(K_\Psi) = \SSi(\cor_U) \cup \SSi(\cor_Z)$. But
   Example~\ref{ex:microsupport}~(iv) says that $\SSi(\cor_Z)$ is bigger than
   the graph of $\Psi$ at the points $T^*\R^{2n+1}$ above
   $\Delta_{\R^n} \times \{0\}$.  In particular the morphism $u$
   in~\eqref{eq:dt_faisceau_flotgeod} is non zero.  The computation
   \begin{align*}
   \rhom(\cor_Z, \cor_U) &\simeq \rhom(\cor_Z, \DD'(\cor_{\ol U}))
   \simeq \rhom(\cor_Z \otimes \cor_{\ol U}, \cor_{\R^{2n+1}})  \\
   &\simeq \rhom(\cor_{\Delta_{\R^n} \times \{0\}}, \cor_{\R^{2n+1}}) \simeq
   \cor_{\Delta_{\R^n} \times \{0\}}[-n-1]
   \end{align*}
   shows that $\Hom(\cor_Z, \cor_U[n+1]) \simeq \cor$. Hence, if $\cor$ is a
   field, we have only one non trivial distinguished triangle
   like~\eqref{eq:dt_faisceau_flotgeod} up to isomorphism (see
   Lemma~\ref{lem:classif_extensions} below).
\end{example}

For a conic subset $A$ of $\dT^*N$ we let $\Der_{[A]}(\cor_N)$ be the full
subcategory of $\Der(\cor_N)$ formed by the $F$ with $\dot\SSi(F) \subset A$
(see Notation~\ref{not:cat_micsupp_fixe}).  Using the notation $\circ^a$
of~\eqref{eq:def_compo_ensembles} we define $A' \subset \dT^*(N \times I)$ by
$A' = \Lambda_\Psi \circ^a A$.  More explicitly we have
\begin{equation}
  \label{eq:imageA_par_isotopie}
A' = \{ (\Psi(x,\xi,s), (s; -h(\Psi(x,\xi,s),s)));
\; (x;\xi) \in A ,\; s\in I\} .
\end{equation}
We see that $A'$ is non-characteristic for the inclusion $i_{s}$, for any
$s\in I$.  Moreover $i_s^\sharp (A') = \Psi_s(A)$.  We obtain an inverse image
\begin{equation}
  \label{eq:restr_NI_Ns}
  \opb{i_{s}} \cl \Der_{[A']}(\cor_{N \times I}) \to \Der_{[\Psi_s(A)]}(\cor_N) .
\end{equation}
We can deduce from Theorem~\ref{thm:GKS} that it is an equivalence
(see~\cite[\S3.4]{GKS12}), with an inverse induced by $K_\Psi$ (however the
categories involved in~\eqref{eq:restr_NI_Ns} and the functor $\opb{i_{s}}$ only
depend on the sets $\Psi_s(A)$, $A'$, not on $\Psi$ itself):

\begin{corollary}
  \label{cor:isot_equivcat}
  {\rm(i)} For any $s \in I$ the composition $F \mapsto K_{\Psi,s} \circ F$
  induces an equivalence of categories
  $\Der_{[A]}(\cor_N) \isoto \Der_{[\Psi_{s}(A)]}(\cor_N)$, where
  $K_{\Psi,s} = K_\Psi|_{N \times \{s\}}$.

  \suiv{(ii)} For any $s \in I$ the inverse image functor~\eqref{eq:restr_NI_Ns}
  is an equivalence of categories, with an inverse given by
  $F \mapsto K_\Psi \circ K_{\Psi,s}^{-1} \circ F$.  In particular, for any
  $F, G \in \Der(\cor_N)$, we have isomorphisms
  \begin{equation}
    \label{eq:restr_NI_Nspleinfid}
    \begin{split}
\Hom(F,G)  &\isofrom \Hom(K_\Psi \circ F, K_\Psi \circ G) \\
&\isoto    \Hom(K_{\Psi,s} \circ F, K_{\Psi,s} \circ G) 
    \end{split}
  \end{equation}
  and  
  $\rsect(N; G) \isofrom \rsect(N\times I; K_\Psi \circ G) \isoto \rsect(N;
  K_{\Psi,s} \circ G)$.
\end{corollary}

\begin{remark}
  The functor~\eqref{eq:restr_NI_Ns} and Corollary~\ref{cor:isot_equivcat} do
  not depend on the whole isotopy $\Psi$ but only on the deformation
  $\Psi_s(A)$, $s\in I$, of $A$.  It is only required that this deformation is
  given by {\em some} Hamiltonian isotopy.
\end{remark}

\begin{remark}
  \label{rem:corGKSparam}
  Instead of composing with $K_\Psi$ on the left in
  Corollary~\ref{cor:isot_equivcat} we can compose on the right. We can also add
  parameters, that is, consider $F,G$ on a product $N\times P$ or $P\times N$
  for some other manifold $P$.  We only quote the following formulas for later
  use: let $F, G \in \Der(\cor_{N\times P})$,
  $F', G' \in \Der(\cor_{P\times N})$ be given; then the restriction at time
  $0 \in I$ gives isomorphisms
  \begin{align}
    \label{eq:restr_NI_Nspleinfid2}
    \Hom_{\Der(\cor_{N\times P \times I})}(K_\Psi \circ F, K_\Psi \circ G)
    &\isoto \Hom_{\Der(\cor_{N\times P})}(F,G) ,  \\
    \label{eq:restr_NI_Nspleinfid3}
    \Hom_{\Der(\cor_{P\times N \times I})}(F' \circ K_\Psi, G' \circ K_\Psi)
    &\isoto \Hom_{\Der(\cor_{P\times N})}(F',G') .
  \end{align}
\end{remark}

\section{Local behaviour}

We check here that the restriction of $K_{\Psi,s} \circ F$ to an open subset $V$
of $N$ only depends on $F|_U$ for some bigger open subset $U$.

\begin{lemma}
  \label{lem:restr_isotopie}
  Let $U$, $V$ be two open subsets of $N$ and let $s \in I$ be given.  We assume
  that $\Psi_t^{-1}(\dT^*V) \subset \dT^*U$ for all $t\in [0,s]$.  Then, for any
  $F \in \Der(\cor_N)$, the morphism $F_U \to F$ induces an isomorphism
  $$
(K_\Psi \circ F_U)|_{V \times [0,s]} \isoto (K_\Psi \circ F)|_{V \times [0,s]} .
$$
In particular $(K_{\Psi,s} \circ F_U)|_V \isoto (K_{\Psi,s} \circ F)|_V$.
\end{lemma}
\begin{proof}
  Let us set $G = K_\Psi \circ F_{N\setminus U} \in \Der(\cor_{N \times I})$ and
  $G_t = G|_{N\times \{t\}}$.  We have a distinguished triangle
  $K_\Psi \circ F_U \to K_\Psi \circ F \to G \to[+1]$ and the assertion of the
  lemma is equivalent to $G|_{V \times [0,s]} \simeq 0$.  Since
  $G_0|_V \simeq 0$ it is enough to see that $G|_{V \times [0,s]}$ is locally
  constant.

  We first remark that $G_t|_V$ is locally constant for each $t \in [0,s]$.
  Indeed
  $\dot\SSi(G_t) = \Psi_t(\dot\SSi(F_{N\setminus U})) \subset \Psi_t(
  \dT^*N\setminus \dT^*U)$ and the hypothesis $\dT^*V \subset \Psi_t(\dT^*U)$
  implies $\dot\SSi(G_t) \cap \dT^*V = \emptyset$.  It follows that $G$ is
  locally on $V\times [0,s]$ of the form
  $G \simeq \opb{q}G'$ for some
  $G' \in \Der(\cor_I)$, where $q\colon N\times I \to I$ is the projection.
  Hence $\SSi(G|_{V\times [0,s]}) \subset T^*_NN \times T^*I$.  On the other
  hand $G$ is non-characteristic for the inclusion of $N \times \{t\}$ in
  $N \times I$, for any $t\in I$. Hence $\dot\SSi(G) =\emptyset$ as required.
\end{proof}

\begin{proposition}\label{prop:restr_morph_isot}
  Let $F,G \in \Der(\cor_N)$. The restriction morphism
  $$
\opb{i_0} \rhom(K_\Psi \circ F, K_\Psi \circ G) \to \rhom(F,G)
$$
is an isomorphism.  
\end{proposition}
Using the bound $A'$ of~\eqref{eq:imageA_par_isotopie} with $A = \dT^*N$ we
could argue as in the proof of Corollary~\ref{cor:homFtGt_indpdtt}. The main
step is then to check that $(A')^a \hplus A'$ is non-characteristic for $i_0$;
here is another proof avoiding this computation.

\begin{proof}
  We set $F' = K_\Psi \circ F$, $G' = K_\Psi \circ G$.  It is enough to check
  that the restriction morphism induces an isomorphism in each degree of the
  stalks at any point $x \in N$.  We recall that
  $H^0(U;\rhom(F,G)) \simeq \Hom(F|_U, G|_U) \simeq \Hom(F_U,G)$. Hence
  \begin{equation}
    \label{eq:restr_morph_isot1}
    \begin{split}
      (H^k\rhom(F,G))_x
      &\simeq \varinjlim_{x \in U} \Hom(F|_U, G[k]|_U),  \\
      (H^k\rhom(F',G'))_{(x,0)}
      &\simeq \varinjlim_{(x,0) \in U \times J} \Hom(F'|_{U \times J}, G'[k]|_{U \times J}) ,
    \end{split}
  \end{equation}
  where $U$ runs over the open neighborhoods of $x \in N$ and $J$ over the
  neighborhoods of $0 \in I$.  For such $U$ and $J$, with $J$ contractible, we
  have
  $\Hom(K_\Psi \circ F_U|_{N\times J}, G'[k]|_{N\times J}) \isoto \Hom(F_U,
  G[k])$ by~\eqref{eq:restr_NI_Nspleinfid}.  If $V \subset U$ satisfies
  $\Psi_t^{-1}(\dT^*V) \subset \dT^*U$ for all $t\in J$, then
  $K_\Psi \circ F_U|_{V\times J} \simeq K_\Psi \circ F|_{V\times J}$ by
  Lemma~\ref{lem:restr_isotopie}. We deduce a natural morphism
  $\Hom(F_U, G[k]) \to \Hom(F'|_{V\times J}, G'[k]|_{V\times J})$ which gives a
  commutative diagram
  $$
  \begin{tikzcd}[column sep=2cm]
    \Hom(F'|_{U\times J}, G'[k]|_{U\times J}) \ar[d] \ar[r]
    & \Hom(F|_U, G[k]|_U) \ar[d]
    \ar[dl, start anchor=south west, end anchor=north east]  \\
   \Hom(F'|_{V\times J}, G'[k]|_{V\times J}) \ar[r]  & \Hom(F|_V, G[k]|_V) .
  \end{tikzcd}
  $$
  It follows that the two limits in~\eqref{eq:restr_morph_isot1} are isomorphic,
  which proves the proposition.  
\end{proof}

Let $A$ be a conic subset of $\dT^*N$ and let $A' \subset T^*(N\times I)$ be as
in~\eqref{eq:imageA_par_isotopie}. Let $U$ be an open subset of $N$.  We would
like to find a neighborhood $V$ of $U\times \{0\}$ in $N\times I$ such that the
inverse image by the inclusion of $U$ in $V$ is an equivalence of categories
$\Der_{[A' \cap T^*V]}(\cor_V) \isoto \Der_{[A \cap T^*U]}(\cor_U)$. We do not know
if such a $V$ always exists and give a weaker statement which will be sufficient
for our purposes.
 
\begin{proposition}\label{prop:equiv_local_isotopie}
  Let $A \subset \dT^*N$ and $A' \subset T^*(N\times I)$ be as
  in~\eqref{eq:imageA_par_isotopie}. Let $j \colon U \to N$ be the inclusion of
  an open subset.

  \suiv{(i)} Let $F \in \Der_{[A \cap T^*U]}(\cor_U)$. Then there exist a neighborhood
  $V$ of $U\times \{0\}$ in $N\times I$ and $G \in \Der_{[A' \cap T^*V]}(\cor_V)$
  such that $F \simeq G|_{U\times \{0\}}$.
  
  \suiv{(ii)} Let $V$ be a neighborhood of $U\times \{0\}$ in $N\times I$ and
  $G \in \Der_{[A' \cap T^*V]}(\cor_V)$. Then there exists a smaller neighborhood
  $V'$ of $U\times \{0\}$ such that $G|_{V'} \simeq (K_\Psi \circ F)|_{V'}$,
  where $F = \reim{j}(G|_{U\times \{0\}})$.  In particular, for
  $G, G' \in \Der_{[A' \cap T^*V]}(\cor_V)$ the inverse image by the inclusion
  gives an isomorphism
  $$
\rhom(G,G')|_{U\times \{0\}} \isoto \rhom(G|_{U\times \{0\}},G'|_{U\times \{0\}})
  $$
  and, if $G|_{U\times \{0\}} \simeq G'|_{U\times \{0\}}$, then there exists a
  smaller neighborhood $V''$ of ${U\times \{0\}}$ such that
  $G|_{V''} \simeq G'|_{V''}$.
\end{proposition}
\begin{proof}
  (i)  We set $F' = \reim{j}F$ and
  $G' = K_\Psi \circ F'$.  We have the rough bounds
  $\dot\SSi(F') \subset A \cup \pi_N^{-1}(\partial U)$ and
  $\dot\SSi(G') \subset A' \cup \dot\pi_{N\times \R}^{-1}(Z)$ where
  $Z = \bigsqcup_{s\in I} Z_s \times \{s\}$ and
  $Z_s = \pi_N(\Psi_s(\dot\pi_{N\times \R}^{-1}(\partial U)))$.  Then
  $V = N\times\R \setminus Z$ is a neighborhood of $U \times \{0\}$ and
  $G = G'|_V$ has the required property.

  \sui(ii) We let $k \colon V \to N\times I$ be the inclusion and set
  $G_1 = \reim{k}G$, $H = K_\Psi^{-1} \circ|_I \, G_1$.  Then
  $H \in \Der(\cor_{N\times I})$ and
  $H_s \simeq K_{\Psi,s}^{-1} \circ (G_1|_{N\times \{s\}})$.  Using the notation
  $\Lambda_\Psi \circ^a-$ (see~\eqref{eq:imageA_par_isotopie}) we have
  $\SSi(H) \subset \Lambda_{\Psi^{-1}} \circ^a|_I \SSi(G_1)$.  Then
  \begin{align*}
    \dot\SSi(H) &\subset \Lambda_{\Psi^{-1}} \circ^a|_I \dot\SSi(G_1) \\
    &\subset
      \Lambda_{\Psi^{-1}} \circ^a|_I  (A' \cup \pi_{N\times I}^{-1}(\partial V)) 
      \subset (A \times T^*_II) \cup B ,
  \end{align*}
  where $B = \Lambda_{\Psi^{-1}} \circ^a|_I \pi_{N\times I}^{-1}(\partial V)$.
  We remark that $W = (N\times I) \setminus \pi(B)$ is a neighborhood of
  $U\times \{0\}$. We let $V_1 \subset V \cap W \cap (U\times I)$ be a smaller
  neighborhood such that the fibers of the projection $p \colon V_1 \to U$ are
  intervals.  Then $\dot\SSi(H|_{V_1}) \subset A \times T^*_II$ and
  Proposition~\ref{prop:iminvproj} implies that $H|_{V_1} \simeq \opb{p}(F)$ for
  some $F \in \Der(\cor_U)$.  Hence $K_\Psi^{-1} \circ|_I \, \reim{k}G$ is
  isomorphic to $\reim{j}F \etens \cor_I$ on some neighborhood of $U$ in
  $N\times I$.  It follows easily that $\reim{k}G$ is isomorphic to
  $K_\Psi \circ|_I (\reim{j}F \etens \cor_I) \simeq K_\Psi \circ \reim{j}F$ on
  some smaller neighborhood $V'$.

  In the same way $G'$ can be written $G'|_{V'} \simeq (K_\Psi \circ F')|_{V'}$,
  maybe up to shrinking $V'$.  Now the last assertions of the proposition follow
  from Proposition~\ref{prop:restr_morph_isot} and
  Corollary~\ref{cor:isot_equivcat}.  
\end{proof}

\section{Non homogeneous case}

Now we consider the case of a non homogeneous Hamiltonian isotopy on the
cotangent bundle $T^*M$ of some manifold $M$.  We reduce this case to the
homogeneous framework by a common trick of adding one variable. Let $(x;\xi)$
and $(t;\tau)$ be the coordinates on $T^*M$ and $T^*\R$.  We define
$\rho_M \cl T^*M\times\dT^*\R \to T^*M$ by
\begin{equation}\label{eq:def_rho}
\rho_M(x,t;\xi,\tau) = (x;\xi/\tau).  
\end{equation}
The fibers of $\rho_M$ have dimension $2$.  They are stable by the translation
in the $t$ variable and are conic, that is, stable by the action of $\rspos$ on
the fibers of $\dT^*(M\times \R)$.  Let $\Phi\cl T^*M\times I \to T^*M$ be a
Hamiltonian isotopy of class $\Cinf$.  We assume that $\Phi$ has compact
support, that is, there exists a compact subset $C\subset T^*M$ such that
$\Phi(p,s) = p$ for all $p\in T^*M\setminus C$ and all $s\in I$.  Then $\Phi$ is
the Hamiltonian flow of a function $h\cl T^*M \times I \to \R$ such that $h$ is
locally constant outside $C \times I$.  In particular, if $M$ does not have a
connected component diffeomorphic to the circle $\cer$, then any Hamiltonian
isotopy on $T^*M$ with compact support can be defined by a Hamiltonian function
$h$ with compact support.

\begin{proposition}[See Prop.~A.6 of~\cite{GKS12}]
\label{prop:homog_isot}
Let $\Phi\cl T^*M\times I \to T^*M$ be a Hamiltonian isotopy with compact
support and let $h\cl T^*M \times I \to \R$ be a function with Hamiltonian flow
$\Phi$.

\suiv{(i)} Let $h' \cl \dT^*(M\times\R) \times I \to \R$ be the Hamiltonian
function given by $h'((x,t;\xi,\tau),s) = \tau h((x; \xi/\tau),s)$.  Then the
flow $\Phi'$ of $h'$ is a homogeneous Hamiltonian isotopy
$\Phi'\cl \dT^*(M\times \R) \times I \to \dT^*(M\times \R)$ whose restriction to
$T^*M\times \dT^*\R \times I$ gives the commutative diagram
\begin{equation}\label{eq:diag-PhiPsi}
\begin{tikzcd}[column sep=2cm]
T^*M\times \dT^*\R \times I \rar{\Phi'} \dar{\rho_M \times\id_I}
                              & T^*M\times \dT^*\R \dar{\rho_M}  \\
T^*M\times I\rar{\Phi}              & T^*M \pointdiag 
\end{tikzcd}
\end{equation}

\suiv{(ii)} The isotopy $\Phi'$ preserves the subset $\{\tau >0\}$ of
$\dT^*(M\times\R)$ and commutes with the vertical translations $T_c$ of
$\dT^*(M\times\R)$ given by $T_c(x,t;\xi,\tau) = (x,t+c;\xi,\tau)$, for all
$c\in\R$.  Defining $q\cl (M\times \R)^2\times I \to M^2 \times \R\times I$,
$(x,t,x',t',s) \mapsto (x,x',t-t',s)$, the graph $\Lambda_{\Phi'}$ defined
in~\eqref{eq:def_twistedgraph_global} satisfies
$\Lambda_{\Phi'} = q_d \opb{q_\pi} (\Lambda'_h)$, where
$\Lambda'_h \subset \dT^*(M^2 \times \R \times I)$ is given by
$\Lambda'_h = q_\pi \opb{q_d} (\Lambda_{\Phi'})$.
\end{proposition}
\begin{proof}
  (i) is Prop.~A.6 of~\cite{GKS12}.

  \sui(ii) Since $h'$ does not depend on $t$, the isotopy $\Phi'$ commutes with
  the Hamiltonian flow of $\tau$ which is $T_c$. The flow $\Phi'$ also preserves
  the variable $\tau$, that is, $\Lambda_{\Phi'}$ is contained in
  $\Sigma \eqdot \{\tau+\tau'=0\}$.  Then $\Sigma = \operatorname{im} q_d$ and
  the quotient map to the symplectic reduction of $\Sigma$ is $q_\pi$.  Hence we
  can write $\Lambda_{\Phi'} = q_d \opb{q_\pi} (\Lambda'_h)$, where
  $\Lambda'_h = q_\pi \opb{q_d} (\Lambda_{\Phi'})$.
\end{proof}

\begin{corollary}\label{cor:quantconic}
  Let $\Phi\cl T^*M\times I \to T^*M$ be a Hamiltonian isotopy with compact
  support and let $h\cl T^*M \times I \to \R$ be a function with Hamiltonian
  flow $\Phi$.  Then there exists a unique
  $K \in \Derlb(\cor_{M^2 \times \R \times I})$ such that
  $\dot\SSi(K) = \Lambda'_h$ and
  $K|_{M^2 \times \R \times \{0\}} \simeq \cor_{\Delta_M \times \{0\}}$, where
  $\Lambda'_h$ is defined in Proposition~\ref{prop:homog_isot}.
\end{corollary}
\begin{proof}
  We use the notations of Proposition~\ref{prop:homog_isot}.  By
  Proposition~\ref{prop:iminvproj} the inverse image functor $\opb{q}$ gives an
  equivalence between $\{K' \in \Derlb(\cor_{M^2 \times \R \times I})$;
  $\dot\SSi(K') = \Lambda'_h\}$ and
  $\{K \in \Derlb(\cor_{(M \times \R)^2 \times I})$;
  $\dot\SSi(K) = \Lambda_{\Phi'}\}$.  Then the existence and uniqueness of $K$
  follows from Theorem~\ref{thm:GKS}.
\end{proof}

\part{Cut-off lemmas}
\label{chap:cutoff}

In this part we recall several results of~\cite{KS90} which are called ``(dual)
(refined) cut-off lemmas''.  We also give another version which removes some
convexity hypotheses.  We apply these results to decompose a sheaf with respect to a
partition of its microsupport (see Proposition~\ref{prop:cut-off_split_local2}
below).  The starting point is the following problem: for a given sheaf
$F\in \Der(\cor_V)$, with $V = \R^n$, and an open cone $C \subset T^*_0V = V^*$, find
a sheaf $G$ such that $\SSi(G) = \SSi(F) \setminus (\SSi(F) \cap C)$ and $F \simeq G$
in the quotient category of $\Der(\cor_V)$ by $\{H$;
$(\SSi(H) \cap T^*_0V) \subset \ol C \}$.  As we have seen in
Corollary~\ref{cor:micsuppLagrlisse} the involutivity theorem~\ref{thm:invol} says
this problem has no solution in general.  However, for a given neighborhood $\Omega$
of $\SSi(F) \cap \partial C$ in $T^*V$, it is possible to find $G$ with
$\SSi(G) = (\SSi(F) \setminus (\SSi(F) \cap C)) \cup \Omega$ near $T^*_0V$ (see
Proposition~\ref{prop:microcutofflevrai1} below).  We first discuss a weaker version
of this problem where we don't ask that
$\SSi(G) = \SSi(F) \setminus (\SSi(F) \cap C)$ but we only ask for a ``cut-off''
functor $F \mapsto P(F)$ so that $\SSi(P(F))$ is contained in
$V \times (V^*\setminus C)$. There are several ways to define such a $P$ and we give
four variations on the cut-off functors of~\cite{KS90} in~\eqref{eq:defPgam}.  If we
want $P$ to be a projector (see Remark~\ref{rem:proj}) these are in fact the only
possible choices (see~\S\ref{sec:Tamproj} for a quick discussion).

One version of the cut-off lemma says that the category of sheaves on $V$ with
microsupport in some convex cone $\gamma$ is equivalent to the category of sheaves on
$V$ endowed with another topology. We use it to prove that $\SSi(H^i(F))$ is
contained in the convex hull of $\SSi(F)$ (see Corollary~\ref{cor:SSHiF} below --
this will be used to construct a graph selector in Part~\ref{part:graphsel}).

In the last paragraph we consider a relative version of the cut-off functor and
recall some its properties already considered by Tamarkin which we will use to
recover non-squeezing results in~\S\ref{sec:nonsqueezing}.

\section{Global cut-off}
\label{sec:Globalcut-off}

Let $V$ be a vector space of dimension $n$ and let $\gamma \subset V$ be a closed
convex cone (with vertex at $0$). We denote by $\gamma^a = -\gamma$ its opposite cone
and by $\gamma^\circ \subset V^*$ its polar cone (see~\eqref{eq:def_polar_cone}).  We
also define $\widetilde\gamma = \{(x,y)\in V^2$; $x-y \in \gamma\}$.  Let
$q_i \cl V^2 \to V$, $i=1,2$, be the projection to the $i^{th}$ factor and let
$\Delta_V$ be the diagonal of $V^2$.  The following functors are introduced
in~\cite{KS90}:
\begin{equation}
  \label{eq:defPgam}
  \begin{alignedat}{2}
P_\gamma &\cl \Der(\cor_V) \to \Der(\cor_V), & \quad
F & \mapsto \roim{q_2}( \cor_{\widetilde\gamma} \tens \opb{q_1} F) , \\
Q_\gamma &\cl \Der(\cor_V) \to \Der(\cor_V), & \quad
F & \mapsto \reim{q_2}( \rhom(\cor_{\widetilde\gamma^a} , \epb{q_1} F)) , \\
P'_\gamma &\cl \Der(\cor_V) \to \Der(\cor_V), & \quad
F & \mapsto \reim{q_2}( \rhom(\cor_{\widetilde\gamma^a \setminus \Delta_V}[1], \epb{q_1} F)) , \\
Q'_\gamma &\cl \Der(\cor_V) \to \Der(\cor_V), & \quad
F &\mapsto \roim{q_2}( \cor_{\widetilde\gamma \setminus \Delta_V}[1] \tens \opb{q_1} F) .
\end{alignedat}
\end{equation}
We will mainly use the first three functors and introduce $Q'_\gamma$ because these
functors come in pairs (adjoint pairs or pairs of projectors -- see
\S\ref{sec:Tamproj}).  We will see the effect of these functors on the microsupport
in~\eqref{eq:SSPgamQgam} and Propositions~\ref{prop:cut-off1}, \ref{prop:cut-off2}
and~\ref{prop:cut-off_split} below.  For $\gamma = \{0\}$, we have
$\widetilde{\{0\}} = \Delta_V$ and $P_{\{0\}}(F) \simeq Q_{\{0\}}(F) \simeq F$.
Using the distinguished triangle
$\cor_{\widetilde\gamma} \to \cor_{\Delta_V} \to \cor_{\widetilde\gamma \setminus
  \Delta_V}[1] \to[+1]$ (and the same with $\gamma^a$) we obtain morphisms of
functors
\begin{equation}
\label{eq:morph_func_gam}
\begin{alignedat}{2}
 u_\gamma&\cl P_\gamma \to \id, &\qquad v_\gamma&\cl \id \to Q_\gamma , \\
 u'_\gamma&\cl P'_\gamma \to \id, &  v'_\gamma&\cl \id \to Q'_\gamma 
\end{alignedat}
\end{equation}
and, for any $F \in \Der(\cor_V)$, the distinguished triangles
\begin{equation}
\label{eq:dist_tri_gam}
\begin{gathered}
P_\gamma(F) \to[\;u_\gamma(F)\;] F \to[\;v'_\gamma(F)\;] Q'_\gamma(F) \to[+1] , \\
P'_\gamma(F) \to[\;u'_\gamma(F)\;] F \to[\;v_\gamma(F)\;] Q_\gamma(F) \to[+1] .
\end{gathered}
\end{equation}

The     functors $P_\gamma$,
$Q_\gamma$, $P'_\gamma$, $Q'_\gamma$ have similarities with the composition
operation~\eqref{eq:def_compo_faisceaux}.  Let us recall the composition: for three
manifolds $X,Y,Z$ and $K \in \Der(\cor_{X\times Y})$, $L \in \Der(\cor_{Y\times Z})$
\begin{equation*}
K \circ L =  \reim{q_{13}}( \opb{q_{12}}K \tens \opb{q_{23}} L) ,
\end{equation*}
where $q_{ij}$ is the projection from $X\times Y \times Z$ to the $i\times j$-factor.
Taking $X=\pt$ and $Y=Z=V$, we have, with the notations of~\eqref{eq:defPgam},
$F \circ \cor_{\widetilde\gamma} = \reim{q_2}( \cor_{\widetilde\gamma} \tens
\opb{q_1} F)$, which is almost $P_\gamma(F)$ up to the change from $\reim{q_2}$ to
$\roim{q_2}$.  The following lemma gives hypothesis on $F$ so that both functors are
the same.
\begin{lemma}
\label{lem:Pgamma-conv}
Let $F\in \Der(\cor_V)$.  We assume that $\supp(F)$ is compact or that there exists a
linear function $\xi\colon V \to \R$ such that
$\gamma \setminus \{0\} \subset \xi^{-1}(\mo]-\infty,0[)$ and
$\supp(F) \subset \xi^{-1}([a,\infty[)$, for some $a\in\R$.  Then
$P_\gamma(F) \simeq F \circ \cor_{\widetilde\gamma}$ and
$\supp(P_\gamma(F) ) \subset \xi^{-1}([a,\infty[)$.
\end{lemma}
\begin{proof}
  As was already remarked, the difference between $P_\gamma(F)$ and
  $F \circ \cor_{\widetilde\gamma}$ is the change from $\roim{q_2}$ to $\reim{q_2}$
  in~\eqref{eq:defPgam}.  Hence it is enough to prove that the projection $q_2$ is
  proper on $S = \supp(\cor_{\widetilde\gamma} \tens \opb{q_1} F)$ to deduce the
  isomorphism. This is clear when $\supp(F)$ is compact. We assume the other
  hypothesis.

  Let $||\cdot||$ be some Euclidean norm on $V$.  Since $\gamma$ is a closed cone, we
  actually have $\gamma \subset \{v\in V;\; \xi(v) \leq -c ||v||\}$ for some $c>0$.
  Hence
  $S \subset \{(v_1,v_2); \; \xi(v_1-v_2) \leq -c||v_1-v_2||,\; \xi(v_1)\geq a\}$.
  For a given $v_2\in V$ we see that
  $S \cap q_2^{-1}(\{v_2\}) \subset \{v_1; \; ||v_1-v_2|| \leq c^{-1}(\xi(v_2)-a)\}$
  is compact. Hence $P_\gamma(F) \simeq F \circ \cor_{\widetilde\gamma}$.

  Moreover, if $(P_\gamma(F))_{v_2} \not=0$, we must have
  $S \cap q_2^{-1}(\{v_2\}) \not=\emptyset$, hence $\xi(v_2) -a \geq 0$. It follows
  that $\supp(P_\gamma(F) ) \subset \xi^{-1}([a,\infty[)$.
\end{proof}

We can also reformulate $Q_\gamma$ and $P'_\gamma$ with the composition.  Since
$\widetilde\gamma^a = d^{-1}(\gamma^a)$ with $d(x,y) = x-y$, we have
$\cor_{\widetilde\gamma^a} \simeq \opb{d}(\cor_{\gamma^a})$ and
Theorem~\ref{thm:iminv} gives
$\SSi(\cor_{\widetilde\gamma^a}) \subset \{(x,y;\xi,-\xi)\}$. We also have
$\SSi(\epb{q_1} F)) \subset \{(x,y;\xi,0)\}$.  Hence the microsupports of
$\cor_{\widetilde\gamma^a}$ and $\epb{q_1} F$ do not intersect outside the zero
section and Theorem~\ref{thm:SSrhom} implies
\begin{equation}
  \label{eq:Qgamma_conv0}
\DD'(\cor_{\widetilde\gamma^a}) \ltens \epb{q_1} F \isoto
\rhom(\cor_{\widetilde\gamma^a} , \epb{q_1} F) .
\end{equation}
Choosing an orientation for $V$ we have $\epb{q_1} F \simeq\opb{q_1} F[n]$ and we
deduce the first isomorphism below; the second one is proved in the same way, using
$\widetilde\gamma^a \setminus \Delta_V = d^{-1}(\gamma^a\setminus \{0\})$,
\begin{equation}
  \label{eq:Qgamma_conv}
  \begin{split}
  Q_\gamma(F) &\simeq F \circ (\DD'(\cor_{\widetilde\gamma^a})[n]) , \\
  P'_\gamma(F) &\simeq F \circ (\DD'(\cor_{\widetilde\gamma^a\setminus \Delta_V})[n-1]) .
\end{split}
\end{equation}
We deduce the following adjunction properties:
\begin{lemma}\label{lem:PgammaQgammaadj}
  For any $F,G \in \Der(\cor_V)$ we have
  $\Hom(Q_\gamma(F),G) \simeq \Hom(F,P_\gamma(G))$ and
  $\Hom(P'_\gamma(F),G) \simeq \Hom(F,Q'_\gamma(G))$.  
\end{lemma}
\begin{proof}
  We only prove the first isomorphism, the second one being similar.
  Using~\eqref{eq:Qgamma_conv} we see $Q_\gamma$ as a composition of left adjoint
  functors and we have
  \begin{align*}
    \Hom(Q_\gamma(F),G)
    &\simeq \Hom( \reim{q_2}( (\DD'(\cor_{\widetilde\gamma^a})[n]) \ltens \opb{q_1}(F),G) \\
    &\simeq \Hom(F, \roim{q_1}( \rhom( (\DD'(\cor_{\widetilde\gamma^a})[n], \epb{q_2}G))) .
  \end{align*}
  By a microsupport argument similar to the proof of~\eqref{eq:Qgamma_conv0} we
  obtain
  $\DD'(\DD'(\cor_{\widetilde\gamma^a})[n]) \ltens \epb{q_2}G \isoto \rhom(
  (\DD'(\cor_{\widetilde\gamma^a})[n], \epb{q_2}G)$.  Since
  $\cor_{\widetilde\gamma^a}$ is cohomologically constructible,
  $\DD'(\DD'(\cor_{\widetilde\gamma^a})) \simeq \cor_{\widetilde\gamma^a}$ and
  we obtain
  $$
  \Hom(Q_\gamma(F),G) \simeq \Hom(F, \roim{q_1}( \cor_{\widetilde\gamma^a}
  \otimes \opb{q_2}G)) .
  $$
  We conclude with the remark that the automorphism $(x,y) \mapsto (y,x)$ of $V^2$
  interchanges $q_1$, $q_2$ and $\widetilde\gamma^a$, $\widetilde\gamma$.
\end{proof}

\begin{examples}
  (i) The   easiest example is the image of
  the skyscraper sheaf $\cor_{\{x\}}$ for a given $x\in V$. Using
  Lemma~\ref{lem:Pgamma-conv} and the formula~\eqref{eq:germ_compo} for the
  stalks of a composition, we find
  $P_\gamma(\cor_{\{x\}}) \simeq \cor_{ x -\gamma}$.  If
  $\Int(\gamma) \not=\emptyset$, we have
  $\DD'(\cor_{\widetilde\gamma^a}) \simeq \cor_{\Int(\widetilde\gamma^a)}$
  by~\eqref{eq:dual_bordlisse}.  Hence~\eqref{eq:Qgamma_conv} gives
  $Q_\gamma(F) \simeq F \circ \cor_{\Int(\widetilde\gamma^a)}[n]$.
  By~\eqref{eq:germ_compo} again
  $Q_\gamma(\cor_{\{x\}}) \simeq \cor_{ \Int(x +\gamma)}[n]$.  We illustrate
  these results in Fig.~\ref{fig:PQgamma1} (for a sheaf of the type $L_Z$, $Z$
  locally closed, we draw the set $Z$ and label it with $L$).

\begin{figure}[ht]
  \begin{tikzpicture}
\foreach \i/\j/\bord in {0/1/black, 6/-1/black, 9/1/dashed}
{
  \begin{scope}[xshift= \i cm, yscale=\j]
\coordinate (A) at (0,0); \coordinate (B) at (-1,-2);
\coordinate (C) at (1,-2);
\fill [fill=gray!20] (B) -- (A) -- (C) -- cycle;
\draw [\bord] (B)--(A)--(C);
\end{scope}
};
\node at (-.5,0) {$\gamma$};
\node at (3,0) {$x$ $\bullet$};
\node at (5.5,-.3) {$P_\gamma(\cor_{\{x\}})$};
\node at (6,1) {$\cor$};
\node at (8.5,.3) {$Q_\gamma(\cor_{\{x\}})$};
\node at (9,-1.3) {$\cor[n]$};
  \end{tikzpicture}
\caption{}   \label{fig:PQgamma1}
\end{figure}

The morphism $v_\gamma(\cor_{\{x\}})$ is given by the inverse image of $1$ in the
sequence of isomorphisms:
  \begin{align*}
  \Hom(\cor_{\{x\}}&, \cor_{ \Int(x +\gamma)}[n]) \simeq \Hom(\cor_{\{x\}},
    \DD(\cor_{x +\gamma}))  \\
    &\simeq \Hom(\cor_{\{x\}} \otimes \cor_{x +\gamma},
  \cor_V[n]) \simeq H^{n}_{\{x\}}(\cor_V) \simeq \cor .    
  \end{align*}

  \sui(ii) Here are other easy examples in dimension $1$.  We set
  $\gamma = \mo]-\infty,0]$.
  $$
  \setlength{\arraycolsep}{3mm}
  \renewcommand{\arraystretch}{1.2}
  \begin{array}{c|cccc}
    F & P_\gamma(F) & Q_\gamma(F) & P'_\gamma(F)  & Q'_\gamma(F) \\[1mm]
    \hline
    \cor_{[0,1]} & \cor_{[0,+\infty[} & (\cor_{\mo]-\infty,0[})[1]
              & \cor_{\mo]-\infty,1]} &  (\cor_{]1,+\infty[})[1]\\
    \cor_{[0,1[} & \cor_{[0,1[} & \cor_{[0,1[} & 0 & 0\\
    \cor_{]0,1]} & 0 & 0 & \cor_{]0,1]} & \cor_{]0,1]} \\
    \cor_{]0,1[} & (\cor_{[1,+\infty[})[-1]  & \cor_{\mo]-\infty,1[}
                 &  (\cor_{\mo]-\infty,0]})[-1] & \cor_{]0,+\infty[} 
  \end{array}
  $$
  It is easy to check here the properties stated in
  Proposition~\ref{prop:cut-off_split}: if $(t;\tau)$ are the coordinates on $T^*\R$
  we have $\dot\SSi( P_\gamma(F)) = \dot\SSi(F) \cap \{\tau >0\}$,
  $\dot\SSi( P'_\gamma(F)) = \dot\SSi(F) \cap \{\tau <0\}$ and $F$ is decomposed ``up
  to a constant sheaf'' as $P_\gamma(F) \oplus P'_\gamma(F)$
  (triangle~\eqref{eq:cut-off_split}); for example we have the triangle
  $$
  \cor_\R \to \cor_{[0,+\infty[} \oplus \cor_{\mo]-\infty,1]}
  \to \cor_{[0,1]} \to[+1] .
  $$
\end{examples}

We use the identifications $T^*_xV = V^*$ for any $x\in V$ and $T^*V = V\times
V^*$.  
\begin{lemma}\label{lem:bornesPFQp}
  Let $F \in \Der(\cor_V)$. We assume that $F$ has a compact support.  Then, for
  any $x\in V$,
\begin{align}
\label{eq:SSPgamF}
\SSi(P_\gamma(F)) \cap T^*_xV
\subset p_2(\SSi(F) \cap \SSi(\cor_{x+\gamma})^a),  \\
\label{eq:SSQgamF}
\SSi(Q_\gamma(F)) \cap T^*_xV
\subset p_2(\SSi(F) \cap \SSi(\cor_{x+\gamma^a})),  
\end{align}
where $p_2 \cl T^*V \to T^*_xV$ is the projection.
\end{lemma}
\begin{proof}
  We only prove the first inclusion, the proof of the second one being similar.
  We set $G = \cor_{\widetilde\gamma} \tens \opb{q_1} F$, hence
  $P_\gamma(F) = \roim{q_2}(G)$. Using $\widetilde\gamma = d^{-1}(\gamma)$ with
  $d(x,y) = x-y$, we see by Theorems~\ref{thm:iminv} and~\ref{thm:SSrhom}:
  \begin{gather*}
    \SSi(\cor_{\widetilde\gamma})
    = \{(x,y;\xi,-\xi);\; (x-y;\xi) \in \SSi(\cor_\gamma) \}, \\
    \SSi(G) \subset  \SSi(\cor_{\widetilde\gamma}) + (\SSi(F) \times T^*_VV) .
  \end{gather*}
  Since $\supp(F)$ is compact, Proposition~\ref{prop:oim} implies that
  $\SSi(P_\gamma(F))$ is bounded by $p_2'( \SSi(G) \cap (T^*_VV\times T^*V))$
  where $p_2' \colon T^*V^2 \to T^*V$ is the second projection.  A point
  $(x_1,x_2;0,\eta)$ of $T^*_VV\times T^*V$ belongs to $\SSi(G)$ if and only if
  there exist $(x_1,x_2;\xi,-\xi) \in \SSi(\cor_{\widetilde\gamma})$ and
  $(x_1;\xi_1) \in \SSi(F)$ such that $(0,\eta) = (\xi,-\xi)+(\xi_1,0)$.  In
  other words $(x_1-x_2;-\eta) \in \SSi(\cor_\gamma)$ and
  $(x_1;\eta) \in \SSi(F)$.  Now, if $x_2=x$ is fixed, we obtain
  $(x_1;\eta) \in \SSi(\cor_{x+\gamma})^a \cap \SSi(F)$ and the result follows.
\end{proof}

By Example~\ref{ex:microsupport}
$\SSi(\cor_{x+\gamma^a}) \subset V \times \gamma^{\circ a}$ and
Lemma~\ref{lem:bornesPFQp} gives
$\SSi(P_\gamma(F))\cup \SSi(Q_\gamma(F)) \subset V \times \gamma^{\circ a}$ if $F$
has compact support; but this   result holds without the compactness
hypothesis:

\begin{lemma}
  For any $F\in \Der(\cor_V)$ we have  
\begin{equation}
\label{eq:SSPgamQgam}
\SSi(P_\gamma(F))\cup \SSi(Q_\gamma(F)) \subset V \times \gamma^{\circ a} . 
\end{equation}
\end{lemma}
\begin{proof}
  We prove the result for $P_\gamma(F)$, the case of $Q_\gamma(F)$ being similar.  We
  set $G = \cor_{\widetilde\gamma} \tens \opb{q_1} F$ as in the proof of
  Lemma~\ref{lem:bornesPFQp}. We have already seen
  $\SSi(G) \subset \SSi(\cor_{\widetilde\gamma}) + (\SSi(F) \times T^*_VV)$ (this
  didn't require the compactness of $\supp(F)$) and we deduce the rough bound
  $\SSi(G) \subset T^*V \times \gamma^{\circ a}$.  If $q_2$ were proper on $\supp(G)$
  we could conclude with Proposition~\ref{prop:oim}. In our case we can use a
  compactification of $V$.  We choose a diffeomorphism $\R \isoto \mo]0,1[$ and
  deduce $V = \R^n \simeq \mo]0,1[^n$ and then
  $\varphi \colon V^2 \isoto \mo]0,1[^n\times V$.  We let
  $j\colon \mo]0,1[^n\times V \hookrightarrow \R^n\times V$ be the inclusion.  Since
  $\varphi$ is proper, we can use Proposition~\ref{prop:oim} and we obtain
  $\SSi(\roim{\varphi} (G)) \subset T^*]0,1[^n \times \gamma^{\circ a}$.  Then
  Theorem~\ref{thm:oim_open} gives
  $\SSi(\roim{j}\roim{\varphi} (G)) \subset T^*V \times \gamma^{\circ a}$.  Since
  $q_2$ is proper on $[0,1]^n \times V$, we can apply Proposition~\ref{prop:oim}
  again and, using $\roim{q_2}\roim{j}\roim{\varphi} (G) \simeq \roim{q_2}(G)$, we
  obtain the result.
\end{proof}

The first cut-off results say roughly that the bound~\eqref{eq:SSPgamQgam}
characterizes sheaves of the type $P_\gamma(F)$ or $Q_\gamma(F)$ (see
Propositions~\ref{prop:cut-off1} and~\ref{prop:cut-off2}).

For a given subset $\Omega$ of $T^*V$, a morphism $a\cl F \to G$ in $\Der(\cor_V)$ is
said to be an {\em isomorphism on $\Omega$} if $\SSi(C(a)) \cap \Omega = \emptyset$,
where $C(a)$ is given by the distinguished triangle $F\to[a] G \to C(a) \to[+1]$.
This implies that $\SSi(F) \cap \Omega = \SSi(G) \cap \Omega$.    We will need the following composition result:
\begin{lemma}\label{lem:compo_isosurOmega}
  If $a\cl F \to G$, $b\colon G \to H$ are isomorphisms on $\Omega$, then so is
  $b\circ a$.  
\end{lemma}
\begin{proof}
  The octahedron axiom implies that we have a distinguished triangle
  $C(a) \to C(b\circ a) \to C(b) \to C(a)[1]$, where $C(a)$, $C(b)$, $C(b\circ a)$
  are the cones of $a$, $b$, $b\circ a$.  Now the result follows from the triangular
  inequality for the microsupport.
\end{proof}

The   next two propositions are
taken from~\cite{KS90} where they are stated using the $\gamma$-topology (see
\S\ref{sec:gamma_top}).  The link between the $\gamma$-topology and the functors
$P_\gamma$, $Q_\gamma$ is also explained in Proposition~3.5.4 of~\cite{KS90}.

\begin{proposition}[see~\cite{KS90} Prop. 5.2.3]
\label{prop:cut-off1}
For $F\in \Der(\cor_V)$ we let $u_\gamma(F)\cl P_\gamma(F) \to F$ be the
morphism in~\eqref{eq:morph_func_gam}. Then
\begin{itemize}
\item [(i)] $u_\gamma(F)$ is an isomorphism on $V \times \Int(\gamma^{\circ
    a})$,
\item [(ii)] $u_\gamma(F)$ is an isomorphism if and only if $\SSi(F) \subset V
  \times \gamma^{\circ a}$.
\end{itemize}
\end{proposition}

\begin{proposition}[see~\cite{KS90} Lem. 6.1.5]
\label{prop:cut-off2}
We assume that $\gamma$ is proper (that is, $\gamma$ contains no line) and
$\Int(\gamma) \not= \emptyset$.  Let $F\in \Der(\cor_V)$.  We assume that $F$ has
compact support.  Then the morphism $v_\gamma(F) \cl F \to Q_\gamma(F)$
in~\eqref{eq:morph_func_gam} is an isomorphism on $V \times \Int(\gamma^{\circ a})$.
\end{proposition}

\begin{remark}\label{rem:proj}
  (1) The bound~\eqref{eq:SSPgamQgam} and Proposition~\ref{prop:cut-off1}-(ii) show
  that $u_\gamma(P_\gamma(F)) \colon P_\gamma(P_\gamma(F)) \to P_\gamma(F)$ is an
  isomorphism.  In other words the morphism of functors
  $u_\gamma \colon P_\gamma \to \id_{\Der(\cor_V)}$ induces an isomorphism
  $P_\gamma \circ P_\gamma \isoto P_\gamma$.  The functor $P_\gamma$ is then called a
  {\em projector} (see for example~\cite[\S4.1]{KS06}).  More precisely it is the
  projector to the subcategory $\Der_{V \times \gamma^{\circ a}}(\cor_V)$ of
  $\Der(\cor_V)$ formed by the $F$ such that
  $\SSi(F) \subset V \times \gamma^{\circ a}$ (Notation~\ref{not:cat_micsupp_fixe}).
  By the dual statements of~\cite[Prop.~4.1.3, 4.1.4]{KS06} the functor
  $R\colon \Der(\cor_V) \to \Der_{V \times \gamma^{\circ a}}(\cor_V)$ induced by
  $P_\gamma$ is both right adjoint and left inverse to the inclusion $\iota$ of
  $\Der_{V \times \gamma^{\circ a}}(\cor_V)$ and we have
  $P_\gamma \simeq \iota \circ R$.

  We will see this again in a slightly different form in \S\ref{sec:gamma_top} where
  $\Der_{V\times\gamma^{\circ a}}(\cor_V)$ is identified with the category of sheaves
  on $V_\gamma$, which is the space $V$ with the ``$\gamma$-topology'' (see
  Proposition~\ref{prop:cut-offgamma} -- then $\iota$, $R$ correspond to
  $\opb{\phi_\gamma}$, $\roim{\phi_\gamma}$).

  \sui(2) Since $Q_\gamma$ is left adjoint to $P_\gamma$, it is also a projector.
  More precisely, $Q_\gamma \circ Q_\gamma$ is adjoint to
  $P_\gamma \circ P_\gamma \simeq P_\gamma$ and by uniqueness of adjoints we have
  $Q_\gamma \circ Q_\gamma \simeq Q_\gamma$.  Instead of a morphism
  $P_\gamma \to \id$ we have here a morphism $\id \to Q_\gamma$.
  By~\cite[Prop.~4.1.3, 4.1.4]{KS06} $Q_\gamma$ is then a projector to some full
  subcategory, say $\catd$, of $\Der(\cor_V)$.  Using the decomposition
  $P_\gamma \simeq \iota \circ R$ in~(1) and $R \circ \iota \simeq \id$ we can prove
  $\Hom(P_\gamma(F), P_\gamma(G)) \isoto \Hom(P_\gamma(F), G)$, for any
  $F,G \in \Der(\cor_V)$.  We deduce
  $\Hom(Q_\gamma(P_\gamma(F)),G) \isoto \Hom(P_\gamma(F), G)$, proving that
  $Q_\gamma \circ P_\gamma \simeq P_\gamma$.  This implies that
  $\Der_{V\times\gamma^{\circ a}}(\cor_V)$ is stable by $Q_\gamma$.  Hence we have
  $\Der_{V\times\gamma^{\circ a}}(\cor_V) \subset \catd$.  In the same way we have
  $P_\gamma \circ Q_\gamma \simeq Q_\gamma$ and we obtain finally
  $\Der_{V\times\gamma^{\circ a}}(\cor_V) = \catd$.
  
  Hence $Q_\gamma$ is also a projector to $\Der_{V\times\gamma^{\circ
      a}}(\cor_V)$. The difference with $P_\gamma$ is that we could write
  $Q_\gamma \simeq \iota \circ L$ where $L$ is {\em left} adjoint to the embedding
  $\iota$.

  \sui(3) Since, by~(2), $Q_\gamma$ is a projector to
  $\Der_{V\times\gamma^{\circ a}}(\cor_V)$, we have as in
  Proposition~\ref{prop:cut-off1}-(ii): $v_\gamma(F)$ is an isomorphism if and only
  if $\SSi(F) \subset V \times \gamma^{\circ a}$, for any $F\in \Der(\cor_V)$ and
  without the additional hypothesis that $\Int(\gamma) \not= \emptyset$. However we
  will only use the statement of Proposition~\ref{prop:cut-off2}.
\end{remark}

A consequence of Proposition~\ref{prop:cut-off1}-(i)   is the
bound
\begin{equation}\label{eq:borne_SSPgammaF}
\SSi(P_\gamma(F)) \cap (V \times \Int(\gamma^{\circ a})) = \SSi(F) \cap
 (V \times \Int(\gamma^{\circ a})) .
\end{equation}
However $\SSi(P_\gamma(F)) \cap (V \times \partial\gamma^{\circ a})$ could be bigger
than $\SSi(F) \cap (V \times \partial\gamma^{\circ a})$.  If we assume that
$\dot\SSi(F)$ does not meet $V \times \partial\gamma^{\circ a}$ and $\supp(F)$ is
compact, then $\SSi(P_\gamma(F))$ does not meet $V \times \partial\gamma^{\circ a}$
by~\eqref{eq:SSPgamF}.  Hence~\eqref{eq:SSPgamQgam} and~\eqref{eq:borne_SSPgammaF}
actually imply $\SSi(P_\gamma(F)) = \SSi(F) \cap (V \times \Int(\gamma^{\circ a}))$.
In Proposition~\ref{prop:cut-off_split} below we see the similar bound
$\dot\SSi(P'_\gamma(F)) = \dot\SSi(F) \cap (V \times (V \setminus \gamma^{\circ a}))$
and that $F$ is split as $P_\gamma(F) \oplus P'_\gamma(F)$ up to a constant sheaf.

Actually,   in Proposition~\ref{prop:cut-off_split} we give a weaker hypothesis
than $\SSi(F) \cap (V \times \partial\gamma^{\circ a}) = \emptyset$.  We want to
split $F$ only on some open subset $W$ of $V$ and we can replace
$V \times \partial\gamma^{\circ a}$ by some smaller set depending on $W$, namely
$\bigcup_{x\in W} S_x^\gamma$ where $S_x^\gamma$ is defined below (remark that
$\bigcup_{x\in V} S_x^\gamma = V \times \partial\gamma^{\circ a}$).  This more
precise result is used in the proof of Proposition~\ref{prop:cut-off_split_local0}.

The motivation to introduce $S_x^\gamma$ is the following.  We first want to ensure
that $\dot\SSi(P_\gamma(F)) \cap T^*_xV$ coincides with
$\dot\SSi(F) \cap (V \times \Int(\gamma^{\circ a})) \cap T^*_xV$.
By~\eqref{eq:SSPgamQgam} and Proposition~\ref{prop:cut-off1} it is enough that
$\dot\SSi(P_\gamma(F)) \cap T^*_xV \cap (\{x\} \times ( \partial\gamma^{\circ a})) =
\emptyset$, where
$\partial\gamma^{\circ a} = \gamma^{\circ a} \setminus \Int\gamma^{\circ a}$.
Using~\eqref{eq:SSPgamF} this holds if we assume that $\dot\SSi(F)$ does not meet
$(V\times \partial\gamma^{\circ a}) \cap \dot\SSi(\cor_{x+\gamma})^a$, which is half
of $S_x^\gamma$. The other half is introduced to bound
$\dot\SSi(Q_\gamma(F)) \cap T^*_xV$ in the same way, using~\eqref{eq:SSQgamF}.

Let $\gamma \subset V$ be a closed convex proper cone.  For $x\in V$ we define
$S_x^\gamma \subset T^*V$ by  
\begin{equation}\label{eq:def_Sgammax}
  \begin{split}
S_x^\gamma  & = (\dot\SSi(\cor_{x + \gamma})^a \cup \dot\SSi(\cor_{x+\gamma^a}) )
 \setminus (\{x\}\times \Int(\gamma^{\circ a})) \\
& = (\dot\SSi(\cor_{x + \gamma})^a \cup \dot\SSi(\cor_{x+\gamma^a}) )
\cap (V\times \partial\gamma^{\circ a}),
\end{split}
\end{equation}
where the second equality follows from~\eqref{eq:borneSScone_bord}. 

\begin{figure}[ht]
  \begin{tikzpicture}[decoration={border, segment length=1.8mm, amplitude=2mm, angle=90}]

\begin{scope}
\coordinate (A) at (0,0); \coordinate (B) at (-1,-2);
\coordinate (C) at (1,-2);
\fill [fill=gray!20] (B) -- (A) -- (C) -- cycle;
\draw (B)--(A)--(C);
\end{scope}

\begin{scope}[xshift= 2.5cm]
\coordinate (A) at (0,0);
\coordinate (B) at (-1,-2); \coordinate (C) at (1,-2);
\coordinate (D) at (.2,.1); 
\draw [postaction={decorate,draw}] (B)--(A)--(C);
\fill [fill=black] (A) -- (D) arc[radius=.22cm, start angle=27, end angle=153]
-- cycle;
\end{scope}

\begin{scope}[xshift= 5cm]
\coordinate (A) at (0,0);
\coordinate (B) at (-1,2); \coordinate (C) at (1,2);
\coordinate (D) at (.2,.1); 
\draw [postaction={decorate,draw}] (B)--(A)--(C);
\fill [fill=black] (A) -- (D) arc[radius=.22cm, start angle=27, end angle=153]
-- cycle;
\end{scope}

\begin{scope}[xshift= 9cm]
\draw [postaction={decorate,draw}] (-1,-2) -- (1,2);
\draw [postaction={decorate,draw}] (-1,2) -- (1,-2);  
\end{scope}

\node at (-.5,-2.5) {$\gamma$};
\node at (2.5,-2.5) {$\dot\SSi(\cor_{x + \gamma})^a$};
\node at (5.5,-2.5) {$\dot\SSi(\cor_{x+\gamma^a})$};
\node at (9,-2.5) {$S_x^\gamma$};
\node at (2.1,0) {$x$};
\node at (4.6,0) {$x$};
\node at (8.6,0) {$x$};
  \end{tikzpicture}
\caption{}   \label{fig:Sxgamma}
\end{figure}

In Fig.~\ref{fig:Sxgamma} we have pictured $\dot\SSi(\cor_{x + \gamma})^a$,
$\dot\SSi(\cor_{x+\gamma^a})$ and $S_x^\gamma$ in dimension $2$ (we identify
$T^*\R^2$ with $T\R^2$; we note that $\dot\SSi(\cor_{x + \gamma})^a$ and
$\dot\SSi(\cor_{x+\gamma^a})$ contain $\{x\}\times \Int(\gamma^{\circ a})$).  We can
give an easy description of $S_x^\gamma$ when $\gamma$ has a smooth boundary away
from $0$, as in the following example (this is in fact the only case we need).
  For $n>1$ we write the coordinates in $\R^n$ as $x=(x',x_n)$.
We define $\gamma = \{ (x',x_n) \in V;\; x_n \leq - \| x'\| \}$ and
$C = \{ (x',x_n) \in V;\; x_n = \pm \| x'\| \}$.  Then
$(x_0 + \gamma) \cup (x_0+\gamma^a)$ has boundary $C_{x_0} = x_0+C$, which is a
smooth hypersurface except at $x_0$.  We set $C'_{x_0} = C_{x_0} \setminus \{x_0\}$
and define $\Lambda = \ol{ T^*_{C'_{x_0}}\R^n} \cap \dT^*\R^n$. Then $\Lambda$ is a
smooth closed conic Lagrangian submanifold of $\dT^*\R^n$ with two connected
components and $S_{x_0}^\gamma$ is the component with $\xi_n >0$.

We deduce from~\eqref{eq:SSPgamF}, \eqref{eq:SSQgamF} and
Example~\ref{ex:microsupport}-(iv) that, if $F \in \Der(\cor_V)$ has compact
support, then, for $G= P_\gamma(F)$ or $G= Q_\gamma(F)$,
\begin{equation}
\label{eq:SSPgamF3}
\dot\SSi(G) \cap (T^*_xV \setminus \Int(\gamma^{\circ a}))
\subset p_2(\dot\SSi(F) \cap S_x^\gamma ) .  
\end{equation}
Since $u_\gamma(F)$ and $v_\gamma(F)$ are isomorphisms on $V \times
\Int(\gamma^{\circ a})$, we also have
\begin{equation}
\label{eq:SSPgamF2}
\SSi(G) \cap (V \times \Int(\gamma^{\circ a}))
= \SSi(F) \cap (V \times \Int(\gamma^{\circ a}))  .
\end{equation}

\begin{proposition}\label{prop:cut-off_split}
Let $F\in \Der(\cor_V)$ be such that $\supp(F)$ is compact
and let $W\subset V$ be an open subset such that
\begin{equation}\label{eq:hypSSF-split-ponct}
  S_x^\gamma \cap  \dot\SSi(F) = \emptyset  \qquad \text{for any $x\in W$.}
\end{equation}
Then $v_\gamma(F) \circ u_\gamma(F)|_W \cl P_\gamma(F)|_W \to Q_\gamma(F)|_W$ is an
isomorphism on $\dT^*W$ and we have a distinguished triangle in $\Der(\cor_W)$
\begin{equation}\label{eq:cut-off_split}
P_\gamma(F)|_W \oplus P'_\gamma(F)|_W 
 \to[\;(u_\gamma(F), u'_\gamma(F) )\;] F|_W \to L \to[+1],
\end{equation}
where $L \in \Der(\cor_W)$ is locally constant and
\begin{align*}
\dot\SSi(P_\gamma(F)|_W) &=
  \dot\SSi(F) \cap (W \times \Int(\gamma^{\circ a})) , \\
\dot\SSi(P'_\gamma(F)|_W) 
 &=  \dot\SSi(F) \cap (W \times (V \setminus \gamma^{\circ a})) .
\end{align*}
\end{proposition}
\begin{proof}
  (i)   Let $G$ be
  $P_\gamma(F)|_W$ or $Q_\gamma(F)|_W$.  By~\eqref{eq:SSPgamF3}
  and~\eqref{eq:hypSSF-split-ponct} we have
  $\dot\SSi(G) \subset W \times \Int(\gamma^{\circ a})$.  By~\eqref{eq:SSPgamF2}
  we deduce
$$
\dot\SSi(G) =  \dot\SSi(F) \cap (W \times \Int(\gamma^{\circ a})) .
$$

\sui(ii) We define $L\in \Der(\cor_W)$  by the distinguished triangle
\begin{equation}\label{eq:cut-off_split-pf1}
P_\gamma(F)|_W \to[v_\gamma(F) \circ u_\gamma(F)] Q_\gamma(F)|_W
 \to[a] L \to[+1].
\end{equation}
By~(i) we have $\dot\SSi(L) \subset W \times \Int(\gamma^{\circ a})$.   On the other hand
$v_\gamma(F) \circ u_\gamma(F)$ is an isomorphism on
$W \times \Int(\gamma^{\circ a})$ by Propositions~\ref{prop:cut-off1}
and~\ref{prop:cut-off2} and Lemma~\ref{lem:compo_isosurOmega}.  Hence
$\SSi(L) \cap (W \times \Int(\gamma^{\circ a})) = \emptyset$ and we find
$\dot\SSi(L) = \emptyset$. This proves the first assertion.

\sui(iii) The distinguished triangle of the proposition follows from the
triangles~\eqref{eq:dist_tri_gam}, \eqref{eq:cut-off_split-pf1} and
Lemma~\ref{lem:tr_dist_somme} applied with $A = P_\gamma(F)|_W$, $B= F|_W$,
$C = P'_\gamma(F)|_W$, $C' = Q_\gamma(F)|_W$, $D=L$.

Since $v_\gamma(F)$ is an isomorphism on $V \times \Int(\gamma^{\circ a})$, the
second triangle in~\eqref{eq:dist_tri_gam} gives
$\dot\SSi(P'_\gamma(F)|_W)\cap (W \times \Int(\gamma^{\circ a})) = \emptyset$.  It
follows that $\dot\SSi(P'_\gamma(F)|_W)$ and $\dot\SSi(P_\gamma(F)|_W)$ are
disjoint.   Hence~\eqref{eq:cut-off_split} implies
$$
\dot\SSi(F|_W) = \dot\SSi(P'_\gamma(F)|_W) \sqcup \dot\SSi(P_\gamma(F)|_W) .
$$
Since~\eqref{eq:hypSSF-split-ponct} implies in particular
$\dot\SSi(F) \cap (W \times \partial(\gamma^{\circ a})) = \emptyset$, the
triangular inequality for the microsupport implies the last equalities of the
proposition.
\end{proof}

\begin{lemma}
\label{lem:tr_dist_somme}
We assume to be given a morphism $a \cl A \to B$ in a triangulated category and
two distinguished triangles
\begin{equation*}
C \to[c] B \to[c'] C' \to[+1], 
\qquad
A \to[c' \circ a] C' \to D \to[+1] .
\end{equation*}
Then there exists a distinguished triangle
$$
A \oplus C \to[(a,c)] B \to D \to[+1].
$$
\end{lemma}
\begin{proof}
  We define $E$ by the distinguished triangle $A \oplus C \to[(a,c)] B \to E
  \to[+1]$. We also have the canonical triangle $C \to[u] A\oplus C \to[v] A
  \to[+1]$ where $u = (0,\id_C)$ and $v = (\id_A,0)$.  Since $(a,c) \circ u = c
  \cl C \to B$, the octahedron axiom applied to this last two triangles and the
  first one in the statement gives a distinguished triangle
$$
A \to[w] C' \to E \to[+1]
$$
such that $w \circ v = c' \circ (a,c)$. This implies $w = c' \circ a$. Hence $E
\simeq D$ and the lemma follows.
\end{proof}

\section{Local cut-off - special case}
\label{sec:localcutoff-special}

In this section we want to give local versions of Propositions~\ref{prop:cut-off1}
and~\ref{prop:cut-off_split}.  For example the conclusion $P_\gamma(F) \simeq F$
requires a bound for $\SSi(F)$ on $V$; in the same way the
hypothesis~\eqref{eq:hypSSF-split-ponct} requires a knowledge of $\SSi(F)$ on some
unbounded set.  Here we only assume that we know $\SSi(F)$ over some open subset $U$
of $V$ and we will give results similar to the conclusions of the mentioned
propositions over some smaller open subset $W \subset U$.

In fact we will apply Propositions~\ref{prop:cut-off1} and~\ref{prop:cut-off_split}
to a sheaf $F_Z$ for some locally closed subset $Z \subset U$ such that
$\ol Z \subset U$ ($F_Z$ is considered as a sheaf on $V$ by extension by $0$).  Then
$\SSi(F_Z)$ coincides with $\SSi(F)$ over $\Int(Z)$ and is empty over
$V\setminus\ol{Z}$. However, over $\ol Z \setminus \Int(Z)$, we only have the bound
$\SSi(F_Z) \subset \SSi(F) \hplus \SSi(\cor_Z)$.  We check in this section that we
can choose $Z$ so that $F_Z$ satisfies the hypotheses of
Propositions~\ref{prop:cut-off1} and~\ref{prop:cut-off_split}.  This is done
in~\cite[Prop. 6.1.4]{KS90} for a convex cone $\gamma$.  In
Proposition~\ref{prop:cut-off_split_local2} we will generalize this result to a more
general cone, using Theorem~\ref{thm:GKS} to reduce the situation to the case of a
convex cone. Since we can as well reduce to a very special cone, we only explain a
particular case of~\cite[Prop. 6.1.4]{KS90} (the proof is not different from that
of~\cite{KS90} but the exposition is easier). 

We write $V = V' \times \R$, where $V'=\R^{n-1}$. We take coordinates
$x=(x',x_n)$ on $V$ and we endow $V'$ and $V$ with the natural Euclidean
structure.  For $c>0$ we let $\gamma_c \subset V$ be the cone
\begin{equation}\label{eq:defgammac}
\gamma_c = \{ (x',x_n) \in V;\; x_n \leq - c \, \| x'\| \}.
\end{equation}

\subsubsection*{Precised cut-off}

  We begin with an analog of Propositions~\ref{prop:cut-off1},
\ref{prop:cut-off2}: we give a functor $R$ similar to $P_\gamma$ or $Q_\gamma$ which
cuts the microsupport along $\{\xi_n\leq0\}$, but locally defined (on $\Der(\cor_U)$
instead of $\Der(\cor_V)$, with $U$ open in $V$) and with a bound for $\SSi(R(F))$ up
to the boundary $\{\xi_n=0\}$.  However $R(F)$ is only defined on a smaller subset
$W \subset U$ and we don't have the ideal bound
$\SSi(R(F)) \subset \SSi(F) \cap \{\xi_n\leq0\}$, but only
$\SSi(R(F)) \subset (\SSi(F) \cap \{\xi_n\leq0\}) \cup B'$, where $B'$ is an
arbitrary neighborhood of $\SSi(F) \cap \{\xi_n=0\}$ chosen in advance.

\begin{proposition}\label{prop:microcutofflevrai0}
  Let $B\subset V^*$ be a closed conic subset and let $B'\subset V^*$ be a conic
  neighborhood of $(B\setminus\{0\}) \cap \{ \xi_n = 0\}$. Let $U\subset V$ be a
  neighborhood of $0$.  Then there exist an open neighborhood $W$ of $0$ in $U$ and a
  functor $R \cl \Der(\cor_U) \to \Der(\cor_W)$, together with a morphism of functors
  $(\cdot)|_W \to R$, such that, for any $F\in \Der(\cor_U)$,
\begin{itemize}
\item [(a)] if $\SSi(F) \subset U \times B$, then
  $$
  \SSi(R(F)) \subset W \times ((B \setminus \{ \xi_n \geq 0\}) \cup B') ,
  $$
\item [(b)] if $\SSi(F) \subset U \times \{ \xi_n \leq 0\}$, then
  $F|_W \isoto R(F)$.
\end{itemize}
\end{proposition}

Since the proof is a bit technical we explain the idea.  We first cut $F$ by some
open set $Y$ with $\ol Y \subset U$ (so that $F_Y$ can be consider as a sheaf on $V$
which is $0$ outside $\ol Y$) and apply $Q'_\gamma$ to $F_Y$.
Using~\eqref{eq:SSPgamF} we have
$\SSi(Q'_\gamma(F_Y)) \cap T^*_xV \subset p_2(\SSi(F_Y) \cap
\SSi(\cor_{x+\gamma})^a)$ and we know that
$\SSi(F_Y) \subset \SSi(F) \hplus \SSi(\cor_Y)$.  Of course this last subset could be
much bigger than $\SSi(F)$, but it turns out that, if $\SSi(\cor_Y)$ is close enough
to $\SSi(\cor_{x+\gamma})$ (at the points $x'$ where both
$\SSi(\cor_Y) \cap T^*_{x'}V$ and $\SSi(\cor_{x+\gamma}) \cap T^*_{x'}V$ are non
trivial), then $\SSi(Q'_\gamma(F_Y))$ is not too much bigger than $\SSi(F)$. This
relies on Lemma~\ref{lem:borne_cutoff_ruse} because, at the interesting points $x'$,
$\SSi(\cor_Y) \cap T^*_{x'}V$ and $\SSi(\cor_{x+\gamma}) \cap T^*_{x'}V$ are
half-lines.

The examples we choose for $Y$ and $\gamma$ with the required property are for $Y$ a
vertical cylinder
\begin{equation}\label{eq:defYrh}
    Y_r^h = \{ (x',x_n) \in V;\; \| x'\| < r, \; |x_n| < h \} 
\end{equation}
and for $\gamma$ a cone $\gamma_c$ with big slope $c$ such that $\partial\gamma_c$
meets $\partial Y_r^h$ along its vertical part, which means that $h>rc$ (as in
Fig.~\ref{fig:Ygamma}).

\begin{figure}[ht]
\begin{tikzpicture}[decoration={border, segment length=1.8mm, amplitude=2mm, angle=90}]

\coordinate (A) at (0,0);
\coordinate (B) at (-1,-2); \coordinate (C) at (1,-2);
\coordinate (D) at (.2,-.1); 
\draw [postaction={decorate,draw}] (C)--(A)--(B);
\fill [fill=black] (A) -- (D) arc[radius=.22cm, start angle=-27, end angle=-153]
-- cycle;

\coordinate (A) at (-.5,2.5); \coordinate (B) at (.5,2.5);
\coordinate (C) at (.5,-2.5); \coordinate (D) at (-.5,-2.5);
\draw  [dotted] (A) -- (B) -- (C) -- (D) -- cycle;
\draw  [decorate] (A) -- (B) -- (C) -- (D) -- cycle;
\fill [fill=black] (A) -- +(0,.2) arc[radius=2mm, start angle=90, end angle=180]
-- cycle;
\fill [fill=black] (B) -- +(0,.2) arc[radius=2mm, start angle=90, end angle=0]
-- cycle;
\fill [fill=black] (C) -- +(0,-.2) arc[radius=2mm, start angle=-90, end angle=0]
-- cycle;
\fill [fill=black] (D) -- +(0,-.2) arc[radius=2mm, start angle=-90, end angle=-180]
-- cycle;

\end{tikzpicture}
\caption{}   \label{fig:Ygamma}
\end{figure}

\begin{lemma}
  \label{lem:borne_cutoff_ruse}
  Let  
  $B,B'$ be as in Proposition~\ref{prop:microcutofflevrai0}.  Let $\sphere$ be the
  unit sphere of $V^*$. Then there exists $\varepsilon>0$ such that, for any
  $(p_1,p_2) \in (\sphere \cap \{ \xi_n = 0\}) \times \sphere$ with
  $||p_1-p_2||<\varepsilon$, we have
  $$
  (B+ \rpos\cdot (-p_1)) \cap (\rspos\cdot p_2) \subset B'.
  $$
  (In other words, if the intersection is non empty, then $p_2 \in B'$.)
\end{lemma}
This lemma is an elaboration on the following remark: for $p\in V^*\setminus B$ we
have $p \not\in (B+\rpos(-p))$ and hence
$(B+ \rpos\cdot (-p)) \cap (\rspos\cdot p) = \emptyset$.  It follows that, for any
$p \in (\sphere \cap \{ \xi_n = 0\})$, we have
$(B+ \rpos\cdot (-p)) \cap (\rpos\cdot p) \subset (B \cap \{ \xi_n = 0\})$.  In the
lemma $(p_1,p_2)$ is near the diagonal and $B \cap \{ \xi_n = 0\}$ is replaced by its
neighborhood $B'$.
\begin{proof}
  For a subset $C$ of the sphere $\sphere$, we set
  $C_{\sphere,\varepsilon} = \{v\in \sphere$; there exists $v'\in C$,
  $||v-v'||<\varepsilon\}$.  For a conic subset $D$ of $V^*$ we set
  $D_\varepsilon = \rspos\cdot ( D\cap \sphere)_{\sphere,\varepsilon}$.  We remark
  that $D_\varepsilon = \{av_1+bv_2$; $v_1 \in D\cap \sphere$, $v_2\in \sphere$,
  $||v_1-v_2||<\varepsilon$, $a,b\geq 0\} \setminus \{0\}$.
 
 A point $p' \in (B+ \rspos\cdot (-p_1)) \cap (\rspos\cdot p_2)$ is written
 $p' = p - ap_1 = bp_2$, with $p\in B$, $a,b\geq0$.  Then $p = ap_1+bp_2$ belongs to
 $B \cap \{\xi_n=0\}_\varepsilon$.  The point $q = \frac1{||p||}p$ belongs to the arc
 $p_1p_2$ on $\sphere$, hence $||q-p_2||<\varepsilon$ and then
 $p_2 \in (B \cap \{\xi_n=0\}_\varepsilon)_\varepsilon$.  Now, for $\varepsilon$
 small enough this last set is contained in $B'$.
\end{proof}

\begin{proof}[Proof of Proposition~\ref{prop:microcutofflevrai0}]
  (i) We will use the functors $P_\gamma$, $Q'_\gamma$ for a cone $\gamma = \gamma_c$
  to be chosen in~(iii) and the distinguished triangle~\eqref{eq:morph_func_gam}
  $P_\gamma(F) \to F \to Q'_\gamma(F) \to[+1]$.
  We have the bound, with the same proof as~\eqref{eq:SSPgamF},
  $$
\SSi(Q'_\gamma(F)) \cap T^*_xV \subset p_2(\SSi(F) \cap S'^\gamma_x) ,
$$
where $S'^\gamma_x = (\SSi(\cor_{(x+\gamma) \setminus \{x\}}))^a$ and
$p_2 \colon T^*V \to T^*_xV \simeq V^*$ is the projection.

\sui(ii) We prove that
$S'^\gamma_x \subset V\times(V^* \setminus \Int(\gamma^{\circ a}))$.    More precisely $S'^\gamma_x = (\SSi(\cor_{x+\gamma}))^a$
outside $T^*_xV$ (which is obvious from the definition of $S'^\gamma_x$) and
$S'^\gamma_x \cap T^*_xV = V^* \setminus \Int(\gamma^{\circ a})$. Indeed this can be
computed directly from the definition of the microsupport, or we can
use~\cite[Lem.~3.7.10, Thm.~5.5.5]{KS90} as follows. The Fourier-Sato transform of a
conic sheaf $G \in \Der(\cor_V)$ is $G^\wedge \in \Der(\cor_{V^*})$ and, identifying
$T^*V = V\times V^* =T^*V^*$, we have $\SSi(G^\wedge) = \SSi(G)$.  In particular
$\SSi(G) \cap T^*_0V = \supp(G^\wedge)$.  We also have
$(\cor_\lambda)^\wedge \simeq \cor_{\Int(\lambda^\circ)}$ for any proper closed
convex cone $\lambda$, in particular $\lambda=\gamma$ and $\lambda=\{0\}$. Hence the
exact sequence
$0 \to \cor_{\gamma\setminus\{0\}} \to \cor_\gamma \to \cor_{\{0\}} \to 0$ gives
$(\cor_{\gamma\setminus\{0\}})^\wedge \simeq \cor_{V^* \setminus \Int(\gamma^\circ)}$
and the result follows.

\sui(iii) We will choose $c,r,h$ such that $h>rc$ as explained
after~\eqref{eq:defYrh}. Hence $\partial\gamma_c$ meets $\partial Y_r^h$ along its
vertical part and the intersection, say $I$, is the sphere
$\{ (x',x_n) \in V;\; \| x'\| = r, \; x_n = -cr \}$.  Let $y\in I$,
$p_1, p_2 \in T^*_yV$ be the unit length vectors in $T^*_yV \cap \SSi(\cor_{Y_r^h})$,
$T^*_yV \cap S'^{\gamma_c}_0$.  Then $||p_1-p_2|| \leq c^{-1}$.  Let $\varepsilon>0$
such that the conclusion of Lemma~\ref{lem:borne_cutoff_ruse} holds.  We choose
$c> 2 \varepsilon^{-1}$ and $h,r$ small enough so that $Y_r^h \subset U$.  We define
$R_0(F) = Q'_{\gamma_c}(F_{Y_r^h})$. We have a natural morphism $F \to R_0(F)$.

Let $W \subset Y_r^h$ be the set of $x$ such that $\partial(x+\gamma_c)$ meets
$\partial Y_r^h$ along its vertical part and, for any
$y\in \partial(x+\gamma_c) \cap \partial Y_r^h$, the unit length vectors
$p_1 \in T^*_yV \cap \SSi(\cor_{Y_r^h})$, $p_2 \in T^*_yV \cap S'^{\gamma_c}_x$ satisfy
$||p_1-p_2|| < \varepsilon$.  The set $W$ is open and contains $0$ by construction.

By Theorem~\ref{thm:SSrhom} we have
$\SSi(F_{Y_r^h}) \subset (V \times B) \hplus \SSi(\cor_{Y_r^h})$.  We can use $+$
instead of $\hplus$ in the last expression because the term $V\times B$ is a
product. By~(i) and~(ii) we have
\begin{align*}
  \SSi(R_0(F)) \cap T^*_xV
  & \subset p_2(((V \times B) + \SSi(\cor_{Y_r^h})) \cap S'^{\gamma_c}_x) \\
  & \subset p_2((V \times B)  \cap (V\times (V^* \setminus \Int(\gamma_c^{\circ a}))) \\
  & \qquad \cup p_2(((V \times B) + \dot\SSi(\cor_{Y_r^h})) \cap S'^{\gamma_c}_x) .
\end{align*}
By Lemma~\ref{lem:borne_cutoff_ruse} and by our definition of $W$ we have
$p_2(((V \times B) + \dot\SSi(\cor_{Y_r^h})) \cap S'^{\gamma_c}_x) \subset B'$ if
$x\in W$.  Hence
$ \SSi(R_0(F)) \cap T^*_xV \subset (B \cap (V^* \setminus \Int(\gamma_c^{\circ a}) )
\cup B'$.  Since
$V^* \setminus \Int(\gamma_c^{\circ a}) = \{ \xi_n \geq c^{-1}||\xi'||\}$, we have
$B \cap (V^* \setminus \Int(\gamma_c^{\circ a})) \subset ((B \setminus \{ \xi_n \geq
0\}) \cup B')$ for $c$ big enough.  Now we set $R(F) = R_0(F)|_W$ and we obtain~(a).

\sui(iv) We prove that $P_{\gamma_c}(F_{Y_r^h})|_W \simeq 0$, which implies~(b) by
the triangle given in~(i).  By~\eqref{eq:germ_compo} we have
$(P_{\gamma_c}(F_{Y_r^h}))_y \simeq \rsect(V; F_{E_y})$ where
$E_y = Y_r^h \cap (y+\gamma_c)$.  If we restrict to the open subset $\{x_n<h\}$, the
microsupports of $F$, $\cor_{Y_r^h}$ and $\cor_{y+\gamma_c}$, for $y\in V$, are all
contained in $\{ \xi_n \leq 0\}$.  For $y\in W$ we have $\ol{E_y} \subset \{x_n<h\}$,
hence $\SSi(F_{E_y}) \subset \{ \xi_n \leq 0\}$ by Theorem~\ref{thm:SSrhom}.  Since
$\supp(F_{E_y}) \subset \ol{Y_r^h}$ is compact, we deduce by
Corollary~\ref{cor:Morse} that $\rsect(V; F_{E_y}) \simeq 0$ (taking $\phi(x) = x_n$
in the corollary we obtain
$\rsect(\phi^{-1}(\mo]-\infty,2h\mc[); F_{E_y}) \isoto
\rsect(\phi^{-1}(\mo]-\infty,-2h\mc[); F_{E_y}) \simeq0$). Hence
$P_{\gamma_c}(F_{Y_r^h})|_W \simeq 0$ as claimed.
\end{proof}

\subsubsection*{Local splitting}

Recall the cone $\gamma_c$ of~\eqref{eq:defgammac} and the subset
$S_x^{\gamma_c} \subset T^*V$ of~\eqref{eq:def_Sgammax}. We have
$\pi_V(S_0^{\gamma_c}) = \{(x',x_n;\xi',\xi_n); \; |x_n| = c \, ||x'||\}$ and, for a non
zero $y = (y',y_n) \in \pi_V(S_0^{\gamma_c})$ we have
\begin{equation}\label{eq:fibreS0gamma}
S_0^{\gamma_c} \cap T^*_yV
= \{(\lambda \tilde y', -c^{-1} \lambda \tilde y_n);\; \lambda y_n >0\} ,
\end{equation}
where $(\tilde y', \tilde y_n) \in V^*$ corresponds to $y$ through the identification
$V' \simeq (V')^*$ given by the Euclidean product.  We will use a cylinder similar to
$Y_r^h$
\begin{equation}\label{eq:defZrh}
  \begin{split}
    Z_r^h &=  \{ (x',x_n) \in V;\; \| x'\| \leq r, \;  -h \leq x_n < 0 \} \\
    & \qquad  \sqcup \{ (x',x_n) \in V;\; \| x'\| < r, \;  0 \leq x_n < h \} 
  \end{split}
\end{equation}
which is locally closed (and not open) and such that $\SSi(\cor_{Z_r^h})$ is close to
$(S_0^{\gamma_c})^a$ along $\partial Z_r^h \cap \pi_V(S_0^{\gamma_c})$ for $c$
big.   Indeed over the vertical
boundary of $Z_r^h$, that is, for $y = (y',y_n)$ with $\| y'\| = r$ and
$0< |y_n| < h$ we have, by Example~\ref{ex:microsupport}-(iii),
\begin{equation}\label{eq:borne_SSZrh}
  \SSi(\cor_{Z_r^h}) \cap T_y^*V
= \{(\lambda \tilde y', 0);\; \lambda y_n >0\} ,
\end{equation}
where $\tilde y' \in (V')^*$ is as in the description of $S_0^{\gamma_c}$ (in
Fig.~\ref{fig:Zgamma} we have pictured $\SSi(\cor_{Z_r^h})$ and $S_0^{\gamma_c}$).

\begin{figure}[ht]
\begin{tikzpicture}[decoration={border, segment length=1.8mm, amplitude=2mm, angle=90}]

\draw [postaction={decorate,draw}] (-1,-2) -- (1,2);
\draw [postaction={decorate,draw}] (-1,2) -- (1,-2);

\coordinate (A) at (-.5,2.5); \coordinate (B) at (.5,2.5);
\coordinate (C) at (.5,-2.5); \coordinate (D) at (-.5,-2.5);
\draw  [decorate] (-.5,0) -- (A) -- (B) -- (.5,0);
\draw  [dotted] (-.5,0) -- (A) -- (B) -- (.5,0);
\draw  [postaction={decorate,draw}] (-.5,0) -- (D) -- (C) -- (.5,0);
\fill [fill=black] (A) -- +(0,.2) arc[radius=2mm, start angle=90, end angle=180]
-- cycle;
\fill [fill=black] (B) -- +(0,.2) arc[radius=2mm, start angle=90, end angle=0]
-- cycle;
\fill [fill=black] (C) -- +(0,.2) arc[radius=2mm, start angle=90, end angle=180]
-- cycle;
\fill [fill=black] (D) -- +(.2,0) arc[radius=2mm, start angle=0, end angle=90]
-- cycle;
\fill [fill=black] (-.5,0) -- +(.2,0) arc[radius=2mm, start angle=0, end angle=180]
-- cycle;
\fill [fill=black] (.5,0) -- +(.2,0) arc[radius=2mm, start angle=0, end angle=180]
-- cycle;

\end{tikzpicture}
\caption{}   \label{fig:Zgamma}
\end{figure}

\begin{proposition}
\label{prop:cut-off_split_local0}
Let $U \subset V$ be an open subset containing $0$ and let $c>0$ be given.  Then
there exist an open neighborhood $W$ of $0$ in $U$ and two functors
$P, P' \cl \Der(\cor_U) \to \Der(\cor_W)$ together with morphisms of functors
$u \cl P \to (\cdot)|_W$, $u' \cl P' \to (\cdot)|_W$, such that, for any
$F\in \Der(\cor_U)$ satisfying
\begin{equation}\label{eq:hypSSF_borne_cutoff0}
\SSi(F) \subset U\times (\gamma_{c}^{\circ} \cup \gamma_{c}^{\circ a} ) ,
\end{equation}
we have
\begin{align*}
\dot\SSi(P(F)) &= \dot\SSi(F|_W) \cap (W \times \gamma_c^\circ ), \\
\dot\SSi(P'(F)) &= \dot\SSi(F|_W) \cap (W \times \gamma_c^{\circ a} )
\end{align*}
and the object $L \in  \Der(\cor_W)$ given by the distinguished triangle
$$
P(F) \oplus P'(F) \to[\,(u(F), \, u'(F))\,] F|_W \to L \to[+1]
$$
satisfies $\dot\SSi(L) = \emptyset$.
\end{proposition}
\begin{proof}
  (i) As in the proof of Proposition~\ref{prop:microcutofflevrai0} we assume to be
  given $d,r,h>0$ such that $h>rd$. Hence $\pi_V(S_0^{\gamma_d})$ meets
  $\partial Z_r^h$ along its vertical part, where $Z_r^h$ is defined
  in~\eqref{eq:defZrh}, and the intersection is the union of the two spheres
  $\{ (x',x_n) \in V;\; \| x'\| = r, \; x_n = \pm dr \}$.  Along these two spheres
  $S_0^{\gamma_d}$ and $\SSi(\cor_{Z_r^h})$ are the two half lines described
  in~\eqref{eq:fibreS0gamma} and~\eqref{eq:borne_SSZrh} (so they make an angle
  $\arctan(d^{-1})$).  Using Lemma~\ref{lem:borne_cutoff_ruse} with  
\begin{equation}
  \label{eq:defgammacopopa}
B=  \gamma_{c}^{\circ} \cup \gamma_{c}^{\circ a}
  = \{ (\xi',\xi_n); \; c\, |\xi_n| \leq \| \xi'\| \}.
\end{equation}
and $B' = \emptyset$ we obtain
\begin{equation}
  \label{eq:S0partialgammac}
S_0^{\gamma_d} \cap (\SSi(\cor_{Z_r^h}) + V\times (\gamma_{c}^{\circ}
\cup \gamma_{c}^{\circ a} )) = \emptyset 
\end{equation}
for $d$ big enough, whatever $r,h$.  We fix such a $d$ and then choose $r,h$ such
that $h>rd$ (as already assumed) and small enough so that $\ol{Z_r^h} \subset U$.

\sui(ii) We define $W$ as the set of $x \in \Int(Z_r^h)$ such that
$\pi_V(S_x^{\gamma_d})$ meets $\partial Z_r^h$ along its vertical part
and~\eqref{eq:S0partialgammac} holds with $S_0^{\gamma_d}$ replaced by
$S_x^{\gamma_d}$.  Hence $W$ is an open neighborhood of $0$.

For $F\in \Der(\cor_U)$ we can extend $F\tens \cor_{Z_r^h}$ by $0$ as an object of
$\Der(\cor_V)$.  We define $P, P'$ by $P(F) = P_{\gamma_{d}}(F\tens \cor_{Z_r^h})|_W$
and $P'(F) = P'_{\gamma_{d}}(F\tens \cor_{Z_r^h})|_W$.  The functors $u,u'$ are
induced by $u_{\gamma_{d}}$, $u'_{\gamma_{d}}$ (we remark that $W\subset Z_r^h$,
hence $(F\tens \cor_{Z_r^h})|_W = F|_W$).  If $F$
satisfies~\eqref{eq:hypSSF_borne_cutoff0}, then
\begin{equation}\label{eq:borne_SS_cutoff2}
  \SSi(F\tens \cor_{Z_r^h}) \subset
  (\SSi(\cor_{Z_r^h}) + V\times (\gamma_{c}^{\circ} \cup \gamma_{c}^{\circ a} ))
\end{equation}
by Theorem~\ref{thm:SSrhom}.  Hence $F\tens \cor_{Z_r^h}$ satisfies the
hypothesis~\eqref{eq:hypSSF-split-ponct} of Proposition~\ref{prop:cut-off_split}
by~\eqref{eq:S0partialgammac} (with $S_0^{\gamma_{d}}$ replaced by
$S_x^{\gamma_{d}}$).  Now the result follows from
Proposition~\ref{prop:cut-off_split}
\end{proof}

\section{Local cut-off - general case}

Here we extend the results of \S\ref{sec:localcutoff-special} replacing the cones
$\{\xi_n\leq0\}$ or $\gamma_c^\circ$ by more general ones.  We deduce
Propositions~\ref{prop:microcutofflevrai1}, \ref{prop:cut-off_split_local2} from
Propositions~\ref{prop:microcutofflevrai0}, \ref{prop:cut-off_split_local0}.  The
process of reduction to the results of \S\ref{sec:localcutoff-special} is the same
for both propositions and we only prove the second one.  (Moreover
Proposition~\ref{prop:microcutofflevrai1} is only a precised version
of~\cite[Prop.~6.1.4]{KS90} but Proposition~\ref{prop:cut-off_split_local2} is new.)

\begin{proposition}\label{prop:microcutofflevrai1}
  Let $U$ be an open subset of $V = \R^n$ and let $A \subset V^* \setminus\{0\}$ be
  an open cone.  We assume that there exists a homotopy
  $\psi\colon (V^* \setminus\{0\}) \times [0,1] \to V^* \setminus\{0\}$ of class
  $C^1$ and homogeneous of degree $1$ such that $\psi_1(A) = \{\xi_n>0\}$.  Let
  $B\subset V^*$ be a closed conic subset and $B'\subset V^*$ be a conic neighborhood
  of $B \cap \partial A$.  Let $x_0\in U$ be given.  Then there exist a neighborhood
  $W$ of $x_0$ and a functor $R \cl \Der(\cor_U) \to \Der(\cor_W)$, of the form
  $R(F) = K \circ F$ for some $K\in \Der(\cor_{W\times U})$, together with a morphism
  of functors $(\cdot)|_W \to R$, such that, for any $F\in \Der(\cor_U)$,
\begin{itemize}
\item [(i)] if $\SSi(F) \subset U \times B$, then
  $$
  \SSi(R(F)) \subset W \times ((B \setminus A) \cup B') ,
  $$
\item [(ii)] if $\dot\SSi(F) \cap (U \times A) = \emptyset$, then
  $F|_W \isoto R(F)$.
\end{itemize}
\end{proposition}

\begin{proposition}
\label{prop:cut-off_split_local2}
Let $U$ be an open subset of $V = \R^n$ and let $A$, $A' \subset V^* \setminus\{0\}$
    be two
disjoint closed conic subsets.  We assume that there exists a homotopy
$\psi\colon (V^* \setminus\{0\}) \times [0,1] \to V^* \setminus\{0\}$ of class $C^1$
and homogeneous of degree $1$ such that $\psi_1(A)$ and $\psi_1(A')$ are separated by
some hyperplane of $V^*$.  Let $x_0\in U$ be given.  Then there exist a neighborhood
$W$ of $x_0$ in $U$ and two functors $P, P' \cl \Der(\cor_U) \to \Der(\cor_W)$
together with morphisms of functors $u \cl P \to (\cdot)|_W$,
$u' \cl P' \to (\cdot)|_W$ such that, for any $F\in \Der(\cor_U)$ satisfying
\begin{equation}\label{eq:hypSSF_borne_cutoff02}
\dot\SSi(F) \subset U\times (A \sqcup A') ,
\end{equation}
we have
\begin{align*}
\dot\SSi(P(F)) &= \dot\SSi(F|_W) \cap (W \times A ), \\
\dot\SSi(P'(F)) &= \dot\SSi(F|_W) \cap (W \times A' )
\end{align*}
and the object $L \in  \Der(\cor_W)$ given by the distinguished triangle
$$
P(F) \oplus P'(F) \to[(u(F), u'(F))] F|_W \to L \to[+1]
$$
satisfies $\dot\SSi(L) = \emptyset$.
\end{proposition}
\begin{proof}
  (i) Let $c>0$ be given.  We recall the notation $\gamma_c$ of~\eqref{eq:defgammac}.
  Up to a linear change of coordinates in $\R^n$ we can assume that $x_0 = 0$, that
  the homotopy $\psi$ in the hypothesis is defined on some open interval $I$
  containing $[0,1]$ and that $\psi_1(A) \subset \gamma_c^\circ$,
  $\psi_1(A') \subset \gamma_c^{\circ a}$.  We first extend $\psi$ to a homogeneous
  Hamiltonian isotopy $\Phi \cl \dT^*U \times I \to \dT^*U$ such that
  $\Phi_t(\dT^*_0U) = \dT^*_0U$ for all $t \in I$ and
$$
\Phi_1( \{0\} \times A) \subset \{0\} \times \gamma_c^\circ ,  
\qquad
\Phi_1( \{0\} \times A') \subset \{0\} \times \gamma_c^{\circ a} .
$$
To see that such a $\Phi$ exists we choose local coordinates $(x;\xi)$ around $0$ and
write
$\frac{\partial \psi}{\partial t}(\xi,t) = \sum_{i=1}^n a_i(\xi,t) \partial_{\xi_i}$
(for $(\xi,t) \in V^*\times I$). Then we choose a Hamiltonian function
$h \cl \dT^*U \times I \to \R$ homogeneous of degree $1$ in $\xi$ and with support in
$C \times V^*$ for some compact subset $C$ of $U$ such that, near $0$, we have
$h(x;\xi) = - \sum_{i=1}^n a_i(\xi,t) x_i$.  Then the Hamiltonian flow $\Phi$ of $h$
satisfies the above relations.

\newcommand{\Rphi}{R_\Phi}

We let $\Rphi \cl \Der(\cor_U) \to \Der(\cor_U)$,
$F \mapsto K_{\Phi,1} \circ F$, be the equivalence of categories given by
Corollary~\ref{cor:isot_equivcat}.  We have in particular
$\dot\SSi(\Rphi(F)) = \Phi_1(\dot\SSi(F))$ for all $F \in \Der(\cor_U)$.

\sui (ii) We can find a neighborhood $U_1$ of $0$ such that
$$
\Phi_1( U \times A) \cap \dT^*U_1 \subset U_1 \times \gamma_{2c}^\circ ,
\quad
\Phi_1( U \times A') \cap \dT^*U_1 \subset U_1 \times \gamma_{2c}^{\circ a} .
$$
Applying Proposition~\ref{prop:cut-off_split_local0} (with $U_1$ and $2c$ instead of
$U$ and $c$) we find a neighborhood of $0$, say $W_1$, and functors
$P_1, P'_1 \cl \Der(\cor_{U_1}) \to \Der(\cor_{W_1})$ together with morphisms of
functors $u_1 \cl P_1 \to (\cdot)|_{W_1}$, $u'_1 \cl P'_1 \to (\cdot)|_{W_1}$
satisfying the conclusion of Proposition~\ref{prop:cut-off_split_local0}.

\sui (iii) We let $W$ be an open neighborhood of $0$ such that
$\Phi_t(\dT^*W) \subset \dT^*W_1$ for all $t\in [0,1]$.  Let $j$ denote the
inclusion of $W_1$ in $U$. We define the functor $P$ by
$$
P(F) = (\Rphi^{-1}\eim{j} P_1(\Rphi(F)|_{U_1}) )|_W
$$
and $P'$ from $P'_1$ by the same formula.  We let $u$, $u'$ be the morphisms of
functors induced by $u_1$, $u_1'$.  By
Proposition~\ref{prop:cut-off_split_local0} we have a distinguished triangle on
$W_1$
$$
P_1(\Rphi(F)|_{U_1}) \oplus P'_1(\Rphi(F)|_{U_1}) \to \Rphi(F)|_{W_1}
\to L \to[+1],
$$
where $L$ is locally constant on $W_1$.  Since $\Phi_1(\dT^*W) \subset \dT^*W_1$
we see that $\Rphi^{-1}(\eim{j}(L))|_W$ is locally constant.  Applying
$(\Rphi^{-1}(\eim{j}(-)))|_W $ to the distinguished triangle, we find that
$G = (\Rphi^{-1}((\Rphi(F))_{W_1}))|_W$ is isomorphic to $P(F) \oplus P'(F)$ up
to a locally constant sheaf on $W$.

It only remains to check that $G$ is isomorphic to $F|_W$.
Lemma~\ref{lem:restr_isotopie} applied with $\Psi_t = \Phi_t^{-1}$, $U = W_1$ and
$V=W$ gives $\Rphi^{-1}(H_{W_1})|_W \simeq \Rphi^{-1}(H)|_W$ for any $H$.  Hence
$G \simeq (\Rphi^{-1}(\Rphi(F)))|_W \simeq F|_W$, as required.
\end{proof}

\section{Cut-off and $\gamma$-topology}
\label{sec:gamma_top}

In~\cite{KS90} the Proposition~\ref{prop:cut-off1} has another formulation in terms
of the {\em $\gamma$-topology}.  It gives directly a decomposition of the projector
$P_\gamma$ as in Remark~\ref{rem:proj}-(1).  As in the previous paragraphs, let
$\gamma \subset V$ be a closed convex cone. The $\gamma$-topology is defined
in~\cite{KS90} as follows.  We say that an open subset $\Omega$ of $V$ is
$\gamma$-stable if $x+y \in \Omega$ for all $(x,y) \in \Omega \times \gamma$.  The
$\gamma$-stable open subsets define a topology on $V$ and we denote by $V_\gamma$
this topological space.  The identity map induces a continuous map
$\phi_\gamma \colon V \to V_\gamma$.    Then Propositions~3.5.3, 3.5.4 and~5.2.3 of~\cite{KS90} give

\begin{proposition}\label{prop:cut-offgamma}
  {\rm (i)} For any $G \in \Der(\cor_{V_\gamma})$ the adjunction morphism
  $G \to \roim{\phi_\gamma} \opb{\phi_\gamma} G$ is an isomorphism.

  \suiv{(ii)} For $F \in \Der(\cor_V)$ the adjunction morphism
  $\opb{\phi_\gamma} \roim{\phi_\gamma} F \to F$ is an isomorphism if and only
  if $\SSi(F) \subset V\times \gamma^{\circ a}$.

  In particular for any $G \in \Der(\cor_{V_\gamma})$ we have
  $\SSi(\opb{\phi_\gamma} G) \subset V\times \gamma^{\circ a}$.

  \suiv{(iii)} There exists an isomorphism of functors
  $P_\gamma \simeq \opb{\phi_\gamma} \circ \roim{\phi_\gamma}$ and the morphism
  $u_\gamma$ of~\eqref{eq:morph_func_gam} corresponds to the adjunction morphism.
\end{proposition}
In other words, using the notation $\Der_{V\times\gamma^{\circ a}}(\cor_V)$ of
Notation~\ref{not:cat_micsupp_fixe}, the functors $\opb{\phi_\gamma}$ and
$\roim{\phi_\gamma}$ give mutually inverse equivalences of categories between
$\Der(\cor_{V_\gamma})$ and $\Der_{V\times\gamma^{\circ a}}(\cor_V)$.

In Proposition~\ref{prop:cut-offgamma} the conditions on the microsupport are
global on $V$. However we can also consider local situations using the following
lemma.

\begin{lemma}\label{lem:local_cond_SSF}
  Let $U$ be an open subset of $V$ and $F \in \Der(\cor_U)$.  We assume that
  $\SSi(F) \subset U \times \gamma^{\circ a}$.  Let $x_0 \in U$. Then there
  exist a neighborhood $U'$ of $x_0$ in $U$ and $G\in \Der(\cor_V)$ such that
  $G|_{U'} \simeq F|_{U'}$ and $\SSi(G) \subset V \times \gamma^{\circ a}$.
\end{lemma}
\begin{proof}
  For $x \in V$, $\xi \in \gamma^{\circ a}$ and $t>0$ we define the truncated cone
  $C = C_{x,\xi,t} = (x-\gamma) \cap \{y$; $\langle \xi, y-x \rangle <t\}$ with
  vertex at $x$.  We have $\SSi(\cor_C) \subset V \times \gamma^{\circ a}$ by
  Example~\ref{ex:microsupport} and Theorem~\ref{thm:SSrhom}.

  We can find $x$, $\xi$, $t$ such that $x_0 \in \Int(C)$ and
  $\ol{C} \subset U$.  By Theorem~\ref{thm:SSrhom} the sheaf $G = F_C$ satisfies
  $\SSi(G) \subset U \times \gamma^{\circ a}$.  Since $G$ has a compact support
  we can extend it by $0$ on $V$ and we still have
  $\SSi(G) \subset V \times \gamma^{\circ a}$. Setting $U' = \Int(C)$, we also
  have $G|_{U'} \simeq F|_{U'}$.
\end{proof}

We thank Pierre Schapira for the following useful result.

\begin{corollary}\label{cor:SSHiF}
  {\rm (i)} Let $U$ be an open subset of $V$ and $F \in \Der(\cor_U)$.  We
  assume that $\SSi(F) \subset U \times \gamma^{\circ a}$. Then
  $\SSi(H^iF) \subset U \times \gamma^{\circ a}$ for all $i\in \Z$.

  \suiv{(ii)} Let $M$ be a manifold and $F \in \Der(\cor_M)$.  Let
  $C \subset T^*M$ be the convex hull of $\SSi(F)$ in the sense that $C$ is the
  intersection of all closed conic subsets $S$ of $T^*M$ which contain $\SSi(F)$
  and are fiberwise convex ($S\cap T^*_xM$ is convex for any $x\in M$).  We
  assume that $C$ does not contain any line.  Then $\SSi(H^iF) \subset C$ for
  all $i\in \Z$.
\end{corollary}
\begin{proof}
  (i) The statement is local on $U$. Let $x_0 \in U$ be given and let $U'$, $G$
  be given by Lemma~\ref{lem:local_cond_SSF}.  By
  Proposition~\ref{prop:cut-offgamma} there exists
  $G' \in \Der(\cor_{V_\gamma})$ such that $G \simeq \opb{\phi_\gamma} G'$
  ($G' = \roim{\phi_\gamma} G$).  Since $\opb{\phi_\gamma}$ is exact we have
  $H^iG \simeq \opb{\phi_\gamma} H^iG'$. By Proposition~\ref{prop:cut-offgamma}
  again we deduce $\SSi(H^iG) \subset V \times \gamma^{\circ a}$. Since
  $G|_{U'} \simeq F|_{U'}$, this gives the required bound for $\SSi(H^iF)$ near
  $x_0$.

  \sui(ii) The statement is local on $M$ and we can assume that $M$ is open in a
  vector space $V$.  Then $T^*M = M \times V^*$. For a given $x\in M$ we can
  find a closed convex cone $\delta$ in $V^*$, contained in an arbitrarily small
  neighborhood of $C \cap (\{x\} \times V^*)$, and a neighborhood $U$ of $x$
  such that $\SSi(F|_U) \subset U \times \delta$.  Then the result follows
  from~(i).
\end{proof}

\begin{remark}\label{rem:supportgammatop}
  1) Let $U$ be an open subset of $V$ and $F \in \Mod(\cor_U)$.  We assume that
  $\SSi(F) \subset U \times \gamma^{\circ a}$. Then, for any $s\in F(U)$,
  $\supp(s)$ is ``locally $\gamma$-closed'', that is, for any $x_0 \in U$, there
  exists a neighborhood $B$ of $x_0$ (for the usual topology) such that
  $B\cap \supp(s) = B \cap Z$ for some $Z\subset V$ which is closed for the
  $\gamma$-topology.

  Indeed, we consider $U'$, $G$ given by Lemma~\ref{lem:local_cond_SSF}.  Then
  $s|_{U'}$ is identified with a section of $G$.  By
  Proposition~\ref{prop:cut-offgamma} there exists
  $G' \in \Mod(\cor_{V_\gamma})$ such that $G \simeq \opb{\phi_\gamma} G'$.
  Since $G_{x_0} \simeq G'_{x_0}$, there exist a section $s'$ of $G'$ (over
  some $\gamma$-open set containing $x_0$) and a neighborhood $U''$ of $x_0$ in
  $V$ such that $s|_{U''}$ is the section of $G$ induced by $s'$. In particular
  $s_x \simeq s'_x$ for all $x\in U''$ and we get
  $\supp(s)\cap U'' = \supp(s')\cap U''$. Since $\supp(s')$ is closed in
  $V_\gamma$ we are done.
  
  \sui 2) If a subset $Z$ of $V$ is locally $\gamma$-closed, then, for any
  $x_0\in Z$, there exists a neighborhood $B$ of $0$ in $V$ such that
  $x_0-y \in Z$ for all $y \in B\cap \gamma$.
\end{remark}

\section{Remarks on projectors - Tamarkin projector}
\label{sec:Tamproj}

We have seen in Remark~\ref{rem:proj} that the functors
$P_\gamma, Q_\gamma \colon \Der(\cor_V) \to \Der(\cor_V)$ are projectors to the
subcategory $\Der_{V \times \gamma^{\circ a}}(\cor_V)$ of $\Der(\cor_V)$ and induce
adjoints to the embedding
$\iota \colon \Der_{V \times \gamma^{\circ a}}(\cor_V) \to \Der(\cor_V)$; namely
$P_\gamma$ induces the right adjoint to $\iota$ and $Q_\gamma$ the left adjoint.  In
this paragraph we give a similar interpretation to $P'_\gamma$, $Q'_\gamma$ and give
a relative version of $P'_\gamma$.

It turns out that our projectors come in pairs.  Using~\cite[Ex. 10.15]{KS06}
or~\cite[Prop. 4.21]{GS11}, we can deduce from the first triangle
in~\eqref{eq:dist_tri_gam} that $Q'_\gamma$ is a projector to the right orthogonal of
$\Der_{V \times \gamma^{\circ a}}(\cor_V)$:
\begin{equation}
  \label{eq:def_right_orthog}
\begin{split}
\Der_{V \times \gamma^{\circ a}}^{\perp, r}(\cor_V)
= \{F \in \Der(\cor_V); &\; \Hom(G,F) \simeq 0  \\
&\text{ for all $G\in \Der_{V \times \gamma^{\circ a}}(\cor_V)$} \} .
\end{split}
\end{equation}
Indeed, applying $P_\gamma$ to~\eqref{eq:dist_tri_gam} and using
$P_\gamma \circ P_\gamma \simeq P_\gamma$ we obtain
$P_\gamma \circ Q'_\gamma \simeq 0$. Then, putting $F = Q'_\gamma(G)$
in~\eqref{eq:dist_tri_gam} we have $Q'_\gamma \isoto Q'_\gamma \circ Q'_\gamma$.
This proves that $Q'_\gamma$ is a projector. Using
$\Hom(P_\gamma(F), P_\gamma(G)) \isoto \Hom(P_\gamma(F), G)$, (see
Remark~\ref{rem:proj}) we deduce $\Hom(P_\gamma(F), Q'_\gamma(G)) \simeq 0$ and we
can deduce that the image of $Q'_\gamma$ is indeed
$\Der_{V \times \gamma^{\circ a}}^{\perp, r}(\cor_V)$. More precisely $Q'_\gamma$
induces a left adjoint to the embedding of
$\Der_{V \times \gamma^{\circ a}}^{\perp, r}(\cor_V)$.

In the same way, we can see that $P'_\gamma$ is also a projector, with image the
left orthogonal of $\Der_{V \times \gamma^{\circ a}}(\cor_V)$:
\begin{equation}
  \label{eq:def_left_orthog}
\begin{split}
\Der_{V \times \gamma^{\circ a}}^{\perp, l}(\cor_V)
= \{F \in \Der(\cor_V); &\; \Hom(F,G) \simeq 0  \\
&\text{ for all $G\in \Der_{V \times \gamma^{\circ a}}(\cor_V)$} \} 
\end{split}
\end{equation}
and that it induces a right adjoint to the embedding of
$\Der_{V \times \gamma^{\circ a}}^{\perp, l}(\cor_V)$.

The relation between $P'_\gamma$, $Q'_\gamma$ is not the same as the relation between
$P_\gamma$, $Q_\gamma$.  In Remark~\ref{rem:proj} we deduced that $Q_\gamma$ was a
projector from the fact that its was adjoint to $P_\gamma$.  To deduce that they have
the same image we used the more precise fact: $P_\gamma$ induces the {\em right}
adjoint to $\iota$ and $P_\gamma$ is the {\em right} adjoint to $Q_\gamma$.  Now
$P'_\gamma$ induces the right adjoint to the embedding of its image but $P'_\gamma$
is left adjoint to $Q'_\gamma$; so we cannot conclude that $P'_\gamma$ and
$Q'_\gamma$ have the same image.

The fact that the embedding $\iota$ has both a left and a right adjoint comes from
the property that $\Der_{V \times \gamma^{\circ a}}(\cor_V)$ has arbitrary small sums
and products and $\iota$ commutes with sums and products (this is the Brown
representability theorem -- see for example~\cite[\S14.2]{KS06}).  However
$\Der_{V \times \gamma^{\circ a}}^{\perp, l}(\cor_V)$ is not stable by infinite
products even for $V=\R$, $\gamma = [0,+\infty[$ (see
Example~\ref{ex:proj_infi_prod}) and we cannot apply Brown theorem to find a new
projector.

\begin{example}\label{ex:proj_infi_prod}
  In the case $V=\R$, $\gamma = [0,+\infty[$, we have, for any $x\in\R$,
  $P'_\gamma(\cor_{[x,+\infty[}) \simeq \cor_{[x,+\infty[}$. Let us check that
  $P'_\gamma(F) \not\simeq F$, where $F = \prod_{i\in\N}\cor_{[-i,+\infty[}$, thus
  proving that $\Der_{V \times \gamma^{\circ a}}^{\perp, l}(\cor_V)$ is not stable by
  infinite products.
  
  We compute the stalks $(P'_\gamma(F))_x$ using~\eqref{eq:Qgamma_conv}
  and~\eqref{eq:germ_compo}. We find $(P'_\gamma(F))_x \simeq \rsect_c(\R; F_{I_x})$,
  with $I_x = \mo]-\infty,x]$.  Using the exact sequence
  $0 \to G \to \cor^\N_\R \to F \to 0$, with $G = \prod_{i\in\N}\cor_{]-\infty,-i[}$,
  and $\rsect_c(\R; L_{I_x}) \simeq 0$ for any constant sheaf $L$, we obtain
  $(P'_\gamma(F))_x \simeq \rsect_c(\R;G_{I_x})[1]$.  Now
  $G \simeq \bigoplus_{i\in\N}\cor_{]-\infty,-i[}$ because this sum is locally
  finite.  The functors $\rsect_c(\R;\cdot)$ and $(\cdot)_{I_x}$ commute with direct
  sums. Hence $\rsect_c(\R; G_{I_x}) \simeq \cor^{(\N)}$ for $x>0$.  On the other
  hand we have $F_x \simeq \cor^\N$ for $x>0$. Hence $P'_\gamma(F) \not\simeq F$.
\end{example}

It should be noted that the cut-off functor $P_\gamma$ is defined in~\cite{KS90} in a
relative situation where $V$ is replaced by $M\times V$ for some manifold $M$.  In
this general setting $P_\gamma$ is defined on $\Der(\cor_{M\times V})$ and projects
onto the subcategory $\Der_{T^*M \times (V \times \gamma)}(\cor_{M\times V})$ formed
by the $F$ such that $\SSi(F) \subset T^*M \times (V \times \gamma^{\circ a})$
(Notation~\ref{not:cat_micsupp_fixe}).  We have
$P_\gamma(F) = \roim{q_2}( \cor_{M\times \widetilde\gamma} \tens \opb{q_1} F)$, where
$q_1, q_2$ are now the projections $M \times V^2 \to M \times V$ and
$\widetilde\gamma$ is defined as in~\eqref{eq:defPgam}.

\medskip

For later use we quote some properties of $P_\gamma$ already considered by Tamarkin
in~\cite{T08} in the case $V=\R$, $\gamma = [0,+\infty[$ or
$\gamma = \mo]-\infty,0]$.  We introduce coordinates $(t;\tau)$ on $T^*\R$ and local
coordinates $(x,t;\xi,\tau)$ on $T^*(M\times\R)$.  We denote for short by
$\{\tau \geq 0\}$ or $\{\tau > 0\}$ the subsets of $T^*(M\times\R)$ defined by the
corresponding conditions on $\tau$.  In this relative setting the projector
$P_{]-\infty,0]} \cl \Der(\cor_{M\times \R}) \to \Der_{\{\tau \geq 0\}}(\cor_{M\times
  \R})$ can be rewritten as 
\begin{equation}
\label{eq:projcutoffconvol}
P_{]-\infty,0]} (F) 
= \roim{s}( \opb{q_1}(F) \tens \opb{q_2}(\cor_{[0,+\infty[})) ,
\end{equation}
where   $s, q_1 \cl M \times \R^2 \to M \times \R$
and $q_2 \cl M \times \R^2 \to \R$ are defined by $s(x,t_1,t_2) = (x, t_1+t_2)$,
$q_1(x,t_1,t_2) = (x, t_1)$ and $q_2(x,t_1,t_2) = t_2$.

In~\cite{T08} Tamarkin was rather interested in the projector $P'_{[0,+\infty[}$,
defined on the category $\Der(\cor_{M\times \R})$ with image
\begin{equation}
  \label{eq:def_taupos_orthog}
\begin{split}
\Der_{\{\tau \leq 0\}}^{\perp, l}(\cor_{M\times \R})
= \{F \in \Der(\cor_{M\times \R}); &\; \Hom(F,G) \simeq 0  \\
&\text{ for all $G\in \Der_{\{\tau \leq 0\}}(\cor_{M\times \R})$} \} 
\end{split}
\end{equation}
as noted in~\eqref{eq:def_left_orthog}.  Using~\eqref{eq:Qgamma_conv} and noticing
that
$\DD'(\cor_{]-\infty,0]^a \setminus\{0\}}) \simeq \DD'(\cor_{]-\infty,0[}) \simeq
\cor_{]-\infty,0]}$, we see that $P_{]-\infty,0]}$ and $P_{[0,+\infty[}'$ have very
similar expressions.  With the notations of~\eqref{eq:projcutoffconvol} we can
rewrite
\begin{equation*}
P_{[0,+\infty[}'(F) 
= \reim{s}( \opb{q_1}(F) \tens \opb{q_2}(\cor_{[0,+\infty[})) .
\end{equation*}
Using this formula we find the bound
$\SSi(P_{[0,+\infty[}'(F)) \subset \{\tau \geq 0\}$, like $\SSi(P_{]-\infty,0]}(F))$
in~\eqref{eq:SSPgamF}. This proves that
$\Der_{\{\tau \leq 0\}}^{\perp, l}(\cor_{M\times \R})$ is in fact contained in
$\Der_{\{\tau \geq 0\}}(\cor_{M\times \R})$.

An important remark of Tamarkin is that the functors $P_{]-\infty,0]}$ or
$P_{[0,+\infty[}'$ have a natural morphism not only to the identity functor but also
to the direct image functor $\oim{T_c}$ for $c\geq 0$, where $T_c$ denotes the
translation along $\R$
$$
T_c \cl M \times \R \to M \times \R, \qquad (x,t) \mapsto (x,t+c) .
$$
In other words the objects $F$ of $\Der_{\{\tau \geq 0\}}(\cor_{M\times \R})$ or
$\Der_{\{\tau \leq 0\}}^{\perp, l}(\cor_{M\times \R})$ come with a natural morphism
$\tau_c(F) \colon F \to \oim{T_c}(F)$ for $c\geq 0$.

To define $\tau_c$ we first remark that we have an isomorphism of functors
$\oim{T_c} \circ P_{]-\infty,0]} \simeq P_{]-\infty,0]} \circ \oim{T_c}$ and that
\begin{equation}
\label{eq:projcutoffconvol_transl}
\oim{T_c} \circ P_{]-\infty,0]} (F) 
\simeq \roim{s}( \opb{q_1}(F) \tens \opb{q_2}(\cor_{[c,+\infty[})) .
\end{equation}
For $c\geq 0$ the natural morphism $\cor_{[0,+\infty[} \to \cor_{[c,+\infty[}$
induces the morphism of functors~\eqref{eq:morph_id_transl} and hence the
morphism~\eqref{eq:morph_id_transl2} for sheaves in the image of
$P_{]-\infty,0]}$:
\begin{align}
\label{eq:morph_id_transl}
  \tau_c &\cl P_{]-\infty,0]} \to \oim{T_c} \circ P_{]-\infty,0]} ,  \\
\label{eq:morph_id_transl2}
  \tau_c(F) &\cl F \to \oim{T_c}(F), \qquad
              \text{for $F\in \Der_{\{\tau \geq 0\}}(\cor_{M\times \R})$.}              
\end{align}
Tamarkin then used this morphism $\tau_c$ to give non-displaceability conditions on
subsets of $T^*M$.  We will consider a variant of Tamarkin criterion by looking at
the minimal $c$ such that $\tau_c(F)=0$ (see~\S\ref{sec:nonsqueezing}).

\begin{remark}\label{rem:autresproj}
  Assuming $\gamma \subset V$ is a closed convex cone which contains no line, we can
  also easily build projectors which cut the microsupport by the complement of
  $\Int(\gamma^\circ)$.  Let us set
  $Z_\gamma = V \times (V^* \setminus \Int(\gamma^\circ))$ and let
  $\Der_{Z_\gamma}(\cor_V)$ be the subcategory of $\Der(\cor_V)$ of sheaves with
  microsupport contained in $Z_\gamma$.  We can generalize Tamarkin's construction in
  higher dimension and define projectors from $\Der(\cor_V)$ to itself with image the
  subcategory $\Der_{Z_\gamma}(\cor_V)$ or its left or right orthogonal.  For example
  it is proved in~\cite[Prop. 4.21]{GS11} that
  $L_\gamma \colon F \mapsto F \circ \cor_{\widetilde\gamma^a}$ is a projector on
  $\Der(\cor_V)$ with image $\Der_{Z_\gamma}^{\perp, l}(\cor_V)$.  We also know that
  $\Der_{Z_\gamma}^{\perp, l}(\cor_V)$ is contained in
  $\Der_{V \times \gamma^{\circ}}(\cor_V)$, which is the image of $P_{\gamma^a}$ (we
  can prove that $\SSi(L_\gamma(F)) \subset V \times \gamma^{\circ}$
  like~\eqref{eq:SSPgamQgam}, or we refer to {\em loc. cit.}). We recall that
  $P_{\gamma^a}$ and $L_\gamma$ often coincide (see Lemma~\ref{lem:Pgamma-conv}).  Of
  course, when $V = \R$ we have $Z_\gamma = V \times \gamma^{\circ a}$,
  $L_\gamma = P'_\gamma$ and we recover Tamarkin's projector (see
  around~\eqref{eq:def_taupos_orthog}).
\end{remark}

\part{Constructible sheaves in dimension $1$}

In this part we apply Gabriel's theorem to describe the constructible sheaves on
the real line and the circle with coefficients in a field $\cor$.

\section{Gabriel's theorem}
\label{sec:Gabrielthm}

We give a quick reminder of a part of Gabriel's theorem that we will use in this
part.  We follow the presentation of Brion's lecture~\cite{B12} on the subject
and we refer the reader to~\cite{B12} for further details.  In this section
$\cor$ is a field.

A {\em quiver} is a finite directed graph, that is, a quadruple $Q = (Q_0$,
$Q_1$, $s, t)$ where $Q_0$, $Q_1$ are finite sets (the set of vertices,
resp. arrows) and $s, t \cl Q_1 \to Q_0$ are maps assigning to each arrow its
source, resp. target.  A {\em representation} of a quiver $Q$ consists of a
family of $\cor$-vector spaces $V_i$ indexed by the vertices $i\in Q_0$,
together with a family of linear maps
$f_\alpha \cl V_{s(\alpha)} \to V_{t(\alpha)}$ indexed by the arrows
$\alpha\in Q_1$.  For a representation $(\{V_i\}, f_\alpha)$ the dimension
vector is $(\dim V_i)_{i \in Q_0}$.  The space $\R^{Q_0}$ of dimension vectors
is endowed with the Tits form defined by
$q_Q(\ul{d}) = \sum_{i\in Q_0} d_i^2 - \sum_{\alpha\in Q_1} d_{s(\alpha)}
d_{t(\alpha)}$.

A quiver is of {\em finite orbit type} if it has only finitely many isomorphism
classes of representations of any prescribed dimension vector.  Gabriel's
theorem describes the quivers of finite orbit type.  It says that they are the
quivers with a positive defined Tits form and also says that this is equivalent
to be of type $A,D,E$.

Another part of Gabriel's theorem gives the structure of the representations of
the quivers of finite type. A representation is said {\em indecomposable} if it
cannot be split as the sum of two non zero representations. A representation $V$
is {\em Schur} if $\Hom(V,V) \simeq \cor \cdot \id_V$.

\begin{theorem}[see Theorem 2.4.3 in~\cite{B12}]
\label{thm:gabriel}
Assume that the Tits form $q_Q$ is positive definite. Then:
\\
(i) Every indecomposable representation is Schur and has no non-zero
self-extensions.
\\
(ii) The dimension vectors of the indecomposable representations are exactly
those $\ul{d} \in \N^{Q_0}$ such that $q_Q(\ul{d}) = 1$.
\\
(iii) Every indecomposable representation is uniquely determined by its
dimension vector, up to isomorphism.
\end{theorem}

\begin{remark}
  \label{rem:Gabriel_An}
  We   are actually only interested in
  quivers of type $A_m$, that is, quivers whose underlying graph is the linear
  graph with $m$ vertices and $m-1$ edges.  In this case we have
  $Q_0 = \{0,\ldots,m-1\}$ and we find
\begin{equation*}
q_Q(\ul{d}) = \frac{1}{2} 
( d_0^2 + \sum_{i=1}^{m-1} (d_i - d_{i-1})^2 + d_{m-1}^2) .  
\end{equation*}
Hence a dimension vector $\ul{d}$ satisfies the condition $q_Q(\ul{d}) = 1$ if
and only if there exist $i\leq j \in \{0,\ldots,m-1\}$ such that $d_k = 1$ if
$i\leq k \leq j$ and $d_k = 0$ else.
\end{remark}

\section{Constructible sheaves on the real line}
\label{sec:conshR}

We apply the results of Section~\ref{sec:Gabrielthm} to sheaves on $\R$ with
coefficients in a field $\cor$.

Let $\ul{x} = \{x_1 < \cdots < x_n \}$ be a finite family of points in $\R$.  We
denote by $\Mod_{\ul{x}}(\cor_\R)$ the category of constructible sheaves on $\R$ with
respect to the stratification induced by $\ul{x}$.  Setting $x_0 = -\infty$ and
$x_{n+1} = +\infty$, a sheaf $F$ belongs to $\Mod_{\ul{x}}(\cor_\R)$ if and only if
the stalks $F_y$ are finite dimensional for all $y\in \R$ and the restrictions
$F|_{]x_k,x_{k+1}[}$ are constant for $k = 0,\ldots,n$.
If $I$ is an interval of $\R$ and $x_1,\dots,x_n \in I$ we define in the same
way $\Mod_{\ul{x}}(\cor_I)$.

We say that a sheaf $F$ on $\R$ is constructible if, for any $n>0$,
$F|_{\mo]-n,n[} \in \Mod_{\ul{x}}(\cor_{\mo]-n,n[})$ for some finite family
$\ul{x} = \{x_1 < \cdots < x_k \}$ of $\mo]-n,n[$.  We denote by
$\Mod_{\rc}(\cor_\R)$ the category of constructible sheaves on $\R$.

A sheaf $F \in \Mod_{\ul{x}}(\cor_\R)$ is determined by the data of the spaces of
sections
\begin{equation}
\label{eq:faisc_carq}
\begin{alignedat}{2}
  V_{2i+1} &= F(]x_i,x_{i+2}[) \isoto F_{x_{i+1}}, 
&\quad  &\text{for $i=0,\ldots,n-1$,} \\
V_{2i} &= F(]x_i,x_{i+1}[), &\quad &\text{for $i=0,\ldots,n$,}
\end{alignedat}
\end{equation}
together with the restriction maps $V_{2i+1} \to V_{2i}$ and
$V_{2i+1} \to V_{2i+2}$ for $i=0,\ldots,n-1$.  Conversely, any such family of
vector spaces $\{V_i\}_{i=0}^{2n}$   and linear maps defines a sheaf in
$\Mod_{\ul{x}}(\cor_\R)$.  Hence~\eqref{eq:faisc_carq} gives an equivalence
between $\Mod_{\ul{x}}(\cor_\R)$ and the category of representations of the
quiver $Q = (Q_0$, $Q_1$, $s, t)$ of type $A_{2n+1}$ where
$Q_0 = \{0,\ldots,2n\}$ and there is exactly one arrow in $Q_1$ from $2i-1$ to
$2i-2$ and from $2i-1$ to $2i$, for $i=1,\ldots,n$.

Since $Q$ is of type $A_{2n+1}$, we can apply Gabriel's theorem and
Remark~\ref{rem:Gabriel_An}. Hence the indecomposable representations of $Q$ are
in bijection with the dimension vectors $\ul{d}$ such that $d_k = 1$ if
$i\leq k \leq j$ and $d_k = 0$ else, for some $i\leq j \in \{0,\ldots,2n\}$.
Through the equivalence~\eqref{eq:faisc_carq} the corresponding sheaves in
$\Mod_{\ul{x}}(\cor_\R)$ are the constant sheaves $\cor_I$ on the intervals $I$
with ends $-\infty$, $x_1,\dots, x_n$ or $+\infty$ (the intervals can be open,
closed or half-closed).

Gabriel's theorem gives the following decomposition result for constructible
sheaves with compact support.  Extensions of Gabriel's theorem gives the general
case (see for example Theorem~1.1 in~\cite{CB14}).  We can also deduce the
general case from the case of compact support; this is done
in~\cite[Thm.~1.17]{KS18} and we reproduce the proof below.

\begin{corollary}
\label{cor:cons_sh_R}
We recall that $\cor$ is a field. Let $F \in \Mod_\rc(\cor_\R)$.  Then there
exist a locally finite family of distinct   intervals $\{I_a\}_{a\in A}$
and integers $\{n_a\}_{a\in A}$ such that
\begin{equation}
\label{eq:cons_sh_R1}
F \simeq \bigoplus_{a\in A} \cor^{n_a}_{I_a} .
\end{equation}
Moreover this decomposition is unique in the following sense.  If we have
another decomposition $F \simeq \bigoplus_{b\in B} \cor^{m_b}_{J_b}$
like~\eqref{eq:cons_sh_R1}, then there exists a bijection
$\sigma \cl A \isoto B$ such that $J_{\sigma(a)} = I_a$ and
$m_{\sigma(a)} = n_a$ for all $a\in A$.
\end{corollary}
\begin{proof}   (i) For an integer $n\geq 1$ we set
  $U_n = \mo]-n,n[$.  Identifying $U_n$ with $\R$, the restriction $F|_{U_n}$
  belongs to $\Mod_{\ul{x}}(\cor_\R)$ for some finite set of points $\ul{x}$.
  Through the equivalence~\eqref{eq:faisc_carq} (and coming back to $U_n$)
  Gabriel's theorem gives a decomposition
  $F|_{U_n} \simeq \bigoplus_{c\in C_n} \cor_{I^n_c}$, where $C_n$ is a finite
  set of intervals of $U_n$. Each factor $\cor_{I^n_c}$ has multiplicity $1$,
  but we may have   $I^n_c = I^n_d$ for $c,d \in C_n$.  This
  decomposition is unique by~(i) of Theorem~\ref{thm:gabriel} and by the
  Krull-Schmidt theorem.

  The uniqueness implies that we can find an   injective map
  $\gamma_n \colon C_n \to C_{n+1}$, for any $n\geq 1$, such that
  $I^n_c = U_n \cap I^{n+1}_{\gamma_n(c)}$ for any $c \in C_n$.  We set
  $C = \varinjlim C_n$.  For $c \in C$, represented by $\tilde c \in C_n$ for
  some $n$, we set $I_c = \bigcup_{m\geq n} I^m_{\gamma_n^m(\tilde c)}$, where
  $\gamma_n^m = \gamma_{m-1} \circ \dots \circ \gamma_n$.  We then have
  $I_c \cap U_m = I^m_{\gamma_n^m(\tilde c)}$.  We remark that the obvious map
  $C_n \to C$ is injective, for any $n$.

  \sui (ii) Let $c \in C$ be given. We claim that we can find
  $i_c \cl \cor_{I_c} \to F$ and $p_c \cl F \to \cor_{I_c}$ such that
  $p_c \circ i_c = \id_{\cor_{I_c}}$. For $n\geq 1$  we define
\begin{align*}
E_n &= \Hom(\cor_{I_c}|_{U_n}, F|_{U_n}) 
\times \Hom(F|_{U_n}, \cor_{I_c}|_{U_n}) ,  \\
E'_n &= \{(i,p) \in E_n; \; p\circ i = \id_{\cor_{I_c}|_{U_n}} \}.
\end{align*}
The restriction morphisms induce $e_m^n \cl E_m \to E_n$ for $m\geq n$.  We
clearly have $e_m^n(E'_m) \subset E'_n$. We remark that $E'_n \not=\emptyset$
because $\cor_{I_c \cap U_n}$ is a direct summand of $F|_{U_n}$.  Let us prove
that $e_m^n(E'_m) = e_m^n(E_m) \cap E'_n$, for any $m\geq n$. This is clear when
$I_c \cap U_n = \emptyset$ (in this case $E_n = E'_n = \{(0,0)\}$).  If
$I_c \cap U_n \not= \emptyset$, then we have  
$$
\cor \simeq \Hom(\cor_{I_c}|_{U_m}, \cor_{I_c}|_{U_m})
\to[r] \Hom(\cor_{I_c}|_{U_n}, \cor_{I_c}|_{U_n}) \simeq \cor ,
$$
and the restriction map $r$ is an isomorphism. In particular
$r(u) = \id_{\cor_{I_c}|_{U_n}}$ implies $u = \id_{\cor_{I_c}|_{U_m}}$ and we
get $e_m^n(E'_m) = e_m^n(E_m) \cap E'_n$.  Since $F$ is constructible, the
spaces $E_n$ are all finite dimensional.  Hence, for a given $n$, the sequence
$\{e_m^n(E_m)\}_{m\geq n}$ stabilizes and it follows that
$\{e_m^n(E'_m)\}_{m\geq n}$ also stabilizes.  We set
$E_n'^\infty = \bigcap_{m\geq n} e_m^n(E'_m)$.  The maps $e_m^n$ induce
surjective maps $E_m'^\infty \to E_n'^\infty$, for all $m \geq n$.  Finally we
have $E_n'^\infty \not= \emptyset$ for all $n$; indeed
$E_n'^\infty = e_m^n(E'_m)$, for some $m$ big enough, and we have
$E'_m\not=\emptyset$.  Hence $\{E_n'^\infty)\}_{n \geq 1}$ is a projective
system of non empty sets, with surjective structural maps. It follows that
$\varprojlim_n E_n'^\infty \not= \emptyset$ and any element in this limit is a
sequence of compatible pairs of morphisms $(i_n,p_n) \in E'_n$ which glue into a
pair $(i_c,p_c)$ as claimed.

\sui(iii) Let $n \geq 1$ be an integer.  We claim that we can write
$F \simeq F_1 \oplus \bigoplus_{a \in A_1} \cor_{I_a}$, where
$F_1|_{U_n} \simeq 0$ and $A_1$ is some finite family of intervals of $\R$.

If $F|_{U_n} \simeq 0$, the claim is trivial. If not, we pick $c\in C$ such that
$I_c \cap U_n \not= \emptyset$.  Using $(i_c,p_c)$ found in~(ii) we write
$F \simeq \cor_{I_c} \oplus F^1$.  Then $F^1$ is constructible.  If
$F^1 \simeq 0$ we are done. If not, we apply the same argument to $F^1$ and
write $F^1|_{U_n} \simeq \cor_{I_1} \oplus F^2$ for some interval $I_1$ meeting
$U_n$ (the interval $I_1$ is in the family $C$ associated with $F$ but we do not
need to know that).  We go on with $F^2$ and write inductively
$F \simeq \cor_{I_c} \oplus \cor_{I_1}\oplus \cdots \oplus \cor_{I_k} \oplus
F^{k+1}$, where the $I_j$'s meet $U_n$, as long as $F^k|_{U_n} \not\simeq 0$.
Since $F$ is constructible, the space $\Hom(F|_{U_n}, F|_{U_n})$ is finite
dimensional.  Since the $I_j$'s meet $U_n$, the elements $\id_{\cor_{I_j}}$ give
a free family in $\Hom(F|_{U_n}, F|_{U_n})$.  Hence the process will stop after
finitely many steps and the claim is proved.

\sui(iv) Using~(iii) we can write inductively for $k\geq 1$,
$$
(D_k) \qquad F \simeq F_k \oplus \bigoplus_{i=1}^k G_i,  
$$
where $F_k|_{U_k} \simeq 0$, $G_i = \bigoplus_{a \in A_i} \cor_{I_a}$, where
$A_i$ is a finite family of intervals of $\R$ and $G_i|_{U_{i-1}} \simeq 0$ for
$i\geq 2$.  Indeed~(iii) applied with $F$ and $n=1$ gives the first step
and~(iii) applied with $F_k$ and $n=k+1$ gives the $(k+1)^{th}$ step. This
inductive construction also makes the decompositions ($D_k$) compatible in the
sense that the projection $u_i \colon F \to G_i$ deduced from ($D_k$) for
$k\geq i$ is actually independent of $k$.  Since $F_k|_{U_k} \simeq 0$, the sum
$\sum_{i=1}^k u_i |_{U_k} \colon F|_{U_k} \to \bigoplus_{i=1}^k G_i|_{U_k}$ is
an isomorphism.  We set $G = \bigoplus_{i=1}^\infty G_i$ and define
$u = \sum_{i=1}^\infty u_i \colon F \to G$.  This last sum makes sense because,
over each interval $U_k$, only the terms $u_i$ for $i=1,\dots,k$ are non zero.
Since $u|_{U_k}$ is an isomorphism for each $k$, $u$ itself is an isomorphism.
Putting together the intervals which appear several times (only finitely many
times by constructibility) we obtain a decomposition of $F$ as stated in the
corollary.

\sui(v)   To prove the uniqueness statement we first consider a bounded
interval $I_a$, for $a\in A$.  We choose an interval $U_n$ such that
$\ol{I_a} \subset U_n$.  Then $\cor_{I_a}$ appears in the decomposition of
$F|_{U_n}$ with the same multiplicity as in $F$. By the uniqueness of the
decomposition of $F|_{U_n}$ we deduce that $I_a = J_b$ for some $b\in B$ and
$m_b = n_a$.  Defining $A' = \{a\in A$; $I_a$ is bounded$\}$ and $B'$ similarly,
we thus have a bijection between $A'$ and $B'$ with respects the multiplicities.

We are left with an isomorphism
$\bigoplus_{a\in A\setminus A'} \cor^{n_a}_{I_a} \simeq \bigoplus_{b\in B\setminus
  B'} \cor^{m_b}_{J_b}$. Now, an unbounded interval $I$ with one end $x$ is
determined by $I \cap U_n$, as soon as $x\in U_n$.  Defining $A'' = \{a\in A$; $I_a$
is unbounded and not equal to $\R\}$ and $B''$ similarly, we can then identify $A''$
with $B''$ as in the bounded case.  Now we are left with the summand $\cor_\R$,
indexed by, say $a_0 \in A$, $b_0\in B$, and we have
$\cor_\R^{n_{a_0}} \simeq \cor_\R^{m_{b_0}}$. Hence $n_{a_0} = m_{b_0}$ and this
concludes the proof.
\end{proof}

We also remark that~(i) of Theorem~\ref{thm:gabriel} is given in our case by the
following easy result, which we quote for later use, and the fact that
$\Ext^1(\cor_I,\cor_I) \simeq 0$.
\begin{lemma}
\label{lem:morph_deux_int0}
Let $I,J$ be two intervals of $\R$.  Then
$$
\Hom(\cor_I,\cor_J) \simeq
\begin{cases}
  \cor  & \text{if $I\cap J \not= \emptyset$ and $I\cap J$ is closed in $I$} \\
  & \hspace{\fill} \text{and open in $J$,} \\
  0 & \text{else.}
\end{cases}
$$
In particular, if $I$ and $J$ are distinct, then we have   $\Hom(\cor_I,\cor_J) \simeq 0$ or
$\Hom(\cor_J,\cor_I) \simeq 0$.
\end{lemma}

\section{Constructible sheaves on the circle}
\label{sec:conshcercle}

In this section we extend Corollary~\ref{cor:cons_sh_R} to the circle.  Like in
the case of $\R$ the result is also a particular case of quiver representation
theory and Auslander-Reiten theory (see for example~\cite{R84} \S3.6 p.153 and
Theorem~5 p.158).  However it is quicker to prove the facts we need than recall
these general results.

We denote by $\cer$ the circle and we let $e \cl \R \to \cer \simeq \R/2\pi\Z$ be the
quotient map.  We use the coordinate $\theta$ on $\cer$ defined up to a multiple of
$2\pi$.  We also denote by $T \cl \R \to \R$ the translation $T(x) = x+2\pi$.  We
recall that $\cor$ is a field.  We say that $F\in \Mod(\cor_\cer)$ is constructible
if $F|_I$ is constructible in the sense of \S\ref{sec:conshR} for any arc
$I\subset \cer$. We denote by $\Mod_\rc(\cor_\cer)$ the category of such sheaves.

Since $e$ is a covering map, we have an isomorphism of functors $\epb{e} \simeq
\opb{e}$, hence an adjunction $(\eim{e},\opb{e})$.  For any $n\in \Z$ we have
$e\circ T^n = e$, hence natural isomorphisms of functors $\opb{(T^n)} \opb{e}
\simeq \opb{e}$ and $\eim{e} T^n_! \simeq \eim{e}$.  For $G \in
\Mod(\cor_\R)$ the isomorphism $\eim{e}( T^n_! (G)) \isoto \eim{e}(G)$
gives by adjunction $i_n(G) \cl T^n_! (G) \to \opb{e} \eim{e}(G)$.  For
$x\in \R$ we have $( \opb{e} \eim{e}(G))_x \simeq (\eim{e}(G))_{e(x)} \simeq
\bigoplus_{n\in \Z} G_{T^n(x)}$. We deduce that the sum of the $i_n(G)$ gives an
isomorphism
\begin{equation}
\label{eq:iminvimdire}
\bigoplus_{n\in \Z} T^n_*(G) \isoto \opb{e} \eim{e}(G) .
\end{equation}

Let $I$ be a bounded interval of $\R$.  We have $\eim{e}(\cor_I) \isoto
\oim{e}(\cor_I)$.  Let $A_I$ be the algebra $A_I = \Hom(\oim{e}(\cor_I),
\oim{e}(\cor_I))$.  The adjunction $(\opb{e},\oim{e})$ gives a morphism $\opb{e}
\oim{e}(\cor_I) \to \cor_I$. Using also the adjunction $(\eim{e},\opb{e})$ we
obtain a natural morphism
\begin{equation}
\label{eq:defvarepsilon_I}
\varepsilon_I \cl A_I \simeq \Hom(\cor_I, \opb{e} \oim{e}(\cor_I))
\to  \Hom(\cor_I, \cor_I) \simeq \cor .
\end{equation}

\begin{lemma}
\label{lem:hom-imdir-interv}
Let $I$ be a bounded interval of $\R$ and let $A_I$ be the algebra
$A_I = \Hom(\oim{e}(\cor_I), \oim{e}(\cor_I))$.  Then the morphism
$\varepsilon_I$ defined in~\eqref{eq:defvarepsilon_I} is an algebra morphism and
$\ker(\varepsilon_I)$ is a nilpotent ideal of $A_I$.  Moreover, a morphism
$u \in A_I$ is an isomorphism if and only if $\varepsilon_I(u) \not= 0$,

More precisely, if $I$ is closed or open, then $\varepsilon_I$ is an
isomorphism.  If $I$ is half-closed, say $I = [a,x[$ or $I = \mo]x,a]$ and we
set $E_a = I \cap \opb{e}(e(a))$, then the identification
$(\oim{e}(\cor_I))_{e(a)} \simeq \cor^{E(a)}$ induces a morphism
$$
A_I \to \Hom(\cor^{E(a)}, \cor^{E(a)}),
\qquad \varphi \mapsto \varphi_{e(a)}
$$
which identifies $A_I$ with the subalgebra of matrices generated by the standard
nilpotent matrix of order $|E(a)|$.
\end{lemma}
\begin{proof}
  Using $\opb{e} \oim{e}(\cor_I) \simeq \bigoplus_{n\in \Z} T^n_*(\cor_I)$
  and Lemma~\ref{lem:morph_deux_int0}, the cases $I$ closed or open are obvious.
  If $I$ is half-closed of length $l$, we find $A_I \simeq \cor^{|E(a)|}$ as a
  vector space.  Assuming $I = [a,x[$ (the case $]x,a]$ is similar), a basis of
  $A_I$ is given by the morphisms $e_*(u_n)$, for $n=0,\ldots, |E(a)|-1$, where
  $u_n \cl \cor_I \to \cor_{T^n(I)}$ is the natural morphism and we use the
  natural identification $\phi_n \cl e_*(\cor_{T^n(I)}) \simeq e_*(\cor_I)$.
  At the level of stalks $\phi_n$ identifies the summands $(\cor_I)_{a+2\pi k}$
  of $(e_*(\cor_I))_{e(a)}$ and $(\cor_{T^n(I)})_{a+2\pi (k+n)}$ of
  $(e_*(\cor_{T^n(I)}))_{e(a)}$.  We obtain that $(e_*(u_n))$ acts on $
  \cor^{E(a)}$ by $(s_1,s_2,\ldots) \mapsto (s_{1+n},s_{2+n},\ldots)$. We deduce
  that the image of $A_I$ in $\End(\cor^{E(a)})$ is as claimed in the lemma.

  The characterization of the isomorphisms then follows from the structure of
  $A_I$.
\end{proof}

\begin{lemma}
\label{lem:interv_facteur}
Let $F\in \Mod(\cor_\cer)$. Let $I$ be a bounded interval of $\R$ such that
$\cor_I$ is a direct summand of $\opb{e}(F)$. Then $\oim{e}(\cor_I)$ is a direct
summand of $F$.
\end{lemma}
\begin{proof}
  We let $i_0 \cl \cor_I \to \opb{e}(F) \simeq \epb{e}(F)$ and
  $p_0 \cl \opb{e}(F) \to \cor_I$ be morphisms such that
  $p_0 \circ i_0 = \id_{\cor_I}$ and we denote by
  $i_0' \cl \oim{e}(\cor_I) \simeq \eim{e}(\cor_I) \to F$ and
  $p_0' \cl F \to \oim{e}(\cor_I)$ their adjoint morphisms.  In general
  $p'_0 \circ i'_0 \not= \id_{\oim{e}(\cor_I)}$ but it is enough to see that
  $p'_0 \circ i'_0$ is an isomorphism. For this we use
  Lemma~\ref{lem:hom-imdir-interv}.  Let us compute
  $\varepsilon_I(p_0' \circ i_0')$.  Let $a\cl \opb{e} \oim{e}(\cor_I) \to \cor_I$
  and $b \cl \cor_I \to \opb{e} \eim{e}(\cor_I)$ be the adjunction morphisms. Then
$$
\varepsilon_I(p_0' \circ i_0') = a \circ
  \opb{e}(p_0' \circ i_0') \circ b = p_0 \circ i_0 = \id_{\cor_I}
$$
and we deduce that $p_0' \circ i_0'$ is an isomorphism.
\end{proof}

\begin{lemma}
\label{lem:interv_facteur2}
Let $F\in \Mod_\rc(\cor_\cer)$.  We assume that $F$ is not locally constant.
Then there exists a bounded interval $I$ of $\R$ such that $\cor_I$ is a direct
summand of $\opb{e}(F)$.
\end{lemma}
\begin{proof}
  By Corollary~\ref{cor:cons_sh_R} there exist a locally finite family of
  intervals $\{I_a\}_{a\in A}$ and integers $\{n_a\}_{a\in A}$ such that
  $\opb{e}(F) \simeq \bigoplus_{a\in A} \cor^{n_a}_{I_a}$.  Since $F$ is not
  locally constant, one of these intervals, say $I$, is not $\R$.  Let $T$ be
  the translation $T(x) = x+2\pi$. Since $\opb{T}\opb{e}(F) \simeq \opb{e}(F)$,
  the intervals $T^n(I)$ also appear in the decomposition, for all $n\in \Z$.
  If $I$ were not bounded, this would contradict the constructibility of $F$.
\end{proof}

\begin{proposition}
\label{prop:cons_sh_cercle}
Let $F \in \Mod_\rc(\cor_\cer)$.  Then there exist a finite family
$\{(I_a,n_a)\}_{a \in A}$, of bounded intervals and integers, and a locally
constant sheaf of finite rank $L\in \Mod(\cor_\cer)$ such that
\begin{equation}
\label{eq:cons_sh_cer1}
F \simeq L \oplus \bigoplus_{a\in A} \oim{e}(\cor^{n_a}_{I_a}) .
\end{equation}
\end{proposition}
\begin{proof}
  (i) We choose a finite stratification $\{\Sigma_i\}_{i\in I}$ of $\cer$ such
  that $F$ is constructible with respect to $\{\Sigma_i\}$.  We choose one point
  $\theta_i \in \Sigma_i$ for each $i\in I$ and set
  $r(F) = \sum_{i\in I} \dim(F_{\theta_i})$.  We prove the proposition by
  induction on $r(F)$. If $r(F) = 0$, then $F\simeq 0$ and the result is clear.

  \sui (ii) We assume $r(F)\not=0$.  If $F$ is locally constant, the result is
  clear.  Else, by Lemmas~\ref{lem:interv_facteur2} and~\ref{lem:interv_facteur}
  there exists a bounded interval $I$ of $\R$ such that
  $F \simeq \oim{e}(\cor_I) \oplus F'$ for some $F' \in \Mod_\rc(\cor_\cer)$.
  Then $r(F') < r(F)$ and the induction proceeds.
\end{proof}

\section{Cohomological dimension $1$}
\label{sec:smallcohomdim}

The decomposition results for sheaves in dimension $1$ extend to the derived
category.  For $M = \R$ or $M = \cer$ we denote by $\Derb_\rc(\cor_M)$ the full
subcategory of $\Derb(\cor_M)$ formed by the $F$ such that
$H^iF \in \Mod_\rc(\cor_M)$ for all $i\in \Z$.  We first recall a well-known
decomposition result (see for example~\cite[Ex.~13.22]{KS06}).

\begin{lemma}
\label{lem:compl_scinde}
Let $\catc$ be an abelian category and $X \in \Derb(\catc)$ a complex such that
$\Ext^k(H^iX , H^jX) \simeq 0$ for all $i,j \in \Z$ and all $k\geq 2$.  Then
$X$ is split, that is, there exists an isomorphism
$X \simeq \bigoplus_{i \in \Z} H^iX [-i]$ in $\Derb(\catc)$.
\end{lemma}

This applies in particular to constructible sheaves in dimension $1$.  Indeed,
if $M = \R$ or $M = \cer$ and $\cor$ is a field, we have $\Ext^k(F, G) \simeq 0$
for all $F,G \in \Mod_\rc(\cor_M)$ and for all $k\geq 2$.  We deduce:
\begin{lemma}
\label{lem:faiscdim1_compl_scinde}
Let $M = \R$ or $M = \cer$ and let $\cor$ be a field.  Then, for all
$F \in \Derb_\rc(\cor_M)$ we have $F \simeq \bigoplus_{i \in \Z} H^iF [-i]$ in
$\Derb_\rc(\cor_M)$.
\end{lemma}

Using Corollary~\ref{cor:cons_sh_R} and Proposition~\ref{prop:cons_sh_cercle}
we obtain immediately:
\begin{corollary}
\label{cor:sheaves_dim1}
let $\cor$ be a field. Let $M$ be $\R$ or $\cer$ and let $F \in \Derb(\cor_M)$
be a constructible object.
\begin{itemize}
\item [(i)] If $M=\R$, then there exist a locally finite
 family of intervals $\{I_a\}_{a\in A}$ and integers
$\{n_a\}_{a\in A}$, $\{d_a\}_{a\in A}$ such that
\begin{equation}
\label{eq:der_cons_sh_R}
F \simeq \bigoplus_{a\in A} \cor^{n_a}_{I_a}[d_a] .
\end{equation}
\item [(ii)] If $M=\cer$, then there exist a finite family of bounded intervals
  $\{I_a\}_{a\in A}$, integers $\{n_a\}_{a\in A}$, $\{d_a\}_{a\in A}$ and $L\in
  \Derb(\cor_\cer)$ with locally constant cohomology sheaves of finite rank
  such that
\begin{equation}
\label{eq:der_cons_sh_cer}
F \simeq L \oplus \bigoplus_{a\in A} \oim{e}(\cor^{n_a}_{I_a})[d_a] ,
\end{equation}
where $e \cl \R \to \cer \simeq \R/2\pi\Z$ is the quotient map.
\end{itemize}
\end{corollary}

The next lemma is related with the results of this section and was already used
in Example~\ref{ex:faisceau_flotgeod}.  This is a classical result
(see~\cite[Ex.~13.20, 13.21]{KS06}).  Let $\catc$ be an Abelian category.  For
$A\in \Der(\catc)$ we let $\Aut(A) \subset \Hom(A,A)$ be the isomorphism group
of $A$. For another $B\in \Der(\catc)$ the product $\Aut(A) \times \Aut(B)$ acts
on $\Hom(A,B)$ by composition.  We recall that we have truncation functors
$\tau_{\leq n}, \tau_{\geq n} \colon \Derb(\catc) \to \Derb(\catc)$ together
with morphisms of functors $c_n$ which yield for any $X \in \Derb(\catc)$ and
$n\in \Z$ a distinguished triangle  
\begin{equation}
\label{eq:dt_ext}
\tau_{\leq n} (X)  \to
X \to \tau_{\geq n+1} (X) \to[c_n(X)] \tau_{\leq n} (X)[1].
\end{equation}
For $A$, $B \in \catc$ and $n \geq 1$, we let $E^n_{A,B} \subset \Derb(\catc)$
be the full subcategory of objects $X$ such that $H^0(X) \simeq B$,
$H^n(X) \simeq A$ and $H^i(X) \simeq 0$ for $i\not= 0,n$.  For $X \in E^n_{A,B}$
and isomorphisms $a\colon A[-n] \isoto \tau_{\geq 1} (X)$,
$b \colon B \isoto \tau_{\leq 0} (X)$, we define  
$\phi_{a,b}(X) = (b[1])^{-1} \circ c_0(X) \circ a \in \Hom(A[-n], B[1]) \simeq
\Hom(A,B[n+1])$.

\begin{lemma}
\label{lem:classif_extensions}
Let $\catc$ be an abelian category and let $A$, $B \in \catc$ and $n \geq 1$ be
given.  Let $\bar E^n_{A,B}$ be the set of isomorphism classes in $E^n_{A,B}$.
  For a given $X\in E^n_{A,B}$ and isomorphisms
$a \cl A[-n] \isoto \tau_{\geq 1} (X)$, $b \cl B \isoto \tau_{\leq 0} (X)$, the image
of $\phi_{a,b}(X)$ in $\Hom(A,B[n+1]) / \Aut(A) \times \Aut(B)$ is independent of the
choice of $a$ and $b$ and yields a bijection between $\bar E^n_{A,B}$ and
$\Hom(A,B[n+1]) / \Aut(A) \times \Aut(B)$.
\end{lemma}
\begin{proof}
  (i) Changing $a$ and $b$ modifies $(b[1])^{-1} \circ c_0(X) \circ a$ by the
  action of an element in $\Aut(A) \times \Aut(B)$. This proves that the image
  of $(b[1])^{-1} \circ c_0(X) \circ a$ in the quotient only depends on $X$.

  \sui (ii) If $u\cl X \to Y$ is an isomorphism in $E^n_{A,B}$, then
  $\tau_{\geq 1} (u)$ and $\tau_{\leq 0} (u)[1]$ are isomorphisms and make a
  commutative square with $c_0(X)$ and $c_0(Y)$.  Hence $c_0(X)$ is conjugate to
  $c_0(Y)$ through $\Aut(A) \times \Aut(B)$.  This defines a map
  $\bar c_0 \cl \bar E^n_{A,B} \to \Hom(A,B[n+1]) / \Aut(A) \times \Aut(B)$.

  \sui (iii) Let $X,Y \in E^n_{A,B}$ be given. If $\bar c_0(X) = \bar c_0(Y)$, then
  there exists $(\alpha, \beta)$ such that the square~$(S)$ below commutes:
$$
\begin{tikzcd}
\tau_{\leq 0} (X)  \ar[r] \ar[d, "\beta{[-1]}", "\wr"'] & X \ar[r] \ar[d,dashed]
& \tau_{\geq 1} (X) \ar[r, "c_0(X)"] \ar[d, "\alpha", "\wr"'] 
\ar[dr, phantom, "(S)"]
& \tau_{\leq 0} (X)[1] \ar[d, "\beta", "\wr"'] \\
\tau_{\leq 0} (Y)  \ar[r]  & Y \ar[r] 
& \tau_{\geq 1} (Y)  \ar[r, "c_0(Y)"']  & \tau_{\leq 0} (Y)[1] .
\end{tikzcd}
$$
By the axioms of triangulated categories, we deduce that $X \simeq Y$.  Hence
$\bar c_0$ is injective.

\sui (iv) For $\phi \in \Hom(A,B[n+1])$ we define $X_\phi$ such that $X_\phi[1]$ is
the cone of $\phi$.  Then $\bar c_0(X_\phi) = [\phi]$ in
$\Hom(A,B[n+1]) / \Aut(A) \times \Aut(B)$.  Hence $\bar c_0$ is surjective.
\end{proof}

\part{Graph selectors}
\label{part:graphsel}

Let $M$ be a manifold and let $\Lambda \subset J^1(M)$ be a closed Legendrian
submanifold.  We let $\Gamma$ be the projection of $\Lambda$ on $M\times \R$.  A
graph selector for $\Lambda$ is a continuous function $\varphi\colon M \to \R$ whose
graph is contained in $\Gamma$.  If $\Lambda$ is generic, then $\Gamma$ is a union of
transverse immersed hypersurfaces outside a set of codimension $1$. In this case
$\varphi$ is differentiable on a dense open set where the graph of $d\varphi$ is
contained in the Lagrangian projection of $\Lambda$ in $T^*M$.  We see $\Lambda$ as a
closed conic Lagrangian submanifold of $\dT^*(M \times\R)$ as follows.  We choose
coordinates $(t;\tau)$ on $T^*\R$ and denote by $T^*_{\tau>0}(M \times\R)$ the open
set $\{\tau>0\}$ in $T^*(M \times\R)$.  Then the quotient by the multiplicative
$\rspos$-action in the fibers gives an identification
$J^1(M) \simeq (T^*_{\tau>0}(M \times\R))/\rspos$ and we abusively write $\Lambda$
for its inverse image by the quotient map; hence
$\Lambda \subset T^*_{\tau>0}(M \times\R)$.  In this short section we prove that
$\Lambda$ has a graph selector as soon as it is the microsupport of a sheaf $F$
satisfying some conditions at infinity; the graph of $\varphi$ is then the boundary
of the support of a section of $F$.  Using Theorem~\ref{thm:quant_canon} below, we
recover Theorem~1.2 of~\cite{AOO18} which says that a compact exact Lagrangian
submanifold of a cotangent bundle has a graph selector.

\medskip

For a map $\varphi \cl M \to \R$ we set $\Gamma_\varphi = \{t = \varphi(x)\}$
and $\Gamma^+_\varphi = \{t \geq \varphi(x)\}$.  We assume in this section that
$M$ is connected. For $F \in \Der(\cor_{M\times\R})$ we consider the following
condition:
\begin{equation}
  \label{eq:cond_graphsel}
  \begin{minipage}{10cm}
    there exists $A>0$ such that $\supp(F) \subset M \times [-A,+\infty[$ and
    $\dot\SSi(F) \subset T^*_{\tau>0}(M \times [-A,A])$.
  \end{minipage}
\end{equation}

We recall that the category of sheaves $\Mod(\cor_{M\times\R})$ is embedded in its
derived category $\Der(\cor_{M\times\R})$ as the subcategory of complexes
concentrated in degree $0$.  We remark the notion of support doesn't make sense for a
class $s\in H^i(M\times\R;F)$ for a general $F\in \Der(\cor_{M\times\R})$, but makes
sense for $s\in H^0(M\times\R;F) = F(M\times\R)$ when $F \in \Mod(\cor_{M\times\R})$.

\begin{proposition}
\label{prop:graphselector}
Let $F \in \Mod(\cor_{M\times\R})$ and let $s\in F(M\times \R)$ be a non-zero
section.  We assume that $F$ satisfies~\eqref{eq:cond_graphsel}.  Then there exists a
unique map $\varphi \cl M \to \R$ such that $\supp(s) = \Gamma_\varphi^+$ and this
map is continuous.  More precisely, for a given chart $U$ in $M$, which is identified
with a ball of $\R^n$ with coordinates $(x;\xi)$ on $T^*U$, if we have a bound
$\dot\SSi(F) \cap T^*(U\times \R) \subset \{ \tau \geq C ||\xi||\}$ for some $C>0$,
where $||\xi|| = (\sum_{i=1}^n \xi_i^2)^{\frac12}$, then the map $\varphi$ is
$C^{-1}$-Lipschitz on $U$, that is, $|\varphi(x) - \varphi(y)| \leq C^{-1}||x-y||$,
for any $x,y\in U$, where $||\cdot||$ is again the norm induced by the coordinates.
\end{proposition}
The link between the Lipschitz condition and the assumption on the microsupport
can also be found in~\cite{Vic13} and~\cite{J18}.

\begin{proof}
  (i) To prove that $\supp(s)$ is of the form $\Gamma_\varphi^+$ it is enough to
  check that $\supp(s) \cap (\{x\}\times \R)$ is an interval of the form
  $[a,+\infty[$ for any $x\in M$.  Let $z = (x_0,t_0) \in M\times\R$ such that
  $s_z \not=0$.  We set $i\colon \R \to M\times\R$, $t\mapsto (x_0,t)$ and
  $G = \opb{i}F$. Then $\SSi(G) \subset \{\tau\geq0\}$ by the hypotheses on $F$ and
  by Theorem~\ref{thm:iminv}. The section $s$ induces a section $s'$ of $G$ over $\R$
  and we have $s'_t = s_{i(t)}$. Hence $\supp(s') = \opb{i}(\supp(s))$ and this is
  non empty since $s_z\not=0$.

  By Corollary~\ref{cor:Morse} the restriction map
  $H^0(]a,c[; G) \to H^0(]b,c[; G)$ is an isomorphism for any $a\leq b <c$.  It
  follows that $s'_t=0$ implies $s'_{t_1} = 0$ for all $t_1 \leq t$. Indeed, we
  can find $a,b,c$ such that $a<t_1$, $t_1<b<t<c$ and $s'|_{]b,c[} = 0$.  Then
  $s'|_{]a,c[} = 0$ and in particular $s'_{t_1} = 0$.  Since $\supp(s')$ is
  closed, non empty and contained in $[A,+\infty[$, this proves that it must be
  an interval of the form $[a,+\infty[$, as required.
  
  It remains to check that, for any $x\in M$, there exists $t$ such that
  $s_{(x,t)}\not= 0$. The hypothesis on $\SSi(F)$ implies that $F|_V$, where
  $V = M \times \mo]A,+\infty[$, is a locally constant sheaf. Since $M$ is
  connected, the support of any global section of $F|_V$ is either empty or $V$.
  We have seen that $\supp(s)$ contains $\{x_0\} \times [a,+\infty[$ for some
  $a$. Hence it must contain $V$.  Finally we obtain that
  $\supp(s) = \Gamma_\varphi^+$ for some function $\varphi \colon M \to \R$.

  \sui(ii) Now we assume that we are given a chart $U$ and $C>0$ as in the second
  part of the statement; hence
  $\dot\SSi(F) \cap T^*(U\times\R) \subset (U\times\R) \times \{ \tau \geq C
  ||\xi||\}$. The chart $U$ is identified with an open ball in $\R^n$ and we let
  $\gamma \subset \R^{n+1}$ be the closed convex cone $\{ t \leq -C^{-1} ||x||\}$.
  We have $\gamma^{\circ a} = \{ \tau \geq C ||\xi||\}$.  By
  Remark~\ref{rem:supportgammatop} $\Gamma_\varphi^+ = \supp(s)$ is locally
  $\gamma$-closed, that is, for any $z \in U\times\R$, there exists a neighborhood
  $B$ of $z$ (for the usual topology) such that $B\cap \Gamma_\varphi^+ = B \cap Z$
  for some $Z\subset \R^{n+1}$ which is closed for the $\gamma$-topology.  By
  definition this means that
  $\R^{n+1} \setminus Z = (\R^{n+1} \setminus Z) + \gamma$; it implies
  $Z = Z + \gamma^a$.

  Let $x\in U$ be given and $z=(x,\varphi(x))$. Hence there exists a neighborhood $B$
  of $z$ such that $B\cap \Gamma_\varphi^+$ contains $B_1 = B \cap (z+\gamma^a)$ and
  is contained in $B\setminus B_2$, where $B_2 = B \cap (z + \gamma)$.  Setting
  $W = p(B_1) \cap p(B_2)$ we thus have
  $|\varphi(x) - \varphi(y)| \leq C^{-1} ||x-y||$ for all $y\in W$ and we can see
  that $W$ is a neighborhood of $x$ in $U$.  Since this holds for any $x\in U$, we
  can deduce that $\varphi$ is $C^{-1}$-Lipschitz on $U$. Indeed for given
  $x,x' \in U$, the segment $[x,x']$ is contained in $U$ ($U$ is a ball) and we can
  find finitely many points $x_i \in [x,x']$ with neighborhoods $W_i$ as above,
  $i=1,\dots,N$, such that $[x,x'] \subset \bigcup_{i=1}^N W_i$.  Then the result
  follows from the triangular inequality.   

  \sui(iii) Since $\dot\SSi(F)$ is conic and closed in $\dT^*(M\times\R)$ and
  contained in $\{\tau>0\}$, for any compact subset $K$ of some coordinate chart
  of $M\times\R$ we can find $C>0$ such that
  $\dot\SSi(F) \cap \opb{\pi_{M\times\R}}(K) \subset K \times \{ \tau \geq C
  ||\xi||\}$. By~(ii) it follows that $\varphi$ is continuous everywhere.
\end{proof}

  In order to apply
Proposition~\ref{prop:graphselector} in the situation of
Corollary~\ref{cor:graphselector} we have to replace a complex of sheaves
$F\in \Der(\cor_{M\times\R})$ by a sheaf in $\Mod(\cor_{M\times\R})$ -- we will take
$H^0F$ -- and ensure that our sheaf has a section.  The next lemma implies that
$H^0F$ still satisfies the hypothesis of the proposition and
Lemma~\ref{lem:graphselector2} deals with the sections.

\begin{lemma}
\label{lem:graphselector1}
Let $F \in \Der(\cor_{M\times\R})$ which satisfies~\eqref{eq:cond_graphsel}.
Then, for any $i\in \Z$, the sheaf $H^iF$ also
satisfies~\eqref{eq:cond_graphsel} and moreover we have
$\pi_{M\times\R}(\dot\SSi(H^iF)) \subset \pi_{M\times\R}(\dot\SSi(F))$.
\end{lemma}
\begin{proof}  
  (i) Let us check~\eqref{eq:cond_graphsel} for $H^iF$.  We have the general
  inclusion $\supp(H^iF) \subset \supp(F)$.  Corollary~\ref{cor:SSHiF} gives
  $\dot\SSi(H^iF)) \subset T^*_{\tau>0}(M \times\R)$.  Since $\dot\SSi(F)$ is
  contained in $T^*(M \times [-A,A])$, $F$ is locally constant outside
  $M \times [-A,A]$ and so is $H^iF$. Hence $\dot\SSi(H^iF)$ is contained in
  $T^*(M \times [-A,A])$ and we have~\eqref{eq:cond_graphsel}.

  \sui(ii) Now we prove the last assertion.  We recall that, for any
  $G \in \Der(\cor_{M\times\R})$, a point $x$ is not in
  $\pi_{M\times\R}(\dot\SSi(G))$ if and only if $G$ is constant in a neighborhood of
  $x$ (see Example~\ref{ex:microsupport}-(i)).  Now, if $F$ is constant, so is
  $H^iF$.  Hence, if $x \not\in \pi_{M\times\R}(\dot\SSi(F))$, we have
  $x \not\in \pi_{M\times\R}(\dot\SSi(H^iF))$, as required.
\end{proof}

\begin{lemma}
\label{lem:graphselector2}
Let $F \in \Der(\cor_{M\times\R})$ which satisfies~\eqref{eq:cond_graphsel}.
Then the restriction morphism
$\rsect(M\times \R; F) \to \rsect(M\times \mo]B,+\infty[; F)$ is an
isomorphism, for any $B\in \R$.
\end{lemma}
We remark that if $M$ were compact, this would follow directly from
Corollary~\ref{cor:Morse}.
\begin{proof}
  We set $Z = M\times \mo]-\infty,B]$ and $U = M\times \mo]B,+\infty[$.  We have
  to prove that
  $r \colon \rsect(M\times \R; F) \to \rsect(M\times \R; \rsect_U(F))$ is an
  isomorphism. The cone of $r$ is $\rsect(M\times\R; \rsect_Z(F))$.  Let
  $p \colon M\times\R \to M$ be the projection.  It is enough to prove
  $\roim{p}\rsect_Z(F) \simeq 0$.

  We remark that $p$ is proper on $\supp(\rsect_Z(F))$.  For a given $x\in M$
  and $i_x \colon \R \to M\times\R$, $t\mapsto (x,t)$, we set
  $G = \opb{i_x}\rsect_Z(F)$. Then the base change formula gives
  $(\roim{p}\rsect_Z(F))_x \simeq \rsect(\R; G)$.  By Theorems~\ref{thm:SSrhom}
  and~\ref{thm:iminv} we have $\SSi(G) \subset \{\tau\geq 0\}$. Since $G$ has
  compact support, we have $\rsect(\R; G) \isoto \rsect(]a,+\infty[; G)$ for
  some $a\in\R$.  By Corollary~\ref{cor:Morse} we also have
  $\rsect(]a,+\infty[; G) \isoto \rsect(]b,+\infty[; G)$ for $a\leq b$ and this
  vanishes for $b\gg0$ since $\supp(G)$ is compact.  
\end{proof}

\begin{corollary}
  \label{cor:graphselector}
  Let $\Lambda \subset T^*_{\tau>0}(M \times\R)$ be a closed conic Lagrangian
  submanifold.  We assume that $\Lambda \subset T^*_{\tau>0}(M \times [-A,A])$
  for some $A$ and that there exists $F \in \Der(\cor_{M\times\R})$ such that
  $\dot\SSi(F) = \Lambda$, $\supp(F) \subset M \times [-A,+\infty[$ and
  $F|_{M\times \mo]A,+\infty[} \simeq \cor_{M\times \mo]A,+\infty[}$.  Then
  $\Lambda$ has a graph selector: there exists a continuous map
  $\varphi \colon M \to \R$ such that
  $\Gamma_\varphi \subset \pi_{M\times\R}(\Lambda)$.

  Moreover, as in Proposition~\ref{prop:graphselector}, for a given chart $U$ in $M$,
  which is identified with a ball of $\R^n$ with coordinates $(x;\xi)$ on $T^*U$, if
  we have a bound $\Lambda \cap T^*(U\times \R) \subset \{ \tau > C ||\xi||\}$ for
  some $C>0$, then $\varphi$ is $C^{-1}$-Lipschitz on $U$.
\end{corollary}
\begin{proof}
  (i) The sheaf
  $F|_{M\times \mo]A,+\infty[} \simeq \cor_{M\times \mo]A,+\infty[} \simeq
  H^0F|_{M\times \mo]A,+\infty[}$ has a section over $M\times \mo]A,+\infty[$,
  say $s$, corresponding to $1\in\cor$.  By Lemma~\ref{lem:graphselector1}
  $H^0F$ satisfies~\eqref{eq:cond_graphsel} and by
  Lemma~\ref{lem:graphselector2} the section $s$ can be extended to
  $s \in H^0(M\times \R; H^0F)$.  Proposition~\ref{prop:graphselector}
  associates a continuous function $\varphi$ to $s$ and we have
  $\Gamma_\varphi \subset \pi_{M\times\R}(\Lambda)$ by
  Lemma~\ref{lem:graphselector1} again.

  \sui(ii) Let $\Xi \subset T^*(M\times\R)$ be the convex hull of $\Lambda$ in
  the sense that $\Xi$ is the intersection of all closed conic subsets $S$ of
  $T^*M$ which contain $\Lambda$ and are fiberwise convex.  For a local chart as
  in the statement we also have $\Xi \cap T^*U \subset \{ \tau > C ||\xi||\}$.
  By Corollary~\ref{cor:SSHiF} we have $\dot\SSi(H^0F) \subset \Xi$.  Hence the
  Lipschitz constant given in Proposition~\ref{prop:graphselector} for $H^0F$ is
  the same as the one given in the current corollary.
\end{proof}

An example of a sheaf satisfying the hypotheses of Corollary~\ref{cor:graphselector}
is given by Corollary~\ref{cor:isot_equivcat} as follows.  We start with
$\Lambda_0 = T^*_{M\times \{0\}}(M\times\R)$ which is the microsupport of   $F^0 = \cor_{M\times [0,+\infty[}$.  Let $\Phi_t$, $t\in I$, where
$I$ is an open interval containing $[0,1]$, be a homogeneous Hamiltonian isotopy of
$\dT^*(M\times\R)$ which preserves $\dT^*_{\tau>0}(M\times\R)$ (for example the
homogeneous lift of a Hamiltonian isotopy of $T^*M$, whose support is proper over
$M$, or the lift of a contact isotopy of $J^1M$).  Corollary~\ref{cor:isot_equivcat}
gives a sheaf $F$ on $M\times \R \times I$ with $F_0 = F^0$ and
$\dot\SSi(F_t) = \Phi_t(\Lambda_0)$, where $F_t = F|_{M\times\R \times\{t\}}$.  We
remark that $\dot\SSi(F) \cap T^*N = \emptyset$ for
$N = M\times \mo]B,+\infty\mc[ \times [0,1]$ and $B\gg 0$.  Hence $F|_N$ is locally
constant. Since
$F_0|_{M\times \mo]B,+\infty\mc[} \simeq \cor_{M\times \mo]B,+\infty\mc[}$, we deduce
that $F_1|_{M\times \mo]B,+\infty\mc[} \simeq \cor_{M\times \mo]B,+\infty\mc[}$.  So
Corollary~\ref{cor:graphselector} implies that $\Phi_1(\Lambda_0)$ has a graph
selector.

More generally, using Corollary~\ref{cor:graphselector} and
Theorem~\ref{thm:quant_canon} below, we recover Theorem~1.2 of~\cite{AOO18}:

\begin{corollary}
  Let $\Lambda \subset T^*_{\tau>0}(M \times\R)$ be a closed conic Lagrangian
  submanifold which is the conification of a compact exact Lagrangian
  submanifold of $T^*M$. Then $\Lambda$ has a graph selector.
\end{corollary}

\part{The Gromov nonsqueezing theorem}
\label{part:nonsqueezing}

In this part we use the microlocal theory of sheaves to give a proof of the famous
Gromov nonsqueezing theorem (see~\cite{Gr85}), which says that there is no
symplectomorphism of $\R^{2n}$ which sends the ball of radius $R$ into a cylinder
$D_r \times \R^{2n-2}$, where $D_r$ is the disc of radius $r$, if $r<R$.  There is
already a proof with generating functions by Viterbo~\cite{V92} and it is no wonder
that we can also find a proof with sheaves.  The proof we give is inspired by the
papers of Chiu~\cite{C17}, to define a projector associated with a square
in~\S\ref{sec:cutoffcarre}, and Tamarkin~\cite{T08} to define a displacement energy
(see~Definition~\ref{def:inv_nonsqueezing} -- this displacement energy is also used
in~\cite{AI17}).  It is in fact a baby case of the main result of~\cite{C17} which is
the nonsqueezing in the contact setting (see also~\cite{Z18} for a survey of
Tamarkin's and Chiu's results and~\cite{F16} for another proof).  We also give a non
squeezing result for a Lagrangian submanifold of the ball; this is related with a
result of Th\'eret in~\cite{T99}.

\section{Cut-off in fiber and space directions}
\label{sec:cutoffcarre}

In this section we prove Corollary~\ref{cor:annulation_tauc} which says that if a
sheaf $F$ on $\R^n$ has its microsupport contained in some prescribed cone, then the
natural morphism $\tau_c(F)$ of~\eqref{eq:morph_id_transl2} vanishes for $c$ big
enough.  This will be used in the next section to recover classical non squeezing
results.  The idea is to construct explicitly a sheaf $K_\infty$ on $\R^{2n}$ such
that $F \circ K_\infty \isoto F$ for $F$ as above and check that $\tau_c(K_\infty)$
vanishes.

In Part~\ref{chap:cutoff} we have seen several functors on the category of
sheaves whose effect is to change the microsupport and make it avoid some given
set.  For a vector space $V$ and a closed convex cone $\gamma \subset V$, we
have seen $P_\gamma \cl \Der(\cor_V) \to \Der(\cor_V)$ which is a projector, in
the sense that $P_\gamma \circ P_\gamma \simeq P_\gamma$, and satisfies
$\SSi(P_\gamma(F)) \subset V \times \gamma^{\circ a}$.  More usual functors
reduce the support of a sheaf, for example $\Der(\cor_M) \to \Der(\cor_M)$,
$F \mapsto F_Z$, for a locally closed set $Z \subset M$.  This functor is also a
projector and satisfies $\SSi(F_Z) \subset \ol{Z} \times V^*$.  If $Z$ is
closed, a sheaf satisfying $\SSi(F) \subset Z \times \gamma^{\circ a}$ is stable
by both functors $P_\gamma$ and $(\cdot)_Z$, hence by the composition
$Q \colon F \mapsto P_\gamma(F_Z)$.  This $Q$ is not a projector but we will see
in a special case that $Q^{\circ i}$ converges when $i\to\infty$ and gives a
projector (in fact we will work with a variant of $Q$).

The construction we give in this section is in fact a baby case of a
construction by Chiu in~\cite{C17} who defines a projector corresponding to a
subset $C$ of $T^*\R^{n+1}$ which is the cone over a ball in $T^*\R^n$; here we
do the case where $C$ is the cone over a square $[-1,1]^2$ in
$T^*\R = \R\times \R^*$.  We will use this projector to recover nonsqueezing
results in the symplectic case.

\smallskip

It will be more convenient to use a composition functor
$F \mapsto F \circ \cor_{\widetilde\gamma}$ rather than $P_\gamma$ because the
composition is associative.  This is similar to the Tamarkin projector
of~\S\ref{sec:Tamproj}: see Remark~\ref{rem:autresproj} where this functor is denoted
$L_{\gamma^a}$.  We recall that, for a manifold $M$ and a conic subset
$A\subset T^*M$, $\Der_A(\cor_M)$ is the full subcategory of $\Der(\cor_M)$ formed by
the $F$ with $\SSi(F) \subset A$ (see Notation~\ref{not:cat_micsupp_fixe}).  The
composition $-\circ \cor_{\widetilde\gamma}$ is still a projector and its image is
$\Der_{A_\gamma}^{\perp, l}(\cor_{V})$, the left orthogonal of
$\Der_{A_\gamma}(\cor_{V})$, where
$A_\gamma = V \times (V^*\setminus \Int(\gamma^{\circ a}))$.  We have
$\Der_{A_\gamma}^{\perp, l}(\cor_{V}) \subset \Der_{V \times \gamma^{\circ
    a}}(\cor_{V})$.  (See Remark~\ref{rem:autresproj} and~\cite[Prop.~4.10]{GS11}.)
However we don't use the results of~\cite{GS11}; the microsupport estimates will rely
on the properties of $P_\gamma$ and Lemma~\ref{lem:Pgamma-conv}.

\smallskip

We introduce some notations.  Let $n\geq 2$ be given.  We consider the vector
spaces $V' = \R^{n-2}$, $V = \R^2 \times V'$ and put coordinates $(x;\xi)$ on
$T^*V$.  We let $\gamma \subset V$ be the cone
$$
\gamma = \{(x) \in \R^n; \; x_2 \leq -|x_1| ,\;  x_3=\cdots=x_n=0\}.
$$
Hence $\gamma^\circ \subset V^*$ is given by $\gamma^\circ = \{ (\xi) \in \R^n$;
$\xi_2 \leq -|\xi_1|\}$.  We recall the notation
$\widetilde\gamma = \{(x,y)\in V^2$; $x-y \in \gamma\}$
of~\S\ref{sec:Globalcut-off}.  We let $Z \subset V$ be the open strip
$Z = \mo]-1,1\mc[ \times \R^{n-1}$.  Let
$R_\gamma, R_Z \colon \Der(\cor_V) \to \Der(\cor_V)$ be the functors
$R_\gamma(F) = F \circ \cor_{\widetilde\gamma}$ and $R_Z(F) = F_Z$.  They come
with natural morphisms $R_\gamma(F) \to F$ and $R_Z(F) \to F$.  We remark that
$R_Z$ can also be written as a composition
$R_Z(F) \simeq F \circ \cor_{\delta_V(Z)}$, where $\delta_V$ is the diagonal
embedding.  We define $R = R_Z \circ R_\gamma \circ R_Z$ and we find
\begin{equation}
  \label{eq:def_foncteurR}
  R(F) \simeq  F \circ \cor_W ,
\end{equation}
where
\begin{align*}
W &= \widetilde\gamma \cap  (Z\times Z)  \\
  & = \{(x,y) \in \R^{2n}; \;  y_2-x_2 \geq |x_1-y_1| , \\
 & \hspace{3cm}  |x_1|<1, \; |y_1|<1, \;  x_i=y_i, \; i=3,\ldots,n \} .
\end{align*}
Since $W \cap \Delta_V$ is closed in $W$ and open in $\Delta_V$ we have a
natural morphism $\cor_W \to \cor_{\Delta_V}$ which gives a morphism of functors
$R \to \id$.  If a sheaf satisfies $R_\gamma(F) \isoto F$ and $F_Z \isoto F$,
then $R(F) \isoto F$ and $R^{\circ i}(F) \isoto F$ for all $i\in \N$.  We now
compute $R^{\circ i}$; of course this is the composition with
$$
K_i = \cor_W \circ \cdots \circ \cor_W \qquad \text{($i$ factors $\cor_W$).}
$$
For $c\in \R$ we define $f, s, T_c \cl V \to V$ by
\begin{align*}
  f(y) &= (-y_1, y_2+2, y_3,\ldots, y_n) \\
  s(y) &= (y_1, -y_2, y_3,\ldots, y_n) \\
  T_c(y) &= (y_1, y_2+c, y_3,\ldots, y_n).
\end{align*}

Let $Z_\pm \subset \R^2$ be the closed half planes
$Z_\pm = \{ y_2 \geq \pm y_1\}$.  We have natural morphisms
$u_\pm \colon \cor_{\R^2} \to \cor_{Z_\pm}$ and we define
$L \in \Der(\cor_{\R^2})$ by
\begin{equation}
  \label{eq:def_L_faisceau_dble_cone}
L = 0 \to \cor_{\R^2} \to \cor_{Z_+} \oplus \cor_{Z_-} \to 0 ,  
\end{equation}
where $\cor_{\R^2}$ is in degree $0$.  We have
$H^0L \simeq \cor_{\R^2 \setminus (Z_+ \cup Z_-)}$,
$H^1L \simeq \cor_{Z_+ \cap Z_-}$ and a distinguished triangle
$\cor_{\R^2 \setminus (Z_+ \cup Z_-)} \to L \to \cor_{Z_+ \cap Z_-}[-1] \to
\cor_{\R^2 \setminus (Z_+ \cup Z_-)}[1]$.  (In fact we have already met $L$
in~\eqref{eq:dt_faisceau_flotgeod}: we have
$L \simeq K_\Psi|_{\{0\} \times \R^2}$ in the case $n=1$.)

\begin{lemma}\label{lem:calcul_K2}
  We set $C_1 = W \cap (\id_V \times f)( (\id_V \times s)(\Int(W)))$ and
  $W_2 = (\id_V \times f)(W)$.  The sheaf $K_2 = \cor_W \circ \cor_W$ appears in
  a distinguished triangle
  \begin{equation}
    \label{eq:calcul_K2}
    \cor_{C_1} \to[u] K_2 \to \cor_{W_2}[-1] \to \cor_{C_1}[1]
  \end{equation}
  and, in a small enough neighborhood of $\{y_1 = -x_1$,
  $y_2 = x_2+2\} \times \Delta_{V'}$, we have
  $K_2 \simeq \opb{p}(L) \etens \cor_{\Delta_{V'}}$, with
  $p \colon \R^4 \to \R^2$, $(x_1,x_2,y_1,y_2) \mapsto (x_1+y_1, y_2-x_2-2)$.
  Moreover $\cor_{C_1} \circ \cor_W \isoto \cor_{C_1}$.
\end{lemma}
\begin{proof}
(i)  Let us forget the variables $x_i$, $y_i$ for $i\geq 3$.  Let
  $q \colon \R^6 \to \R^4$ be the projection
  $q(\ul x, \ul y, \ul z) = (\ul x, \ul z)$. Then $K_2 = \reim{q}(\cor_A)$ where
  $A = (W \times \R^2) \cap (\R^2 \times W)$.  For a subset $E \subset \R^6$ and
  $(\ul x, \ul z) \in \R^4$ we set
  $E_{(\ul x, \ul z)} = E \cap \opb{q}(\ul x, \ul z)$.  If $\ul x \not\in Z$ or
  $\ul z \not\in Z$ we have $A_{(\ul x, \ul z)} = \emptyset$.  This proves that
  $(K_2)_{Z\times Z} \isoto K_2$.  Now we assume $(\ul x, \ul z) \in Z^2$ and we
  find $A_{(\ul x, \ul z)} = \{\ul y \in \R^2$; $y_2-x_2 \geq |x_1-y_1|$,
  $|y_1|<1$, $z_2 - y_2 \geq |y_1 - z_1|\}$.  Then $A_{(\ul x, \ul z)}$ is a
  bounded convex polytope, but the cohomology of $\cor_{A_{(\ul x, \ul z)}}$
  depends on $(\ul x, \ul z)$ because we have two types of boundary conditions
  (open or closed).

  \sui(ii) We define $A'$, $A^\pm \subset \opb{q}(Z^2) \subset \R^6$ by their
  fibers $A'_{(\ul x, \ul z)} = \ol{A_{(\ul x, \ul z)}}$ and
  $A^\pm_{(\ul x, \ul z)} = \partial A_{(\ul x, \ul z)} \cap \{y_1 = \pm1\}$.
  We have an exact sequence
  $0 \to \cor_A \to \cor_{A'} \to \cor_{A^+} \oplus \cor_{A^-} \to 0$.  Now the
  fibers of $A'$, $A^\pm$ are always compact convex polytopes.  For any such
  polytope $B$ we have $\reim{q}(\cor_B) \simeq \cor_{q(B)}$ and we deduce a
  resolution of $K_2$ as the complex
  $0 \to \cor_{q(A')} \to[d_1] \cor_{q(A^+)} \oplus \cor_{q(A^-)} \to 0$.  We
  have $q(A') = W$ and
\begin{equation*}
  q(A^\pm)  = \{(x,z) \in \R^4; \;
              z_2-x_2 \geq 2 \mp(x_1+z_1) ,  |x_1|<1, \; |z_1|<1 \} .
\end{equation*}
Hence $\ker(d_1) \simeq \cor_{C_1}$, $\coker(d_1) = \cor_{W_2}$ and, up to a
change of coordinates, we have the complex~\eqref{eq:def_L_faisceau_dble_cone}
defining $L$.

\sui(iii) To prove the last assertion we can compute $\cor_{C_1} \circ \cor_W$
directly or use the fact that $(\cor_{C_1})_Z \simeq \cor_{C_1}$, because
$C_1 \subset Z^2$, and $R_\gamma(\cor_{C_1}) \simeq \cor_{C_1}$, because $C_1$
is relatively compact and
$\SSi(\cor_{C_1}) \subset T^*V \times (V \times \gamma^{\circ a})$.  We deduce
$\cor_{C_1} \circ \cor_W \simeq \cor_{C_1}$ and this proves the lemma.
\end{proof}

We can see that $-\circ \cor_W$ commutes with the translation $T_c$.  More
precisely
$T_{c*}(F) \circ \cor_W \simeq T_{c*}(F \circ \cor_W) \simeq F \circ
\cor_{(\id_V \times T_c)(W)}$.  The same holds for the map $f$ and, since
$W_2 = (\id_V \times f)(W)$, we can deduce $\cor_{W_2} \circ \cor_W$ from the
triangle~\eqref{eq:calcul_K2}. We find a similar triangle
\begin{equation}
    \label{eq:td_calcul_Ki1}
\cor_{C_2} \to \cor_{W_2} \circ \cor_W \to \cor_{W_3}[-1] \to \cor_{C_2}[1] ,
\end{equation}
where $C_2 = (\id_V \times f)(C_1)$, $W_3 = (\id_V \times f)(W_2)$.  Since
$\cor_{C_1} \circ \cor_W \isoto \cor_{C_1}$, applying $-\circ \cor_W$
to~\eqref{eq:calcul_K2} gives the triangle
\begin{equation}
    \label{eq:td_calcul_Ki2}
  \cor_{C_1} \to K_3 \to \cor_{W_2} \circ \cor_W[-1] \to  \cor_{C_1}[1] .
\end{equation}
The triangle~\eqref{eq:td_calcul_Ki1} implies that
$H^0(\cor_{W_2} \circ \cor_W) \simeq \cor_{C_2}$ and
$H^1(\cor_{W_2} \circ \cor_W) \simeq \cor_{W_3}$.  Then~\eqref{eq:td_calcul_Ki2}
gives $H^0(K_3) \simeq \cor_{C_1}$, $H^1(K_3) \simeq \cor_{C_2}$ and
$H^2(K_3) \simeq \cor_{W_3}$.  Now we can compute $\cor_{W_3} \circ \cor_W$ in
the same way and an induction gives the following result.

\begin{proposition}\label{prop:calcul_Ki}
  For $i\geq 1$ we define $C_i = (\id_V \times f)^i(C_1)$ and
  $W_i = (\id_V \times f)^{i-1}(W)$.  Then, for any $n\geq 1$, 
  \begin{itemize}
  \item [(i)] we have a distinguished triangle
    $$
    \cor_{C_1} \to[u] K_{n+1} \to (\id_V \times f)_*(K_n)[-1] \to
    \cor_{C_1}[1],
    $$
  \item [(ii)] $K_n$ is concentrated in degrees $0,\dots, n-1$ and
    $H^iK_n \simeq \cor_{C_{i+1}}$, for $i=0,\dots, n-2$,
    $H^{n-1}K_n \simeq \cor_{W_n}$,
  \item [(iii)] setting
    $\Delta^t_i := \{y_1 = (-1)^ix_1, \; y_2 = x_2+2i\} \times \Delta_{V'}$ and
    $p_i \colon \R^4 \to \R^2$,
    $(x_1,x_2,y_1,y_2) \mapsto (x_1-(-1)^i y_1, y_2-x_2-2i)$, we have, in a
    small enough neighborhood of $\Delta^t_i$ and for $i=1,\dots, n-1$:
   $$
   K_n \simeq (\opb{p_i}(L) \etens \cor_{\Delta_{V'}})[1-i],
    $$
  \item [(iv)] setting $Y_i = \{y_2-x_2 \geq 2i \}$, we have, for
    $i=1,\dots, n-1$, $(K_n)_{Y_i} \simeq (\id_V \times f)^i_*(K_{n-i})[-i]$.
\end{itemize}
\end{proposition}

\begin{lemma}\label{lem:rhomKnKn}
  We set $S_n = \supp(K_n) = \ol{W_n} \cup \bigcup_{i=1}^{n-1} \ol{C_i}$.  Then we
  have an isomorphism $\cor_{S_n} \isoto \rhom(K_n,K_n)$ mapping $1$ to the identity
  morphism.
\end{lemma}
\begin{proof}
  (i) The statement is local on $V^2$.  In a neighborhood of a point $x_0$ which
  is away from the sets $\Delta^t_i$ defined in Proposition~\ref{prop:calcul_Ki}
  our sheaf $K_n$ is up to shift a constant sheaf on some subset $Y$ of $V^2$
  with $Y = C_i$ or $Y = W_n$.  Since the sets $C_i$ and $W_n$ are locally
  closed convex, we get $\rhom(K_n,K_n) \simeq \cor_{S_n}$ near $x_0$.

  \sui(ii) In a neighborhood of some $\Delta^t_i$ the lemma is reduced to the proof
  of $\rhom(L,L) \simeq \cor_D$, where $D = \supp(L)$.  Let us set $H = \rhom(L,L)$.
  Since the morphism $\cor_D \to H$ is given, it only remains to check
  $H_x \simeq \cor$ at any point $x \in D$.  This is clear for $x\not=0$.  We can
  compute $H_0$ directly from the definition~\eqref{eq:def_L_faisceau_dble_cone}.
  Alternatively we can remark that $\SSi(L) = \Lambda_\Psi \circ^a \dT^*_0\R$ (see
  the notation~\eqref{eq:imageA_par_isotopie}), where $\Psi$ is the Hamiltonian
  isotopy of $\dT^*\R$ defined by $\Psi_s(x;\xi) = (x-s\,\xi/|\xi|; \xi)$.  It
  follows from Corollary~\ref{cor:isot_equivcat} that the restriction morphisms
  $\RHom(L,L) \to \RHom(L_s, L_s)$, where $L_s = L|_{\R \times \{s\}}$ are all
  isomorphisms.  We deduce that $\rsect(\R^2;H) \simeq \cor$.  Since $H$ is a conic
  sheaf, we have $H_0 \simeq \rsect(\R^2;H)$, hence $H_0 \simeq \cor$.
\end{proof}

We define an increasing sequence of open subsets
$U_i = \{y_2-x_2 < 2i \} \subset \R^{2n}$, $i\in \Z$.  By
Lemma~\ref{lem:rhomKnKn} the natural morphism $u_j \colon K_{j+1} \to K_j$,
induced by $\cor_W \to \cor_{\Delta_V}$, gives an isomorphism
$K_{j+1}|_{U_i} \isoto K_j|_{U_i}$ if $j>i$.  Hence we can define a sheaf
$K_\infty \in \Der(\cor_{\R^{2n}})$ such that
$K_\infty|_{U_i} \simeq K_j|_{U_i}$ if $j>i$, for example as follows.  We set
$U'_i = U_i \setminus \ol{U_{i-1}}$ and $U''_i = U_i \setminus \ol{U_{i-2}}$.
The natural restriction morphism and $u_j$ induce
$v_i \colon (K_{i+1})_{U'_{i-1}} \to (K_{i+1})_{U''_i}$ and
$v'_i \colon (K_{i+1})_{U'_{i-1}} \to \oplus (K_i)_{U''_{i-1}}$ whose
restrictions to $U'_{i-1}$ are isomorphisms.  Now we define $K_\infty$ by the
distinguished triangle
\begin{equation}
  \label{eq:defkinf}
\bigoplus_{i\in \Z} (K_i)_{U'_{i-2}} \to[v] \bigoplus_{i\in \Z} (K_i)_{U''_{i-1}}
\to K_\infty  \to[+1] ,
\qquad
v = \bigoplus_i \begin{pmatrix} v_{i-1}\\ v'_{i-1} \end{pmatrix} .
\end{equation}

\begin{proposition}\label{prop:rhomKinfKinf}
  We set $S_\infty = \supp(K_\infty) = \bigcup_{i=1}^{\infty} \ol{C_i}$ and
  $Y_i = \{y_2-x_2 \geq 2i \}$.  Then we have
  $\rhom(K_\infty,K_\infty) \simeq \cor_{S_\infty}$ and
  $(K_\infty)_{Y_i} \simeq \rsect_{\Int(Y_i)}(K_\infty) \simeq (\id_V \times
  f)^i_*(K_\infty)[-i]$, for $i\geq 1$.  In particular
  \begin{equation}
    \label{eq:rhomkinf}
    \RHom(K_\infty, T_{4*} (K_\infty)) \simeq \cor[2] .
  \end{equation}
\end{proposition}
\begin{proof}
  The first assertion follows from Lemma~\ref{lem:rhomKnKn}.
  Proposition~\ref{prop:calcul_Ki}-(iv) gives
  $(K_\infty)_{Y_i} \simeq (\id_V \times f)^i_*(K_\infty)[-i]$.  The inductive
  description of $K_n$ in~(i) of the same proposition shows that
  $\SSi(K_\infty) \subset \{\xi_2+\eta_2 \geq 0\}$.  Hence $\dot\SSi(K_\infty)$
  does not meet $\SSi(\cor_{\Int(Y_i)})$ and Theorem~\ref{thm:SSrhom} gives
  $(K_\infty)_{Y_i} \simeq \rsect_{\Int(Y_i)}(K_\infty)$.  For the last
  assertion we remark that $T_4 = (\id_V \times f)^2$.  Hence
  \begin{align*}
    \rhom(K_\infty, T_{4*} (K_\infty))
    &\simeq \rhom(K_\infty, \rsect_{\Int(Y_2)}(K_\infty)[2]) \\
    &\simeq \rsect_{\Int(Y_2)}\rhom(K_\infty, K_\infty[2]) \\
    &\simeq \rsect_{\Int(Y_2)}(\cor_{S_\infty})[2] \\
    &\simeq (\cor_{S_\infty \cap Y_2})[2] 
  \end{align*}
  and the result follows.  
\end{proof}

The morphism $\cor_W \to \cor_{\Delta_V}$ gives by iteration
$K_i \to \cor_{\Delta_V}$ and then $K_\infty \to \cor_{\Delta_V}$.  In
particular for any $F\in \Der(\cor_V)$ we have a natural morphism
$F \circ K_\infty \to F$.

We recall that we used the composition functor
$F \mapsto F \circ \cor_{\widetilde\gamma}$ rather than the cut-off functor
$P_\gamma$ since the composition is associative.  The difference between these two
functors is a switch between proper and usual direct image in the
definition~\eqref{eq:def_compo_faisceaux} of the composition.  Here we will also use
the usual direct image. For later use we introduce a space of parameters which is a
manifold $N$.  Let $q_1$, $q_2\colon N\times V^2 \to N\times V$ and
$q_3 \colon N\times V^2 \to V^2$ be the projections.  For
$F \in \Der(\cor_{N\times V})$ we define $F \circ K_\infty$,
$F \circ_{np} K_\infty \in \Der(\cor_{N\times V})$ by
\begin{align}
  \label{eq:rappel_compo}
F \circ K_\infty &= \reim{q_2}( \opb{q_1}F \otimes \opb{q_3}K_\infty ) , \\
  \label{eq:def_compo_np}
  F \circ_{np} K_\infty
                 &= \roim{q_2}( \opb{q_1}F \otimes \opb{q_3}K_\infty ) .
\end{align}

\begin{proposition}\label{prop:faisceauKinfstable}
  Let $F\in \Der(\cor_{N\times V})$. We assume
  $\SSi(F) \subset T^*N \times (V \times \gamma^{\circ a})$ and
  $\supp(F) \subset N \times \ol Z$ (recall $Z= \mo]-1,1\mc[ \times \R^{n-1}$).  Then
  $F \circ_{np} K_\infty \isoto F$.

  If, moreover, there exists $x_2^0 \in \R$ such that
  $\supp(F) \subset \{x_2 \geq x_2^0\}$, then $F \circ K_\infty \isoto F$.  
\end{proposition}
\begin{proof}
  The second assertion is a particular case of the first one, since the
  hypothesis gives that the map $q_2$ in~\eqref{eq:def_compo_np} is proper on
  the support of $\opb{q_1}F \otimes \opb{q_3}K_\infty$.  However we first prove
  the second assertion in~(i) and~(ii) below and we deduce the first one
  in~(iii).

  \sui (i) To prove that the morphism $F \circ K_\infty \to F$ is an
  isomorphism, it is enough to see that the restrictions to $\{p\}\times V$, for
  any $p \in N$, are isomorphisms.  Using the base change formula and the bound
  of Theorem~\ref{thm:iminv} for $\SSi(F|_{\{p\}\times V})$, this means that we
  can assume that $N$ is a point.
  
  By Proposition~\ref{prop:cut-off1} we have $P_\gamma(F) \isoto F$, where
  $P_\gamma$ is defined in~\eqref{eq:defPgam} by
  $P_\gamma(F) = \roim{q_2}( \cor_{\widetilde\gamma} \tens \opb{q_1} F)$.  Since
  $\supp(F) \subset \{x_2 \geq x_2^0\}$ we can replace $\roim{q_2}$ by
  $\reim{q_2}$ in this expression and we find
  $F \circ \cor_{\widetilde\gamma} \isoto F$.

  We also have $F_Z \isoto F$.  Indeed, the support condition implies
  $F \simeq F_{\ol Z}$ and $\rsect_{\{\varphi \geq 0\}}F \simeq F$ where
  $\varphi(x) = 1+x_1$.  Now the microsupport condition implies
  $(\rsect_{\{\varphi \geq 0\}}F)_x \simeq 0$ if $x \in \{x_1=-1\}$.  Hence
  $F|_{\{x_1=-1\}} \simeq 0$. The same result holds for $x_1=1$ and we get
  $F|_{\partial Z} \simeq 0$. Finally $F_Z \isoto F$ as claimed.

  We deduce $F \circ \cor_W \isoto F$.  Iterating $i$ times we get
  $F \circ K_i \isoto F$ for all $i\in \N$.

  \sui(ii) We recall that $K_\infty|_{U_i} \simeq K_j|_{U_i}$ if $j>i$, where
  $U_i = \{y_2-x_2 < 2i \} \subset \R^{2n}$, $i\in \Z$.  For any sheaf $K$ on
  $V^2$ supported in $V^2 \setminus U_i$ we have
  $\supp(F \circ K) \subset \{x_2 \geq x_2^0 + 2i\}$.  Using the triangle
  $(K_\infty)_{U_i} \to K_\infty \to (K_\infty)_{V^2 \setminus U_i} \to[+1]$ and
  the same one with $K_\infty$ replaced by $K_i$, we deduce that, for any given
  $x' \in V$, if $j>i> x'_2 - x_2^0$, we have
  $(F \circ K_\infty)_{x'} \isoto (F \circ (K_\infty)_{U_i} )_{x'} \simeq (F
  \circ (K_j)_{U_i} )_{x'} \isofrom (F \circ K_j)_{x'}$.  Hence
  $(F \circ K_\infty)_{x'} \simeq F_{x'}$ which proves that the morphism
  $F \circ K_\infty \to F$ is an isomorphism.

  \sui(iii) Now we deduce the first assertion.  The idea is to write $F$ as the
  limit of $F_{Z_k}$, $k\in \N$, where $Z_k = \{x_2 \geq -k\}$, and check that
  $-\circ_{np} \opb{q_3}K_\infty$ commutes with limits.  Since
  $F_{Z_k}\circ_{np} \opb{q_3}K_\infty \simeq F_{Z_k}\circ \opb{q_3}K_\infty$,
  the result follows.

  Since we work in the derived category (and not the dg-derived category), we
  cannot use limits. Instead we use the distinguished triangle
  $F \to \prod_{k\in \N} F_{Z_k}\to[u] \prod_{k\in \N} F_{Z_k}\to[+1]$, where
  $u = \id -s$ and $s$ is the product of the natural morphisms
  $F_{Z_k} \to F_{Z_{k-1}}$.  It is then enough to check that
  $-\circ_{np} \opb{q_3}K_\infty$ commutes with products.  The
  definition~\eqref{eq:def_compo_np} is the composition of the three functors
  $\opb{q_1}$, $\otimes$ and $\roim{q_2}$.  Since $q_1$ is a submersion we have
  $\opb{q_1} \simeq \epb{q_1}[-n]$ and this is a right adjoint. Hence
  $\opb{q_1}$ commutes with products; the same holds for $\roim{q_2}$.  The
  tensor product is not a right adjoint, but we can write
  $\opb{q_1}F \otimes \opb{q_3}K_\infty \simeq \rhom(\DD'(\opb{q_3}K_\infty),
  \opb{q_1}F)$ and
  $\opb{q_1}F_{Z_k} \otimes \opb{q_3}K_\infty \simeq
  \rhom(\DD'(\opb{q_3}K_\infty), \opb{q_1}F_{Z_k})$.  Indeed the microsupports
  of $F$, $\cor_{Z_k}$ and $K_\infty$ are all contained in $\{\xi_2 >0\}$ (away
  from the zero section) and $K_\infty$ is constructible.  Hence the claimed
  isomorphisms follow from several applications of Theorem~\ref{thm:SSrhom}.

  Now we obtain
  $F \circ_{np} \opb{q_3}K_\infty \simeq \roim{q_2}(\rhom(
  \DD'(\opb{q_3}K_\infty), \epb{q_1}F[-n]))$ (and the same with $F_{Z_k}$
  instead of $F$) and this is a composition of right adjoint functors. Hence it
  commutes with products and the result follows.
\end{proof}

For a real number $c$ we let $T_c \cl N \times V \to N\times V$ be the
translation along the direction $y_2$ of $V$,
$T_c(p,y_1,\ldots,y_n) = (p,y_1,y_2+c, y_3,\ldots,y_n)$ ($p\in N$).  For
$F \in \Der(\cor_{N\times V})$ such that $\SSi(F) \subset \{ \eta_2 \geq 0\}$ we
let $\tau_c(F) \cl F \to \oim{T_c}(F)$ be the
morphism~\eqref{eq:morph_id_transl2}.

\begin{corollary}
\label{cor:annulation_tauc}
We have $\tau_c(K_\infty) = 0$ for all $c \geq 4$.  In particular, for any
manifold $N$ and $F \in \Der(\cor_{N\times V})$ such that
$\SSi(F) \subset T^*N \times (V \times \gamma^{\circ a})$ and
$\supp(F) \subset N \times \ol Z$, we have $\tau_c(F) = 0$ for all $c \geq 4$.
\end{corollary}
\begin{proof}
  By~\eqref{eq:rhomkinf} we have $\Hom(K_\infty, T_{c*} K_\infty) \simeq 0$ if
  $c\geq 4$ and a fortiori $\tau_c(K_\infty)=0$.  The second assertion then
  follows from the isomorphism $F \circ_{np} K_\infty \isoto F$ of
  Proposition~\ref{prop:faisceauKinfstable} and the fact that $\tau_c(F)$
  coincides with $\id_F \circ_{np} \tau_c(K_\infty)$.
\end{proof}

\section{Nonsqueezing results}
\label{sec:nonsqueezing}

Here we use Corollary~\ref{cor:annulation_tauc} to prove classical nonsqueezing
results in the symplectic case.  The morphism $\tau_c$ of~\eqref{eq:morph_id_transl},
introduced by Tamarkin in~\cite{T08}, gives the following invariant, that we can call
a {\em displacement energy}.  We refer to~\cite{AI17} where a similar (more refined)
invariant is used to obtain bounds on the displacement energy of some subsets of a
cotangent bundle and to~\cite{Z18} for a survey of Tamarkin's and Chiu's results.  We
put coordinates $(t;\tau)$ on $T^*\R$. For a manifold $M$ we thus consider
$\{\tau\geq 0\}$ as a subset of $T^*(M\times\R)$ and we denote by
$\Der_{\tau \geq 0}(\cor_{M\times \R})$ the category of sheaves $F$ on $M\times \R$
with $\SSi(F) \subset \{\tau\geq 0\}$.

\begin{definition}
  \label{def:inv_nonsqueezing}
  Let $M$ be a manifold and $F\in \Der_{\tau \geq 0}(\cor_{M\times \R})$.  We
  set $\de(F) = \sup \{c\geq 0$; $\tau_c(F) \not= 0\}$.
\end{definition}
We check in Proposition~\ref{prop:deFinvariant} below that $\de(F)$ is invariant
by Hamiltonian isotopies of $T^*M$ with compact support.  We can reformulate
Corollary~\ref{cor:annulation_tauc} as follows: for any
$F \in \Der(\cor_{N\times V})$ such that
$\SSi(F) \subset (N \times \ol Z) \times (V \times \gamma^{\circ a})$ we have
$\de(F) \leq 4$.

\begin{remark}\label{rem:tauc_fonctoriel}
  The morphism $\tau_c(F)$ for $F\in \Der_{\tau \geq 0}(\cor_{M\times \R})$
  is functorial in $M$: if $N \subset M$ is a submanifold, then
  $\tau_c(F|_{N \times \R}) = (\tau_c(F))|_{N \times \R}$.  This implies
  $\de(F|_{N \times \R}) \leq \de(F)$.
\end{remark}

\subsection{Invariance of the displacement energy}

Let $M$ be a manifold and let $h\cl T^*M \times I \to \R$ be a function of class
$\Cinf$.  We assume that its Hamiltonian flow $\Phi\cl T^*M\times I \to T^*M$ is
defined and has compact support.  As in Proposition~\ref{prop:homog_isot} we
associate with $h$ a homogeneous Hamiltonian isotopy
$\Phi'\cl \dT^*(M\times \R) \times I \to \dT^*(M\times \R)$ lifting $\Phi$ and we let
$K_{\Phi'} \in \Der(\cor_{(M \times \R)^2 \times I})$ be the sheaf given by
Theorem~\ref{thm:GKS}.  The isotopy $\Phi'$ preserves the subset $\{\tau \geq 0\}$ of
$\dT^*(M\times\R)$ and commutes with the transpose derivative of $T_c$ acting on
$\dT^*(M\times\R)$ (also denoted $T_c$ abusively, so
$T_c(x,t;\xi,\tau) = (x,t+c;\xi,\tau)$) for all $c\in\R$.  In other words
$\Phi' = T_{-c} \circ \Phi' \circ T_c$ and the functor $K_{\Phi'} \circ -$ coincides
with $\oim{T_{-c}} \circ (K_{\Phi'} \circ -) \circ \oim{T_c}$.  We obtain: for any
$F \in \Der_{\tau \geq 0}(\cor_{M\times \R})$ we have
$K_{\Phi'} \circ F \in \Der_{\tau \geq 0}(\cor_{M\times \R \times I})$ and an
isomorphism
$\alpha \colon K_{\Phi'} \circ \oim{T_c}F \simeq \oim{T_c}(K_{\Phi'} \circ F)$.  We
can then ask if the image of $\tau_c(F)$ by the composition functor
$K_{\Phi'} \circ -$ coincides with $\tau_c(K_{\Phi'} \circ F)$ through $\alpha$.  To
see this we recall the equivalence~\eqref{eq:restr_NI_Ns} of
Corollary~\ref{cor:isot_equivcat}.

We use~\eqref{eq:restr_NI_Ns} with $N = M\times\R$,
$A = \{\tau\geq 0\}\setminus T^*_{M\times\R}(M\times\R) \subset T^*(M\times\R)$ and
$A' \subset T^*(M\times\R \times I)$ associated with $A$ by $\Phi'$ in the sense
of~\eqref{eq:imageA_par_isotopie}.  We let $A''$ be the union of $A'$ and the zero
section.  Then $A''$ is half of the hypersurface of $T^*(M\times\R \times I)$ defined
by the graph of the lift of $h$:
$A'' = \{ ((x,t;\xi,\tau), (s; -\tau h(x,\xi/\tau,s)));\; \tau \geq 0 \}$.  However we
do not need a precise expression of $A'$, we only need to know that
$i_s^\sharp(A') = \Phi'_s(A)$ coincides with $A$, for all $s\in I$, where $i_s$ is
the inclusion of $M\times \R \times \{s\}$ in $M\times \R \times I$.
Corollary~\ref{cor:isot_equivcat} says that, for any $s\in I$, the inverse image
functor
\begin{equation}
  \label{eq:opbis_equiv}
  \opb{i_s} \colon  \Der_{A''}(\cor_{M\times \R \times I})
  \to \Der_{\tau\geq 0}(\cor_{M\times \R})
\end{equation}
is an equivalence of categories.  Now we have $A'' \subset \{\tau\geq 0\}$ and
any $G \in \Der_{A''}(\cor_{M\times \R \times I})$ also comes with the morphism
$\tau_c(G) \colon G \to \oim{T_c}G$. By the functoriality of $\tau_c$
(Remark~\ref{rem:tauc_fonctoriel}) we have
$\tau_c(\opb{i_s} (G)) = \opb{i_s}(\tau_c(G))$.  Since $\opb{i_s}$ is an
equivalence, we deduce that, for any $s\in I$, $\tau_c(\opb{i_s} (G))$ vanishes
if and only if $\tau_c(G)$ vanishes.  We thus get the invariance of the
displacement energy:

\begin{proposition}
  \label{prop:deFinvariant}
  Let $\Phi'\cl \dT^*(M\times \R) \times I \to \dT^*(M\times \R)$ be a
  homogeneous Hamiltonian isotopy lifting some Hamiltonian isotopy of $T^*M$ in
  the sense of Proposition~\ref{prop:homog_isot}.  Let
  $F \in \Der_{\tau \geq 0}(\cor_{M\times \R})$.  Then
  $\de(F) = \de(K_{\Phi'} \circ F) = \de(K_{\Phi',s} \circ F)$, for any
  $s\in I$.
\end{proposition}

\subsection{Nonsqueezing for a flying saucer}

We put the natural Euclidean structure on $\R^n$ and $T^*\R^n \simeq \R^{2n}$
and we denote by $B_1(E)$ and $S_1(E)$ the closed unit ball and unit sphere of
an Euclidean space $E$.  We first define a subset $\Lambda_0$ of $S_1(T^*\R^n)$
which is the image of a Legendrian of $J^1(\R^n)$ whose front projection in
$\R^{n+1}$ is a flying saucer with conic points.

We define $\Lambda_0$ as the union of the graphs of the differential of two
functions $f_1$, $f_2 \cl B_1(\R^n) \to \R$.  We choose these functions to be
rotation invariant with a differential belonging to the unit sphere of
$T^*\R^n$.  In other words we write $f_i(x) = g_i(||x||)$ for a function $g_i$
such that $r^2 + (g_i'(r))^2 =1$.  This determines $f_i$ up to a constant and we
find $f_1(x) = \int_0^{||x||} \sqrt{1-u^2} du$ and $f_2(x) = \pi/2 - f_1(x)$.
We let $W_0 = \{ f_1(||x||) \leq t < f_2(||x||) \}$ be the region in $\R^{n+1}$
bounded by the graphs of $f_1$, $f_2$ and define
\begin{equation}
  \label{eq:def_Lagr_soucvol}
\Lambda_0 = \rho_{\R^n}(\dot\SSi(\cor_{W_0})) \subset T^*\R^n .  
\end{equation}
The functions $f_i$ are not differentiable at $0$ and $W_0$ has a conic point at
$(0,0)$ and $(0,\pi/2)$:
$$
\SSi(\cor_{W_0}) \cap T^*_{(0,0)} \R^{n+1}  = \SSi(\cor_{W_0}) \cap T^*_{(0,\pi)}  \R^{n+1}
= \{(\xi, \tau); \; \tau \geq ||\xi|| \} 
$$
and we have $df_{i,x} = g_i'(||x||) \cdot \frac{x}{||x||}$ for a non zero $x$.  We
deduce
\begin{equation}
  \label{eq:def_Lagr_soucvol2}
    \left\{  \begin{aligned}
\Lambda_0 &=  L_0 \cup B_1(T^*_0\R^n) , \\
L_0 &:= \{(x;\xi) \in S_1(T^*\R^n); \; \text{$\xi$ is a scalar multiple of $x$} \} .
      \end{aligned} \right.
  \end{equation}
  We remark that $L_0$ is the image of $S_1(T^*_0\R^n)$ by the characteristic
  flow of $S_1(T^*\R^n)$. More precisely we have $S_1(T^*\R^n) = \{h=1\}$ where
  $h(x;\xi) = ||x||^2+ ||\xi||^2$.  The flow of
  $X_h = 2 \sum_i \xi_i \frac{\partial}{\partial x_i} - x_i
  \frac{\partial}{\partial \xi_i}$ has orbits of period $\pi$.  Identifying
  $T^*\R^n$ with $\C^n$ by $(x;\xi) \mapsto x+i\xi$ the flow is given by
  $\Phi_{h,s}(x+i\xi) = \exp(2si) \cdot (x+i\xi)$.  For an orbit $\gamma$ of the
  flow the action is $A= \int_\gamma \alpha$ where $\alpha = \sum \xi_i dx_i$ is
  the Liouville form. Here we find $A = \pi$ for all orbits in $S_1(T^*\R^n)$.
  We thus have
\begin{equation}
  \label{eq:formuleL0}
  L_0 = \{\exp(2si) \cdot (0;\xi) \in T^*\R^n \simeq \C^n; \;
  s \in [0,\pi], \; ||\xi|| = 1 \}  .
\end{equation}
If we smooth our functions $f_i$ near the origin (in which case the front looks
like a flying saucer), we obtain an approximation of $\Lambda_0$ by an immersed
Lagrangian sphere with one double point.

\begin{proposition}
  \label{prop:nonsqueezLagr}
  Let $r < 1/\sqrt 2$ be given and let $D_r \subset \R^2$ be the closed disc of
  radius $r$.  There is no Hamiltonian isotopy
  $\Phi \cl \R^{2n}\times I \to \R^{2n}$ such that
  $\Phi_1(\Lambda_0) \subset D_r \times \R^{2n-2}$, where $\Lambda_0$ is defined
  in~\eqref{eq:def_Lagr_soucvol}.
\end{proposition}
\begin{proof}
  (i) Let us assume that such an isotopy $\Phi$ exists.  We can find an isotopy
  $\Psi$ of $\R^{2n} = T^*\R^n$ such that
  $\Psi_1(D_r \times \R^{2n-2}) \subset \mo]-a,a\mc[^2 \times T^*\R^{n-1}$, for
  some $a$ with $(2a)^2 < \pi /2$.  Hence we can as well assume that
  $\Phi_1(\Lambda_0) \subset \mo]-a,a\mc[^2 \times T^*\R^{n-1}$.  We can also
  assume that $\Phi$ has compact support. We let $\Phi'$ be a homogeneous
  isotopy of $T^*\R^{n+1}$ lifting $\Phi$ as in
  Proposition~\ref{prop:homog_isot} and we let
  $K_{\Phi'} \in \Der(\cor_{\R^{2n+2} \times I})$ be the sheaf associated with
  $\Phi'$ by Theorem~\ref{thm:GKS}.

  \sui(ii) We set $F_0 = \cor_{W_0}$ and define $F_1 = K_{\Phi',1} \circ F_0$.
  Hence $\dot\SSi(F_1) = \Phi'_1(\dot\SSi(F_0))$ and $F_1$ has compact support.  In
  particular $\rho_{\R^n}(\dot\SSi(F_1)) = \Phi_1(\Lambda_0)$.  By
  Proposition~\ref{prop:deFinvariant} we have $\de(F_0) = \de(F_1)$.

  \sui(iii) Let $\Gamma_1 = \pi_{\R^{n+1}}(\dot\SSi(F_1))$ be the projection of
  $\dot\SSi(F_1)$ to the base.  Since
  $\rho_{\R^n}(\dot\SSi(F_1)) \subset \mo]-a,a\mc[^2 \times T^*\R^{n-1}$, we
  have $\Gamma_1 \subset Z_a = \mo]-a,a\mc[ \times \R^n$. Hence $F_1$ is locally
  constant outside $Z_a$. Since it has compact support, it has to vanish outside
  $Z_a$.  Hence $F_1$ satisfies the hypotheses of
  Corollary~\ref{cor:annulation_tauc}, up to a rescaling of $Z$ and $\gamma$ by
  $a$ and we obtain $\de(F_1) \leq 4a^2 < \pi/2$.

  On the other hand
  $\rhom(F_0, \oim{T_c}(F_0)) \simeq \cor_{\ol{W_0 \cap T_c(W_0)}}$ for
  $c\in [0,\pi/2[$.  Since the topology of $\ol{W_0 \cap T_c(W_0)}$ is unchanged
  when $c$ runs over $[0,\pi/2[$ we deduce that $\de(F_0) = \pi/2$.  We thus
  have a contradiction.
\end{proof}

\subsection{Nonsqueezing for $L_0$}

By~\eqref{eq:def_Lagr_soucvol2} the Lagrangian subset $\Lambda_0$ of
$B_1(T^*\R^n)$ consists of two parts, $B_1(T^*_0\R^n)$ and $L_0$ which can be
identified with the image of $S_1(T^*_0\R^n)$ by the geodesic flow.  The part
$B_1(T^*_0\R^n)$ corresponds to the microsupport of $\cor_{W_0}$ at the conic
points $(\ul 0,0)$ and $(\ul 0,\pi/2)$.  We can define another sheaf $G_0$ on
$\R^{n+1}$ such that $\rho_{\R^n}(\dot\SSi(G_0)) = L_0$.  Let us recall the
sheaf $K_\Psi \in \Der(\cor_{\R^{2n+1}})$ of Example~\ref{ex:faisceau_flotgeod}
associated with the normalized geodesic flow of $\dT^*\R^n$.  It fits in the
distinguished triangle~\eqref{eq:dt_faisceau_flotgeod} which is an extension
between constant sheaves on two opposite cones (one closed, the other open). If
we restrict $K_\Psi$ to a slice $V_0 = \R^n \times \{0\} \times \R$, we obtain a
picture similar to the following.  We define $W_1 = T_{\pi/2}(W_0)$.  Then there
exists a diffeomorphism $f$ from a neighborhood $\Omega$ of $(\ul 0,\pi/2)$ in
$\R^{n+1}$ to a neighborhood of $0$ in $\R^{n+1}$ such that
$f(W_0) = U \cap V_0$ and $f(W_1) = Z \cap V_0$ (with the notations $U$, $Z$
of~\eqref{eq:dt_faisceau_flotgeod}).  Then $\opb{f}(K_\Psi|_{V_0})[-n]$ extends
as a sheaf $F$ which fits in a distinguished triangle
\begin{equation}
\label{eq:def_exten_pt_conique}
\cor_{W_0}\to F \to \cor_{W_1}[-n] \to \cor_{W_0}[1] .  
\end{equation}
The microsupport of $F$ is the image of $\SSi(K_\Psi|_{V_0})$ by $df$ and we
find $\rho_{\R^n}(\dot\SSi(F|_\Omega)) \subset L_0$, as required.

However $F$ still has a too big microsupport at the conic points $(0,0)$ and
$(0,\pi)$. We can repeat the above gluing process at these points and
iterate. We obtain a sheaf $G_0 \in \Derlb(\cor_{\R^{n+1}})$, in the same way we
defined $K_\infty$ (see~\eqref{eq:defkinf}), with the following property
\begin{equation}\label{eq:defG0soupvol}
\begin{minipage}[c]{10cm}
  $G_0|_{\Omega_k} \simeq (T_{k\frac\pi2})_*(F)[-kn]|_{\Omega_k}$, for any $k\in\Z$,
  where $F$ is defined in~\eqref{eq:def_exten_pt_conique} and
  $\Omega_k = \R^n \times \mo]k \frac\pi2, (k+2)\frac\pi2[$
\end{minipage}
\end{equation}
and we have $\rho_{\R^n}(\dot\SSi(G_0)) = L_0$.

\begin{proposition}
  \label{prop:nonsqueezLagr2}
  Proposition~\ref{prop:nonsqueezLagr} holds with $\Lambda_0$ replaced by $L_0$.
\end{proposition}
\begin{proof}
  It is enough to prove that $\de(G_0) \geq \pi/2$; the rest of the proof of
  Proposition~\ref{prop:nonsqueezLagr} works the same.  By
  Remark~\ref{rem:tauc_fonctoriel} we have
  $\de(G_0) \geq \de(G_0|_{\{x_0\} \times \R})$ for any $x_0 \in \R^n$.  The
  above description of $G_0$ implies
  $G_0|_{\{0\} \times \R} \simeq \bigoplus_{k\in \Z} \cor_{[k \frac\pi2,
    (k+1)\frac\pi2[} [-kn]$.  Since $\de(\cor_{[a,b[}) = b-a$ we obtain
  $\de(G_0) \geq \pi/2$, as claimed.
\end{proof}

\subsection{Nonsqueezing for the ball}

We set $B = \Int(B_1(\R^n))$ and $S = S_1(T^*\R^n)$.  For $y \in B$, let
$L_y \subset S$ be the image of $T^*_y\R^n \cap S$ by the characteristic flow of $S$.
Then $L_y$ is a Lagrangian submanifold of $T^*\R^n$.  To prove the nonsqueezing for
the ball we use the Lagrangian submanifolds $L_y$, $y \in B$, in family.  We first
prove that there exists a sheaf $G$ on $B \times \R^{n+1}$ such that
$\rho_{\R^n}(\dot\SSi(G_y)) = L_y$, where $G_y = G|_{\{y\} \times \R^{n+1}}$.

\smallskip

Let us describe the family $L_y$. For a given $y\in B$ we can see $L_y$ as
in~\eqref{eq:formuleL0} as the image of a map
\begin{equation*}
  f_y \colon (\R/\pi\Z) \times S_r(T^*_y\R^n) \to T^*\R^n, \; (s,\xi) \mapsto\exp(2si) \cdot (y;\xi) ,
\end{equation*}
where $r = \sqrt{1-||y||^2}$.  We set
$L^0 = \bigsqcup_{y\in B} L_y \subset \R^n \times T^*\R^n$.  We can describe a
Lagrangian submanifold $L \subset (T^*\R^n)^2$ above $L^0$ as follows.  In
$T^*\R^{2n}$ we consider the hypersurface $Z = T^*\R^n \times S$ and the Lagrangian
submanifold $T^*_\Delta\R^{2n}$.  Then $L^1 = Z \cap T^*_\Delta\R^{2n}$ is isotropic
and its image $L$ by the characteristic flow of $Z$ is still isotropic, hence
Lagrangian since it is of dimension $2n$.  We can see that
$L \cap (T^*B \times T^*\R^n)$ maps bijectively onto $L^0$ through
$T^*B \times T^*\R^n \to B \times T^*\R^n$.  Like $L_y$ the Lagrangian $L$ is the
image of a map
\begin{equation}
\label{eq:plongt_L}
\begin{aligned}
  f \colon (\R/\pi\Z) \times S &\to (T^*\R^n)^2, \\
  (s,(y;\xi)) &\mapsto ((y;-\xi), \exp(2si) \cdot (y;\xi)) .  
\end{aligned}
\end{equation}
The important point is that $f$ is injective whereas $f_y$ is not.  Let
$L' \subset J^1(\R^{2n}) = (T^*\R^{2n}) \times \R$ be a Legendrian lift of $L$.  Then
$L' \to L$ is an infinite cyclic cover. A loop on $L$ given by a characteristic flow
line lifts to a path $l$ on $L'$ with $l(1) = T_\pi(l(0))$, where $T_\pi$ is as
before the translation in the last variable by $\pi$.  Hence $T_\pi(L') = L'$.  Since
$f$ is injective we have
\begin{equation}
  \label{eq:Lpasautointersection}
  T_c(L') \cap L' = \emptyset \quad\text{ for } 0<c<\pi.
\end{equation}
Moreover, we remark in~\eqref{eq:plongt_L} that, if $y\in \partial B$, then $\xi=0$
(because $(y;\xi) \in S$).  Hence above $\partial B \times \R^n$ we have
\begin{equation}
  \label{eq:LsurbordB}
  L' \cap ((\partial B \times \R^n) \times_{\R^{2n}} J^1(\R^{2n}) )
  \subset T^*_BB \times T^*\R^n \times \R .
\end{equation}

For $y\in B$ we let $L'_y \subset J^1(\{y\} \times \R^n)$ be the projection of
$L' \cap (\{y\} \times \R^n) \times_{\R_{2n}} J^1(\R^{2n})$ to
$J^1(\{y\} \times \R^n)$.  Then $L'_y$ is a Legendrian lift of $L_y$.  For $y\in B$
these manifolds $L'_y$ are all diffeomorphic to $L'_0$ and there exists a $\Cinf$ map
$u \colon B \times L'_0 \to J^1(\R^n)$ such that $u(\{y\} \times L'_0) = L'_y$.  We
can lift this map $u$ to a contact isotopy (see for example Theorem~2.6.2
of~\cite{G08}). Indeed this is possible for a family of compact Legendrian manifolds.
However, since $T_\pi(L'_y) = L'_y$ for all $y\in B$, we can find a family $L''_y$,
$y\in B$, of compact Legendrian submanifolds of $(T^*\R^n) \times (\R/\pi\Z)$ such
that $L'_y$ is a covering of $L''_y$. Then we find a contact isotopy of
$(T^*\R^n) \times (\R/\pi\Z)$ and lift it to $J^1\R^n$.  We thus have a contact
isotopy $\Phi \colon B \times J^1(\R^n) \to J^1(\R^n)$ such that
$\Phi_y(L'_0) = L'_y$, for any $y\in B$.

\smallskip

A Legendrian submanifold of $J^1(M)$ gives a conic Lagrangian submanifold of
$\{\tau>0\} \subset T^*(M\times\R)$.  Hence $L'$, $L'_y$ give conic Lagrangian
submanifolds of $\dT^*\R^{2n+1}$ or $\dT^*\R^{n+1}$ that we also denote $L'$,
$L'_y$.  The equivalence of categories~\eqref{eq:restr_NI_Ns} of
Corollary~\ref{cor:isot_equivcat} applied with $I=B$, $A_0=L$, $A' = L'$ gives a
unique sheaf $G \in \Der_{[L']}(\cor_{B \times \R^{n+1}})$ such that
$G|_{\{0\} \times \R^{n+1}} \simeq G_0$, where $G_0$ is defined
in~\eqref{eq:defG0soupvol}.

\begin{remark}
  A sheaf similar to $G$ is constructed in another way in~\cite{C17} and is
  shown to be a projector from sheaves on $\R^{n+1}$ to sheaves with a
  microsupport contained in $\rho_{\R^n}^{-1}(B_1(T^*\R^n))$.
\end{remark}

\begin{proposition}
  \label{prop:nonsqueezBall}
  Let $r < 1$ be given and let $D_r \subset \R^2$ be the closed disc of radius
  $r$.  There is no Hamiltonian isotopy $\Phi \cl \R^{2n}\times I \to \R^{2n}$
  such that $\Phi_1(B_1(\R^{2n})) \subset D_r \times \R^{2n-2}$.
\end{proposition}
\begin{proof}
  It is enough to prove that $\de(G) \geq \pi$; the rest of the proof of
  Proposition~\ref{prop:nonsqueezLagr} works the same (recall that
  Corollary~\ref{cor:annulation_tauc} works with a parameter space $N$ which is
  $B$ in our case).  We set $\Omega = B\times \R^{n+1} $ and let
  $j \colon \Omega \to \R^{2n+1}$ be the inclusion.  We let
  $u \in \Hom( \reim{j}G, \roim{j}G)$ be the natural morphism.  It is enough to
  prove that $u_c := \roim{j}( \tau_c(G)) \circ u$ is non zero for
  $0 \leq c< \pi$.

  As in the proof of Proposition~\ref{prop:nonsqueezLagr2} we have $u_c \not=0$
  for small $c$ (in fact for $c<\pi/2$).  We have
  $u_c \in H_c := \Hom(\reim{j}G, \oim{T_c}\roim{j}G)$.  It is thus enough to
  see that the morphism $H_\varepsilon \to H_c$ is an isomorphism for
  $0 < \varepsilon \leq c< \pi$.

  We set $L_\Omega = \SSi(\cor_\Omega)$.  We recall that $L_\Omega$ is the union of
  the zero section over $\Omega$ and one half of $T^*_{\partial \Omega}\R^{2n+1}$.
  By Theorem~\ref{thm:oim_open} we have
  $\dot\SSi(\reim{j}G) \subset L' \hplus L_\Omega$ and
  $\dot\SSi(\roim{j}G) \subset L' \hplus L_\Omega^a$.  Hence
  $\dot\SSi(\reim{j}G) \cap \dot\SSi(\oim{T_c}\roim{j}G) = \emptyset$ for
  $0 < c< \pi$: this follows from~\eqref{eq:Lpasautointersection} above $\Omega$ and
  from ~\eqref{eq:Lpasautointersection} and~\eqref{eq:LsurbordB} above
  $\partial \Omega$.  The hypotheses of Corollary~\ref{cor:homFtGt_indpdtt} are then
  almost satisfied except for the compactness of $\supp(\reim{j}G)$.  However $G$ is
  $\pi$-periodic up to shift in the sense that $\oim{T_\pi}(G) \simeq G[-2n]$ and we
  deduce that $\rhom(\reim{j}G, \oim{T_c}\roim{j}G)$ is $\pi$-periodic. Hence we can
  apply Lemma~\ref{lem:homFtGt_indpdtt-period} below which a variant of
  Corollary~\ref{cor:homFtGt_indpdtt}.
\end{proof}

\begin{lemma}
  \label{lem:homFtGt_indpdtt-period}
  Let $M$ be a manifold and $J$ an open interval of $\R$.  Let $F$,
  $G \in \Der(\cor_{M\times \R \times J})$. We assume
  \begin{itemize}
  \item [(i)] the projection $\supp(F) \cap \supp(G) \to \R \times J$ is proper,
  \item [(ii)] $F$, $G$ are non-characteristic for all maps
    $i_u \colon M\times \R \times \{u\} \to M\times \R \times J$, $u\in J$, that is,
    $\dot\SSi(A) \cap (T^*_{M\times\R}(M\times\R) \times T^*_tJ) = \emptyset$ for
    $A=F$, $G$,
  \item [(iii)] setting $\Lambda_u = i_u^\sharp(\dot\SSi(F))$ and
    $\Lambda'_u = i_u^\sharp(\dot\SSi(G))$, we have
    $\Lambda_u \cap \Lambda'_u = \emptyset$ for all $u\in J$,
  \item [(iv)] $H := \rhom(F,G)$ is $a$-periodic for some $a\in\R$ in the sense
    that $\opb{T_a}H \simeq H$, where $T_a$ is the translation
    $T_a(x,t,u) = (x,t+a,u)$.
  \end{itemize}
  Then $\RHom(\opb{i_u}F, \opb{i_u}G)$ is independent of $u \in J$.
\end{lemma}
\begin{proof}
  As in the proof of Corollary~\ref{cor:homFtGt_indpdtt} we set $H = \rhom(F,G)$
  and the microsupport estimates imply
  $\rhom(\opb{i_u}F, \opb{i_u}G) \simeq \opb{i_u}H$ for all $u\in J$.  Hence
  $\RHom(\opb{i_u}F, \opb{i_u}G) \simeq \rsect(M\times \R;\opb{i_u}H)$.

  Let $e \colon \R \to \cer = \R/a\Z$ be the projection.  We also write $e$ for
  $\id_M \times e$ or $\id_M \times e \times \id_J$.  Since $H$ is periodic
  there exists $H' \in \Der(\cor_{M\times \cer \times J})$ such that
  $H \simeq \opb{e}H'$. Since $e$ is a covering map, we have
  $\oim{e} \opb{e}H' \simeq H' \tens L_{M\times \cer \times J}$, where
  $L_{M\times \cer \times J}$ is the local system
  $L_{M\times \cer \times J} = \oim{e}(\cor_{M\times \R \times J})$. In the same
  way $\oim{e} \opb{e} \opb{j_u}H' \simeq\opb{j_u} H' \tens L_{M\times \cer}$,
  where $j_u$ is the inclusion of $M\times \cer \times \{u\}$.  Finally
  \begin{align*}
    \RHom(\opb{i_u}F, \opb{i_u}G)
    &\simeq \rsect(M\times \R;\oim{e} \opb{e} \opb{j_u}H') \\
    &\simeq \rsect(M\times \cer; \opb{j_u}( H' \tens L_{M\times \cer \times J})).
  \end{align*}
  Let $q\colon M\times \cer \times J\to J$ be the projection. Then $q$ is proper
  on $\supp(H')$ and the base change formula gives
  $\RHom(\opb{i_u}F, \opb{i_u}G) \simeq (\roim{q}(H' \tens L_{M\times \cer
    \times J}))_u$.  Hence it is enough to see that
  $\roim{q}(H' \tens L_{M\times \cer \times J})$ is locally constant, which is
  proved by microsupport estimates as in the proof of
  Corollary~\ref{cor:homFtGt_indpdtt} (we remark that $H'$ and
  $H' \otimes L_{M\times \cer \times J}$ have the same microsupport since
  $L_{M\times \cer \times J}$ is locally constant).
\end{proof}

\part{The Gromov-Eliashberg theorem}
\label{part:GEthm}

The Gromov-Eliashberg theorem (see~\cite{E87, Gr86}) says that the group of
symplectomorphisms of a symplectic manifold is $C^0$-closed in the group of
diffeomorphisms. This can be translated into a statement about the Lagrangian
submanifolds which are graphs of symplectomorphisms.  It can be deduced from the
Gromov nonsqueezing theorem but we want to stress the relation with the
involutivity theorem of Kashiwara-Schapira (stated here as
Theorem~\ref{thm:invol}).

Let us explain the idea of the proof.  We assume for a while a stronger assumption
than the Gromov-Eliashberg theorem: let $M$ be a manifold and let $\phi_n$ be a
sequence of homogeneous Hamiltonian isotopies of $\dT^*M$ which converges in $C^0$
norm to a diffeomorphism $\phi_\infty$ of $\dT^*M$. Let $K_n \in \Der(\cor_{M^2})$ be
the sheaf associated with $\phi_n$ by Theorem~\ref{thm:GKS}. Hence $\dot\SSi(K_n)$ is
$\Gamma_{\phi_n}$, the graph of $\phi_n$, and $H^0(M^2; K_n) \simeq \cor$.  We define
a kind of limit $K_\infty$ by the distinguished triangle
$\bigoplus_{n\in \N} K_n \to \prod_{n\in \N} K_n \to K_\infty \to[+1]$.  Then we can
check that $\dot\SSi(K_\infty) \subset \Gamma_{\phi_\infty}$ and
$H^0(M^2; K_\infty) \simeq \prod_{k\in \N} \cor / \bigoplus_{k\in \N} \cor$.  Hence
$K_\infty$ is not the zero sheaf. If we could prove moreover that $K_\infty$ is not
locally constant, we would deduce from the involutivity theorem that
$\Gamma_{\phi_\infty}$ is coisotropic, then that $\phi_\infty$ is a symplectic map.

The Gromov-Eliashberg theorem is a local non homogeneous version of the previous
(unproved) statement and the $\phi_n$ are only symplectic diffeomorphisms.  It is not
difficult to modify the $\phi_n$ away from a given point to turn them into
Hamiltonian isotopies and make them homogeneous by adding a variable.  Our sheaves
$K_n$ live now on $M^2 \times\R$.  The problem is that the convergence of
$\Gamma_{\phi_n}$ to $\Gamma_{\phi_\infty}$ is now true only in a neighborhood of a
point and we have no control on $\Gamma_{\phi_n}$ away from this point; more
precisely we have a subset $\Gamma'_n$ of $\Gamma_{\phi_n}$ such that $\Gamma'_n$
converges to a subset of $\Gamma_{\phi_\infty}$.  We use a cut-off lemma of
Part~\ref{chap:cutoff} to split $K_n$ in a small ball $B$ as
$K_n = K'_n \oplus K''_n$, with $\SSi(K'_n) \subset \Gamma'_n$.  The main difficulty
is to prove that the above ``limit'' of $K'_n$ is not locally constant.  For this we
restrict $K_n$ to a line $D = \{x_0\} \times \R$ and decompose $K_n$ as a sum of
constant sheaves on intervals using Corollary~\ref{cor:sheaves_dim1}, say
$K_n|_D \simeq \bigoplus_{a\in A_n} \cor_{I_a^n}[d_a^n]$.  To ensure that $K'_n$ does
not vanish when $n\to \infty$, we prove that there are intervals $I_a^n$ bigger than
$B\cap D$ as follows.  Let $\pi$ be the projection from $T^*(M^2 \times\R)$ to the
base.  If an interval $I_a^n$ is contained in $B$, the splitting
$K_n = K'_n \oplus K''_n$ prevents it from having one end in $\pi(\Gamma'_n)$ and the
other in $\pi(\Gamma_{\phi_n} \setminus \Gamma'_n)$.  In other word, the intervals
with exactly one end in $\pi(\Gamma'_n)$ are big.  Hence it is enough to see that the
projection of $\Gamma'_n$ to $M^2$ is of degree one to find a big interval.

\section{The involutivity theorem}

The main tool in our proof of the Gromov-Eliashberg theorem is the involutivity
theorem of~\cite{KS90}.  We recall its statement (see Theorem~6.5.4
of~\cite{KS90} restated here as Theorem~\ref{thm:invol}).  For a manifold $N$, a
subset $S$ of $N$ and $p\in N$, we use the notations $C_p(S)$,
$C_p(S,S) = C(S,S) \cap T_pN$ of~\eqref{eq:form_cone1} and\eqref{eq:form_cone2}
for the tangent cones of $S$ at $p$.  Let $M$ be a manifold, $\cor$ any
coefficient ring and $F\in \Der(\cor_M)$.  The involutivity theorem says that
the microsupport $S = \SSi(F)$ of $F$ is a coisotropic subset of $T^*M$ in the
sense that $(C_p(S,S))^{\perp \omega_p} \subset C_p(S)$, for all $p\in S$.  We
quote the following lemmas.
\begin{lemma}\label{lem:inclu_cois}
Let $X$ be a symplectic manifold and let $S \subset S'$ be locally closed
subsets of $X$. Let $p\in S$. We assume that $S$ is coisotropic at $p$.
Then $S'$ is also coisotropic at $p$.
\end{lemma}
\begin{proof}
This is obvious since we have the inclusions\\
$(C_p(S',S'))^{\perp\omega_p} \subset (C_p(S,S))^{\perp\omega_p} \subset
C_p(S) \subset C_p(S')$.
\end{proof}

We recall the map $\rho_M \cl T^*M\times\dT^*\R \to T^*M$,
$(x,t;\xi,\tau) \mapsto (x;\xi/\tau)$, defined in~\eqref{eq:def_rho}.
\begin{lemma}\label{lem:invol_invol}
  Let $M$ be a manifold and $S\subset T^*M$ a locally closed subset.  Let
  $p\in S$ and $q\in \opb{\rho_M}(p)$.  Then $S$ is coisotropic at $p$ if and
  only if $\opb{\rho_M}(S)$ is coisotropic at $q$.
\end{lemma}
\begin{proof}
  We use coordinates $(x,t;\xi,\tau)$ on $T^*(M\times\R)$ and corresponding
  coordinates $(X,T;\Xi,\Sigma)$ on $T_qT^*(M\times\R)$.  We set
  $S' = \opb{\rho_M}(S)$ and we write $q=(x_0,t_0;\xi_0,\tau_0)$.  We have
  $d\rho_{M,q}(X,T;\Xi,\Sigma) = (X; \frac{1}{\tau_0} \Xi -
  \frac{\xi_0}{\tau_0^2} \Sigma)$.  Since $S'$ is conic, we may assume
  $\tau_0=1$.  Using the symplectic transformations
  $(x;\xi) \mapsto (x;\xi-\xi_0)$ on $T^*M$ and
  $(x,t;\xi,\tau)\mapsto (x,t+\langle \xi_0,x\rangle; \xi-\tau \xi_0,\tau)$ on
  $T^*(M\times \R)$, which commute with $\rho_M$, we may also assume $\xi_0=0$.
  Then we have $d\rho_q (X,T;\Xi,\Sigma) = (X;\Xi)$ and we deduce
  $C_q(S') = C_p(S) \times T_{(t_0;1)}T^*\R$ and
  $C_q(S',S') = C_p(S,S) \times T_{(t_0;1)}T^*\R$.  Now the result follows
  easily.
\end{proof}

\section{Approximation of symplectic maps}

Let $(E,\omega)$ be a symplectic vector space which we identify with $\R^{2n}$.
We recall a standard application of the Alexander trick which says that a
symplectic map $\varphi \cl B_R^E\to E$ defined on some ball of $E$ coincides
with a Hamiltonian isotopy of $E$ on some smaller ball $B_r^E$.

We endow $E$ with the Euclidean norm of $\R^{2n}$.  For an open subset $U
\subset E$ and a map $\psi\cl U \to E$ we set
\begin{align}
\| \psi \|_U &= \sup \{ \| \psi(x) \|;\; x\in U \}, \\
\label{eq:norm1}
\| \psi \|^1_U &= \sup \{ \| \psi(x) \|,\, \| d\psi_x(v)\|;\;
x\in U,\, \|v\| =1 \}, \text{ if $\psi$ is $C^1$.}
\end{align}

\begin{lemma}\label{lem:appr_symplC1}
Let $R>r$ and $\varepsilon$ be positive numbers.  Let $\varphi \cl B_R^E\to E$
be a symplectic map of class $C^1$. Then there exists $R'>r$ and a symplectic
map $\psi \cl B_{R'}^E\to E$ which is of class $\Cinf$ such that
$\| \varphi - \psi \|_{B_r^E} \leq \varepsilon$.
\end{lemma}
\begin{proof}
  We set $r_1=(R+r)/2$ and we choose a (non symplectic) map $\varphi' \cl
  B_R^E\to E$ of class $\Cinf$ such that $\| \varphi - \varphi' \|^1_{B_{r_1}^E}
  \leq \varepsilon$.  We set $\omega' = \varphi'^*(\omega)$.  We have $\omega -
  \omega' = (\varphi - \varphi')^*\omega$.  Hence, if we consider $\omega$ and
  $\omega'$ as maps from $E$ to $\wedge^2E$ and we endow $\wedge^2E$ with the
  Euclidean structure induced by $E$, we have $\| \omega - \omega'
  \|_{B_{r_1}^E} \leq C\varepsilon$, where the constant $C$ only depends on $n$.

  We set $r_2 = (r_1+r)/2$.  By Moser's lemma we can find a flow
  $\Phi \cl B_{r_1}^E \times [0,1] \to E$ such that
  $\Phi_t(B_{r_2}^E) \subset B_{r_1}^E$ for all $t\in [0,1]$ and
  $\omega|_{B_{r_2}^E} = \Phi_1^*(\omega')|_{B_{r_2}^E}$.  The flow $\Phi$ is
  the flow of a vector field $X_t$ which satisfies
  $\iota_{X_t}(\omega_t) = -\alpha$ over $B_{r_1}^E$, where
  $\omega_t = t \omega' - (1-t)\omega$ and $d\alpha = \omega' - \omega$.  We can
  assume that $\alpha$ satisfies the bound
  $\| \alpha \|_{B_{r_1}^E} \leq C'\| \omega' - \omega \|_{B_{r_1}^E}$ for some
  $C'>0$ only depending on $r_1$.  Hence $X_t$ satisfies
  $\| X_t \|_{B_{r_1}^E} \leq C''\varepsilon$, for some constant $C''>0$ and all
  $t\in [0,1]$.

We may assume from the beginning that $C''\varepsilon < r_1-r_2$.
Hence $\Phi_1(B_{r_2}^E) \subset B_{r_1}^E$ and we have
$\| \Phi_1 - \id \|_{B_{r_2}^E} \leq C''\varepsilon$.
The map $\psi = \varphi' \circ \Phi_1 \cl B_{r_2}^E \to E$ is a symplectic map
such that $\| \varphi - \psi \|_{B_r^E} \leq (1+C'')\varepsilon$, which
gives the lemma (up to replacing $\varepsilon$ by $\varepsilon/(1+C'')$).
\end{proof}

\begin{proposition}\label{prop:appr_Hamisot}
  Let $R>r>0$ be given.  Let $\varphi \cl B_R^E\to E$ be a symplectic map of
  class $\Cinf$.  Then there exists a Hamiltonian isotopy $\Phi \cl E\times \R
  \to E$ of class $\Cinf$ and with compact support such that $\Phi_1|_{B_r^E} =
  \varphi|_{B_r^E}$.
\end{proposition}
\begin{proof}
  (i) Up to composing with a translation or a symplectic linear map, we may
  assume that $\varphi(0) = 0$ and $d\varphi_0 = \id_E$.

  We first show that there exists a Hamiltonian isotopy
  $\Phi \cl E\times \R \to E$ with compact support such that
  $\Phi_1^{-1} \circ \varphi = \id_E$ near $0$.  We choose a symplectic
  isomorphism $E^2 \simeq T^*\Delta$ where $\Delta$ is the diagonal. Through
  this isomorphism the graph of $\varphi$ is a Lagrangian subset, say $\Lambda$,
  of $T^*\Delta$.  Since $\varphi(0) = 0$ and $d\varphi_0 = \id_E$, the set
  $\Lambda$ is tangent to the zero section at $0$ and there exists a $1$-form
  $\alpha$ on $\Delta$ such that $\Lambda$ coincides with the graph of $\alpha$
  near $0$. Since $\Lambda$ is Lagrangian, $\alpha$ is closed and we can find a
  function $f \cl \Delta \to \R$ with compact support such that $\Lambda$
  coincides with the graph of $df$ in some neighborhood $U$ of $0$.  Up to
  restricting $U$ we can assume that $f$ is small enough in $C^2$-norm so that
  the graph $\Lambda_t$ of $t\, df$, viewed as a subset of $E^2$, is the graph
  of a diffeomorphism $\Phi_t \cl E \to E$ for all $t\in [-1,2]$.  Then $\Phi$
  is a Hamiltonian isotopy with compact support and
  $\Phi_1^{-1} \circ \varphi = \id_E$ near $0$.

  \sui (ii) By~(i) we can assume that $\varphi|_{B^E_\varepsilon} = \id_E$ for some
  small ball $B^E_\varepsilon$.  We define $U \subset E \times \mo]0,+\infty[$ and
  $\psi \cl U \to E$ by
$$
U = \{(x,t); \; \| x \| < R/t \} ,
\qquad\qquad
\psi(x,t) = t^{-1}\varphi(tx) .
$$
Then $\psi(\cdot,t)$ is a symplectic map for all $t>0$.  We let $V =
\{(\psi(x,t),t)$; $(x,t) \in U\}$ be the image of $U$ by $\psi \times \id_\R$.
Then $V$ is contractible and we can find $h\cl V \to \R$ such that $\psi$ is
the Hamiltonian flow of $h$.

We define $U_0\subset U$ by $U_0 = \{(x,t)$; $t>0$, $\| x \| < \varepsilon/t
\}$.  Since $\varphi|_{B^E_\varepsilon} = \id_E$, we have $\psi(x,t) = x$ for
all $(x,t) \in U_0$. Hence $U_0 \subset V$. Moreover $h$ is constant on $U_0$.
We can assume $h|_{U_0} =0$ and extend $h$ by $0$ to a $\Cinf$ function defined
on $V' = \mo{]}-\infty,0] \cup V$.

We set $Z = \{(\psi(x,t),t)$; $t\in{} ]0,1]$ and $\|x\| \leq r\}$.  For $t \leq
\varepsilon / r$ and $\|x\| \leq r$ we have $\psi(x,t) = x$.  Hence
$$
Z = (\ol{B_r^E} \times \mo]0, \varepsilon / r]) 
\cup (\psi \times \id_\R) (\ol{B_r^E} \times [\varepsilon / r,1])
$$
and it follows that $\ol{Z}$ is compact.  We choose a $\Cinf$ function
$g\cl E\times \R \to \R$ with compact support such that $g=h$ on $Z$.  Then the
Hamiltonian isotopy $\Phi$ defined by $g$ has compact support contained in $C$
and satisfies $\Phi_1 = \varphi$ on $B_r^E$.  This proves the proposition.
\end{proof}

\section{Degree of a continuous map}

We recall the definition of the degree of a continuous map.  Let $M,N$ be two
oriented manifolds of the same dimension, say $d$.  We assume that $N$ is
connected. We have a morphism $H^d_c(M;\Z_M) \to \Z$ and an isomorphism
$H^d_c(N;\Z_N) \isoto \Z$.
Let $f\cl M\to N$ be a proper continuous map.  Applying $H^d_c(N;\cdot)$ to the
morphism $\Z_N \to \roim{f} \opb{f}\Z_N \simeq \reim{f} \Z_M$ we find
$$
\Z \isofrom H^d_c(N;\Z_N) \to H^d_c(M;\Z_M) \to \Z .
$$ 
The degree of $f$, denoted $\deg f$, is the image of $1$ by this morphism.

\begin{lemma}\label{lem:degree}
Let $M,N$ be two oriented manifolds of dimension $d$.  We assume
that $N$ is connected.

\noindent
{\rm(i)} Let $f\cl M\to N$ be a proper continuous map
and let $V \subset N$ be a connected open subset.
Then $\deg f = \deg(f|_{\opb{f}(V)} \cl \opb{f}(V) \to V)$.

\noindent
{\rm(ii)} Let $I\subset \R$ be an interval.
Let $U \subset M\times I$, $V \subset N\times I$ be open subsets and
let $f\cl U \to V$ be a continuous map which commutes with the
projections $U\to I$ and $V \to I$. We set
$U_t = U\cap (M\times \{t\})$, $V_t = V\cap (N\times \{t\})$ and
$f_t = f|_{U_t} \cl U_t \to V_t$, for all $t\in I$. We assume that $f$
is proper and that $V$ and all $V_t$, $t\in I$, are non empty and
connected.
Then $\deg f = \deg f_t$, for all $t\in I$.
\end{lemma}
\begin{proof}
(i) and (ii) follow respectively from the commutative diagrams
$$
\begin{tikzcd}[row sep=6mm]
\Z \dar[equal] & H^d_c(N;\Z_V) \lar[swap]{\sim} \rar \dar 
  & H^d_c(M;\Z_{\opb{f}(V)}) \rar  \dar  & \Z \dar[equal]   \\
\Z & H^d_c(N;\Z_N) \lar[swap]{\sim} \rar & H^d_c(M;\Z_M) \rar  & \Z \virgdiag
\end{tikzcd}
$$
$$
\begin{tikzcd}[row sep=6mm]
\Z \dar[equal] & H^d_c(V_t;\Z_{V_t}) \lar[swap]{\sim} \rar \dar 
  & H^d_c(U_t;\Z_{U_t}) \rar  \dar  & \Z \dar[equal]   \\
\Z  & H^{d+1}_c(V;\Z_V) \lar[swap]{\sim} \rar  & H^{d+1}_c(U;\Z_U) \rar  
 & \Z \pointdiag
\end{tikzcd}
$$
\end{proof}

\begin{proposition}\label{prop:degree}
Let $B_R$ be the open ball of radius $R$ in $\R^d$. Let $U, V \subset \R^d$ be
open subsets and let $f \cl U \to B_R$, $g \cl V \to B_R$ be proper
continuous maps. We assume that there exists $r<R$ such that
$\opb{f}(\ol{B_r}) \subset U\cap V$, and that $d(f(x),g(x)) < r/2$, for all
$x\in U\cap V$.  Then $\deg f = \deg g$.
\end{proposition}
\begin{proof}
(i) We define $h\cl (U\cap V) \times [0,1] \to \R^{d+1}$ by
$h(x,t) = (tf(x) + (1-t)g(x),t)$.  Let us prove that
$\opb{h}( \ol{B_{r/2}}\times [0,1])$ is compact.
Since $\opb{f}(\ol{B_r})$ is compact and contained in $U\cap V$,
it enough to prove that
$\opb{h}( \ol{B_{r/2}}\times [0,1]) \subset \opb{f}(\ol{B_r}) \times [0,1]$.
Let $(x,t) \in (U\cap V) \times [0,1]$ be such that $\| h(x,t) \| \leq r/2$.
Since $h(x,t)$ belongs to the line segment $[f(x),g(x)]$ which is of length
$< r/2$, we deduce $f(x) \in B_r$, as required.

\medskip\noindent
(ii) We define $W =\opb{h}(B_{r/2} \times [0,1])$,
$W_t = W\cap (\R^d\times \{t\})$ for $t\in [0,1]$ and
$h'_t = h|_{W_t} \cl W_t \to B_{r/2}$.
By~(i) $h|_W \cl W \to B_{r/2} \times [0,1]$ is proper.
Hence Lemma~\ref{lem:degree}~(ii) implies that $\deg h'_0 = \deg h'_1$.
We conclude with Lemma~\ref{lem:degree}~(i) which implies
$\deg h'_0 = \deg g$ and $\deg h'_1 = \deg f$.
\end{proof}

\section{The Gromov-Eliashberg theorem}

Let $(E,\omega)$ be a symplectic vector space which we identify with $\R^{2n}$.
We endow $E$ with the Euclidean norm of $\R^{2n}$.  For $R>0$ we let $B_R^E$ be
the open ball of radius $R$ and center $0$.  For a map $\psi\cl B_R^E \to E$ we
set $\| \psi \|_{B_R^E} = \sup\{ \| \psi(x) \|;$ $x\in B_R^E \}$.

For a map $f \cl E \to E$ we denote by $i_f \colon E \to E\times E$ the
embedding $x \mapsto (x, i_f(x))$ and by $\Gamma_f = i_f(E)$ the graph of $f$.

\begin{lemma}
\label{lem:goodpos}
Let $V = \R^n$ and $E = T^*V$.  Let $f \cl E \to E$ be a map of class $C^1$ and
$0<R$ be given.  Then there exist a Hamiltonian isotopy
$\Theta \cl T^*V^2 \times \R \to T^*V^2$ with compact support, $\varepsilon>0$ and
three balls centered at $0$: $B_V \subset V^2$ and
$B^*_0 \subset B^*_1 \subset (V^*)^2$ such that for any other map
$g \cl E \to E$ with $||f - g||_{B_R^E} < \varepsilon$ we have
\begin{itemize}
\item [(a)] $\Theta_1 \circ i_f(0) = (0;0)$,
\item [(b)] $\Theta_1(\Gamma_g) \cap (B_V\times (B_1^*\setminus B_0^*)) = \emptyset$,
\item [(c)] $\Gamma'_g := \Theta_1(\Gamma_g) \cap (B_V\times B_0^*)$ is
  contained in $\Theta_1(i_g(B_R^E))$,
\item [(d)] the restriction of $\pi_{V^2}$ to $\Gamma'_g$ gives a proper map
  $\Gamma'_g \to B_V$ of degree $1$.
\end{itemize}
\end{lemma}
\begin{proof}
  (i) We set $p = (0,f(0)) \in E^2$, $q = (0;0) \in T^*V^2$ and
  $F = T_p\Gamma_f$.  We can find a symplectic map
  $\psi \cl T_pE^2 \to T_qE^2$ such that $d(\pi_{V^2})_q \cl T_qE^2 \to T_0V^2$
  induces an isomorphism $\psi(F) \isoto T_0V^2$.  We choose a Hamiltonian
  isotopy $\Theta$ such that $\Theta_1(p) = q$ and $d\Theta_1 = \psi$.

  \sui (ii) We can find balls $B_V$, $B^*_0$, $B^*_1$ such that (b-d) hold for
  $g=f$. Indeed we first choose a neighborhood $W \subset B_R^E$ of $0$ in $E$
  such that $\pi_{V^2} \circ \Theta_1 \circ i_f$ induces a diffeomorphism from $W$
  to $W' = \pi_{V^2} (\Theta_1 (i_f(W)))$. Then $\pi_{V^2}$ is a diffeomorphism
  from $\Gamma' = \Theta_1 (i_f(W))$ to $W'$ and it is easy to find $B_V$, $B_0^*$
  such that (c) and (d) hold. Up to shrinking $B_V$, $B_0^*$, we can also find
  $B_1^*$ such that (b-d) hold.

  \sui(iii) For $g$ close enough to $f$ the property~(d) holds by
  Proposition~\ref{prop:degree} (apply the proposition with $f'$, $g'$, $U'$,
  $V'$ where $f' := \pi_{V^2} \circ \Theta_1 \circ i_f$,
  $g' := \pi_{V^2} \circ \Theta_1 \circ i_g$, $U' = f'^{-1}(B_V)$,
  $V' = g'^{-1}(B_V)$).

  To check~(c), we ask the additional condition
  $\Theta_1^{-1}(B_V \times B_1^*) \subset B_R^E \times E$, which is satisfied up
  to shrinking $B_V$, $B_1^*$.

  For~(b) we choose balls $B'_V$, $B'_1$ slightly smaller than $B_V$, $B_1^*$
  and $B'_0$ slightly bigger than $B_0^*$.  Let us assume that there exists
  $x\in E$ such that $\Theta_1(x,g(x)) \in B'_V\times (B_1'\setminus B_0')$.  Then
  $(x,g(x)) \in B_R^E \times E$ and $||f(x)-g(x)|| < \varepsilon$.  For
  $\varepsilon$ small enough this implies
  $\Theta_1(x,f(x)) \in B_V\times (B_1^* \setminus B_0^*)$ which is empty.  Hence,
  up to replacing $B_V$, $B^*_0$, $B^*_1$ by $B'_V$, $B'_0$, $B'_1$, we also
  have~(b) for $g$.
\end{proof}

Now we can give a proof of the Gromov-Eliashberg rigidity theorem
(see~\cite{E87, Gr86}).

\begin{theorem}
\label{thm:GE}
Let $R>0$. Let $\varphi_k\cl B_R^E\to E$, $k\in \N$, and
$\varphi_\infty\cl B_R^E \to E$ be $C^1$ maps.  We assume
\begin{itemize}
\item [(i)] $\varphi_k$ is a symplectic map, that is,
  $\varphi_k^*(\omega) = \omega$, for all $k\in \N$,
\item [(ii)] $\| \varphi_k - \varphi_\infty \|_{B_R^E} \to 0 $ when
  $k\to \infty$,
\item [(iii)] $d \varphi_{\infty,x} \cl T_xE \to T_{\varphi_\infty(x)}E$ is an
  isomorphism, for all $x\in B_R^E$.
\end{itemize}
Then $\varphi_\infty|_{B_R^E}$ is a symplectic map.
\end{theorem}
\begin{proof}
  (i) We will prove $d\varphi_\infty|_x$ is a symplectic linear map for any
  given $x\in B_R^E$.  Up to composition with a translation we can as well
  assume $x=0$.  By Lemma~\ref{lem:appr_symplC1} and
  Proposition~\ref{prop:appr_Hamisot} we can also assume, up to shrinking $R$,
  that $\varphi_k = \Phi_1^k|_{B_R^E}$ is the restriction of (the time $1$ of) a
  globally defined Hamiltonian isotopy $\Phi^k \cl E \times \R \to E$ with
  compact support, for each $k\in \N$.

  Let us choose an isomorphism $E \simeq T^*V$ where $V= \R^n$.  We let
  $\Gamma_k \subset T^*V^2$ be the graph of $\Phi_1^k$, twisted by
  $(x,x';\xi,\xi') \mapsto (x,x';\xi,-\xi')$.  We let $\Gamma_\infty$ be the
  graph of $\varphi_\infty$ with the same twist. It is enough to prove that
  $\Gamma_\infty$ is coisotropic at $(0; \varphi_\infty(0))$.

  We assume $n\geq 2$ (the case $n=1$ is about volume preserving maps and is
  easy).  Then $\Phi^k$ can be defined by a compactly supported Hamiltonian
  function.

  \sui(ii) By Corollary~\ref{cor:quantconic} there exists
  $F_k \in \Derb(\cor_{V^2 \times \R})$ such that $\dot\SSi(F_k)$ is a conic
  Lagrangian submanifold of $\dT^*(V^2 \times \R)$ which is the graph of a
  homogeneous lift of $\Phi^k_1$.  Applying the cut-off functor $P_\gamma$
  of~\eqref{eq:defPgam}, with $\gamma = \{(0,t) \in V^2 \times\R$; $t\leq 0\}$,
  we can as well assume $\dot\SSi(F_k) \subset \{\tau>0\}$.  We recall the map
  $\rho \cl T^*V^2 \times \dT^*\R \to T^*V^2$, $(x,t;\xi,\tau) = (x;\xi/\tau)$,
  defined in~\eqref{eq:def_rho}. Then $\rho(\dot\SSi(F_k)) = \Gamma_k$.

 \sui(iii) We apply Lemma~\ref{lem:goodpos} with the function $f\colon E \to E$
 given by $f(x;\xi) = \varphi_\infty(x;\xi)^a$.  It yields a Hamiltonian isotopy
 $\Theta$ with compact support and three balls centered at $0$: $B_V \subset V^2$
 and $B^*_0 \subset B^*_1 \subset (V^*)^2$ such that the conditions~(a-d) of the
 lemma hold for $g(x;\xi) = \varphi_k(x;\xi)^a$ if $k$ is big enough.

 By Proposition~\ref{prop:homog_isot} we can lift $\Theta$ to a homogeneous
 Hamiltonian isotopy $\widetilde\Theta$ of $\dT^*(V^2 \times \R)$ such that the
 diagram~\eqref{eq:diag-PhiPsi} commutes.  Then Theorem~\ref{thm:GKS} gives
 $K \in \Derb(\cor_{(V^2 \times \R)^2})$ such that $\dot\SSi(K)$ is the graph of
 $\widetilde\Theta_1$.  We set $G_k = K \circ F_k$ and $\Lambda_k = \dot\SSi(G_k)$
 for each $k\in \N$.  Then $G_k \in \Derb(\cor_{V^2 \times \R})$ and
 $\Lambda_k = \widetilde\Theta_1(\dot\SSi(F_k))$. We still have
  $$
  \Lambda_k \subset \{\tau > 0\},
  \qquad \rho(\Lambda_k) = \Theta_1(\Gamma_k).
  $$

  \sui(iv) We can find a point $x_0 \in B_V \subset V^2$, as closed to $0$ as we
  want, such that, for any $k$, the Lagrangian submanifold
  $\Lambda_k \subset T^*(V^2 \times \R)$ is in generic position with respect to
  the line $D = \{x_0\} \times \R$ in the following sense: there exists a
  neighborhood $W_k$ of $D$ and a hypersurface $S_k$ of $W_k$ meeting $D$
  transversely such that
  $\Lambda_k\cap \dT^*W_k = \dT^*_{S_k}W_k \cap \{\tau>0\}$.  (This implies
  that $x_0$ is not on the diagonal.)  Since $\Phi^k$ and $\Theta$ have compact
  supports and $x_0$ is not on the diagonal, $S_k$ has finitely many connected
  components.

  Then $G_k|_D$ is a constructible sheaf and we can decompose it as a finite sum
  of constant sheaves on intervals
  $G_k|_D \simeq \bigoplus_{\alpha\in A_k} \cor_{I_\alpha^k}[d_\alpha^k]$ by
  Corollary~\ref{cor:sheaves_dim1}.  Since $\SSi(G_k) \subset \{\tau\geq 0\}$,
  the $I_\alpha^k$ are of the form $[a,b[$ or $\mo]-\infty,b[$ with
  $b\in \R\cup \{+\infty\}$.  The finite ends of the intervals $I_\alpha^k$ are in
  bijection with
  $$
  E_k = \Theta_1(\Gamma_k)\cap T^*_{x_0}(V^2)
  = (\Lambda_k \cap (T^*_{x_0}(V^2) \times T^*\R)) / \rspos .
$$
  
\sui(v) We set $\Gamma_k^0 = \Theta_1(\Gamma_k) \cap (B_V \times B^*_0)$ and
$E_k^0 = E_k \cap \Gamma_k^0$.  By Lemma~\ref{lem:goodpos} the map
$\Gamma_k^0 \to B_V$ is of degree $1$ and it follows that $E_k^0$ is of odd
cardinality. Hence there exists one interval $I_\alpha^k$ with one end, say $a_k$, in
$E_k^0$ and the other, say $b_k$, in $(E_k\setminus E_k^0) \cup\{\pm\infty\}$.
We translate $G_k$ vertically so that $a_k = 0$ and we shift its degree so that
$d_\alpha^k = 0$ (we still have $\rho(\Lambda_k) = \Theta_1(\Gamma_k)$).

Now $G_k|_D$ has one direct summand which is $\cor_{I_\alpha^k}$ with
$I_\alpha^k = [0,b_k[$ or $]b_k,0[$, $b_k$ finite or infinite.  One of these
possibilities occurs infinitely many times and, up to taking a subsequence, we
assume $I_\alpha^k = [0,b_k[$ with $b_k \in \R$, for all $k$, the other cases being
similar.

\sui(vi) By Lemma~\ref{lem:goodpos} again
$\rho(\dot\SSi(G_k)) \cap (B_V\times (B_1^*\setminus B_0^*)) = \emptyset$.
Hence, by Proposition~\ref{prop:cut-off_split_local2}, there exists an open ball
$W$ with center $(x_0,0)$ and radius $r$ such that, for any $k$, we have a
distinguished triangle on $W$
$$
G'_k \oplus G''_k \to  G_k|_W \to L_k \to[+1] ,
$$
where $L_k$ is a constant sheaf,
$\dot\SSi(G'_k) = \Lambda'_k$ with
$$
\Lambda'_k = \Lambda_k \cap \opb{\rho}(B_V\times B_0^*) \cap T^*W 
$$
and $\dot\SSi(G''_k) = (\Lambda_k \cap T^*W) \setminus \Lambda'_k$.  In
particular $\dot\SSi(G''_k)$ does not meet $T_{(x_0,0)}^*W$ and $G''_k$ is
constant near $(x_0,0)$.  In the same way, if $(x_0, b_k)$ belongs to $W$, then
$G'_k$ is constant near $(x_0,b_k)$.

On the other hand, if $(x_0, b_k)$ belongs to $W$,
Lemma~\ref{lem:deux_decompositions1} below implies that the direct summand
$\cor_{[0,b_k[}|_{D\cap W}$ of $G_k|_{D\cap W}$ is also a direct summand of $H$
with $H = G'_k|_{D\cap W}$ or $H = G''_k|_{D\cap W}$.  Then $H$ would be non
constant at $(x_0,0)$ and $(x_0,b_k)$, which gives a contradiction.

It follows that $(x_0,b_k) \not\in W$ and $G'_k|_{D\cap W}$ has a direct summand
which is $\cor_{[0,r[}$ (recall $r$ is the radius of $W$).

\sui (vii) We define $G \in \Der(\cor_W)$ by the distinguished triangle
\begin{equation}
  \label{eq:GErigid1}
\bigoplus_{k\in \N} G'_k \to \prod_{k\in \N} G'_k \to G \to[+1] .
  \end{equation}
For any $N \in \N$ we also have the triangle
$\bigoplus_{k \geq N} G'_k \to \prod_{k\geq N} G'_k \to G \to[+1]$.  We have
$\dot\SSi(\bigoplus_{k \geq N} G'_k) \subset \ol{\bigcup_{k \geq N} \Lambda'_k}$
and the same bound holds for $\dot\SSi(\prod_{k \geq N} G'_k)$. Hence the same
bound also holds for $\dot\SSi(G)$ and, when $N\to \infty$, we obtain
$$
\dot\SSi(G) \subset \opb{\rho}(\Theta_1(\Gamma_\infty) \cap (W \times B_0^*)) .
$$
We let $i \cl D \cap W \to W$ be the inclusion.  We remark that $i$ is
non-characteristic for $\opb{\rho}(W \times B_0^*)$. Hence $i$ is non-characteristic
for $F = G'_k$ or $F = \prod_{k \in N} G'_k$ and Theorem~\ref{thm:iminv} gives
$\opb{i}F \simeq \epb{i}F[1]$.  Since $\opb{i}$ commutes with $\bigoplus$ and
$\epb{i}$ with $\prod$, we deduce the following distinguished triangle by applying
$\opb{i}$ to~\eqref{eq:GErigid1}
$$
\bigoplus_{k\in \N} (G'_k|_D) \to \prod_{k\in \N} (G'_k|_D) \to G|_D \to[+1] .
$$
We set $l = \prod_{k\in \N} \cor / \bigoplus_{k\in \N} \cor$. Since
$\cor_{[0,r[}$ is a direct summand of $G'_k|_{D \cap W}$, the sheaf $l_{[0,r[}$
is a direct summand of $G|_{D \cap W}$.  In particular $G$ is not constant
around $(x_0,0)$ and
$\dot\SSi(G) \cap T^*_{(x_0,0)}(V^2 \times \R) \not= \emptyset$.

We choose $p \in \dot\SSi(G) \cap T^*_{(x_0,0)}(V^2 \times \R)$.  By the
involutivity Theorem and Lemma~\ref{lem:invol_invol} we obtain that
$\Theta_1(\Gamma_\infty)$ is coisotropic at $\rho(p)$.  Since
$\rho(p) \in T^*_{x_0}V^2$ and $\Theta_1(\Gamma_\infty) \cap T^*_{x_0}V^2$
consists of a single point, we have $\rho(p) = (x_0; \xi(x_0))$.  The point
$x_0$ can be chosen arbitrarily close to $0$. Hence $\Theta_1(\Gamma_\infty)$ is
coisotropic at $(0; \xi(0))$ and it follows that $\Gamma_\infty$ is coisotropic
at $(0; \varphi_\infty(0))$, as required.
\end{proof}

\begin{lemma}\label{lem:deux_decompositions1}
  Let $\cor$ be a field.  Let $I$ be an interval in $\R$ and let $a,b \in I$
  with $a<b$.  Let $F_1, F_2, F,L \in \Derb(\cor_I)$. We assume that
  $\dot\SSi(L) = \emptyset$, that we have a distinguished triangle
\begin{equation*}
 F_1 \oplus F_2 \to[u] F \to[v] L \to[+1] 
\end{equation*}
and that $\cor_{[a,b[}$ is a direct summand of $F$.  Then $\cor_{[a,b[}$ is a
direct summand of $F_1$ or $F_2$.
\end{lemma}
\begin{proof}
  By hypotheses there exist $i \cl \cor_{[a,b[} \to F$ and
  $p \cl F \to \cor_{[a,b[}$ such that $p \circ i = \id_{\cor_{[a,b[}}$.  Since
  $L$ has constant cohomology sheaves, we have $\Hom(\cor_{[a,b[},L) \simeq 0$.
  Hence $v \circ i = 0$ and there exists
  $i' = (\begin{smallmatrix} i_1 \\ i_2 \end{smallmatrix}) \colon \cor_{[a,b[}
  \to F_1 \oplus F_2$ such that $i = u \circ i'$.  Let
  $p_1 \colon F_1 \to \cor_{[a,b[}$, $p_2 \colon F_2 \to \cor_{[a,b[}$ be the
  components of $p\circ u$. Then
  $\id_{\cor_{[a,b[}} = p_1 \circ i_1 + p_2 \circ i_2$ and we deduce
  $p_1 \circ i_1 \not= 0$ or $p_2 \circ i_2 \not=0$.  Since
  $\Hom(\cor_{[a,b[}, \cor_{[a,b[}) = \cor$, we can multiply our morphisms by a
  scalar to have $p_1 \circ i_1$ or $p_2 \circ i_2$ equals $\id$, proving the
  lemma.
\end{proof}

\part{The three cusps conjecture}

In~\cite{A96} Arnol'd states a theorem of M\"obius ``a closed smooth curve
sufficiently close to the projective line (in the projective plane) has at least
three points of inflection'' and conjectures that ``the three points of
flattening of an immersed curve are preserved so long as under the deformation
there does not arise a tangency of similarly oriented branches''.  This is a
statement about oriented curves in $\RP^2$.  Under the projective duality it is
turned into a statement about Legendrian curves in the projectivized cotangent
bundle of $\RP^2$: if $\{\Lambda_s\}_{s\in [0,1]}$ is a generic path in the
space of Legendrian knots in $PT^*(\RP^2) = (T^*\RP^2 \setminus \RP^2)/\Rm$ such
that $\Lambda_0$ is a fiber of the projection $\pi \cl PT^*(\RP^2) \to \RP^2$,
then the front $\pi(\Lambda_s)$ has at least three cusps.  This statement is
given by Chekanov and Pushkar in~\cite{CP05}, where the authors prove a local
version, replacing $\RP^2$ by the plane $\R^2$, and also another similar
conjecture ``the four cusps conjecture''.  Here we prove the conjecture when the
base is the sphere $\sph$ (which implies the case of $\RP^2$ since we can lift a
deformation in $PT^*(\RP^2)$ to a deformation in $PT^*\sph = \dT^*\sph/\Rm$).

  Let us give the precise statement.  Let
$\{\ol\Lambda_s\}_{s\in [0,1]}$ be a path of Legendrians in $PT^*\sph$ starting
at $\ol\Lambda_0 = \dT^*_{x_0}\sph/\Rm$ for some point $x_0 \in \sph$.  We let
$\Sigma_s$ be the projection of $\ol\Lambda_s$ to $\sph$.
\begin{theorem}
\label{thm:threecusps-intro}
Let $s\in [0,1]$ be such that $\Sigma_s$ is a curve with only cusps and double
points as singularities. Then $\Sigma_s$ has at least three cusps.
\end{theorem}

Here is a sketch of the proof.  A first remark is that, for a connected curve
$\Sigma$ in $\sph$ with only cusps and double points as singularities and smooth
part $\Sigma_{reg}$, the closure of $\dT^*_{\Sigma_{reg}}\sph$ in $\dT^*\sph$ is
connected if and only if the number of cusps is odd.  Hence we only have to show
that $\Sigma_s$ does not have only one cusp.  For any path of Legendrians
$\{\ol\Lambda_s\}_{s\in [0,1]}$ there exists a contact isotopy
$\{\ol\Phi_s\}_{s\in [0,1]}$ of the ambient contact manifold such that
$\ol\Lambda_s = \ol\Phi_s(\ol\Lambda_0)$ (see for example Theorem~2.6.2
of~\cite{G08}).  We will apply Theorem~\ref{thm:GKS} to this contact isotopy.
To fit with the framework of this theorem we lift the isotopy
$\ol{\Phi} \cl PT^*\sph \times [0,1] \to PT^*\sph$ to a homogeneous Hamiltonian
isotopy $\Phi \cl \dT^*\sph \times [0,1] \to \dT^*\sph$ which satisfies
$\Phi_0 = \id$ and $\Phi_s(x;\lambda\,\xi) = \lambda \cdot \Phi_s(x;\xi)$ for
all $\lambda \in \Rm$ (not only for $\lambda \in \rspos$) and all
$(x;\xi) \in \dT^*\sph$.  By Theorem~\ref{thm:GKS}, for each $s\in [0,1]$ we
obtain an   auto-equivalence, say $R_{\Phi,s} = K_{\Phi,s} \circ -$, of
the category $\Der(\cor_\sph)$, with the property that
$\dot\SSi(R_{\Phi,s} (F)) = \Phi_s(\dot\SSi(F))$ for all
$F \in \Der(\cor_\sph)$.  We set $F_0 = \cor_{\{x_0\}}$ and
$F_s = R_{\Phi,s}(F_0)$. Hence $\dot\SSi(F_0) = \dT^*_{x_0}\sph$ and it follows
that $\dot\SSi(F_s) = \Phi_s(\dT^*_{x_0}\sph)$.

The fact that $\Phi$ is homogeneous for the action of $\Rm$, and not only $\rspos$,
implies that $R_{\Phi,s}$ commutes with the duality functor $\DD_\sph$.  The sheaf
$F_0$ is self-dual and simple. Hence so is $F_s$.  Now we prove the following result
(see Theorem~\ref{thm:ext1FF} and the beginning of the proof of
Theorem~\ref{thm:threecusps}): if   $F \in \Der(\cor_\sph)$ satisfies
\begin{itemize}
\item [(a)] $F$ is self-dual,
\item [(b)] $F$ is simple,
\item [(c)] $\rsect(\sph; F) \simeq \cor$,
\item [(d)] $\Lambda = \dot\SSi(F)/\rspos$ is a smooth curve whose projection to
  $\sph$ is a generic curve $\Sigma$ with only one cusp,
\end{itemize}
then $\Hom(F,F[1]) \not=0$.  Since $\Hom(F_0,F_0[1])=0$ there cannot exist an
auto-equivalence $R$ of $\Der(\cor_\sph)$ such that $R(F_0) = F$.

\smallskip

We describe the hypotheses of Theorem~\ref{thm:ext1FF} and see at the same time how
(a-d) imply them.  We choose a Morse function $q \cl \sph \to \R$ with only two
critical points and sufficiently generic with respect to $\Sigma$.  We decompose
$\roim{q}(F)$ using Corollary~\ref{cor:sheaves_dim1}.  The hypothesis~(c) implies
that all intervals, but one, appearing in the corollary are half-closed.  Then~(a)
implies that the non half-closed interval is reduced to a point.  We thus have
$\roim{q}(F) \simeq \cor_{\{t_0\}} \oplus \bigoplus_{a\in A} \cor^{n_a}_{I_a}[d_a]$
where the $I_a$ are half-closed intervals, $d_a\in \Z$ and $t_0$ is some point in
$\R$.  We can assume $t_0 = 0$ and, choosing $q$ sufficiently generic with respect to
$\Sigma$, we can assume that $0 \not\in \ol{I_a}$, for all $a$.  Then the
decomposition of $\roim{q}(F)$, restricted to a neighborhood of $0$, gives the
hypothesis~\eqref{eq:decomp_FJ2} of Theorem~\ref{thm:ext1FF} (up to rescalling), that
is, $\roim{q}(F)|_{\mo]-1,1[} \simeq \cor_{\{0\}} \oplus B_{\mo]-1,1[}$ for some
$B \in \Derb(\cor)$.

We write $C_t = q^{-1}(t)$.  The decomposition of $\roim{q}(F)$ and
Proposition~\ref{prop:oim} imply that $\Sigma$ is tangent to $C_0$.  For $q$ generic,
we can choose a diffeomorphism $q^{-1}(\mo]-1,1[) \simeq \cer \times \mo]-1,1[$ such
that $\Sigma \cap q^{-1}(\mo]-1,1[)$ is the union of one branch, say $\GammaE_0$,
tangent to $C_0$ and contained in $q^{-1}([0,1[)$, and other branches of the form
$\{\theta\} \times \mo]-1,1[$ (see~\eqref{eq:notJELambda0},
\eqref{eq:hypLambda_gentang} and Fig.~\ref{fig:hypLambda}).  We are now in the
settings of Theorem~\ref{thm:ext1FF}, which says that, if $\Sigma$ only has one cusp
(hypothesis~(d) above), then $\Hom(F,F[1]) \not= 0$.

The proof of Theorem~\ref{thm:ext1FF} distinguishes two cases and uses two criteria
to ensure the non-vanishing of $\Hom(F,F[1])$.  We decompose $F|_{C_{1/2}}$ using
Corollary~\ref{cor:sheaves_dim1} as
$F|_{C_{1/2}} \simeq L \oplus \bigoplus_{a\in A'} \oim{e}(\cor_{I'_a})[d'_a]$, where
$L$ is locally constant, the $I'_a$ are intervals of $\R$, $d'_a \in \Z$ and
$e \cl \R \to C_{1/2} \simeq \R/2\pi\Z$ is the quotient map.  The branch $\GammaE_0$
meets $C_{1/2}$ in two points, say $\theta_1$, $\theta_2$, and some intervals $I'_a$
have one end in $e^{-1}(\{\theta_1, \theta_2\})$ (a priori there are four such
intervals but some could coincide).  Now the two cases depend on these intervals.

If one of the intervals with one end in $e^{-1}(\{\theta_1, \theta_2\})$ is
closed or open, then the non vanishing of $\Hom(F,F[1])$ is given by
Propositions~\ref{prop:interv_ferme} and~\ref{prop:FC0-monodunip}.  This result
is in fact local around $C_0$ and we do not use the fact that $\Sigma$ only has
one cusp.

In the other case all intervals with one end in $e^{-1}(\{\theta_1, \theta_2\})$ are
half-closed.  Here we apply a more global criterion to ensure $\Hom(F,F[1]) \not=0$.
We define notions of $F$-linked and $F$-conjugate points of $\Lambda$ in
Section~\ref{sec:micro_linked}. Our criterion is roughly that, if there exist a pair
$(p_0,p_1)$ of $F$-linked points and a pair $(q,q')$ of $F$-conjugate points such
that $(p_0,p_1)$ and $(q,q')$ are intertwined on the circle $\Lambda/\rspos$, then
$\Hom(F,F[1]) \not=0$ (see Proposition~\ref{prop:mongammasurj}).  Examples of
conjugate points are the points of $\Lambda$ corresponding to the ends of one
interval $I_a$ in the decomposition of $F|_{C_{1/2}}$ recalled above.  In
Proposition~\ref{prop:otherFlinkedpoint} we see that in our situation we can situate
some $F$-linked points.  The hypothesis that $\Sigma$ has only one cusp is used as
follows.    The sheaf $F$ has a {\em shift} (or
{\em Maslov potential} -- see Example~\ref{ex:shift}) at each point of
$\Lambda$, which is a half integer and changes by $1$ when we cross a cusp.  If
$\Sigma$ has only one cusp, then $\Lambda/\rspos$ is decomposed in two arcs, say
$\Lambda_+$ and $\Lambda_-$, according to the value of the shift.  In the above
decomposition of $F|_{C_{1/2}}$, the points above the ends of an interval $I'_a$
which is half-closed are not in the same component $\Lambda_\pm$.  This gives several
pairs of conjugate points $(q_+,q_-)$ with $q_\pm \in \Lambda_\pm$.  Using the linked
points of Proposition~\ref{prop:otherFlinkedpoint} it is then possible to find two
intertwined pairs of linked/conjugate points and apply
Proposition~\ref{prop:mongammasurj}.

\smallskip

  This part is organized as follows.
In Section~\ref{sec:exemples_trois_cusps} we describe the sheaf associated with
the standard curve with three cusps and give an example of a curve with one cusp
whose conormal bundle is the microsupport of a sheaf.  In
Section~\ref{sec:generic_tgt} we give the first criterion (local around $C_0$)
for the non vanishing of $\Hom(F,F[1])$.  In Section~\ref{sec:micro_linked} we
introduce the notions of linked and conjugate points and prove that the
existence of two intertwined pairs of linked/conjugate points also implies the
non vanishing of $\Hom(F,F[1])$.  Examples of conjugate points are given in the
next section.  In Section~\ref{sec:trois_points_lies} we prove the existence of
three linked points (under the same hypotheses as in
Section~\ref{sec:generic_tgt}).  In the next two sections we apply these
criteria to prove the three cusps conjecture.  In the last section we give a
sketch of proof of the four cusp conjecture with the same method.

\smallskip

In this part we assume that $\cor$ is an infinite field (we use it in
Proposition~\ref{prop:otherFlinkedpoint} and we use Gabriel's theorem).

\section{Examples}
\label{sec:exemples_trois_cusps}

A first idea to prove Theorem~\ref{thm:threecusps-intro} could be the more
ambitious statement: if $\Lambda \subset \dT^*\sph$ is a smooth conic Lagrangian
submanifold and $\Sigma = \pi_{\sph}(\Lambda)$ is a curve with only one cusp and
otherwise only double points as singularities, then there is no
$F \in \Der(\cor_{\sph})$ such that $\dot\SSi(F) = \Lambda$.  Indeed
Theorem~\ref{thm:GKS} would imply that there is no homogeneous isotopy $\Phi$ of
$\dT^*\sph$ such that $\Lambda = \Phi_1(\dT^*_{x_0}\sph)$.  But this statement
is false; we give a counterexample in this section.  The examples given here are
in fact in $\R^2$.  We put coordinates $(x,y)$ on $\R^2$ and $(x,y;\xi,\eta)$ on
$T^*\R^2$.

\subsubsection*{Sheaf associated with a cusp}

Let $\Sigma_{cusp} = \{(x,y)$; $x^2 = y^3 \}$ be the ordinary cusp in $\R^2$.
Let $\Lambda_{cusp}$ be the closure of
$\dT^*_{\Sigma_{cusp} \setminus \{0\}}(\R^2 \setminus \{0\})$ in $\dT^*\R^2$.
Then $\Lambda_{cusp}$ is a smooth conic Lagrangian submanifold of $\dT^*\R^2$
consisting of two connected components, say $\Lambda_1 = \{(t^3,t^2; -2u,3tu)$;
$t\in\R$, $u\in \rspos\}$ and $\Lambda_2 = \Lambda_1^a$.  We set
$W_1 = \{(x,y)$; $-y^{3/2} < x \leq y^{3/2}$, $y>0 \}$ and $W_2 = \{(x,y)$;
$-y^{3/2} \leq x < y^{3/2}$, $y>0 \}$ (see Fig.~\ref{fig:pictcusp1}).  By
Example~5.3.4 of~\cite{KS90} we know that $\dot\SSi(\cor_{W_i}) = \Lambda_i$,
for $i=1,2$ (outside $T^*_0\R^2$ this follows from
Example~\ref{ex:microsupport}-(iii) and a direct computation gives the
microsupport over $0$).

\begin{figure}[ht]
 \begin{tikzpicture}[decoration={border, segment length=2mm, angle=90}]

\fill [fill=gray!20] plot [domain=-1:1] (\x^3, \x*\x);
\draw [postaction={decorate,draw}] plot [domain=0:1] (\x^3, \x*\x);
\draw [decorate] plot [domain=0:-1] (\x^3, \x*\x);
\draw [dotted] plot [domain=0:-1] (\x^3, \x*\x);
\node at (0,.7) {$W_1$};

\fill [fill=gray!20] plot [domain=-1:1] (3+\x^3, \x*\x);
\draw [postaction={decorate,draw}]  plot [domain=-1:0] (3+\x^3, \x*\x);
\draw [decorate] plot [domain=1:0] (3+\x^3, \x*\x);
\draw [dotted] plot [domain=1:0] (3+\x^3, \x*\x);
\node at (3,.7) {$W_2$};
\node at (7,.7) {$F = \cor_{W_1}[d_1] \oplus \cor_{W_2}[d_2]$};
\end{tikzpicture}

\bigskip
\begin{tikzpicture}

  \draw  plot [domain=-1:1] (\x^3, \x*\x);
  \draw (-1.5,.4) -- (1.1,.4) ;
  \node at (-1.5,.7) {$L$};
  \node at (-.5,.2) {$z$};
  \node at (.5,.18) {$z'$};

  \pgftransformxshift{3.5cm}
  \draw (-1.5,.4) -- (1.1,.4) ;
  \draw (-.3,.3) -- (-.3,.5);
  \draw (.3,.3) -- (.3,.5);

  \node at (-1.5,.7) {$L$};
  \node at (-.5,.2) {$z$};
  \node at (.5,.18) {$z'$};  
  
  \fill (-.3,.7)  circle[radius=.04cm] ;
  \draw (-.3,.7) -- (.3,.7) arc(180:90:.04) arc(90:270:.04) ;
  \fill (.3,.1)  circle[radius=.04cm] ;
  \draw (.3,.1) -- (-.3,.1) arc(0:90:.04) arc(90:-90:.04) ;
  
\node at (3.5,.4) {$F|_L \simeq \cor_{[z,z'[}[d_1]\oplus \cor_{]z,z']}[d_2]$};  
\end{tikzpicture}
\caption{}   \label{fig:pictcusp1}
\end{figure}

\begin{lemma}\label{lem:sheafcusp}
  Let $F \in \Der(\cor_{\R^2})$ be such that $\dot\SSi(F) = \Lambda_{cusp}$ and $F$ is
  simple along $\Lambda_{cusp}$.  Then there exist $E\in \Der(\cor)$ and
  $d_1, d_2\in \Z$ such that
  $F \simeq \cor_{W_1}[d_1] \oplus \cor_{W_2}[d_2] \oplus E_{\R^2}$.
\end{lemma}

With the notations of Lemma~\ref{lem:sheafcusp}, if $\supp(F) \subset \ol{W_1}$
and $L$ is a horizontal line cutting $\Sigma_{cusp}$ in two points $z, z'$, then
$F|_L \simeq \cor_{[z,z'[}[d_1] \oplus \cor_{]z,z']}[d_2]$.

\begin{proof}
  This follows from the description of sheaves on the plane in~\cite{STZ17} (see
  for example (3.3) in the proof of Thm.~3.12 or Prop.~5.8) but we give a sketch
  of proof for the reader convenience.

  \sui(i) We first consider the case where $\dot\SSi(F) = \Lambda_1$ and
  $\supp(F)$ is contained in $\ol{W_1}$. The set $U = W_2 \cup \{x<0\}$ is open
  with a boundary of class $C^1$. We set $Z = \R^2 \setminus U$.  The
  microsupport condition gives $(\rsect_Z(F))_p \simeq 0$ for all
  $p\in \partial U$, and then for all $p\in Z$ because
  $\supp(F) \subset \ol{U}$. Hence $\rsect_Z(F) \simeq 0$ and we obtain, by
  excision, $F \simeq \rsect_U(F) \simeq \roim{j}(F|_U)$, where $j$ is the
  inclusion of $U$ in $\R^2$.  Using Example~\ref{ex:SS=conormal_hypersurface}
  and the fact that $F$ is simple, we see that $F|_U \simeq \cor_{U,W_1}[d_1]$
  for some $d_1 \in \Z$. We deduce $F\simeq \cor_{\R^2,W_1}[d_1]$.

  \sui(ii)   Now we assume $\dot\SSi(F) = \Lambda_1$
  but we assume nothing on $\supp(F)$.  We set $q = (0,-1)$ and let $B_r$ be the open
  disc with center $q$ and radius $r$.  By Corollary~\ref{cor:Morse} we have
  isomorphisms $\rsect(B_r; F) \isoto \rsect(B_s; F)$ for all $r\geq s >0$.  We
  deduce $\rsect(\R^2;F) \isoto F_q$. We set $E'= \rsect(\R^2;F)$. We have a natural
  morphism $u\colon E'_{\R^2} \to F$ and $u_q$ is an isomorphism. Hence the cone of
  $u$, say $F'$, satisfies $\dot\SSi(F') = \Lambda_1$ and $F'_q \simeq 0$.  By~(i) we
  have $F'\simeq \cor_{W_1}[d_1]$ for some $d_1 \in \Z$. We can check that
  $\Hom(F', E'_{\R^2}[1]) \simeq 0$ and it follows that
  $F \simeq F' \oplus E'_{\R^2}$.

  \sui(iii) Now we assume $\dot\SSi(F) = \Lambda_{cusp}$.  By
  Proposition~\ref{prop:cut-off_split_local2} there exists a neighborhood $V$ of
  $0$ and a distinguished triangle in $V$,
  $F_1 \oplus F_2 \to F|_V \to E''_{\R^2}\to[+1]$, where
  $\dot\SSi(F_i) = \Lambda_i$ and $E''_{\R^2}$ is a constant sheaf.  We can find
  an isotopy from $\R^2$ to a small neighborhood of $0$ which preserves
  $\Lambda_{cusp}$. Hence Proposition~\ref{prop:iminvproj} implies that the
  distinguished triangle can be extended to $\R^2$. By~(ii) we know $F_1$ and
  $F_2$ and, using $\Hom(\cor_{\R^2}, F_i[1]) \simeq 0$, we deduce the result.
\end{proof}

\subsubsection*{Sheaf associated with a deltoid}

We deform $\dT^*_0\R^2$ by the Hamiltonian flow, say $\Phi$, of the function
$h(x,y;\xi,\eta) = \eta^3/(\xi^2+\eta^2)$ (which is one of the simplest example
of function homogeneous of degree $1$).  We set
$\Lambda_{delt} = \Phi_1(\dT^*_0\R^2)$ and
$\Sigma_{delt} = \pi_{\R^2}(\Lambda_{delt})$.  Then $\Sigma_{delt}$ is a curve
which bounds a star shaped domain, say $D$, and which has three cusps, all
pointing to the outward direction ($\Sigma_{delt}$ is a kind of deltoid):
$$
\begin{tikzpicture}

\draw plot [domain=0:3.14] (  {-2*cos(\x r)*cos(\x r)*cos(\x r)*sin(\x r)},
  {cos(\x r)*cos(\x r)*cos(\x r)*cos(\x r) +3*sin(\x r)*sin(\x r)*cos(\x r)*cos(\x r)} );
\end{tikzpicture}
$$
Let $K_\Phi$ be the sheaf associated with $\Phi$ by Theorem~\ref{thm:GKS}.  The
composition with $K_{\Phi,1}$ gives an equivalence between simple sheaves with
microsupport $\dT^*_0\R^2$ and simple sheaves with microsupport
$\Lambda_{delt}$.  We set $F = K_{\Phi,1} \circ \cor_{\{0\}}$. Then
$\dot\SSi(F) = \Phi_1(\dot\SSi(\cor_{\{0\}})) = \Lambda_{delt}$ and $F$ is
simple along $\Lambda_{delt}$.  Since $\cor_{\{0\}}$ has compact support, the
same holds for $F$. Since $\dot\SSi(F) = \Lambda_{delt}$, $F$ is locally
constant outside $\Sigma_{delt}$.  Hence $F$ must be zero outside $\ol D$.  By
Lemma~\ref{lem:sheafcusp} we have $\supp(F) = \ol D$ and, if $L$ is a line cutting
$\Sigma_{delt}$ transversely in two points $z, z'$, then
$F|_L \simeq \cor_{[z,z'[}[d_1] \oplus \cor_{]z,z']}[d_2]$ for some integers
$d_1$, $d_2$.

\begin{lemma}\label{lem:sheafdeltoid}
  Let $F = K_{\Phi,1} \circ \cor_{\{0\}} \in \Der(\cor_{\R^2})$ be the simple
  sheaf along $\Lambda_{delt}$ corresponding to $\cor_{\{0\}}$.  Then $F$ is
  concentrated in degree $-1$ and its restriction to a line cutting transversely
  $\Sigma_{delt}$ at $z,z'$ is $(\cor_{[z,z'[} \oplus \cor_{]z,z']})[1]$.
\end{lemma}
We summarize these results in the following picture (where the deltoid is slightly
deformed)
\begin{figure}[ht]
\begin{tikzpicture}    
\draw (0,0) .. controls (0,1) and (-1,1) .. (-1,2) ; 
\draw (0,0) .. controls (0,1) and (1,1) .. (1,2) ;  
\draw (-.6,2) arc[radius=.6cm, start angle=180, end angle=360]; 
\draw (-1,2) .. controls (-1,2.3) and (-.8,2.4) .. (-.8,2.7); 
\draw (-.6,2) .. controls (-.6,2.3) and (-.8,2.4) .. (-.8,2.7); 
\draw (1,2) .. controls (1,2.3) and (.8,2.4) .. (.8,2.7); 
\draw (.6,2) .. controls (.6,2.3) and (.8,2.4) .. (.8,2.7); 
\draw [dotted] (-1.5,.7) rectangle (1.5,1.9);
\node at (-1.3,.9) {$B$};
\draw (-1.5,2.1) -- (-.3,2.1);
\draw (1.5,2.1) -- (.3,2.1);
\draw (-.6,.5) -- (.6,.5);
\node at (-.5,.2) {$L_1$};
\node at (-1.3,2.4) {$L_2$};
\node at (1.5,2.3) {$L_3$};
\node at (5,1.3) 
{$(F[-1])|_{L_i} \simeq \cor_{\tikz{\path (-.1,0); \fill (0,0)  circle[radius=.04cm] ;
    \draw (0,0) -- (.5,0) arc(180:90:.04) arc(90:270:.04) ;}}
\oplus
\cor_{\tikz{\path (-.1,0); \fill (0.5,0)  circle[radius=.04cm] ;
    \draw (.5,0) -- (0,0) arc(0:90:.04) arc(90:-90:.04) ;}}$};
\end{tikzpicture}
\caption{}  \label{fig:deltoid2} 
\end{figure}

\begin{proof}
  We have already seen that the restriction of $F$ to a transverse line is of
  the form $\cor_{[z,z'[}[d_1] \oplus \cor_{]z,z']}[d_2]$ for some integers $d_1$,
  $d_2$.

  Let us first prove that $d_1 = d_2$.  Let $E_1$, $E_2$, $E_3$ be the edges of the
  domain $D$.   We set
  $U_i = \R^2 \setminus (E_{i-1} \cup E_{i+1} )$ (with $E_4= E_1$).  Then $F|_{U_i}$
  is of the form
  $F^{d_1,d_2}_i := \cor_{U_i \cap D}[d_1] \oplus \cor_{U_i \cap \ol D}[d_2]$ or
  $F^{d_2,d_1}_i$. Moreover, if $F|_{U_i} \simeq F^{d_1,d_2}_i$, then
  $F|_{U_{i+1}} \simeq F^{d_2,d_1}_{i+1}$.  Turning once around $D$ we obtain
  $F^{d_1,d_2}_1 \simeq F^{d_2,d_1}_1$, which gives $d_1 = d_2$.
  
    Now we prove that $d_1=1$.  Let $q \colon \R^2 \to \R$ be the projection
    $q(x,y) = y$.  By Corollary~\ref{cor:isot_equivcat} we know that
    $\rsect(\R^2;F) \simeq \rsect(\R^2;\cor_{\{0\}}) \simeq \cor$.  We compute
    $\rsect(\R^2;F)$ using $\rsect(\R^2;F) \simeq \rsect(\R; \roim{q}(F))$ and
    we will deduce $d_1$.

    The line $q^{-1}(y)$ is transverse to $\Sigma_{delt}$ except for one value
    $y=y_0$.  For $y\not= y_0$, $F|_{q^{-1}(y)}$ is a sheaf on $\R$ which is the
    sum of two or four sheaves of the type $\cor_I[d_1]$, where $I$ is a half
    closed interval.  Hence
    $(\roim{q}(F))_y \simeq \rsect(q^{-1}(y); F|_{q^{-1}(y)}) \simeq 0$ for
    $y\not=y_0$.  It follows that
    $\rsect(\R; \roim{q}(F)) \simeq (\roim{q}(F))_{y_0}$.  Let $z,z'$ be the
    transverse intersections of $\Sigma_{delt}$ and $q^{-1}(y_0)$ and let $z''$
    be their tangent intersection.  Then $F|_{q^{-1}(y_0)}$ is of the type
  \begin{gather*}
    (\cor_{[z,-[} \oplus \cor_{]z,-[})[d_1]  \quad \text{near $z$} , \qquad
    (\cor_{]-,z'[} \oplus \cor_{]-,z']})[d_1]  \quad \text{near $z'$} , \\
    (\cor_{]-,z''[} \oplus \cor_{]z'',-[} \oplus \cor_\R)[d_1]
                                           \quad \text{near $z''$.} 
  \end{gather*}
  By Corollary~\ref{cor:sheaves_dim1} we have
  $F|_{q^{-1}(y_0)} \simeq (\cor_{I_1} \oplus \cor_{I_2} \oplus \cor_{I_3})[d_1]$,
  where the intervals $I_1,I_2,I_3$ have two closed ends and four open ends. This
  leaves two possibilities: (1) two of these intervals are half-closed and one is
  open, or (2) two are open and one is closed.  In the first case we find
  $(\roim{q}(F))_{y_0} \simeq \cor[d_1-1]$ and in the second case
  $(\roim{q}(F))_{y_0} \simeq \cor^2[d_1-1] \oplus \cor[d_1]$.  Since the result is
  $\cor$, this excludes the second case and we obtain $d_1=1$.
\end{proof}

\subsubsection*{Sheaf associated with a curve with one cusp}

We define a new curve $\Sigma$ from $\Sigma_{delt}$ as follows.  We cut
$\Sigma_{delt}$ by the rectangle $B$ of Fig.~\ref{fig:deltoid2}; the bottom edge
$\partial^-B$ of $B$ cuts $\Sigma_{delt}$ in one pair of points near the bottom
cusp and the top edge $\partial^+B$ of $B$ cuts $\Sigma_{delt}$ in two pairs of
points near the top cusps.  We put two small copies of $B$, say $B_2$ and $B_3$,
above another copy $B_1$, so that the corresponding copies of
$\Sigma_{delt} \cap B$ glue together, as in Fig.~\ref{fig:onecusp}.  We attach a
cusp to the pair of points of $\Sigma_{delt} \cap \partial^- B_1$ and we attach
four arcs of circles, say $C_1, \dots, C_4$, to the four pairs of points of
$\Sigma_{delt} \cap (\partial^+ B_2 \cup \partial^+ B_3)$ so that the resulting
curve, say $\Sigma$, is connected.  The closure of the conormal of the regular
part of $\Sigma$ is a smooth conic Lagrangian submanifold of $\dT^*\R^2$, say
$\Lambda$.

\begin{figure}[ht]
\begin{tikzpicture}[scale=1.2]
\draw (0,0) .. controls (0,1) and (-2.6,1) .. (-2.6,3) ;
\draw (0,0) .. controls (0,1) and (2.6,1) .. (2.6,3) ;

\draw (-.6,2) -- (-.6,3);
\draw (.6,2) -- (.6,3);

\draw (-.6,2) arc[radius=.6cm, start angle=180, end angle=360];
\draw (-2.2,3) arc[radius=.6cm, start angle=180, end angle=360];
\draw (1,3) arc[radius=.6cm, start angle=180, end angle=360];

\draw (-2.6,3) arc[radius=1.8cm, start angle=180, end angle=0];
\node [fill=white] at (-2,4.5) {$C_1$};
\draw (-2.2,3) arc[radius=1.4cm, start angle=180, end angle=0];
\node [fill=white] at (-1.9,3.6) {$C_2$};
\node at (-1.2,4.55) {$D$};

\draw (2.6,3) arc[radius=1.8cm, start angle=0, end angle=180];
\node [fill=white] at (2,4.5) {$C_3$};
\draw (2.2,3) arc[radius=1.4cm, start angle=0, end angle=180];
\node [fill=white] at (1.9,3.6) {$C_4$};
\node at (1.2,4.55) {$D'$};

\draw [dotted] (-3.2,.5) rectangle (3.2,2);
\draw [dotted] (-3.2,2) rectangle (-.2,3);
\draw [dotted] (3.2,2) rectangle (.2,3);

\node at (-2.9,.8) {$B_1$}; \node at (-4.5,-.2) {$\partial^- B_1$};
\draw [->] (-3.8,-.2) -- (-2,0.4);
\node at (-2.9,2.3) {$B_2$}; \node at (-4.5,3.7) {$\partial^+ B_2$};
\draw [->] (-3.8,3.7) -- (-2.9,3.1);
\node at (2.9,2.3) {$B_3$}; \node at (4.3,3.7) {$\partial^+ B_3$};
\draw [->] (3.7,3.7) -- (3,3.1);
\end{tikzpicture}
\caption{}  \label{fig:onecusp} 
\end{figure}

Let $F$ be the sheaf described in Lemma~\ref{lem:sheafdeltoid}.  Let $F_i$ be a copy
of $F|_B$ in the rectangle $B_i$.  By the description of $F$, the sheaves $F_1$,
$F_2$, $F_3$ glue into a sheaf, $G$, on $B_1 \cup B_2 \cup B_3$. We can then glue $G$
with a sheaf associated with the cusp to obtain a sheaf $G'$.   We let $D \simeq C_1 \times \mo]0,1[$ be the domain bounded by $C_1$,
$C_2$ and we set $D_1 = D \cup C_1$, $D_2 = D \cup C_2$.  We define similarly the
strips $D'$, $D'_3$, $D'_4$ associated with the two other arcs $C_3, C_4$.  Now we
glue $G'$ with the sheaf
$(\cor_{D_1} \oplus \cor_{D_2} \oplus \cor_{D'_3} \oplus \cor_{D'_4})[1]$ and we
obtain a sheaf $F'$ such that $\dot\SSi(F') = \Lambda$.

\section{Simple sheaf at a generic tangent point}
\label{sec:generic_tgt}

In this section we consider a sheaf $F$ on a surface $M$ whose microsupport
$\dot\SSi(F) = \Lambda \subset \dT^*M$ is a smooth Lagrangian submanifold.  We
give a criterion for the non vanishing of $\Hom(F,F[1])$ (see
Propositions~\ref{prop:interv_ferme} and~\ref{prop:FC0-monodunip}).  The
geometric situation and the hypotheses on $F$ are described in
Section~\ref{sec:threecusps_nothyp}.  Our criterion is in fact local around an
embedded circle of $M$. We begin with a general statement to go from a local non
vanishing to a global one (Lemma~\ref{lem:cohom_supp}).

\subsection{Local cohomology}

We make general remarks about the (local) extensions of sheaves.    Let $M$ be a manifold and $Z$ a closed
subset of $M$.  Let $F,G \in \Der(\cor_M)$ be given.  We will often use the
isomorphisms (see Proposition~\ref{prop:formulaire}-(a-c))  
\begin{equation}
  \label{eq:derivedHom}
 \Hom(F,G[i]) \simeq   H^i\RHom(F,G) \simeq H^i(M; \rhom(F,G)) ,
\end{equation}
where $\RHom$ denotes the derived $\Hom$ functor, with values in $\Der(\cor)$,
and $\Hom$ denotes the $\Hom$ functor in $\Der(\cor_M)$, with values in
$\Mod(\cor)$.  We also have the following isomorphisms, deduced
from~\eqref{eq:cohom_support} and the adjunction $(\otimes, \rhom)$
\begin{equation}
  \label{eq:adj_cohom_supp}
\rsect_Z\rhom(F,G) \simeq \rhom(F,\rsect_Z(G)) \simeq \rhom(F_Z,G) .
\end{equation}
Since $\rsect_M(F) \simeq F \simeq F_M$, the morphism $\cor_M \to \cor_Z$
induces the natural morphisms~\eqref{eq:rem_cohom_supp1} and
then~\eqref{eq:rem_cohom_supp2} by composition
\begin{gather}
  \label{eq:rem_cohom_supp1}
i_Z(F) \colon \rsect_Z(F) \to F, \qquad j_Z(F) \colon F \to F_Z , \\
  \label{eq:rem_cohom_supp2}
  H^i\RHom(F_Z,\rsect_Z(G)) \to   H^i\RHom(F,G) .
\end{gather}

The next result gives conditions about some morphisms of sheaves supported on
$Z$ to ensure the non triviality of $H^i\RHom(F,G)$.

\begin{lemma}\label{lem:cohom_supp}
  Let $i\in\Z$ and $u \colon F_Z \to \rsect_Z(G)[i]$ be given.  We assume that
  $j_Z(G) \circ i_Z(G) \circ (u[-i]) \not=0$. Then the image of $u$ in
  $H^i\RHom(F,G)$ by~\eqref{eq:rem_cohom_supp2} is non zero.
\end{lemma}
\begin{proof}
  Composing with $j_Z(F)$, $i_Z(G)$ and $j_Z(G)$ gives the commutative diagram
$$
\begin{tikzcd}[column sep=3cm]
  \RHom(F_Z,\rsect_Z(G)) \ar[r, "i_Z(G) \circ \; - \; \circ j_Z(F)"]
  \ar[d, "j_Z(G) \circ i_Z(G) \circ \; -"]
  & \RHom(F,G) \ar[d,  "j_Z(G) \circ \; -"] \\
\RHom(F_Z,G_Z) \ar[r, "\sim", "- \; \circ j_Z(F)"'] & \RHom(F,G_Z) ,
\end{tikzcd}
$$ 
where the top arrow is~\eqref{eq:rem_cohom_supp2} and the left arrow is
composing with $j_Z(G) \circ i_Z(G)$.  Since $G_Z$ is supported in $Z$, we have
$\rsect_Z(G_Z) \simeq G_Z$ and the isomorphisms~\eqref{eq:adj_cohom_supp} show
that the bottom arrow is an isomorphism. Taking $H^i(-)$ gives the result.
\end{proof}

\subsection{Generic tangent point - notations and hypotheses}
\label{sec:threecusps_nothyp}

   We consider a surface $M$, a smooth closed conic Lagrangian
submanifold $\Lambda \subset \dT^*M$ and $F\in \Derb(\cor_M)$ such that
$\dot\SSi(F) = \Lambda$.  We introduce some hypotheses on $\Lambda \cap T^*\Uu$ and
$F|_\Uu$, where $\Uu$ is an open subset of $M$ diffeomorphic to a cylinder.

Since these hypotheses are not  
completely obvious, we give a quick justification why we introduce them (see also the
introduction and Lemma~\ref{lem:autodual}).  Let us assume for a while (as we will do
when we apply the results of this section) that $\Sigma= \dot\pi_{M}(\Lambda)$ is an
immersed curve, with only cusps and transverse double points as singularities.  We
choose a Morse function $q \colon M \to \R$ generic enough so that the fibers
$q^{-1}(t)$ are tangent to $\Sigma$ only for finitely many values $t_1,\dots,t_N$ and
that $q^{-1}(t_i)$ only contains one tangent point and no cusp nor double points.  By
Corollary~\ref{cor:sheaves_dim1} there exists a decomposition of $\roim{q}(F)$ as a
finite sum $\roim{q}(F) \simeq \bigoplus_{a\in A} \cor^{n_a}_{I_a}[d_a]$, where the
$I_a$'s are intervals.  By Proposition~\ref{prop:oim} the ends of these intervals can
only be the points $t_i$ and the critical values of $q$.  In our application the
sheaf $F\in \Derb(\cor_M)$ will satisfy $\DD_M(F) \simeq F$ (for this, it is
necessary that $\Lambda^a = \Lambda$, by Theorem~\ref{thm:SSrhom}) and
$\rsect(M;F) \simeq \cor$.  By Lemma~\ref{lem:autodual} this implies that all
intervals $I_a$ are half-closed except one which is reduced to a point.  We can check
that this point cannot be a critical value, hence it is $t_{i_0}$ for some $i_0$.
Thus there is one special tangent fiber, with respect to $F$, and in this section we
try to understand $F$ around this fiber.  Up to adding $-t_{i_0}$ to $q$ we assume
$t_{i_0} = 0$.  Then $q^{-1}(0)$ is a union of circles, one of which is tangent to
$\Sigma$.  Restricting to a neighborhood $\Uu$ of this circle, we obtain the
situation described in~\eqref{eq:notJELambda0} and~\eqref{eq:hypLambda_gentang}
below, with $\GammaE = \Sigma \cap \Uu$, and $F|_\Uu$
satisfies~\eqref{eq:autodualite} and~\eqref{eq:decomp_FJ2} (we restrict the
isomorphism $\roim{q}(F) \simeq \bigoplus_{a\in A} \cor^{n_a}_{I_a}[d_a]$ to a
neighborhood $\Jj$ of $0$ which does not contain the ends of the $I_a$'s, for
$I_a \not= \{0\}$, and rescale to have $\Jj = \mo]-1,1\mc[$). In
Fig.~\ref{fig:hypLambda} we give a typical curve $\dot\pi_M(\Lambda)$ in $M$ and more
precisely the situation in the open subset $\Uu$.

We set $\Jj = \mo]-1,1\mc[$ and we choose a diffeomorphism
$\Uu \simeq \cer \times \Jj$. We let
$$
q \colon \Uu \to \Jj
$$
be the projection.  We take the coordinates $\theta \in \mo]-\pi,\pi]$ on
$\cer$, $t$ on $\Jj$ and $(\theta;\xi)$ on $T^*\cer$, $(t;\tau)$ on $T^*\Jj$.
We set
\begin{equation}\label{eq:notJELambda0}
\left\{
\begin{alignedat}{2}
  C_t &= \opb{q}(t), \; t\in \Jj,  & \qquad  \Uu_0& =  \mo]-1,1\mc[ \times \Jj , \\
\GammaE_0 &= \{(\theta,t) \in \Uu_0; \; t =  \theta^2 \} , 
& \Omega &= \{(\theta,t) \in \Uu_0; \; t > \theta^2 \}, \\
 \Lambda_0 &= T^*_{\GammaE_0}\Uu_0 , 
&  p^\pm &= (0,0;0,\pm 1) \in \Lambda_0 .
\end{alignedat}
\right.
\end{equation}
We assume that $\Lambda$ satisfies
\begin{equation}\label{eq:hypLambda_gentang}
\left\{     \hspace*{-1cm}
  \begin{minipage}[c]{11cm}
    \begin{itemize}
    \item [-] $\Lambda^a = \Lambda$ (that is, $\Lambda$ is stable by the antipodal
      map $(x;\xi) \mapsto (x;-\xi)$),
    \item [-] $\Lambda \cap T^*\Uu = \dT^*_\GammaE \Uu$, where
      $\GammaE = (\{\theta_1,\dots,\theta_N\} \times \Jj) \cup \GammaE_0$ and
      $\theta_1,\dots,\theta_N \in \cer \setminus \mo]-1,1\mc[$.
    \end{itemize}
\end{minipage}
\right.
\end{equation}

\begin{figure}[ht] 
\begin{tikzpicture}[scale=.5]
  \draw   (-1,1) parabola bend (0,0) (1,1)
  -- plot [smooth] coordinates {(1, 1) (1.2, 2) (0, 3) (-.5, 2.5)
    (0, 2.2) (2.3, 2.5) (3, 3) };

  \draw  plot [smooth] coordinates {(3, 3) (2.7, 2.5) (2.9, 2)
    (3,1) (3,-1) (2.5, -1.8) (1,-3) (-1, -2.3) (-1.5, -1.5) (-1.5,1) (-1,1.5)
    (1, 1.5) (2,1) (2,-1) (0.5, -2) (-2, -2.3) (-2.5,-1) (-2.5,1)
  (-2,2.5) (-1.5,2) (-1,1) };

  \draw (0.5,0) circle [radius=4.3];
\foreach \x/\y\z in {-3.8/4.8/0}
\draw (\x,\z) -- (\y,\z);
\foreach \x/\y\z in {-3.7/4.7/-1, -3.7/4.7/1}
\draw [dashed] (\x,\z) -- (\y,\z);

  \draw [->] (6,0) -- (8,0) node[midway, above] {$q$};
  
  \draw (10,-1) -- (10,1) ;
  \node at (10,2) {$\Jj$} ;
  \node at (-7,0) {$\Uu \simeq \cer \times \Jj$} ;
  \begin{scope}[decoration=brace]
\draw [decorate] (-4.3,-1) -- (-4.3,1);
\end{scope}
\foreach \x in {-1,0,1} 
{  \node at (11,\x) {$\x$} ;
  \draw (9.9,\x) -- (10.1,\x) ; };
  
  \node at (-3.5,3.5) {$M$} ;   
\end{tikzpicture}

\bigskip

\begin{tikzpicture}[scale=1.3]
  \draw[dashed] (-3.5,0) -- (-3, 0);
\draw (-3,0) -- (2,0) ;
\draw[dashed] (2,0) -- (2.5, 0);

  \draw (-1,1) parabola bend (0,0) (1,1);
  \fill [gray!20] (-1,1) parabola bend (0,0) (1,1);
  \node [fill=white] at (-.6,.36) {$\scriptstyle \GammaE_0$};
  \node  at (.1,.7) {$\Omega$};
  \node at (-1.2,-.2) {$-1$} ;
  \draw [dashed] (-1,-1) -- (-1,1) ;
  \node at (.9,-.2) {$1$} ;
  \draw [dashed] (1,-1) -- (1,1) ;
  \node at (-2.5,-.2) {$\theta_i$} ;
  \draw (-2.3,-1) -- (-2.3,1) ;
  \node at (1.5,-.2) {$\theta_j$} ;
  \draw (1.7,-1) -- (1.7,1) ;
  \node at (-4.3,.8) {$\Uu$};
  \node at (-3.8,0) {$C_0$};
  \node at  (0,-1.5) {$\Uu_0 \simeq \mo]-1,1[ \times \Jj$};
  \begin{scope}[decoration=brace]
\draw [decorate] (1,-1.2) -- (-1,-1.2);
  \end{scope}
\end{tikzpicture}
\caption{} \label{fig:hypLambda}
\end{figure}

In fact, if $\Lambda \subset \dT^*M$ is any smooth closed conic Lagrangian
submanifold such that $\Lambda^a = \Lambda$ and $\pi_M(\Lambda)$ is a smooth
curve in a neighborhood of $C_0$ which has one tangent point with $C_0$ with a
tangency of order $1$, then we can find an isotopy $\{\varphi_t\}_{t\in [0,1]}$
of $M$ such that $d\varphi_1(\Lambda)$ satisfies~\eqref{eq:hypLambda_gentang},
up to maybe changing $q$ to $-q$.

We will often make the following hypotheses on the restriction of $F$ to $\Uu$
\begin{align}
\label{eq:autodualite}
& \DD_\Uu(F|_\Uu) \simeq F|_\Uu , \\
\label{eq:decomp_FJ2}
& \roim{q}(F|_\Uu) \simeq \cor_{\{0\}} \oplus B_{\Jj} , \qquad
\text{for some $B \in \Derb(\cor)$},
\end{align}
where   $B_{\Jj}$ denotes the constant sheaf on $\Jj$ with stalk $B$
(that is, $B_{\Jj} = \opb{a_{\Jj}}(B)$ where $a_{\Jj}\colon \Jj \to \pt$ is the
map to a point).

\subsection{Local study around $C_0$}
We prove that $H^1\RHom(F,F)$ is non zero for a sheaf $F$
satisfying~\eqref{eq:autodualite}, \eqref{eq:decomp_FJ2} and some additional
hypothesis (see Proposition~\ref{prop:FC0-monodunip}). For this we use
Lemma~\ref{lem:cohom_supp} with $G=F$ and $Z= C_0$.  To construct the morphism $u$ of
the lemma we use the decomposition of $F_{C_0}$ and $\rsect_{C_0}(F)$ (which are
sheaves on the circle) given by Corollary~\ref{cor:sheaves_dim1} and, more
specifically,   the fact that a non
zero locally constant sheaf with unipotent monodromy appears in the decomposition
of $F_{C_0}$. This fact is proved in Proposition~\ref{prop:interv_ferme}.  We begin
with a lemma which describes the restriction of $F$ to the open set $\Uu_0$ defined
in~\eqref{eq:notJELambda0}.  For this lemma and for
Proposition~\ref{prop:interv_ferme} we consider a sheaf
satisfying~\eqref{eq:decomp_FJ2} but maybe not~\eqref{eq:autodualite} (the argument
of the proof constructs an intermediate sheaf which inherits~\eqref{eq:decomp_FJ2}
but a priori not~\eqref{eq:autodualite}).

In this section we use the notations~\eqref{eq:notJELambda0}
and~\eqref{eq:hypLambda_gentang},  in particular $\Uu$, $\Uu_0$, $\GammaE$,
$\Omega$, $\Lambda$, $\Lambda_0$.  
\begin{lemma}
  \label{lem:F_degree-1}
  Let $\Lambda \subset \dT^*M$ be a smooth closed conic Lagrangian submanifold
  satisfying~\eqref{eq:hypLambda_gentang}.  Let
  $F \in \Derb_{[\Lambda\cap T^*\Uu]}(\cor_\Uu)$ be given. Then $F$ is decomposed
  as $F \simeq \bigoplus_{i\in \Z} (H^iF)[-i]$ and  
  $\dot\SSi(H^iF) \subset \Lambda\cap T^*\Uu$, for each $i\in\Z$.  Moreover, if
  $F$ satisfies~\eqref{eq:decomp_FJ2}, then
  \begin{itemize}
  \item [-] $(H^{-1}F)[1]$ satisfies~\eqref{eq:decomp_FJ2} (for some other
    $B\in \Der(\cor)$),
  \item [-] $\roim{q}(H^iF)$ is a constant sheaf on $\Jj$, for all $i\not=-1$,
  \item [-]
    $H^{-1}F|_{\Uu_0} \simeq \cor_{\Uu_0}^p \oplus \cor_{\Uu_0 \setminus \Omega}^q
    \oplus \cor_{\Uu_0 \setminus \ol\Omega}^r$ with $p,q,r \in\N$, $q,r \geq 1$.
  \end{itemize}
\end{lemma}
\begin{proof} 
  (i) Since $\Lambda \cap T^*\Uu \subset T^*_\GammaE \Uu$ where $\GammaE$ is a smooth
  curve, $F$ is   weakly constructible for the stratification
  $\Uu = \GammaE \cup (\Uu\setminus \GammaE)$ by Proposition~\ref{prop:mustratif}
  and Remark~\ref{rem:tensor_prod_construct}.  It follows that the same holds
  for all $H^iF$.  By Corollary~\ref{cor:tensor_prod_construct}, for two such
  weakly constructible sheaves $G$, $G'$, the complex $H = \rhom(G,G')$ is also
  weakly constructible for the same stratification.  If $G$, $G'$ are
  concentrated in degree $0$, then $H_{\Uu\setminus \GammaE}$ is concentrated in
  degree $0$ and $H_\GammaE$ in degrees $0$ and $1$.  Moreover $H_\GammaE$ has
  stalks $0$ outside $\GammaE$ and $H_\GammaE|_\GammaE$ is locally constant;
  similarly $H_{\Uu\setminus \GammaE}$ has stalks $0$ outside $\Uu\setminus \GammaE$
  and $H_{\Uu\setminus \GammaE}|_{\Uu\setminus \GammaE}$ is locally constant.  Since
  $\GammaE$ consists only of segments, we deduce $H^2(\Uu; H_\GammaE) \simeq 0$.
  We can check that the same vanishing holds for
  $H^2(\Uu; H_{\Uu\setminus \GammaE})$ and we deduce by excision that
  $\Ext^2(G,G') \simeq 0$.  Lemma~\ref{lem:compl_scinde} gives
  $F \simeq \bigoplus_i (H^iF)[-i]$ and this implies the bound for
  $\dot\SSi(H^i(F))$.

  \sui (ii) Now we assume that $F$ satisfies~\eqref{eq:decomp_FJ2}.  By~(i) we
  have $\roim{q}F \simeq \bigoplus_{i\in \Z}\roim{q}(H^iF)[-i]$.  Each
  $\roim{q}(H^iF)$ is weakly constructible for the stratification
  $\Jj = \{0\} \sqcup (\Jj \setminus \{0\})$ and the same argument as in~(i)
  shows that $\roim{q}(H^iF)$ is decomposed by the cohomological degree. Hence
  there exists $i_0$ such that $\roim{q}(H^{i_0}F)[-i_0]$ has $\cor_{\{0\}}$ as
  a direct summand and $\roim{q}(H^iF)$ are constant sheaves on $\Jj$ for
  $i\not= i_0$.  Let us check that $i_0 = -1$.   We see
  $H^{i_0}F$ as an object of $\Derb(\cor_\Uu)$ concentrated in degree $0$.  Since
  the fibers of $q$ have dimension $1$, the direct image $\roim{q}(H^{i_0}F)$ is
  concentrated in degrees $0$ and $1$ and so is its summand $\cor_{\{0\}}[i_0]$.
  Hence $i_0$ can only be $0$ or $-1$.  If $i_0 = 0$, then taking $H^0$ of
  $\roim{q}(H^{0}F)$ we find that $\cor_{\{0\}}$ is a direct summand of
  $H^0\roim{q}(H^{0}F)$ (in the category $\Mod(\cor_{\Jj})$ of non derived
  sheaves).  We recall that $H^0\roim{q}(H^{0}F) \simeq \oim{q}(H^{0}F)$ (the
  non derived direct image).  It follows that there exists
  $s \in \sect(\Jj; \oim{q}(H^{0}F))$ with $\supp(s) = \{0\}$.  Since
  $\sect(\Jj; \oim{q}(H^{0}F)) \simeq \sect(\Uu; H^{0}F)$ this section $s$ gives a
  section $s'$ of $H^0F$ with $\supp(s') \subset q^{-1}(0)$.  But we have seen
  that $H^0F$ is constructible for the stratification
  $\Uu = \GammaE \cup (\Uu\setminus \GammaE)$.  The support of $s'$ should be a
  union of strata and there is no stratum contained in $q^{-1}(0)$. This
  excludes the case $i_0=0$ and we have $i_0=-1$.

  \sui(iii) We recall that $\Uu_0$ is an open subset of $\Uu$ which contains the
  curve $\GammaE_0$ and no other component of $\GammaE$
  (see~\eqref{eq:notJELambda0}).  We can find a submersion $f \cl \Uu_0\to \R$
  whose fibers are intervals and such that $\Omega = \opb{f}(]0,+\infty[)$.
  Since $\dot\SSi(H^{-1}(F)) \subset \Lambda$, Proposition~\ref{prop:iminvproj}
  implies that $H^{-1}(F)|_{\Uu_0} \simeq \opb{f}G$ for some sheaf $G$ on
  $\R$. Then $G$ must be concentrated in degree $0$ and we have
  $\dot\SSi(G) \subset T^*_0\R$.  By Corollary~\ref{cor:sheaves_dim1} $G$ is
  decomposed as a sum of the sheaves $\cor_\R$, $\cor_{]0,+\infty[}$,
  $\cor_{[0,+\infty[}$, $\cor_{\{0\}}$, $\cor_{]-\infty,0[}$ and
  $\cor_{]-\infty,0]}$ with some multiplicities.  Hence $H^{-1}(F)|_{\Uu_0}$ is
  decomposed as a sum of the sheaves $\cor_{\Uu_0}$, $\cor_\Omega$,
  $\cor_{\ol{\Omega}}$, $\cor_{\GammaE_0}$, $\cor_{\Uu_0 \setminus \ol{\Omega}}$ and
  $\cor_{\Uu_0 \setminus \Omega}$.  The sheaves $\cor_\Omega$,
  $\cor_{\ol{\Omega}}$, $\cor_{\GammaE_0}$ have a support which is closed in $\Uu$
  (not only in $\Uu_0$). It follows that, if one of them, say $F'$, is a summand
  of $H^{-1}(F)|_{\Uu_0}$, then it is also a summand of $H^{-1}(F)$ and
  $\roim{q}(F')$ is a summand of $\roim{q}(H^{-1}(F))$.  Since $H^{-1}(F)[1]$
  satisfies~\eqref{eq:decomp_FJ2}, the summands of $\roim{q}(H^{-1}(F))$ can
  only have $\{0\}$ or $\Jj$ as possible support.  On the other hand
  $\roim{q}\cor_\Omega$, $\roim{q}\cor_{\ol{\Omega}}$, $\roim{q}\cor_{\GammaE_0}$
  all have support $[0,1[$.  Hence $\cor_\Omega$, $\cor_{\ol{\Omega}}$,
  $\cor_{\GammaE_0}$ cannot appear in the decomposition of $H^{-1}(F)|_{\Uu_0}$.

  The only possible summands of $H^{-1}(F)|_{\Uu_0}$ are then $\cor_{\Uu_0}$,
  $\cor_{\Uu_0 \setminus \ol{\Omega}}$ and $\cor_{\Uu_0 \setminus \Omega}$. If the
  last two do not appear both, then $\dot\SSi(H^{-1}(F))$ does not meet $\Lambda_0$
  or only contains one of the two components of $\Lambda_0$. By
  Proposition~\ref{prop:oim} we obtain that $\dot\SSi(\roim{q}(H^{-1}(F))$ does not
  meet $\dT_0\R$ or only contains one half of $\dT_0\R$, which contradicts the fact
  that $\cor_{\{0\}}[-1]$ is a summand of $\roim{q}(H^{-1}(F))$. This concludes the
  proof.
\end{proof}

Let $F\in \Derb(\cor_\Uu)$ be given such that $\dot\SSi(F) = \Lambda \cap \dT^*\Uu$
and $F$ is simple.  We assume that $F$ is constructible (for the stratification
$\Uu = \GammaE \sqcup (\Uu\setminus \GammaE)$ -- see
Section~\ref{sec:constructibility}). Then $F|_{C_{1/2}}$ is constructible and we
can decompose $F|_{C_{1/2}}$ according to Corollary~\ref{cor:sheaves_dim1}:
there exist   $L\in \Derb(\cor_{C_{1/2}})$ and
$\{(I_a, d_a)\}_{a\in A}$, where $L$ has locally constant cohomology sheaves of
finite rank and $A$ is a finite family of bounded intervals and integers, such
that
\begin{equation}
\label{eq:decomp_F_C}
F|_{C_{1/2}} \simeq L \oplus \bigoplus_{a\in A} \oim{e}(\cor_{I_a})[d_a] ,
\end{equation}
where $e \cl \R \to C_{1/2} \simeq \R/2\pi\Z$ is the quotient map.   
  Since $F$ is simple, so is $F|_{C_{1/2}}$ by
Corollary~7.5.13 of~\cite{KS90} (see also Lemma~\ref{lem:bonne_im_inv} below). In
particular, the intervals $I_a$ appear with multiplicity $1$ and, for any two
distinct intervals $I_a$, $I_b$, we have
$\dot\SSi(\oim{e}(\cor_{I_a})) \cap \dot\SSi(\oim{e}(\cor_{I_b})) = \emptyset$.  This
implies:
\begin{equation}
  \label{eq:multipl_un}
  \begin{minipage}{10cm}
    if $x$ is an end of $I_a$, $y$ an end of $I_b$ and there exists
    $\pi>\varepsilon>0$ such that
    $e(I_a \cap \mo ]x-\varepsilon,x+\varepsilon[) = e(I_b \cap \mo
    ]y-\varepsilon,y+\varepsilon[)$, then $I_a=I_b$.    
  \end{minipage}
\end{equation}

  In the next proposition we consider a sheaf
$F \in \Dersf_{[\Lambda]}(\cor_\Uu)$ satisfying~\eqref{eq:decomp_FJ2}. By
Lemma~\ref{lem:F_degree-1} we know that $F$ is decomposed according to the
cohomological degree and that $H^{-1}F|_{\Uu_0}$ has $\cor_{\Uu_0 \setminus \Omega}$
and $\cor_{\Uu_0 \setminus \ol\Omega}$ as direct summands.  We remark that
$C_{1/2} \cap (\Uu_0 \setminus \Omega)$ is the union of the two intervals
$\mo]-1,-1/\sqrt2]$ and $[1/\sqrt2,1[$ of length $l = 1-1/\sqrt2$.
By~\eqref{eq:multipl_un}, in the decomposition~\eqref{eq:decomp_F_C} of
$F|_{C_{1/2}}$, there are two intervals (uniquely defined but maybe equal) in
the family $\{I_a\}_{a\in A}$, say $I_b$, with right end $x$ and $I_c$, with
left end $y$, such that
\begin{equation}
  \label{eq:def_Ib_Ic}
  \begin{split}
  e(I_b \cap \mo]x-l, +\infty[) &= \mo]-1,-1/\sqrt2]  , \\
  e(I_c \cap \mo]-\infty, y+l[) &= [1/\sqrt2,1[ .    
  \end{split}
\end{equation}

We recall that a locally constant sheaf $G$ on the circle is described up to
isomorphism by its stalk $G_\theta$ at a given point and its monodromy
$m\colon G_\theta \to G_\theta$.  We then have
$H^0(C_0;G) \simeq \{v\in G_\theta$; $m(v)=v\}$. Hence a locally constant sheaf
has a section if and only if its monodromy has a unipotent factor.

\begin{proposition}
\label{prop:interv_ferme}
Let $\Lambda \subset \dT^*M$ be a smooth closed conic Lagrangian submanifold
satisfying~\eqref{eq:hypLambda_gentang}.  We use the
notations~\eqref{eq:notJELambda0}.  Let $F \in \Derb_{[\Lambda]}(\cor_\Uu)$ be a
simple sheaf satisfying~\eqref{eq:decomp_FJ2}.  We assume that $F$ is
constructible (see Section~\ref{sec:constructibility}).    We consider the intervals $I_b$, $I_c$ appearing in
the decomposition~\eqref{eq:decomp_F_C} of $F|_{C_{1/2}}$ which are defined
in~\eqref{eq:def_Ib_Ic}.  We assume that $I_b$ or $I_c$ is closed.  Then there
exists a decomposition $F|_{C_0} \simeq F' \oplus L[1]$ such that
$L \in \Mod(\cor_{C_0})$ is a locally constant sheaf with unipotent monodromy
and $L\not=0$.
\end{proposition}
\begin{proof} 
  (i) We assume that $I_b$ is closed, the other case being similar.  We argue by
  contradiction and assume that there exists $F$ satisfying the hypotheses of the
  proposition but not the conclusion.  We choose $F$ (among those $F$ contradicting
  the proposition) so that the integer $\lfloor l(I_b)/2\pi \rfloor$ is minimal,
  where $l(I_b)$ is the length of $I_b$.  We write $I_b = [\alpha, \beta]$.  We set
  $C = C_{1/2}$ for short.  By Lemma~\ref{lem:F_degree-1} we know that $F$ is
  decomposed according to the cohomological degree, that $(H^{-1}F)[1]$
  satisfies~\eqref{eq:decomp_FJ2} and that $H^{-1}F|_{\Uu_0}$ has
  $\cor_{\Uu_0 \setminus \Omega}$ and $\cor_{\Uu_0 \setminus \ol\Omega}$ as direct
  summands.     Since $F$ is
  simple, only this degree $-1$ term involves intervals with one end in
  $\opb{e}(\GammaE_0\cap C)$. We can then assume from the beginning that $F$ is
  concentrated in degree $-1$ and we write $F = F_0[1]$, where   $F_0 \in \Mod(\cor_\Uu)$ is seen as usual as an object of
  $\Derb(\cor_\Uu)$ that is concentrated in degree $0$.  The hypotheses say that we
  can write   $F_0|_C \simeq \oim{e}(\cor_{[\alpha, \beta]}) \oplus
  F_1$. We define $s = (\mathbf{1},0) \in H^0(C; F_0|_C)$ according to this
  decomposition, where $\mathbf{1}$ is the natural section of
  $\oim{e}(\cor_{[\alpha, \beta]})$.

  By the base change formula we have  
  $H^0(C; F_0|_C) \simeq (\oim{q}(F_0))_{1/2}$ (the stalk of the {\em non
    derived} direct image $\oim{q}(F_0)$ at the point $1/2 \in \Jj$).  The
  hypothesis~\eqref{eq:decomp_FJ2} implies that
  $\oim{q}(F_0) \simeq H^0(\roim{q}F_0)$ is a constant sheaf on $\Jj$, hence the
  restriction morphism $H^0(\Jj;\oim{q}(F_0)) \to (\oim{q}(F_0))_{1/2}$ is an
  isomorphism.  Since $H^0(\Jj;\oim{q}(F_0)) = H^0(\Uu; F_0)$, there exists
  $s' \in H^0(\Uu; F_0)$ such that $s'|_C = s$.  We interpret the sections $s$ and
  $s'$ as morphisms
$$
u \cl \cor_C \to F_0|_C, \qquad u' \cl \cor_\Uu \to F_0 .
$$

\sui (ii) If $\lfloor (\beta-\alpha)/2\pi \rfloor = 0$, then
$I' = e([\alpha, \beta])$ is an arc of $C$ (smaller than $C$).  Then
$S = \supp( s')$ is a closed subset of $\Uu$ which satisfies $S \cap C = I'$.  By
Lemma~\ref{lem:F_degree-1} and the fact that $F$ is simple we have
$F_0|_{\Uu_0} \simeq \cor_{\Uu_0}^p \oplus \cor_{\Uu_0 \setminus \Omega} \oplus
\cor_{\Uu_0 \setminus \ol\Omega}$.   Hence $S\cap \Uu_0$ can only be
$\emptyset$, $\Uu_0$ or $\Uu_0 \setminus \Omega$.  The first two cases are excluded
because one end of the arc $I'$ is in $\partial \Omega$ and $I' \not= C$. Hence
$S\cap \Uu_0 = \Uu_0 \setminus \Omega$. It follows that
$I' = C \setminus (C\cap \Omega)$.      Outside $\Uu_0$ the sheaf $F_0$ is constant on the vertical segments
$\{\theta\} \times \Jj$, by Theorem~\ref{thm:iminv} and the
hypotheses~\eqref{eq:hypLambda_gentang} on $\Lambda$, and
$S\setminus (S\cap \Uu_0)$ must be a union of such segments.  Hence we have
$S = \Uu \setminus \Omega$.

The morphism $u'$ factorizes through $\cor_\Uu \to \cor_S$ and gives
$v \cl \cor_S \to F_0$.  By construction the restriction of $v$ to $\Uu_0$,
$v|_{\Uu_0} \colon \cor_{\Uu_0 \setminus \Omega} \to F_0|_{\Uu_0}$, is the morphism
induced by the above decomposition of $ F_0|_{\Uu_0}$.  Let us define
$G \in \Derb(\cor_\Uu)$ by the distinguished triangle
\begin{equation}
\label{eq:interv_ferme1}
\cor_S \to[v] F_0 \to G \to[+1].
\end{equation}
Then $G|_{\Uu_0} \simeq \cor_{\Uu_0}^p \oplus \cor_{\Uu_0 \setminus \ol\Omega}$.  The
image of the triangle~\eqref{eq:interv_ferme1} by $\roim{q}$ gives
\begin{equation}
\label{eq:interv_ferme2}
\roim{q}(\cor_S) \;\to[w]\; \cor_{\{0\}}[-1] \oplus B_{\Jj}[-1] 
 \;\to\;  \roim{q}(G)  \;\to[+1] ,
\end{equation}
with $B \in \Derb(\cor)$ as in~\eqref{eq:decomp_FJ2} and $w = \roim{q}(v)$. Let us
write $w = \bpmat w_0 \\ w_1 \epmat$. The morphism $w_0$ has target
$\cor_{\{0\}}[-1]$ and thus is determined by $H^1(w|_{\{0\}})$.  By the base change
formula $H^1(w|_{\{0\}}) = H^1(C_0; v|_{C_0})$.  We have
$\cor_S|_{C_0} \simeq \cor_{C_0}$ and we have assumed that $F_0|_{C_0}$ does not have
a direct summand which is locally constant with unipotent monodromy.  This implies
that $H^1(C_0; v|_{C_0})$ vanishes.  Indeed, decomposing
$F_0|_{C_0} \simeq L_0 \oplus \bigoplus_{a\in A'} \oim{e}(\cor^{n_a}_{J_a})$ as
in~\eqref{eq:cons_sh_cer1}, we write $v|_{C_0} = v_0 + \sum_{a\in A'} v_a$.  Our
hypothesis says that $L_0$ has no global section, hence $v_0 = 0$.  If $v_a\not=0$,
then the corresponding interval $J_a$ is closed and
$H^1(C_0; \oim{e}(\cor^{n_a}_{J_a})) \simeq 0$, hence $H^1(C_0; v_a)=0$.  Finally
$H^1(C_0; v|_{C_0}) =0$, as claimed.

We thus have $w_0 = 0$ and we deduce from~\eqref{eq:interv_ferme2} that
$\cor_{\{0\}}[-1]$ is a direct summand of $\roim{q}(G)$.  On the other hand, the
triangular inequality for the microsupport together
with~\eqref{eq:interv_ferme1} and
$G|_{\Uu_0} \simeq \cor_{\Uu_0}^p \oplus \cor_{\Uu_0 \setminus \ol\Omega}$ give the
bound $\dot\SSi(G) \subset (\Lambda \cap T^*\Uu) \setminus \Lambda_0^-$, where
$\Lambda_0^-$ is the connected component of $\Lambda_0$ which contains $p^-$.
This implies $\dot\SSi( \roim{q}(G)) \subset \{(0;\tau)$; $\tau>0\}$ and
contradicts the fact that $\cor_{\{0\}}[-1]$ is a direct summand of
$\roim{q}(G)$.

\sui (iii) Now we assume $\lfloor (\beta - \alpha)/2\pi \rfloor > 0$.  We define
$G \in \Derb(\cor_\Uu)$ by the distinguished triangle
$\cor_\Uu \to[u'] F_0 \to G \to[+1]$. Applying $\roim{q}$ gives the distinguished
triangle
\begin{equation}
\label{eq:interv_ferme3}
\roim{q}(\cor_\Uu) \to \cor_{\{0\}}[-1] \oplus B_{\Jj}[-1] 
\;\to\;  \roim{q}(G)  \;\to[+1] .
\end{equation}
The same argument as in~(ii) works and we find that $\cor_{\{0\}}[-1]$ is a
direct summand of $\roim{q}(G)$.  Since
$\roim{q}(\cor_\Uu) \simeq \cor_{\Jj} \oplus \cor_{\Jj}[-1]$ the other summands of
$\roim{q}(G)$ are constant sheaves on $\Jj$. Hence $G[1]$
satisfies~\eqref{eq:decomp_FJ2}.  Since $\dot\SSi(\cor_\Uu) = \emptyset$, we also
have $\dot\SSi(G) = \dot\SSi(F_0)$ and $G$ is simple.  By~(iv) below we have
$G|_C \simeq \oim{e}(\cor_{[\alpha,\beta-2\pi]}) \oplus F_1$.  
Hence $G[1]$
satisfies the same hypotheses as $F$ with $(\beta - \alpha)/2\pi$ replaced by
$(\beta - \alpha)/2\pi -1$.  Hence $G|_{C_0}$ has a direct summand, say $L_u$,
which is a locally constant sheaf with unipotent monodromy (recall that $F$ was
chosen to contradict the result with $\beta-\alpha$ minimal).  As in~(ii) we
write
$F_0|_{C_0} \simeq L_0 \oplus \bigoplus_{a\in A'} \oim{e}(\cor^{n_a}_{J_a})$ and
we write $u'|_{C_0} = \sum_{a\in A''} u'_a$, where $A'' \subset A'$ is the set of
closed intervals.  Setting
$F'_0 = L_0 \oplus \bigoplus_{a\in A'\setminus A''} \oim{e}(\cor^{n_a}_{J_a})$,
$F''_0 = \bigoplus_{a\in A''} \oim{e}(\cor^{n_a}_{J_a})$ we have
$F_0|_{C_0} \simeq F'_0 \oplus F''_0$ and $u'|_{C_0} = (0,u'')$.  Hence
$G|_{C_0} \simeq F'_0 \oplus \coker(u'')$. The summand $L_u$ of $G|_{C_0}$
cannot appear in $F'_0$ (by the assumption on $F_0|_{C_0}$) hence it appears in
$\coker(u'')$.  Hence $L_u$ is a quotient of $F''_0$. On the other hand, if
$J_a$ is closed, then $\Hom(\oim{e}(\cor^{n_a}_{J_a}), L_u) = 0$ and we have a
contradiction. This proves the proposition.

\sui(iv)   It remains to explain the decomposition of $G|_C$.  We have
$\beta - \alpha \geq 2\pi$ and we set $I = [\alpha, \beta]$,
$J = [\alpha, \beta-2\pi]$, $J' = [\alpha+2\pi, \beta]$.  We have the canonical
isomorphisms
$\Hom(\cor_C, \oim{e}(\cor_I)) \simeq \Hom(\cor_\R, \cor_I) \simeq
\sect(\R;\cor_I) \simeq \cor$ where the first one is given by the adjunction
$(\opb{e}, \oim{e})$.  Let $i\colon \cor_C \to \oim{e}(\cor_I)$ be the morphism
corresponding to $1\in \cor$.  We have a natural isomorphism
$\varphi \colon \oim{e}(\cor_J) \simeq \oim{e}(\cor_{J'})$ and two restriction
morphisms $r \colon \cor_I \to \cor_J$, $r' \colon \cor_I \to \cor_{J'}$.  For
$\theta \in C$ the vector space $(\oim{e}(\cor_I))_\theta$ has a basis
$\{p_1,\dots,p_n\}$ identified with $I \cap e^{-1}(\theta)$ ($n$ depends on
$\theta$). We order the basis by the order induced by $\R$, that is,
$I \cap e^{-1}(\theta) = \{p_1< \dots <p_n\}$.  In the same way, a basis of
$(\oim{e}(\cor_J))_\theta$ is $\{p_1,\dots,p_{n-1}\}$.  The morphisms
$\oim{e}(r)$ and $\varphi^{-1} \circ \oim{e}(r')$ induce in the stalks the
morphisms $(x_1,\dots,x_n) \mapsto (x_1,\dots,x_{n-1})$ and
$(x_1,\dots,x_n) \mapsto (x_2,\dots,x_n)$.  We also have
$(\cor_C)_\theta \simeq \cor$ and the morphism $i$ induces the diagonal morphism
in the stalks, $x \mapsto (x,\dots,x)$.  It follows that the sequence
  $$
  0 \to \cor_C \to[i] \oim{e}(\cor_I)
  \to[\oim{e}(r) - \varphi^{-1} \circ \oim{e}(r')] \oim{e}(\cor_J) \to 0
  $$
  is exact in the stalks, hence exact.
\end{proof}

\begin{example} 
  Here is an example of a sheaf satisfying the hypotheses of
  Proposition~\ref{prop:interv_ferme}.   We
  consider the closed subset $S = \Uu \setminus \Omega$ of $\Uu$ introduced in
  part~(ii) of the proof and define its interior $S' = \Uu \setminus \ol\Omega$.
  Then $\Hom(\cor_{S'}, \cor_S[1]) \simeq H^1(S'; \cor_S) \simeq \cor$ and we define
  $F \in \Derb(\cor_\Uu)$ by the distinguished triangle
  $\cor_{S'} \to[u] \cor_S[1] \to F \to[+1]$, where $u$ is the image of $1$ by this
  isomorphism.  The image of the triangle by $\roim{q}$ gives the distinguished
  triangle
  $$
\cor_\Jj[-1] \oplus \cor_{]-1,0[} \to[v]
\cor_\Jj[1] \oplus \cor_{]-1,0]} \to \roim{q}F \to[+1] ,
  $$
  where the morphism $v$ is non zero because its stalk at a point
  $t \in \mo]-1,0[$ is the image of the morphism $\cor_\cer \to \cor_\cer[1]$ by
  the functor of global sections and this is non zero.  Writing $v$ as a
  $2\times 2$ matrix of morphisms, the only possibly non zero entry is
  $v_{22} \colon \cor_{]-1,0[} \to \cor_{]-1,0]}$ and $v_{22}$ must be a
  multiple of the canonical morphism $\cor_{]-1,0[} \to \cor_{]-1,0]}$ whose
  cone is $\cor_{\{0\}}$.  Since the other entries of $v$ are zero, we find
  $\roim{q}F \simeq \cor_{\{0\}} \oplus \cor_\Jj \oplus \cor_\Jj[1]$, which is
  the hypothesis~\eqref{eq:decomp_FJ2}.
\end{example}

\begin{proposition} 
\label{prop:FC0-monodunip}
Let $\Lambda \subset \dT^*M$ be a Lagrangian submanifold
satisfying~\eqref{eq:hypLambda_gentang} and let $F \in \Derb_{[\Lambda]}(\cor_\Uu)$ be a
simple sheaf.  We assume that $F$ is constructible (see
Section~\ref{sec:constructibility}) and satisfies~\eqref{eq:autodualite}
and~\eqref{eq:decomp_FJ2}.   
We also assume that
$H^{-1}F|_{C_0}$ has a direct summand which is a (non zero) locally constant sheaf
with unipotent monodromy.  Then $H^1\RHom(F,F)$ is non zero.
\end{proposition}
\begin{proof}
  \newcommand{\ww}{w}
  \newcommand{\vv}{v}
  \newcommand{\uu}{u}
  \newcommand{\unip}{\mathrm{uni}}
  \newcommand{\nunip}{\mathrm{nu}}

  (i) We will use Lemma~\ref{lem:cohom_supp} and look for
  $\uu \colon F_{C_0}[-1] \to \rsect_{C_0}(F)$ such that $i \circ \uu \not=0$, where
  $i = j_{C_0}(F) \circ i_{C_0}(F)$ is the composition of the natural morphisms
  $i_{C_0}(F) \colon \rsect_{C_0}(F) \to F$, $j_{C_0}(F) \colon F \to F_{C_0}$.
  In fact, to ensure that $i \circ \uu \not=0$ we first find
  $\ww \colon \cor_{C_0} \to \rsect_{C_0}(F)$ satisfying $i\circ \ww \not=0$ and
  prove the existence of a factorization $\ww = \uu \circ \vv$ as in the diagram
$$
\begin{tikzcd}
  \cor_{C_0} \ar[rr, "\ww"] \ar[dr, dashed, "\vv"]
  & &  \rsect_{C_0}(F)  \ar[r, "i"]  &  F_{C_0} . \\
& F_{C_0}[-1] \ar[ur, dashed, "\uu"]
\end{tikzcd}
$$
It is then clear that $i \circ \uu \not=0$.  We first define $\ww$ such that
$i\circ \ww \not= 0$.  Then in parts~(ii-iv) of the proof we actually show that
any $\ww \colon \cor_{C_0} \to \rsect_{C_0}(F)$ can be written
$\ww = \uu \circ \vv$ as above.

The decomposition~\eqref{eq:decomp_FJ2} gives a morphism
$\cor_{\{0\}} \to \roim{q}(F)$, hence
$\ww' \colon \cor_{\{0\}} \to \rsect_{\{0\}}\roim{q}(F)$.  Using  
$\rsect_{\{0\}}\roim{q}(F) \simeq \roim{q}\rsect_{C_0}(F)$ and the adjunction
$(\opb{q},\roim{q})$, we see that $\ww'$ induces a morphism
$\ww \colon \cor_{C_0} \to \rsect_{C_0}(F)$.   Then
$i\circ \ww \colon \opb{q}(\cor_{\{0\}}) \simeq \cor_{C_0} \to F_{C_0}$ corresponds
to $\roim{q}(i) \circ \ww' \colon \cor_{\{0\}} \to \roim{q}(F_{C_0})$ via the
adjunction.  Through the isomorphisms
$\rsect_{\{0\}}\roim{q}(F) \simeq \roim{q}\rsect_{C_0}(F)$ and
$(\roim{q}(F))_{\{0\}} \simeq \roim{q}(F_{C_0})$ the direct image
$\roim{q}(i) \colon \rsect_{\{0\}}\roim{q}(F) \to (\roim{q}(F))_{\{0\}}$ coincides
with the natural morphism $j_{\{0\}}(\roim{q}(F)) \circ i_{\{0\}}(\roim{q}(F))$.  It
follows that $\roim{q}(i) \circ \ww'$ induces the identity morphism on the summand
$\cor_{\{0\}}$ of $\roim{q}F$. Hence $\roim{q}(i) \circ \ww'$ is non zero and its
image by the adjunction, $i\circ \ww$, is also non zero.

\sui(ii) The existence of $\vv, \uu$ only depend on the restriction of $F$ to a
neighborhood of $C_0$.  So from now on we restrict over $\Uu$, but still write $F$
instead of $F|_\Uu$.  We recall that $\DD F = \DD_\Uu F = \rhom(F,\omega_\Uu)$ and,
since $\Uu$ is oriented of dimension $2$, $\DD_\Uu F \simeq (\DD'_\Uu F)[2]$, where
$\DD'_\Uu F = \rhom(F,\cor_\Uu)$.

Using Lemma~\ref{lem:F_degree-1} we have $F \simeq \bigoplus_{i\in \Z} H^iF[-i]$
and $H^{-1}F$ satisfies~\eqref{eq:decomp_FJ2}.  We deduce
$\DD F \simeq \bigoplus_{i\in \Z}\DD(H^iF[-i])$ where $\DD(H^iF[-i])$ is
concentrated in degree $2-i$. Hence $H^{-1}F$ satisfies~\eqref{eq:autodualite}.
Since $\RHom(H^{-1}F[-1], H^{-1}F[-1])$ is a direct summand of $\RHom(F,F)$, we
may assume from the beginning that $F$ is concentrated in degree $-1$.

  \sui(iii) We recall that a locally constant sheaf $G$ on the circle is
  described up to isomorphism by its stalk $G_\theta$ at a given point and its
  monodromy $m\colon G_\theta \to G_\theta$.  The stalk of $\DD'G$ is
  $\Hom(G_\theta,\cor)$ and its monodromy is ${}^tm^{-1}$.  Two locally constant
  sheaves with isomorphic stalks and conjugate monodromies are isomorphic.
    In particular, if $G$ has unipotent monodromy, then
  $\DD'G \simeq G$.  Using Corollary~\ref{cor:sheaves_dim1} we decompose
  $F|_{C_0}$ as
  \begin{equation}
    \label{eq:FC0-monodunip1}
F|_{C_0} \simeq L_\unip[1] \oplus L_\nunip[1] \oplus
\bigoplus_{a\in A'} \oim{e}(\cor_{I_a})[1] ,
  \end{equation}
  where $L_\unip$, $L_\nunip$ are locally constant sheaves, $L_\unip$ has
  unipotent monodromy, no summand of $L_\nunip$ has unipotent monodromy, and $A'$
  is a finite family of bounded intervals or $\R$.  Since $\DD_\Uu F \simeq F$ we
  have
  $\rsect_{C_0}(F)|_{C_0} \simeq \rsect_{C_0}(\DD_\Uu F)|_{C_0} \simeq
  \DD_{C_0}(F|_{C_0})$.  Since $\DD'L_\unip \simeq L_\unip$, we obtain
  \begin{equation}
    \label{eq:FC0-monodunip2}
\rsect_{C_0}(F)|_{C_0} \simeq L_\unip \oplus \DD'L_\nunip \oplus
\bigoplus_{a\in A'} \oim{e}(\cor_{I_a^*}) ,    
  \end{equation}
where $I_a^*$ is the interval such that $\DD'\cor_{I_a} \simeq \cor_{I_a^*}$.

\sui(iv) We see $\ww$ as a section of $\rsect_{C_0}(F)$ using
$\Hom(\cor_{C_0}, \rsect_{C_0}F) \simeq H^0(C_0; \rsect_{C_0}(F))$.  For a
locally constant sheaf $G$ as in~(iii) with monodromy
$m\colon G_\theta \to G_\theta$, we have $H^0(C_0;G) \simeq \{v\in G_\theta$;
$m(v)=v\}$. Hence a locally constant sheaf has a section if and only if its
monodromy has a unipotent factor.  Using the
decomposition~\eqref{eq:FC0-monodunip2} we can then write
$\ww = (\ww_\unip, 0,\sum_a\ww_a)$ where $\ww_\unip$ is a section of $L_\unip$
and $\ww_a$ a section of $\oim{e}(\cor_{I_a^*})$.

We choose a non zero section $\vv_\unip^1$ of $L_\unip$ (recall that
$L_\unip \not= 0$ by assumption) and define $\vv_\unip$ by
$$
\vv_\unip =  \ww_\unip \quad \text{ if $\ww_\unip \not= 0$,}
\qquad\qquad
\vv_\unip =  \vv_\unip^1   \quad \text{ if $\ww_\unip=0$.}
$$
Using the decomposition~\eqref{eq:FC0-monodunip1} we set
$\vv = (\vv_\unip, 0, 0) \colon \cor_{C_0} \to F_{C_0}[-1]$ (viewing $\vv_\unip$ as a
morphism $\cor_{C_0} \to L_\unip$).

  Now, for each $a\in A'$,
we look for $\uu_a \colon L_\unip \to \oim{e}(\cor_{I_a^*})$ such that
$\ww_a = \uu_a \circ \vv_\unip$.  The morphism $\vv_\unip$ is injective (at each
stalk, it is a non zero morphism
$(\cor_{C_0})_\theta = \cor \to (L_\unip)_\theta$, hence injective).  Let $L'$
be its cokernel.  For each $a\in A'$, the exact sequence
$0 \to \cor_{C_0} \to L_\unip \to L' \to 0$ yields the following part of a long
exact sequence
$$
\Hom(L_\unip, \oim{e}(\cor_{I_a^*})) \to[\varphi] \Hom(\cor_{C_0}, \oim{e}(\cor_{I_a^*}))
\to  \Ext^1(L',  \oim{e}(\cor_{I_a^*})) ,
$$
where $\varphi$ is the morphism $f \mapsto f \circ \vv_\unip$.
By the adjunction $(\opb{e}, \oim{e})$ we have
$$
\Ext^1(L', \oim{e}(\cor_{I_a^*})) \simeq \Hom(L', \oim{e}(\cor_{I_a^*})[1])
\simeq \Hom(\opb{e}(L'), \cor_{I_a^*}[1]).
$$
Now $L'$ is locally constant, hence $\opb{e}(L')$ is constant, say
$\opb{e}(L') \simeq \cor_\R^N$, and we have
$\Hom(\opb{e}(L'), \cor_{I_a^*}[1]) \simeq (H^1(\R; \cor_{I_a^*}))^N$.  We
remark that a sheaf $\cor_{I_a^*}$ has a non zero section if and only if $I_a^*$
is closed, in which case we have $H^1(\R; \cor_{I_a^*}) \simeq 0$ and the
morphism $\varphi$ is surjective.  Hence for any $a\in A'$ there exists
$\uu_a \colon L_\unip \to \oim{e}(\cor_{I_a^*})$ such that
$\ww_a = \uu_a \circ \vv_\unip$.

Using the decompositions~\eqref{eq:FC0-monodunip1}, \eqref{eq:FC0-monodunip2} we
define $\uu \colon F_{C_0}[-1] \to \rsect_{C_0}(F)$ by
$$
\uu = \bpmat \id_{L_\unip} & 0 & 0 \\ 0 & 0 &0 \\ \sum \uu_a & 0 &0 \epmat
\text{ if $\ww_\unip \not= 0$},
\qquad
\uu = \bpmat 0 & 0 & 0 \\ 0 & 0 &0 \\ \sum \uu_a & 0 &0 \epmat
\text{ if $\ww_\unip = 0$.}
$$
The equality $\uu \circ \vv = \ww$ follows; indeed it reads respectively in  the
cases $\ww_\unip \not= 0$ and $\ww_\unip = 0$:
$$
\bpmat \id_{L_\unip} & 0 & 0 \\ 0 & 0 &0 \\ \sum \uu_a & 0 &0 \epmat
\bpmat  \ww_\unip  \\ 0 \\ 0 \epmat 
= \bpmat  \ww_\unip \\ 0 \\ \sum_a\ww_a \epmat ,
\qquad
\bpmat 0 & 0 & 0 \\ 0 & 0 &0 \\ \sum \uu_a & 0 &0 \epmat
\bpmat \vv_\unip^1 \\ 0 \\ 0 \epmat
= \bpmat  0 \\ 0 \\ \sum_a\ww_a \epmat .
$$
\end{proof}

\section{Microlocal linked points}
\label{sec:micro_linked}

Let $M$ be a manifold (in Proposition~\ref{prop:mongammasurj} $M$ will be a
surface) and let $\Lambda$ be a smooth conic Lagrangian submanifold of $\dT^*M$.
We consider $F\in \Derb(\cor_M)$ such that $\dot\SSi(F) = \Lambda$ and $F$ is
simple along $\Lambda$.  In this section we give a criterion which implies that
$H^1\RHom(F,F)$ is non zero (see Proposition~\ref{prop:mongammasurj}).
For $U,V$ two open subsets of $M$ such that $M=U\cup V$ and for $G\in \Derb(\cor_M)$
we denote by
\begin{equation}
\label{eq:defHUVG}
H^1(\{U,V\}; G) = H^0(U\cap V; G) / (H^0(U; G) \times H^0(V; G) )
\end{equation}
the first \v Cech group of $G$ associated with the covering $\{U,V\}$ of $M$ (we
do not assume that $U$ and $V$ are connected).  By the Mayer-Vietoris long exact
sequence we have an injective map
\begin{equation}
\label{eq:May-VietUVG}
H^1(\{U,V\}; G) \hookrightarrow H^1(M;G) .
\end{equation}
In particular it is enough for our purpose to find a covering $\{U,V\}$ with
$H^1(\{U,V\}; \rhom(F,F)) \not=0$.  To better understand this latter group, we
use the natural morphism from $\rhom(F,F)$ to $\RR\dot\pi_{M*} \muhom(F,F)$.
Since $F$ is simple we have a canonical isomorphism
$$
\cor_\Lambda \isoto \muhom(F,F)|_{\dT^*M}
$$
sending $1$ to $\id_F$.  We deduce morphisms
$\rhom(F,F) \to \RR\dot\pi_{M*}(\cor_\Lambda)$ and
$$
H^1(\{U,V\};\rhom(F,F)) \to H^1(\{U',V'\}; \cor_\Lambda) ,
$$
where $U' = T^*U\cap \Lambda$, $V' = T^*V \cap \Lambda$.  However we lose too
  much information in this way and we consider a map to another \v Cech
group (see~\eqref{eq:defmongamma}).  The construction of this map relies on the
notion of {\em linked points} in $\Lambda$ introduced in
Definition~\ref{def:linkedpair} below.

We first introduce a general notation.  For $G,G' \in \Derb(\cor_M)$ we recall
the canonical isomorphism~\eqref{eq:proj_muhom_oim}
\begin{equation}
\label{eq:Sato-bis}
\rhom(G,G') \simeq \roim{\pi_M} \muhom(G,G') .
\end{equation}
For an open subset $W$ of $M$ and $p\in T^*W$, we deduce the morphisms
\begin{alignat}{2}
\label{eq:def_umu_general}
\Hom(G|_W,G'|_W) &\to H^0(T^*W; \muhom(G,G')) , &\qquad u &\mapsto u^{\mu+} , \\
\label{eq:def_umup_general}
\Hom(G|_W,G'|_W) &\to  H^0 \muhom(G,G')_p , &\qquad u &\mapsto u^{\mu+}_p . 
\end{alignat}
For $G=G'=F$ with $F$ simple and for $p\in \dot\SSi(F)$ we obtain
\begin{alignat}{2}
\label{eq:def_umu}
\Hom(F|_W,F|_W) &\to H^0(\dT^*W; \cor_\Lambda) , &\qquad u &\mapsto u^\mu , \\
\label{eq:def_umup}
\Hom(F|_W,F|_W) &\to  \cor , &\qquad u &\mapsto u^\mu_p 
\end{alignat}
and we have $u^{\mu+}_p = u^\mu_p \cdot (\id_F)^{\mu+}_p$, where $\cdot$ is the
scalar multiplication (indeed the identification $H^0 \muhom(F,F)_p \simeq \cor$
sends $(\id_F)^{\mu+}_p$ to $1$).

\begin{definition}
\label{def:linkedpair}
Let $W\subset M$ be an open subset and let $p,q \in \Lambda \cap T^*W$ be given
points. We say that $p$ and $q$ are $F$-linked over $W$ if $u^\mu_p = u^\mu_q$
for all $u \in \Hom(F|_W, F|_W)$.
\end{definition}

  Because of the isomorphism $\cor_\Lambda \isoto \muhom(F,F)|_{\dT^*M}$ the
scalar $u^\mu_p$ only depends on the component of $\Lambda \cap T^*W$ containing $p$
and we could also speak of $F$-linked connected components of $\Lambda \cap T^*W$.
In particular, if $\Lambda$ is connected and $W=M$ all points of $\Lambda$ are
$F$-linked over $M$. The notion of $F$-linked points will be used when
$\Lambda \cap T^*W$ is a priori non connected.

\begin{remark}
\label{rem:umu_autredef} 
(a) Let $p=(x;\xi) \in \Lambda$.  Let $\varphi \cl M\to \R$ be a function of class
$C^\infty$ such that $\Lambda$ and $\Lambda_\varphi$ intersect transversely
at $p$.  We set $Z = \{ \varphi \geq \varphi(x) \}$.  For
$G,G' \in \Derb_{[\Lambda]}(\cor_M)$, we have by~\eqref{eq:stalk_muhom}
\begin{equation*}
H^0 \muhom(G,G')_p \simeq
\Hom( (\rsect_Z(G))_x , (\rsect_Z(G'))_x ).
\end{equation*}
If $G=G'=F$ with $F$ simple at $p$, then $(\rsect_Z(F))_x$ is $\cor$ in some degree. For
$u \in \Hom(F,F)$ the morphism $(\rsect_Z(u))_x$ is the multiplication by the
scalar $u^\mu_p$.  

\noindent (b) We keep the notations in~(a). We assume that $u^\mu_p \not=0$.
Defining $H \in \Derb_{[\Lambda]}(\cor_M)$ by the distinguished triangle
$F \to[u] F \to H \to[+1]$, we thus have $(\rsect_Z(H))_x \simeq 0$. On the
other hand $\dot\SSi(H) \subset \Lambda$ by the triangular inequality for the
microsupport and the vanishing of $(\rsect_Z(H))_x$ implies that $\dot\SSi(H)$
does not meet the connected component of $\Lambda$ containing $p$.

\noindent (c) We keep the notations in~(a). We assume that $u$ factorizes
through $H \in \Derb(\cor_M)$ and $p\not\in \SSi(H)$. Then $(\rsect_Z(u))_x$
factorizes through $(\rsect_Z(H))_x \simeq 0$ and we obtain $u^\mu_p = 0$.
\end{remark}

By Theorem~\ref{thm:germemuhom} the functors $u \mapsto u^{\mu+}_p$ or
$u \mapsto u^\mu_p$ are well-defined in the quotient category
$\Derb(\cor_M;p)$. We can also express this result as follows.
\begin{lemma}
\label{lem:umup_DMp}
Let $G,G' \in\Derb(\cor_M)$ be such that
$\dot\SSi(G), \dot\SSi(G') \subset \Lambda$ and let $p\in \Lambda$ be given. We
assume that there exists a distinguished triangle $G \to[g] G' \to H \to[+1]$
with $p\not\in \SSi(H)$.  Then the composition with $g$ induces
isomorphisms
$$
\muhom(G,G)_p \isoto[a] \muhom(G,G')_p \isofrom[b] \muhom(G',G')_p
$$
and we have $a((\id_G)^{\mu+}_p) = g^{\mu+}_p = b((\id_{G'})^{\mu+}_p)$.  In
particular, if $G$ and $G'$ are simple and $u\cl G \to G$ and $v\cl G' \to G'$
satisfy $v \circ g = g \circ u$ then $u^\mu_p = v^\mu_p$.
\end{lemma}
\begin{proof}
  We apply the functor $\muhom(G,\cdot)$ to the given distinguished triangle
  and we take the stalks at $p$.  By the bound~\eqref{eq:suppmuhom} we have $\mu
  hom(G,H)_p \simeq 0$ and we deduce the isomorphism $a$. The isomorphism $b$ is
  obtained in the same way.  Then the last assertion follows from the relations
  $u^{\mu+}_p = u^\mu_p \cdot (\id_G)^{\mu+}_p$ and $v^{\mu+}_p = v^\mu_p \cdot
  (\id_{G'})^{\mu+}_p$.
\end{proof}

We recall that $U,V$ are two open subsets of $M$ such that $M = U \cup V$.  Let
$\gamma \cl [0,1] \to \Lambda$ be a path such that  
\begin{equation}
\label{eq:hyppathgamma}
\begin{minipage}[c]{10cm}
  $p_0 = \gamma(0)$ and $p_1 = \gamma(1)$ belong to
  $(T^*U \setminus T^*V) \cap \Lambda$ and are $F$-linked over $U$; moreover
  $\im(\gamma)$ is not entirely contained in $T^*U \cap \Lambda$.
  \end{minipage} 
\end{equation}
We define a circle $C$ by identifying $0$ and $1$ in $[0,1]$ and  
let $r\colon [0,1] \to C$ be the quotient map.  The natural orientation of
$[0,1]$ induces an orientation on $C$.  We set
$U' = r(\opb{\gamma}(T^*U \cap \Lambda))$,
$V' = r(\opb{\gamma}(T^*V \cap \Lambda))$.    We have
$U'\cup V' = C$ and, by~\eqref{eq:hyppathgamma}, $U'$ and $V'$ are neither empty
nor equal to $C$.  Hence they are unions of non trivial arcs of $C$ and we have
a canonical isomorphism
$H^1(\{U',V'\}; \cor_C) \simeq H^1(C;\cor_C) \simeq \cor$.

For $u \in H^0(U;\rhom(F,F))$ the inverse image of $u^\mu$ by $\gamma$
gives a well-defined section of $H^0(U';\cor_C)$ because $p_0$ and
$p_1$ are $F$-linked over $U$. An element of $H^0(V;\rhom(F,F))$ also
induces a section of $H^0(V';\cor_C)$ because $p_0, p_1 \not\in T^*V
\cap \Lambda$.  We deduce a well-defined map
\begin{equation}
\label{eq:defmongamma}
\mon_\gamma \cl H^1(\{U,V\};\rhom(F,F)) \to H^1(\{U',V'\}; \cor_C) \simeq \cor.
\end{equation}

Now we describe a situation where the map $\mon_\gamma$ is surjective.  We first
introduce a notion of conjugate pair of points.

\begin{definition}
\label{def:conjugate_points}
Let $M$ be a manifold, $\cor$ a field and $F\in \Derb(\cor_M)$.  Let
$q_0 = (x_0;\xi_0)$ and $q_1 = (x_1;\xi_1)$ be given points of $\dot\SSi(F)$,
generating two distinct half-lines $\rspos\cdot q_0 \not= \rspos\cdot
q_1$.   Let $I$ be either an open interval of $\R$ or the
circle $\cer$ and let $i \cl I \to M$ be an embedding.  We say that $q_0$ and
$q_1$ are $F$-conjugate with respect to $i$ if
\begin{itemize}
\item [(i)] $x_0, x_1 \in \im(i)$; we write $t_0 = i^{-1}(x_0)$,
  $t_1 = i^{-1}(x_1)$,
\item [(ii)] for $k=0,1$, there exist a neighborhood $U_k \subset M$ of $x_k$
  and a hypersurface $N_k \subset U_k$ such that
  $\SSi(F) \cap \dT^*U_k \subset \dT^*_{N_k}U_k$ and $i$ is transverse to $N_k$
  at $t_k$,
\item [(iii)] $F$ is simple along $\SSi(F)$ at $q_k$ for $k=0,1$,
\item [(iv)] $\opb{i}F$ has a direct summand $F'$ such that
  $(t_k;i_d(\xi_k)) \in \SSi(F')$ for $k=0,1$ and $F'$ is isomorphic to
$$
\begin{cases}
\text{$\cor_J[d]$ for some interval $J$ of $I$} 
& \text{if $I$ is an interval,} \\
\text{$\oim{e}(\cor_J)[d]$ for some interval $J$ of $\R$} 
& \text{if $I=\cer$,}
\end{cases}
$$
where $d\in \Z$ and $e \cl \R \to \cer$ is the covering map.
\end{itemize}
\end{definition}

We will see in Proposition~\ref{prop:mulink-restrdim1}   that the notion of
conjugate and linked points are related.

\begin{proposition}
\label{prop:mongammasurj}
Let $M$ be a surface and let $\Lambda$ be a smooth conic Lagrangian submanifold of
$\dT^*M$ such that $\Lambda/\rspos$ is a circle.  Let
$F\in \Derb_{[\Lambda]}(\cor_M)$ such that $F$ is simple along $\Lambda$.  We assume
that there exists an embedding of the circle $i \cl \cer \to M$ and $p_0$, $p_1$,
$q_0$, $q_1 \in \Lambda$ such that
\begin{itemize}
\item [(a)]  $\im(i)$ meets $\pi_M(\Lambda)$ at smooth points and transversely,
\item [(b)] $M\setminus \im(i)$ has two connected components $U$, $V$,
\item [(c)] $q_0$ and $q_1$ are $F$-conjugate with respect to $i$,
\item [(d)] $p_0$ and $p_1$ are $F$-linked over $U$,
\item [(e)] the pairs $\{[p_0], [p_1]\}$ and $\{[q_0], [q_1]\}$ in $\Lambda/\rspos$
  are intertwined, i.e. each   of the two arcs from $[p_0]$ to $[p_1]$ contains exactly
  one point of $\{[q_0], [q_1]\}$ (where $[p]$ denotes the image of $p\in \Lambda$ in
  $\Lambda/\rspos$).
\end{itemize}
Let $\gamma$ be a path from $p_0$ to $p_1$ in $\Lambda$ lifting an arc in
$\Lambda/\rspos$.  Then we can increase $U$, $V$ so that the map $\mon_\gamma$
of~\eqref{eq:defmongamma} is defined and surjective.  In particular
$H^1\RHom(F,F) \not= 0$.
\end{proposition}
The hypotheses are illustrated in Fig.~\ref{fig:dem_un_cusp}
and~\ref{fig:dem_un_cusp2} for the proof of Theorem~\ref{thm:ext1FF}: the embedding
$i$ is the inclusion of the circle $C_{1/2}$ and the points $p^+_l$, $q_l$ are
$F$-conjugate with respect to $i$; the points $p^+$, $p_0$ are $F$-linked over the
lower open set bounded by $C_{1/2}$; the dotted path in Fig.~\ref{fig:dem_un_cusp} is
$\im(\pi_M \circ \gamma)$.

\begin{proof}
  (i) By the hypothesis~(a) we can assume that $\im(i)$ has a neighborhood
  $\shn$ of the type $\shn = \cer \times K$, where $K = \mo]-1,1\mc[$ and
  $i$ corresponds to the embedding of $\cer \times \{0\}$, such that
  $\Lambda \cap T^*\shn = \Lambda_0 \times T^*_KK$ for some subset
  $\Lambda_0 \subset \dT^*\cer$. Since $p_0$, $p_1 \in T^*(M\setminus\cer)$ we
  can also assume that $p_0$, $p_1 \not\in \ol{T^*\shn}$.  By
  Theorem~\ref{thm:iminv} we can write $F|_\shn \simeq \opb{p}F_0$ where
  $p \colon \shn \to \cer$ is the projection and $F_0 \in \Der(\cor_\cer)$. We
  must have $F_0 = \opb{i}F$.  We write   $q_0 = (a,0; \alpha,0)$,
  $q_1 = (b,0; \beta,0)$ with $a$, $b\in \cer$ and $\alpha$, $\beta \not= 0$.
  By the definition of conjugate points, $F_0$ has a direct summand isomorphic
  to $\oim{e}(\cor_J)[d]$, for some interval $J$ of $\R$ whose endpoints are
  mapped to $a$ and $b$ by $e \cl \R \to \cer$   and such that
  $\dot\SSi(\oim{e}(\cor_J)) = \rspos \cdot (a;\alpha) \sqcup \rspos \cdot
  (b;\beta)$.  Hence we can write
\begin{equation}
\label{eq:decFOmega}
F|_\shn \simeq F' \oplus F'',
\end{equation}
where $F' = \opb{p}(\oim{e}(\cor_J)[d])$ and
$\dot\SSi(F') = \dot\SSi(\oim{e}(\cor_J)) \times T^*_KK$. In other words
$\dot\SSi(F')$ consists of the two components of $\Lambda \cap T^*\shn$
containing $q_0$ and $q_1$.    We remark that the components of
$\Lambda \cap T^*\shn$ are of the form
$\Lambda_i = \rspos \cdot (a_i;\alpha_i) \times T^*_KK$, $i=1,\dots,n$, with
$(a_i;\alpha_i) \in T^*\cer$.  Since $\gamma \colon [0,1] \to \Lambda$ lifts an
arc of $\Lambda/\rspos$ whose two ends, $p_0$ and $p_1$, do not belong to
$\ol{T^*\shn}$, it follows that $I_i = \gamma^{-1}(\Lambda_i)$ is either empty
or an interval of $[0,1]$. Moreover $\ol{I_i}$ does not contain $0$, $1$.  The
hypothesis~(e) says that $\im(\gamma)$ meets exactly one of the two components
of $\dot\SSi(F')$.  Hence $\opb{\gamma}(\SSi(F'))$ consists of a single
interval.  We choose the indices so that $\opb{\gamma}(\SSi(F')) = I_1$,
$\im(\gamma)$ meets $\Lambda_1,\dots,\Lambda_{n'}$ and does not meet
$\Lambda_{n'+1},\dots,\Lambda_n$, for some $n'\leq n$.

\sui(ii) We set $U^+ = U \cup \shn$ and $V^+ = V \cup \shn$.  Hence
$U^+\cap V^+ = \shn$.   Let $x\in \cor$ be given.  Using the
decomposition~\eqref{eq:decFOmega} we define a morphism
$\alpha(x) \cl F|_\shn \to F|_\shn$ by
\begin{equation}
\label{eq:defalphaFUV}
\alpha(x) = \begin{pmatrix}
  x \cdot \id_{F'}  \\ & 0_{F''} 
\end{pmatrix}.
\end{equation}
As described after~\eqref{eq:hyppathgamma} we set $C = [0,1]/(0\sim 1)$ and let
$r\colon [0,1] \to C$ be the quotient map. We set
$U' = r(\opb{\gamma}(T^*U^+ \cap \Lambda))$ and
$V' = r(\opb{\gamma}(T^*V^+ \cap \Lambda))$. Then $U'$ and $V'$ are unions of
possibly several arcs of $C$, say $U' = \bigsqcup_{i = 1}^l U'_i$,
$V' = \bigsqcup_{i = 1}^m V'_i$.    We have
$U' \cap V' = r(\opb{\gamma}(\Lambda \cap T^*\shn)) = \bigsqcup_{i = 1}^{n'}
r(I_i)$ and in fact $m = l = n'/2$.  By definition $H^1(\{U',V'\}; \cor_C)$ is
the cokernel of the morphism
$$
d \colon \bigoplus_{i = 1}^l H^0(U'_i; \cor_C) \oplus
\bigoplus_{i = 1}^m H^0(V'_i; \cor_C) 
\to \bigoplus_{i = 1}^{n'} H^0(r(I_i); \cor_C),
$$
where each $H^0(-; \cor_C)$ is isomorphic to $\cor$ and $d$ is induced by the
obvious restriction maps (we remark that
$H^0(U'_i; \cor_C) \to H^0(r(I_j); \cor_C)$ is the identity map if
$r(I_j) \subset U'_i$, the zero map else).  We know that
$H^1(\{U',V'\}; \cor_C) \simeq H^1(C, \cor_C) \simeq \cor$ and the quotient map
induces $H^0(r(I_i); \cor_C) \isoto \cor$, for each $i = 1,\dots,n'$.

By the definition of $\alpha(x)$, the section $\alpha(x)^\mu$ of
$\muhom(F,F) \simeq \cor_\Lambda$ is the scalar $x$ over the two components of
$\Lambda \cap T^*\shn$ given by $\dot\SSi(F')$ and $0$ over all other
components of $\Lambda \cap T^*\shn$.  Taking the pull-back by $\gamma$ and
the image by $r$, and remembering that $\opb{\gamma}(\SSi(F')) = I_1$ is a
single interval, we obtain that $\mon_\gamma([\alpha(x)])$ is represented by
$(x_i)_{i=1}^{n'} \in \bigoplus_{i = 1}^{n'} H^0(r(I_i); \cor_C)$ with $x_1 = x$
and $x_i = 0$ for $i\not=1$.  It follows that $\mon_\gamma([\alpha(x)])$ is non
zero if $x\not=0$, hence $\mon_\gamma$ is surjective.

In particular $H^1(\{U,V\};\rhom(F,F)) \not=0$ and, by~\eqref{eq:May-VietUVG},
we have $H^1\RHom(F,F) \simeq H^1(M; \rhom(F,F)) \not= 0$.
\end{proof}

\section{Examples of microlocal linked points}
\label{sec:exmiclinkpt}

We begin with the following easy example over $\R$. In higher dimension we give
Propositions~\ref{prop:mulink-restrdim1} and~\ref{prop:mulink-projdim1} below
which reduce to this case by inverse or direct image.

\begin{proposition}
\label{prop:mulink-dim1}
Let $t_0 \leq t_1 \in \R$ and $p_0 = (t_0;\tau_0)$, $p_1 = (t_1;\tau_1) \in \dT^*\R$
be given.  Let $F \in \Derb(\cor_\R)$ be a constructible sheaf with
$p_0 , p_1 \in \SSi(F)$ such that $F$ is simple at $p_0, p_1$.  We assume that there
exists a decomposition $F \simeq G \oplus \cor_I[d]$ where $G \in \Derb(\cor_\R)$,
$d\in \Z$ and $I$ is an interval with ends $t_0,t_1$ such that
$p_0,p_1 \in \SSi(\cor_I)$.  Then $p_0$ and $p_1$ are $F$-linked over any open  
interval containing $\ol{I}$.
\end{proposition}
\begin{proof}
  Let $i \cl \cor_I[d] \to F$ and $q \cl F \to \cor_I[d]$ be the morphism
  associated with the decomposition of $F$.  Then $i$ and $q$ give a morphism
  $\Hom(F,F) \to \Hom(\cor_I,\cor_I) \simeq \cor$.  Since $F$ is simple at
  $p_k$, for $k=0,1$, and $p_k \in \SSi(\cor_I)$, we have $p_k \not\in \SSi(G)$.
  Let $u\colon F \to F$ be given and let ${}_Iu \colon \cor_I[d] \to \cor_I[d]$
  be the morphism induced by $u$.  Then $u^\mu_{p_k} = ({}_Iu)^\mu_{p_k}$ by
  Lemma~\ref{lem:umup_DMp}.  Since $\Hom(\cor_I,\cor_I) \simeq \cor$ we have
  $({}_Iu)^\mu_{p_0} = ({}_Iu)^\mu_{p_1}$ for any $u$ defined in a neighborhood
  of $\ol{I}$.  The result follows.
\end{proof}

Let $f \cl M \to N$ be a morphism of manifolds. We recall the notations $f_d \cl
M \times_N T^*N \to T^*M$ and $f_\pi \cl M \times_N T^*N \to T^*N$ for the
natural maps induced on the cotangent bundles.

We study easy cases of inverse or direct images of simple sheaves.  We let
$F \in \Derb(\cor_M)$ be such that $\Lambda = \dot\SSi(F)$ is a smooth
Lagrangian submanifold and $F$ is simple along $\Lambda$.  More general, but
local, statements are given in~\cite{KS90} (see Corollaries~7.5.12 and~7.5.13).

\begin{lemma}
\label{lem:bonne_im_inv}
Let $i \cl L \to M$ be an embedding and let $x_0$ be a point of $L$.  We assume
that there exist a neighborhood $U \subset M$ of $i(x_0)$ and a submanifold $N
\subset U$ such that $\Lambda \cap \dT^*U \subset \dT^*_NU$ and $N$ is
transverse to $L$ at $x_0$.  Then, up to shrinking $U$ around $x_0$, the set
$\Lambda' = \dot\SSi((\opb{i}F)|_{U \cap L})$ is contained in $\dT^*_{N \cap
  L}L$ and $(\opb{i}F)|_{U \cap L}$ is simple along $\Lambda'$.

Moreover, if $u \cl F \to F$ is defined in a neighborhood of $x_0$ and $v \cl
\opb{i}F \to \opb{i}F$ is the induced morphism, then for any $p = (i(x_0);\xi_0)
\in \Lambda$ and $q = (x_0; i_d(\xi_0))$ we have $v^\mu_q = u^\mu_p$.
\end{lemma}
We note that the inclusion $\Lambda \cap T^*U \subset \dT^*_NU$ implies that
$\Lambda$ is a union of components of $\dT^*_NU$ (hence this is an equality if
$N$ is connected of codimension $\geq 2$).
\begin{proof}
  Up to shrinking $U$ we can find a submersion $f \cl U \to L\cap U$ such that
  $N = \opb{f} (L \cap N)$ and $f\circ i = \id_{L\cap U}$.  By
  Proposition~\ref{prop:iminvproj} we can write $F|_U \simeq \opb{f}G$ for some
  $G \in \Derb(\cor_{L\cap U})$. We must have $G = (\opb{i}F)|_{U \cap L}$.
  Then $\muhom(F,F) \simeq \oim{f_d} \opb{f_\pi}(\muhom(G,G))$ over $U$ and
  the lemma follows.
\end{proof}

\begin{lemma}
\label{lem:bonne_im_dir}
Let $q \cl M \to \R$ be a function of class $C^2$ and let $t_0$ be a regular
value of $q$.  We assume that there exist an open interval $J$ around $t_0$, a
connected hypersurface $L$ of $U \eqdot \opb{q}(J)$ and a connected component
$\Lambda_0$ of $\dT^*_LU$ such that
\begin{itemize}
\item [(a)] $q|_U \colon U \to J$ is proper on $\supp F$,
\item [(b)] $q|_L$ is Morse with a single critical point $x_0$ and $q(x_0) =
  t_0$,
\item [(c)] $\Lambda_0 \subset \Lambda \cap T^*U$ and $T^*_{x_0}M \cap \Lambda
  = T^*_{x_0}M \cap \Lambda_0$,
\item [(d)] $((\Lambda \cap T^*U) \setminus \Lambda_0) \cap q_d(M \times_\R
  \dT^*\R) = \emptyset$.
\end{itemize}
Let $p = (x_0;\xi_0)$ be the point of $\Lambda_0$ ($\xi_0$ is unique up to a
positive scalar) above $x_0$. We have $p \in \im(q_d)$ and we set $p' =
(t_0;\tau_0) = q_\pi(\opb{q_d}(p))$.  Then $p' \in \SSi(\roim{q}F)$ and
$\roim{q}(F)$ is simple at $p'$.  Moreover, for any morphism $u \cl F|_U \to
F|_U$, denoting by $v = \roim{q}(u) \cl \roim{q}F|_J \to \roim{q}F|_J$ the
induced morphism, we have $v^\mu_{p'} = u^\mu_p$.
\end{lemma}
\begin{proof}
  (i) We set $N = \opb{q}(t_0)$.    Then $N$ is a smooth hypersurface of $M$ and, up to shrinking
  $J$ and restricting to a neighborhood of $\supp(F) \cap N$, we can assume that
  $U = N \times J$.  The hypersurfaces $N$ and $L$ of $U$ are tangent at the
  point $x_0$.   We choose a distance function
  $d_N$ on $N$ and define $f\colon U \to \R$, $(x,t) \mapsto d_N(x_0,x)$.  For
  $r>0$ we set $V_r = f^{-1}([0,r[)$ and $Z_r = U \setminus V_r$.  We have a
  distinguished triangle $F_{V_r} \to F \to F_{Z_r} \to[+1]$.  By~(c) we can
  find a ball $B$ around $x_0$ in $U$ such that
  $T^*B \cap \Lambda = T^*B \cap \Lambda_0$.  By Theorem~\ref{thm:SSrhom} we
  obtain $\dot\SSi(F_{Z_r}) \cap T^*B \subset \Lambda(r)$ where
  $$
  \Lambda(r) = ((\Lambda_0 \cap \pi_U^{-1}(Z_r)) \cup T^{*,out}_{\partial V_r}U
  \cup (\Lambda_0 +  T^{*,out}_{\partial V_r}U)) \cap T^*B,
  $$
  and $T^{*,out}_{\partial V_r}U$ is the outer conormal bundle of $\partial V_r$
  in $U$.  We claim that $\Lambda(r) \cap T^*_NU = \emptyset$ for a dense set of
  $r$. Indeed it is enough to see that
  $(T^*_LU + T^*_{f^{-1}(r)}U) \cap T^*_NU \cap T^*B = \emptyset$.  Since $x_0$
  is the only critical point of $q|_L$, $L \cap N$ is smooth outside $x_0$ and
  $(T^*_LU + T^*_{f^{-1}(r)}U) \cap T^*_NU$ is empty if and only if
  $T^*_{L\cap N}U \cap T^*_{f^{-1}(r)}U$ is empty; this means that $r$ is a
  regular value of $f|_{L\cap N}$ and happens for a dense set of $r$.

  \sui(ii)   We choose $r$ and
  $\varepsilon_0 >0$ such that $\Lambda(r) \cap T^*_NU = \emptyset$ and
  $\ol{V_r} \cap q^{-1}([-\varepsilon_0, \varepsilon_0]) \subset B$.  Then, for
  $0<\varepsilon<\varepsilon_0$ small enough, we have
  $\Lambda(r) \cap q_d((N \times [-\varepsilon, \varepsilon]) \times_\R \dT^*\R) =
  \emptyset$. It follows from~(d) that
  $\dot\SSi(F_{Z_r}) \cap q_d((N \times [-\varepsilon, \varepsilon]) \times_\R
  \dT^*\R) = \emptyset$.  Hence $\roim{q}(F_{Z_r})|_{[-\varepsilon, \varepsilon]}$ is
  a constant sheaf and we can assume from the beginning that $F = F_{V_r}$.

  \sui(iii) We set $W = V_r \cap (N \times [-\varepsilon, \varepsilon])$.
  Taking $r$ and $\varepsilon$ smaller if necessary we are in the situation of
  Example~\ref{ex:SS=conormal_hypersurface} and there exist
  $E, E' \in \Derb(\cor)$ and a distinguished triangle
  $E'_W \to[u] E_{W'} \to F|_W \to[+1]$ in $\Derb(\cor_W)$, where $W'$ is the
  closure of one component of $W\setminus L$.  Since $F$ is simple, we must have
  $E = \cor[d]$ for some integer $d$.  Since
  $(\roim{q}(E'_{W}))|_{[-\varepsilon, \varepsilon]}$ is constant, we deduce
  $(\rsect_Y(\roim{q}(F_{V_r})))_{t_0} \simeq
  (\rsect_Y(\roim{q}(\cor_{W'}[d])))_{t_0}$, where $Y = \mo]-\infty,0]$ or
  $Y = [0,+\infty[$ according to the sign of $\tau_0$.  Now the problem is
  reduced to the computation of $\roim{q}(\cor_{W'})$ in a neighborhood of
  $t_0$, which is a classical computation in Morse theory (in our case this is
  done in the proof of Proposition~7.4.2 of~\cite{KS90}).
\end{proof}

Proposition~\ref{prop:mulink-dim1} and Lemmas~\ref{lem:bonne_im_inv}
and~\ref{lem:bonne_im_dir} imply the following results.
\begin{proposition}
\label{prop:mulink-restrdim1}
Let $M$ be a manifold, $\cor$ a field and $F\in \Derb(\cor_M)$.  Let
$p_0, p_1 \in \dot\SSi(F)$ be given.  We assume that there exists an embedding
$i \cl I \to M$ of the circle or an open interval of $\R$ such that $p_0$ and
$p_1$ are $F$-conjugate with respect to $i$.  Then $p_0$ and $p_1$ are
$F$-linked over any open subset containing $i(I)$.
\end{proposition}

\begin{proposition}
\label{prop:mulink-projdim1}
Let $q \cl M \to \R$ be a map of class $C^2$ and let $t_0 \leq t_1$ be regular
values of $q$.  For $k=0,1$ we assume that there exist an open interval $J_k$
around $t_k$, a connected hypersurface $L_k$ of $U_k \eqdot \opb{q}(J_k)$ and a
connected component $\Lambda_k$ of $\dT^*_{L_k}U_k$ such that the
hypotheses~(a)-(d) of Lemma~\ref{lem:bonne_im_dir} are satisfied.  We assume
moreover that $\roim{q}F$ has a decomposition $\roim{q}F \simeq G \oplus
\cor_J[d]$ where $G \in \Derb(\cor_\R)$, $d\in \Z$ and $J$ is an interval with
ends $t_0, t_1$.  Let $p_k = (x_k;\xi_k) \in \Lambda$ be such that $q(x_k) =
t_k$ and $q_\pi(\opb{q_d}(p_k)) \in \dot\SSi(\cor_J)$. Then $p_0$ and $p_1$ are
$F$-linked over any open subset containing $\opb{q}([t_0,t_1])$.
\end{proposition}

  Now we check that in a generic situation a given point has a
unique conjugate.

\begin{proposition}
  \label{prop:exist_conj_pt}
  Let $M$ be a manifold, let $\Lambda \subset \dT^*M$ be a smooth conic
  Lagrangian submanifold and let $i \colon \cer \to M$ be an embedding of the
  circle.  We assume that, in some neighborhood $U$ of $i(\cer)$, there exists a
  hypersurface $N$ of $U$ which is transverse to $i$ and such that
  $\Lambda \subset T^*_NU$. Let $F \in \Derb(\cor_M)$ be such that
  $\dot\SSi(F) = \Lambda$, $F$ is simple along $\Lambda$ and $\opb{i}F$ is
  constructible.  Then, for any $q_0 \in \Lambda \cap (i(\cer) \times_M T^*M)$
  there exists a unique $q_1 \in \Lambda \cap (i(\cer) \times_M T^*M)$ (unique
  up to the action of $\rspos$) such that $q_0$ and $q_1$ are $F$-conjugate with
  respect to $i$.
\end{proposition}
\begin{proof}
  We set $\Lambda' = i_d\opb{i_\pi}(\Lambda)$. The hypothesis on $\Lambda$ implies
  that $\Lambda'$ is in bijection with $\Lambda \cap (i(\cer) \times_M T^*M)$.  Up to
  shrinking $U$ we can find a diffeomorphism $f \colon U \isoto \cer \times \R^d$,
  $d=\dim M -1$, and a finite subset $Z$ of $\cer$ such that $f(N) = Z \times \R^d$
  and $f\circ i$ is the inclusion of $\cer \times \{0\}$.
  By
  Proposition~\ref{prop:iminvproj} we can write
  $\oim{f}(F|_U) \simeq \opb{p}\opb{i}(F)$ where $p \colon \cer \times \R^d \to \cer$
  is the projection.  It follows that $\SSi(\opb{i}(F)) = \Lambda'$. Moreover
  $\opb{i}(F)$ is simple along $\Lambda'$ by Lemma~\ref{lem:bonne_im_inv}.

  Using Corollary~\ref{cor:sheaves_dim1}, we write
  $\opb{i}(F) \simeq L \oplus \bigoplus_{a\in A} \oim{e}(\cor^{n_a}_{I_a})[d_a]$,
  where $L$ is locally constant and $A$ is a finite family of bounded intervals of
  $\R$.  Since the $I_a$'s are bounded we can identify the set of ends of these
  intervals with $B = A \times \{left, right\}$.  Taking the microsupport gives a map
  $\mu$ from $B$ to the set of half-lines of $\dT^*\cer$.   More precisely, for an interval $I_a$ with $\ol{I_a} = [x,y]$ we have
  $\mu(I_a, left) = \{(e(x); \xi)$; $\varepsilon \xi>0\}$ with $\varepsilon = 1$ if
  $I_a$ is closed near $x$, $\varepsilon = -1$ if $I_a$ is open near $x$, and
 $\mu(I_a, right) = \{(e(y); \xi)$; $\varepsilon \xi>0\}$ with $\varepsilon = -1$ if
  $I_a$ is closed near $y$, $\varepsilon = 1$ if $I_a$ is open near $y$.

  Since $\opb{i}(F)$ is simple, the integers $n_a$ are all equal to $1$ and the
  map $\mu$ is injective, inducing a bijection between $B$ and
  $\Lambda'/\rspos$.  Now, two points
  $q_0, q_1 \in \Lambda \cap (i(\cer) \times_M T^*M)$ are $F$-conjugate with
  respect to $i$ if and only if their images in $\Lambda'$ correspond, via
  $\mu$, to the right and left ends of the same interval.  The result follows.
\end{proof}

\section{Generic tangent point - global study}
\label{sec:trois_points_lies}

In this section we find some $F$-linked points in order to be able to apply
Proposition~\ref{prop:mongammasurj} in the proof of Theorem~\ref{thm:ext1FF}.
We first describe the geometric situation.  We assume now that $M = \sph$ is the
sphere.  We keep the notations and assumptions of~\S\ref{sec:threecusps_nothyp}:
$\Uu\subset M$ is an open subset diffeomorphic to $\cer \times \Jj$, with
$\Jj = \mo]-1,1\mc[$, and $\Lambda \subset \dT^*M$ is a smooth closed conic
Lagrangian submanifold satisfying~\eqref{eq:hypLambda_gentang}.  We will also
use the notations $\GammaE$, $\Omega$, $\Lambda_0$ and $p^\pm \in\Lambda_0$
in~\eqref{eq:notJELambda0} and~\eqref{eq:hypLambda_gentang}.

Since $M$ is the sphere, the boundary $\cer \times \{-1\}$ of $\Uu$ bounds a disc
$D$.  For $t \in [-1,1]$ we set
\begin{equation}
  \label{eq:defMt}
M_t = D \cup (\cer \times [-1,t[).   
\end{equation}

\begin{proposition}
\label{prop:otherFlinkedpoint}
Let $F \in \Derb_{[\Lambda]}(\cor_M)$ be a simple sheaf.  We assume that $F|_\Uu$ is
constructible (see Section~\ref{sec:constructibility}) and
satisfies~\eqref{eq:decomp_FJ2}.  Then there exists $p \in \Lambda \cap T^*M_0$
such that $p$, $p^-$ and $p^+$ are $F$-linked over $M_t$, for any
$t \in \mo]0,1]$.
\end{proposition}
\begin{proof}
  (i)   By~\eqref{eq:notJELambda0} and~\eqref{eq:hypLambda_gentang} the group
  $\Hom(F|_{M_t}, F|_{M_t})$ is independent of $t \in \mo]0,1]$ and we can
  assume $t=1$.  By Proposition~\ref{prop:mulink-projdim1} (applied with
  $t_0 = t_1 =0$ and $J = \{0\}$) we know that $p^-$ and $p^+$ are $F$-linked
  over $M_1$.  Hence it is enough to find $p \in \Lambda \cap T^*M_0$ which is
  $F$-linked with $p^+$ over $M_1$.  We argue by contradiction and assume that
  there exists no such $p$.  Since $\Lambda/\rspos$ is compact and
  $\Lambda \cap T^*\Uu$ has the form given in~\eqref{eq:hypLambda_gentang}, we see
  that $\Lambda \cap T^*M_1$ is a finite union of connected components, say
  $$
\Lambda \cap T^*M_1 = \bigsqcup_{i=0}^N \Lambda_i,
  $$
  where $\Lambda_0$ is already defined in~\eqref{eq:notJELambda0}.  The
  components of $\Lambda \cap T^*M_0$ are then $\Lambda_i \cap T^*M_0$,
  $i=1,\dots,N$.

  \smallskip
  
    Before we go on we discuss the idea of the proof
  since it is rather long.  The statement implies in particular that $N\geq 1$.  We
  could indeed see now that $N\geq 1$ as follows: if $N=0$, then $F|_{M_1}$ is a
  simple sheaf on the disc $M_1$ with $\dot\SSi(F|_{M_1}) = \Lambda_0$.  Using
  Example~\ref{ex:SS=conormal_hypersurface} we could describe $F$ and see that it
  cannot satisfy the decomposition hypothesis~\eqref{eq:decomp_FJ2}. The idea is to
  {\em try to} reduce to this situation and obtain a contradiction. We define
  $u \cl F|_{M_1} \to F|_{M_1}$ satisfying~\eqref{eq:otherFlink1} below.  The cone of
  $u$, say $C(u)$, is such that $\dot\SSi(C(u)) = \Lambda_0$ but a priori $C(u)$ does
  not satisfy~\eqref{eq:decomp_FJ2}.  We modify $u$ in~(ii) and~(iii) below to obtain
  another morphism, say $u^1$, ($u^1 = u+v$ in~\eqref{eq:otherFlink6}) such that
  $u^1$ is an isomorphism on $M_0$, and not only a ``microlocal isomorphism'' along
  $\Lambda \cap T^*M_0$ as is the case for $u$. (For this we modify $u$ by a morphism
  $v$ which factorizes through a constant sheaf.)  Now the cone of $u^1$ still does
  not satisfy~\eqref{eq:decomp_FJ2} but it is zero on $M_0$ and we easily obtain a
  contradiction in~(iv), using~\eqref{eq:decomp_FJ2} for $F$ (in fact the case $N=0$
  was the motivation for introducing $u$, $u^1$ and their cones, but the proof is
  finally not a reduction to $N=0$).

  \smallskip

  Now we resume the course of the proof and we build $u$.  We have
  $p^+ \in \Lambda_0$ and we choose $p_i \in \Lambda_i \cap T^*M_0$ for each
  $i=1,\dots,N$. By Definition~\ref{def:linkedpair} there exists
  $u_i \cl F|_{M_1} \to F|_{M_1}$ such that $(u_i)^\mu_{p_i} \not= (u_i)^\mu_{p^+}$,
  for each $i =1,\dots,N$.  Adding a multiple of $\id_F$ and   rescaling we
  can even assume $(u_i)^\mu_{p_i} = 1$ and $(u_i)^\mu_{p^+} = 0$.

  Since $\cor$ is infinite, we can find $\ul a = (a_i)_{i=1,\dots,N} \in \cor^N$
  outside the union of hyperplanes $\bigcup_{i=1,\dots,N} P_i$, where
  $P_i = \{ \ul a$; $\sum_{j=1,\dots,N} a_j (u_j)^\mu_{p_i} = 0\}$ (we remark
  that $P_i \not= \cor^N$ because the coefficient of $a_i$ is $1$).  Then
  $u = \sum_{j=1,\dots,N} a_j \, u_j \cl F|_{M_1} \to F|_{M_1}$ satisfies
\begin{equation}
\label{eq:otherFlink1}
\text{$u^\mu_{p^+} = 0$ and
$u^\mu_{p_i} \not= 0$ for all $i =1,\dots,N$.} 
\end{equation}

\sui(ii) Since $F|_\Uu$ is constructible the algebra $\Hom(F|_\Uu,F|_\Uu)$ is
finite dimensional and there exists a non zero polynomial $P \in \cor[X]$ such
that $P(u|_\Uu)=0$.  Let us write $P(X) = Q(X) \cdot X^k$, where $Q(0) \not=0$.
Then $R(X) = Q(X) + X$ is prime with $P(X)$ and it follows that $R(u|_\Uu)$ is
an isomorphism (we can find $A,B\in \cor[X]$ with $AR+BP=1$ and we obtain
$A(u)R(u) = \id_F$).  It follows that $R(u)$ is an isomorphism (on
$M_1$). Indeed, let $F'$ be the cone of $R(u)$. By the triangular inequality for
the microsupport we have $\dot\SSi(F') \subset \Lambda \cap T^*M_1$.    Since $R(u|_\Uu)$ is an isomorphism we also have
$\dot\SSi(F') \cap T^*\Uu = \emptyset$.  But a microsupport cannot be a proper
subset of a smooth connected Lagrangian submanifold by
Corollary~\ref{cor:micsuppLagrlisse}.  Since all components of
$\Lambda \cap T^*M_1$ meet $T^*\Uu$ we deduce $\dot\SSi(F') = \emptyset$. Hence
$F'$ is locally constant on $M_1$. Since it vanishes on $\Uu$, it vanishes
everywhere on $M_1$. This means that $R(u)$ is an isomorphism on $M_1$.

\sui(iii) We claim that there exists  
$v \colon F|_{M_1} \to F|_{M_1}$ such that
\begin{equation}
\label{eq:otherFlink2}
\text{$v|_{M_0} = Q(u)|_{M_0}$ and
$v^\mu_p = 0$ for any $p\in \Lambda \cap T^*M_1$.}
\end{equation}
To construct $v$ we first define $G \in \Derb(\cor_{M_{1}})$ and $u'$ by the
distinguished triangle
\begin{equation*}
G \to[u'] F|_{M_1} \to[u^k] F|_{M_1} \to G[1] . 
\end{equation*}
By the triangular inequality for the microsupport we have
$\dot\SSi(G) \subset \Lambda \cap T^*M_1$.  By~\eqref{eq:otherFlink1} and
Remark~\ref{rem:umu_autredef}-(b) $\dot\SSi(G)$ avoids the components
$\Lambda_i$ for $i=1,\dots,N$. Hence $\dot\SSi(G) \subset \Lambda_0$.  In
particular $\dot\SSi(G|_{M_0})$ is empty and $G|_{M_0}$ is locally constant,
hence constant since $M_0$ is a disc. Let us write
$G|_{M_0} = K_{M_0} = \opb{a_{M_0}}(K)$ for some $K \in \Derb(\cor)$ where
$a_{M_0}\colon M_0 \to \pt$ is the projection,

Since $u^k \circ Q(u) = 0$ there exists, $w' \colon F \to G$ such that
$Q(u) = u' \circ w'$.  Now both morphisms
$u'|_{M_0} \colon K_{M_0} \to F|_{M_0}$ and
$w'|_{M_0} \colon F|_{M_0} \to K_{M_0}$ extend to $M_1$ as
$u'' \colon K_{M_1} \to F|_{M_1}$ and $w'' \colon F|_{M_1} \to K_{M_1}$.  To see
this we use $K_{M_0} = \opb{a_{M_0}}(K)$, $K_{M_0} \simeq \epb{a_{M_0}}(K)[-2]$ and
the adjunctions  
$(\opb{a_{M_0}}, \rsect(M_0;-)$ and $(\rsect_c(M_0;-), \epb{a_{M_0}})$, which yield
\begin{align}
  \label{eq:otherFlink3}
\Hom(K_{M_0}, F|_{M_0}) &\simeq  \Hom(K, \rsect(M_0; F)), \\
\Hom(F|_{M_0}, K_{M_0}) &\simeq  \Hom(\rsect_c(M_0; F), K[-2]) .
\end{align}
We have similar isomorphisms with $M_0$ replaced by $M_1$.  We extend the
projection $q \colon \Uu \to \Jj$ as $q\colon M_1 \to \R$ by setting
$q(\ol D) = -1$.  Using the hypothesis~\eqref{eq:decomp_FJ2} and
$\rsect(M_1; F) \simeq \rsect(\R; \roim{q}F)$, we have
$\rsect(M_1; F) \simeq \rsect(M_0; F) \oplus \cor$.  Hence the image of $u'$
by~\eqref{eq:otherFlink3} can be extended to $M_1$ and thus $u'$ also can
be extended to $M_1$. In the same way
$\rsect_c(M_1; F) \simeq \rsect_c(M_0; F) \oplus \cor$ and $w'$ can be extended
to $M_1$.

We choose such extensions $u''$, $w''$ and we set $v = u'' \circ w''$. Then
$v|_{M_0} = (u' \circ w')|_{M_0} = Q(u)|_{M_0}$.  Since $v$ factorizes through
$K_{M_1}$ and $\dot\SSi(K_{M_1})$ is empty, Remark~\ref{rem:umu_autredef}-(c) gives
$v^\mu_p = 0$ for any $p\in \Lambda\cap T^*M_1$.

  \sui(iv) We define $G' \in \Derb(\cor_{M_{1}})$ by the distinguished triangle
\begin{equation}
\label{eq:otherFlink6}
G' \to F|_{M_1} \to[u+v] F|_{M_1} \to G'[1] . 
\end{equation}
Then $\dot\SSi(G') \cap T^*M_1 \subset \Lambda \cap T^*M_1$.  Since
$(u+v)|_{M_0} = (u+ Q(u))|_{M_0} = (R(u))|_{M_0}$ is an isomorphism by~(ii), we
have $G'|_{M_0} \simeq 0$.  In particular $\dot\SSi(G') \cap T^*M_0 = \emptyset$
and we obtain that $\dot\SSi(G') \cap T^*M_1 \subset \Lambda_0$.  Then $G'$ is
locally constant on $M_1 \setminus \ol{\Omega}$, hence zero since it is zero on
$M_0$.   Using
Proposition~\ref{prop:iminvproj} as in the part~(iii) of the proof of
Lemma~\ref{lem:F_degree-1} we deduce that $G'$ is a direct sum of sheaves of the
type $\cor_\Omega$, $\cor_{\ol{\Omega}}$ and $\cor_{\GammaE_0}$ with some shifts
in degree.  It follows easily that either $G' \simeq 0$ or
$\supp(\roim{q}(G')) = [0,1[$.

On the other hand, applying $\roim{q}$ to the triangle~\eqref{eq:otherFlink6},
we obtain the distinguished triangle
$\roim{q}(G')|_{\Jj} \to (\cor_{\{0\}} \oplus B_{\Jj}) \to[w] (\cor_{\{0\}}
\oplus B_{\Jj}) \to[+1]$.  Whatever the morphism $w$ this implies
$\supp(\roim{q}(G')) \cap \Jj = \Jj$, $\{0\}$ or $\emptyset$. This prevents
$\supp(\roim{q}(G')) = [0,1[$ and proves that $G' \simeq 0$.  Hence $u+v$ is an
isomorphism. But we also have $(u+v)^\mu_p = u^\mu_p + v^\mu_p = 0$ for any
$p\in \Lambda_0$. This gives a contradiction and concludes the proof.
\end{proof}

\section{Front with one cusp}

In this section we keep the hypotheses of Section~\ref{sec:trois_points_lies} and
add genericity hypotheses on $\Sigma = \dot\pi_{M}(\Lambda)$.  So $M = \sph$ is
the sphere, $\Lambda \subset \dT^*M$ is a smooth closed conic Lagrangian
submanifold and there exists a cylinder $\Uu = \cer \times \Jj$ in $M$ such that
$\Lambda$ satisfies~\eqref{eq:hypLambda_gentang}.  We assume moreover that
$\Sigma$ has exactly one simple cusp $c$ and that
$\Sigma \cap (M\setminus \{c\})$ is an immersed curve, with transverse double
points.  Since $\Lambda$ is stable by the antipodal map, $\Lambda \cap T^*_cM$
consists of two opposite half lines and
$\Lambda \setminus (\Lambda \cap T^*_cM)$ has two connected components.  The
antipodal map has no fixed point and it follows that it exchanges the two
connected components of $\Lambda \setminus (\Lambda \cap T^*_cM)$.

We consider $F \in \Dersf_{[\Lambda]}(\cor_M)$ satisfying~\eqref{eq:autodualite}
and~\eqref{eq:decomp_FJ2}.  Since $F$ is simple along $\Lambda$, it has a shift
$s(p) \in \demi\Z$ at any $p\in \Lambda$ (see~\eqref{eq:def_shift_simple}).  As
recalled in Example~\ref{ex:shift}, the shift is locally constant outside the
cusps and changes by $1$ when $p$ crosses a cusp.  Hence $s(p)$ takes two
distinct values over $\Lambda \setminus (\Lambda \cap T^*_cM)$, one for each
connected component.  We recall that we have distinguished two points  
$p^\pm$ of $\Lambda$ (see notations~\eqref{eq:notJELambda0}) and we have
$p^{+a} = p^-$.  Hence $F$ has different shifts at the points $p^+$ and $p^-$.
We denote by $\Lambda^\pm$ the connected component of
$\Lambda \setminus(\Lambda \cap T^*_cM)$ containing $p^\pm$.

By Example~\ref{ex:shift} the sheaf $\cor_{\mo]-\infty,0[}$ on $\R$ has shift
$-\frac12$ at $(0;1)$ and the sheaf $\cor_{\mo]-\infty,0]}$ has shift $\frac12$
at $(0;-1)$.  In particular the constant sheaf on a closed or open interval has
the same shifts at both ends. The constant sheaf on a half-closed interval has
different shifts at both ends.

\begin{figure}[ht]
  \centering
\begin{tikzpicture}[scale=1]
  \draw   (-1,1) parabola bend (0,0) (1,1)
  -- plot [smooth] coordinates {(1, 1) (1.2, 2) (0, 3) (-.5, 2.5)
    (0, 2.2) (2.3, 2.5) (3, 3) };

  \draw  plot [smooth] coordinates {(3, 3) (2.7, 2.5) (2.9, 2)
    (3,1) (3,-1) (2.5, -1.8) (1,-3) (-1, -2.3) (-1.5, -1.5) (-1.5,1) (-1,1.5)
    (1, 1.5) (2,1) (2,-1) (0.5, -2) (-2, -2.3) (-2.5,-1) (-2.5,1)
    (-2,2.5) (-1.5,2) (-1,1) };

\node at (3.2,3.2) {$c$};  
  \node at (-5.7,-.2) {$C_0$};
  \draw (-5,0) -- (4,0);
  \node at (-5.7,.9) {$C_{\frac 12}$};
  \draw (-5,0.7) -- (4,0.7);

\draw [->] (0,0) -- +(0,.3); \node  at (.3,.4) {$\textstyle p^+$};
\draw [->] (0,0) -- +(0,-.3); \node  at (.3,-.4) {$\textstyle p^-$};
\draw [->] (-0.85,0.7) -- +(0.3,.2); \node  at (-.5,1.15) {$\textstyle p^+_l$};
\draw [->] (0.85,0.7) -- +(-0.3,.2); \node  at (.5,1.15) {$\textstyle p^+_r$};
\draw [->] (0.5, -2) -- +(0.1,-.3); \node  at (.9, -2.3) {$\textstyle p_0$};
\draw [thick, ->] (-2.5, .7) -- +(0.4,0); \node  at (-2.1, .4) {$\textstyle q_l$};

  \draw[loosely dotted, ultra thick]  plot [smooth] coordinates { (0.5, -2) (-2, -2.3) (-2.5,-1) (-2.5,1)
    (-2,2.5) (-1.5,2) (-1,1) (-.5,.25) (0,0) };
\end{tikzpicture}
  \caption{} \label{fig:dem_un_cusp}
\end{figure}

\begin{theorem}
  \label{thm:ext1FF}
  Let $M$ be the sphere and let   $\Lambda \subset \dT^*M$ be a
  closed conic Lagrangian submanifold such that $\Lambda$
  satisfies~\eqref{eq:hypLambda_gentang}. We assume that
  $\Sigma = \dot\pi_{M}(\Lambda)$ is a curve with only cusps and ordinary double
  points as singularities. We assume that $\Sigma$ has exactly one cusp.  Let
  $F \in \Derb_{[\Lambda]}(\cor_M)$ be a simple sheaf.  We assume that $F|_\Uu$ is
  constructible and satisfies~\eqref{eq:autodualite} and~\eqref{eq:decomp_FJ2}.
  Then $H^1\RHom(F,F) \not= 0$.
\end{theorem}
\begin{proof}
  The notations introduced in~(i) are illustrated in Fig.~\ref{fig:dem_un_cusp},
  where the curve is $\Sigma$, the dotted path is $\im(\pi_M \circ \gamma)$ and
  $\varepsilon = +$.

  \sui(i)   We let $i$ be the inclusion of
  $\cer = C_{1/2}$ in $M$.  We use the notation $M_t$ of~\eqref{eq:defMt}.  By
  Proposition~\ref{prop:otherFlinkedpoint} there exists
  $p_0 = (z_0;\xi_0) \in \Lambda \cap T^*M_0$ such that $p^+$, $p^-$ and $p_0$
  are $F$-linked over $M_{1/2}$.  The point $p_0$ is in $\Lambda^+$ or
  $\Lambda^-$.  We define $\varepsilon = +$ or $\varepsilon = -$ so that
  $p_0 \in \Lambda^\varepsilon$.

  We denote by $z_l = (-1/\sqrt2,1/2)$, $z_r = (1/\sqrt2,1/2)$ the intersections
  of $\GammaE_0$ and $C_{1/2}$.  We also denote by $p_l^\pm$, $p_r^\pm$ the
  points of $\Lambda_0$ (well-defined up to a positive multiple) above $z_l$,
  $z_r$ in the same connected component as $p^\pm$.

  Let $\gamma \cl [0,1] \to \Lambda$ be an embedding such that
  $\gamma(0) = p^\varepsilon$, $\gamma(1) =p_0$,
  $\gamma([0,1]) \subset \Lambda^\varepsilon$ and $\pi_M \circ \gamma$ is an
  immersion (hence $\pi_M \circ \gamma$ describes the portion of the curve
  $\Sigma$ between $(0,0)$ and $z_0$ which does not contain the cusp).  Then the
  image of $\gamma$ meets either the half-line $\Rp\cdot p_l^\varepsilon$ or the
  half-line $\Rp\cdot p_r^\varepsilon$.  We assume it meets
  $\Rp\cdot p_l^\varepsilon$, the other case being similar.  (In
  Fig.~\ref{fig:dem_un_cusp2} the circle represents $\Lambda/\rspos$ and the
  positions of the points correspond to Fig.~\ref{fig:dem_un_cusp}; the points
  $c', c''$ are the preimages of the cusp $c$; we have $\varepsilon=+$, the
  image of the path $\gamma$ is the arc $(p^+p_0)$, which contains $p^+_l$.)

  By   Proposition~\ref{prop:exist_conj_pt} there exists a unique
  $q_l \in \Lambda$ (up to a positive scalar) which is $F$-conjugate to
  $p_l^\varepsilon$ with respect to the embedding $i$ of $C_{1/2}$ in the sense
  of Definition~\ref{def:conjugate_points}.  This means that there exists an
  interval $I_a$ in the decomposition~\eqref{eq:decomp_F_C} of $F|_{C_{1/2}}$
  such that  
  $\dot\SSi(\oim{e}(\cor_{I_a})) = i_d\opb{i_\pi}(\Rp\cdot p_l^\varepsilon
  \sqcup \Rp\cdot q_l)$.  Let $x$ be the right end of $I_a$. Then $e(x) =
  z_l$. If $I_a$ is closed near $x$, then $I_a$ is the same as $I_b$ introduced
  in~\eqref{eq:def_Ib_Ic}.

\sui(ii) If the interval $I_a$ is closed, then the non vanishing of
$H^1\RHom(F,F)$ follows from Propositions~\ref{prop:interv_ferme}
and~\ref{prop:FC0-monodunip}.

\sui(iii) We assume that $I_a$ is open. Then    
$\DD'_{C_{1/2}}(\oim{e}(\cor_{I_a})) \simeq \oim{e}(\cor_{\ol{I_a}})$. Since
$\DD_M(F) \simeq F$, the interval $\ol{I_a}$, maybe translated by a multiple of
$2\pi$, also appears in the decomposition of $F|_{C_{1/2}}$.  Hence we can still
apply Propositions~\ref{prop:interv_ferme} and~\ref{prop:FC0-monodunip}.

\sui(iv) If the interval $I_a$ is half-closed, then $F$ has different shifts at
the points $p_l^\varepsilon$ and $q_l$. Hence $q_l$ belongs to
$\Lambda^{-\varepsilon}$ and the path $\gamma$ does not meet $\Rp\cdot q_l$.
This means that the pairs $\{[p_l^\varepsilon], [q_l]\}$ and
$\{[p^\varepsilon], [p_0]\}$ are intertwined in $\Lambda/\rspos$ (see
Fig.~\ref{fig:dem_un_cusp2}).  The result follows from
  Proposition~\ref{prop:mongammasurj} (applied with the embedding $i$ of
  $C_{1/2}$ in $M$ and the path $\gamma$).
\end{proof}

\begin{figure}[ht]
  \centering
\begin{tikzpicture}
  \draw (0,0) circle [radius=2];

  \foreach \x/\y in {90/c', -90/c'', 0/p^-, 20/p^-_l, -20/p^-_r,
    180/p^+, 160/p^+_r, 200/p^+_l, 240/p_0, 40/q_l}
{\draw (\x:1.9cm) -- (\x:2.1cm) ;
  \node at (\x:2.4cm) {$\y$}; };

\draw[dotted] (180:2cm) .. controls (210:1cm) .. (240:2cm) ;
\draw[dotted] (200:2cm) -- (40:2cm) ;

\end{tikzpicture}
  \caption{}   \label{fig:dem_un_cusp2}
\end{figure}

\section{Proof of the three cusps conjecture}
\label{sec:proof3cusps}

We first make a general remark about the sheaf $K_\Phi$ associated with a
homogeneous Hamiltonian isotopy.   Lemma~\ref{lem:qhisym}
is a general statement and $M$ is any manifold (we come back to $M = \sph$ for
the three cusps conjecture), $\It$ is an open interval containing $0$ and
$\Phi \cl \dT^*M \times \It \to \dT^*M$ is a Hamiltonian isotopy such that
$\Phi_0 = \id$.  We assume that
$\Phi_s(x;\lambda\,\xi) = \lambda \cdot \Phi_s(x;\xi)$ for all $\lambda \in \Rm$
(not only $\lambda \in \rspos$) and all $(x;\xi) \in \dT^*M$, where we set as
usual $\Phi_s(\cdot) = \Phi(\cdot,s)$.  Theorem~\ref{thm:GKS} associates with
$\Phi$ a sheaf $K_\Phi \in\Derlb(\cor_{M^2\times \It})$ and the next lemma says
that it is self-dual in the sense
$K_\Phi \simeq \rhom(K_\Phi, \omega_M \etens \cor_{M\times \It})$.  Let us first
  recall some fact about the dualizing sheaf.
For a manifold $M$ we set $\omega_M = \epb{a_M}(\cor)$, where
$a_M\colon M \to \pt$ is the projection.  We have $\omega_M \simeq or_M[d_M]$,
where $or_M$ is the orientation sheaf and $d_M$ the dimension of $M$. In our
relative case
$\omega_M \etens \cor_{M\times \It} \simeq \epb{p}(\cor_{M\times \It})$, where
$p \colon M^2 \times \It \to M\times \It$ is $(x,y,s)\mapsto (y,s)$.

\begin{lemma}
  \label{lem:qhisym}
  Let $K_\Phi \in\Derlb(\cor_{M^2\times \It})$ be the sheaf associated with $\Phi$
  by Theorem~\ref{thm:GKS}.  Then
  $K_\Phi \simeq \rhom(K_\Phi, \omega_M \etens \cor_{M\times \It})$.  Moreover,
  for any $F \in \Der(\cor_M)$ and any $s\in \It$, we have
  $\DD_M(K_{\Phi,s} \circ F) \simeq K_{\Phi,s} \circ \DD_M(F)$.
\end{lemma}
\begin{proof}
  (i) We use the uniqueness part of Theorem~\ref{thm:GKS}.  We set
  $K' = \rhom(K_\Phi, \omega_M \etens \cor_{M\times \It})$.  Since $\omega_M$ is
  locally constant, Theorem~\ref{thm:SSrhom} gives
  $\dot\SSi(K') = (\dot\SSi(K_\Phi))^a = \Lambda_\Phi^a$, where $\Lambda_\Phi$
  is the graph of $\Phi$, as defined in~\eqref{eq:def_twistedgraph_global}.  The
  hypothesis on $\Phi$ implies $\Lambda_\Phi^a = \Lambda_\Phi$, hence
  $\dot\SSi(K') = \Lambda_\Phi$.

  In particular the inclusion $i \colon M^2\times \{0\} \to M^2 \times \It$ is
  non-characteristic for $\SSi(K')$ and Theorem~\ref{thm:iminv} gives
  $\epb{i}(K') \simeq \opb{i}(K') \otimes \omega_i$, where
  $\omega_i = \epb{i}(\cor_{M^2 \times \It}) \simeq \cor_{M^2}[-1]$.
  By adjunction we have   
  \begin{align*}
  \epb{i}(K')
  &\simeq \rhom(\opb{i}K_\Phi, \epb{i}(\omega_M \etens \cor_{M\times \It})) \\
& \simeq \rhom(\cor_{\Delta_M}, \omega_M \etens \cor_M[-1] ) .
  \end{align*}
  Now $\rhom(\cor_{\Delta_M}, -) \simeq \roim{\delta}\epb{\delta}(-)$, where
  $\delta$ is the inclusion of $\Delta_M$.  We also have
  $\omega_M \etens \cor_M \simeq \epb{p_2}(\cor_M)$, where
  $p_2 \colon M^2 \to M$ is the second projection.  Since
  $p_2 \circ \delta = \id_M$ we obtain
  $\rhom(\cor_{\Delta_M}, \omega_M \etens \cor_M) \simeq \cor_{\Delta_M}$.
  Hence $\opb{i}(K') \simeq \cor_{\Delta_M}$.  Summing up we have
  $\dot\SSi(K') = \Lambda_\Phi$ and
  $K'|_{M^2 \times \{0\}} \simeq \cor_{\Delta_M}$.  The uniqueness property of
  $K_\Phi$ gives $K' \simeq K_\Phi$.

  \sui(ii) The proof of the second statement is the same as~(i). We define two
  sheaves $F'$, $F''$ on $M\times \It$ by $F' = K_\Phi \circ \DD_M(F)$ and
  $F'' = \rhom(K_\Phi \circ F, \omega_M \etens \cor_\It)$.  Then
  $F'|_{M \times \{0\}} \simeq F''|_{M \times \{0\}}$ and $F'$, $F''$ have the
    same microsupport, which is $\Lambda_\Phi \circ^a \SSi(F)$ in the
  notation~\eqref{eq:imageA_par_isotopie}. Then
  Corollary~\ref{cor:isot_equivcat} gives $F' \simeq F''$.  In particular
  $F'|_{M \times \{s\}} \simeq F''|_{M \times \{s\}}$, which is the required
  isomorphism.
\end{proof}

\begin{lemma}
\label{lem:autodual}
Let $G \in \Derb(\cor_\R)$ be a constructible sheaf.  We assume that $G$ has
compact support, that there exists an isomorphism $G \simeq \DD_\R(G)$ and that
$\rsect(\R;G) \simeq \cor$.  Then there exist $t_0 \in \R$ and a decomposition
$G \simeq \cor_{\{t_0\}} \oplus \bigoplus_{a\in A} \cor^{n_a}_{I_a}[d_a]$ where the
$I_a$ are half-closed intervals.
\end{lemma}
\begin{proof}
  By Corollary~\ref{cor:sheaves_dim1} there exists a   decomposition of $G$
  as a finite sum $G \simeq \bigoplus_{a\in A} \cor^{n_a}_{I_a}[d_a]$. Then
  $\rsect(\R;G) \simeq \bigoplus_{a\in A} \rsect(\R;\cor_{I_a})^{n_a}[d_a]$.  Since
  $\rsect(\R;G) \simeq \cor$ all the intervals $I_a$ but one, say $I_b$, have
  cohomology zero, which means that they are half-closed.  For $\alpha<\beta$ and
  $I = \mo]\alpha, \beta\mc[$, $J = [\alpha,\beta]$, we set $I^* = J$, $J^*=I$.  We
  have $\DD_\R(\cor_I[d]) \simeq \cor_{I^*}[1-d]$.  If the remaining interval $I_b$
  is open or is closed with non empty interior, then $I_b^*$ also appears in the
  family $I_a$, $a\in A$, because $\DD_\R(G) \simeq G$.  Then $\cor_{I^*_b}$ also
  contributes to $\rsect(\R;G)$ and it would imply that
  $\bigoplus_{k\in \Z} H^k(\R; G)$ has dimension at least $2$.  Hence $I_b$ is
  reduced to one point, say $t_0$.  Using
  $\DD_\R(\cor_{\{t_0\}}[d]) \simeq \cor_{\{t_0\}}[-d]$ we deduce the lemma.
\end{proof}

Now we can prove the three cusps conjecture.  We let $PT^*\sph = \dT^*\sph/\Rm$
be the projectivized cotangent bundle of $\sph$.  Let $\It$ be an open interval
  containing $0$ and let $\ol{\Phi} \cl PT^*\sph \times \It \to PT^*\sph$
be a contact isotopy of $PT^*\sph$.  It lifts to a homogeneous Hamiltonian
isotopy $\Phi \cl \dT^*\sph \times \It \to \dT^*\sph$ which satisfies
$\Phi_0 = \id$ and $\Phi_s(x;\lambda\,\xi) = \lambda \cdot \Phi_s(x;\xi)$ for
  all $\lambda \in \Rm$ (not only for $\lambda \in \rspos$) and all
  $(x;\xi) \in \dT^*\sph$.  Let $x_0\in \sph$ be a given point.  We set
$\Lambda_0 = \dT^*_{x_0}\sph$, $\Lambda_s = \Phi_s(\Lambda_0)$ and
$\Sigma_s = \pi_\sph(\Lambda_s)$.

\begin{theorem}
\label{thm:threecusps}
Let $s\in \It$ be such that $\Sigma_s$ is a curve with only cusps and double points as
singularities. Then $\Sigma_s$ has at least three cusps.
\end{theorem}
\begin{proof}
  (i) We remark that the number of cusps of $\Sigma_s$ is odd because   $\Lambda_s$ is connected. Hence we have to prove that $\Sigma_s$ does
  not have one cusp.

  \sui(ii) Let $K = K_\Phi \in \Derlb(\cor_{(\sph)^2 \times \It})$ be the sheaf
  associated with $\Phi$ by Theorem~\ref{thm:GKS}. For $s\in \It$ we set
  $K_s = K|_{(\sph)^2 \times \{s\}} \in \Derb(\cor_{(\sph)^2})$.   We also set 
  $F_0 = \cor_{\{x_0\}}$ and $F_s = K_s \circ F_0$.  Then
  $\dot\SSi(F_s) = \Lambda_s$.  We have $\DD_{\sph}(F_0) \simeq F_0$ and by
  Lemma~\ref{lem:qhisym} we deduce that $\DD_{\sph}(F_s) \simeq F_s$.  By
  Corollary~\ref{cor:isot_equivcat} we also have
  $\rsect(\sph; F_s) \simeq \rsect(\sph; F_0) \simeq \cor$ for all $s\in \It$.

  \sui(iii) Now we consider a given $s\in \It$ so that $\Sigma_s$ satisfies the
    hypotheses of the theorem.  We choose a Morse function
  $q \cl \sph \to \R$ with only one minimum $x_-$ and one maximum $x_+$ such
  that $x_-, x_+ \not\in \Sigma_s$.  We can choose $q$ such that any fiber
  $q^{-1}(t)$, $t\in\R$, contains at most one of the following special points:
    tangent point between $q^{-1}(t)$ and
  $\Sigma_s$, double point or cusp (in particular a cusp is not tangent to
  $q^{-1}(t)$).

  We set $G = \roim{q}(F_s)$.  We have isomorphisms
  $\DD_\R(G) \simeq \roim{q}(\DD_{\sph} F_s) \simeq G$ and
  $\rsect(\R;G) \simeq \rsect(\sph; F_s) \simeq \cor$.  By
  Lemma~\ref{lem:autodual} there exist $t_0 \in \R$ and a decomposition
  $G \simeq \cor_{\{t_0\}} \oplus \bigoplus_{a\in A} \cor^{n_a}_{I_a}[d_a]$
where the $I_a$ are half-closed intervals. We set $t_\pm = q(x_\pm)$.  Since
  $F_s$ is constant near $x_-$, we have
  $G \simeq L \otimes (\cor_{[t_-,+\infty[} \oplus
  \cor_{]t_-,+\infty[}[-1])$ near $t_-$, where $L = (F_s)_{x_-}$.  The same
  holds near $t_+$.  Hence $t_0 \not= t_-$ and $t_0 \not= t_+$.

  Since $\cor_{\{t_0\}}$ is a direct summand of $G$, we have
  $T^*_{t_0}\R \subset \SSi(G)$.  By Proposition~\ref{prop:oim} this implies
  $\Lambda_s \cap T^*_{\opb{q}(t_0)}\sph \not= \emptyset$.  By the assumption on
  $q$ it follows that $\opb{q}(t_0)$ is tangent to $\Sigma_s$ at a single point
  and that $\opb{q}(t_0)$ contains no cusp or double point.  By
  Lemma~\ref{lem:bonne_im_dir} the intervals $I_a$ cannot have an end at $t_0$.
  
  Up to a change of coordinates we can assume that $t_0 =0$ and that
  $\Lambda \cap T^*\Uu$, with $\Uu = q^{-1}(\mo]-1,1[)$,
  satisfies~\eqref{eq:hypLambda_gentang} and $F_s$
  satisfies~\eqref{eq:autodualite}
  and~\eqref{eq:decomp_FJ2}.   In particular $F_s|_\Uu$
  is weakly constructible (see Proposition~\ref{prop:mustratif} and
  Remark~\ref{rem:tensor_prod_construct}).  Moreover
  $F = K \circ F_0 \in \Derlb(\cor_{\sph \times \It})$ has finite dimensional
  stalks at any point of
  $(\sph \times \It) \setminus \dot\pi_{\sph \times \It}(\dot\SSi(F))$ by
  Lemma~\ref{lem:stalk_simple_sh}. Hence $F_s|_\Uu$ is constructible.  By
  Theorem~\ref{thm:ext1FF} we deduce $\Hom(F_s, F_s[1]) \not\simeq 0$.  By
  Corollary~\ref{cor:isot_equivcat} we have
  $$
  \Hom(F_s, F_s[1]) \simeq \Hom(F_0, F_0[1]) \simeq 0
  $$
  and this gives a contradiction.
\end{proof}

\section{The four cusps conjecture}
\newcommand{\ttt}{s}
\newcommand{\sss}{\sigma}

In this short section we sketch a proof of Arnol'd's four cusps conjecture along
the same lines as the proof of Theorem~\ref{thm:threecusps}.  It is proved by
Chekanov and Pushkar in~\cite{CP05} and says the following.  Let $M = \R^2$ and
$f\colon M \to \R$, $x\mapsto ||x||$.  Let $\cer = f^{-1}(1)$ be the unit circle
and let $\Lambda^\pm = \{(x; \lambda df)$; $x\in \cer$, $\pm \lambda >0\}$ be
the outer and inner conormal bundles of $\cer$.   
Let $\Phi\cl \dT^*M \times \It \to \dT^*M$ be a homogeneous Hamiltonian isotopy
with $0,1 \in \It$.  We set $\Lambda_\ttt = \Phi_\ttt(\Lambda^-)$ for
$\ttt\in \It$ and we assume that $\Lambda_1 = \Lambda^+$.       We make the same genericity hypotheses as
in~\cite{CP05}:
\begin{itemize}
\item [-] for $\ttt\in \It$ outside a finite set $\{\ttt_1,\dots, \ttt_N\}$,
  $\Sigma_\ttt = \dot\pi_{M}(\Lambda_\ttt)$ is a curve with only cusps and
  ordinary double points as singularities,
\item [-] for $\ttt \in \{\ttt_1,\dots, \ttt_N\}$, $\Sigma_\ttt$ may have a
  birth/death of swallowtail, a triple point, a self-tangency or a double point
  where one of the points is a cusp.
\end{itemize}
Then there exists $\ttt\in \It$ such that $\Sigma_\ttt$ has at least four cusps.

Here is a sketch of proof.

\sui(i) We let $F_0 = \cor_{\{f\leq 1\}}$ be the constant sheaf on the closed unit
ball.  Hence $\dot\SSi(F_0) = \Lambda_0$.  We let
$\Lambda' \subset \dT^*(M\times \It)$ be the image of $\Lambda_0$ by the whole
isotopy (see~\eqref{eq:imageA_par_isotopie}).  It is a conic Lagrangian submanifold
which is non-characteristic for all inclusions
$i_\ttt \colon M \times \{\ttt\} \to M\times \It$ (that is,
$\Lambda' \cap (T^*_MM \times T^*_\ttt\It) = \emptyset$) and satisfies
$i_\ttt^\sharp(\Lambda') = \Lambda_\ttt$.  By Corollary~\ref{cor:isot_equivcat} there
exists a unique $F \in \Der(\cor_{M\times \It})$ such that $\dot\SSi(F) = \Lambda'$
and $\opb{i_0}F \simeq F_0$.  We set $F_\ttt = \opb{i_\ttt}F$.  We know that $F$ is
simple along $\Lambda'$, $F_\ttt$ has compact support,
$\dot\SSi(F_\ttt) = \Lambda_\ttt$ and $\rsect(M; F_\ttt)$ is independent of $\ttt$.

\sui(ii) We compute $F_1$. We have $\dot\SSi(F_1) = \Lambda^+$, hence $F_1$ is
locally constant on $M \setminus \cer$.  Since $F_1$ has compact support, it
must be supported on the unit ball.  The microsupport condition then implies
that it is of the form $F_1 \simeq E_{\{f<1\}}$ for some $E\in \Der(\cor)$.
Since $F_1$ is simple, we have $E = \cor[d]$ for some $d\in \Z$.  Finally the
condition $\rsect(M;F_1) \simeq \rsect(M;F_0) \simeq \cor$ gives $d=2$ and we
have $F_1 \simeq \cor_{\{f<1\}}[2]$.

\sui(iii)   We choose a projection $q\colon M \to \R$
which is generic with respect to the family of curves $\Sigma_\ttt$ (in particular,
for all but finitely many $\ttt$, $\Sigma_\ttt$ has no more than one tangent point to
a fiber of $q$ and, for the remaining values of $\ttt$, $\Sigma_\ttt$ has no more
than two tangent points to a fiber of $q$).  We set
$G = \roim{(q\times \id_\It)}(F) \in \Der(\cor_{\R \times \It})$,
$G_\ttt= G|_{\R \times \{\ttt\}} \simeq \roim{q}(F_\ttt) \in \Der(\cor_{\R})$.  We
have $\rsect(\R; G_\ttt) \simeq \rsect(M; F_\ttt) \simeq \cor$ for all $\ttt\in \It$.
We can decompose $G_\ttt$ as a sum of constant sheaves on intervals, with degree
shifts, using Corollary~\ref{cor:sheaves_dim1}.  Since
$\rsect(\R; G_\ttt) \simeq \cor$ all these intervals are half-closed but one, say
$I(\ttt)$.  Then either $I(\ttt)$ is closed and $G_\ttt$ has $\cor_{I(\ttt)}$ as a
direct summand (this is the case for $\ttt$ close to $0$), or $I(\ttt)$ is open and
$G_\ttt$ has $\cor_{I(\ttt)}[1]$ as a direct summand (this is the case for $\ttt$
close to $1$).

\sui(iv) We prove that there exists $\ttt_0 \in \It$ such that $I(\ttt_0)$ is
reduced to a point. For this we argue as in Part~\ref{part:graphsel}. We set
$G' = H^0G$.  We have a natural morphism
$H^0(\R \times \It; G) \to H^0(\R \times \It; \tau_{\geq 0}G) \simeq H^0(\R
\times \It; G')$. Hence the element $\sss \in H^0(\R \times \It; G) \simeq \cor$
corresponding to $1$ induces a section $\sss'$ of $G'$.  We consider its
support, say $Z = \supp(\sss')$.  Then $Z$ is a closed set and coincides with the
set of $(z,\ttt) \in \R \times \It$ such that $\sss'_{(z,\ttt)} \not= 0$. We set
$G'_\ttt = G'|_{\R \times \{\ttt\}} \simeq H^0G_\ttt$ and we let
$\sss'(\ttt) \in H^0(\R; G'_\ttt)$ be the section of $G'_\ttt$ induced by
$\sss'$.  We have $\supp(\sss'(\ttt)) = Z \cap (\R \times \{\ttt\})$.  By
Corollary~\ref{cor:sheaves_dim1} $G'_\ttt$ is the direct sum of the summands of
$G_\ttt$ which are in degree $0$.  A sheaf $\cor_J$, $J$ some interval, has a
global section if and only if $J$ is closed; in this case the support of the
section is $J$.  Hence we deduce from~(iii) that, if $\sss'(\ttt) \not= 0$, then
$I(\ttt)$ is closed and   $\supp(\sss'(\ttt)) = I(\ttt)$.  We define
$\It' \subset \It$ by $\It' = \{\ttt$; $\sss'(\ttt) \not= 0\}$.  Then $\It'$ is
closed and $Z = \bigsqcup_{\ttt\in \It'} (I(\ttt) \times \{\ttt\})$.  We define
$\ttt_0$ as the upper bound of $\It' \cap [0,1]$. We remark that $\ttt_0 < 1$
because $I(\ttt)$ is open for $\ttt$ close to $1$.

  On the other hand $\dot\SSi(G)$ is bounded
by $\Xi = (q\times \id_\It)_\pi (q\times \id_\It)_d^{-1}(\Lambda')$.  We set
$C = \dot\pi_{\R\times \It}(\Xi)$ (see Fig.~\ref{fig:four_cuspsCetZ}). If $q$ is
generic with respect to $\Lambda'$, then   $C$ is a (singular)
curve of $\R\times \It$, with a finite set, say $C_{sing}$, of singular points (cusps
or multiple points).  We set $C_{reg} = C \setminus C_{sing}$, which is a finite union
of disjoint smooth arcs.  If $C'$ is a connected component of $C_{reg}$, then $\Xi$
must contain one of the two connected components of $T^*_{C'}(\R\times \It)$.  Using
the fact that $\Lambda'$ is non-characteristic for all inclusions $i_\ttt$, we can
see that $\Xi \cap (T^*_\R\R \times T^*\It) = \emptyset$.   It follows that $C'$ is not tangent to any line $\R \times \{\ttt\}$,
$\ttt\in\It$.

Now $G$, hence $H^0G$, is weakly constructible with respect to the stratification
induced by $C$ (see Section~\ref{sec:constructibility}).  The strata are the
components of $(\R\times \It) \setminus C$, the components of $C_{reg}$ and the
points of $C_{sing}$.  The set $Z$, which is the support of a section of $H^0G$, is a
union of strata.  Hence, if $I(\ttt_0) \times \{\ttt_0\}$ meets an open stratum, then
$Z \cap (\R \times \{\ttt_0+\varepsilon\})$ is not empty for small $\varepsilon>0$,
which contradicts the definition of $\ttt_0$.  Hence $I(\ttt_0) \times \{\ttt_0\}$ is
contained in $C$.  If $I(\ttt_0)$ is not reduced to a point, then
$C_{reg} \cap (I(\ttt_0) \times \{\ttt_0\})$ is non trivial and $C_{reg}$ is tangent
to $\R \times \{\ttt_0\}$, which is impossible as already remarked.  Hence
$I(\ttt_0)$ is a point.

\begin{figure}[ht]
  \centering
\begin{tikzpicture}[yscale = 3]
\fill [fill=gray!20] plot [smooth, dotted] coordinates {(1, 0) (.4,.3) (0.1, .7) (-0.1,.84)}
--  plot [smooth, dotted] coordinates {(-0.1,.84) (-0.1, .5)} 
-- plot [smooth, dotted] coordinates {(-0.1, .5)(-.3,.3) (-1, 0)} ;

\draw plot [smooth] coordinates {(1, 0) (.4,.3) (0, .8) (-1,1) };

\draw plot [smooth] coordinates {(-1, 0) (-.3,.3) (.2, .8)};
\draw plot [smooth] coordinates {(.2, .8) (.25, .7) (.5, .4) (.25, .2) (.2,.15) };
\draw plot [smooth] coordinates {(.2, .15) (.2,.1) (-.1, .5) (-.1, .7) (0,.9) (1,1)};

\draw plot [smooth] coordinates {((-1,.25) (-1,.3) (-1.3, .5) (-1,.7) (-1,.75)};
\draw plot [smooth] coordinates {((-1,.25) (-.9,.3) (-.7, .5) (-.9,.7) (-1,.75)};

\draw [->] (-2.5,-.3) -- (3,-.3); \node at (3.2,-.2) {$\R$};
\draw [->] (-3,-.2) -- (-3, 1.2); \node at (-3.4, 1.2) {$\It$};
\draw [dotted] (-3,0) -- (2,0);  \node at (-3.4, 0) {$0$};
\draw [dotted] (-3,1) -- (2,1); \node at (-3.4, 1) {$1$};

\draw [dotted]  let \p1 = (-0.1,.84) in (-3,\y1) -- (\p1);
\node at (-3.4, .84) {$\ttt_0$};
\end{tikzpicture}
\caption{The curve $C$ and the support $Z$ of $\sss'$} \label{fig:four_cuspsCetZ}
\end{figure}

\sui(v)     The situation at time $\ttt_0$ is similar to the setting
of~\S\ref{sec:threecusps_nothyp}. In particular, assuming $I(\ttt_0) = \{0\}$, we
have $(\roim{q}(F_{\ttt_0}))|_{\Jj} \simeq \cor_{\{0\}} \oplus B_{\Jj}$, for some
interval $\Jj$ around $0$ and some $B \in \Derb(\cor)$ (like~\eqref{eq:decomp_FJ2}).
We can see also that $\Lambda_{\ttt_0}$ meets $T^*_{q^{-1}(0)}M$ along two half
lines, generated by, say, $p^\pm$.

These two points belong to the conormal bundles of two branches of the intersection
of $\Sigma_{\ttt_0}$ with some neighborhood of $q^{-1}(0)$, say $\Gamma^\pm$, tangent
to $q^{-1}(0)$ (unlike the case of the three cusps conjecture $\Lambda_{\ttt_0}$ is
not equal to $\Lambda_{\ttt_0}^a$ and the map $\Lambda_{\ttt_0}^a/\rspos \to M$ is
generically one to one).    The fact
that $M$ is $\R^2$ and not $\sph$ makes the situation easier.  In particular an
argument as in the step~(ii) in the proof of Proposition~\ref{prop:interv_ferme}
gives the following: let $p^+_0$, $p^+_1$ be the two points in
$\Lambda_{\ttt_0} \cap T^*_{\Gamma^+}M \cap \pi^{-1}q^{-1}(\pm \varepsilon)$, for
$\varepsilon>0$ small enough (we write $\pm \varepsilon$ because we do not know which
side of $q^{-1}(0)$ the branch $\Gamma^+$ is situated).  Let $q^+_0$, $q^+_1$ be the
conjugate points of $p^+_0$, $p^+_1$ with respect to $q^{-1}(\pm \varepsilon)$.  Then
$F_{\ttt_0}$ has different shifts at $p^+_i$ and $q^+_i$.  (See
Fig.~\ref{fig:four_cusps0} where we give two examples of possible $\Sigma_{\ttt_0}$
and some notations -- we do not claim that there actually exist sheaves corresponding
to these pictures.)

\begin{figure}[ht]
  \centering 
  \begin{tikzpicture}[scale=.9]
    \draw   (-1,1) parabola bend (0,0) (1,1)
  -- plot [smooth] coordinates {(1, 1) (1.3, 1.2) (1.5, -1.2)
    (2.3, 2) (3, 3) };
  \draw  plot [smooth] coordinates {(3, 3) (2.7, 2.5) (2.9, 2)
    (3,1) (3,-1) (2.5, -2) (1,-3) (-1.5, -1.5) (-1.5,1) (-1,1.5)
    (1, 1.5) (2,1) (2,-1) (0.5, -2) (-3,-2.4) (-4,-1.6) (-4,-1)}
  -- (-4,-1) parabola bend (-3,0) (-2,-1)
  plot [smooth] coordinates {(-2,-1) (-2, -2) (-4,-2) (-4.7,-2.5)};
  \draw  plot [smooth] coordinates {(-4.7,-2.5) (-4.3,-1.5) (-4,1) (-2,2.5)
    (-1,1) };
  \draw[loosely dotted, ultra thick] (-1,1) parabola bend (0,0) (1,1);
   \draw[loosely dotted, ultra thick] (-4,-1) parabola bend (-3,0) (-2,-1);
  
  \node at (3.3,3) {$c_2$};
  \node at (-5,-2.5)(3.2,3.2) {$c_1$};  
  \node at (-7,0) {$q^{-1}(0)$};
  \draw (-6,0) -- (4,0);
  \node at (-7,.7) {$q^{-1}(\varepsilon)$};
  \draw (-6,0.7) -- (4,0.7);
    \node at (-7,-.7) {$q^{-1}(-\varepsilon)$};
  \draw (-6,-0.7) -- (4,-0.7);

\draw [->] (0,0) -- +(0,.3); \node  at (.3,.4) {$\textstyle p^+$};
\draw [->] (-0.85,0.7) -- +(0.3,.2); \node  at (-.5,1.15) {$\textstyle p^+_0$};
\draw [->] (-1.6,0.7) -- +(0.3,-0.1); \node  at (-1.2,.4) {$\textstyle q^+_0$};
\draw [->] (0.85,0.7) -- +(-0.3,.2); \node  at (.4,1.15) {$\textstyle p^+_1$};
\draw [->] (-3,0) -- +(0,-.3); \node  at (-2.7,-.4) {$\textstyle p^-$};
\draw [->] (-3.85,-0.7) -- +(0.3,-.2); \node  at (-3.4,-1.1) {$\textstyle p^-_1$};
\draw [->] (-2.15,-0.7) -- +(-0.3,-.2); \node  at (-2.4,-1.1) {$\textstyle p^-_0$};
\end{tikzpicture}

\begin{tikzpicture}[scale=.9]
  \draw   (-1,1) parabola bend (0,0) (1,1)
  -- plot [smooth] coordinates {(1, 1) (2, 2.2) (3.3, 2.2)
    (3.3, -1.5) (-4.5,-1.5) (-5, 1.5) (0,2) (3, 3) };
  \draw  plot [smooth] coordinates {(3, 3) (2.7, 2.5) (2.9, 2)
    (3,1) (3,-1) (2.5, -2) (1,-3) (-1.5, -1.5) (-1.5,1) (-1,1.5)
    (1, 1.5) (2,1) (2,-1) (0.5, -2) (-4, -2) (-5,1) (-4.4,1.4) (-4,1)}
  -- (-4,1) parabola bend (-3,0) (-2,1)
  plot [smooth] coordinates {(-2,1) (-1.6, 2) (-1.5, 2.5)};
  \draw  plot [smooth] coordinates {(-1.5, 2.5) (-1.4,2) (-1,1) };
  \draw[loosely dotted, ultra thick] (-1,1) parabola bend (0,0) (1,1);
   \draw[loosely dotted, ultra thick] (-4,1) parabola bend (-3,0) (-2,1);
  
  \node at (3.3,3) {$c_1$};
  \node at (-1.8,2.5) {$c_2$};  
  \node at (-7,0) {$q^{-1}(0)$};
  \draw (-6,0) -- (4,0);
  \node at (-7,.7) {$q^{-1}(\varepsilon)$};
  \draw (-6,0.7) -- (4,0.7);
    \node at (-7,-.7) {$q^{-1}(-\varepsilon)$};
  \draw (-6,-0.7) -- (4,-0.7);

\draw [->] (-0,0) -- +(0,.3); \node  at (.3,.4) {$\textstyle p^+$};
\draw [->] (-0.85,0.7) -- +(0.3,.2); \node  at (-.5,1.15) {$\textstyle p^+_1$};
\draw [->] (0.85,0.7) -- +(-0.3,.2); \node  at (.4,1.15) {$\textstyle p^+_0$};
\draw [->] (3,0.7) -- +(-0.3,.1); \node  at (2.6,1.15) {$\textstyle q^+_0$};
\draw [->] (-3,0) -- +(0,-.3); \node  at (-2.7,-.4) {$\textstyle p^-$};
\draw [->] (-3.85,0.7) -- +(-0.3,-.2); \node  at (-4,.3) {$\textstyle p^-_0$};
\draw [->] (-5.4,0.7) -- +(-0.3,.1); \node  at (-5.6,1.2) {$\textstyle q^-_0$};
\draw [->] (-2.15,0.7) -- +(0.3,-.2); \node  at (-1.9,.3) {$\textstyle p^-_1$};
\end{tikzpicture}
  \caption{two examples of $\Sigma_{\ttt_0}$ -- the branches $\Gamma^\pm$ are
  the dotted parts} \label{fig:four_cusps0}
\end{figure}

We have similar pairs $\{p^-_i, q^-_i\}$. Using these four pairs of conjugate points
we can now prove directly an analog of Theorem~\ref{thm:ext1FF} (there is no need for
a third linked point as in Proposition~\ref{prop:otherFlinkedpoint}) as follows.  Let
us assume that $\Sigma_{\ttt_0}$ has only two cusps, say $c_1$, $c_2$.  Then
$\Lambda_{\ttt_0} \setminus (\Lambda_{\ttt_0} \cap (T^*_{c_1}M \cup T^*_{c_2}M))$
consists of two connected components, corresponding to two different shifts, say
$\Lambda^0$, $\Lambda^1$.  We assume that $p^+ \in \Lambda^0$.  By
Proposition~\ref{prop:mulink-projdim1} (with $t_0 = t_1 =0$ and $J = \{0\}$) we know
that $p^-$ and $p^+$ are $F$-linked over any open subset containing $q^{-1}(0)$.  We
distinguish three cases, (a), (b-i), (b-ii) (see Fig.~\ref{fig:fourcusp} -- the first
picture in Fig.~\ref{fig:four_cusps0} is compatible with~(b-i) and the second one
with~(b-ii)):

\noindent (a) If $p^- \in \Lambda^0$, the arc of $\Lambda^0$ from $p^+$ to $p^-$
contains $p_0^+$ or $p_1^+$.  We choose the notations so that it contains $p_0^+$
(see Fig.~\ref{fig:fourcusp}~(a)).  Then $q_0^+ \in \Lambda^1$ and we see that the
pairs $\{[p^+], [p^-]\}$ and $\{[p_0^+], [q_0^+]\}$ in $\Lambda_{\ttt_0}/\rspos$ are
intertwined. By Proposition~\ref{prop:mongammasurj} we deduce
$H^1\RHom(F_{\ttt_0},F_{\ttt_0}) \not= 0$, contradicting $H^1\RHom(F_0,F_0) = 0$.

\sui (b)   If $p^- \in \Lambda^1$, we choose the
notations so that $p_0^+$, $p_0^-$ and $c_1$ belong to the same arc of
$\Lambda_{\ttt_0}/\rspos$ joining $p^+$ and $p^-$.  We remark that the open arc
$]p_0^-,p_1^-[$ is contained in $q^{-1}(\mo]-\varepsilon, \varepsilon[)$.  Since
$q(q_0^+) = \pm \varepsilon$, $q_0^+$ belongs to the arc $]c_1, p_0^-]$ or to the arc
$]c_2,p_1^-]$.  It could happen that $q_0^+ = p_0^-$, but this implies that
$\Gamma^+$ and $\Gamma^-$ are on the same side of $q^{-1}(0)$ and that
$\pi_M(p_1^+) < \pi_M(p_0^+) < \pi_M(p_0^-) < \pi_M(p_1^-)$ for the order on the line
$q^{-1}(\varepsilon)$; hence we cannot have both $q_0^+ = p_0^-$ and $q_1^+ = p_1^-$.
Up to switching $0$ and $1$, we can assume $q_0^+ \not= p_0^-$.  In other words
$p_0^+$ and $p_0^-$ are not conjugate, and we also have $q_0^- \not=
p_0^+$. Similarly as $q_0^+$, $q_0^-$ belongs to the arc $]c_1, p_0^+[$ or to the arc
$]c_2,p_1^+]$.

\noindent(b-i) If $q_0^+$ belongs to $]c_2,p_1^-]$ (see
Fig.~\ref{fig:fourcusp}~(b-i)), then the pairs $\{[p^+], [p^-]\}$ and
$\{[p_0^+], [q_0^+]\}$ are intertwined.  If $q_0^-$ belongs to $]c_2,p_1^+]$, then
the pairs $\{[p^+], [p^-]\}$ and $\{[p_0^-], [q_0^-]\}$ are intertwined.  In both
cases we conclude as in case~(a).

\noindent(b-ii) If none of these two cases occurs, then $q_0^+$ belongs to
$]c_1, p_0^-[$, $q_0^-$ belongs to $]c_1, p_0^+[$ and the pairs
$\{[p_0^+], [q_0^+]\}$ and $\{[p_0^-], [q_0^-]\}$ are intertwined (see
Fig.~\ref{fig:fourcusp}~(b-ii)).  Since $F$-conjugate points are $F$-linked (see
Proposition~\ref{prop:mulink-restrdim1}), we can again apply
Proposition~\ref{prop:mongammasurj}, after moving the pairs so that they are not on
the same line (choose $p_0^-$ on a line $q^{-1}(\varepsilon')$ for $\varepsilon'$
slightly different from $\varepsilon$).

\begin{figure}[ht]
  \centering
  \begin{tikzpicture}
  \draw (0,0) circle [radius=1.1];
  
  \foreach \x/\y in {90/c_1, -90/c_2, 220/p^-, 200/p^-_0, 240/p^-_1,
    140/p^+, 120/p^+_1, 160/p^+_0, 30/q^+_0 }
{\draw (\x:1cm) -- (\x:1.2cm) ;
  \node at (\x:1.5cm) {$\y$}; };

\draw[dotted] (160:1.1cm) -- (30:1.1cm) ;
\draw[dotted] (140:1.1cm) .. controls (180:.5cm) .. (220:1.1cm) ;

\node at (0,-2.2) {(a)};
\end{tikzpicture}
\quad
\begin{tikzpicture}
  \draw (0,0) circle [radius=1.1];

  \foreach \x/\y in {90/c_1, -90/c_2, 0/p^-, 20/p^-_0, -20/p^-_1,
    180/p^+, 160/p^+_0, 200/p^+_1, -45/q^+_0 }
{\draw (\x:1cm) -- (\x:1.2cm) ;
  \node at (\x:1.5cm) {$\y$}; };

\draw[dotted] (180:1.1cm) -- (0:1.1cm) ;
\draw[dotted] (160:1.1cm) -- (-45:1.1cm) ;

\node at (0,-2.2) {(b-i)};
\end{tikzpicture}
\quad
\begin{tikzpicture}
  \draw (0,0) circle [radius=1.1];

  \foreach \x/\y in {90/c_1, -90/c_2, 0/p^-, 20/p^-_0, -20/p^-_1,
    180/p^+, 160/p^+_0, 200/p^+_1, 45/q^+_0, 135/q^-_0 }
{\draw (\x:1cm) -- (\x:1.2cm) ;
 \node at (\x:1.5cm) {$\y$}; };

\draw[dotted] (135:1.1cm) -- (20:1.1cm) ;
\draw[dotted] (160:1.1cm) -- (45:1.1cm) ;

\node at (0,-2.2) {(b-ii)};
\end{tikzpicture}
  \caption{} \label{fig:fourcusp}
\end{figure}

\part{Triangulated orbit categories for sheaves}
\label{part:orbcat}

In the study of Lagrangian exact submanifolds of cotangent bundles in
Part~\ref{chap:ex_Lag} we will use triangulated orbit categories.  For a   triangulated category $\catt$, the triangulated envelope of the
orbit category of $\catt$ is a triangulated category $\catt'$ with a functor
$\iota \colon \catt \to \catt'$ such that $\iota(F) \simeq \iota(F)[1]$ for any
$F\in \catt$ and
$\Hom_{\catt'}(\iota(F), \iota(G)) \simeq \bigoplus_{i\in Z} \Hom_{\catt}(F,G[i])$.
Such categories are constructed by Keller in~\cite{Ke05}. We specify his construction
in the case of categories of sheaves and check that we can define a microsupport in
this framework.

\section{Definition of triangulated orbit categories}
\label{sec:def_orb_cat}

We will use a very special case of the triangulated hull of an orbit category as
described by Keller in~\cite{Ke05}.  More precisely
Definition~\ref{def:categorie-orb} below is inspired by \S7 of~\cite{Ke05} that we
apply to the simple case where we quotient $\Derb(\cor_M)$ by the autoequivalence
$F \mapsto F[1]$ (in~\cite{Ke05} much more general equivalences are
considered). However we apply this construction for sheaves instead of modules over
an algebra.  We use the name triangulated orbit category for the category
$\Orb(\cor_M)$ introduced in Definition~\ref{def:categorie-orb} by analogy with the
categories introduced by Keller.  However we only show that we have a functor
$\iota_M \cl \Derb(\cor_M) \to \Orb(\cor_M)$ such that
$\iota_M(F) \simeq \iota_M(F)[1]$ and which satisfies
Corollary~\ref{cor:morph-QRF-QRG} below (the proof of this result is also inspired
by~\cite{Ke05}).  In this section we assume $\cor = \Z/2\Z$.

\subsection*{Quick reminder on localization}

  We recall quickly
some facts about quotients of triangulated categories.  We follow the exposition
of~\cite[\S1.6]{KS90} or~\cite[\S10.2]{KS06}.
Let $\catd$ be a triangulated category. A multiplicative system
$\shs$ in  $\catd$ is a family of morphisms such that
\begin{itemize}
\item [(S1)] any isomorphism is in $\shs$,
\item [(S2)] $\shs$ is stable by composition of morphisms,
\item [(S3)] for given $f\colon X \to Y$ and $s\colon X \to X'$, with $s\in\shs$,
  there exist $t\colon Y \to Y'$ and $g\colon X' \to Y'$, with $t\in \shs$, such that
  $g \circ s = t \circ f$, and the same holds with all arrows reversed (visualized by
\begin{tikzcd}[row sep=.5cm]
X \ar[r, "f"] \ar[d, "s"] & Y \ar[d, "t"] \\
X' \ar[r, "g"] & Y' 
\end{tikzcd}
and
\begin{tikzcd}[row sep=.5cm]
X \ar[r, leftarrow, "f"] \ar[d, leftarrow, "s"] & Y \ar[d, leftarrow, "t"] \\
X' \ar[r, leftarrow, "g"] & Y' 
\end{tikzcd}) 
\item [(S4)] For $f,g \colon X \rightrightarrows Y$, there exists $s \colon W \to X$
  in $\shs$ such that $f \circ s = g \circ s$ if and only if there exists
  $t \colon Y \to Z$ in $\shs$ such that $t \circ f = t\circ g$.
\end{itemize}
The localization of $\catd$ by $\shs$ is then the category $\catd_\shs$ with the same
objects as $\catd$ and with morphisms $\Hom_{\catd_\shs}(X,Y) = \{(X',s,f)$;
$s \colon X' \to X$, $f\colon X' \to Y$ with $s\in \shs$, $f$ any morphism$\}/\sim$,
where the equivalence relation $\sim$ is generated by: $(X',s,f) \sim (X'',t,g)$ if there exists a
commutative diagram
\begin{tikzcd}[row sep=.2cm, baseline=(Y.base)]
& X' \ar[dl, "s"'] \ar[dr, "f"] \ar[dd] \\
X && |[alias=Y]| Y .\\
& X'' \ar[ul, "t"] \ar[ur, "g"'] 
\end{tikzcd}
The composition is defined using the property (S3).  The
localization comes with a functor $Q_\shs \colon \catd \to \catd_\shs$ such that, for
any $s\in \shs$, $Q_\shs(s)$ is an isomorphism. Moreover the pair
$(\catd_\shs, Q_\shs)$ is universal with respect to this property.  (Up to now we did
not use the triangulated structure.)  We now define a distinguished triangle in
$\catd_\shs$ as a triangle $A \to B \to C \to A[1]$ which is isomorphic in
$\catd_\shs$ to the image by $Q_\shs$ of a distinguished triangle of $\catd$.  This
turns $\catd_\shs$ into a triangulated category.

In the framework of triangulated categories, a convenient way to obtain a
multiplicative system is to start with a full triangulated subcategory $\shn$ of
$\catd$ which is saturated in the following sense: for any isomorphism $X \simeq Y$
in $\catd$, if $X\in \shn$, then $Y\in \shn$.  We then define $\shs_\shn$ as the
family of morphisms $s$ in $\catd$ such that the cone of $s$ belongs to $\shn$.  We
can check that $\shs_\shn$ is a multiplicative system and consider the localization
$\catd_{\shs_\shn}$. Then, for any $N \in \shn$, $Q_{\shs_\shn}(N) \simeq 0$;
moreover $(\catd_{\shs_\shn}, Q_{\shs_\shn})$ is universal with respect to this
property.  We call $\catd_{\shs_\shn}$ the quotient of $\catd$ by $\shn$ and denote
it by $\catd/\shn$.

\smallskip

The localization is used to define the derived category.  If $\catc$ is an abelian
category, we let $\CC(\catc)$ be the category of complexes of objects of $\catc$ and
$\CC^{\mathrm b}(\catc)$ its full subcategory of bounded complexes.  We define
$\CK(\catc)$ as the category with the same objects as $\CC(\catc)$ and morphisms
$\Hom_{\CK(\catc)}(X,Y) = \Hom_{\CC(\catc)}(X,Y)/\sim$ where $\sim$ is the homotopy
equivalence of morphisms. We define $\CK^{\mathrm b}(\catc)$ as the full subcategory
of $\CK(\catc)$ of bounded complexes.  These categories $\CK(\catc)$,
$\CK^{\mathrm b}(\catc)$ are triangulated.  We recall that $u\colon X \to Y$ is a
quasi-isomorphism if $H^n(u) \colon H^n(X) \to H^n(Y)$ is an isomorphism for all
$n\in \Z$.  Then $\Der(\catc) = (\CK(\catc))_{qis}$ and
$\Derb(\catc) = (\CK^{\mathrm b}(\catc))_{qis}$, where $qis$ denotes the family of
quasi-isomorphisms.  We have a natural functor $\Derb(\catc) \to \Der(\catc)$ which
identifies $\Derb(\catc)$ as a full subcategory of $\Der(\catc)$.  It is classical
that $\Derb(\catc)$ is also equivalent to the full subcategory of $\Der(\catc)$:
\begin{equation}
  \label{eq:defderb}
  \{X \in \Der(\catc); \; H^n(X) \simeq 0 \text{ for $n\ll0$ and $n\gg0$}\} .
\end{equation}
We use similar constructions to define the derived categories of complexes bounded
from below, $\Der^+(\catc)$, or above, $\Der^-(\catc)$.

If $\catc'$ is another abelian category, any additive functor
$F \colon \catc \to \catc'$ induces a functor,
$F_1 \colon \CK(\catc) \to \CK(\catc')$, compatible with the triangulated structures.

  If $F$ is exact, then $F_1$ sends
quasi-isomorphisms to quasi-iso\-mor\-phisms and $Q \circ F_1$ sends
quasi-isomorphisms to isomorphisms, where  
$Q \colon \CK(\catc') \to \Der(\catc')$ is the localization functor. By the universal
property, $Q \circ F_1$ factorizes through a functor
\begin{equation}
  \label{eq:funcder0}
F \colon  \Der(\catc) \to \Der(\catc')
\end{equation}
(compatible with the triangulated structures).  In the same way $F$ induces functors,
all denoted $F$, $\Der^*(\catc) \to \Der^*(\catc')$, where $*$ stands for
$\mathrm{b}$, $+$ or $-$.

In general, if $F$ is only additive, $Q \circ F_1$ does not factorizes through
$\Der(\catc)$ but, under some assumptions (for example, $\catc$ has enough
injectives), we can define a ``best possible approximation'' to a factorization,
which is called the right derived functor of $F$,
\begin{equation}
  \label{eq:funcder1}
\RR F \colon \Der^+(\catc) \to \Der^+(\catc').
\end{equation}
If $F$ is left exact, we have $F \simeq H^0\circ \RR F \circ \iota$, where $\iota$ is
the embedding of $\catc$ as the full subcategory of $\Der^+(\catc)$ of complexes
concentrated in degree $0$ (if $F$ is not left exact, $F$ and $\RR F$ are not related
in a useful way).  If $F$ is exact, then $F \simeq \RR F$.

If $\catc$ has enough projectives, we can also derive functors on the left. However
we are interested in categories of sheaves, which do not have enough projectives.
For a specific functor $F$, we can still derive $F$ on the left, if $\catc$ has
enough $F$-projective objects.  We are interested in the case of the tensor product
of sheaves; the $\otimes$-projective sheaves are the flat sheaves and there are enough
flat sheaves.  So we can define $\ltens$ on the bounded from above derived categories
of sheaves:
\begin{equation}
  \label{eq:funcder2}
\ltens_{\cor_M} \colon \Der^-(\cor_M) \times \Der^-(\cor_M)  \to \Der^-(\cor_M) ,
\end{equation}
where $M$ is any manifold and $\cor$ any ring.

\subsection*{Definition of the orbit category}

We recall that $\cor = \Z/2\Z$ and we set $\Cor = \cor[X]/\langle X^2 \rangle$.  We
let $\varepsilon$ be the image of $X$ in $\Cor$. Hence $\Cor = \cor[\varepsilon]$
with $\varepsilon^2=0$.  Let $M$ be a manifold.  We denote as usual by $\Mod(\cor_M)$
(resp. $\Mod(\Cor_M)$) the category of sheaves of $\cor$-modules
(resp. $\Cor$-modules) on $M$. We set $\Derb(\cor_M) = \Derb(\Mod(\cor_M))$ and
$\Derb(\Cor_M) = \Derb(\Mod(\Cor_M))$, in the sense of~\eqref{eq:defderb}.

The natural ring morphisms $\cor \to \Cor$ and $\Cor\to \cor$ induce two pairs of
adjoint functors $(e_M,r_M)$ and $(E_M,R_M)$, where $e_M,E_M$ are scalar extensions
and $r_M,R_M$ restrictions of scalars (or rather, their derived versions, see the
discussion around~\eqref{eq:funcder0}-\eqref{eq:funcder2} for a quick reminder):
\begin{alignat*}{3}
&\Derb(\cor_M) \overset{e_M}{\underset{r_M}{\rightleftarrows}} \Derb(\Cor_M),
&\qquad e_M(F) &= \Cor_M \tens_{\cor_M} F,
&\quad r_M(G) &= G, \\
&\Derb(\Cor_M) \xfrom[R_M]{} \Derb(\cor_M),
&& & R_M(F) &= F , \\
&\Der^-(\Cor_M) \overset{E_M}{\underset{R_M}{\rightleftarrows}} \Der^-(\cor_M),
& E_M(G) &= \cor_M \ltens_{\Cor_M} G,
& R_M(F) &= F . 
\end{alignat*}

We will sometimes use the fact that $r_M$ is conservative, that is, for a
morphism $u\colon F \to G$ in $\Derb(\Cor_M)$, if $r_M(u)$ is an isomorphism, so
is $u$.  Indeed $r_M$ is the derived functor of an exact functor and the
assertion follows from the case of the module categories $\Mod(\Cor_M)$,
$\Mod(\cor_M)$, where it is clear.

The ring morphism $\Cor\to \cor$ gives a basis of $\Hom_{\cor}(\Cor,\cor)$ (which is a free
$\Cor$-module). It induces $\Cor_M \isoto \rhom_{\cor_M}(\Cor_M,\cor_M)$ and then a canonical
isomorphism, for all $F\in \Derb(\cor_M)$,
\begin{equation}\label{eq:eMtorhom}
\Cor_M \tens_{\cor_M} F \simeq \rhom_{\cor_M}(\Cor_M,F)
\end{equation}
and we deduce the isomorphisms:
\begin{align}
\notag
\Hom_{\Derb(\Cor_M)}(G,e_M(F))
& \simeq \Hom_{\Derb(\Cor_M)}(G, \rhom_{\cor_M}(\Cor_M,F))  \\
\label{eq:adjrMeM}
& \simeq \Hom_{\Derb(\cor_M)}( G\tens_{\Cor_M} \Cor_M ,F)  \\
\notag
& \simeq \Hom_{\Derb(\cor_M)}(r_M(G),F)  .
\end{align}

\begin{definition}\label{def:categorie-orb}
We let $\perf(\Cor_M)$ be the full triangulated subcategory of $\Derb(\Cor_M)$
generated by the image of $e_M$, that is, by the objects of the form
$\Cor_M \tens_{\cor_M} F$ with $F\in \Derb(\cor_M)$.

We denote by $\Orb(\cor_M)$ the quotient category $\Derb(\Cor_M)/\perf(\Cor_M)$.  We
let $Q_M\cl \Derb(\Cor_M) \to \Orb(\cor_M)$ be the quotient functor and we set
$\iota_M = Q_M \circ R_M \cl \Derb(\cor_M) \to \Orb(\cor_M)$.
\end{definition}

An object of $\perf(\Cor_M)$ is obtained by taking iterated cones of objects in
$e_M(\Derb(\cor_M))$.  There are objects in $\perf(\Cor_M)$ which are not of the
form $e_M(F)$. For example, when $M$ is a point, the objects of $\Derb(\cor)$
are split (sum of their cohomology) and so are the objects of
$e_{\pt}(\Derb(\cor))$, but the object $L^{p,q} \in \perf(\Cor)$ defined
in~\eqref{eq:def-Lpq} below is not split.

\smallskip

Let $\perf'(\Cor_M)$ be the subcategory of $\Derb(\Cor_M)$ formed by the $P$ such
that $Q_M(P) \simeq 0$. Then $\perf(\Cor_M) \subset \perf'(\Cor_M)$ and we also have
$\Orb(\cor_M) \isoto \Derb(\Cor_M)/\perf'(\Cor_M)$. We do not know whether
$\perf'(\Cor_M) = \perf(\Cor_M)$. A general result (see~\cite{KS06} Ex.~10.11) says
that $P\in \perf'(\Cor_M)$ if and only if $P\oplus P[1] \in \perf(\Cor_M)$.

\begin{notation}
If the context is clear, we will not write the functor $Q_M$ or $R_M$, that is,
for $F\in \Derb(\cor_M)$ we often write $F$ instead of $R_M(F)$, and, for $G\in
\Derb(\Cor_M)$, we often write $G$ instead of $Q_M(G)$.  In particular for a
locally closed subset $Z\subset M$, we consider $\cor_Z = R_M(\cor_Z) \in
\Derb(\Cor_M)$ and $\cor_Z = Q_M(\cor_Z) \in \Orb(\cor_M)$.
\end{notation}

The exact sequence of $\Cor$-modules  $0\to \cor \to \Cor \to \cor \to 0$
induces a morphism
\begin{equation}\label{eq:def_shift_morph0}
s_M \cl \cor_M \to \cor_M[1] \qquad \qquad \text{in $\Derb(\Cor_M)$}
\end{equation}
and a distinguished triangle, for any $F\in \Derb(\cor_M)$,
\begin{equation}\label{eq:td_shift_morph0}
R_M(F) \to e_M(F)  \to  R_M(F) \to[s_M \tens \id_F] R_M(F)[1] .
\end{equation}
We thus obtain an isomorphism $s_M \tens \id_F \cl R_M(F) \isoto R_M(F)[1]$ in
$\Orb(\cor_M)$, for any $F\in \Derb(\cor_M)$.  This would work for any field
$\cor$. In characteristic $2$ we can generalize this isomorphism to any $F\in
\Derb(\Cor_M)$ (see Remark~\ref{rem:def_shift_morph2}).

\subsection*{Internal tensor product and homomorphism}

For two $\Cor$-modules $E_1,E_2$ we define $E_1 \epstens_\cor E_2 \in
\Mod(\Cor)$ as follows. The underlying $\cor$-vector space is $E_1 \otimes_\cor
E_2$ and $\varepsilon$ acts by
$$
\varepsilon \cdot(x\otimes y)
 = (\varepsilon x) \otimes y + x \otimes (\varepsilon y) .
$$
Since the characteristic is $2$, we can check that $\varepsilon^2$ acts by $0$ and
this defines an object of $\Mod(\Cor)$ that we denote $E_1 \epstens_\cor E_2$. We
obtain in this way a bifunctor
$\epstens_\cor \cl \Mod(\Cor) \times \Mod(\Cor) \to \Mod(\Cor)$.  For
$F_1,F_2 \in \Mod(\Cor_M)$, we define $F_1 \epstens_{\cor_M} F_2 \in \Mod(\Cor_M)$ as
the sheaf associated with the presheaf $U \mapsto F_1(U) \epstens_\cor F_2(U)$.  We
remark that $r_M(F_1 \epstens_{\cor_M} F_2) \simeq r_M(F_1) \tens_{\cor_M} r_M(F_2)$,
where $r_M$ is seen here as a functor $\Mod(\Cor_M) \to \Mod(\cor_M)$, and it follows
easily that $\epstens_{\cor_M}$ is an exact functor.   Hence it induces a
functor on the derived category (see the discussion around~\eqref{eq:funcder0}),
denoted in the same way:
\begin{equation}\label{eq:def-epstens}
\epstens_{\cor_M} \cl \Derb(\Cor_M) \times \Derb(\Cor_M) \to \Derb(\Cor_M) .  
\end{equation}
For any $F,G \in \Derb(\Cor_M)$ we have canonical isomorphisms
\begin{gather}
\label{eq:unit-epstens}
\cor_M \epstens_{\cor_M} F \simeq F \epstens_{\cor_M} \cor_M \simeq F
\quad \text{ in $\Derb(\Cor_M)$,}  \\
\label{eq:tens-epstens}
  r_M(F  \epstens_{\cor_M} G) \simeq r_M(F) \tens_{\cor_M} r_M(G) 
\quad \text{ in $\Derb(\cor_M)$.}
\end{gather}
Using~\eqref{eq:unit-epstens} and the exact sequence $0\to \cor \to \Cor \to
\cor \to 0$, we obtain as
in~(\ref{eq:def_shift_morph0}-\ref{eq:td_shift_morph0}) a morphism $s_M(F) \cl F
\to F[1]$, for any $F\in \Derb(\Cor_M)$, and a distinguished triangle
\begin{equation}\label{eq:td_shift_morph}
F \to  \Cor_M \epstens_{\cor_M} F  \to  F \to[s_M(F)] F[1] .
\end{equation}
Using the adjunction $(e_M,r_M)$ and~\eqref{eq:tens-epstens} we have the
isomorphism, for any $F\in \Derb(\cor_M)$ and $G \in \Derb(\Cor_M)$,  
\begin{equation}\label{eq:eM-epstens}
  \begin{split}
\Hom&_{\Derb(\Cor_M)} ( e_M(F \tens_{\cor_M} r_M(G)) , e_M(F)\epstens_{\cor_M} G)  \\
& \simeq
\Hom_{\Derb(\cor_M)}( F \tens_{\cor_M} r_M(G) , (r_Me_M(F))\tens_{\cor_M} r_M(G) ) .
  \end{split}
\end{equation}
By adjunction we have a morphism $a_F \cl F \to r_Me_M(F)$.  The inverse image
of $a_F \otimes \id_{r_M(G)}$ by~\eqref{eq:eM-epstens} gives a canonical
morphism, for any $F\in \Derb(\cor_M)$ and $G \in \Derb(\Cor_M)$,
\begin{equation}\label{eq:eM-epstens2}
e_M(F \tens_{\cor_M} r_M(G)) \to  e_M(F)\epstens_{\cor_M} G .
\end{equation}

\begin{lemma}\label{lem:epstens-perf}
  Let $F\in \Derb(\cor_M)$ and $G \in \Derb(\Cor_M)$.  Then the
  morphism~\eqref{eq:eM-epstens2} is an isomorphism.  In the same way
  $G \epstens_{\cor_M} e_M(F) \simeq e_M(r_M(G) \tens_{\cor_M} F)$. In
  particular, for $F,G \in \Derb(\Cor_M)$ such that $F$ or $G$ belongs to
  $\perf(\Cor_M)$, we have $F \epstens_{\cor_M} G \in \perf(\Cor_M)$ and
  $\epstens_{\cor_M}$ induces a functor
$$
\epstens_{\cor_M} \cl \Orb(\cor_M) \times \Orb(\cor_M) \to \Orb(\cor_M) .
$$
\end{lemma}
\begin{proof}
(i) Let us denote by $u_{F,G}$ the morphism~\eqref{eq:eM-epstens2}. Using the
distinguished triangle $\tau_{\leq i}F \to F \to \tau_{>i}F \to[+1]$ and the
similar one for $G$ we can argue by induction on the length of $F$ and $G$ to
prove that $u_{F,G}$ is an isomorphism.  Then we are reduced to the case where
$F$ and $G$ are concentrated in degree $0$.
Writing $\Cor = \cor \oplus \varepsilon \cor$ we have $e_M(F \tens_{\cor_M}
r_M(G)) = (F\otimes G) \oplus \varepsilon (F\otimes G)$ and $e_M(F)
\epstens_{\cor_M} G = (F \oplus \varepsilon F) \otimes G$.  For sections $f$ and
$g$ of $F$ and $G$ we have
\begin{align*}
u_{F,G} ( f\otimes g) &= f\otimes g, \\
u_{F,G} ( \varepsilon(f\otimes g)) &=  \varepsilon \cdot (f\otimes g)
= (\varepsilon f)\otimes g + f \otimes (\varepsilon g )
\end{align*}
and we can check directly that $u_{F,G}$ is an isomorphism with inverse
$$
u_{F,G}^{-1} ((f_0+\varepsilon f_1)\otimes g) 
= (f_0\otimes g + f_1 \otimes (\varepsilon g)) + \varepsilon(f_1\otimes g) .
$$
(ii) For any $F \in \perf(\Cor_M)$, there exist a sequence
$F_1, F_2,\ldots,F_n=F \in \Derb(\Cor_M)$ and distinguished triangles
$F_i \to F'_i \to F''_i \to[+1]$ in $\Derb(\Cor_M)$, $i=1,\dots,n$, such that
$F'_i$, $F''_i$ belong to $e_M(\Derb(\cor_M)) \cup \{F_1, F_2,\dots, F_{i-1}\}$.
Then~(i) and an induction on $n$ give $F \epstens_{\cor_M} G \in \perf(\Cor_M)$.
\end{proof}

\begin{remark}\label{rem:def_shift_morph2}
An easy case of Lemma~\ref{lem:epstens-perf} is $F=\cor_M$
in~\eqref{eq:eM-epstens2}. We obtain $e_Mr_M(G) \isoto \Cor_M\epstens_{\cor_M}
G$, for any $G\in \Derb(\Cor_M)$.  Hence the distinguished
triangle~\eqref{eq:td_shift_morph} becomes
\begin{equation}\label{eq:td_shift_morph-bis}
F \to  e_Mr_M(F)  \to  F \to[s_M(F)] F[1]
\quad\text{ for any $F\in\Derb(\Cor_M)$}.
\end{equation}
Applying $Q_M$ to this triangle gives an isomorphism $s_M(F) \cl F \isoto F[1]$
in $\Orb(\cor_M)$.
\end{remark}

We can define an adjoint $\homeps$ to $\epstens$ by a similar construction.
For $F_1,F_2 \in \Mod(\Cor_M)$, we define $\homeps(F_1,F_2) \in \Mod(\Cor_M)$ as
the sheaf of $\cor$-vector spaces $\hom_\cor(F_1,F_2)$ with the action of
$\varepsilon$ given by
$$
(\varepsilon \cdot \varphi)(x)
 = \varepsilon \varphi(x) + \varphi(\varepsilon x), 
$$
where $\varphi$ is a section of $\hom_\cor(F_1,F_2)$ over an open set $U$ and
$x$ a section of $F_1$ over a subset $V$ of $U$.  Then we see that $\homeps$ is
right adjoint to $\epstens$, hence left exact.  We check also that its derived
functor $\rhomeps$ is right adjoint to $\epstens$ in $\Derb(\Cor_M)$ and that,
for any $F,G \in \Derb(\Cor_M)$,
\begin{equation}\label{eq:hom-homeps}
r_M(\rhomeps(F,G) ) \simeq \rhom(r_M(F), r_M(G)).
\end{equation}
We have the similar result as Lemma~\ref{lem:epstens-perf}.
\begin{lemma}\label{lem:homeps-perf}
Let $F,G \in \Derb(\Cor_M)$. We assume that $F$ or $G$ belongs to
$\perf(\Cor_M)$.  Then $\rhomeps(F,G) \in \perf(\Cor_M)$. The induced functor
$$
\rhomeps \cl \Orb(\cor_M)^{op} \times \Orb(\cor_M) \to \Orb(\cor_M) .
$$
is right adjoint to $\epstens_{\cor_M}$.
\end{lemma}

\subsection*{Morphisms in the triangulated orbit category}
We prove the formula~\eqref{eq:mor-orbA} which describes the morphisms in
$\Orb(\cor_M)$.

\begin{lemma}\label{lem:mor-perf-borne}
Let $F,P \in \Derb(\Cor_M)$. We assume that $P \in \perf(\Cor_M)$.
Then $\RHom(P,F)$ and $\RHom(F,P)$ belong to $\Derb(\Cor)$.
\end{lemma}
\begin{proof}
  Since $\perf(\Cor_M)$ is generated by $e_M(\Derb(\cor_M))$, the same argument as
  in~(ii) of the proof of Lemma~\ref{lem:epstens-perf} implies that we can assume
  $P=e_M(Q)$, for some $Q\in \Derb(\cor_M)$. Then $\RHom_{\Cor}(P,F)$ is isomorphic
  to $\RHom_\cor(Q,r_M(F))$ and is bounded. Using~\eqref{eq:adjrMeM} the same proof
  gives that $\RHom(F,P)$ is bounded.
\end{proof}

We define the following objects $L^{p,q}$ of $\Derb(\Cor)$, for any
two integers $p\leq q$, by
\begin{equation}\label{eq:def-Lpq}
  L^{p,q} \; = \; 0 \to \Cor \to[\varepsilon] \Cor \to[\varepsilon] \dots
\to[\varepsilon] \Cor \to 0,
\end{equation}
where the first $\Cor$ is in degree $p$ and the last one in degree $q$.  Then
$L^{p,q} \in \perf(\Cor)$ and there is a distinguished triangle in $\Derb(\Cor)$,
\begin{equation}\label{eq:dt-Lpq}
\cor[-p] \to L^{p,q} \to \cor[-q] \to[s^{p,q}] \cor[-p+1],
\end{equation}
with $s^{p,q} = s_{\pt}^{q-p+1}[-q]$, where $s_{\pt}$
is~\eqref{eq:def_shift_morph0}.  For $F \in \Derb(\Cor_M)$ we define
$s^{p,q}_M(F) \eqdot s^{p,q}\tens \id_F \cl F[-q] \to F[-p+1]$.  We deduce the
triangle, for any $F \in \Derb(\Cor_M)$ and any $n\geq 1$,
$$
L^{1,n}_M \epstens_{\cor_M} F \to F[-n] \to[{s_M^{1,n}(F)}] F \to[+1] .
$$
\begin{lemma}\label{lem:morph-orb}
We consider a distinguished triangle $P \to F' \to F \to[+1]$ in $\Derb(\Cor_M)$
and we assume that $P\in \perf(\Cor_M)$. Then there exist $n\in \N$ and a
morphism of triangles
\begin{equation*}
\begin{tikzcd}
L^{1,n}_M \epstens_{\cor_M} F \ar[r]\ar[d] & F[-n] \ar[rr, "s_M^{1,n}(F)"] \ar[d]
&&  F \ar[r, "+1"] \ar[d, equal] & {} \\
P \ar[r] & F'  \ar[rr]  && F \ar[r, "+1"]  &  \pointdiag 
\end{tikzcd}
\end{equation*}
\end{lemma}
\begin{proof}
We set for short $s_n = {s_M^{1,n}(F)}$. We consider the diagram
$$
\begin{tikzcd}
& F[-n] \ar[r, "s_n"]  \ar[dr, "s_n"]  \ar[d, dashed, "a"] &  F \ar[d, equal] \\
P \ar[r] & F'  \ar[r]  & F \ar[r, "w"]  &  P[1] \pointdiag 
\end{tikzcd}
$$
We have $w \circ s_n \in \Hom(F[-n],P[1])$.  Since $P\in \perf(\Cor_M)$ this
group vanishes for $n$ big enough, by Lemma~\ref{lem:mor-perf-borne}. In this
case we have $w \circ s_n =0$ and there exists a morphism $a$ as in the diagram
making the square commute.  Then we can extend the square to a commutative
diagram as in the lemma.
\end{proof}

For $F \in \Derb(\Cor_M)$ we have $s^{n,n}_M(F) \cl F[-n] \to F[-n+1]$.  Then
$\{F[-n], s^{n,n}_M(F)\}_{n\in \N}$ gives a projective system.  For $G\in
\Derb(\Cor_M)$ we define
\begin{equation}\label{eq:mor-Der-Orb}
\varinjlim_{n\in \N} \Hom_{\Derb(\Cor_M)}(F[-n],G) \to \Hom_{\Orb(\cor_M)}(F,G),
\end{equation}
by sending $\varphi_n \cl F[-n]\to G$ to $\varphi_n \circ (s_M^{1,n}(F))^{-1}$.
This is well defined since $s_M^{1,n}(F)$ becomes invertible in $\Orb(\cor_M)$
and $s_M^{1,n}(F) = s_M^{1,n-1}(F) \circ s_M^{n,n}(F)$.

\begin{proposition}\label{prop:morph-orb}
Let $F,G \in \Derb(\Cor_M)$. Then the inductive limit in the left hand side
of~\eqref{eq:mor-Der-Orb} stabilizes and the morphism~\eqref{eq:mor-Der-Orb} is
an isomorphism. More precisely, if $H^i(F) = H^i(G) =0$ for all $i$ outside an
interval $[a,b]$, then
\begin{equation}\label{eq:mor-orbA}
\Hom_{\Derb(\Cor_M)}(F[-n],G) \isoto \Hom_{\Orb(\cor_M)}(F,G),
\end{equation}
for all $n > b-a + \dim M +1$.
\end{proposition}
\begin{proof}
(i) We prove that the limit stabilizes. We chose $a\leq b$ such that $H^i(F) =
H^i(G) =0$ for all $i$ outside $[a,b]$.  By~\eqref{eq:td_shift_morph-bis} we
have the distinguished triangle
\begin{equation}\label{eq:morph-orb1}
\begin{split}
\Hom_{\Derb(\Cor_M)}( e_M r_M(F)[-n],G)
&\to \Hom_{\Derb(\Cor_M)}(F[-n],G) \\
\to[s'_n] {}& \Hom_{\Derb(\Cor_M)}(F[-n-1],G) \to[+1] ,
\end{split}
\end{equation}
for all $n\in \Z$. By adjunction we have
$$
\Hom_{\Derb(\Cor_M)}( e_M r_M(F)[-n],G) \simeq
\Hom_{\Derb(\cor_M)}(r_M(F)[-n],r_M(G) )
$$
and this is zero for $n> b-a + \dim M +1$ (recall that the flabby dimension
of a manifold $M$ is $\dim M +1$ and that injective over a field is the same as
flabby). It follows that the morphism $s'_n$ in~\eqref{eq:morph-orb1} is an
isomorphism for $n> b-a + \dim M +1$.

\medskip\noindent
(ii) We prove that~\eqref{eq:mor-Der-Orb} is an isomorphism.  We recall that
$$
\Hom_{\Orb(\cor_M)}(F,G) \simeq \varinjlim_{i\cl F'\to F}
\Hom_{\Derb(\Cor_M)}(F',G),
$$
where the limit runs over the morphisms $i\cl F'\to F$ whose cone belongs to
$\perf(\Cor_M)$ (and a morphism $u\cl F' \to G$ is send to $u \circ i^{-1}$ in
$\Orb(\cor_M)$).  By Lemma~\ref{lem:morph-orb} we can restrict to the family of
morphisms $s_M^{1,n}(F) \cl F[-n] \to F$ for $n\in \N$. This gives the result.
\end{proof}

\begin{corollary}\label{cor:morph-QRF-QRG}
  Let $F,G \in \Derb(\cor_M)$. Let $\iota_M = Q_M \circ R_M$.  We have  
$$
\Hom_{\Orb(\cor_M)}(\iota_M(F), \iota_M(G))
\simeq \bigoplus_{n\in\Z} \Hom_{\Derb(\cor_M)}(F[-n],G) .
$$
\end{corollary}
\begin{proof}
By Proposition~\ref{prop:morph-orb} the left hand side of the formula is
isomorphic to
$$
\Hom_{\Derb(\Cor_M)}(R_M(F)[-n_0], R_M(G))
\simeq \Hom_{\Der^-(\cor_M)}(E_M R_M(F)[-n_0], G),
$$
for any big enough $n_0\in\N$.  Using the resolution of $\cor$ as a
$\Cor$-module given by $\cdots\to \Cor \to[\varepsilon] \Cor \to[\varepsilon]
\Cor \to \cor \to 0$, we see that $E_MR_M(F) \simeq \bigoplus_{i\in \N} F[i]$.
The result follows easily.
\end{proof}

This last result says in particular that $\iota_M$ is faithful.  We define
$\iota^0_M \colon \Mod(\cor_M) \to \Orb(\cor_M)$ as the composition of $\iota_M$
and the embedding $\Mod(\cor_M) \to \Derb(\cor_M)$ which sends a sheaf to a
complex concentrated in degree $0$.  For $F \in \Orb(\cor_M)$ we let $h^0_M(F)$
be the sheaf associated with the presheaf
$U \mapsto \Hom_{\Orb(\cor_U)}(\cor_U,F|_U)$.  This defines a functor
$h^0_M \colon \Orb(\cor_M) \to \Mod(\cor_M)$.

\begin{corollary}\label{cor:iota_equivcat}
  For all $F \in \Mod(\cor_M)$ we have $h^0_M(\iota^0_M(F)) \simeq F$.  When $M$
  is a point, the functors $\iota^0$ and $h^0$ are mutually inverse equivalences
  of categories between $ \Mod(\cor)$ and $\Orb(\cor)$.
\end{corollary}
\begin{proof}
  (i) We have
  $\Hom_{\Orb(\cor_U)}(\cor_U,\iota_M(F)|_U) \simeq \bigoplus_{n\in\Z} H^n(U;
  F)$ by Corollary~\ref{cor:morph-QRF-QRG}.  The sheaf associated with
  $U \mapsto H^n(U; F)$ is $H^nF$ and we obtain $h^0_M(\iota_M(F)) \simeq F$
  when $F$ is in degree $0$.

  \sui(ii) When $M$ is a point, Corollary~\ref{cor:morph-QRF-QRG} says that
  $\iota^0$ is fully faithful. Let us prove that it is essentially surjective.
  Any $G \in \Mod(\Cor)$ can be written as an extension
  $0 \to R(F) \to G \to R(F') \to 0$ with $F,F' \in \Mod(\cor)$.  It follows
  that any object in $\Orb(\cor)$ is obtained from objects in
  $\iota^0(\Mod(\cor))$ by taking iterated cones.  Hence to prove that $\iota^0$
  is essentially surjective, it is enough to see that if we have a distinguished
  triangle $\iota^0(F) \to[u] \iota^0(F') \to G \to[+1]$ in $\Orb(\cor)$, with
  $F,F' \in \Mod(\cor)$, then $G \simeq \iota^0(F'')$ for some
  $F'' \in \Mod(\cor)$.  Since $\iota^0$ is fully faithful, we have
  $u \in \Hom_{\Mod(\cor)}(F,F')$.  Then
  $G \simeq \coker(u) \oplus \ker(u)[1] \simeq \coker(u) \oplus \ker(u)$ and the
  result follows.
\end{proof}

\subsection*{Direct sums}

We recall that we denote by $Q_M$ the quotient functor $\Derb(\Cor_M) \to
\Orb(\cor_M)$.
\begin{lemma}\label{lem:direct-sum}
Let $I$ be a small set and $\{F_i\}_{i\in I}$ a family in $\Derb(\Cor_M)$.  We
assume that there exist two integers $a\leq b$ such that $H^k(F_i)=0$ for all
$k$ outside $[a,b]$ and all $i\in I$.
Then $\bigoplus_{i\in I} Q_M(F_i)$ exists in $\Orb(\cor_M)$ and
$\bigoplus_{i\in I} Q_M(F_i) \simeq Q_M(\bigoplus_{i\in I} F_i)$.
\end{lemma}
\begin{proof}
By the hypothesis on the degrees the sum $\bigoplus_{i\in I} F_i$ exists in
$\Derb(\Cor_M)$ and we have $H^k(\bigoplus_{i\in I} F_i)=0$ for $k$ outside
$[a,b]$. Let $G \in \Derb(\Cor_M)$ and let $a'\leq a$, $b'\geq b$ be such that
$H^k(G)=0$ for $k$ outside $[a',b']$. We set $n=b'-a' + \dim M +2$.  If $F =
F_i$ for some $i\in I$, or $F = \bigoplus_{i\in I} F_i$, we have
$$
\Hom_{\Derb(\Cor_M)}(F[-n],G) \isoto \Hom_{\Orb(\cor_M)}(Q_M(F),Q_M(G))
$$
by Proposition~\ref{prop:morph-orb}.  Now the lemma follows from the universal
property of the sum.
\end{proof}

\subsection*{Direct and inverse images}

Let $f\cl M\to N$ be a morphism of manifolds.  We have functors $\roim{f}$,
$\reim{f}$, $\opb{f}$ and $\epb{f}$, between $\Derb(\Cor_M)$ and
$\Derb(\Cor_N)$. Indeed, the functors
$\roim{f}, \reim{f} \cl \Der(\Cor_M) \to \Der(\Cor_N)$ commute with the functors
$r_N \cl \Der(\Cor_N) \to \Der(\cor_N)$ and $r_M$.  We remark that, for
$G \in \Der(\Cor_N)$, if $r_N(G) \in \Derb(\cor_N)$, then $G \in \Derb(\Cor_N)$.
Hence $\roim{f}$ and $\reim{f}$ induce functors
$\Derb(\Cor_M) \to \Derb(\Cor_N)$ between the bounded categories.

The case of $\opb{f}$ is clear since it is an exact functor.  To prove the
existence of $\epb{f}$, right adjoint to $\reim{f}$, it is enough, by
factorizing through the graph embedding, to consider the cases where $f$ is an
embedding or $f$ is a submersion. The usual formulas work in our case.  If $f$
is an embedding, then $\epb{f}(\cdot) = \opb{f} \rhomeps(\cor_M,\cdot)$ is
adjoint to $\reim{f}$. If $f$ is a submersion, then $\epb{f}(\cdot) =
\opb{f}(\cdot) \epstens_{\cor_M} \omega_{M|N}$.  We also remark that $\opb{f}$ and
$\epb{f}$ commute with $r_N$ and $r_M$.

\begin{lemma}\label{lem:im-dir-inv_Orb}
Let $f\cl M\to N$ be a morphism of manifolds.  Then the functors $\roim{f}$,
$\reim{f}$, $\opb{f}$ and $\epb{f}$, between $\Derb(\Cor_M)$ and
$\Derb(\Cor_N)$, preserves the categories $\perf(\Cor_M)$ and $\perf(\Cor_N)$.
They induce pairs of adjoint functors (that we denote in the same way) between
$\Orb(\cor_M)$ and $\Orb(\cor_N)$.
\end{lemma}
\begin{proof}
  We only consider the case of $\roim{f}$, the other cases being similar.  For
  $F\in \Derb(\cor_M)$ we have a natural morphism
  $u\cl \Cor_N \tens_{\cor_N}\roim{f} F \to \roim{f}(\Cor_M \tens_{\cor_M}
  F)$. Since $r_N(\Cor_N) \simeq \cor_N^2$ we see easily that $r_N(u)$ is an
  isomorphism. Since $r_N$ is conservative, $u$ is an isomorphism and we obtain
  $\roim{f}(\Cor_M \tens_{\cor_M} F) \in \perf(\Cor_N)$.  It follows as in~(ii)
  of the proof of Lemma~\ref{lem:epstens-perf} that
  $\roim{f}(\perf(\Cor_M)) \subset \perf(\Cor_N)$, as required.
\end{proof}

For $F \in \Orb(\cor_M)$ and $j\cl U \to M$ the inclusion of a locally closed
subset, we use the standard notations $F|_U = \opb{j}F$, $F_U =
\reim{j}\opb{j}F$, $\rsect_U(F) = \roim{j}\opb{j}F$, $\rsect(U;F) =
\roim{a_U}(F|_U) \in \Orb(\cor)$, where $a_U$ is $U \to \pt$, and $\rsect_c(U;F)
= \reim{a_U}(F|_U) \in \Orb(\cor)$. We have the same formulas as in
$\Derb(\cor_M)$:
$$
F_U  \simeq F \epstens_{\cor_M} \cor_U, \qquad \rsect_U(F) \simeq \rhomeps(\cor_U,F) .
$$
We also define
$\RHom^\varepsilon(F,G) = \rsect(M; \rhomeps(F,G)) \in \Orb(\cor)$.  The
adjunctions $(\opb{a_M}, \roim{a_M})$ and $(\epstens_{\cor_M}, \rhomeps)$ give
\begin{equation}
  \label{eq:Hom_RHomepsilon}
  \begin{aligned}
  \Hom_{\Orb(\cor_M)}(F,G)
  &\simeq \Hom_{\Orb(\cor)}(\cor, \RHom^\varepsilon(F,G)) \\
&\simeq \Hom_{\Orb(\cor_M)}(\cor_M, \rhomeps(F,G)) .
\end{aligned}
\end{equation}

Let $N$ be a submanifold of $M$.  We recall that Sato's microlocalization is a
functor $\mu_N\cl \Derb(\cor_M) \to \Derb(\cor_{T^*_NM})$.  It is defined by
composing direct and inverse images functors and Lemma~\ref{lem:im-dir-inv_Orb}
implies that it induces functors, denoted in the same way:
\begin{align}
\mu_N&\cl \Derb(\Cor_M) \to \Derb(\Cor_{T^*_NM}) , \\
\mu_N&\cl \Orb(\cor_M) \to \Orb(\cor_{T^*_NM}) .  
\end{align}

\begin{definition}
Let $q_1,q_2\cl M\times M \to M$ be the projections.  We identify
$T^*_{\Delta_M}(M\times M)$ with $T^*M$ through the first projection.  For
$F,G\in \Orb(\cor_M)$ we define as in~\eqref{eq:def_muhom}
\begin{equation*}
\muhom^\varepsilon (F,G) = \mu_{\Delta_M}(\rhomeps(\opb{q_2}F,\epb{q_1}G))
\quad \in \quad \Orb(\cor_{T^*M}). 
\end{equation*}
\end{definition}

The following result follows from the analogous one in $\Derb(\Cor_M)$.
\begin{lemma}\label{lem:dir-sum_im-inv-dir}
Let $\{F_i\}_{i\in I}$ a small family in $\Derb(\Cor_M)$ satisfying the
hypotheses of Lemma~\ref{lem:direct-sum}.  Let $f\cl M' \to M$ and $g\cl M\to
M''$ be morphisms of manifolds and let $G\in \Orb(\cor_M)$.  Then we have
canonical isomorphisms
\newcommand{\somme}{{\textstyle\bigoplus_{i\in I}}\;}
\begin{align*}
\opb{f} (\somme Q_M(F_i)) &\simeq \somme \opb{f}Q_M(F_i), \\
\reim{g}(\somme Q_M(F_i)) &\simeq \somme \reim{g}\, Q_M(F_i), \\
(\somme Q_M(F_i)) \epstens_{\cor_M} G 
&\simeq \somme (Q_M(F_i)) \epstens_{\cor_M} G).
\end{align*}
\end{lemma}

\begin{lemma}\label{lem:modif-repr}
Let $U$ be an open subset of $M$. Let $F\in \Derb(\Cor_M)$ and $F' \in
\Derb(\Cor_U)$. We assume that there exists an isomorphism $F|_U\simeq F'$ in
$\Orb(\cor_U)$.  Then there exists $F_1 \in \Derb(\Cor_M)$ such that $F_1|_U
\simeq F'$ in $\Derb(\Cor_U)$ and $F_1\simeq F$ in $\Orb(\cor_M)$.
\end{lemma}
\begin{proof}
We let $j\cl U \to M$ be the inclusion and we set $Z=M\setminus U$.
Let $u \cl F|_U \to F'$ be an isomorphism in $\Orb(\cor_U)$.  By
Proposition~\ref{prop:morph-orb} there exist $n\in \Z$ and a morphism $F[-n] \to
F'$ in $\Derb(\Cor_M)$ which represents $u$.  Defining $P$ by the distinguished
triangle $F|_U[-n] \to F' \to P \to[+1]$ we have $Q_U(P) \simeq 0$.  We apply
$\eim{j}$ to this triangle and get~\eqref{eq:modif-repr2} below; we also
consider the excision triangle~\eqref{eq:modif-repr1} and the
triangle~\eqref{eq:modif-repr3} built on the composition $F_Z[-n-1] \to F_U[-n]
\to \eim{j} F'$:
\begin{gather}
\label{eq:modif-repr1}
F_Z[-n-1] \to[a] F_U[-n] \to F[-n] \to[+1] , \\
\label{eq:modif-repr2}
F_U[-n] \to[b]  \eim{j} F' \to \eim{j} P \to[+1] , \\
\label{eq:modif-repr3}
F_Z[-n-1] \to[b\circ a] \eim{j} F' \to F_1  \to[+1] .
\end{gather}
Then the octahedron axiom gives the triangle $F[-n] \to F_1 \to \eim{j} P
\to[+1]$. We have $Q_M(\eim{j} P) \simeq \eim{j} Q_M(P) \simeq 0$, hence
$F_1\simeq F[-n] \simeq F$ in $\Orb(\cor_M)$.  Applying $\opb{j}$ to the
triangle~\eqref{eq:modif-repr3} gives $F' \simeq F_1|_U$, as required.
\end{proof}

\begin{definition}\label{def:supporb}
For $F \in \Orb(\cor_M)$ we define $\supporb(F) \subset M$ as the complement of
the union of the open subsets $U\subset M$ such that $F|_U \simeq 0$.
\end{definition}
For an open subset $U\subset M$ we have $F|_U \simeq 0$ in $\Orb(\cor_U)$ if and only
if $F_U \simeq 0$ in $\Orb(\cor_M)$.  By the Mayer-Vietoris triangle we deduce that,
for a finite covering $U=\bigcup_{i=1}^n U_i$, we have $F|_U \simeq 0$ if and only if
$F|_{U_i} \simeq 0$ for all $i$.    For an increasing countable union $U=\bigcup_{i=1}^\infty U_i$ we have an
exact sequence
$$
0 \to \bigoplus_{i=1}^\infty \cor_{U_i} \to[\id -t] \bigoplus_{i=1}^\infty
\cor_{U_i} \to[s] \cor_U \to 0
$$
in $\Mod(\cor_M)$, where $t$ is the sum of the morphisms
$t_i \colon \cor_{U_i} \to \cor_{U_{i+1}}$ induced by the inclusions
$U_i \subset U_{i+1}$ and $s$ the sum of the morphisms
$s_i \colon \cor_{U_i} \to \cor_{U}$ (the exactness is easily checked in the stalks).
Turning this sequence into a distinguished triangle in $\Der(\cor_M)$, then in
$\Orb(\cor_M)$, and applying $F \epstens_{\cor_M}-$ we obtain a similar triangle
\begin{equation}\label{eq:lim_crois_ouverts}
\bigoplus_{i=1}^\infty F_{U_i} \to \bigoplus_{i=1}^\infty F_{U_i} \to F_U
\to[+1]
\end{equation}
and we deduce as in the finite case that $F|_U \simeq 0$ if
$F|_{U_i} \simeq 0$ for all $i$.  We obtain finally, for any
$F \in \Orb(\cor_M)$,
\begin{equation}\label{eq:supporb-OK}
F|_{M\setminus\supporb(F)} \simeq 0 \quad \text{and} \quad
F \isoto F_{\supporb(F)} .
\end{equation}

\section{Microsupport in the triangulated orbit categories}
\label{sec:micsup_orb_cat}

We define the microsupport of objects of $\Orb(\cor_M)$ and check that it
satisfies the same properties as the usual microsupport.

We recall that the microsupport is invariant by restriction of scalars, that is,
for $F \in \Der(\Cor_M)$, we have $\SSi(r_M(F)) = \SSi(F)$.  We deduce that
Theorem~\ref{thm:SSrhom}, about the microsupports of $F\ltens G$ and
$\rhom(F,G)$, for $F,G \in \Der(\Cor_M)$, is still true if we replace $\ltens$
and $\rhom$ by $\epstens_{\cor_M}$ and $\rhomeps$, because of~\eqref{eq:tens-epstens}
and~\eqref{eq:hom-homeps}.

\subsection{Definition and first properties}
We define the microsupport $\SSo(F)$ of an object of $F\in\Orb(\cor_M)$ from the
microsupports of its representatives in $\Derb(\Cor_M)$.  We prove in
Proposition~\ref{prop:repr-bonmicsup} that, for a given $x_0\in M$ and $F\in
\Orb(\cor_M)$, we can find a representative $F'\in \Derb(\Cor_M)$ with
$T^*_{x_0}M \cap \SSi(F')$ contained in an arbitrary neighborhood of $T^*_{x_0}M
\cap \SSo(F)$ .

\begin{definition}\label{def:SSorb}
Let $F \in \Orb(\cor_M)$. We define $\SSo(F) \subset T^*M$ by $\SSo(F) =
\bigcap_{F'} \SSi(F')$ where $F'$ runs over the objects of $\Derb(\Cor_M)$ such
that $F'\simeq F$ in $\Orb(\cor_M)$. We set $\dSSo(F) = \SSo(F) \cap \dT^*M$.
\end{definition}

We remark that $\SSo(F)$ is a closed conic subset of $T^*M$.  We deduce from
Lemma~\ref{lem:modif-repr} that $\SSo(F)$ is a local notion, that is, for
$U\subset M$ open, we have
\begin{equation}\label{eq:SSorb_local}
\SSo(F|_U) = \SSo(F) \cap T^*U.
\end{equation}
In other words $p=(x;\xi) \not\in \SSo(F)$ if and only if there exist a
neighborhood $U$ of $x$ and $F'\in \Derb(\Cor_U)$ such that $F|_U \simeq F'$ in
$\Orb(\cor_U)$ and $p\not\in \SSi(F')$.  We also have
$\supporb(F) = T^*_MM \cap \SSo(F)$.  Indeed we have
$T^*_MM \cap \SSo(F) \subset \supporb(F)$ by~\eqref{eq:SSorb_local}. Conversely,
if $(x;0) \not\in \SSo(F)$, then $F$ has a representative $F' \in \Derb(\Cor_M)$
such that $(x;0) \not\in \SSi(F')$.  Hence $F'$, and thus $F$, vanishes in some
neighborhood of $x$.

\begin{lemma}\label{lem:repr-bonmicsup}   Let $F \in \Orb(\cor_M)$
  and $F', F'' \in \Derb(\Cor_M)$ such that $Q_M(F') \simeq Q_M(F'') \simeq F$.  Let
  $x_0 \in M$ and let $A \subset \dT^*_{x_0}M$ be an open contractible cone with a
  smooth boundary such that $(\ol{A} \setminus \{x_0\}) \cap \SSi(F'') = \emptyset$.
  Let $C \subset \dT^*_{x_0}M$ be a conic neighborhood of $\SSi(F') \cap \partial A$.
  Then there exists $G \in \Derb(\Cor_M)$ such that $Q_M(G) \simeq F$ and
  $\SSi(G) \cap T^*_{x_0}M \subset (\SSi(F') \setminus A) \cup C$.
\end{lemma}
\begin{proof}
  Let $U$ be a chart around $x_0$ such that $T^*U \simeq U \times T^*_{x_0}M$ and $(U
  \times A) \cap \SSi(F'') = \emptyset$.  We apply
  Proposition~\ref{prop:microcutofflevrai1} with $B = \SSi(F') \cap T^*_{x_0}M$ and $B' =
  C$.    We obtain a neighborhood $W$ of $x_0$ and a functor $R \cl
  \Der(\Cor_U) \to \Der(\Cor_W)$ together with a morphism of functors $(-)|_W \to R$.
  This functor $R$ is a composition of usual sheaf operations and, by the results of the
  previous section, it induces a functor that we denote by the same letter $R \colon
  \Orb(\Cor_U) \to \Orb(\Cor_W)$.  By~(ii) of Proposition~\ref{prop:microcutofflevrai1} we
  have $F''|_W \isoto R(F''|_U)$. Hence $F|_W \isoto R(F|_U)$ and it follows that
  $Q_W(R(F'|_U)) \simeq R(Q_U(F'|_U)) \simeq F|_W$.  By~(i) of
  Proposition~\ref{prop:microcutofflevrai1} the sheaf $R(F'|_U)$, defined on $W$,
  satisfies the conclusion of the lemma.  By Lemma~\ref{lem:modif-repr} we can extend
  $R(F'|_U)$ into $G$, defined on $M$.
\end{proof}

\begin{proposition}\label{prop:repr-bonmicsup}
  Let $F\in \Orb(\cor_M)$.  Let $x_0\in M$ be given and let
  $B \subset \dT^*_{x_0}M$ be a closed conic subset such that
  $\SSo(F) \cap B = \emptyset$.  Then there exists $F' \in \Derb(\Cor_M)$ such
  that $Q_M(F') \simeq F$ and $\SSi(F') \cap B = \emptyset$.
\end{proposition}
\begin{proof}
  (i) We first prove the following claim to reduce the problem to
  Lemma~\ref{lem:repr-bonmicsup}:
  \begin{quote}
    Let $A_1,\dots,A_n$ be open cones and $S$ a closed cone in $T^*_{x_0}M$. Let
    $C$ be a conic neighborhood of $S \cap \partial (\bigcup_{i=1}^n A_i)$. Then
    there exist conic subsets $C_1,\dots,C_n$ of $T^*_{x_0}M$ such that, defining
    inductively $S_0 =S$ and $S_i = (S_{i-1} \setminus A_i) \cup C_i$, we have: $C_i$
    is a neighborhood of $S_{i-1} \cap \partial A_i$ and
    $S_n \subset (S \setminus \bigcup_{i=1}^n A_i) \cup C$.
  \end{quote}
  We prove the claim by induction on $n$.  For $n=1$ we take $C_1=C$ and the claim is
  clear.  Let us assume we have proved it for $n-1$.  Since
  $(S \cap \partial A_1 \cap \partial (\bigcup_{i=2}^{n} A_i)) \subset C$ and
  $(S \cap \partial A_1) \subset C \cup \bigcup_{i=2}^{n} A_i$, we can choose a
  neighborhood $C_1$ of $S \cap \partial A_1$ such that
  $(C_1 \cap \partial(\bigcup_{i=2}^{n} A_i)) \subset C$ and
  $C_1 \subset C \cup \bigcup_{i=2}^{n} A_i$.  We set
  $S_1 = (S \setminus A_1) \cup C_1$.  Then $C$ is a neighborhood of
  $S_1 \cap \partial (\bigcup_{i=2}^{n} A_i)$ and we can apply the induction
  hypothesis with $A_2,\dots,A_n$, $S=S_1$ and $C$. The claim follows (using
  $(S_1 \setminus \bigcup_{i=2}^n A_i) \cup C \subset (S \setminus \bigcup_{i=1}^n
  A_i) \cup C)$).

  \sui(ii) For any $p\in B$ we can find $F^p \in \Derb(\Cor_M)$ representing $F$ such
  that $p\not\in \SSi(F^p)$. We can find an open convex cone $A^p$ around $p$ such
  that $(\ol{A^p}\setminus \{x_0\}) \cap \SSi(F^p) = \emptyset$. We can choose
  finitely many such cones, say $A_1,\dots,A_n$, such that
  $B \subset \bigcup_{i=1}^n A_i$ and we let $F_1,\dots, F_n \in \Derb(\Cor_M)$ be
  the corresponding sheaves.  We also choose an arbitrary representative $F_0$ of $F$
  and set $S = \SSi(F_0)$.  Finally we choose a neighborhood $C$ of
  $S \cap \partial (\bigcup_{i=1}^n A_i)$ such that $C \cap B = \emptyset$.

  We let $C_1,\dots,C_n$ be the subsets of $T^*_{x_0}M$ given by the claim in~(i).
  
  Now we define $G_i \in \Derb(\Cor_M)$, $i=0,\ldots,n$ inductively: we start with
  $G_0 = F_0$, and using Lemma~\ref{lem:repr-bonmicsup}, with $F' = G_{i-1}$,
  $F'' = F_i$, $A = A_i$, $C=C_i$, we set $G_i = G$ ($G$ given by the lemma). Then
  $G_i$ satisfies $\SSi(G_i) \subset S_i$, for each $i$. In particular the sheaf
  $G_n$ satisfies the conclusion of the proposition.
  \end{proof}

\begin{proposition}\label{prop:trian-ineq-SSo}
Let $F \to F' \to F'' \to[+1]$ be a distinguished triangle
in $\Orb(\cor_M)$. Then $\SSo(F'') \subset \SSo(F) \cup \SSo(F')$.
\end{proposition}
\begin{proof}
  Let $p=(x_0;\xi_0) \in T^*M$ be given. We assume that
  $p\not\in \SSo(F) \cup \SSo(F')$.  Let $u \cl F \to F'$ be the morphism of the
  triangle.  By definition we can find $F_0, F'_0 \in \Derb(\Cor_M)$ such that
  $Q_M(F_0) \simeq F$, $Q_M(F'_0) \simeq F'$ and $p\not\in \SSi(F_0)$,
  $p\not\in \SSi(F'_0)$.  By Proposition~\ref{prop:morph-orb} there exist
  $n\in \Z$ and a morphism $u_0 \colon F_0[n] \to F'_0$ such that
  $Q_M(u_0) = u$.  Then the cone of $u_0$, say $F''_0$, represents $F''$ and
  $p\not\in \SSi(F''_0)$ by the triangular inequality for the usual
  microsupport. The result follows.
\end{proof}

\subsection{Functorial behavior}
We prove that $\SSo(\cdot)$ satisfies the same properties as $\SSi(\cdot)$
with respect to the usual sheaf operations.

\begin{proposition}\label{prop:im-inv-SSorb}
Let $f\cl M \to N$ be a morphism of manifolds.  Let $G\in \Orb(\cor_N)$. We
assume that $f$ is non-characteristic for $\SSo(G)$. Then $\SSo(\opb{f}G) \cup
\SSo(\epb{f}G) \subset f_d\opb{f_\pi}\SSo(G)$.
\end{proposition}
\begin{proof}
  (i) The cases of $\opb{f}$ and $\epb{f}$ are similar and we only consider
  $\opb{f}$.  We can write $f = p\circ i$ where $i\cl M \to M\times N$ is the
  graph embedding and $p\cl M\times N \to N$ is the projection. Since the result
  is compatible with the composition it is enough to consider the case of an
  embedding and a submersion separately.

  \sui (ii) We assume that $f$ is a submersion. Let $x \in M$ and set
  $y=f(x)$. Then
$$
\SSo(\opb{f}G)\cap T^*_xM
= \bigcap_{F'} \SSi(F')\cap T^*_xM 
\subset \bigcap_{G'} \SSi(\opb{f}G')\cap T^*_xM ,
$$
where $F'$ runs over the objects of $\Derb(\Cor_M)$ such that
$F' \simeq \opb{f}G$ in $\Orb(\cor_M)$ and $G'$ over the objects of
$\Derb(\Cor_N)$ such that $G' \simeq G$ in $\Orb(\cor_N)$.  Now the result
follows from Theorem~\ref{thm:iminv} and the fact that
${}^tf'_x \colon T^*_yN \to T^*_xM$ is injective.

\sui (iii) We assume that $f$ is an embedding.  Let $(x_0;\xi_0) \in T^*M$ be
such that $(x_0;\xi_0) \not\in f_d(\SSo(G)\cap T^*_{x_0}N)$.  Let $l$ be the
half line $\rpos \cdot(x_0;\xi_0)$. Since $f$ is non-characteristic for
$\SSo(G)$, we have $\opb{f_d}(l) \cap \SSo(G) \subset \{x_0\}$.  By
Proposition~\ref{prop:repr-bonmicsup} there exist a neighborhood $V$ of $x_0$
and $G'\in \Derb(\Cor_V)$ such that $G' \simeq G$ in $\Orb(\cor_V)$ and
$\opb{f_d}(l) \cap \SSi(G') \subset \{x_0\}$. Then $\opb{f}G' \simeq \opb{f}G$
in $\Orb(\cor_{M\cap V})$ and
$(x_0;\xi_0) \not\in f_d(\SSi(G')\cap T^*_{x_0}N)$.  This proves
$(x_0;\xi_0) \not\in \SSo(\opb{f}G)$, hence the inclusion of the proposition.
\end{proof}

\begin{proposition}\label{prop:im-dir-SSorb}
Let $f\cl M \to N$ be a morphism of manifolds.  Let $F\in \Orb(\cor_M)$.  We
assume that $f$ is proper on $\supporb(F)$. Then $\SSo(\reim{f}F)\subset
f_\pi\opb{f_d}\SSo(F)$.
\end{proposition}
\begin{proof}
  (i) As in the proof of Proposition~\ref{prop:im-inv-SSorb} we can reduce the
  problem to the cases where $f$ is an embedding or a projection. The case of an
  embedding is similar to the part~(ii) of the proof of
  Proposition~\ref{prop:im-inv-SSorb}.  

  \sui (ii) We assume now that $M=M'\times N$ and $f$ is the projection.  Let
  $q=(y;\eta) \in T^*N$ be such that $q\not\in f_\pi\opb{f_d}\SSo(F)$.  Let us
  prove that $q\not\in \SSo(\reim{f}F)$.  Since $f$ is proper on $\supporb(F)$,
  we can assume, up to restricting to a neighborhood of $y$, that
  $\supporb(F) \subset C \times N$, for some compact set $C\subset M'$.  We can
  then find an open convex cone $A$ in $T^*_yN$ containing $\eta$ and a
  neighborhood $\Omega$ of $C$, such that, for any $x\in \Omega$ there exists
  $G \in \Derb(\cor_M)$ representing $F$ with
  $\SSi(G) \cap (T^*_{x}M' \times A) = \emptyset$.

  As in the proof of Lemma~\ref{lem:repr-bonmicsup} we apply
  Proposition~\ref{prop:microcutofflevrai1} to obtain two neighborhoods
  $W \subset U$ of $y$ in $N$ and a functor
  $R^+ \colon \Derb(\Cor_{M' \times U}) \to \Derb(\Cor_{M' \times W})$, of the
  form $R^+(H) = K \circ H$, where $K\in \Derb(\cor_{W\times U})$ is given by
  the proposition, which satisfies
  \begin{itemize}
  \item [(a)] $R^+$ induces a functor on $\Orb(\cor_{M' \times U})$,
  \item [(b)] for $G \in \Derb(\Cor_{M' \times U})$ and $x\in M'$ such that
    $\SSi(G) \cap (T^*_{x}M' \times A) = \emptyset$ we have $R^+(G) \simeq G|_W$
    around $\{x\} \times W$,
  \item [(c)]for any $G \in \Derb(\Cor_{M' \times U})$ we have
    $\SSi(R^+(G)) \cap (T^*_{M'}M' \times \{p\}) = \emptyset$.
  \end{itemize}
  Then~(a) and~(b) imply $R^+(F) \simeq F|_W$.  Now we choose a representative
  $F'\in \Derb(\cor_M)$ of $F$.  Then $R^+(F')$ is another representative of $F$
  and we deduce the result from the condition~(c) and
  Proposition~\ref{prop:oim}.
\end{proof}

\begin{proposition}\label{prop:SSo-tens-hom}
Let $F,G\in\Orb(\cor_M)$. 
\begin{itemize}
\item [(i)] We assume that $\SSo(F)\cap\SSo(G)^a\subset T^*_MM$. Then \\
  $\SSo(F\epstens_{\cor_M} G)\subset \SSo(F)+\SSo(G)$.
\item [(ii)] We assume that $\SSo(F)\cap\SSo(G)\subset T^*_MM$. Then \\
  $\SSo(\rhomeps(F,G))\subset \SSo(F)^a+\SSo(G)$. 
\end{itemize}
\end{proposition}
\begin{proof}
  Let us prove~(i). Let $x_0 \in M$ and let $A , B \subset T^*_{x_0}M$ be conic
  neighborhoods of $\SSo(F)\cap T^*_{x_0}M$, $\SSo(G)\cap T^*_{x_0}M$ such that
  $A\cap B^a \subset \{x_0\}$.  By Proposition~\ref{prop:repr-bonmicsup} we can
  find representatives $F', G' \in \Derb(\Cor_M)$ of $F,G$ such that
  $\SSi(F')\cap T^*_{x_0}M \subset A$, $\SSi(G')\cap T^*_{x_0}M \subset B$.
  Since microsupports are closed we have $\SSi(F')\cap\SSi(G')^a\subset T^*_UU$
  for some neighborhood $U$ of $x_0$.  Then Theorem~\ref{thm:SSrhom} gives
  $\SSi(F'\epstens_{\cor_M} G') \cap T^*_{x_0}M \subset A+B$.  Since $A$ and $B$
  are arbitrarily close to our microsupports we deduce~(i). The proof of~(ii) is
  the same.
\end{proof}

\subsection{Microsupport in the zero section}
In Proposition~\ref{prop:SSorb-sectnulle} we give a special case of
Proposition~\ref{prop:iminvproj} for $\SSo$.

\begin{lemma}\label{lem:recouvrt-cubes}
Let $C=[a,b]^d$ be a compact cube in $\R^d$ and let $\{U_i\}_{i\in I}$ be a
family of open subsets of $\R^d$ such that $C\subset \bigcup_{i\in I}
U_i$. Then there exists a finite family of open subsets $\{V_n\}$,
$n=1,\ldots,N$, such that
\begin{itemize}
\item [(i)] for each $n=1,\ldots,N$ there exists $i\in I$ such that $V_n
  \subset U_i$,
\item [(ii)] $C\subset \bigcup_{n=1}^N V_n$,
\item [(iii)] $(\bigcup_{k=1}^{n} V_k) \cap V_{n+1}$ is contractible, for each
  $n=1,\ldots,N-1$.
\end{itemize}
\end{lemma}
\begin{proof}
  For $x\in \R^d$ and $\varepsilon>0$ we set  
  $C_x^\varepsilon = x+\mo]-\varepsilon,\varepsilon[^d$.  We can choose
  $\varepsilon>0$ such that, for any $x\in C$, there exists $i\in I$ satisfying
  $C_x^\varepsilon \subset U_i$ ($\varepsilon$ is a {\em Lebesgue number} of the
  covering).  We let $x_n$, $n=1,\ldots,N$, be the points of the lattice
  $C\cap (\varepsilon \Z)^d$ ordered by the lexicographic order of their coordinates.
  Then the family $V_n = C_{x_n}^\varepsilon$, $n=1,\ldots,N$, satisfies the required
  properties.
\end{proof}

\begin{lemma}\label{lem:const_ouv_contr}
  Let $M$ be a manifold and $U$, $V \subset M$ two open subsets. Let
  $F\in \Orb(\cor_M)$.  We assume that $U\cap V$ is contractible
  and that there exist $A,A' \in \Mod(\cor)$ such that $F|_U \simeq A_U$ and
  $F|_V \simeq A'_V$.  Then $A\simeq A'$ and $F|_{U\cup V} \simeq A_{U\cup V}$.
\end{lemma}
\begin{proof}
  We set $W = U \cap V$ and $X = U \cup V$.  Taking the stalks at some $x\in W$
  gives $A \simeq A'$. The Mayer-Vietoris triangle yields in our case
  $A_W \to[u] A_U \oplus A_V \to F_X \to[+1]$, where $u$ is of the form $(1,v)$
  for some isomorphism $v \in \Hom_{\Orb(\cor_M)}(A_W,A_W)$.  By
  Corollary~\ref{cor:morph-QRF-QRG} we have
  \begin{align*}
    \Hom_{\Orb(\cor_M)}(A_V, A_V)
    &\simeq \bigoplus_{i\in\Z} \Hom_{\Derb(\cor_M)}(A_V, A_V[i])  \\
    &\simeq \Hom(A,A) \otimes \bigoplus_{i\in\N} H^i(V;\cor) .
  \end{align*}
  In the same way $\Hom_{\Orb(\cor_M)}(A_W, A_W) \simeq \Hom(A,A)$ since $W$ is
  contractible.  Hence $v$ can be extended as an isomorphism to $V$ and we can
  define the commutative square $(C)$ below:
\begin{equation*}
\newcommand{\fleche}{\left(\begin{smallmatrix}
    1 & 0 \\ 0 & v^{-1} \end{smallmatrix} \right)}
\begin{tikzcd}[column sep=1.5cm]
A_W \ar[r, "{(1,v)}"] \ar[d, equal] \ar[dr, phantom, "(C)" description]
& A_U \oplus A_V \ar[r] \ar[d, "\fleche", "\wr"']
& F_X \ar[r, "+1"]  \ar[d, dashed] &  {} \\
A_W \ar[r, "{(1,1)}"']
& A_U \oplus A_V \ar[r]
& A_X \ar[r, "+1"] & \pointdiag
\end{tikzcd}
\end{equation*}
We extend this square to an isomorphism of triangles and obtain
$F_X \simeq A_X$, as required.
\end{proof}

\begin{proposition}\label{prop:SSorb-sectnulle}
  Let $E=\R^d$ and $F\in \Orb(\cor_E)$.  We assume that
  $\SSo(F) \subset T^*_EE$. Then there exists $A\in \Mod(\cor)$ such that
  $F\simeq A_E$.
\end{proposition}
\begin{proof}
  (i) By Proposition~\ref{prop:repr-bonmicsup}, for any $x\in E$ there exists a
  representative of $F$, say $F^x$, in $\Derb(\Cor_E)$ such that
  $\SSi(F^x) \cap T^*_xE \subset \{x\}$.  Since microsupports are closed there
  exists an open neighborhood of $x$, say $U_x$, such that
  $\SSi(F^x) \cap \dT^*U_x = \emptyset$, that is, $F^x|_{U_x}$ is constant. In
  other words, there exists $B^x\in \Derb(\Cor)$ such that
  $F^x|_{U_x} \simeq B^x_{U_x}$.  By Corollary~\ref{cor:iota_equivcat} there
  exists $A_x\in \Mod(\cor)$ such that $B^x \simeq A^x$ in $\Orb(\cor)$ and we
  have $F|_{U_x} \simeq A^x_{U_x}$ in $\Orb(\cor_{U_x})$.

  \sui(ii) We set $I_n = \mo]-n,n[^d$ for $n\in \N\setminus\{0\}$.  The family
  $\{U_x\}_{x\in \ol{I_n}}$ covers $\ol{I_n}$. Hence, by Lemmas~\ref{lem:recouvrt-cubes}
  and~\ref{lem:const_ouv_contr} there exists $A\in \Mod(\cor)$ such that $F_{I_n} \simeq
  A_{I_n}$ for all $n \in \N\setminus\{0\}$.    We can assume that these isomorphisms are
  compatible with the morphisms $i_n\cl (\cdot)_{I_n} \to (\cdot)_{I_{n+1}}$.  We obtain
  the commutative square (D) below between triangles deduced as
  in~\eqref{eq:lim_crois_ouverts}
\begin{equation*}
\begin{tikzcd}
  \oplus_n F_{I_n} \ar[r, "u"] \ar[d, "\wr"] \ar[dr,  phantom, "(D)" description]
  & \oplus_n F_{I_n} \ar[r]\ar[d, "\wr"]
 & F \ar[r, "+1"]\ar[d, dashed] & {}\\
 \oplus_n A_{I_n} \ar[r, "u"] & \oplus_n A_{I_n} \ar[r] & A_E \ar[r, "+1"] &,
\end{tikzcd}
\end{equation*}
where the $n^{th}$-component of $u$ is $\id - i_n$. We extend this
square to an isomorphism of triangles and we see that $F \simeq A_E$.
\end{proof}

We deduce a version of the Corollary~\ref{cor:Morse} for the orbit category.  We
state a particular case on the real line (the only case we will use), but the
statement of Corollary~\ref{cor:Morse} follows by applying this case to
$\roim{\phi}F$ and using the triangle
$\cor_{]-\infty,a[} \to \cor_{]-\infty,b[} \to \cor_{[a,b[} \to[+1]$.
\begin{corollary}\label{cor:Morse-orb}
  Let $F\in \Orb(\cor_\R)$ and let $a<b$ be given.  We assume that
  $\SSo(F) \cap \pi_\R^{-1}([a,b[)$ is contained in
  $T^*_{\tau\leq 0}\R = \{(t;\tau) \in T^*_\R\R$; $\tau\leq0\}$. Then
  $\RHom^\varepsilon(\cor_{[a,b[},F) \simeq 0$.
\end{corollary}
\begin{proof}
  We set
  $G = \roim{s}\rhomeps(\opb{q_2} \cor_{]0,\infty[}, \epb{q_1}
  \rhomeps(\cor_{[a,b[},F))$, where $q_1,q_2,s\colon \R^2 \to \R$ are the two
  projections and the sum.  Using the bounds in Propositions~\ref{prop:im-inv-SSorb},
  \ref{prop:SSo-tens-hom} and~\ref{prop:im-dir-SSorb}) we obtain
  \begin{align*}
    \SSo(\rhomeps(\cor_{[a,b[},F))
    &\subset \{(t;\tau);\;  a\leq t< b,\, \tau\leq0 \} \cup T^*_b\R, \\
  \SSo(G) &\subset \{(t;0);\; t\not=b\} \cup \{(b;\tau);\; \tau\geq0 \}.
  \end{align*}
  We deduce that $G|_{]-\infty,b[}$ is constant by
  Proposition~\ref{prop:SSorb-sectnulle}.  Let $i_t\colon \R \to \R^2$,
  $t' \mapsto (t',t-t')$ be the inclusion of $s^{-1}(t)$. Then $q_1 \circ i_t = \id$
  and $q_2 \circ i_t$ is the reflexion along $\frac t2$. Hence
  Proposition~\ref{prop:formulaire}-(h-j) gives
  \begin{align*}
    \rsect_{\{t\}}G
    &\simeq \RHom^\varepsilon(\cor_{]-\infty,t[},  \rhomeps(\cor_{[a,b[},F)) \\
    &\simeq \RHom^\varepsilon( \cor_{]-\infty,t[}\epstens \cor_{[a,b[},F) \\
    &\simeq \RHom^\varepsilon( \cor_{[a,c[}, F) ,
  \end{align*}
  where $c = \min\{ b, t\}$.  In particular $\rsect_{\{t\}}G \simeq 0$ for $t<a$ and
  hence $G|_{]-\infty,b[}$ vanishes.  We also obtain
  $\rsect_{\{b\}}G \simeq \RHom^\varepsilon(\cor_{[a,b[},F)$.

  We can represent $G$ by $\tilde G \in \Der(\Cor_\R)$ with
  $(b;-1) \not \in \SSi(\tilde G)$.  The definition of the microsupport implies
  $(\rsect_{]-\infty,b[}\tilde G)_b[-1] \isoto (\rsect_{\{b\}}\tilde G)_b$.  The same
  isomorphism holds in $\Orb(\cor_\R)$ and we deduce $\rsect_{\{b\}}G \simeq 0$
  because $G|_{]-\infty,b[}$ vanishes.
\end{proof}

\part{The Kashiwara-Schapira stack}
\label{chap:KSstack}

Let $M$ be a manifold and $\Lambda$ a locally closed conic Lagrangian submanifold of
$\dT^*M$.  In this part we define the Kashiwara-Schapira stack $\kss(\cor_\Lambda)$
and its orbit category version.  It is obtained by quotienting the category of
sheaves on $M$ by subcategories defined by microsupport conditions.  These categories
are introduced by Kashiwara-Schapira in~\cite{KS90}.  For a conic subset $S$ of
$T^*M$ we denote by $\Derb(\cor_M;S)$ the quotient of $\Derb(\cor_M)$ by the
subcategory of sheaves $F$ with $\SSi(F) \cap S = \emptyset$.   Then $\kss(\cor_\Lambda)$ is the stack on $\Lambda$
associated with the prestack whose value over $\Lambda_0 \subset \Lambda$ is the
subcategory of $\Derb(\cor_M;\Lambda_0)$ generated by the $F$ such that there exists
a neighborhood $\Omega$ of $\Lambda_0$ in $T^*M$ with
$\SSi(F) \cap \Omega \subset \Lambda_0$.  We remark that we lose the triangulated
structure in the stackification process.

An important result of~\cite{KS90} says that the $\hom$ sheaf in
$\kss(\cor_\Lambda)$ is given by $H^0\muhom$ (see
Corollary~\ref{cor:defbiskss}).  It is then possible to describe an object of
$\kss(\cor_\Lambda)$ by local data (see Remark~\ref{rem:description_KSstack}).
We check in Lemma~\ref{lem:simple_local_base} that, locally on $M$, there exist
simple sheaves with microsupport $\Lambda$: if $B$ is a small enough ball in $M$
and $\Lambda$ is in generic position, there exist simple sheaves on $B$ with
microsupport $\Lambda \cap T^*B$.  We define two classes
$\mu^{sh}_1(\Lambda) \in H^1(\Lambda; \Z)$ and
$\mu^{sh}_2(\Lambda) \in H^2(\Lambda; \cor^\times)$ which are obstructions for
the existence of a global object of $\kss(\cor_\Lambda)$ (that is, an object of
$\kss(\cor_\Lambda)(\Lambda)$).  In the last two sections we make the link
between these classes and the usual Maslov class $\mu_1(\Lambda)$ and another
obstruction class $\mu^{gf}_2(\Lambda)$ (an obstruction class to trivialize the
Gauss map of $\Lambda$).  In fact it would not be difficult to deduce
$\mu^{sh}_1(\Lambda) = \mu_1(\Lambda)$ directly from
Proposition~\ref{prop:inv_microgerm} and the description of the Maslov class by
\v Cech cohomology (for example in~\cite{GSt77}).  However the class
$\mu^{sh}_2(\Lambda)$ requires more work; we only prove the useful implication
that the vanishing of $\mu^{sh}_2(\Lambda)$ implies the vanishing of
$\mu^{gf}_2(\Lambda)$.

When we work with sheaves of $\cor$-modules we only have the obstruction classes
$\mu^{sh}_1(\Lambda)$ and $\mu^{sh}_2(\Lambda)$ for the existence of a global section
of $\kss(\cor_\Lambda)$.  If we work with sheaves of spectra, there are infinitely
many of them.  Of course this requires to work with dg-categories or infinity
categories (see~\cite{To07, To08} or~\cite{Lu09}). In this framework the stack
$\kss(\cor_\Lambda)$ has the structure of a stable category, like a category of
sheaves (whereas the triangulated structures are not suited for stackification).
This is explained in~\cite{Jin19} where the higher classes $\mu^{sh}_i(\Lambda)$ are
described (they do not coincide with the obstruction classes $\mu^{gf}_i(\Lambda)$ --
see also~\cite{JT17}).    In~\cite{AK16} it is
proved that these classes vanish when $\Lambda$ is a Lagrangian embedding in $T^*M$
in the homotopy class of the base.

\begin{notation}\label{not:micro_categories}  
  For a conic subset $S$ of $T^*M$ we recall the categories $\Der_{S}(\cor_M)$,
  $\Der_{[S]}(\cor_M)$ and $\Der_{(S)}(\cor_M)$ of
  Notation~\ref{not:cat_micsupp_fixe}.  We also denote by $\Der(\cor_M;S)$ the
  quotient of $\Der(\cor_M)$ by $\Der_{T^*M\setminus S}(\cor_M)$ (see the reminder on
  localization in section~\ref{sec:def_orb_cat}).

  We use similar notations for the (locally) bounded derived categories:
  $\Der^*_{S}(\cor_M)$, $\Der^*_{[S]}(\cor_M)$, $\Der^*_{(S)}(\cor_M)$ and
  $\Der^*(\cor_M;S)$, where $*=\mathrm{b}$ or $*=\mathrm{lb}$.  We also define
  $\OrbL{S}(\cor_M)$, $\OrbL{[S]}(\cor_M)$, $\OrbL{(S)}(\cor_M)$ and $\Orb(\cor_M;S)$
  in the same way, replacing $\Der$ by $\Orb$ and $\SSi$ by $\SSo$.
\end{notation}

As   recalled in section~\ref{sec:def_orb_cat} the
objects of $\Derb(\cor_M;S)$ are those of $\Derb(\cor_M)$.  A morphism $u\cl F\to G$
in $\Derb(\cor_M;S)$ is represented by a triple $(F',s,u')$ where
$F'\in \Derb(\cor_M)$ and $s$, $u'$ are morphisms
\begin{equation}
\label{eq:repr_morphism_qoutient}
F\from[s] F' \to[u'] G
\end{equation}
such that the $L$ defined (up to isomorphism) by the distinguished triangle
$F'\to[s] F \to L \to[+1]$ satisfies $S \cap \SSi(L) = \emptyset$.  Two such triples
$(F'_i,s_i,u'_i)$, $i=1,2$, represent the same morphism if there exists a third
triple $(F',s,u')$ and two morphisms $v_i \cl F' \to F'_i$ such that
$s = s_i \circ v_i$, $i=1,2$, and $u'_1 \circ v_1 = u'_2 \circ v_2$.

The notion of stack used here is that of ``sheaf of categories''.  We refer for
example to~\cite[\S 19]{KS06}.  A prestack $\catc$ on a topological space $X$
consists of the data of a category $\catc(U)$, for each open subset $U$ of $X$,
restriction functors $r_{V,U} \cl \catc(U) \to \catc(V)$, for $V\subset U$, and
isomorphisms of functors $r_{W,V} \circ r_{V,U} \simeq r_{W,U}$, for
$W\subset V\subset U$, satisfying compatibility conditions.

A stack is a prestack satisfying some gluing conditions.  In particular, if
$\catc$ is a stack and $A,B \in \catc(U)$, then the presheaf
$V \mapsto \Hom_{\catc(V)}(A|_V,B|_V)$ is a sheaf on $U$.  Moreover, if
$U = \bigcup_{i\in I} U_i$ and $A_i \in \catc(U_i)$ are given objects with
compatible isomorphisms between their restrictions on the intersections
$U_i \cap U_j$, then these objects glue into an object of $\catc(U)$.

For any given prestack we can construct its associated stack, similar to the
associated sheaf of a presheaf.

\section{Definition of the Kashiwara-Schapira stack}
\label{sec:KSstack}

We use the categories associated with a subset of $T^*M$ introduced in
Notation~\ref{not:micro_categories}.
\begin{definition}
\label{def:KSstack}
Let $\Lambda \subset T^*M$ be a locally closed conic subset.
We define a prestack $\kss^0_\Lambda$ on $\Lambda$ as follows.
Over an open subset $\Lambda_0$ of $\Lambda$ the objects of
$\kss^0_\Lambda(\Lambda_0)$ are those of $\Derb_{(\Lambda_0)}(\cor_M)$.
For $F,G \in \kss^0_\Lambda(\Lambda_0)$ we set
$$
\Hom_{\kss^0_\Lambda(\Lambda_0)}(F,G) \eqdot
\Hom_{\Derb(\cor_M;\Lambda_0)}(F,G) .
$$
We define the Kashiwara-Schapira stack of $\Lambda$ as the stack associated with
$\kss^0_\Lambda$. We denote it by $\kss(\cor_\Lambda)$.  For
$\Lambda_0 \subset \Lambda$ we usually write abusively $\kss(\cor_{\Lambda_0})$
instead of $\kss(\cor_\Lambda)(\Lambda_0)$.

We denote by
$\kssfunc_\Lambda \cl \Derb_{(\Lambda)}(\cor_M) \to \kss(\cor_\Lambda)$ the
obvious functor. However, for $F \in \Derb_{(\Lambda)}(\cor_M)$, we often write
$F$ instead of $\kssfunc_\Lambda(F)$ if there is no risk of ambiguity.
\end{definition}

Several results in the next sections give links between $\kss(\cor_\Lambda)$
and stacks of the following type.

\begin{definition}
\label{def:Dloc}
Let $X$ be a topological space. We let $\Dloc^0(\cor_X)$ be the subprestack of
$U \mapsto \Derb(\cor_U)$, $U$ open in $X$, formed by the $F\in \Derb(\cor_U)$
with locally constant cohomologically sheaves.  We let $\Dloc(\cor_X)$ be the
stack associated with $\Dloc^0(\cor_X)$.
We denote by $\loc(\cor_X)$ the substack of $\Mod(\cor_X)$ formed by the
locally constant sheaves.
\end{definition}
\begin{remark}
  \label{rem:Dloc}
We remark that $\Dloc(\cor_X)$ is only a stack of additive categories (the
triangulated structure is of course lost in the ``stackification'').  However
the cohomological functors $H^i \cl \Derb(\cor_U) \to \Mod(\cor_U)$ induce
functors of stacks $H^i \cl \Dloc(\cor_X) \to \loc(\cor_X)$ and the natural
embedding $\Mod(\cor_U) \hookrightarrow \Derb(\cor_U)$ induces $i \cl
\loc(\cor_X) \to \Dloc(\cor_X)$.  We have $H^0 \circ i \simeq
\id_{\loc(\cor_X)}$. Hence $i$ is faithful and $\loc(\cor_X)$ is a subcategory
of $\Dloc(\cor_X)$.
\end{remark}

\subsection*{Link with microlocalization}

We recall now some results of~\cite{KS90} which explain the link between the
localized categories $\Derb(\cor_M;S)$, for $S \subset T^*M$, and the
microlocalization, making these categories easier to describe.

First we recall some properties of the functor $\muhom$.  For $U \subset M$ and
$F,G \in \Der(\cor_U)$, the isomorphism
$\rhom(F,G) \simeq \roim{\pi_M}\muhom(F,G)$ of~\eqref{eq:proj_muhom_oim}
implies $\Hom(F,G) \isoto H^0(T^*U; \muhom(F,G))$; for $u\colon F\to G$ we
denote by
\begin{equation}
\label{eq:def_umu-bis}
u^\mu \in H^0(T^*U; \muhom(F,G))  
\end{equation}
the image of $u$ by this isomorphism (this notation has been introduced
in~\eqref{eq:def_umu}).  For $F,G,H \in \Derb(\cor_M)$ we have a composition
morphism (see~\cite[Cor.~4.4.10]{KS90})
\begin{equation}\label{eq:comp_muhom}
  \muhom(F,G) \ltens \muhom(G,H) \to \muhom(F,H) .
\end{equation}
(With the notations of Definition~\ref{def:muhom}, it is induced by a natural
morphism
$\mu_{\Delta_M}(A) \ltens \mu_{\Delta_M}(B) \to \mu_{\Delta_M}(A \ltens B)$ and
the composition morphism for $\rhom(\opb{q_2}F,\epb{q_1}G)$.)

\begin{notation}
\label{not:mucomposition}
For $F,G,H \in \Derb(\cor_M)$, $S \subset T^*M$ and sections of
$\muhom(-,-)$, say $a \in H^i(S; \muhom(F,G)|_S)$,
$b \in H^j(S; \muhom(G,H)|_S)$, we denote by
$$
b\mucirc a \in H^{i+j}(S; \muhom(F,H)|_S)
$$
the image of $a\tens b$ by the morphism induced by~\eqref{eq:comp_muhom} on
sections.
\end{notation}
By construction $\mucirc$ is compatible with the usual composition morphism for
$\rhom$ through the isomorphism~\eqref{eq:proj_muhom_oim}: for $u\colon F\to G$,
$v\colon G \to H$ we have $(v\circ u)^\mu = v^\mu \mucirc u^\mu$.

In the same way $\mucirc$ is also compatible with the functoriality of $\muhom$
as follows. A morphism $v\colon G \to H$ induces 
\begin{equation}
\label{eq:fonctorialitemuhom}
\muhom(F,v) \colon \muhom(F,G) \to \muhom(F,H)  
\end{equation}
and we have, for a section $a$ of $\muhom(F,G)$,
\begin{equation}
\label{eq:mucirc_fonct}
\muhom(F,v)(a) =  v^\mu \mucirc a .
\end{equation}

\smallskip

We come back to $\Derb(\cor_M;S)$ for $S\subset T^*M$.  Let $u\cl F\to G$ be a
morphism in $\Derb(\cor_M;S)$, represented by a triple $(F',s,u')$ as
in~\eqref{eq:repr_morphism_qoutient}.  Let $L$ be defined (up to isomorphism) by the
distinguished triangle $F'\to[s] F \to L \to[+1]$. Then $L$ satisfies
$S \cap \SSi(L) = \emptyset$.  By~\eqref{eq:suppmuhom} we have
$\muhom(L,G)|_S \simeq 0$ and $\muhom(s,G)$ as in~\eqref{eq:fonctorialitemuhom} gives
an isomorphism
$$
\muhom(s,G) \colon \muhom(F,G)|_S \isoto \muhom(F',G) |_S .
$$
Hence the triple $(F',s,u')$ yields a section $(\muhom(s,G))^{-1}((u')^\mu)$ of
$H^0(S;\muhom(F,G)|_S)$.  We can check that this section only depends on $u$
and not its representative $(F',s,u')$.  Thus we obtain a well-defined morphism
\begin{equation}\label{eq:HomDkMOmega_vers_muhom}
\Hom_{\Derb(\cor_M;S)}(F,G) \to H^0(S;\muhom(F,G)|_S) .
\end{equation}
\begin{theorem}[Thm.~6.1.2 of~\cite{KS90}]
\label{thm:HomDkMp=muhomp}
If $S = \rspos \cdot p$ for some $p\in T^*M$,
then~\eqref{eq:HomDkMOmega_vers_muhom} is an isomorphism.
\end{theorem}
When $p\in M \simeq T^*_MM$ we have $S = \{p\}$ and the statement has the
following meaning: we remark that $\Derb(\cor_M;\{p\})$ is the localization of
$\Derb(\cor_M)$ by the subcategory of sheaves which are zero in some neighborhood of
$p$.  We can then see that $\Hom_{\Derb(\cor_M;\{p\})}(F,G)$ identifies with
$H^0(\rhom(F,G))_p$.  We also have $\muhom(F,G)|_{T^*_MM} \simeq \rhom(F,G)$
(use~\eqref{eq:proj_muhom_oim} and the fact that $\muhom$ is conic), hence
$(\muhom(F,G))_p \simeq (\rhom(F,G))_p$.

Let $F,G \in \Derb(\cor_M)$ be given.  It follows from
Theorem~\eqref{thm:HomDkMp=muhomp} that the sheaf associated with the presheaf
$\Omega \mapsto \Hom_{\Derb(\cor_M;\rspos\cdot \Omega)}(F,G)$, where $\Omega$ runs
over the open subsets of $T^*M$, is $H^0\muhom(F,G)$ (here $\rspos\cdot \Omega$ is
the image of $\rspos \times \Omega \to T^*M$, $(a,p) \mapsto a\cdot p$). We obtain an
alternative definition of $\kss(\cor_\Lambda)$:
\begin{corollary}\label{cor:defbiskss}
Let $\Lambda \subset T^*M$ be as in Definition~\ref{def:KSstack}.
We define a prestack $\kss^1_\Lambda$ on $\Lambda$ as follows.
Over an open subset $\Lambda_0$ of $\Lambda$ the objects of
$\kss^1_\Lambda(\Lambda_0)$ are those of $\Derb_{(\Lambda_0)}(\cor_M)$.
For $F,G \in \kss^1_\Lambda(\Lambda_0)$ we set
$\Hom_{\kss^1_\Lambda(\Lambda_0)}(F,G)
\eqdot H^0(\Lambda_0;\muhom(F,G)|_{\Lambda_0})$.
The composition is induced by~\eqref{eq:comp_muhom}.
Then, the natural functor of prestacks $\kss^0_\Lambda \to \kss^1_\Lambda$
induces an isomorphism on the associated stacks.
\end{corollary}

\begin{remark}\label{rem:description_KSstack}
By Corollary~\ref{cor:defbiskss} an object of $\kss(\cor_\Lambda)$ is determined
by the data of an open covering $\{\Lambda_i\}_{i\in I}$ of $\Lambda$, objects
$F_i \in \Derb_{(\Lambda_i)}(\cor_M)$, for any $i\in I$, and sections $u_{ji}\in
H^0(\Lambda_{ij};\muhom(F_i,F_j)|_{\Lambda_{ij}})$, for any $i,j\in I$, such that
\begin{itemize}
\item [(i)] $u_{ii}$ is induced by $\id_{F_i}$, for any $i\in I$,
\item [(ii)] $u_{kj} \mucirc u_{ji} = u_{ki}$, for any $i,j,k \in I$.
\end{itemize}
For a complex of sheaves $A$, the sheaf associated with the presheaf
$U \mapsto H^0(U;A)$ is $H^0A$.  Hence, for $F,G \in \Derb_{(\Lambda)}(\cor_M)$,
we find that the homomorphism sheaf
$\hom_{\kss(\cor_\Lambda)}(\kssfunc_\Lambda(F), \kssfunc_\Lambda(G))$ is
$H^0(\muhom(F,G))|_\Lambda$.  In particular, for an open subset
$\Lambda_0 \subset \Lambda$,
\begin{equation}\label{eq:hom_sheaf_kss}
\Hom(\kssfunc_\Lambda(F)|_{\Lambda_0}, \kssfunc_\Lambda(G)|_{\Lambda_0})
\simeq H^0(\Lambda_0; H^0(\muhom(F,G))) .
\end{equation}
\end{remark}

\begin{remark}\label{rem:description_KSstack2}
  We will only consider $\kss(\cor_\Lambda)$ when $\Lambda$ is conic Lagrangian
  submanifold of $\dT^*M$.  In this case, for $F,G\in \Derb_{(\Lambda)}(\cor_M)$
  we know by Corollary~\ref{cor:muhom_loccst} that $\muhom(F,G)$ has locally
  constant cohomology sheaves on $\Lambda$.  Moreover, for a given
  $p = (x;\xi) \in \Lambda$ we have
\begin{equation*}
\muhom(F,G)_p \simeq
\RHom( (\rsect_{\{ \varphi_0 \geq 0\}}(F))_x ,
(\rsect_{\{ \varphi_0 \geq 0\}}(G))_x ),
\end{equation*}
where $\varphi_0$ is such that $\Lambda$ and $\Lambda_{\varphi_0}$ intersect
transversely at $p$ (see~\eqref{eq:stalk_muhom}).  Hence, if $F$ and $G$ are
pure along $\Lambda$, $\muhom(F,G)$ is concentrated in one degree.  If this
degree is $0$, the right hand side of~\eqref{eq:hom_sheaf_kss} coincides with
$H^0(\Lambda_0; \muhom(F,G))$.  If $F$ and $G$ are simple, then $\muhom(F,G)$ is
moreover of rank one.
\end{remark}

\begin{remark}
\label{rem:iso_sur_Omega}
For a conic subset $\Omega$ of $T^*M$, we recall that a morphism $a\cl F \to G$
in $\Der(\cor_M)$ is an {\em isomorphism on $\Omega$} if
$\SSi(C(a)) \cap \Omega = \emptyset$, where $C(a)$ is given by the distinguished
triangle $F\to[a] G \to C(a) \to[+1]$.  In particular $a$ induces an isomorphism
in $\Derb(\cor_M; \Omega)$.  We assume that $\Lambda$ is a conic Lagrangian
submanifold and $F,G\in \Derb_{(\Lambda)}(\cor_M)$. Let $\Lambda_0$ be an open
subset of $\Lambda$ and let $a\cl F \to G$ be an isomorphism on $\Lambda_0$.
Then the morphism
$\kssfunc_\Lambda(a)|_{\Lambda_0} \colon \kssfunc_\Lambda(F)|_{\Lambda_0} \to
\kssfunc_\Lambda(G)|_{\Lambda_0}$ is an isomorphism. It follows from
Remark~\ref{rem:description_KSstack2} that there exists
$b \in H^0(\Lambda_0; H^0(\muhom(G,F)))$ such that
$$
a^\mu \mucirc b = \id_G^\mu , \qquad b \mucirc a^\mu =  \id_F^\mu.
$$
\end{remark}

  The functor
$\kss^{0,op}_\Lambda \times \kss^0_\Lambda \to \Derb(\cor_\Lambda)$,
$(F,G) \mapsto \muhom(F,G)$ induces a functor of stacks
\begin{equation}
  \label{eq:def_muhombar}
  \ol{\muhom} \cl \kss(\cor_\Lambda)^{op} \times \kss(\cor_\Lambda)
  \to \Dloc(\cor_\Lambda) .
\end{equation}

\section{Simple sheaves}\label{sec:simpsheaves}

In this section we assume that $\Lambda$ is a locally closed conic Lagrangian
submanifold of $\dT^*M$.  We have seen the notion of pure and simple sheaves
along $\Lambda$ in Section~\ref{sec:simplesheaves}. We give here some additional
properties.

It is easy to describe the simple sheaves along a Lagrangian submanifold at a
generic point. They are given in the following example.

\begin{example}
\label{ex:simple_sheaf}
We consider the hypersurface $S= \R^{n-1}\times \{0\}$ in $M=\R^n$. We let
$\Lambda = \{(\ul x, 0; 0, \xi_n); \xi_n>0\}$ be the ``positive'' half part of
$T^*_SM$. We set $Z= \R^{n-1}\times \rpos$.  The sheaf $\cor_Z$ is simple along
$\Lambda$. More generally, by Example~\ref{ex:SS=conormal_hypersurface}, the
simple sheaves $F$ along $\Lambda$ fit in a distinguished triangle
$$
E'_M \to \cor_Z[i] \to F \to[+1] ,
$$
for some integer $i$ and some $E' \in \Der(\cor)$.  Let
$F\in \Derb_{[\Lambda]}(\cor_M)$. Then, there exists
$L\in \Derb(\cor)$ such that the image of $F$ in the quotient category
$\Derb(\cor_M; \dT^*M)$ is isomorphic to $L_Z = L_M \otimes \cor_Z$.  The pure
sheaves correspond to the case where $L$ is concentrated in one degree and free.
The simple sheaves correspond to the case where $L\simeq \cor[i]$ for some
degree $i\in \Z$.
\end{example}

For any $p\in \Lambda$ we can find a homogeneous Hamiltonian isotopy that sends
a neighborhood of $p$ in $\Lambda$ to the conormal bundle of a smooth
hypersurface.  Then Theorem~\ref{thm:GKS} reduces the general case to
Example~\ref{ex:simple_sheaf} (since this is a local statement in $T^*M$ we can
also use Theorem~7.2.1 of~\cite{KS90}). We deduce:

\begin{lemma}\label{lem:simple_local}
Let $p=(x;\xi)$ be a given point of $\Lambda$. Then there exists a neighborhood
$\Lambda_0$ of $p$ in $\Lambda$ such that
\begin{itemize}
\item [(i)] there exists $F\in \Derb_{(\Lambda_0)}(\cor_M)$ which is
simple along $\Lambda_0$,
\item [(ii)] for any $G\in \Derb_{(\Lambda_0)}(\cor_M)$ there exist a
  neighborhood $\Omega$ of $\Lambda_0$ in $T^*M$ and an isomorphism
  $F\ltens L_M \isoto G$ in $\Derb(\cor_M;\Omega)$, where $L\in \Derb(\cor)$ is
  given by $L = \muhom(F,G)_p$.
\end{itemize}
\end{lemma}
  The subset $\Lambda_0$ of the lemma is
contractible by construction if it is obtained from the $\Lambda$ of
Example~\ref{ex:simple_sheaf} by a homogeneous Hamiltonian isotopy.  Conversely if
the condition~(ii) of the lemma is satisfied, we must have $\pi_1(\Lambda_0) = 0$ and
$H^i(\Lambda_0;\cor) \simeq 0$ for all $i>0$.

\begin{definition}\label{def:KSstacksimple}
Let $\Lambda \subset \dT^*M$ be a locally closed conic Lagrangian submanifold.
We let $\kss^p(\cor_\Lambda)$ (resp. $\kss^s(\cor_\Lambda)$) be the substack of
$\kss(\cor_\Lambda)$ formed by the pure (resp. simple) sheaves along $\Lambda$.
\end{definition}

Lemma~\ref{lem:simple_local} implies the following result.

\begin{proposition}\label{prop:KSstack=Dloc}
  Let $\Lambda \subset \dT^*M$ be a locally closed conic Lagrangian submanifold.
  We assume that there exists a simple sheaf $F\in \kss^s(\cor_\Lambda)$.  Then
  the functor $\ol{\muhom}$ defined in~\eqref{eq:def_muhombar} induces an
  equivalence of stacks
\begin{equation*}
\ol{\muhom}(F,\cdot) \cl \kss(\cor_\Lambda) \isoto   \Dloc(\cor_\Lambda) ,
\qquad G \mapsto \ol{\muhom}(F,G).
\end{equation*}
\end{proposition}

By Lemma~\ref{lem:simple_local} we can find a simple sheaf with microsupport
$\Lambda$ locally around a given point $p\in \Lambda$. When $\Lambda$ is in a
good position we can improve this result as follows.

\begin{lemma}\label{lem:simple_local_base}
    Let $M$ be a manifold and let
  $\Lambda \subset \dT^*M$ be a locally closed conic Lagrangian submanifold such that
  the projection $\Lambda/\R_{>0} \to M$ has finite fibers. Let
  $p=(x;\xi) \in \Lambda$.  Then there exist a neighborhood $U$ of $x$ and
  $F\in \Derb(\cor_U)$ such that $\dot\SSi(F) = \Lambda\cap T^*U$ and $F$ is simple
  along $\Lambda\cap T^*U$.
\end{lemma}
\begin{proof}
(i) By hypotheses $\Lambda\cap T^*_xM$ consists of finitely many half-lines, say
$\rspos\cdot p_i$, with $p_i = (x;\xi_i)$, $i=1,\ldots,n$. Up to a restriction
to a neighborhood of $x$ we can assume that the $p_i$ belong to distinct
connected components of $\Lambda$, say $\Lambda_i$, $i=1,\ldots,n$. If $F_i$ is
simple along $\Lambda_i$, then the direct sum $\oplus_i F_i$ is simple along
$\Lambda$.  Hence we can assume that $\Lambda\cap T^*_xM = \rspos\cdot p$ for
some $p=(x;\xi)$.

\sui(ii) By Lemma~\ref{lem:simple_local} there exists a neighborhood $\Omega$ of $p$
in $T^*M$ and $F_0\in \Derb(\cor_M)$ such that  
$\SSi(F_0) \cap \Omega = \Lambda$ and $F_0$ is simple along $\Lambda$ at $p$.  For a
neighborhood $V$ of $x$ we choose a trivialization $T^*V \simeq V \times T^*_{x}M$.
Up to shrinking $V$ we can find two disjoint closed conic contractible subsets
$A, A' \subset \dT^*_{x}M$ such that $\dot\SSi(F_0) \subset V\times (A \sqcup A')$
and $\dot\SSi(F_0) \cap (V\times A) = \Lambda \cap \dT^*V$.  By
Proposition~\ref{prop:cut-off_split_local2} there exists a distinguished triangle
$F \oplus F_1 \to F_0|_U \to L \to[+1]$ in $\Der(\cor_U)$ on some smaller
neighborhood $U$ of $x$ such that $\dot\SSi(F) = \dot\SSi(F_0) \cap (U\times A)$,
$\dot\SSi(F_1) = \dot\SSi(F_0) \cap (U\times A')$ and $L$ is constant.  Then
$\dot\SSi(F) = \Lambda \cap T^*U$ and $F$ is simple.
\end{proof}

\section{Obstruction classes}
\label{sec:obstr_classes}

In this section we see that there are two obstructions to the existence of a
global simple object in $\kss(\cor_\Lambda)$ (by this we mean an object of
$\kss^s(\cor_\Lambda)(\Lambda)$).  They are classes
$\mu^{sh}_1(\Lambda) \in H^1(\Lambda; \Z)$ and
$\mu^{sh}_2(\Lambda) \in H^2(\Lambda; \cor^\times)$, where $\cor^\times$ is the
group of units in $\cor$ (when $\cor = \Z/2\Z$ this last class is automatically
zero).

On the other hand, it is proved in~\cite{Gi88} that $\Lambda$ has a local
generating function if and only if the stable Gauss map
$g\colon \Lambda \to U/O$ is homotopic to the constant map
$v\colon \Lambda \to U/O$ which sends a point to the vertical fiber of
$T^*M$. Here $U/O = \varinjlim_n U(n)/O(n)$ is the (stable) Lagrangian
Grassmannian (some choices are needed to define the Gauss map, but it is
well-defined up to homotopy).  The obstruction classes to find such a homotopy
are classes $\mu^{gf}_i(L) \in H^i(\Lambda; \pi_i(U/O))$ for $i=1,\dots, \dim M$
($gf$ stands for generating function).  The first class is the Maslov class
$\mu^{gf}_1(\Lambda) = \mu_1(\Lambda)$.  We will see in~\S\ref{sec:monodromy}
that, for $i=1,2$, the vanishing of $\mu^{sh}_i(\Lambda)$ implies the vanishing
of $\mu^{gf}_i(\Lambda)$.

To define the classes $\mu^{sh}_i(\Lambda)$ we recall how we can describe an
object of $\kss^s(\cor_\Lambda)(\Lambda)$.  By
Remark~\ref{rem:description_KSstack} a global simple object of
$\kss(\cor_\Lambda)$ is determined by the data of an open covering
$\{\Lambda_i\}_{i\in I}$ of $\Lambda$, objects
$F_i \in \Derb_{(\Lambda_i)}(\cor_M)$, for all $i\in I$, which are simple along
$\Lambda_i$, and sections
$u_{ji}\in H^0(\Lambda_{ij};\muhom(F_i,F_j)|_{\Lambda_{ij}})$, for any
$i,j\in I$, such that
\begin{itemize}
\item [(i)] $u_{ii}$ is induced by $\id_{F_i}$, for any $i\in I$,
\item [(ii)] $u_{kj} \mucirc u_{ji} = u_{ki}$, for any $i,j,k \in I$.
\end{itemize}
We try to find such a set of data.  First we choose a covering
$\{\Lambda_i\}_{i\in I}$ of $\Lambda$ by small open subsets. We assume that the
$\Lambda_i$'s and all intersections $\Lambda_{ij}$, $\Lambda_{ijk}$ are
contractible.  We have seen in the previous section that, if the $\Lambda_i$'s
are small enough, we can choose $F_i \in \Derb_{(\Lambda_i)}(\cor_M)$, for all
$i\in I$, which are simple along $\Lambda_i$.  Moreover, for each $i\in I$, if
$F'_i\in \Derb_{(\Lambda_i)}(\cor_M)$ is another simple sheaf along $\Lambda_i$
we have $\kssfunc_{\Lambda_i}(F'_i) \simeq \kssfunc_{\Lambda_i}(F_i)[d]$ for some
shift $d$.

The sheaf $\muhom(F_i,F_j)|_{\Lambda_{ij}}$ is constant on $\Lambda_{ij}$, free
of rank one.  Hence there exist isomorphisms
$\varphi_{ji} \colon \muhom(F_i,F_j)|_{\Lambda_{ij}} \simeq
\cor_{\Lambda_{ij}}[-d_{ij}]$, for some integers $d_{ij}$. They induce
$\varphi_{ji} \in H^0(\Lambda_{ij}; \muhom(F_i,F_j)[d_{ij}])$.  In view
of~\eqref{eq:hom_sheaf_kss} the $\varphi_{ji}$'s give isomorphisms in
$\kss^s(\cor_\Lambda)(\Lambda_{ij})$:
$$
v_{ji} \colon \kssfunc_{\Lambda_i}(F_i)|_{\Lambda_{ij}} \isoto
\kssfunc_{\Lambda_j}(F_j)|_{\Lambda_{ij}}[d_{ij}] .
$$
We deduce that the \v Cech cochain $\{d_{ij}\}_{i,j \in I}$ is a cocyle and
defines
\begin{equation}
  \label{eq:Maslov_premiere}
  \mu^{sh}_1(\Lambda) = [\{d_{ij}\}]\in H^1(\Lambda; \Z) .
\end{equation}
By the remark that $ \kssfunc_{\Lambda_i}(F_i)$ is well defined up to shift,
this class only depends on $\Lambda$.  If there exists a global simple object
$\shf$ in $\kss(\cor_\Lambda)$, for some coefficient ring $\cor$, then we can
choose $F_i$'s which represent $\shf|_{\Lambda_i}$ and this implies $d_{ij} = 0$
for all $i,j$ and thus $\mu^{sh}_1(\Lambda) = 0$.

\smallskip

Let us assume that $\mu^{sh}_1(\Lambda) = 0$.  Then we can write $d_{ij} = d_j - d_i$
for some family of integers $d_i$, $i\in I$.  We set $F'_i = F_i[d_i]$ and
obtain isomorphisms
$$
w_{ji} \colon \kssfunc_{\Lambda_i}(F'_i)|_{\Lambda_{ij}} \isoto
\kssfunc_{\Lambda_j}(F'_j)|_{\Lambda_{ij}} .
$$
For $i,j,k \in I$ we define an automorphism $c_{ijk}$ of
$\kssfunc_{\Lambda_i}(F'_i)|_{\Lambda_{ijk}}$ by
$c_{ijk} = w_{ik} \circ w_{kj} \circ w_{ij}$.  Since
$$
\hom(\kssfunc_{\Lambda_i}(F'_i), \kssfunc_{\Lambda_i}(F'_i))
\simeq H^0\muhom(F'_i, F'_i) \simeq \cor_{\Lambda_i} ,
$$
we can canonically identify the automorphisms group of
$\kssfunc_{\Lambda_i}(F'_i)|_{\Lambda_{ijk}}$ with
$\cor^\times \subset H^0(\Lambda_{ijk}; \cor_{\Lambda_{ijk}}) \simeq \cor$.
Hence the $c_{ijk}$'s give a \v Cech cochain with coefficient in $\cor^\times$.
It is easy to see that it is a cocycle and defines
\begin{equation}
  \label{eq:Maslov_deuxieme}
  \mu^{sh}_2(\Lambda) = [\{c_{ijk}\}]\in H^2(\Lambda; \cor^\times) .
\end{equation}
The isomorphism $H^0\muhom(F'_i, F'_i) \simeq \cor_{\Lambda_i}$ also implies
that $w_{ji}$ is well defined up to multiplication by a unit.  It follows that
$\mu^{sh}_2(\Lambda)$ only depends on $\Lambda$.  If $\mu^{sh}_2(\Lambda) = 0$, then we
can write $\{c_{ijk}\}$ as the boundary of a $2$-cochain, say $\{b_{ij}\}$, with
$b_{ji} \in \cor^\times$.  Defining
$$
w'_{ji} = b_{ji}^{-1} w_{ji} \colon \kssfunc_{\Lambda_i}(F'_i)|_{\Lambda_{ij}}
\isoto \kssfunc_{\Lambda_j}(F'_j)|_{\Lambda_{ij}}
$$
we have $w'_{ik} \circ w'_{kj} = w'_{ij}$ and the $\kssfunc_{\Lambda_i}(F'_i)$
glue into a global simple object of $\kss(\cor_\Lambda)$.

\section{The Kashiwara-Schapira stack for orbit categories}
\label{sec:KSstack_orbit}

In this section we set $\cor = \Z/2\Z$.  We have defined the usual sheaf operations
for the triangulated orbit categories $\Orb(\cor_M)$ and we can define a
Kashiwara-Schapira stack in this situation.  We give quickly the analogs of the
results obtained in the previous sections.  For a conic subset $S$ of $T^*M$ we
recall the categories $\OrbL{S}(\cor_M)$, $\OrbL{[S]}(\cor_M)$, $\OrbL{(S)}(\cor_M)$
and $\Orb(\cor_M;S)$ (see Notation~\ref{not:micro_categories}).

Let $\Lambda \subset \dT^*M$ be a locally closed conic
Lagrangian submanifold. We define a stack $\kss_{/[1]}(\cor_\Lambda)$ on $\Lambda$ as
in Definition~\ref{def:KSstack}, again replacing $\Derb$ by $\Orb$.  It comes with a
functor
$\kssfunc_{/[1],\Lambda} \cl \OrbL{(\Lambda)}(\cor_M) \to \kss_{/[1]}(\cor_\Lambda)$.

We say that $F\in \Orb(\cor_M)$ is simple along $\Lambda$ if
$\SSo(F)\cap \dT^*M \subset \Lambda$ and, for any $p\in \Lambda$, there exists
$F'\in \Derb(\cor_M)$ such that $\iota_M(F') \simeq F$, $\SSi(F') = \Lambda$ in a
neighborhood of $p$ and $F'$ is simple along $\Lambda$ at $p$.  As in
section~\ref{sec:KSstack} we can define the substack $\kss_{/[1]}^{s}(\cor_\Lambda)$
of $\kss_{/[1]}(\cor_\Lambda)$ associated with the simple sheaves.  For
$\Omega \subset T^*M$, we have a morphism similar
to~\eqref{eq:HomDkMOmega_vers_muhom}
$$
\Hom_{\Orb(\cor_M;\Omega)}(F,G) 
\to \Hom_{\Orb(\cor_\Omega)}(\cor_\Omega,\muhom^\varepsilon(F,G)|_\Omega) 
$$
and, as in Theorem~\ref{thm:HomDkMp=muhomp}, it is an isomorphism if
$\Omega = \{p\}$ for some $p\in T^*M$.

We remark that Proposition~\ref{prop:SSorb-sectnulle} implies in particular
that, if $B$ is homeomorphic to a ball and $L\in \Orb(\cor_B)$ is locally of the
form $\cor_U$, then there exists an isomorphism $u\colon L \simeq \cor_B$.
Moreover $\Hom_{\Orb(\cor_B)}(\cor_B, \cor_B) \simeq \cor$ (and
$\cor = \Z/2\Z$), hence $u$ is unique.  For simple sheaves this gives:

\begin{lemma}\label{lem:unique-fais-simple-orb}
Let $\Lambda \subset \dT^*M$ be a locally closed conic Lagrangian
submanifold. We assume that $\Lambda$ is contractible.  Let $F,F' \in
\Orb(\cor_M)$ be two simple sheaves along $\Lambda$ and let $\Omega$ be a
neighborhood of $\Lambda$ such that $\SSo(F) \cap \Omega =\SSo(F') \cap \Omega =
\Lambda$.  Then we have a unique iso\-mor\-phism $\muhom^\varepsilon(F,F')|_\Omega
\simeq \cor_\Lambda$ in $\Orb(\cor_\Omega)$.
\end{lemma}

By Lemma~\ref{lem:unique-fais-simple-orb} there exists a unique simple sheaf in
$\kss_{/[1]}(\cor_{\Lambda_0})$ for any contractible open subset $\Lambda_0
\subset \Lambda$, up to a unique isomorphism.  In other words
$\kss_{/[1]}^{s}(\cor_{\Lambda})$ has locally a unique object with the identity as
unique isomorphism. Hence gluing is trivial.  Since
$\kss_{/[1]}^{s}(\cor_{\Lambda})$ is a stack it follows that it has a unique
global object.

We recall that $\loc(\cor_X)$ is the substack of $\Mod(\cor_X)$ formed by the
locally constant sheaves.

\begin{definition}\label{def:Orbloc}
Let $X$ be a manifold. We let $\Oloc^0(\cor_X)$ be the subprestack of
$U \mapsto \Orb(\cor_U)$, $U$ open in $X$, formed by the $F\in \Orb(\cor_U)$
such that $\SSo(F) \subset T^*_UU$.  We let $\Oloc(\cor_X)$ be the
stack associated with $\Oloc^0(\cor_X)$.
\end{definition}
By Proposition~\ref{prop:SSorb-sectnulle} the condition $\SSo(F) \subset T^*_UU$
is equivalent to: $F$ is locally isomorphic to $A_U$ for some
$A \in \Mod(\cor)$.  We recall the functors
$\iota^0_X \colon \Mod(\cor_X) \to \Orb(\cor_X)$ and
$h^0_X \colon \Orb(\cor_X) \to \Mod(\cor_X)$ defined before
Corollary~\ref{cor:iota_equivcat}.  They induce functors of stacks
$i_X\cl \loc(\cor_X) \to \Oloc(\cor_X)$ and
$h_X \cl \Oloc(\cor_X) \to \loc(\cor_X)$.

\begin{lemma}\label{lem:equiv_loc-Oloc}
  The functors $i_X$ and $h_X$ are mutually inverse equivalences of stacks.
\end{lemma}
\begin{proof}
  We have seen in Corollary~\ref{cor:iota_equivcat} that
  $h_X \circ i_X \simeq \id_{\loc(\cor_X)}$.  Hence it is enough to see that
  $i_X$ is locally an equivalence, that is, essentially surjective and fully
  faithful. Let $U \subset X$ be an open subset homeomorphic to a ball.  By
  Proposition~\ref{prop:SSorb-sectnulle} the functor $i_U$ is essentially
  surjective and, for $F,G \in \loc(\cor_U)$, Corollary~\ref{cor:morph-QRF-QRG}
  gives
\begin{align*}
\Hom_{\Orb(\cor_U)}(i_U(F) ,i_U(G)) 
& \simeq \bigoplus_{n\in\Z} \Hom_{\Derb(\cor_U)}(F[-n],G) \\
& \simeq \Hom_{\loc(\cor_U)}(F,G) ,
\end{align*}
which proves that $i_U$ is fully faithful.
\end{proof}

As we remarked after Lemma~\ref{lem:unique-fais-simple-orb}
$\kss_{/[1]}^{s}(\cor_\Lambda)$ has a unique global object.  As
in~\eqref{eq:def_muhombar} the functor $\muhom^\varepsilon$ induces a functors
of stacks
\begin{equation*}
  \ol{\muhom^\varepsilon} \cl
  (\kss_{/[1]}(\cor_\Lambda))^{op} \times \kss_{/[1]}(\cor_\Lambda)
  \to \Oloc(\cor_\Lambda) \simeq  \loc(\cor_\Lambda) .
\end{equation*}
We have an analog of Proposition~\ref{prop:KSstack=Dloc} in the orbit category
case.

\begin{proposition}\label{prop:KSstackorb}
  The stack $\kss_{/[1]}^{s}(\cor_\Lambda)$ has a unique object, say $\shf_0$,
  defined over $\Lambda$.  Moreover the functor
  $\ol{\muhom^\varepsilon}(\shf_0,-)$ induces an equivalence of stacks
  $\kss_{/[1]}(\cor_\Lambda) \isoto \loc(\cor_\Lambda)$.
\end{proposition}

\section{Microlocal germs}
\label{sec:defmicrogerms}

In this section and the next one we   see the link between the classes introduced
in~\S\ref{sec:obstr_classes}, $\mu^{sh}_1(\Lambda) \in H^1(\Lambda; \Z)$,
$\mu^{sh}_2(\Lambda) \in H^2(\Lambda; \cor^\times)$ and
$\mu^{gf}_1(\Lambda) \in H^1(\Lambda; \Z)$,
$\mu^{gf}_2(\Lambda) \in H^2(\Lambda; \Z/2\Z)$.  We only prove the useful
implication that the vanishing of $\mu^{sh}_i(\Lambda)$ (for a ring $\cor$ with
$2\not= 0$) implies the vanishing of $\mu^{gf}_i(\Lambda)$ but a little more work
would show that they coincide.

In the definition of the microsupport of a sheaf $F$ we consider whether some
local cohomology group vanishes, namely
$(\rsect_{\{x;\, \phi(x)\geq \phi(x_0)\}} (F))_{x_0}$ for some function $\phi$.
A natural question is then how this group depends on $\phi$ and not only on
$\xi_0 = d\phi(x_0)$.  In general it really depends on $\phi$, but it is proved
in Proposition~7.5.3 of~\cite{KS90} that it is independent of $\phi$ (up to a
shift $d_\phi$) if we assume that $\Lambda = \SSi(F)$ is a Lagrangian
submanifold near $(x_0;\xi_0)$ and that $\Gamma_{d\phi}$ is transverse to
$\Lambda$ at $(x_0;\xi_0)$ (see Proposition~\ref{prop:inv_microgerm}).  Moreover
the shift $d_\phi$ is related with the Maslov index of the three Lagrangian
subspaces of $T_{(x_0;\xi_0)}T^*M$ given by $\lambda =T_{(x_0;\xi_0)}\Lambda$,
$\lambda_\phi = T_{(x_0;\xi_0)}\Gamma_{d\phi}$ and
$\lambda_0 = T_{(x_0;\xi_0)}(\pi^{-1}(x_0))$. In particular
$(\rsect_{\{x;\, \phi(x)\geq \phi(x_0)\}} (F))_{x_0}$ only depends on
$\lambda_\phi$ which is a Lagrangian subspace of $T_{(x_0;\xi_0)}T^*M$
transverse both to $\lambda$ and $\lambda_0$.  In this section we precise a
little bit this result and prove that there exists a locally constant sheaf on
some open subset of the Lagrangian Grassmannian of $\Lambda \times_{T^*M} TT^*M$
whose stalks at $\lambda_\phi$ is
$(\rsect_{\{x;\, \phi(x)\geq \phi(x_0)\}} (F))_{x_0}$.  We will also see in
\S\ref{sec:monodromy} that it has a non trivial monodromy.

\bigskip

We first introduce some notations.  In this section $M$ is a manifold of
dimension $n$ and $\Lambda$ is a locally closed conic Lagrangian submanifold of
$\dT^*M$.  We recall the notations~\eqref{eq:def_lambdas_p}: for a given point
$p=(x;\xi) \in \Lambda$ we have the following Lagrangian subspaces of
$T_p(T^*M)$
\begin{equation*}
\lambda_0(p) = T_p(T_x^*M), \qquad 
\lambda_\Lambda(p) = T_p\Lambda 
\end{equation*}
and, for   a function $\varphi \colon M \to \R$,
we set $\Lambda_\varphi = \Gamma_{d\varphi} = \{(x; d\varphi(x)); \; x\in M\}$ and
$$
\lambda_\varphi(p) = T_p\Lambda_\varphi .
$$
We let
\begin{equation}\label{eq:def_shl_M}
  \sigma_{T^*M}\cl \lag_M \to T^*M, \qquad
\sigma_{T^*M}^0 \cl \lag_M^0 \to T^*M  
\end{equation}
be respectively the fiber bundle of Lagrangian Grassmannian of $T^*M$ and the
subbundle whose fiber over $p\in T^*M$ is the set of Lagrangian subspaces of
$T_pT^*M$ which are transverse to $\lambda_0(p)$.  Then $\lag_M^0$ is an open
subset of $\lag_M$. For a given $p\in T^*M$ we set $V=T_{\pi_M(p)}M$ and we
identify $T_pT^*M$ with $V\times V^*$.  We use coordinates $(\nu;\eta)$ on
$T_pT^*M$.  Then we can see that any $l\in (\lag_M^0)_p$ is of the form
\begin{equation}\label{eq:LM0=matsym}
l = \{ (\nu;\eta)\in T_pT^*M;\; \eta = A\cdot \nu \},
\end{equation}
where $A \cl V \to V^*$ is a symmetric matrix.  This identifies the fiber
$(\lag_M^0)_p$ with the space of $n\times n$-symmetric matrices.

For a function $\varphi$ defined on a product $X\times Y$ and for a given
$x\in X$ we use the general notation $\varphi_x = \varphi|_{\{x\}\times Y}$.
\begin{lemma}\label{lem:exist-fcttest}
There exists a function $\psi\cl \lag_M^0\times M \to \R$ of class
$C^\infty$ such that, for any $l\in \lag_M^0$ with $\sigma_{T^*M}(l) = (x;\xi)$,
$$
\psi_l(x) = 0, \quad
d\psi_l(x) = \xi, \quad
\lambda_{\psi_l}( \sigma_{T^*M}(l)) = l .
$$
\end{lemma}
\begin{proof}
(i) We first assume that $M$ is the vector space $V=\R^n$. 
We identify $T^*M$ and $M\times V^*$. For $p=(x;\xi) \in M\times V^*$ the fiber
$(\lag_M^0)_p$ is identified with the space of quadratic forms on $V$
through~\eqref{eq:LM0=matsym}.  For $l\in (\lag_M^0)_p$ we let $q_l$ be the
corresponding quadratic form.  Now we define $\psi_0$ by  
$$
\psi_0(l,y) = \langle y-x, \xi \rangle  + \pdemi\, q_l(y-x),
\quad \text{where $(x;\xi) = \sigma_{T^*M}^0(l)$.}
$$
We can check that $\psi_0$ satisfies the conclusion of the lemma.

\sui (ii) In general we choose an embedding $i\cl M\hookrightarrow \R^N$.  For a
given $p' =(x;\xi') \in M\times_{\R^N} T^*\R^N$ the subspace
$T_{p'}(M\times_{\R^N} T^*\R^N)$ of $T_{p'}T^*\R^N$ is coisotropic.  The
symplectic reduction of $T_{p'}T^*\R^N$ by $T_{p'}(M\times_{\R^N} T^*\R^N)$ is
canonically identified with $T_pT^*M$, where $p = i_d(p')$.  The symplectic
reduction sends Lagrangian subspaces to Lagrangian subspaces and we deduce a
map, say $r_{p'} \cl \lag_{\R^N,p'} \to \lag_{M,p}$.  The restriction of
$r_{p'}$ to the set of Lagrangian subspaces which are transverse to
$T_{p'}(M\times_{\R^N} T^*\R^N)$ is an actual morphism of manifolds.  In
particular it induces a morphism
$r^0_{p'} \cl \lag^0_{\R^N,p'} \to \lag^0_{M,p}$.  We can see that $r^0_{p'}$ is
onto and is a submersion.  When $p'$ runs over $M\times_{\R^N} T^*\R^N$ we
obtain a surjective morphism of bundles, say $r$:
$$
\begin{tikzcd}
\lag_{\R^N}^0|_{M\times_{\R^N} T^*\R^N} \ar[r, "r"] \ar[d] & \lag_M^0 \ar[d]\\
M\times_{\R^N} T^*\R^N \ar[r, "i_d"] & T^*M. 
\end{tikzcd}
$$
We can see that $r$ is a fiber bundle, with fiber an affine space.  Hence we can
find a section, say $j\cl \lag_M^0 \to \lag_{\R^N}^0$.  For
$(l,x) \in \lag_M^0\times M$ we set $\psi(l,x) = \psi_0(j(l), i(x))$, where
$\psi_0$ is defined in~(i). Then $\psi$ satisfies the conclusion of the lemma.
\end{proof}

We come back to the Lagrangian submanifold $\Lambda$ of $\dT^*M$.  We let
\begin{equation}\label{eq:def_U_Lambda}
U_\Lambda \subset \lag_M^0|_\Lambda
\end{equation}
be the subset of $\lag_M^0|_\Lambda$ consisting of Lagrangian subspaces of
$T_pT^*M$ which are transverse to $\lambda_\Lambda(p)$. We define
$\sigma_\Lambda = \sigma_{T^*M}|_{U_\Lambda}$,
$\tau_M = \pi_M|_\Lambda \circ \sigma_\Lambda$ and
$i_\Lambda \colon U_\Lambda \hookrightarrow U_\Lambda\times M$,
$l \mapsto (l,\tau_M(l))$:
\begin{equation}\label{eq:def_sigmaLambdaetc}
\begin{tikzcd}[column sep=2cm]
U_\Lambda \ar[r, "\sigma_\Lambda"] \ar[d, "i_\Lambda"] \ar[dr, "\tau_M"]
&  \Lambda \ar[d, "\pi_M|_\Lambda"]  \\
U_\Lambda\times M \ar[r, "q_2"] &M.
\end{tikzcd}
\end{equation}
We note that $U_\Lambda$ is not a fiber bundle over $\Lambda$ but only an open
subset of $\lag_M^0|_\Lambda$. However, for a given $p\in \Lambda$, we will use
the notation $U_{\Lambda,p} = \opb{\sigma_\Lambda}(p)$.

\begin{definition}\label{def:microgerm}
  Let $\psi\cl \lag_M^0\times M \to
  \R$ be a function satisfying the conclusions of Lemma~\ref{lem:exist-fcttest}
  and let $\varphi \cl U_\Lambda\times M \to
  \R$ be its restriction to $U_\Lambda\times M$.  For
  $F\in \Derb_{(\Lambda)}(\cor_M)$ we define
  $m_\varphi(F) \in \Derb(\cor_{U_\Lambda})$ by
  \begin{equation*}
m_\varphi(F) = \opb{i_\Lambda}(\rsect_{\opb{\varphi}([0,+\infty[)}(\opb{q_2}F)),
\end{equation*}
where $q_2\cl U_\Lambda \times M \to M$ is the projection.
\end{definition}

\begin{proposition}\label{prop:SSmicrogerm}
  Let $F\in \Derb_{(\Lambda)}(\cor_M)$.  Let
  $\varphi \cl U_\Lambda\times M \to \R$ be as in
  Definition~\ref{def:microgerm}.  Then the object
  $m_\varphi(F) \in \Derb(\cor_{U_\Lambda})$ has locally constant cohomology
  sheaves and its stalks are
  $$
  (m_\varphi(F))_l \simeq (\rsect_{\opb{\varphi_l}([0,+\infty[)}(F))_x,
  $$
  for any $l\in U_\Lambda$ and $x= \tau_M(l)$.
\end{proposition}
\begin{proof}
  (i) We first prove that $m_\varphi(F)$ is locally constant.  For this we give
  another expression of $m_\varphi(F)$.  We define
  $G \in \Derb(\cor_{T^*(U_\Lambda \times M)})$ by
  $G = \muhom(\cor_{\opb{\varphi}([0,+\infty[)},\opb{q_2}F)$.  We use the notations
  in~\eqref{eq:def_sigmaLambdaetc} and we define
  $I_\Lambda = \im(i_\Lambda) \subset U_\Lambda\times M$ and
  $J_\Lambda \subset \dT^*(U_\Lambda\times M)$,
  $J_\Lambda = \{ (l,x; 0, \lambda \xi)$; $(x;\xi) = \sigma_\Lambda(l), \lambda>0\}$.
  We remark that $J_\Lambda$ is a fiber bundle over $I_\Lambda$ with fiber $\rspos$.
  We prove in~(ii) and~(iii) below that there exists a neighborhood $V$ of
  $I_\Lambda$ in $U_\Lambda \times M$ such that
  \begin{itemize}
  \item [(a)] $\supp(G) \cap \dT^*V \subset J_\Lambda$,
  \item [(b)] $\SSi(G|_{\dT^*V}) \subset T^*_{J_\Lambda}T^*(U_\Lambda\times M)$,
  \item [(c)]
    $(\rsect_{\opb{\varphi}([0,+\infty[)}(\opb{q_2}F))_{I_\Lambda} \simeq
    \roim{\dot\pi_{V}{}}( G|_ {\dT^*V})$.
  \end{itemize}
  By   Proposition~\ref{prop:iminvproj} the
  properties~(a-b) imply that $G|_{\dT^*V}$ has support in $J_\Lambda$ and is locally
  constant along $J_\Lambda$. Since $J_\Lambda$ is a fiber bundle over $I_\Lambda$,
  we deduce by~(c) that
  $(\rsect_{\opb{\varphi}([0,+\infty[)}(\opb{q_2}F))_{I_\Lambda}$ is locally constant
  on $I_\Lambda$, hence $m_\varphi(F)$ is locally constant on $U_\Lambda$.

  \sui(ii)   We prove (i-a) and (i-b).  By
  Proposition~\ref{prop:SSmuhom} and Lemma~\ref{lem:lagr-clean-inter} below we have:
  for $\Lambda_1,\Lambda_2$ two conic Lagrangian submanifolds of a cotangent bundle
  $\dT^*X$ with a clean intersection $\Xi = \Lambda_1 \cap \Lambda_2$ and for
  $F_i \in \Derb_{[\Lambda_i]}(\cor_X)$, $i=1,2$, we have
  $\supp(\muhom(F_1,F_2)|_{\dT^*X}) \subset \Xi$ and
  $\SSi(\muhom(F_1,F_2)|_{\dT^*X}) \subset T^*_\Xi T^*X$.

    By Example~\ref{ex:microsupport}~(iii) we have
  $\dot\SSi(\cor_{\opb{\varphi}([0,+\infty[)}) = \Lambda'_\varphi$, where
 $$
  \Lambda'_\varphi = \{(l,x; \lambda \cdot d\varphi(l,x)); \; (l,x)\in U_\Lambda
  \times M, \; \lambda >0,\; \varphi(l,x) = 0\} .
  $$
  We also have $\dot\SSi(\opb{q_2}F) = T^*_{U_\Lambda}U_\Lambda \times \Lambda$.  It
  is then enough to find a neighborhood $V$ of $I_\Lambda$ in $U_\Lambda \times M$
  such that $T^*V \cap \Lambda'_\varphi$ and
  $T^*V \cap (T^*_{U_\Lambda}U_\Lambda \times \Lambda)$ have a clean intersection,
  which is $J_\Lambda$.

  \smallskip

  Let us first prove that $\Lambda_\varphi (= \Gamma_{d\varphi})$ is transverse to
  $T^*U_\Lambda \times \Lambda$.  For $l_0 \in U_\Lambda$, with
  $\sigma_\Lambda(l_0) = (x_0;\xi_0)$, we know that $\Lambda_{\varphi_{l_0}}$ is
  transverse to $\Lambda$ at the point $(x_0;\xi_0)$ and
  $\xi_0 = d\varphi_{l_0}(x_0)$. Hence we can find a neighborhood $V_{l_0}$ of $x_0$
  in $M$ such that
  $\Lambda_{\varphi_{l_0}} \cap \Lambda \cap T^*V_{l_0} = \{(x_0;\xi_0)\}$.  Since
  $\Lambda_{\varphi_{l_0}}$ is the projection of
  $(T^*_{l_0}U_\Lambda \times T^*M) \cap \Lambda_\varphi$ to $T^*M$, it follows that,
  in $T^*_{l_0}U_\Lambda \times T^*V_{l_0}$, the submanifolds
  $(T^*_{l_0}U_\Lambda \times T^*V_{l_0}) \cap \Lambda_\varphi$ and
  $T^*_{l_0}U_\Lambda \times (T^*V_{l_0} \cap \Lambda)$ are transverse at the point
  $(l_0,x_0; \frac{\partial \varphi}{ \partial l}, \frac{\partial\varphi}{\partial
    x})$ (and this is the only intersection point).  We can make $V_{l_0}$ move
  nicely enough with $l_0$ so that $V = \bigsqcup_{l \in U_\Lambda} \{l\} \times V_l$
  is a neighborhood of $I_\Lambda$ in $U_\Lambda \times M$.  Then
  $T^*V \cap \Lambda_\varphi$ is transverse to
  $T^*V \cap (T^*U_\Lambda \times \Lambda)$, with intersection
$$
J^1_\Lambda = \{ (l,x; \frac{\partial \varphi}{ \partial l}, 
\frac{\partial\varphi}{\partial x}); \; l\in U_\lambda, \, x = \tau_M(l)\} .
$$
Let us prove that $J_\Lambda = \rspos \cdot J^1_\Lambda$. For $l_0\in U_\lambda$ and
$x_0 = \tau_M(l_0)$ we have
$(x_0;\frac{\partial\varphi}{\partial x}(l_0,x_0)) = \sigma_\Lambda(l_0)$ by the
definition of $\varphi$. It remains to see that
$\frac{\partial \varphi}{ \partial l} (l_0,x_0) = 0$.  We recall that
$\varphi(l,\tau_M(l)) = 0$ for all $l \in U_\Lambda$.  Differentiating this
relation we obtain
$\frac{\partial \varphi}{ \partial l}(l_0,x_0) + \xi_0 \circ d\tau_M(l_0) = 0$,
for all $l_0\in U_\lambda$ and $(x_0; \xi_0) = \sigma_\Lambda(l_0)$ (here we view
$\xi_0 = \frac{\partial\varphi}{\partial x}(l_0,x_0)$ as a map from $T_{x_0}M$ to
$\R$).  Since the map $\tau_M \colon U_\Lambda \to M$ factorizes through
$\sigma_\Lambda \colon U_\Lambda \to \Lambda \subset T^*M$, we have
$$
\xi_0 \circ d\tau_M(l_0) = \xi_0 \circ d\pi_M(x_0;\xi_0) \circ
d\sigma_\Lambda(l_0) = \alpha_M \circ d\sigma_\Lambda(l_0),
$$
where $\alpha_M$ is the Liouville $1$-form on $T^*M$.  Now $\Lambda$ is conic
Lagrangian, hence the pull-back of $\alpha_M$ to $\Lambda$ vanishes and we obtain
$\xi_0 \circ d\tau_M(l_0) = 0$, hence
$\frac{\partial \varphi}{ \partial l} (l_0,x_0) = 0$, as required.

It follows from this discussion that $A = T^*V \cap\rspos \cdot \Lambda_\varphi$ is
transverse to $B = T^*V \cap (T^*U_\Lambda \times \Lambda)$ with intersection
$J_\Lambda$.  Now $J_\Lambda$ is contained in $A_1 = \{\varphi=0\}$ and in
$B_1 = T^*_{U_\Lambda}U_\Lambda \times \Lambda$.  We deduce that $A \cap A_1$ and
$B \cap B_1$ have a clean intersection which is still $J_\Lambda$.  Since
$A\cap A_1 = T^*V \cap \Lambda'_\varphi $ and
$B\cap B_1 = T^*V \cap (T^*_{U_\Lambda}U_\Lambda \times \Lambda)$, this concludes the
proof of~(i-a) and (i-b).

  \sui(iii) Now we prove the claim (c) of (i).  Sato's
  triangle~\eqref{eq:SatoDTmuhom1} gives
  \begin{align*}
(\DD'(\cor_{\opb{\varphi}([0,+\infty[)}) \tens \opb{q_2}F)_{I_\Lambda}
    \to (&\rsect_{\opb{\varphi}([0,+\infty[)}(\opb{q_2}F))_{I_\Lambda} \\
    &\qquad\to \roim{\dot\pi_{U_\Lambda \times M}{}}( G )_{I_\Lambda} \to[+1].
  \end{align*}
By definition $d\varphi$ does not vanish in a neighborhood of $I_\Lambda$.
Hence $\opb{\varphi}(0)$ is a smooth hypersurface near $I_\Lambda$ and
$\DD'(\cor_{\opb{\varphi}([0,+\infty[)}) \simeq
\cor_{\opb{\varphi}(]0,+\infty[)}$.  Since $I_\Lambda \subset \opb{\varphi}(0)$,
the first term of the above triangle is zero. By~\mbox{(i-a)} the support of
$\roim{\dot\pi_{V}{}}( G|_ {\dT^*V})$ is already contained in $I_\Lambda$. So we
can forget the subscript $I_\Lambda$ in the third term and we obtain~(i-c).

\sui(iv) We prove the last assertion of the proposition. Let $l_0 \in U_\Lambda$ be
given and $(x_0; \xi_0) = \sigma(l_0)$.  Since $\Lambda'_{\varphi_l}$ is transverse
to $\Lambda$ at $(x;\xi) = \sigma(l)$, by Lemma~\ref{lem:isotopie_graphes} below we
can find neighborhoods $U$ of $l_0$ and $W$ of $x_0$ and a homogeneous Hamiltonian
isotopy of $\dT^*W$ parameterized by $U$, say
$\Psi \colon U \times \dT^*W \to \dT^*W$, such that
$\Psi_l(\Lambda) \cap T^*W = \Lambda \cap T^*W$ and
$\Psi_l(\Lambda'_{\varphi_{l_0}}) \cap T^*W = \Lambda'_{\varphi_l} \cap T^*W$, for
all $l\in U$.  We set $\Lambda^+ = \Lambda'_\varphi \cup (T^*_UU \times \Lambda)$ and
$\Lambda^+_l = \Lambda'_{\varphi_l} \cup \Lambda$. Then
$\cor_{\opb{\varphi}([0,+\infty[)}$ and $\opb{q_2}F$ belong to
$\Der_{[\Lambda^+]}(\cor_{U\times W})$ and $\Psi_l(\Lambda^+_{l_0}) = \Lambda^+_l$.
By Proposition~\ref{prop:equiv_local_isotopie}~(ii) we deduce
\begin{align*}
  \rhom(&\cor_{\opb{\varphi}([0,+\infty[)}, \opb{q_2}F)|_{\{l_0\}\times W} \\
  &\isoto \rhom(\cor_{\opb{\varphi}([0,+\infty[)}|_{\{l_0\}\times W},
    \opb{q_2}F|_{\{l_0\}\times W}) \\
  & \simeq \rhom(\cor_{\opb{\varphi_{l_0}}([0,+\infty[)}, F)|_W .
\end{align*}
Taking the germs at $x_0 \in W$ we obtain the required isomorphism.
\end{proof}

Let $X$ be a manifold and $Y,Z$ two submanifolds of $X$.  We recall that $Y$ and
$Z$ have a clean intersection if $W=Y\cap Z$ is a submanifold of $X$ and $TW =
TY \cap TZ$. This means that we can find local coordinates $(\ul x,\ul y, \ul
z,\ul w)$ such that $Y = \{ \ul x = \ul z =0\}$ and $Z = \{ \ul x = \ul y
=0\}$. Using these coordinates the following lemma is easy.

\begin{lemma}\label{lem:clean-cone}
Let $X$ be a manifold and $Y,Z$ two submanifolds of $X$ which have a clean
intersection. We set $W=Y\cap Z$. Then $C(Y,Z) = W\times_X TY + W\times_X TZ$.
\end{lemma}

\begin{lemma}\label{lem:lagr-clean-inter}
  Let $X$ be a manifold and $\Lambda_1,\Lambda_2$ be two Lagrangian submanifolds
  of $\dT^*X$. Let $F_1\in \Derb_{(\Lambda_1)}(\cor_X)$ and
  $F_2\in \Derb_{(\Lambda_2)}(\cor_X)$.  We assume that $\Lambda_1$ and
  $\Lambda_2$ have a clean intersection and we set
  $\Xi= \Lambda_1 \cap \Lambda_2$.  Then there exists a neighborhood $U$ of
  $\Xi$ in $T^*X$ such that $\SSi(\muhom(F_1,F_2) |_U) \subset T^*_\Xi T^*X$,
  that is, $\muhom(F_1,F_2)|_U$ is supported on $\Xi$ and has locally constant
  cohomology sheaves on $\Xi$.
\end{lemma}
\begin{proof}
We have $\SSi(\muhom(F_1,F_2) ) \subset (H^{-1}(C(\SSi(F_2) , \SSi(F_1) )))^a$
by the bound~\eqref{eq:SSmuhom}.
Let $U_i$ be a neighborhood of $\Lambda_i$ such that
$\SSi(F_i)\cap U_i  \subset \Lambda_i$, $i=1,2$.
Then $U = U_1 \cap U_2$ is a neighborhood of $\Xi$ and we have
$H^{-1}(C(\SSi(F_2) , \SSi(F_1) )) \cap T^*U
 \subset H^{-1}(C(\Lambda_2,\Lambda_1))$.

Since $\Lambda_i$ is Lagrangian we have
$H^{-1}(T\Lambda_i) = T^*_{\Lambda_i}T^*X$, for $i=1,2$. In particular
$H^{-1}( \Xi\times_{T^*X} T\Lambda_i) \subset T^*_\Xi T^*X$ and the result
follows from Lemma~\ref{lem:clean-cone}.
\end{proof}

\begin{lemma}\label{lem:isotopie_graphes}
  Let $B$ be a neighborhood of $0$ in $\R^N$.  Let
  $\varphi \colon B \times \R^n \to \R$ be a family of functions.
  \\
  (i) We assume that $\Gamma_{d\varphi_0}$ is transverse to the zero-section
  $T^*_{\R^n}\R^n$ of $T^*\R^n$ and
  $\Gamma_{d\varphi_0} \cap T^*_{\R^n}\R^n = \{0\}$.  Then there exist neighborhoods
  $B'$ of $0$ in $\R^N$ and $V$ of $0$ in $\R^n$ and a family of Hamiltonian
  isotopies of $T^*\R^n$ parameterized by $B'$, say
  $\Psi \colon B' \times T^*\R^n \to T^*\R^n$, such that
  $\Psi_b(T^*_{\R^n}\R^n) = T^*_{\R^n}\R^n$ and
  $\Psi_b(\Gamma_{d\varphi_0}) \cap T^*V = \Gamma_{d\varphi_b} \cap T^*V$, for all
  $b\in B'$.
  \\
  (ii) Let $\Lambda \subset \dT^*\R^n$ be a closed conic Lagrangian submanifold.  We
  assume that $\Gamma_{d\varphi_0}$ is transverse to $\Lambda$ with
  $\Gamma_{d\varphi_0} \cap \Lambda = \{(0; \xi_0)\}$,
  $\Gamma_{d\varphi_b} \cap \Lambda = \{(x_b; \xi_b)\}$ and $\varphi_b(x_b) = 0$ for
  all $b\in B$.  Then there exist neighborhoods $B'$ of $0$ in $\R^N$ and $V$ of $0$
  in $\R^n$ and a family of homogeneous Hamiltonian isotopies of $\dT^*\R^n$
  parameterized by $B'$, say $\Psi \colon B' \times \dT^*\R^n \to \dT^*\R^n$, such
  that  
  $\Psi_b(\Lambda) \cap \dT^*V = \Lambda \cap \dT^*V$ and
  $\Psi_b(\Lambda'_0) \cap \dT^*V = \Lambda'_b \cap \dT^*V$, for all $b\in B'$, where
  $\Lambda'_b = \{(x; \lambda \cdot d\varphi_b(x)); \; \lambda >0,\; \varphi_b(x) =
  0\}$.
\end{lemma}
\begin{proof}
  (i) The transversality hypothesis implies that $d\varphi_b$ viewed as a
  function from $\R^n$ to $(\R^n)^*$ is invertible near $0$, for $b$ small
  enough.  We set $\theta_b = (d\varphi_b)^{-1}$ and view $\theta_b$ as a
  $1$-form on $(\R^n)^*$ defined in some neighborhood of $0$.  Since the graph
  of $\theta_b$ is Lagrangian, it is a closed $1$-form and we can write
  $\theta_b = dh_b$ near $0$.  We consider $h_b(\xi)$ as a Hamiltonian function
  on $T^*\R^n$.  By construction its Hamiltonian vector field is
  $X_{h_b}(x;\xi) = \sum_i (\theta_b)_i(\xi) \partial_{x_i}$ and the time $1$ of
  its flow satisfies $\phi_{h_b}^1(\{0\} \times (\R^n)^*) = \Gamma_{d\varphi_b}$
  near $0$. Moreover
  $\phi_{h_b}^t(\R^n \times \{0\}) \subset \R^n \times \{0\}$.  Now we set
  $\Psi_b = \phi_{h_b}^1 \circ (\phi_{h_0}^1)^{-1}$.

  \sui(ii) We can find a homogeneous Hamiltonian isotopy $\Phi$ arbitrarily close to
  $\id$ and a neighborhood $W$ of $\Phi(0;\xi_0)$ such that $T^*W \cap \Phi(\Lambda)$
  is half of the conormal bundle of a smooth hypersurface $X$ and
  $T^*W \cap \Phi(\Gamma_{d\phi_b})$ is still the graph of a function for $b$ close
  enough to $0$.  We take coordinates on $W$ such that $X = \R^{n-1} \times \{0\}$,
  $\Phi(0;\xi_0) = (\ul 0, 0; \ul 0,1)$ and we write
  $\Phi(x_b;\xi_b) = (y_b,0; \ul 0, \eta_b)$.  Then $\Phi(\Lambda'_b)$ is the
  conormal bundle of a hypersurface which is the graph of a function
  $\varphi'_b \colon \R^{n-1} \to \R$.  Moreover $\Gamma_{d\varphi'_b}$ is transverse
  to $T^*_{\R^{n-1}}\R^{n-1}$.  Part~(i) of the proof gives a family $\Psi'$ of
  Hamiltonian isotopies of $T^*\R^{n-1}$ such that
  $\Psi'_b(T^*_{\R^{n-1}}\R^{n-1}) = T^*_{\R^{n-1}}\R^{n-1}$ and
  $\Psi'_b(\Gamma_{d\varphi'_0}) \cap T^*V = \Gamma_{d\varphi'_b} \cap T^*V$.  We
  lift $\Psi'$ into a family $\Psi''$ of homogeneous Hamiltonian isotopies of
  $\dT^*\R^n$ and we set $\Psi_b = \Phi^{-1} \circ \Psi''_b \circ \Phi$.
\end{proof}

\begin{remark}\label{rem:def_microgerm}  
  The assignment $F \mapsto m_\varphi(F)$ of Definition~\ref{def:microgerm} is a
  functor $m_\varphi \colon \Derb_{(\Lambda)}(\cor_M) \to \Derb(\cor_{U_\Lambda})$.
  By Proposition~\ref{prop:SSmicrogerm} it factorizes through
  $\Dloc^0(\cor_{U_\Lambda})$ (see Definition~\ref{def:Dloc}).  Let us check that
  $m_\varphi$ induces a functor of stacks 
  $m_\varphi \colon \kss(\cor_\Lambda) \to \Dloc(\cor_{U_\Lambda})$.

  We recall that $\kss(\cor_\Lambda)$ is associated with some prestack
  $\kss^0_\Lambda$; the objects of $\kss^0_\Lambda(\Lambda_0)$ are those of
  $\Derb_{(\Lambda_0)}(\cor_M)$ and the Hom set between $F$ and $G$ is
  $\Hom_{\Derb(\cor_M;\Lambda_0)}(F,G)$ (see Definition~\ref{def:KSstack}).  Now we
  remark that the definition of $m_\varphi(F)$ applies to any $F \in \Derb(\cor_M)$
  and gives in fact $m_\varphi \colon \Derb(\cor_M) \to \Derb(\cor_{U_\Lambda})$.  By
  the definition of the microsupport we see that $m_\varphi(F) \simeq 0$ if
  $\SSi(F) \cap \Lambda = \emptyset$, that is, $m_\varphi(F) \simeq 0$ if
  $F \in \Derb_{T^*M \setminus \Lambda}(\cor_M)$.  By the definition of
  $\Derb(\cor_M;\Lambda)$ this means that $m_\varphi$ factorizes through
  $\Derb(\cor_M;\Lambda)$ (see the reminder on localization in
  section~\ref{sec:def_orb_cat}).  Hence $m_\varphi$ induces a functor from
  $\kss^0_\Lambda(\Lambda)$ to $\Derb(\cor_{U_\Lambda})$, still with image in
  $\Dloc^0(\cor_{U_\Lambda})$. We can replace $\Lambda$ by any of its open subset
  $\Lambda_0$ and obtain in this way a functor of prestacks from $\kss^0_\Lambda$ to
  $\Dloc^0(\cor_{U_\Lambda})$.  Passing to the associated stacks it induces a functor
  of stacks $m_\varphi \colon \kss(\cor_\Lambda) \to \Dloc(\cor_{U_\Lambda})$.
\end{remark}

\section{Monodromy morphism}
\label{sec:monodromy}

We keep the notations of Definition~\ref{def:microgerm} and
Remark~\ref{rem:def_microgerm}.  In particular we have a choice of function
$\varphi \cl U_\Lambda\times M \to \R$ which gives a functor
$m_\varphi \colon \Derb_{(\Lambda)}(\cor_M) \to \Derb(\cor_{U_\Lambda})$.  For
$F\in \Derb_{(\Lambda)}(\cor_M)$ we know that $m_\varphi(F)$ is a locally constant
object on $U_\Lambda$.  Here we describe the monodromy of its restriction to a fiber
$U_{\Lambda,p}$ of $\sigma_\Lambda \cl U_\Lambda \to \Lambda$.

\medskip

We first recall well-known results on locally constant sheaves and introduce
some notations.  Let $X$ be a manifold and $L\in \Derb(\cor_X)$ such that
$\SSi(L) \subset T^*_XX$, that is, $L$ has locally constant cohomology sheaves.
Then any path $\gamma \cl [0,1] \to X$ induces an isomorphism
\begin{equation}\label{eq:def_action_chemin}
M_\gamma(L) \cl L_{\gamma(0)} \isofrom \rsect([0,1]; \opb{\gamma}L)
\isoto L_{\gamma(1)}.
\end{equation}
Moreover, $M_\gamma(L)$ only depends on the homotopy class of $\gamma$ with
fixed ends and $M_\gamma(L) \circ M_{\gamma'}(L) = M_{\gamma \circ \gamma'}(L)$
if $\gamma$, $\gamma'$ are composable.  In particular, if we fix a base point
$x_0\in X$, we obtain the monodromy morphism
\begin{equation}\label{eq:def_monodr}
  \begin{split}
M(L) \cl \pi_1(X;x_0) &\to \Iso(L_{x_0})  \\
\gamma & \mapsto M_\gamma(L),
  \end{split}
\end{equation}
where $\Iso(L_{x_0})$ is the group of isomorphisms of $L_{x_0}$ in
$\Derb(\cor)$.

\smallskip

Now we go back to the situation of section~\ref{sec:defmicrogerms}.  For a given
$p\in \Lambda$ we set $U_{\Lambda,p} = \sigma_\Lambda^{-1}(p)$.  This is the open
subset of Lagrangian Grassmannian manifold $\lag(T_pT^*M)$ formed by the Lagrangian
subspaces of $T_p\Lambda$ which are transverse to $\lambda_0(p)$ and
$\lambda_\Lambda(p)$.  Let us describe its connected components and their fundamental
groups.

Let $(V,\omega)$ be a symplectic vector space of dimension $2n$.  Let $l_1, l_2$
be two Lagrangian subspaces.  We can assume that $V=\R^{2n}$ with
$\omega = \sum e_i\wedge f_i$ in the canonical base
$(e_1,\ldots,e_n, f_1,\ldots,f_n)$ and $l_1 = \langle e_1,\ldots,e_n \rangle$,
$l_2 = \langle e_1,\ldots,e_k$, $f_{k+1},\ldots, f_n \rangle$.  Let
$U(l_1) \subset \lag(V)$ be the open subset of Lagrangian subspaces which are
transverse to $l_1$.  Then $U(l_1)$ is diffeomorphic to $\Sym_n$, the space of
symmetric matrices of size $n\times n$, through
$M \mapsto l_M := \{(M \ul y, \ul y)$; $\ul y = (y_1,\ldots,y_n)\}$.  Writing
$M = \begin{pmatrix} A & B \\ {}^tB & C \end{pmatrix}$, with $A$ of size
$k\times k$, we see that $l_M$ is also transverse to $l_2$ if and only if $C$ is
invertible.  Hence $U(l_1) \cap U(l_2)$ is diffeomorphic to
$\R^d \times \Sym^0_{n-k}$ where $d$ is some integer and $\Sym^0_{n-k}$ is the
subset of invertible matrices in $\Sym_{n-k}$.

Let us recall the topology of $\Sym_n^{p,q}$ the subset of $\Sym_n$ of matrices
with $p$ positive eigenvalues and $q$ negative eigenvalues, $p+q=n$.  The action
of $\GL^+_n$ on $\Sym_n^{p,q}$, $A\cdot M = {}^tAMA$, gives
$\Sym_n^{p,q} \simeq \GL^+_n/\SO(p,q)$.  In particular $\Sym_n^{p,q}$ is
connected.  Now a maximal compact subgroup of $\SO(p,q)$ is
$K = S(\GO(p)\times \GO(q))$ and $\SO(p,q)$ is diffeomorphic to $K \times \R^d$
for some $d$. We deduce an exact sequence of fundamental groups:
$$
\pi_1(K) \to \pi_1(\GL^+_n) \to \pi_1(\Sym_n^{p,q})
\to \pi_0(K) \to \pi_0(\GL^+_n)= \{1\} . 
$$
We recall that $\pi_1(\SO(n)) = \pi_1(\GL^+_n)$ is $\Z$ for $n=2$ and $\Z/2\Z$
for $n\geq 3$.  We will only need the case where $n$ is big. In particular we
can assume $n\geq 3$ and $p$ or $q\geq 2$.  Since $K$ contains
$\SO(p) \times \SO(q)$ it follows that the map $\pi_1(K) \to \pi_1(\GL^+_n)$ is
surjective. Hence $\pi_1(\Sym_n^{p,q}) \isoto \pi_0(K)$.  If $p$ and $q$ are
both $\geq 1$, then $\GO(p) \times \GO(q)$ has four components and
$\pi_0(S(\GO(p)\times \GO(q))) = \Z/2\Z$.  If $p$ or $q$ vanishes, then
$\SO(p,q) = \SO(n)$ and $\Sym_n^{p,q}$ is contractible.  In conclusion, for
$n\geq 3$ we have $\pi_1(\Sym_n^{p,q}) \simeq \Z/2\Z$ if $p\not=0$ and $q\not=0$
and $\pi_1(\Sym_n^{p,q}) \simeq 0$ if $p=0$ or $q=0$.

\smallskip

Let us write down the above results for
$U_{\Lambda,p} = \sigma_\Lambda^{-1}(p)$.  We set $n = \dim M$ and
$k = \dim(\lambda_0(p) \cap \lambda_\Lambda(p))$.  We assume $n-k \geq 3$.  Then
$U_{\Lambda,p}$ is topologically equivalent to
$\Sym^0_{n-k} = \bigsqcup_{p=0}^{n-k} \Sym_{n-k}^{p,n-k-p}$, which has $n-k+1$
components, two of them being contractible and the other ones having
$\pi_1 = \Z/2\Z$.

The inertia index $\tau_{T_pT^*M}$ introduced in~\eqref{eq:def_inertia_ind}
gives a function on 
$$
\tau_p \colon U_{\Lambda,p} \to \Z, \qquad
l \mapsto \tau_{T_pT^*M}(\lambda_0(p), \lambda_\Lambda(p), l) 
$$
which is constant on each component of $U_{\Lambda,p}$.  In the coordinates chosen
above for $l_1$, $l_2$ and with $l = l_M$,
$M = \begin{pmatrix} A & B \\ {}^tB & C \end{pmatrix}$   we can
see that $\tau_p(l_M) = \sgn(C)$, using Lemma~A.3.3 of~\cite{KS90} (here $\sgn(C)$ is
$p_+-p_-$, where $p_+$ and $p_-$ are the numbers of positive and negative eigenvalues
of $C$).  Since $C$ is an invertible symmetric matrix of size $(n-k)$, we obtain that
the values of $\tau_p$ are $\{-n+k, -n+k+2,\dots, n-k\}$. Hence $\tau_p$
distinguishes the components of $U_{\Lambda,p}$ and we can index them as follows:
\begin{equation}\label{eq:def_ULambdai}
  \text{$U_{\Lambda,p}^i$ is the connected component of $U_{\Lambda,p}$ where
    $\tau_p = i$.}
\end{equation}
We have to be careful that the components of $U_\Lambda$ cannot be indexed in
this way: the function $\tau_p$ is locally constant on $U_{\Lambda,p}$ but the
function
$l \mapsto \tau_{T_{\sigma(l)}T^*M}(\lambda_0(\sigma(l)),
\lambda_\Lambda(\sigma(l)), l)$ is not locally constant on $U_\Lambda$.  For
example, when $\dim(\lambda_0(p) \cap \lambda_\Lambda(p))$ changes by $1$ (which
happens when $p$ moves from a generic point to a cusp), the parity of the
possible values of $\tau_p$ also changes.  However we have the following result.

\begin{lemma}\label{lem:diff_Maslov}
  The function $\delta \colon U_\Lambda \times_\Lambda U_\Lambda \to \Z$,
  $(l,l') \mapsto \tau_p(l) - \tau_p(l')$, where $p = \sigma(l) = \sigma(l')$,
  is locally constant on $U_\Lambda \times_\Lambda U_\Lambda$.  
\end{lemma}
\begin{proof}
  We recall that $\tau$ satisfies a cocycle relation (see for
  example~\cite{KS90}, Thm~A.3.2)
  $\tau(l_1,l_2,l_3) - \tau(l_2,l_3,l_4) + \tau(l_3,l_4,l_1) - \tau(l_4,l_1,l_2)
  = 0$ and that $\tau(l_1,l_2,l_3)$ is constant when $(l_1,l_2,l_3)$ moves but
  the dimensions of $l_i\cap l_j$, $i,j=1,2,3$, and $l_1 \cap l_2 \cap l_3$
  do not change.

  The function $\delta$ is locally constant on
  $U_{\Lambda,p} \times U_{\Lambda,p}$ for a given $p$ because so are
  $\tau_p(l)$ and $\tau_p(l')$.  When $p$ moves we choose a local trivialization
  of $T^*M$ and we consider $l$, $l'$, $\lambda_0(p)$, $\lambda_\Lambda(p)$ as
  subspaces of a fixed symplectic space.  Then
  $\tau_p(l) - \tau_p(l') = \tau(l,l',\lambda_0(p)) -
  \tau(l,l',\lambda_\Lambda(p))$ is constant for $l,l'$ fixed and
  $\lambda_0(p)$, $\lambda_\Lambda(p)$ remaining transverse to $l,l'$.
\end{proof}

\begin{proposition}\label{prop:monodr=sign}
  Let $F\in \Derb_{(\Lambda)}(\cor_M)$.   For
  $p\in \Lambda$ let $U_{\Lambda,p}^i$ and $U_{\Lambda,p}^j$ be two components of
  $U_{\Lambda,p}$ (see~\eqref{eq:def_ULambdai}). Then
  \begin{itemize}
  \item [(a)] for $l \in U_{\Lambda,p}^i$ and $l' \in U_{\Lambda,p}^j$ we have
    $m_\varphi(F)_l \simeq m_\varphi(F)_{l'}[(i-j)/2]$,
  \item [(b)] if $\pi_1(U_{\Lambda,p}^i) = \Z/2\Z$, the monodromy of
    $m_\varphi(F)|_{U_{\Lambda,p}^i}$ along the non trivial loop is the
    multiplication by $-1$.
  \end{itemize}
\end{proposition}
By the discussion before~\eqref{eq:def_ULambdai} we have
$\pi_1(U_{\Lambda,p}^i) = \Z/2\Z$ if $|i|$ is not maximal, that is,
$|i|\not= \dim M - \dim(\lambda_0(p) \cap \lambda_\Lambda(p))$.

\begin{proof}
  The point~(a) is already proved in~\cite{KS90} and stated here as
  Proposition~\ref{prop:inv_microgerm}. It will be recovered in the course of
  the proof of~(b).

  \sui(i) Since $m_\varphi(F)$ is locally constant on $U_\Lambda$ and $j-i$ is
  well-defined in a neighborhood of $U_{\Lambda,p}^i \times U_{\Lambda,p}^j$ by
  Lemma~\ref{lem:diff_Maslov}, we can assume for the proof of~(a) that $p$ is a
  generic point of $\Lambda$.  This also works for the proof of~(b) since any
  loop in $U_{\Lambda,p}^i$ can be deformed into a loop in a nearby fiber
  $U_{\Lambda,q}^{i'}$.

  Hence we assume that $\Lambda = T^*_NM$ in a neighborhood of $p$, for some
  submanifold $N \subset M$.   By Example~\ref{ex:shift} we know that $\cor_N$ is simple.  By
  Lemma~\ref{lem:simple_local} there exists a neighborhood $\Omega$ of $p$ in $T^*M$
  such that $F$ is isomorphic to $\cor_N \otimes L_M \simeq \Z_N \tens_\Z L_M$ in
  $\Derb(\cor_M;\Omega)$, for some $L\in \Derb(\cor)$.  Hence
  $m_\varphi(F) \simeq m_\varphi(\Z_N)\tens_\Z L_{U_\Lambda}$ and we can assume that
  $\cor=\Z$ and $F=\Z_N$.

  \sui(ii) We take coordinates $(x_1,\ldots,x_n)$ so that
  $N = \{x_1 =\cdots= x_k = 0\}$ and $p= (0;1,0)$. We identify $(\lag_M^0)_p$
  with a space of matrices as in~\eqref{eq:LM0=matsym}.  Then $U_{\Lambda,p}$ is
  the space of symmetric matrices $A$ such that $\det(A_k) \not= 0$, where $A_k$
  is the matrix obtained from $A$ by deleting the $k$ first lines and columns.
  The component $U_{\Lambda,p}^i$ is defined by $\sgn(A_k) = i$. We can choose a
  base point $B \in U_{\Lambda,p}^i$ represented by a diagonal matrix $B$ with
  entries $0$, $1$ or $-1$.  More precisely the diagonal consists of $k$ $0$'s,
  $\alpha$ $1$'s and $\beta$ $-1$'s.  We have $\alpha - \beta = i$, hence
  $2\alpha = i + n-k$.  Since $\pi_1(U_{\Lambda,q}^i) \not=0$ we also have
  $\alpha$, $\beta \geq 1$.

  We choose $a$, $b$ such that $B_{aa} = 1$ and $B_{bb}=-1$.  For
  $\theta \in [0,2\pi]$, we define the matrix $B(\theta)$ which is equal to $B$
  except
$$
\begin{pmatrix}
B_{aa}(\theta) & B_{ab}(\theta) \\
B_{ba}(\theta) & B_{bb}(\theta)
\end{pmatrix}
=
\begin{pmatrix}
\cos(\theta) & \sin(\theta) \\ 
\sin(\theta) & -\cos(\theta)
\end{pmatrix} .
$$
Then $\gamma\cl \theta \mapsto B(\theta)$ defines a non trivial loop in
$U_{\Lambda,p}^i$ and we want to prove that the monodromy of $m_\varphi(\Z_N)$
around $\gamma$ is the multiplication by $-1$.  Since $m_\varphi(\Z_N)$ has
stalk $\Z$ up to some shift, we only have to check that this monodromy is not
trivial.

\sui(iii) We define $\varphi\cl [0,2\pi] \times M \to \R$ by
$\varphi(\theta, \ul x) = x_1 + \ul x \, B(\theta) \, {}^t\ul x$.  Then
$C_\theta = \{\varphi_\theta\geq 0\} \cap N$ is a quadratic cone and the pair
$(N,C_\theta)$ is homotopically equivalent to the pair $(N,V_\theta)$, where
$V_\theta$ is the subspace
$$
V_\theta = \langle e_\theta, e_p; p=k+1,\ldots, k+\alpha, \, p\not= a \rangle,
$$
of dimension $\alpha$ with
$e_\theta = (0, \underset{\scriptscriptstyle a}\cos(\frac{\theta}{2}),0,
\underset{\scriptscriptstyle b}\sin(\frac{\theta}{2}) , 0)$ and
$e_p = (0,\underset{\scriptscriptstyle p}{1},0)$.  The stalk of
$m_\varphi(\Z_N)$ at $B(\theta) \in U_{\Lambda,p}^i$ is
\begin{equation*}
m_\varphi(\Z_N)_{B(\theta)} \simeq  \rsect_{\{\varphi_\theta\geq 0\}}(\Z_N)
\simeq (H^{n-k-\alpha}_{V_\theta}(\Z_N))_0[k+\alpha-n]
\end{equation*}
and a non zero germ $s_\theta \in m_\varphi(\Z_N)_{B(\theta)}$ gives a choice of
relative orientation of $V_\theta$.   In particular a non zero section of $m_\varphi(\Z_N)$ defined on some
neighborhood of $\gamma$ would induce relative orientations of all $V_\theta$
together for $\theta \in [0,2\pi]$, varying continuously with $\theta$ and coinciding
for $\theta=0$ and $\theta=2\pi$, which is impossible.  Hence the monodromy of
$m_\varphi(\Z_N)$ is not $1$, which proves~(b).

The part~(a) follows from the fact that $m_\varphi(\Z_N)_{B(\theta)}$ is
concentrated in cohomological degree $n-k-\alpha = (n-k)/2 - i/2$.
\end{proof}

Let $M$ be a manifold and $L \subset T^*M$ a closed Lagrangian submanifold.  We
set $\shl = \lag_M|_L$.  We have already defined the Gauss map
$g \colon L \to \shl$ as $g(p) = \lambda_\Lambda(p)$.  We have also considered
the tangent to the vertical fiber and we see it now as a map
$v \colon L \to \shl$, $p \mapsto \lambda_0(p)$.  We can define the same maps
after stabilisation.  For an integer $N$ we let $V_N$ be the symplectic vector
space $V_N = \C^N$ and we let $l_0^N = \R^N$, $l_1^N = i\R^N$ be two Lagrangian
subspaces. We define $\shl_N \to L$ the fiber bundle whose fiber at $p\in L$ is
$\lag(T_pT^*M \oplus V_N)$.  We extend $g$ and $v$ into two sections
$g_N, v_N \colon L \to \shl_N$ defined by
$g_N(p) = \lambda_\Lambda(p) \oplus l_1^N$ and
$v_N(p) = \lambda_0(p) \oplus l_0^N$.  Taking the limit for $N\to \infty$, we
set $\shl_\infty = \varinjlim_N \shl_N$ and obtain the sections $g_\infty$ and
$v_\infty$ of $\shl_\infty$.  As recalled in~\S\ref{sec:obstr_classes} it is
proved in~\cite{Gi88} that $\Lambda$ has a local generating function if and only
if the sections $g_\infty$ and $v_\infty$ are homotopic.

In this situation it is classical to consider the obstructions classes to find a
homotopy of sections between $g_\infty$ and $v_\infty$. The $i^{th}$ class
belongs to $H^i(L; \pi_i(U/O))$; we denote it by $\mu_i^{gf}(L)$ (more
precisely, these classes are defined inductively and we need the vanishing of
the first $i-1$ classes to define the $i^{th}$ class).  Let us assume that $L$
has a triangulation and denote by $S_k(L)$ the $k$-skeleton of $L$.  Then
$\mu^{gf}_i=0$ for $i=1,\ldots,k$, if and only if there exists a homotopy
between $g_\infty|_{S_k}$ and $v_\infty|_{S_k}$.

\smallskip

The question of finding an isotopy between $g$ and $v$ is related to finding a
section of the map $U_L \to L$, where   $U_L$ is the open subset of
$\lag_M|_L$ introduced in~\eqref{eq:def_U_Lambda}.  Indeed, for $p\in L$ and
$l \in U_{L,p}$ the open subset $U(l)$ of $\lag_{M,p}$ consisting of Lagrangian
subspaces transverse to $l$ is an affine chart of $\lag_{M,p}$ isomorphic to a space
of symmetric matrices. It has a natural structure of affine space.  Since $g(p)$ and
$v(p)$ belong to $U(l)$, we obtain an isotopy
$h_l \colon t\mapsto t g(p) + (1-t) v(p)$, $t\in [0,1]$.  Hence any section of the
map $\sigma \colon U_L \to L$ over a subset $S$ of $L$ gives an isotopy between
$g|_S$ and $v|_S$.

For an integer $N$ we set $\Xi_N = \dT^*_{\R^N}\R^{N+1}$ and
$\xi_N = (\ul0,0; \ul0,1) \in \Xi_N$.  For $F \in \Der(\cor_M)$ we consider
$F\etens \cor_{\R^N} \in \Der(\cor_{M\times \R^{N+1}})$.  Since
$F\etens \cor_{\R^N} \simeq \oim{i} \opb{p} F$, where
$p \colon M \times \R^N \to M$ is the projection and
$i \colon M \times \R^N \to M\times \R^{N+1}$ the inclusion, we have
$\SSi(F \etens \cor_{\R^N}) = \SSi(F) \times \Xi_N$.  If
$\dot\SSi(F) = \Lambda$, then $\dot\SSi(F\etens \cor_{\R^N})$ contains
$\Lambda \times \Xi_N$ as an open set and $\Lambda \times \Xi_N$ is a
neighborhood of $\Lambda \simeq \Lambda \times \{\xi_N\}$ which retracts to
$\Lambda$.  We can identify $\lag(M \times \R^N)|_{\Lambda \times \{\xi_N\}}$
with the Lagrangian Grassmannian $\shl_N$ of the stabilisation of $T^*M$
introduced above.

\smallskip

Summing up the discussion we obtain the following result.  We assume that
$\Lambda$ is triangulated.  If the map
$\sigma \colon U_{\Lambda \times \Xi_N} \to \Lambda \times \Xi_N$ has a section
over the $k$-skeleton of $\Lambda \times \{\xi_N\}$, then
$\mu_i^{gf}(\Lambda) = 0$ for $i\leq k$.

\begin{corollary}\label{cor:mush_et_mu}
  Let $\Lambda \subset \dT^*M$ be a closed conic Lagrangian submanifold.  We assume
  that there exists $\shf \in \kss(\cor_\Lambda)(\Lambda)$ which is simple along
  $\Lambda$.  Then $\mu_1^{gf}(\Lambda) = 0$.  Moreover, if  
  $\cor = \Z$, then $\mu_2^{gf}(\Lambda) = 0$.
\end{corollary}
\begin{proof}
(i)  We choose a triangulation of $\Lambda$.  By the discussion before the
  corollary it is enough to prove that, for $N$ big enough, the map
  $\sigma \colon U_{\Lambda \times \Xi_N} \to \Lambda \times \Xi_N$ has a
  section over the $k$-skeleton of $\Lambda \times \{\xi_N\}$, for $k=1,2$.

  We consider the connected components of
  $U_{\Lambda \times \Xi_N}|_{\Lambda \times \{\xi_N\}}$, that we denote by $U_a^N$,
  $a\in A_N$.  We let $U_a$, $a\in A$, be the connected components of $U_\Lambda$.
  To prove the vanishing of $\mu_1$ it is enough to see that there exists $N$ and
  $U_a^N$ such that $\sigma|_{U_a^N} \colon U_a^N \to \Lambda$ is   surjective with connected fibers (then
  it is possible to find a section on the $1$-skeleton, since $U_a^N$ is an open
  subset of a fiber bundle over $\Lambda$).

  \sui(ii) By Remark~\ref{rem:def_microgerm} we can define
  $m_\varphi(\shf) \in \Dloc(\cor_\Lambda)$.  Since $\shf$ is simple,
  $m_\varphi(\shf)$ is (locally) concentrated in one degree with germs $\cor$ in this
  degree.  Hence, for any $a\in A$ there exists $d_a = d_a(\shf) \in \Z$ such that
  $m_\varphi(\shf)|_{U_a}[d_a]$ is locally constant with germs $\cor$.

  Let $a,a' \in A$ and $p\in \Lambda$.  We recall that the connected components
  of $U_{\Lambda,p}$ are distinguished by the inertia index and that we denote
  by $U_{\Lambda,p}^i$ the component with index $i$.  We assume that
  $U_{\Lambda,p}^i$ is a component of $U_a \cap U_{\Lambda,p}$ and
  $U_{\Lambda,p}^{i'}$ is a component of $U_{a'} \cap U_{\Lambda,p}$.  By
  Proposition~\ref{prop:monodr=sign}-(a) this implies $i-i' = 2(d_a-d_{a'})$.
  We thus obtain
  \begin{itemize}
  \item [(ii-a)] $U_a$ cannot intersect $U_{\Lambda,p}$ in more than one
    connected component,
  \item [(ii-b)] if $d_a(\shf) = d_{a'}(\shf)$ and
    $\sigma(U_a) \cap \sigma(U_{a'}) \not= \emptyset$, then $U_a = U_{a'}$.
  \end{itemize}
  Applying (ii-a) to $\shf\etens \cor_{\R^N}$ (writing abusively $\cor_{\R^N}$ for
  $\kssfunc_{\Xi_N}(\cor_{\R^N})$) we obtain that $\sigma|_{U_a^N}$ has connected
  fibers, for any component $U_a^N$ introduced in~(i).

  \sui(iii) We denote by $V_b^N$, $b\in B_N$, the components of $U_{\Xi_N}$ and by
  $d_b(\cor_{\R^N})$ (as in~(ii)) the cohomological degree such that
  $m_{\varphi_N}(\cor_{\R^N})$ has germs $\cor$ in degree $d_b(\cor_{\R^N})$, for a
  function $\varphi_N \colon U_{\Xi_N} \times \R^{N+1} \to \R$ similar to $\varphi$.
  We can see directly on the formula of Proposition~\ref{prop:SSmicrogerm} that
  $d_b(\cor_{\R^N})$ takes the values $0$, $1$,\dots,$N$ (more precisely
  $H^*_{\{f\geq 0\}}\cor_{\R^N}$ is concentrated in degree $N-i$ when
  $f(\ul x) = x_{N+1} + q(x_1,\ldots,x_N)$ and $q$ is a non degenerate quadratic form
  with $i$ negative eigenvalues).  We can then identify $B_N$ with $\{0,\ldots,N\}$
  and we get $d_b(\cor_{\R^N}) = b$.

  We remark that the product of two Lagrangian subspaces gives a natural inclusion of
  $U_\Lambda \times U_{\Xi_N}$ in $U_{\Lambda \times \Xi_N}$.  For $a \in A$ and
  $b\in B_N$ the component $U_c^N$ of $U_{\Lambda \times \Xi_N}$ containing
  $U_a \times V_b^N$ satisfies
  $d_c(\shf\etens \cor_{\R^N}) = d_a(\shf) + d_b(\cor_{\R^N}) = d_a(\shf) + b$.

  We set $d' = \max\{d_a(\shf);\, a\in A\}$, $d'' = \min\{d_a(\shf);\, a\in A\}$ and
  we choose $N$ bigger than $d' - d''$.  For $a\in A$ we let $U_{c(a)}^N$,
  $c(a) \in A_N$, be the component which contains $U_a \times V_{d'-d_a(\shf)}^N$.
  Then $d_{c(a)}(\shf\etens \cor_{\R^N}) = d'$, for all $a\in A$.  Since
  $\sigma(U_{c(a)}^N)$ contains  
  $\sigma(U_a) \times \Xi_N$, the open subsets $\sigma(U_{c(a)}^N)$ cover
  $\Lambda \times \Xi_N$.  By~(ii-b) we deduce that these components $U_{c(a)}^N$ are
  in fact a single component.

  By the final remark of~(i) this proves that $\mu_1 = 0$.

  \sui(iv)   Now we assume $\cor = \Z$ and prove
  $\mu^{gf}_2 =0$. Let us set
  $U_0^N =U_{c(a)}^N \cap \sigma_{\Lambda \times \Xi_N}^{-1}(\Lambda \times
  \{\xi_N\})$ ($U_{c(a)}^N$ is the component of $U_{\Lambda \times \Xi_N}$ found
  in~(iii)).  Hence
  $\sigma := \sigma_{\Lambda \times \Xi_N}|_{U_0^N} \colon U_0^N \to \Lambda$ is
  surjective with connected fibers.  We recall that the components of $U_{\Lambda,p}$
  have fundamental group $\Z/2\Z$ except the two with extremal inertia index.  Hence,
  up to taking the product with $V_1^2$ (the component of $U_{\Xi_2}$ with index $1$
  -- see~(iii)), we can assume that the fibers of $\sigma|_{U_0^N}$ have fundamental
  group $\Z/2\Z$.

    Let us recall how $\mu^{gf}_2$ is
  defined.  We assume that our triangulation is fine enough so that
  $\pi_1(\sigma^{-1}(T)) = \pi_1(\sigma^{-1}(p)) = \Z/2\Z$ for any triangle $T$ and
  $p\in T$.  We choose a section of $\sigma$ on the $1$-skeleton, say
  $i \colon S_1(\Lambda) \to U_0^N$.  Then $\mu^{gf}_2$ is the obstruction to extend
  it to the $2$-skeleton and is defined as follows.  The boundary of each triangle
  $T$ gives a loop $i(\partial T)$ in $\sigma^{-1}(T)$, hence an element
  $c(T) \in \pi_1(\sigma^{-1}(T)) = \Z/2\Z$.  Then $i$ can be extended to the
  $2$-skeleton if and only if the chain $c \colon T \mapsto c(T)$ vanishes.  It is
  easy to see that $c$ is a cocycle and, by definition $\mu^{gf}_2 = [c]$.  If
  $[c] = 0$, we write $c = \partial b$ where $b$ is a $1$-chain and we can modify $i$
  by $b$ and obtain a new section $i' \colon S_1(\Lambda) \to U_0^N$ which can be
  extended to $S_2(\Lambda)$.

  Now we see how we can use our sheaf to define $i$ such that the chain $c$ vanishes.
  By Proposition~\ref{prop:monodr=sign}-(b) the sheaf
  $G = m_\varphi(\shf \etens \Z_{\R^N})|_{U_0^N }$ is locally constant with germs
  $\Z$ and its restriction to the fibers has monodromy $-1$.  We first define $i$ on
  $S_0(\Lambda)$ arbitrarily.  For each $p \in S_0(\Lambda)$ we also choose a
  generator $u_p$ of $G_{i(p)} \simeq \Z$. We have two such generators $u_p$ and
  $-u_p$.

  Let $E \subset S_1(\Lambda)$ be an edge with boundaries $p$, $p'$.  The
  fundamental group of $\sigma^{-1}(E)$ is $\pi_1(\sigma^{-1}(E)) =
  \Z/2\Z$. Hence we have two sections $j$, $j'$ of $\sigma$ over $E$ up to
  homotopy such that $j(p) = j'(p) = i(p)$ and $j(p') = j'(p') = i(p')$.  Then
  $\opb{j}(G)$ is a constant sheaf on $E$ and we have a canonical isomorphism
  $$
  a_j \colon G_{i(p)} \simeq (\opb{j}G)_p \isofrom \sect(E; \opb{j}G)
  \isoto (\opb{j}G)_{p'} \simeq G_{i(p')} 
  $$
  and a similar one $a_{j'}$.  Since $G$ has monodromy $-1$, we have
  $a_j = - a_{j'}$.  Hence we can choose one section $j$ or $j'$, that we call
  $i$, such that $a_i(u_p) = u_{p'}$.

  We do this for all edges and we obtain $i \colon S_1(\Lambda) \to U_0^N$. With
  this definition the monodromy of $G$ along $i(\partial T)$ is $1$, for any
  triangle $T$.  Using again $\pi_1(\sigma^{-1}(T)) = \Z/2\Z$ and the fact that
  $G$ has monodromy $-1$ along the non trivial loop, we deduce that
  $i(\partial T)$ is a trivial loop. Hence we can extend $i$ to
  $S_2(\Lambda)$. This proves $\mu^{gf}_2 = 0$.
\end{proof}

\part{Convolution and microlocalization}
\label{chap:conv_mic}

Let $N$ be a manifold and $\Lambda \subset \dT^*N$ a locally closed conic
Lagrangian submanifold.  As explained in Remark~\ref{rem:description_KSstack}
the objects of $\kss^s(\cor_\Lambda)$ are described by simple sheaves along a
covering $\{\Lambda_i\}_{i\in I}$ of $\Lambda$, say
$F_i \in \Derb_{(\Lambda_i)}(\cor_N)$, and gluing ``isomorphisms''
$u_{ji}\in H^0(\Lambda_{ij};\muhom(F_i,F_j)|_{\Lambda_{ij}})$.  For a given
object of $\kss^s(\cor_\Lambda)$ we want to find a representative in
$\Derlb_{(\Lambda)}(\cor_N)$, or better, $\Derlb_{[\Lambda]}(\cor_N)$.  For this
we would like to glue the $F_i$'s in the category $\Derlb(\cor_N)$ instead of
$\kss^s(\cor_\Lambda)$.  A first step for this is to find other representatives
of the $\kssfunc_{\Lambda_i}(F_i)$'s for which the $u_{ji}$ arise from morphisms
in $\Derlb(\cor_N)$.  In this part we introduce a functor, $\Psi$, which gives
an answer to this question (see Proposition~\ref{prop:muhom=hompsi} and
Corollary~\ref{cor:kssfuncPsi} below).  This functor is a variation on
Tamarkin's projector of~\S\ref{sec:Tamproj}.  To define $\Psi$ we need to choose
a direction on $N$ and we assume that $N$ is a product $N=M \times\R$.  For an
open subset $U$ of $N$ we define
$\Psi_U \colon \Der(\cor_U) \to \Der(\cor_{U\times \mo]0,+\infty[})$ with the
following properties: setting  
$\Psi_U^\varepsilon(F) = \Psi_U(F)|_{U \times \{\varepsilon\}}$ for
$\varepsilon>0$, we have, for $F,G \in \Der(\cor_U)$ with a microsupport
contained in $T^*M \times \{\tau\geq 0\}$ ($\tau$ the fiber coordinate of
$T^*\R$ -- see~\eqref{eq:def_tau_positif} below),
\begin{itemize}
\item [(i)]
  $\dot\SSi(\Psi_U^\varepsilon(F)) = \dot\SSi(F) \cup
  T_\varepsilon(\dot\SSi(F))$, where $T_\varepsilon$ is the translation by
  $\varepsilon$ along $\R$,
\item [(ii)] $\Psi_U^\varepsilon(F)$ is isomorphic to $F$ along $\dot\SSi(F)$
  in the sense that we have a triangle
  $\Psi_U^\varepsilon(F) \to F \to H \to[+1]$ with
  $\SSi(H) \cap \dot\SSi(F) = \emptyset$,
\item [(iii)] if $F \simeq G$ in $\Der(\cor_U; \dT^*U)$, then
  $\Psi_U(F) \simeq \Psi_U(G)$,
\item [(iv)]
  $H^0(\dT^*V; \muhom(F,G)) \simeq \varinjlim_{\varepsilon\rightarrow 0}
  \Hom(\Psi_U^\varepsilon(F)|_V, \Psi_U^\varepsilon(G)|_V)$, if $V$ is a
  relatively compact open subset of $U$.
\end{itemize}
The property~(iv) will be used in the next part to glue representatives of objects of
$\kss^s(\cor_\Lambda)$ as follows.  For $W\subset M\times\R$ we set
$\Lambda_W = \Lambda \cap T^*W$. We are given
$\shf \in \kss^s(\cor_\Lambda)(\Lambda_W)$, a covering $W =W_1 \cup W_2$ and
$F_i \in \Der(\cor_{W_i})$ representing $\shf|_{\Lambda_{W_i}}$.  We set
$U = W_1 \cap W_2$.  We then have isomorphisms
$\kssfunc_{\Lambda_U}(F_1|_U) \simeq \shf|_{\Lambda_U} \simeq
\kssfunc_{\Lambda_U}(F_2|_U)$, hence a section of $H^0\muhom(F_1,F_2)$ over $\dT^*U$.
Using the property~(iv) we deduce an isomorphism
$\Psi_U^\varepsilon(F_1) \simeq \Psi_U^\varepsilon(F_2)$ and we can glue
$\Psi_U^\varepsilon(F_1)$ and $\Psi_U^\varepsilon(F_2)$ into a representative of
$\shf$ over $W$.

With this procedure we can construct sheaves representing objects of
$\kss^s(\cor_\Lambda)(\Lambda_0)$ when $\Lambda_0 \subset \Lambda$ is of the type
$\Lambda \cap T^*U_0$, $U_0$ open in $M\times \R$.  Unfortunately we will need to
consider more general open subsets $\Lambda_0$ in~\S\ref{sec:quant_orbcat}.  Our
$\Lambda_0$ will not be a union of subsets $\Lambda \cap T^*W_i$, but only a union of
connected components of $\Lambda \cap T^*W_i$.  This means that
$\pi_{M\times \R}(\Lambda_0) \cap \pi_{M\times \R}(\partial\Lambda_0)$ is a priori
non empty. Let $p=(x;\xi) \in \partial\Lambda_0$, $\Xi$ a neighborhood of $p$ in
$\Lambda$ and $U = \pi_{M\times \R}(\Xi)$,
$U_0 = \pi_{M\times \R}(\Xi\cap \Lambda_0)$.   Near $x$ we will have to consider sheaves of the
form $F = \rsect_{U_0}(F')$ with $F'\in \Der(\cor_U)$, $\SSi(F') = \Xi$.  For sheaves
of this kind the above formula~(iv) may still hold but the microsupport of $F$ is
bigger than $\Xi \cap T^*U_0$.  If $G = \rsect_{U_1}(G')$ is another sheaf of the
same type, what we really need in~(iv) is not $H^0(\dT^*V; \muhom(F,G))$ (for some
$V\subset U$) but $H^0(\dT^*V \cap \Lambda_0; \muhom(F,G)|_{\Lambda_0})$.  In fact,
for generic open subsets $U_0, U_1$ (in the sense that the conormal bundles of their
boundaries are well positioned with respect to $\Lambda$), these two groups are
isomorphic.  We check this in~\S\ref{sec:dblsh}, where we introduce a category of
sheaves which are locally of the form $\Psi_U(\rsect_{U_0}(F'))$.

\bigskip

We introduce some notations.  We set for short $\rspos = \mo]0,+\infty[$ and
$\rpos = [0,+\infty[$. We usually endow the factor $\R$ in $M\times \R$ with the
coordinate $t$.  We will need an extra parameter, usually denoted $u$, running
over $\rpos$ or $\rspos$. The associated coordinates in the cotangent bundles
are $(t;\tau)$ for $T^*\R$ and $(u;\upsilon)$ for $T^*\rspos$.  We set
$T^*_{\tau\geq 0}\R = \{ (t,\tau) \in T^*\R;$ $\tau\geq 0\}$ and we define
$T^*_{\tau>0}\R$ similarly.  For a manifold $M$ and an open subset
$U\subset M\times\R$ we define
\begin{equation}
  \label{eq:def_tau_positif}
  \begin{alignedat}{2}
T^*_{\tau\geq 0}U &= (T^*M \times T^*_{\tau\geq 0}\R) \cap T^*U, 
& \qquad T^*_{\tau\leq 0}U & = (T^*_{\tau\geq 0}U)^a ,  \\
T^*_{\tau >0}U &= (T^*M \times T^*_{\tau >0}\R) \cap T^*U ,
& \qquad T^*_{\tau < 0}U & = (T^*_{\tau > 0}U)^a.
\end{alignedat}
\end{equation}
\begin{definition}\label{def:derbtp}
  Let $U$ be an open subset of $M\times\R$.    For $* = \emptyset$, $\mathrm{b}$ or
  $\mathrm{lb}$, we let $\Der^*_{\tau >0}(\cor_U)$ (resp.
  $\Der^*_{\tau \geq 0}(\cor_U)$) be the full subcategory of $\Der^*(\cor_U)$ of
  sheaves $F$ satisfying $\dot\SSi(F) \subset T^*_{\tau >0}U$ (resp.
  $\SSi(F) \subset T^*_{\tau \geq 0}U$).
\end{definition}

\section{The functor $\Psi$}\label{sec:func_psi}

The convolution product is a variant of the ``composition of kernels''
considered in~\cite{KS90} (denoted by $\circ$ --
see~\eqref{eq:def_compo_faisceaux}).  It is used in~\cite{T08} to study the
localization of $\Der(\cor_{M\times\R})$ by the objects with microsupport in
$T^*_{\tau\leq 0}(M\times\R)$, in a framework similar to the present
one. Namely, Tamarkin proves that the functor
$F\mapsto \cor_{[0,+\infty[} \star F$ is a projector from
$\Der(\cor_{M\times\R})$ to the left orthogonal of the subcategory
$\Der_{T^*_{\tau\leq 0}(M\times\R)}(\cor_{M\times\R})$ of objects with
microsupport in $T^*_{\tau\leq 0}(M\times\R)$ (see~\cite{GS11} for a survey).
We will use a variant of Tamarkin's definition.

We will use the product $\star$ in the following special situation.  We define
the subsets of $\R\times \rspos$:
\begin{equation}\label{eq:def_cones}
  \begin{split}
\gammaof & = \{(t,u); \; 0\leq t <u\}, \\
  \lambda_0 & = \{0\} \times \rspos, \quad
\lambda_1=  \{(t,u) \in \R\times\rspos; \; t=u\} .
  \end{split}
\end{equation}

\begin{definition}\label{def:conv_gamma}
  Let $M$ be a manifold and let $U \subset M\times\R$ be an open subset.   We define
  $U_\gammaof \subset M\times\R \times \rspos$ by  
\begin{equation*}
U_\gammaof 
= \{ (x,t,u) \in M\times \R \times \rspos;\; \{x\}\times [t-u,t] \subset U\}.
\end{equation*}
For  $F \in \Der(\cor_U)$ and $G \in \Der(\cor_{\R\times \rspos})$ with
$\supp(G) \subset \ol{\gamma}$, we define $G\star F \in
\Der(\cor_{U_\gammaof})$ by
\begin{equation}\label{eq:defGstarF}
G\star F = (\reim{s}(F \letens G))|_{U_\gammaof} ,
\end{equation}
where  
$s = s_U\cl U \times \R \times \rspos \to M\times \R \times \rspos$ is the sum
$s(x,t_1,t_2,u) = (x,t_1+t_2,u)$ (in general $U$ is understood and we write $s$ for
$s_U$).  We define the functor $\Psi_U \cl \Der(\cor_U) \to \Der(\cor_{U_\gammaof})$
by
\begin{equation}\label{eq:defPsiU}
\Psi_U(F) = \cor_{\gammaof} \star F
= (\reim{s}(F\etens \cor_{\{(t,u);\; 0\leq t < u \}}))|_{U_\gammaof}  .
\end{equation}
\end{definition}

\begin{remark}\label{rem:def-Psi}
  (i)   We see easily from the
  definition of $U_\gammaof$ that, for any submanifold $M'$ of $M$, we have
  $$
  U_\gammaof \cap (M' \times\R \times \rspos) = (U \cap (M' \times\R))_\gammaof.
  $$
  In particular
  $U_\gammaof \cap (\{x\} \times\R \times \rspos) = (U \cap (\{x\}
  \times\R))_\gammaof$, for any point $x\in M$, and we can write
  $U_\gammaof = \bigsqcup_{x\in M} \big( \{x\} \times (U \cap (\{x\}
  \times\R))_\gammaof \big)$.

  For a disjoint union $U = \bigsqcup_{i\in I} U_i$ we have
  $U_\gammaof = \bigsqcup_{i\in I} U_{i,\gammaof}$.  Hence we can reduce the
  description of $U_\gammaof$ to the case where $M$ is a point and $U= \mo]a,b[$ is
  an interval of $\R$. Then we have
\begin{equation}\label{eq:Ugamma-interval}
  (]a,b[)_\gammaof = \{ (t,u) \in \R\times\rspos;\; a+u<t<b \}.
\end{equation}

\sui (ii)   We have
$\opb{s}(U_\gammaof) \cap (M\times \R \times \ol{\gammaof}) \subset U \times
\ol{\gammaof}$.  Since $\supp(G) \subset \ol{\gamma}$, it follows that we also have
$G\star F = (\reim{s_{M\times\R}}(F' \letens G))|_{U_\gammaof}$ where
$F' \in \Der(\cor_{M\times\R})$ is any object such that $F'|_U = F$.

\sui (iii) For the same reason the restriction of $s$ to
$\opb{s}(U_\gammaof) \cap (M\times \R \times \ol{\gammaof})$ is a proper map.
Hence we can replace $\reim{s}$ by $\roim{s}$ in~\eqref{eq:defGstarF}.
\end{remark}

\begin{example}\label{ex:calculPsiF}  
  In Fig.~\ref{fig:imagePsi} we give the easiest examples of $\Psi_U(F)$ when
  $M$ is a point and $U=\R$.  The three pictures describe $\Psi_U(F_i)$
  respectively for $F_1=\cor_{[0,+\infty[}$, $F_2=(\cor_{]-\infty,0[})[1]$,
  $F_3=\cor_{[0,1[}$.  We have $\Psi_U(F_1) \simeq \Psi_U(F_2) \simeq \cor_W$,
  where $W = \{(t,u)$; $0\leq t <u\} \subset \R\times\rspos$, and $\Psi_U(F_3)$
  is concentrated in degrees $0,1$, with $H^0\Psi_U(F_3) = \cor_{W_0}$,
  $H^1\Psi_U(F_3) = \cor_{W_1}$, for some locally closed subsets $W_0$, $W_1$ of
  $\R\times \rspos$ (however
  $\Psi_U(F_3) \not\simeq \cor_{W_0} \oplus \cor_{W_1}[-1]$).  We have sketched
  $W$, $W_0$, $W_1$ in the pictures.

  We remark that we have a distinguished triangle
  $\cor_\R \to F_1 \to F_2 \to[+1]$ and Lemma~\ref{lem:suppPsiuL} below implies
  $\Psi_U(\cor_\R) \simeq 0$, hence $\Psi_U(F_1) \simeq \Psi_U(F_2)$ as we have
  said. Setting $\Lambda_{x_0} = \{(x_0;\tau)$, $\tau>0\} \subset \dT^*\R$, for
  $x_0\in\R$, we also have $\dot\SSi(F_1) = \dot\SSi(F_2) = \Lambda_0$ and
  $\kssfunc_{\Lambda_0}(F_1) \simeq \kssfunc_{\Lambda_0}(F_2)$.

  Similar things happen for $F_3$ and
  $F_4 := F_1 \oplus \cor_{[1,+\infty[}[-1]$.  We have
  $\dot\SSi(F_3) = \dot\SSi(F_4) = \Lambda_0 \sqcup \Lambda_1$ and
  $\kssfunc_{\Lambda_0 \sqcup \Lambda_1}(F_3) \simeq \kssfunc_{\Lambda_0 \sqcup
    \Lambda_1}(F_4)$.  However $\Psi_U(F_3) \not\simeq \Psi_U(F_4)$; indeed
  $\Psi_U$ is additive and the result for $F_1$ gives here
  $\Psi_U(F_4) \simeq \cor_W \oplus \cor_{W'}[-1]$, where $W' = \{(t,u)$;
  $1\leq t <1+u\}$.  What remains true is that, for $\varepsilon>0$ small enough
  (here $\varepsilon \leq 1$),
  $\Psi_U(F_3)|_{\R \times \mo]0,\varepsilon[} \simeq \Psi_U(F_4)|_{\R \times
    \mo]0,\varepsilon[}$.

  In Corollary~\ref{cor:kssfuncPsi} we will see: for a closed conic Lagrangian
  submanifold $\Lambda \subset T^*_{\tau >0}U$ and
  $F,G \in \Der_{[\Lambda]}(\cor_U)$, if
  $\kssfunc_\Lambda(F) \simeq \kssfunc_\Lambda(G)$, then
  $\Psi_U(F)|_V \simeq \Psi_U(G)|_V$, for some open subset $V$ of $U_\gammaof$ such
  that $(U\times \{0\})\sqcup V$ is open in $U \times \rpos$.  So $\Psi_U$ gives
  canonical representatives of objects of $\kss(\cor_\Lambda)$, but with microsupport
  ``doubled'' (see Remark~\ref{rem:SSPsiF} -- the problem is addressed
  in~\S\ref{sec:transl_micsup}).
\end{example}

\begin{figure}[ht]

\begin{tikzpicture}
\draw (-2,-0.5) -- (2,-0.5) ;
\draw (0,-0.5) node {$\bullet$} ;
\draw [very thick] (0,-0.5) -- (2,-0.5) ;
\node at (1,-0.8) {$\cor_{[0,\infty[}$};
\coordinate (A) at (0,0) ; \coordinate (B) at (0,2) ; 
\coordinate (E) at (2,2) ;
\fill [fill=gray!20] (A) -- (B) -- (E) -- cycle;
\draw (-2,0) -- (2,0) ;
\draw (A) -- (B) ;
\draw [dotted] (A) -- (E) ;
\draw (0.5,1.2) node{$\cor_W$} ;
\end{tikzpicture}
\hspace{1.5cm}
\begin{tikzpicture}
\draw (-2,-0.5) -- (2,-0.5) ;
\draw [very thick] (0.07,-.57) arc (-100:-260:.07) ;
\draw [very thick] (-2,-0.5) -- (0,-0.5) ;
\node at (-1,-0.8) {$(\cor_{\mo]-\infty,0[})[1]$};
\coordinate (A) at (0,0) ; \coordinate (B) at (0,2) ; 
\coordinate (E) at (2,2) ;
\fill [fill=gray!20] (A) -- (B) -- (E) -- cycle;
\draw (-2,0) -- (2,0) ;
\draw (A) -- (B) ;
\draw [dotted] (A) -- (E) ;
\draw (0.5,1.2) node{$\cor_W$} ;
\end{tikzpicture}

\vspace{8mm}

\begin{tikzpicture}
  \begin{scope}[xshift=5cm]
  \draw[->] (0,-.5) -- (1,-.5) ;
  \node at (1.1,-.7) {$t$};
  \draw[->] (0,0) -- (1,0) ;
  \draw[->] (0,0) -- (0,1) ;
  \node at (1.1,.2) {$t$};
    \node at (-.3,.8) {$u$};
\end{scope}
\draw (-1,-.5) -- (3,-.5) ;
\draw (0,-.5) node {$\bullet$} ;
\draw [very thick] (1.07,-.57) arc (-100:-260:.07) ;
\draw [very thick] (0,-.5) -- (1,-.5) ;
\node at (0.5,-0.8) {$\cor_{[0,1[}$};
\coordinate (A) at (0,0) ; \coordinate (B) at (0,2) ; 
\coordinate (C) at (1,0) ; \coordinate (D) at (1,1) ; 
\coordinate (E) at (1,2) ; \coordinate (F) at (3,2) ; 
\coordinate (G) at (2,2) ; 
\fill [fill=gray!20] (A) -- (B) -- (E) -- (D) -- cycle;
\fill [fill=gray!20] (C) -- (D) -- (G) -- (F) -- cycle;
\draw (-1,0) -- (3,0) ;
\draw (A) -- (B) ;
\draw [dotted] (A) -- (D) ;
\draw (C) -- (D) ;
\draw [dotted] (D) -- (E) ;
\draw [dotted] (C) -- (F) ;
\draw (D) -- (G);
\draw (0.5,1.2) node{$\cor_{W_0}$} ;
\draw (2,1.2) node{$\cor_{W_1}[-1]$} ;  
\end{tikzpicture}

  \caption{} \label{fig:imagePsi}
\end{figure}

We define the projections
\begin{equation}\label{eq:def_proj}
  \begin{alignedat}{2}
q \cl M\times \R \times \rspos &\to M\times \R ,
& \qquad  (x,t,u) &\mapsto (x,t),  \\
r \cl M\times \R \times \rspos &\to M\times \R ,
 & (x,t,u) &\mapsto (x,t-u) 
  \end{alignedat}
\end{equation}
and, for an open subset $U$ of $M\times\R$, we denote by
\begin{equation}\label{eq:def_projU}
q_U,r_U \cl U_\gammaof \to U 
\end{equation}
the restrictions of $q,r$ to $U_\gammaof$. Using the
notations~\eqref{eq:def_cones} we have
$\cor_{\lambda_0} \star F \simeq \opb{q_U}(F)$ and
$\cor_{\lambda_1} \star F \simeq \opb{r_U}(F)$, for any $F\in \Der(\cor_U)$.
The closed inclusions $\lambda_0 \subset \gammaof$,
$\lambda_1 \subset \ol{\gammaof}$ and
$\lambda_1 \subset \ol{\gammaof} \setminus \lambda_0$ give excision
distinguished triangles
\begin{gather*}
\cor_{\Int(\gammaof)} \to[b''] \cor_\gammaof \to[a] \cor_{\lambda_0} \to[+1],
\qquad
\cor_{\lambda_1} [-1] \to[b] \cor_\gammaof \to \cor_{\ol\gammaof} \to[+1] ,\\
\cor_{\lambda_1} [-1] \to[b'] \cor_{\Int(\gammaof)}
\to \cor_{\ol\gammaof\setminus  \lambda_0} \to[+1] 
\end{gather*}
and we have $b = b'' \circ b'$.  The convolution $- \star F$ turns the morphisms
$a$, $b$, $b'$, $b''$ into morphisms of functors:
\begin{equation}\label{eq:def_alphabeta}
\alpha(F) \cl \Psi_U(F) \to  \opb{q_U}(F),
\qquad
\beta(F) \cl \opb{r_U}(F)[-1] \to \Psi_U(F) , \\
\end{equation}
and $\beta'(F) \cl \opb{r_U}(F)[-1] \to \cor_{\Int(\gammaof)} \star F$, \
$\beta''(F) \cl \cor_{\Int(\gammaof)} \star F \to \Psi_U(F)$.  We have
$\beta(F) = \beta''(F) \circ \beta'(F)$ and two distinguished triangles
\begin{gather}
  \label{eq:dtrPsiq0}
\cor_{\Int(\gammaof)} \star F  \to[\;\;\beta''(F)\;\;] \Psi_U(F) \to[\;\;\alpha(F)\;\;] 
\opb{q_U}(F)  \to[+1] , \\
\label{eq:dtrPsiq1}
 \opb{r_U}(F)[-1] \to[\;\;\beta'(F)\;\;] \cor_{\Int(\gammaof)} \star F
\to \cor_{\ol\gammaof\setminus  \lambda_0} \star F  \to[+1] 
\end{gather}

\begin{lemma}\label{lem:dtgrPsiq}
  For $F\in \Dertpn(\cor_U)$ the morphism $\beta'(F)$ is an isomorphism
  and~\eqref{eq:dtrPsiq0} gives the distinguished triangle
\begin{equation}\label{eq:dtgamqr2}
\opb{r_U}(F)[-1] \to[\;\;\beta(F)\;\;] \Psi_U(F) \to[\;\;\alpha(F)\;\;] 
 \opb{q_U}(F)  \to[+1] .
\end{equation}
\end{lemma}
\begin{proof}
  Since $\beta(F) = \beta''(F) \circ \beta'(F)$ the second part of the lemma follows
  from the claim that $\beta'(F)$ is an isomorphism.  In view of~\eqref{eq:dtrPsiq1}
  we only have to prove that
  $\cor_{\ol\gammaof\setminus \lambda_0} \star F \simeq 0$.  For
  $(x,t,u) \in U_\gammaof$ we have the Cartesian square  
\begin{equation*}
\begin{tikzcd}
\R \ar[r, "i_{(x,t,u)}"] \ar[d] & M\times\R^2\times \rspos \ar[d, "s"] \\
\{(x,t,u)\} \ar[r] & M\times\R\times \rspos ,
\end{tikzcd}
\end{equation*}
where $i_{(x,t,u)}(t') = (x, t', t-t',u)$.  Writing
$i_{(x,t,u)} = (i_x \times i_{(t,u)}) \circ \delta$, with $\delta(t') = (t',t')$,
$i_x(t') = (x,t')$ and $i_{(t,u)}(t') = (t-t',u)$, we deduce by the base change
formula that
\begin{align*}
(\cor_{\ol\gammaof\setminus \lambda_0} \star F)_{(x,t,u)}
  & \simeq \rsect_c(\R;  \opb{i_x} F
    \tens \opb{i_{(t,u)}} \, \cor_{\ol\gammaof\setminus \lambda_0}) \\
& \simeq \rsect_c(\R;  \opb{i_x} F\tens \cor_{[t-u,t[}) .
\end{align*}
Then $G = \opb{i_x} F\tens \cor_{[t-u,t[}$ has compact support and satisfies
$\SSi(G) \subset T^*_{\tau\geq 0}\R$ by Theorems~\ref{thm:iminv}
and~\ref{thm:SSrhom}.  Using Corollary~\ref{cor:Morse} we deduce that
$\rsect_c(\R; G) \simeq 0$.  It follows that
$\cor_{\ol\gammaof\setminus \lambda_0} \star F \simeq 0$.
\end{proof}

Let     $U$ be an open subset of
$M\times\R$ and $V$ an open subset of $U$.  Let $X$ be a submanifold of $M$ and
$U' = U \cap (X\times \R)$. We have
\begin{align}
\label{eq:restr-PsiU-ouvert}
\Psi_V(F|_V) &\simeq (\Psi_U(F))|_{V_\gammaof} ,  \\
\label{eq:restr-PsiU-sousvar}
\Psi_{U'}(F|_{U'}) &\simeq (\Psi_U(F))|_{U'_\gammaof} ,
\end{align}
where the first isomorphism follows from supports estimates as in
Remark~\ref{rem:def-Psi}~(ii) and the second one follows from the base change
formula.

\medskip

In the next lemma we use an analog of the convolution for sets.  For $A \subset
M\times \R$ and $B \subset \R\times \rspos$ we define $B\star A \subset M\times
\R\times \rspos$ by   
\begin{equation}\label{eq:convset}
B\star A = s_{M\times\R}(A\times B) ,
\end{equation}
where $s_{M\times\R}\cl M \times \R^2 \times \rspos \to M\times \R \times \rspos$ is
the sum as in Definition~\ref{def:conv_gamma}.

\begin{lemma}\label{lem:suppPsiuL}
Let $F\in \Der(\cor_U)$ and let $V\subset U$ be an open subset.  We assume
that
\begin{equation}\label{eq:hyp_F_vert_cst}
\text{$F|_{V\cap (\{x\}\times \R)}$ is locally constant, for any $x\in M$.}
\end{equation}
Then $\Psi_U(F)|_{V_\gammaof} \simeq 0$.  As a special case, if $\SSi(F|_V)
\subset T^*_VV$, then $\Psi_U(F)|_{V_\gammaof} \simeq 0$. In particular
$\supp(\Psi_U(F)) \subset (\gammaf \star \dot\pi_U(\dot\SSi(F))) \cap
U_\gammaof$.
\end{lemma}
\begin{proof}
  We set $V_x = V \cap (\{x\}\times \R)$. Then
  $V_\gammaof = \bigsqcup_{x\in M} ( \{x\} \times (V_x)_\gammaof)$ and we have to
  prove $\Psi_U(F)|_{\{x\} \times (V_x)_\gammaof} \simeq 0$, for all $x\in
  M$. By~\eqref{eq:restr-PsiU-sousvar} we have
  $\Psi_U(F)|_{\{x\} \times (V_x)_\gammaof} \simeq \Psi_{V_x}(F|_{V_x})$.  The set $V_x$ is a
  disjoint union of open intervals of $\R$ and $F|_{V_x}$ is constant on each of
  these intervals. A direct computation gives $\Psi_{V_x}(F|_{V_x}) \simeq 0$ and we
  obtain the result.
\end{proof}

\begin{lemma}\label{lem:annulation_qr}
Let $F \in \Der(\cor_U)$.
\begin{itemize}
\item [(i)] We have $\reim{q_U} \epb{q_U} (F) \isoto F$ and
  $\reim{r_U}(\Psi_U(F)) \simeq 0$.
\item [(ii)] If $F\in \Dertpn(\cor_U)$, then $\reim{q_U} \opb{r_U} (F)$
  satisfies~\eqref{eq:hyp_F_vert_cst} (with $V=U$). In particular $\Psi_U(
  \reim{q_U} \opb{r_U} (F)) \simeq 0$.
\item [(iii)] We assume that $U = M \times \R$, that $F\in \Dertpn(\cor_U)$
  and that $\supp(F) \subset M \times [a,+\infty[$ for some $a\in \R$.  Then
  $\reim{q_U} \opb{r_U} (F) \simeq 0$.
\end{itemize}
\end{lemma}
\begin{proof}
  (i) The first morphism is the adjunction for $(\reim{q_U}, \epb{q_U})$.  It is
  an isomorphism because the fibers of $q_U$ are intervals.  Let us prove the
  second isomorphism.  By the base change formula
  (see~\eqref{eq:restr-PsiU-sousvar}) we may as well assume that $M$ is a
  point. Then $U$ is an open subset of $\R$ and we can restrict to one component
  of $U$. Hence we assume $U$ is an interval.  We define
  $r'\cl \R^2 \times\rspos \to \R^2$, $(t_1,t_2,u) \mapsto (t_1,t_2-u)$ and
  $s' \cl \R^2\to \R$, $(t_1,t_2) \mapsto (t_1+t_2)$.  We have the commutative
  diagram
\begin{equation*}
\begin{tikzcd}
\R^2 \times\rspos \ar[r, "r'"] \ar[d, "s"']
& \R^2 \ar[d, "s'"] \\
\R \times\rspos \ar[r, "r"] & \R . 
\end{tikzcd}
\end{equation*}
We let $j\cl U \to \R$ be the inclusion and we set $F' = \eim{j}F$.     Then
$\Psi_U(F) \simeq (\cor_{\gammaof}\star F')|_{U_\gammaof}$.  Let
$q_2 \colon \R \times \R \times\rspos \to \R\times\rspos$ be the projection on the
last two factors. We have
\begin{align*}
\reim{r_U}(\Psi_U(F))
&\simeq \reim{r} ( \reim{s}(F' \etens \cor_{\gammaof}) \tens \cor_{U_\gammaof})  \\
&\simeq \reim{(r\circ s)} ( (F'\etens \cor_{\R\times\rspos}) \tens
 \cor_{\opb{q_2}\gammaof \cap \opb{s}U_\gammaof})  \\
&\simeq \reim{(s'\circ r')}( \opb{r'}(F'\etens \cor_\R) \tens 
 \cor_{\opb{q_2}\gammaof \cap \opb{s}U_\gammaof})  \\
&\simeq \reim{s'} ( (F'\etens \cor_\R) \tens
\reim{r'}( \cor_{\opb{q_2}\gammaof \cap \opb{s}U_\gammaof})).
\end{align*}
Hence it is enough to prove that
$\reim{r'}( \cor_{\opb{q_2}\gammaof \cap \opb{s}U_\gammaof}) \simeq 0$. We write
$U = \mo]a,b \mc[$.  Then we have $U_\gammaof = \{(t,u) \in \R\times\rspos$;
$a+u<t<b\}$.  For any $(t_1,t_2) \in \R^2$ the fiber
$\opb{r'}(t_1,t_2) \cap (\opb{q_2}\gammaof \cap \opb{s}U_\gammaof)$ is
identified with
\begin{align*}
\{u>0; \; (t_1,&t_2+u,u) \in \opb{q_2}\gammaof \cap \opb{s}U_\gammaof\} \\
&= \{u>0; \; 0\leq t_2+u<u \text{ and } (t_1+t_2+u,u) \in U_\gammaof\} \\
&= \{u>0; \;  -t_2\leq u \text{ and }  u < b-t_1-t_2 \},
\end{align*}
where we assume $t_2<0$ and $a<t_1+t_2$ (otherwise the fiber is empty).  Since
$-t_2>0$ we see that the fiber is either empty or a half closed interval. This
implies $\reim{r'}( \cor_{\opb{q_2}\gammaof \cap \opb{s}U_\gammaof}) \simeq 0$,
as required.

\sui (ii)     We choose
$x\in M$ and we set $U_x = U\cap (\{x\}\times \R)$.  By the base change formula we
have
$(\reim{q_U} \opb{r_U} (F))|_{U_x} \simeq \reim{q_{U_x}} \opb{r_{U_x}}
(F|_{U_x})$. Hence we can assume that $M$ is a point and that $U$ is an interval.  It
is enough to prove that $(\reim{q_U} \opb{r_U} (F))|_{]a,b[}$ is constant for any
$a,b \in U$.  We set $V = U \cap \mo]-\infty,a[$, $Z = U\setminus V$. In view of the
distinguished triangle $F_V \to F \to F_Z \to[+1]$ it is enough to see that
$G_1 = (\reim{q_U} \opb{r_U} (F_V))|_{]a,b[}$ and
$G_2 = (\reim{q_U} \opb{r_U} (F_Z))|_{]a,b[}$ are constant.

The maps $q_U$ and $r_U$ restricted to $W = q_U^{-1}(]a,b[) \cap r_U^{-1}(V)$
identify $W$ with $]a,b\mc[ \times V$ and we deduce, by the base change formula, that
$G_1$ is constant with stalks $\rsect_c(V; F|_V)$.

Let $j\cl U \to \R$ be the inclusion and let $q',r' \colon \R^2 \to \R$ be the maps
$(t,u) \mapsto t$, $(t,u) \mapsto t-u$.  We set $F' = \reim{j}(F_Z)$. Then we have
$G_2 \simeq (\reim{q'} (\opb{r'} (F') \tens \cor_{\R\times \rspos}))|_{]a,b[}$ and
$\SSi(F') \subset T^*_{\tau\geq 0}\R$.  Using Theorems~\ref{thm:SSrhom}
and~\ref{thm:iminv} we obtain, with coordinates $(t,u;\tau,\upsilon)$ on $T^*\R^2$:
\begin{align*}
 \SSi(\opb{r'}(F')|_{U\times \R}) &\subset \{(t,u;\tau,-\tau); \;\tau\geq 0\} , \\
 \SSi((\opb{r'} (F') \tens \cor_{\R\times \rspos})|_{U\times \R})
&\subset \{(t,u;\tau, -\tau + \upsilon); \;\tau\geq 0,\, \upsilon \leq 0\} .
\end{align*}
This last set intersects $T^*\R \times T^*_\R\R$ along $T^*_{\R^2}\R^2$.  Since $q'$
is proper on $r'^{-1}(Z) \cap (\R \times \rpos)$ it follows from
Proposition~\ref{prop:oim} that $\SSi(G_2)$ is contained in the zero section.  Hence
$G_2$ is constant.

\medskip\noindent
(iii) By~(ii) $\reim{q_U} \opb{r_U} (F)$ is constant on the fibers $\{x\}
\times \R$ for all $x\in M$.  By the hypothesis its restriction to $M\times
\{a-1\}$ vanishes.  Hence it is zero.
\end{proof}

\begin{lemma}\label{lem:SSPsiF}
  We let $j\colon U \times \rspos \to U\times\R$ be the inclusion.   We set
  $A = (U\times \{0\}) \times_{U\times\R} T^*(U\times\R)$.  We use the maps $q_U,r_U$
  of~\eqref{eq:def_projU} and the notations~\eqref{eq:def_derivee_morph}.  Let
  $F\in \Dertpn(\cor_U)$.  Then
\begin{gather*}
  \SSi(\Psi_U(F))   \,\subset\,
q_{U,d}\opb{q_{U,\pi}}(\SSi(F)) \cup r_{U,d}\opb{r_{U,\pi}}(\SSi(F)) , \\
\begin{split}
\SSi(\reim{j} \Psi_U(F))\cap A \,\subset\, \{(x,t,0&;\xi,\tau,\upsilon); \\
  & (x,t;\xi,\tau) \in \SSi(F) ,\, -\tau \leq \upsilon \leq 0\} .
\end{split}
\end{gather*}
In particular $\reim{j} \Psi_U(F) \isoto \roim{j} \Psi_U(F)$.
\end{lemma}
\begin{remark}\label{rem:SSPsiF}
  If we restrict to a ``slice''
  $U_\gammaof^\varepsilon = U_\gammaof \cap (U \times \{\varepsilon\})$, we obtain
  $\dot\SSi(\Psi_U(F)|_{U_\gammaof^\varepsilon}) \subset (\dot\SSi(F) \cup
  T_\varepsilon( \dot\SSi(F) ) )\cap T^*U_\gammaof^\varepsilon$, where
  $T_\varepsilon$ is the translation
  $T_\varepsilon(x,t;\xi,\tau) = (x,t+c;\xi,\tau)$.
\end{remark}
\begin{proof}
  (i) The first inclusion follows from the triangle~\eqref{eq:dtgamqr2} and the
  triangular inequality for the microsupport.  To prove the second inclusion we
  consider $\gammaof$ as a subset of $\R^2$ rather than $\R \times \rspos$.  We
  also consider the sum $s$, $(x,t_1,t_2,u) \mapsto (x,t_1+t_2,u)$, as a map
  from $U \times\R^3$ to $M\times \R^2$.  Then by the base change formula we
  have $\reim{j} \Psi_U(F) \simeq \reim{s}(F\etens \cor_{\gammaof})$.

  Setting $V = (U\times \mo]-\infty,0]) \cup U_\gammaof$ we see that $s$ is
  proper as a map from $\opb{s}(V) \cap (U \times \ol{\gammaof})$ to $V$.  Hence
  we can use Proposition~\ref{prop:oim} to bound
  $\SSi(\reim{s}(F\etens \cor_{\gammaof})|_V)$.  We see that we only need to
  know $\SSi(F\etens \cor_{\gammaof})$ above $U\times\R\times \{0\}$.  Now
  $\SSi(\cor_\gammaof) \cap T^*_{(t,0)}\R^2$ is empty for $t\not=0$ and is the
  cone $\{-\tau \leq \upsilon \leq 0\}$ for $t=0$.  Hence
  $\SSi(F\etens \cor_{\gammaof}) \cap\{u=0\} \subset \{(x,t_1,0,0;
  \xi,\tau_1,\tau_2,\upsilon)$; $(x,t_1;\xi,\tau_1) \in \SSi(F)$,
  $-\tau_2 \leq \upsilon \leq 0\}$.  We conclude with
  Proposition~\ref{prop:oim}.

  \sui(ii) We deduce $\reim{j} \Psi_U(F) \isoto \roim{j} \Psi_U(F)$.  Let us set
  $Z = U\times \mo]-\infty,0]$ and $W = U\times \mo]0,+\infty[$.  Then
  $\roim{j} \Psi_U(F) \simeq \rsect_W( \reim{j} \Psi_U(F))$ and, by the excision
  distinguished triangle, we are reduced to prove that
  $\rsect_Z( \reim{j} \Psi_U(F)) \simeq 0$.  Since the support of
  $\reim{j} \Psi_U(F)$ is contained in $\ol W$, it is enough to prove
  $(\rsect_Z( \reim{j} \Psi_U(F)))_z \simeq 0$ for any $z\in U\times \{0\}$.
  Going back to the definition of the microsupport, this follows from
  $(z,0;0,-1)\not\in \SSi(\reim{j} \Psi_U(F))$.
\end{proof}

\begin{proposition}\label{prop:decomposition_muhomPsi}
  Let $F, G \in \Dertpn(\cor_U)$.  We set $\Omega = T^*_{\tau>0}U_\gammaof$ for
  short.  Then we have a natural decomposition
\begin{align*}
  \muhom(\Psi_U(F), \Psi_U(G))|_\Omega
&\simeq \muhom(\opb{q_U}(F), \opb{q_U}(G))|_\Omega \\
  &\qquad  \oplus \muhom(\opb{r_U}(F), \opb{r_U}(G))|_\Omega 
\end{align*}
such that the corresponding projection from $\muhom(\Psi_U(F), \Psi_U(G))|_\Omega$ to
$\muhom(\opb{q_U}(F), \opb{q_U}(G))|_\Omega$ coincides with
  $\alpha_2^{-1} \circ \alpha_1$, where
$\alpha_1$, $\alpha_2$ are respectively induced by the morphisms $\alpha(G)$ and
$\alpha(F)$ of~\eqref{eq:def_alphabeta} as follows:
\begin{align*}
  \muhom(\Psi_U(F), \Psi_U(G))|_\Omega
  & \to[\alpha_1] \muhom(\Psi_U(F), \opb{q_U}(G))|_\Omega \\
  & \isofrom[\alpha_2] \muhom(\opb{q_U}(F), \opb{q_U}(G))|_\Omega  .
\end{align*}
\end{proposition}
\begin{proof}
  (i) We set for short $h(A) = \muhom(\Psi_U(F), A)|_\Omega$ for a sheaf $A$ on
  $U_\gammaof$. The triangle~\eqref{eq:dtgamqr2} induces a distinguished
  triangle
  $$
  h(\opb{r_U}(G))[-1] \to h(\Psi_U(G)) \to[\alpha_1] h(\opb{q_U}(G))
  \to[\gamma_1] h(\opb{r_U}(G))
  $$
  on $\Omega$.  We recall the bound
  $\supp \muhom(F_1,F_2) \subset \SSi(F_1) \cap \SSi(F_2)$. By
  Theorem~\ref{thm:iminv} the microsupports of $\opb{q_U}(G)$ and $\opb{r_U}(F)$
  are respectively contained in $\{ (x,t,u;\xi,\tau,0) \}$ and
  $\{ (x,t,u;\xi,\tau,-\tau) \}$. Hence their intersection is contained in
  $\{\tau=0\}$. Since we work on $\Omega = T^*_{\tau>0}U_\gammaof$ it follows
  that the supports of $h(\opb{q_U}(G))$ and $h(\opb{r_U}(G))$ are disjoint,
  hence that $\gamma_1 = 0$. This proves that
  $h(\Psi_U(G)) \simeq h(\opb{q_U}(G)) \oplus h(\opb{r_U}(G))$.
  
  \sui(ii) Now we check that
  $\alpha_2 \colon \muhom(\opb{q_U}(F), \opb{q_U}(G))|_\Omega \to
  h(\opb{q_U}(G))$ is an isomorphism.  Using the triangle~\eqref{eq:dtgamqr2}
  again, we see that the cone of $\alpha_2$ is
  $\muhom(\opb{r_U}(F), \opb{q_U}(G))|_\Omega$ whose support is contained in
  $\SSi(\opb{r_U}(F)) \cap \SSi(\opb{q_U}(G))$.  The same bound as in~(i) shows
  that this set is contained in $\{\tau=0\}$ and does not meet $\Omega$.  Hence
  the cone of $\alpha_2$ vanishes and $\alpha_2$ is an isomorphism.

  We obtain $\muhom(\opb{r_U}(F), \opb{r_U}(G))|_\Omega \isoto h(\opb{r_U}(G))$
  by swapping $q_U$ and $r_U$ in the previous argument.  This concludes the
  proof of the proposition.
\end{proof}

\section{Adjunction properties}\label{sec:adj_prop}

Let $U\subset M\times \R$ be an open subset.   We let
$\Derra(\cor_{U_\gammaof})$ be the full subcategory of $\Dertpn(\cor_{U_\gammaof})$
consisting of $\reim{r_U}$-acyclic objects, that is, the objects $G$ such that
$\reim{r_U}G \simeq 0$.  This is a triangulated category.  By
Lemma~\ref{lem:annulation_qr}~(i) the functor $\Psi_U$ takes values in
$\Derra(\cor_{U_\gammaof})$.  Using Theorem~\ref{thm:oim_open} for the embedding
$U_\gammaof \hookrightarrow U \times \R$ and Proposition~\ref{prop:oim} we see that
the functor $\reim{q_U}$ sends $\Dertpn(\cor_{U_\gammaof})$ into $\Dertpn(\cor_U)$.
Moreover the morphism of functors $\alpha \cl \Psi_U \to \opb{q_U}$
in~\eqref{eq:def_alphabeta} and the adjunction morphism
$\reim{q_U} \epb{q_U} \simeq \reim{q_U} \opb{q_U} [1] \to \id$ induce:
\begin{equation}\label{eq:psiq-id}
  b_U(F) \cl \reim{q_U} \Psi_U (F) [1] \to F,
\qquad \text{for all $F\in \Dertpn(\cor_U)$.}
\end{equation}

\begin{lemma}\label{lem:adj-Psi-q}
The functor $\reim{q_U}[1] \cl \Derra(\cor_{U_\gammaof}) \to \Dertpn(\cor_U)$ is
left adjoint to $\Psi_U \cl \Dertpn(\cor_U) \to \Derra(\cor_{U_\gammaof})$.
In particular we have an adjunction morphism
\begin{equation}\label{eq:id-psiq}
  b'_U(G) \cl G \to \Psi_U \reim{q_U} (G) [1] ,
\qquad \text{for all $G\in \Derra(\cor_{U_\gammaof})$.}
\end{equation}
\end{lemma}
\begin{proof}
Since $\Dertpn(\cor_U)$ and $\Derra(\cor_{U_\gammaof})$ are full
subcategories of $\Der(\cor_U)$ and $\Der(\cor_{U_\gammaof})$, it is enough
to prove
\begin{equation}\label{eq:psiq-id1}
\Hom_{\Der(\cor_{U_\gammaof})}(G, \Psi_U(F))
\simeq \Hom_{\Der(\cor_U)}(\reim{q_U}G[1],F)
\end{equation}
for any $F\in \Dertpn(\cor_U)$ and $G\in \Derra(\cor_{U_\gammaof})$.  Since
$r_U$ is a smooth map with fibers homeomorphic to $\R$ we have a canonical
isomorphism of functors $\epb{r_U} \simeq \opb{r_U}[1]$; hence an adjunction
$(\reim{r_U},\opb{r_U}[1])$. The same holds for $q_U$.  Applying
$\RHom(G,\cdot)$ to~\eqref{eq:dtgamqr2} we obtain the distinguished triangle
\begin{align*}
\RHom(G,\opb{r_U}F[-1]) \to \RHom(G, \Psi_U(F)) \to
\RHom(G, \opb{q_U}F)  \to[+1] .
\end{align*}
The adjunction $(\reim{r_U},\opb{r_U}[1])$ and the hypothesis $G\in
\Derra(\cor_{U_\gammaof})$ give $\RHom(G,\opb{r_U}F[-1]) \simeq 0$. We deduce
$$
\RHom(G, \Psi_U(F)) \isoto \RHom(G, \opb{q_U}F) \simeq \RHom(\reim{q_U}G[1],F),
$$
which implies~\eqref{eq:psiq-id1}.
\end{proof}

\begin{lemma}\label{lem:qpsi-projA}
Let $F\in \Dertpn(\cor_U)$.  Then the morphisms $\Psi_U(b_U(F))$ and
$b'_U(\Psi_U (F))$ are mutually inverse isomorphisms:
\begin{align*}
\Psi_U(b_U(F)) &\cl \Psi_U \reim{q_U} \Psi_U (F) [1] \isoto \Psi_U(F) , \\
b'_U(\Psi_U (F)) &\cl \Psi_U (F) \isoto \Psi_U \reim{q_U} \Psi_U (F) [1] .
\end{align*}
\end{lemma}
\begin{proof}
  (i)   We prove the first
  isomorphism.  We apply $\reim{q_U}[1]$ to the distinguished
  triangle~\eqref{eq:dtgamqr2}.  Since $q_U$ has fibers isomorphic to $\R$ the
  adjunction morphism
  $\reim{q_U} \opb{q_U}(F)[1] \simeq \reim{q_U} \epb{q_U}(F) \to F$ is an isomorphism
  and we obtain the distinguished triangle:
\begin{equation}\label{eq:qpsi-projA1}
L \to \reim{q_U} \Psi_U (F) [1] \to[b_U(F)]  F \to[+1],
\end{equation}
where $L = \reim{q_U} \opb{r_U}(F)$. By Lemma~\ref{lem:annulation_qr}~(ii) we
have $\Psi_U(L) \simeq 0$. Hence applying $\Psi_U$ to~\eqref{eq:qpsi-projA1}
gives the lemma.

\medskip\noindent
(ii) The composition $\Psi_U(b_U(F)) \circ b'_U(\Psi_U (F))$ is the identity
morphism of $\Psi_U(F)$, by general properties of adjunctions.  Hence
the lemma follows from~(i).
\end{proof}

\begin{proposition}\label{prop:Psipresqueff}
We assume that $U = M \times \R$, that $F\in \Dertpn(\cor_U)$ and that
$\supp(F) \subset M \times [a,+\infty[$ for some $a\in \R$.  Then the
adjunction morphism $b_U(F) \cl \reim{q_U} \Psi_U (F) [1] \to F$
of~\eqref{eq:psiq-id} is an isomorphism and for any $G\in \Dertpn(\cor_U)$ we
have
$$
\Hom(F,G) \isoto \Hom(\Psi_U(F), \Psi_U(G)) .
$$
\end{proposition}
\begin{proof}
  By Lemma~\ref{lem:annulation_qr}~(i) and~(iii) we have
  $\reim{q_U} \epb{q_U} (F) \isoto F$ and $\reim{q_U} \opb{r_U} (F) \simeq 0$.
  Hence the first part follows from the distinguished
  triangle~\eqref{eq:dtgamqr2}.  Then the second part is given by the adjunction
  $(\reim{q_U}, \Psi_U)$ of Lemma~\ref{lem:adj-Psi-q}.
\end{proof}

\section{Link with microlocalization}

In this section we prove Proposition~\ref{prop:muhom=hompsi} below which says that
the group of homomorphisms from $\Psi_U(F)$ to $\Psi_U(G)$ is isomorphic to the group
of sections of $\muhom(F,G)$ outside the zero section.  We first deduce from
Proposition~\ref{prop:decomposition_muhomPsi} a morphism from
$\muhom(\Psi_U(F), \Psi_U(G))$ to $\muhom(F,G)$.  Then we use ``boundary values'' of
sheaves on $U_\gammaof$ in the following sense.  Let $U \subset M\times \R$ and
$V \subset M \times\R \times \rspos$ be open subsets satisfying  
\begin{equation}
  \label{eq:valbord_ouvert} 
  \text{$i_U(U) \sqcup j_V(V)$ is open in $M \times\R \times \rpos$,}
\end{equation}
where $i_U,j_V$ are the natural inclusions
\begin{equation}\label{def:iUjU}
  \begin{alignedat}{2}
i_U \cl U  &\to M\times \R \times \R  ,
& \qquad  (x,t) &\mapsto (x,t,0) ,   \\
j_V  \cl V &\to M\times\R \times \R .
  \end{alignedat}
\end{equation}
Then, for $G \in \Der(\cor_{V})$, its boundary value is
$\opb{i_U} \RR j_{V*}(G) \in \Der(\cor_U)$. We remark that we can shrink $V$ as long
as~\eqref{eq:valbord_ouvert} is satisfied:
\begin{lemma}\label{lem:valbord}
  Let $V' \subset V$ be an open subset such that $i_U(U) \sqcup j_{V'}(V')$ is open
  in $M \times\R \times \rpos$.  Then the morphism
  $\RR j_{V*}(G) \to \RR j_{V'*}(G|_{V'})$, obtained from
  $\opb{j_{V'}} \RR j_{V*}(G) \simeq G|_{V'}$ by the adjunction
  $(\opb{j_{V'}}, \RR j_{V'*})$, induces an isomorphism
  $\opb{i_U} \RR j_{V*}(G) \isoto \opb{i_U} \RR j_{V'*}(G|_{V'})$.
\end{lemma}
\begin{proof}
  To check that a given morphism is an isomorphism it is enough to see that it gives
  an isomorphism on the cohomology of the stalks at any point.  For $(x,t) \in U$ and
  $k\in \Z$ we have
  $H^k(\opb{i_U} \RR j_{V*}(G))_{(x,t)} \simeq \varinjlim_W H^k(W \cap V; G)$, where
  $W$ runs over the neighborhoods of $(x,t,0)$ in $M \times\R \times \rpos$.  For a
  given $(x,t) \in U$ we have $W\cap V = W\cap V'$ when $W$ is small enough and the
  result follows.  
\end{proof}

\begin{notation}\label{rem:valbord}
  Because of Lemma~\ref{lem:valbord} we can forget the subscript $V$ and write the
  boundary value as $\opb{i_U} \roim{j} (G)$ for $G \in \Der(\cor_{V})$, when $V$ is
  any open subset of $M \times\R \times \rspos$ satisfying~\eqref{eq:valbord_ouvert}.
  Most of the time $V$ will be $U_\gammaof$ or $U\times\rspos$ but sometimes it will
  be convenient to shrink $V$ in the course of a proof.
\end{notation}

Let $\pi_U^{>0} \colon T^*_{\tau>0}U \to U$,
$\pi_{U_\gammaof}^{>0} \colon T^*_{\tau>0}U_\gammaof \to U_\gammaof$ be the
projections.  Proposition~\ref{prop:decomposition_muhomPsi} yields a morphism, for
$F,G \in \Dertpn(\cor_U)$,
\begin{equation}
  \label{eq:hom_Psi-vers-muhom0}
  \begin{split}
  \rhom(\Psi_U(F)&,\Psi_U(G))   \\
 & \to  \roim{\pi_{U_\gammaof}^{>0}{}}
(\muhom(\opb{q_U}(F), \opb{q_U}(G))|_{T^*_{\tau>0}U_\gammaof})  
\end{split}
\end{equation}
which can be written as the composition
\begin{align}
\notag  \rhom(\Psi_U(F)&,\Psi_U(G))  \\
  \label{eq:hom_Psi-vers-muhomA}
                       & \to \rhom(\Psi_U(F),\opb{q_U}(G))  \\
  \notag & \simeq  \roim{\pi_{U_\gammaof}{}} \muhom(\Psi_U(F), \opb{q_U}(G)) \\
  \label{eq:hom_Psi-vers-muhomC}
                       & \to  \roim{\pi_{U_\gammaof}^{>0}{}}
   ( \muhom(\Psi_U(F), \opb{q_U}(G))|_{T^*_{\tau>0}U_\gammaof} ) \\
\notag  & \isofrom \roim{\pi_{U_\gammaof}^{>0}{}}
(\muhom(\opb{q_U}(F), \opb{q_U}(G))|_{T^*_{\tau>0}U_\gammaof})  
\end{align}
where the first line is induced by
$\alpha(G) \colon \Psi_U(G) \to \opb{q_U}(G)$, the second line
is~\eqref{eq:proj_muhom_oim}, the third line is the restriction to
$T^*_{\tau>0}U_\gammaof$ and the fourth line is given by
Proposition~\ref{prop:decomposition_muhomPsi}.  Now it is easy to describe the
boundary value of the right hand side of~\eqref{eq:hom_Psi-vers-muhom0} as
follows.

\begin{lemma}
  \label{lem:def_homPsiversmuhom}
  For $F,G,H \in \Der(\cor_U)$ and $\shf\in \Der(\cor_{T^*U})$ we have natural
  isomorphisms
  \begin{align}
     \label{eq:hom_Psi-vers-muhom1}
    \muhom(\opb{q_U}(F),\opb{q_U}(G))
    &\simeq \eim{{q_{U,d}}} \, \opb{q_{U,\pi}} \muhom(F,G) , \\
    \label{eq:hom_Psi-vers-muhom2}
    \roim{\pi_{U_\gammaof}^{>0}} 
((\eim{{q_{U,d}}} \, \opb{q_{U,\pi}}(\shf))|_{T^*_{\tau>0}U_\gammaof})
    &\simeq \opb{q_U} \RR\pi_{U*}^{>0}(\shf|_{T^*_{\tau>0}U}), \\
    \label{eq:hom_Psi-vers-muhom3}
    \opb{i_U} \roim{j} \opb{q_U} H &\simeq H 
  \end{align}
which induce
\begin{equation}\label{eq:hom_Psi-vers-muhom4}
  \begin{split}
  \opb{i_U} \roim{j} \roim{\pi_{U_\gammaof}^{>0}{}}
(\muhom(\opb{q_U}(F)&, \opb{q_U}(G))|_{T^*_{\tau>0}U_\gammaof})  \\
& \simeq  \RR\pi_{U*}^{>0}( \muhom(F,G)|_{T^*_{\tau>0}U} ) .
\end{split}
\end{equation}
\end{lemma}
\begin{proof}
  (i) The behaviour of $\muhom$ under an inverse image by a submersion is
  described in~\cite[Prop.~4.4.7]{KS90} and gives in our case the
  isomorphism~\eqref{eq:hom_Psi-vers-muhom1}.

  \sui(ii)   We remark that
  $\eim{{q_{U,d}}} = \oim{{q_{U,d}}}$, since $q_{U,d}$ is an embedding, and
  $\opb{q_U} \simeq \epb{q_U}[-1]$, $\opb{q_{U,\pi}} \simeq \epb{q_{U,\pi}}[-1]$,
  since $q_U$ is a projection with fibers $\R$. Now~\eqref{eq:hom_Psi-vers-muhom2}
  follows from the base change formula
  $\roim{a} \epb{q_{U,\pi}} \simeq \epb{q_U} \RR\pi_{U*}^{>0}$ with
  $a = \pi_{U_\gammaof}^{>0} \circ q_{U,d}$.
  
  \sui(iii) We have
  $\opb{i_U} \roim{j} \opb{q_U} H \simeq \opb{i_U} \rsect_{U \times
    \rspos}(\opb{q_1} H)$, where $q_1 \cl U \times \R \to U$ is the projection.
  Since $\SSi(\cor_{U \times \rspos}) \subset T^*_UU \times T^*\R$,
  Theorem~\ref{thm:SSrhom} gives
$$
\rsect_{U \times \rspos}(\opb{q_1} H) \simeq \rhom(\cor_{U \times \rspos},
\opb{q_1} H) \simeq \cor_{U \times \rpos} \tens \opb{q_1} H
$$
and we obtain $\opb{i_U} \roim{j} \opb{q_U} H \simeq \opb{i_U} (\cor_{U \times
  \rpos} \tens \opb{q_1} H) \simeq H$.
\end{proof}

\begin{definition}
\label{def:hom_Psi-vers-muhom}
For $F,G \in \Dertpn(\cor_U)$ we define the
morphism (functorial in $F$ and $G$)
\begin{equation*}
b(F,G) \colon  \opb{i_U} \roim{j}  \rhom(\Psi_U(F),\Psi_U(G))
  \to  \RR\pi_{U*}^{>0}( \muhom(F,G)|_{T^*_{\tau>0}U} ) 
\end{equation*}
as the composition of~\eqref{eq:hom_Psi-vers-muhom0}
and~\eqref{eq:hom_Psi-vers-muhom4}.
\end{definition}

We prove below that $b(F,G)$ is an isomorphism if $G \in \Dertp(\cor_U)$.  We
need some remarks on the $\muhom$ functor.  We will use Sato's distinguished
triangle~\eqref{eq:SatoDTmuhom0} and introduce the following notation.

\begin{notation}\label{not:homprime}
Let $X$ be a manifold. Let $q_{X,1},q_{X,2}\cl X\times X \to X$ be the
projections and $\delta_X \cl X \to X\times X$ the diagonal embedding.  Let
$F,F'\in \Der(\cor_X)$. We set
\begin{equation}\label{eq:not_homprime}
\hom'(F,F') \eqdot \opb{\delta_X}\rhom(\opb{q_{X,2}}F,\opb{q_{X,1}}F').
\end{equation}
\end{notation}
Then Sato's distinguished triangle becomes, for $F,F'\in \Der(\cor_X)$,
\begin{equation}\label{eq:SatoDTmuhom2}
  \begin{split}
 \hom'(F,F') \to \rhom (F,F') &  \\
\to \roim{\dot\pi_X{}} & (\muhom  (F,F') |_{\dT^*X}) \to[+1].
  \end{split}
\end{equation}

\begin{lemma}\label{lem:im_inv_homprime}
(i) Let $f\cl X\to Y$ be a morphism of manifolds. Let $F,F'\in \Der(\cor_Y)$
such that $f$ is non-characteristic for $\SSi(F)$ and $\SSi(F')$. Then
$$
\opb{f}\hom'(F,F') \isoto \hom'(\opb{f}F,\opb{f}F') .
$$
(ii) For $F,F'\in \Der(\cor_X)$ we have
 $\SSi(\hom'(F,F'))\subset \SSi(F)^a \hplus \SSi(F')$.
\end{lemma}
\begin{proof}
  (i) We set $G = \rhom(\opb{q_{Y,2}}F,\opb{q_{Y,1}}F')$ (using similar
  notations as in~\eqref{eq:not_homprime}). Then $\SSi(\opb{q_{Y,1}}F')$ and
  $\SSi(G) \subset \SSi(F)^a \times \SSi(F')$ are non-characteristic for
  $f\times f$. By Theorem~\ref{thm:iminv} we deduce
\begin{align*}
\epb{(f\times f)} \opb{q_{Y,1}}F' &\simeq \opb{(f\times f)} \opb{q_{Y,1}}F'
\tens \omega_{X\times X| Y\times Y}, \\
\opb{(f\times f)} G &\simeq
\epb{(f\times f)}G \tens \omega_{X\times X| Y\times Y}^{\otimes -1}.
\end{align*}
By Proposition~\ref{prop:formulaire}-(h) this gives the third isomorphism in the
following sequence:
\begin{align*}
\opb{f}\hom'(F,F')
& \simeq  \opb{f}\opb{\delta_Y}\rhom(\opb{q_{Y,2}}F,\opb{q_{Y,1}}F')  \\
& \simeq \opb{\delta_X}\opb{(f\times f)} \rhom(\opb{q_{Y,2}}F,\opb{q_{Y,1}}F') \\
& \isoto \opb{\delta_X}
\rhom(\opb{(f\times f)}\opb{q_{Y,2}}F,\opb{(f\times f)}\opb{q_{Y,1}}F') \\
& \simeq \opb{\delta_X}\rhom(\opb{q_{X,2}}\opb{f}F,\opb{q_{X,1}}\opb{f}F').
\end{align*}

\noindent
(ii) follows from Remark~\ref{rem:SSexttens} and Theorem~\ref{thm:iminv}.
\end{proof}

\begin{proposition}
  \label{prop:muhom=hompsi}
  Let $F, G \in \Dertpn(\cor_U)$.  If $\dot\SSi(F) \cap \dot\SSi(G)$ is
  contained in $\{\tau>0\}$, then we have
  $$
  \roim{\dot\pi_U{}}(\muhom(F,G)|_{\dT^*U} ) \isoto
\RR\pi_{U*}^{>0}( \muhom(F,G)|_{T^*_{\tau>0}U} ) 
  $$
  and the morphism $b(F,G)$ of Definition~\ref{def:hom_Psi-vers-muhom} is an
  isomorphism.
\end{proposition}
\begin{proof}
  (i) We recall that $\supp \muhom(F_1,F_2) \subset \SSi(F_1) \cap \SSi(F_2)$.  Hence
  the hypothesis implies that $\muhom(F,G)|_{\dT^*U}$ is supported in $T^*_{\tau>0}U$
  and this gives the first isomorphism.  Let us call $u$ and $v$ the
  morphisms~\eqref{eq:hom_Psi-vers-muhomA} and~\eqref{eq:hom_Psi-vers-muhomC}.  To
  see that $b(F,G)$ is an isomorphism, we prove that the induced morphisms
  $\opb{i_U}\roim{j}(u)$ and $\opb{i_U}\roim{j}(v)$ are isomorphisms, respectively
  in~(ii) and~(iii) below.

  \sui(ii) Let us prove that $\opb{i_U}\roim{j}(u)$ is an isomorphism.  By the
  distinguished triangle~\eqref{eq:dtgamqr2} its cone is
  $A = \opb{i_U}\roim{j} \rhom(\Psi_U(F), \opb{r_U}(G))$ and we have to prove the
  vanishing of $A$.  For a given $(x,t) \in U$ and $k\in \Z$ we have
\begin{equation}\label{eq:germopbiroimj}
  H^k(A)_{(x,t)} \simeq \varinjlim_W \Hom(\Psi_U(F)|_W, \opb{r_U}(G)|_W[k]),
\end{equation}
where $W$ runs over the open subsets of $M\times\R\times\rspos$ such that
$\ol{W}$ is a neighborhood of $(x,t,0)$ in $U\times [0,+\infty[$.  We may take
$W = V_\gammaof$, where $V$ runs over the open neighborhoods of $(x,t)$ in $U$.
By~\eqref{eq:restr-PsiU-ouvert} we have $\Psi_U(F)|_{V_\gammaof} \simeq
\Psi_V(F|_V)$.  We also have $\opb{r_U}(G)|_{V_\gammaof} \simeq
\opb{r_V}(G|_V)$.  Since $\epb{r_V} \simeq \opb{r_V}[1]$, the adjunction
$(\reim{r_V}, \epb{r_V})$ gives
$$
\Hom(\Psi_V(F|_V), \opb{r_V}(G|_V)[k])
 \simeq \Hom(\reim{r_V}\Psi_V(F|_V), G|_V[k-1]).
$$
By Lemma~\ref{lem:annulation_qr} we have $\reim{r_V}\Psi_V(F|_V) \simeq 0$ and
we deduce the vanishing of~\eqref{eq:germopbiroimj} for all $(x,t)$ in $U$.
Hence $A \simeq 0$, as required.

\sui(iii) Now we prove that $\opb{i_U}\roim{j}(v)$ is an isomorphism.  Using the
hypothesis on the microsupport, we remark as in~(i) that the right hand side
of~\eqref{eq:hom_Psi-vers-muhomC} is isomorphic to
$\roim{\dot\pi_{U_\gammaof}{}} ( \muhom(\Psi_U(F), \opb{q_U}(G))|_{\dT^*U_\gammaof}
)$.  Hence, by the triangle~\eqref{eq:SatoDTmuhom2}, the cone of
$\opb{i_U}\roim{j}(v)$ is (up to a shift by $1$)
$B := \opb{i_U} \roim{j} \hom'(\Psi_U(F),\opb{q_U}(G))$.  Let us prove that $B$
vanishes.

Let $p_U \colon U\times\R \to U$ be the projection.  We set
$U_+ = U \times \rspos$ and $C = \hom'(\reim{j} \Psi_U(F),\opb{p_U}(G))$.  Then
$$
\roim{j} \hom'(\Psi_U(F),\opb{q_U}(G)) \simeq  \rsect_{U_+}C .
$$
By Lemma~\ref{lem:SSPsiF} we have
$\SSi(\reim{j} \Psi_U(F)) \subset \{(\xi,\tau, \upsilon)$;
$-\tau \leq \upsilon \leq 0\}$.  We also have
$\SSi(\opb{p_U}(G)) \subset \{(\xi,\tau, 0)$; $\tau \geq 0\}$ and we obtain
$\dot\SSi(C) \subset \{ \upsilon \geq 0\}$ by Lemma~\ref{lem:im_inv_homprime}.  Since
$\dot\SSi(\cor_{U_+}) = T^*_UU \times \{\upsilon < 0\}$, we deduce
$\rsect_{U_+}C \simeq \DD'(\cor_{U_+}) \otimes C \simeq C_{\ol{U_+}}$ by
Theorem~\ref{thm:SSrhom}.  In particular $B \simeq \opb{i_U}C$.

We also obtain that $i_U$ is non-characteristic for $\SSi(\reim{j} \Psi_U(F))$ and
$\SSi(\opb{p_U}(G))$.  Hence
$\opb{i_U}C \simeq \hom'(\opb{i_U}\reim{j} \Psi_U(F), G)$ by
Lemma~\ref{lem:im_inv_homprime}.  Since $\opb{i_U}\reim{j} \Psi_U(F) \simeq 0$, we
obtain $\opb{i_U}C \simeq 0$, as required.
\end{proof}

For   the next result we use the notion of
pure sheaves (see Definition~\ref{def:simple_pure}) along a Lagrangian
submanifold $\Lambda \subset \dT^*U$ and the stack $\kss(\cor_\Lambda)$ together
with the functor
$\kssfunc_\Lambda \cl \Derb_{(\Lambda)}(\cor_M) \to \kss(\cor_\Lambda)$ (see
Definition~\ref{def:KSstack}).

\begin{corollary}\label{cor:kssfuncPsi}
  Let $\Lambda \subset T^*_{\tau >0}U$ be a closed conic Lagrangian submanifold.
  Let $F, G \in \Der_{[\Lambda]}(\cor_U)$ be pure sheaves with the same shift.
  Then
  $$
  \varinjlim_V \Hom(\Psi_U(F)|_V, \Psi_U(G)|_V)
\simeq \Hom(\kssfunc_\Lambda(F), \kssfunc_\Lambda(G)) ,
  $$
  where $V$ runs over the open subsets of $U_\gammaof$ such that
  $(U \times \{0\}) \sqcup V$ is open in $M \times\R \times \rpos$.  In
  particular, if $\kssfunc_\Lambda(F) \simeq \kssfunc_\Lambda(G)$, then there
  exists such an open subset $V$ such that $\Psi_U(F)|_V \simeq \Psi_U(G)|_V$.
\end{corollary}
\begin{proof}
  (i) We recall that the $\hom$ sheaf in $\kss(\cor_\Lambda)$ is induced by
  $H^0\muhom$.  By the purity hypothesis $\muhom(F,G)$ is concentrated in degree $0$
  and~\eqref{eq:hom_sheaf_kss} gives
  $\Hom(\kssfunc_\Lambda(F), \kssfunc_\Lambda(G)) \simeq H^0(\Lambda; \muhom(F,G))$.
  Hence the isomorphism of the corollary follows from
  Proposition~\ref{prop:muhom=hompsi} and the remark that, for any
  $F' \in \Der(\cor_{U_\gammaof})$,
  $$
  H^0(U; \opb{i_U} \roim{j}(F'))
  \simeq \varinjlim_W H^k(W \cap U_\gammaof; F'),
  $$
  where $W$ runs over the neighborhoods of $U$ in $M \times\R \times \rpos$.
  
  \sui(ii) Let $u \colon \kssfunc_\Lambda(F) \isoto \kssfunc_\Lambda(G)$ be an
  isomorphism. By~(i) there exist $V$ as in the statement of the corollary and
  $a \colon \Psi_U(F)|_V \to \Psi_U(G)|_V$,
  $b \colon \Psi_U(G)|_V \to \Psi_U(U)|_V$ representing $u$, $u^{-1}$.
  Using~(i) again, and maybe shrinking $V$, the relations $u\circ u^{-1} = \id$,
  $u^{-1}\circ u = \id$ give $a\circ b = \id$, $b\circ a = \id$.  Hence $a$ is
  an isomorphism.  
\end{proof}

\section{Doubled sheaves}
\label{sec:dblsh}

The isomorphism of Proposition~\ref{prop:muhom=hompsi} will be used mainly when
$\dot\SSi(F) = \dot\SSi(G) = \Lambda$ is a smooth Lagrangian submanifold of
$\dT^*_{\tau>0}(M\times\R)$.  In this case  
$\muhom(F,G)|_{\dT^*(M\times\R)}$ is a locally constant sheaf on $\Lambda$.

Unfortunately we will also need to consider the case where $F$ and $G$ have
microsupport $\Lambda$ only in a neighborhood of some open subset $\Lambda_0$ of
$\Lambda$.  To easily handle this case we actually assume that $\Lambda_0$ is
locally of the form $\Lambda \cap T^*(U \times I)$, for some open set $U$ of $M$
and some open interval $I$ of $\R$, and that $F$ (and $G$ as well) is locally of
the form $\rsect_UF'$ for some $F'$ with $\SSi(F') \subset \Lambda$.  In the
proof of Theorem~\ref{thm:quant_legendrian_dbl} below we will glue such sheaves
$F$ or, rather, their images $\Psi(F)$. For this reason we define a subcategory
of $\Der(\cor_{M\times\R \times\rspos})$ of sheaves which are locally of the
form $\Psi(\rsect_UF')$.  We then give the analog of
Proposition~\ref{prop:muhom=hompsi} for this subcategory (see
Proposition~\ref{prop:muhom=hompsiEtendu} below).

  Since we deal with
sheaves of the form $\rsect_UF'$, we will face the problem of the commutation of the
main functors introduced so far, $\Psi$ and $\opb{i}\roim{j}$, with the functors
$\rsect_U$, $\rsect_{U_\gammaof}$,\dots  Lemma~\ref{lem:formules_dbl} answers this
question and its proof requires some conditions on $\Lambda$ and $\SSi(\cor_U)$,
essentially to be able to use Theorem~\ref{thm:SSrhom}.  This is the aim of the
following definition.

For a family of subsets $U_a$, $a\in A$, of some set $E$ and for $B\subset A$ we
set $U_B = \bigcap_{a\in B} U_a$ and $U^B = \bigcup_{a\in B} U_a$.

\begin{definition}\label{def:horiz_subset}
    Let $\Lambda \subset T^*_{\tau >0}(M\times \R)$
  be a conic Lagrangian submanifold such that $\Lambda/\rspos$ is compact, the map
  $\Lambda/\rspos \to M$ has finite fibers.    A finite
  family $\shv = \{V_a$; $a\in A\}$ of open subsets of $M\times\R$ is said to be {\em
    adapted} to $\Lambda$ if the following conditions hold:  
  \begin{itemize}
  \item [(i)] for each $a\in A$ we have $V_a = U_a \times I_a$ and
    $\Lambda \cap T^*V_a$ is contained in $\pi_{M\times\R}^{-1}(U_a \times K)$
    for some compact interval $K$ of $I_a$,
  \item [(ii)] for all $B,B' \subset A$ we have
    $\DD'(\cor_{V^B}) \simeq \cor_{\ol{V^B}}$ and, setting
    $\Lambda_+ = \Lambda \cup T^*_{M\times\R}(M\times\R)$,
    \begin{equation}
      \label{eq:def_horiz_subset}
      (\SSi(\cor_{V^B}) \hplus  \SSi(\cor_{V^{B'}})^a)
      \cap (\Lambda_+^a \hplus \Lambda_+) \subset T^*_{M\times \R}(M\times\R) .
    \end{equation}
  \end{itemize}
\end{definition}

\begin{notation}\label{not:cutwithadapfam}
  In this section we use adapted families of open subsets to cut sheaves defined on
  an open subset $W$ of $M\times\R$.  To avoid too heavy notations we set abusively,
  for $F \in \Der(\cor_W)$ and $B\subset A$:
  $$
\rsect_{V^B}(F) := \rsect_{W \cap V^B}(F) \in \Der(\cor_W)
$$
and similarly
$\rsect_{V^B\times \rspos}(-) := \rsect_{W_\gammaof \cap (V^B\times \rspos)}(-)$ for
sheaves defined over $W_\gammaof$ or
$\rsect_{T^*V^B}(-) := \rsect_{T^*W \cap T^*V^B}(-)$ over $T^*W$,\dots
\end{notation}

We first check that we have enough such adapted families.

\begin{lemma}\label{lem:exist_adapted_family}
  Let $\Lambda \subset T^*_{\tau >0}(M\times \R)$ be a conic Lagrangian submanifold
  such that $\Lambda/\rspos$ is compact and let
  $\Lambda = \bigcup_{i\in I} \Lambda_i$ be a finite covering by conic open
  subsets.   Then there exists a Hamiltonian
  isotopy $\Phi$, as closed to $\id$ as desired, and a finite family
  $\{V_a\}_{a\in A}$ of open subsets of $M\times\R$ which is adapted to
  $\Lambda' = \Phi_1(\Lambda)$ such that each connected component of
  $\Lambda' \cap T^*V_a$, for $a\in A$, is contained in $\Phi_1(\Lambda_i)$, for some
  $i\in I$.     Moreover we can assume that the family $\{V_a\}_{a\in A}$ is
  stable by intersection and that $\Lambda' \cap T^*V_a$ has finitely many connected
  components, for each $a\in A$.
\end{lemma}
\begin{proof}
  (i) We recall Definition~\ref{def:mustrat}: a stratification $\Sigma = \{\Sigma_j$,
  $j\in J\}$ of $M\times\R$ satisfies the {\em $\mu$-condition} if
  $\Lambda_\Sigma \hplus \Lambda_\Sigma = \Lambda_\Sigma$, where
  $\Lambda_\Sigma = \bigcup_{j\in J} T^*_{\Sigma_j}(M\times\R)$.    By
  Proposition~\ref{prop:mustratif}, any sheaf $F\in \Der(\cor_{M\times\R})$ which is
  constructible with respect to $\Sigma$ satisfies $\SSi(F) \subset \Lambda_\Sigma$.
  By Proposition~\ref{prop:deformeversmustrat}, there exists an $\rspos$-homogeneous
  Hamiltonian isotopy $\Phi$ and a stratification $\Sigma$ of $M\times\R$ satisfying
  the $\mu$-condition such that $\Lambda' := \Phi_1(\Lambda) \subset \Lambda_\Sigma$.
  Now to have~\eqref{eq:def_horiz_subset} it is enough that the family
  $\{V_a\}_{a\in A}$ satisfies
  \begin{equation}\label{eq:verifier_def_horiz_subset}
    \begin{aligned}
    (\SSi(\cor_{V^B}) \hplus \SSi(\cor_{V^{B'}})^a) \cap \Lambda_\Sigma
    \cap \pi_{M\times\R}^{-1}(& \dot\pi_{M\times\R}(\Lambda') ) \\
   & \subset   T^*_{M\times \R}(M\times\R) .
  \end{aligned}
  \end{equation}

  \sui(ii) We can assume that $\Lambda'$ is also contained in
  $T^*_{\tau >0}(M\times \R)$.  Hence, for a given point $y_0= (x_0,t_0)$ in
  $\dot\pi_{M\times\R}(\Lambda')$ the strata $\Sigma_j$ such that
  $y_0 \in \ol{\Sigma_j} \subset \dot\pi_{M\times\R}(\Lambda')$ do not meet the
  truncated cone given in local coordinates by $\{\eta > |t-t_0| > C||x-x_0||\}$, for
  $\eta>0$ small enough and $C$ big enough.  Then, for
  $U = \{ ||x-x_0|| < \eta/2C\} \subset M$ and $I = \{ |t-t_0| < \eta \} \subset \R$,
  we have $y_0 \in U \times I$ and
  $\dot\pi_{M\times\R}(\Lambda') \cap (\ol{U} \times \partial I) = \emptyset$.  We
  can then cover $\dot\pi_{M\times\R}(\Lambda')$ by open subsets of this kind and,
  taking a finite subcover, we obtain a finite family of open subsets of $M\times\R$,
  $W_a$, $a\in A$, such that
  $\dot\pi_{M\times\R}(\Lambda') \subset \bigcup_{a\in A}W_a$, $W_a = U_a \times I_a$
  as in~(i) of Definition~\ref{def:horiz_subset} and $U_a = \{ f_a < 0\}$ for a
  $\Cinf$ function $f_a \colon M \to \R$.  Moreover, for each $a\in A$ there exists a
  compact interval $K_a \subset I_a$ such that $\Lambda' \cap T^*W_a$ is contained in
  $\pi_{M\times\R}^{-1}(U_a \times K_a)$.  We can also assume that each $W_a$ is
  small enough so that each component of $\Lambda' \cap T^*W_a$ is contained in
  $\Phi_1(\Lambda_i)$, for some $i\in I$.

  Let $\ul \varepsilon = \{\varepsilon_a$; $a\in A\}$ be a family of negative
  numbers.  We define $W_{a,\varepsilon_a} = \{ f_a < \varepsilon_a \} \times I_a$.
  We choose $\delta>0$ such that
  $\dot\pi_{M\times\R}(\Lambda') \subset \bigcup_{a\in A} W_{a,-\delta}$.  Let
  $E \subset [-\delta, 0[^A$ be the subset formed by the $\ul \varepsilon$ such that:
  \begin{itemize}
  \item [(a)] the hypersurfaces
    $X_{a,\varepsilon_a} = \{ f_a = \varepsilon_a \}$, $a\in A$, are smooth and
    intersect transversely, in the sense that their union is locally
    diffeomorphic to the embedding of coordinates hyperplanes in $\R^n$,
  \item [(b)] for any $B \subset A$, the manifold
    $X_{B,\ul\varepsilon} = \bigcap_{a\in B} (X_{a,\varepsilon_a} \times I_a)$
    intersects each stratum $\Sigma_j$ of $\Sigma$ transversely.
  \end{itemize}
  By the transversality theorem $E$ is dense in $[-\delta, 0[^A$.  Hence we only have
  to prove that, for a given $\ul \varepsilon \in E$, the family
  $V_a = W_{a,\varepsilon_a}$, $a\in A$, satisfies the conclusion of the lemma.

  The condition $\DD'(\cor_{V^B}) \simeq \cor_{\ol{V^B}}$ is local on $M\times\R$ and
  follows from the above condition~(a), up to modifying slightly the intervals
  $I_a$'s so that they have distinct ends.

  We set $Y = \bigcup_{a\in A} (X_{a,\varepsilon_a} \times \partial I_a)$ and
  $\Omega = (M\times\R) \setminus Y$.  We remark that
  $\dot\pi_{M\times\R}(\Lambda') \cap Y = \emptyset$.   Hence, to prove~\eqref{eq:verifier_def_horiz_subset} it is
  enough to see
  $(\SSi(\cor_{V^B}) \hplus \SSi(\cor_{V^{B'}})^a) \cap \Lambda_\Sigma \cap T^*\Omega
  \subset T^*_\Omega\Omega$.

  In $\Omega$ we have the family of hypersurfaces
  $(X_{a,\varepsilon_a} \times I_a) \cap \Omega$, $a\in A$, which are closed (not
  only locally closed), smooth and intersect transversely.    This family
  generates a stratification $\Sigma'(\varepsilon)$ of $\Omega$ such that the
  closures of the strata are the $X_{B,\ul\varepsilon}\cap \Omega$, $B\subset A$.  By
  the transversality assumptions
  $\Lambda_{\Sigma'(\varepsilon)} \hplus \Lambda_{\Sigma'(\varepsilon)} =
  \Lambda_{\Sigma'(\varepsilon)}$ and
  $\Lambda_{\Sigma'(\varepsilon)} \cap \Lambda_\Sigma \cap T^*\Omega \subset
  T^*_\Omega\Omega$.  For $B\subset A$ the open set $V^B \cap \Omega$ is
  constructible with respect to $\Sigma'(\varepsilon)$ and we deduce
  $\SSi(\cor_{V^B}) \cap T^*\Omega \subset \Lambda_{\Sigma'(\varepsilon)}$.  This
  gives~\eqref{eq:verifier_def_horiz_subset} and the family $\{V_a\}_{a\in A}$ is
  adapted to $\Lambda'$.

  \sui(iii) Let $A'$ be the set of subsets of $A$ and, for $a'\in A'$, set
  $\tilde V_{a'} = \bigcap_{a\in a'} V_a$.  An open subset $\tilde V^{B'}$, for
  $B'\subset A'$, is still bounded by the hypersurfaces $X_{a,\varepsilon_a}$
  introduced in the condition~(a) of~(ii).  So we also have
  $\DD'(\cor_{\tilde V^{B'}}) \simeq \cor_{\ol{\tilde V^{B'}}}$ and
  $\SSi(\cor_{\tilde V^{B'}}) \subset \Lambda_{\Sigma'(\varepsilon)}$ as
  in~(ii). Replacing $A$ by $A'$ we have an adapted family which is stable by
  intersection.  To ensure that $\Lambda \cap T^*\tilde V_{a'}$ has finitely many
  connected components, for each $a' \in A'$, we choose $\ul\varepsilon$ so that
  $\varepsilon_a$ is a regular value of
  $f_a \circ \pi \colon \Lambda'\cap T^*W_a \to \R$, for each $a\in A$, where $\pi$
  is the projection $\Lambda'\to M$.
\end{proof}

\begin{definition}\label{def:cat_dbl_sheaves}
    Let $\Lambda$ and $\shv = \{V_a$; $a\in A\}$ be as in
  Definition~\ref{def:horiz_subset}.  Let $V \subset M\times\R$ be an open subset.
  We denote by $\Der_{\Lambda, \shv}^\dbl(\cor_V)$ the subcategory of
  $\Derb(\cor_{V_\gammaof})$ formed by the $G$ such that any point of $V$ has an
  open neighborhood $W \subset V$ such that $\Lambda \cap T^*W$ has finitely many
  connected components, say $\{\Lambda_i\}_{i\in I}$, and for each $i\in I$ there
  exist $A_i \subset A$ and $F_i \in \Derb_{[\Lambda_i]}(\cor_W)$ satisfying
  \begin{equation}
    \label{eq:defDdbl}
    G|_{W_\gammaof}
    \simeq \bigoplus_{i\in I} \rsect_{V^{A_i} \times \rspos}  (\Psi_W(F_i))
  \end{equation}
  where $V^{A_i} = \bigcup_{a\in A_i} V_a$ and we use the abusive
  Notation~\ref{not:cutwithadapfam}.
\end{definition}

For $G$ and $W$ as in~\eqref{eq:defDdbl} we have
$\supp(G) \subset W_\gammaof \cap (\gammaf \star \dot\pi_{M\times \R}(\Lambda))$ by
Lemma~\ref{lem:suppPsiuL} (see~\eqref{eq:convset} for $\star$).  We can cover $V$ by
open subsets $W_j$, $j\in J$, for which a decomposition~\eqref{eq:defDdbl} holds;
setting $V' = \bigcup_{j\in J} W_{j,\gammaof}$ we then have
\begin{equation}
  \label{eq:supp_faiscdbl}
\supp(G) \cap V' \subset (\gammaf \star \dot\pi_{M\times \R}(\Lambda)) \cap V'  
\end{equation}
and $(V \times \{0\}) \sqcup V'$ is open in $M \times\R \times \rpos$ (as
in~\eqref{eq:valbord_ouvert}).

    For an open subset
$W\subset M\times\R$ we recall the maps $q_W,r_W \colon W_\gammaof \to M\times \R$
of~\eqref{eq:def_projU}.  For subsets $U$ of $M\times\R$ or $\Xi$ of $T^*(M\times\R)$
we introduce the notations  
\begin{equation}
  \label{eq:def_iminvpr}
  \begin{split}
    U^q &= U^{q_W}= q_W^{-1}(U \cap W) , \\
    \Xi^q &= \Xi^{q_W} = q_{W,d}\opb{q_{W,\pi}}(\Xi \cap T^*W)
    = (\Xi \times T^*_{\rspos}\rspos) \cap T^*W_\gammaof , \\
\Xi^r &= \Xi^{r_W} = r_{W,d}\opb{r_{W,\pi}}(\Xi \cap T^*W) .
\end{split}
\end{equation}
($W$   is in general
suppressed from notation if there is no ambiguity.)  The next lemma gives results
about the sheaves appearing in~\eqref{eq:defDdbl}.

\begin{lemma}\label{lem:formules_dbl}
  Let $\Lambda$ and $\shv = \{V_a$; $a\in A\}$ be as in
  Definition~\ref{def:horiz_subset}.  Let $W$ be an open subset of $M\times\R$
  and let $\Lambda_1$ be a connected component of $\Lambda \cap T^*W$.  Let
  $F\in \Derb_{[\Lambda_1]}(\cor_W)$ and $B \subset A$ be given.  Then, using
  Notation~\ref{rem:valbord} and~\eqref{eq:def_iminvpr}, we have
  \begin{itemize}
  \item [(i)]
    $\rsect_{V^{B,q}}(\Psi_W(F)) \simeq  (\Psi_W(F))_{\ol{V^{B,q}}}$,
  \item [(ii)]
    $\SSi(\rsect_{V^{B,q}}(\Psi_W(F))) \cap \Lambda^q 
  \subset (\ol{\Lambda_1 \cap T^*V^B})^q$,
  \item [(iii)] for any $G \in \Der_{\Lambda, \shv}^\dbl(\cor_W)$ we have
    \begin{equation}
    \label{eq:formules_dblA}
    \begin{split}
      \opb{i_W}\roim{j} \rhom(&G, \rsect_{V^{B,q}}(\Psi_W(F))) \\
      &\simeq \rsect_{V^B} (\opb{i_W}\roim{j} \rhom(G, \Psi_W(F)) ).
    \end{split}
  \end{equation}
\item [(iv)] For any $x\in W \setminus \dot\pi_{M\times \R}(\Lambda_1)$ there exists a
  neighborhood $W'$ of $x$ such that
  $(\rsect_{V^{B,q}}(\Psi_W(F)))|_{W'_\gammaof} \simeq 0$,
\item [(v)] For any $x\in W \cap \dot\pi_{M\times \R}(\Lambda_1)$ there exists a
  neighborhood $W'$ of $x$ such that
    \begin{equation}
    \label{eq:formules_dblB}  
  (\rsect_{V^{B,q}}(\Psi_W(F)))|_{W'_\gammaof} \simeq
  \Psi_{W'}(\rsect_{W' \cap V^B}(F|_{W'}))
\end{equation}
and $\SSi(\rsect_{W' \cap V^B}(F|_{W'})) \subset T^*_{\tau\geq0}W'$.
\end{itemize}
\end{lemma}
\begin{proof}
  (i)-(ii)   We will show that
  the first two claims follow from Theorem~\ref{thm:SSrhom}, the bound
  $\dot\SSi(\Psi_W(F)) \,\subset\, \Lambda_1^q \cup \Lambda_1^r$ of
  Lemma~\ref{lem:SSPsiF} and the hypotheses of Definition~\ref{def:horiz_subset}.

  More precisely, to deduce~(i) from Theorem~\ref{thm:SSrhom} we need to check that
  $\dot\SSi(\Psi_W(F))$ and $\SSi(\cor_{V^{B,q}}) = (\SSi(\cor_{V^B}))^q$ do not
  intersect  
  (here $F$ is a sheaf on $W$ and $\cor_{V^B}$ means
  $\cor_{V^B}|_W = \cor_{V^B\cap W}$ with the same abuse as in
  Notation~\ref{rem:valbord}).  We take coordinates $(\xi,\tau,\upsilon)$ in a fiber
  of $T^*(M\times\R\times\rspos)$.  The points of $\Lambda_1^q$, $\Lambda_1^r$ and
  $\SSi(\cor_{V^{B,q}})$ are respectively of the form $(\xi_1,\tau_1,0)$,
  $(\xi_2,\tau_2,-\tau_2)$, $(\xi_3,\tau_3,0)$ with $\tau_1,\tau_2>0$ (since
  $\Lambda_1 \subset \Lambda$). Hence the subsets $\Lambda_1^r$ and
  $\SSi(\cor_{V^{B,q}})$ cannot intersect. So it remains to check that
  $\SSi(\cor_{V^{B}})$ and $\dot\SSi(F)$ do not intersect.  We recall that
  $\dot\SSi(\cor_{V^{B}})$ is contained in the fibers over $\partial V^B$ and that
  $V^B$ is a finite union of products $U_a\times I_a$.  We remark that
  $\partial V^B \subset \bigcup_{a\in B} \partial (U_a\times I_a)$.  Let us call the
  ``vertical part'' of $\partial V^B$ its subset
  $\partial V^B \setminus (\partial V^B \cap \bigcup_{a\in B} (U_a\times \partial
  I_a))$.  Near a point in the vertical part $\partial V^B$ is then of the form
  $\partial V \times I$ for some open subsets $V\subset M$, $I\subset \R$ and a point
  of $\SSi(\cor_{V^{B,q}})$ is of the form $(\xi_3,0,0)$.  The hypothesis~(i) in
  Definition~\ref{def:horiz_subset} implies that $\pi_{M\times \R}(\Lambda)$ and
  $\partial V^B$ can only meet at points in the vertical part of $\partial V^B$.
  Near such a point $\SSi(\cor_{V^{B}})$ is contained in $\{\tau=0\}$ and cannot meet
  $\Lambda$. This proves~(i).

  \smallskip

  Since $\dot\SSi(\Psi_W(F))$ and $\SSi(\cor_{V^{B,q}})$ do not
  intersect, we know by Theorem~\ref{thm:SSrhom} and
  Example~\ref{ex:hplus_noncar} that
  $\SSi(\rsect_{V^{B,q}}(\Psi_W(F)))$ is bounded by the pointwise sum
  $\Xi = \Lambda_1' + \SSi(\cor_{V^{B,q}})$, where
  $\Lambda_1' = \Lambda_1^q \cup \Lambda_1^r \cup T^*_{W_\gammaof}W_\gammaof$.
  We need to understand the intersection of $\Xi$ and $\Lambda^q$.  It is clear
  that $\Lambda_1' + (T^*_{V^B}V^B)^q = \Lambda_1' \cap T^*(q^{-1}(V^B))$ and
  this satisfies the required bound. Hence we can concentrate on the part of
  $(\SSi(\cor_{V^B}))^q$ outside the zero section, hence above
  $q^{-1}(\partial V^B)$.  Since we are interested in $\Lambda^q \cap \Xi$, by
  the preceding discussion we only need to consider points above the vertical
  part of $q^{-1}(\partial V^B)$ and a point of $(\SSi(\cor_{V^B}))^q$ is then
  of the form $(\xi_0,0,0)$. It follows that a point of
  $\Lambda_1^r + (\SSi(\cor_{V^B}))^q$ is of the form $(\xi,\tau,-\tau)$,
  $\tau>0$, and cannot belong to $\Lambda^q$.  We also see that
  $T^*_{W_\gammaof}W_\gammaof + (\SSi(\cor_{V^B}))^q$ does not meet $\Lambda^q$.
  So it only remains to understand
  $(\Lambda_1^q + (\SSi(\cor_{V^B}))^q) \cap \Lambda^q$ and the statement
  follows from the hypothesis~(ii) of Definition~\ref{def:horiz_subset}.

  \sui(iii-a) The isomorphism~(iii) is local on $W$ and we can shrink $W$ if
  necessary.  Since $G \in \Der_{\Lambda, \shv}^\dbl(\cor_W)$ we can thus assume
  that $G$ satisfies~\eqref{eq:defDdbl}.  Taking one summand we assume in fact
  $G = \rsect_{V^{B',q}} (\Psi_{W}(F'))$, for some $B' \subset A$ and
  $F'\in \Derb_{[\Lambda]}(\cor_W)$.  We will prove
  \begin{equation}
    \label{eq:formules_dbl3}
    \begin{split}
      \opb{i_W}\roim{j} &\rhom(\rsect_{V^{B',q}} (\Psi_{W}(F')),
      \rsect_{V^{B,q}} (\Psi_{W}(F)) )  \\
      &\simeq \rsect_{\ol{V^{B'}} \cap V^B} (\opb{i_W}\roim{j}
      \rhom(\Psi_W(F'),  \Psi_W(F)) ) .
    \end{split}
  \end{equation}
  This implies~\eqref{eq:formules_dblA}: use~\eqref{eq:formules_dbl3} as it is
  stated and also in the case where $V^B \supset W$, together with
  $\rsect_{\ol{V^{B'}} \cap V^B}(-) \simeq \rsect_{V^B}(
  \rsect_{\ol{V^{B'}}}(-))$.

  We set $H = \rhom(\Psi_W(F'), \Psi_W(F))$.  Using part~(i) of the lemma and
  the isomorphism
  $\rhom((-)_{Z'}, \rsect_Z(-)) \simeq \rsect_{Z'\cap Z}\rhom(-,-)$, we deduce
  that the left hand side of~\eqref{eq:formules_dbl3} is
  $$
  \opb{i_W}\roim{j} \rsect_{(\ol{V^{B'}} \cap V^B)^q}(H)
  \simeq \opb{i_W}\rsect_{(\ol{V^{B'}} \cap V^B)\times \R}(\roim{j}  H)
  $$
  and we are reduced to prove
    \begin{equation*}
      \opb{i_W}\rsect_{(\ol{V^{B'}} \cap V^B)\times \R}(\roim{j} H)
      \simeq   \rsect_{\ol{V^{B'}} \cap V^B} (\opb{i_W}\roim{j}H)  .
    \end{equation*}
    For this it is enough to check, for some $\varepsilon>0$, setting
    $J_\varepsilon = \mo]-\infty,\varepsilon[$,
      \begin{equation}
    \label{eq:formules_dbl5}
    \begin{split}
    (\SSi(\cor_{\ol{V^{B'}} \cap V^B})\times T^*_{J_\varepsilon}J_\varepsilon)
    \cap \dot\SSi(\roim{j} (H )) = \emptyset, \\
    \SSi(\cor_{\ol{V^{B'}} \cap V^B})
    \cap \dot\SSi(\opb{i_W}\roim{j} (H )) = \emptyset .
  \end{split}
  \end{equation}
  Indeed, \eqref{eq:formules_dbl5},~\eqref{eq:def_horiz_subset} and
  Theorem~\ref{thm:SSrhom} imply the isomorphisms  
  \begin{gather*}
  \begin{split}
    \rsect_{(\ol{V^{B'}} \cap V^B)\times J_\varepsilon}(\roim{j} H)
    &\simeq
    \DD'(\cor_{(\ol{V^{B'}} \cap V^B)\times J_\varepsilon}) \ltens \roim{j} H  \\
    &\simeq (\DD'(\cor_{\ol{V^{B'}} \cap V^B}) \boxtimes \cor_{\ol J_\varepsilon})
    \ltens \roim{j} H , 
\end{split}  \\
\rsect_{\ol{V^{B'}} \cap V^B} (\opb{i_W}\roim{j}H) \simeq
\DD'(\cor_{\ol{V^{B'}} \cap V^B}) \ltens \opb{i_W}\roim{j} H .
\end{gather*}
We then use the commutativity of $\opb{i_W}$ with the tensor product.

  \sui(iii-b) Let us prove \eqref{eq:formules_dbl5}.
  Proposition~\ref{prop:formulaire}-(i) gives
  $$
  \roim{j}H \simeq \rhom( \reim{j}\Psi_W(F'), \roim{j}\Psi_W(F)) .
  $$
  For   $F'' \in \Dertpn(\cor_W)$
  Lemma~\ref{lem:SSPsiF} says $\reim{j}\Psi_W(F'') \simeq \roim{j}\Psi_W(F'')$ and
  gives a bound for $\SSi(\reim{j}\Psi_W(F''))$.  We deduce the following less
  precise bound which is easier to handle.  For $S \subset T^*W$ and $\varepsilon>0$
  we define $N_\varepsilon(S) = \bigcup_{c\in [0,\varepsilon]} T_c(S)$ where $T_c$ is
  the vertical translation $T_c(x,t;\xi,\tau) = (x,t+c;\xi,\tau)$.
  Lemma~\ref{lem:SSPsiF} implies
$$
\SSi((\reim{j}\Psi_W(F''))|_{W\times J_\varepsilon}) \subset
N_\varepsilon(\SSi(F'')) \times T^*J_\varepsilon .
$$
By Theorem~\ref{thm:SSrhom} we deduce
  \begin{equation}
    \label{eq:formules_dbl1}
    \SSi(\roim{j} (H )) \subset
    N_\varepsilon(\Lambda_+^a \hplus \Lambda_+) \times T^*J_\varepsilon.
  \end{equation}
  When   we take the inverse image by $i_W$
  in~\eqref{eq:formules_dbl1} we can assume $\varepsilon$ as small as we want and
  Theorem~\ref{thm:iminv} gives
  \begin{equation}
    \label{eq:formules_dbl2}
    \SSi(\opb{i_W}\roim{j} (H )) \subset  \Lambda_+^a \hplus \Lambda_+ .
  \end{equation}
  Finally we remark
  $\SSi(\cor_{\ol{V^{B'}} \cap V^B}) \subset \SSi(\cor_{V^B}) \hplus
  \SSi(\cor_{V^{B'}})^a$.  Now we deduce the relations~\eqref{eq:formules_dbl5}
  from~\eqref{eq:formules_dbl1} (taking $\varepsilon$ as small as required),
  \eqref{eq:formules_dbl2} and~\eqref{eq:def_horiz_subset}.

  \sui (iv) follows from the bound~\eqref{eq:supp_faiscdbl} and the remark that, for
  any subset $S$ of $M\times \R$ and any $\varepsilon>0$,
$$
(\gammaf \star S) \cap (M\times\R \times \mo]0,\varepsilon])
\subset \big( \bigcup_{c\in [0,\varepsilon]} t_c(S) \big)\times \mo]0,\varepsilon] ,
$$
where $t_c$ is the vertical translation $t_c(x,t) = (x,t+c)$.

\sui(v) If $x\in V^B$, then we choose $W' = W \cap V^B$ and the result is obvious. If
$x\not\in \ol{V^B}$, we choose $W'$ such that $W'\cap V^B = \emptyset$ and both sides
are zero.  So we can assume $x \in \partial V^B$.  Since
$x \in \dot\pi_{M\times \R}(\Lambda_1)$ we have seen in the proof of~(i) that $x$
belongs to the ``vertical part'' of $\partial V^B$.   
Hence we can find a neighborhood $W'$ of $x$ such that
$W' \cap V^B = W' \cap (U \times \R)$, for some open subset $U$ of $M$: recalling
that $V^B = \bigcup_{a\in B} (U_a\times I_a)$, we have
$U = \bigcup_{a \in B, x \in \ol{U_a \times I_a}} U_a$.  It follows from the
projection formula (see Proposition~\ref{prop:formulaire}-(g)) that
$\Psi_{W'}((H)_{W' \cap (Z \times \R)}) \simeq (\Psi_{W'}(H))_{W'_\gammaof \cap (Z
  \times \R \times \rspos)}$, for any sheaf $H$ on $W'$ and any locally closed subset
$Z$ of $M$.  Using part~(i) of the lemma (and the similar isomorphism
$\rsect_{V^B}(F) \simeq (F)_{\ol{V^B}}$) we deduce the result.
\end{proof}

We will state the analog of Proposition~\ref{prop:muhom=hompsi} for
$\Der_{\Lambda, \shv}^\dbl(\cor_V)$. We recall that the restriction of
$\muhom(\Psi_U(F), \Psi_U(G))$ to $\{\tau>0\}$ is decomposed, by
Proposition~\ref{prop:decomposition_muhomPsi}, as the sum of
$\muhom(\opb{q_U}(F), \opb{q_U}(G))$ and $\muhom(\opb{r_U}(F), \opb{r_U}(G))$.  Hence
in an analog of Proposition~\ref{prop:muhom=hompsi} we can expect that $\muhom(F,G)$
(for $F,G \in \Der(\cor_V)$) should be replaced by $\muhom(F',G')|_{\Lambda^q}$ (for
$F', G' \in \Der_{\Lambda, \shv}^\dbl(\cor_V)$). We will make this more precise soon
and we will use the following lemma.

\begin{lemma}\label{lem:muhom_dbl0}
  Let $G, G' \in \Der_{\Lambda, \shv}^\dbl(\cor_V)$. Then there exist a uniquely
  defined sheaf $\muhom^\dbl(G,G')$ in $\Derb(\cor_{\Lambda \cap T^*V})$ and an open
  subset $V'$ of $V_\gammaof$ such that $(V \times \{0\}) \sqcup V'$ is open in
  $M \times\R \times \rpos$ and
  $$
  \muhom(G,G')|_{\Lambda^q \cap T^*V'} \simeq
  \opb{p_\Lambda}(\muhom^\dbl(G,G'))|_{\Lambda^q \cap T^*V'} ,
  $$
  where $p_\Lambda \colon \Lambda^q \cap T^*V' \to \Lambda$ is the projection.
    Moreover, for $W \subset V$ open, if
  $G|_{W_\gammaof} \simeq \Psi_W(F)$, $G'|_{W_\gammaof} \simeq \Psi_W(F')$ for some
  $F,F' \in \Derb_{\tau\geq0}(\cor_W)$, then
  $\muhom^\dbl(G,G')|_{\Lambda \cap T^*W} \simeq \muhom(F,F')|_{\Lambda \cap T^*W}$.
\end{lemma}
\begin{proof}
  By Lemma~\ref{lem:formules_dbl}-(iv-v), any point $x\in V$ has a neighborhood $W$
  such that $G|_{W_\gammaof}$, $G'|_{W_\gammaof}$ are of the form $\Psi_W(F)$,
  $\Psi_W(F')$ for some $F,F' \in \Derb_{\tau\geq0}(\cor_W)$.  Then
  Proposition~\ref{prop:decomposition_muhomPsi} implies that the restriction of
  $\muhom(G,G')$ to $\Lambda^q \cap T^*W_\gammaof$ is of the form
  $\muhom(\opb{q_W}(F), \opb{q_W}(F')) \simeq \eim{{q_{W,d}}} \, \opb{q_{W,\pi}}
  \muhom(F,F')$.  In particular
  $$
  \muhom(G,G')|_{\Lambda^q \cap T^*W_\gammaof}
  \simeq (p_\Lambda|_{\Lambda^q \cap T^*W_\gammaof})^{-1} \muhom(F,F') ,
  $$
  which proves the last assertion and shows that
  $\muhom(G,G')|_{\Lambda^q \cap T^*W_\gammaof}$ is constant on the fibers of the
  projection $\Lambda^q \cap T^*W_\gammaof \to \Lambda$.

  Now we choose a family $W_i$, $i\in I$, of such open subsets $W$ which covers $V$
  and is locally finite (each compact subset of $V$ meets finitely many $W_i$'s).  We
  set $V' = \bigcup_{i\in I} W_{i,\gammaof}$.  Then $(V \times \{0\}) \sqcup V'$ is
  open in $M \times\R \times \rpos$ and, for any $x\in V$, there exists $i\in I$ such
  that $(\{x\} \times \rspos )\cap V' = (\{x\} \times \rspos )\cap W_{i,\gammaof}$.
  It follows that $\muhom(G,G')|_{\Lambda^q \cap T^*V'}$ is constant on the fibers of
  the projection $p_\Lambda \colon \Lambda^q \cap T^*V' \to \Lambda$, which are open
  intervals.  Hence it is the inverse image of some uniquely defined sheaf
  $\muhom^\dbl(G,G')$ by $p_\Lambda$ (and we have
  $\muhom^\dbl(G,G') = \roim{p_\Lambda}(\muhom(G,G')|_{\Lambda^q \cap T^*V'})$).
\end{proof}

The following result follows easily from the definition of
$\Der_{\Lambda, \shv}^\dbl(\cor_V)$ and Lemma~\ref{lem:formules_dbl}-(ii).
\begin{lemma}\label{lem:SS_dbl}
  For $G \in \Der_{\Lambda, \shv}^\dbl(\cor_V)$ there exists a unique open subset
  $\Lambda_G \subset \Lambda\cap T^*V$, that we denote $\SSi^\dbl(G)$,
  such that
  \begin{itemize}
  \item [(i)] for any open subset $W \subset V$ where a
    decomposition~\eqref{eq:defDdbl} holds we have
    $\Lambda_G \cap T^*W = \bigcup_{i\in I} \dot\SSi(F_i) \cap T^*V^{A_i}$ (with the
    notations of Definition~\ref{def:cat_dbl_sheaves} -- remark that $\dot\SSi(F_i)$
    is $\emptyset$ or $\Lambda_i$).
  \item [(ii)] there exist an open subset $V'$ of $V_\gammaof$ and a neighborhood
    $\Omega$ of $\Lambda_G$ in $T^*V$ such that $(V \times \{0\}) \sqcup V'$ is open
    in $M \times\R \times \rpos$ and
    \begin{align*}
      \SSi(G) \cap \Lambda^q \cap T^*V' &= {\ol{\Lambda_G}}^q \cap T^*V' , \\
      \SSi(G) \cap \Omega^q \cap T^*V' &= \Lambda_G^q \cap T^*V'.
    \end{align*}    
\end{itemize}
\end{lemma}

\begin{remark}   The microsupport of
  $G \in \Der_{\Lambda, \shv}^\dbl(\cor_V)$ is in general bigger than
  $(\SSi^\dbl(G))^q \cup (\SSi^\dbl(G))^r$; the last formula in
  Lemma~\ref{lem:SS_dbl} says $\dot\SSi(G)$ coincides with $(\SSi^\dbl(G))^q$ in some
  neighborhood of $(\SSi^\dbl(G))^q$, but $\dot\SSi(G)$ is a priori not contained in
  $\Lambda^q \cup \Lambda^r$.  Namely, the microsupport of
  $\rsect_{V^{A_i} \times \rspos} (\Psi_W(F_i))$ could be as big as the bound given
  in the proof of Lemma~\ref{lem:formules_dbl}~(ii).
\end{remark}

    We can define the
notions of pure or simple doubled sheaf and also $\kssfunc_{\Lambda_0}(G)$ for a
doubled sheaf $G$; this will be used in Part~\ref{part:quant}.  Let
$G \in \Der_{\Lambda, \shv}^\dbl(\cor_V)$.  Lemma~\ref{lem:SS_dbl} defines an open
subset $\SSi^\dbl(G)$ of $\Lambda$. Let $\Lambda_0$ be an open subset of
$\SSi^\dbl(G)$.  With $V'$ as in the lemma we have
$G|_{V'} \in \Derb_{(\Lambda_0^q \cap T^*V')}(\cor_{V'})$ (we recall that this means
that $\SSi(G)$ coincides with $\Lambda_0^q$ on some neighborhood of $\Lambda_0^q$).
We say that $G$ is pure (or simple) along $\Lambda_0$ if $G|_{V'}$ is pure (or
simple) in the usual sense along $\Lambda_0^q \cap T^*V'$.  Since
$G|_{V'} \in \Derb_{(\Lambda_0^q \cap T^*V')}(\cor_{V'})$, we can also consider
$\kssfunc_{\Lambda_0^q \cap T^*V'}(G|_{V'})$ which is an object of
$\kss(\cor_{\Lambda_0^q \cap T^*V'})$.  Up to shrinking $V'$ we can assume that the
fibers of the projection $p_{\Lambda_0} \colon \Lambda_0^q \cap T^*V' \to \Lambda_0$
are open intervals.  Then the inverse image by $p_{\Lambda_0}$ induces an equivalence
\begin{equation}
  \label{eq:equivksspourdbl}
  \kss(\cor_{\Lambda_0}) \isoto \kss(\cor_{\Lambda_0^q \cap T^*V'}) .
\end{equation}
In this way we can identify $\kssfunc_{\Lambda_0^q \cap T^*V'}(G|_{V'})$ with an
object of $\kss(\cor_{\Lambda_0})$ that we denote by $\kssfunc_{\Lambda_0}^\dbl(G)$.
We obtain a functor, for any open subset $\Lambda_0$ of $\Lambda \cap T^*V$,
\begin{equation}
  \label{eq:def_kssfuncLambda0dbl}
  \kssfunc_{\Lambda_0}^\dbl \cl \{G \in \Der_{\Lambda, \shv}^\dbl(\cor_V); \;
  \Lambda_0 \subset \SSi^\dbl(G)\}  \to \kss(\cor_{\Lambda_0}) 
\end{equation}
and~\eqref{eq:hom_sheaf_kss} together with~\eqref{eq:equivksspourdbl} give, for
$G, G' \in \Der_{\Lambda, \shv}^\dbl(\cor_V)$ such that
$\Lambda_0 \subset \SSi^\dbl(G)$, $\Lambda_0 \subset \SSi^\dbl(G')$,
\begin{equation}
  \label{eq:hom_kssfuncLambda0dbl}
  \Hom(\kssfunc_{\Lambda_0}^\dbl(G), \kssfunc_{\Lambda_0}^\dbl(G'))
\simeq H^0(\Lambda_0; H^0(\muhom^\dbl(G,G'))) .
\end{equation}
(Here   the condition
$\Lambda_0 \subset \SSi^\dbl(G)$ ensures that
$G \in \Derb_{(\Lambda_0^q)}(\cor_{V'})$.)  Let us set
$V_u = V_\gammaof \cap (M\times\R \times \{u\})$ for $u>0$.  We have a natural
inclusion $T^*V_u \subset T^*V_\gammaof$.  We remark that, if $V = V^B$ for some
$B\subset A$ and $u$ is small enough, we have $\Lambda_0^q \cap T^*V_u = \Lambda_0$.
Then $G|_{V_u}$ belongs to $\Derb_{(\Lambda_0)}(\cor_{V_u})$ and we have, by the
construction of $\kssfunc_{\Lambda_0}^\dbl$,
\begin{equation}
  \label{eq:mLambda=mLambdadbl}
  \kssfunc_{\Lambda_0}(G|_{V_u}) \simeq \kssfunc_{\Lambda_0}^\dbl(G) .
\end{equation}

Now we define a version of the morphism $b(F,G)$ of
Definition~\ref{def:hom_Psi-vers-muhom}.  Let
$G, G' \in \Der_{\Lambda, \shv}^\dbl(\cor_V)$ and set $\Lambda_0 = \SSi^\dbl(G)$,
$\Lambda'_0 = \SSi^\dbl(G')$ (they are open subsets of $\Lambda$).  We choose an open
subset $V'$ of $W_\gammaof$ satisfying the conclusions of Lemmas~\ref{lem:muhom_dbl0}
and~\ref{lem:SS_dbl} for $G,G'$.  We denote by $\pi_{\Lambda^q}$, $\pi_\Lambda$ the
projections from $\Lambda^q \cap T^*V'$, $\Lambda \cap T^*V$ respectively to $V'$,
$V$. The support of $\muhom(G,G')|_{\Lambda^q \cap T^*V'}$ is contained in
$\Lambda^q \cap T^*V'\cap {\ol{\Lambda_0}}^q \cap {\ol{\Lambda'_0}}^q$, hence the
support of $\muhom^\dbl(G,G')$ is contained in $\ol{\Lambda_0} \cap \ol{\Lambda'_0}$.
Since $\Lambda'_0$ is open in $\Lambda$ we have a natural morphism
$(-) \to \rsect_{\Lambda'_0}(-)$ and we deduce the following sequence of morphisms
\begin{equation*}
  \begin{split}
    \rhom(G,G')|_{V'}   &\simeq  \RR \pi_{V'*} (\muhom(G,G')|_{T^*V'}) \\
    &\to    \RR \pi_{\Lambda^q *} \,  (\muhom(G,G')|_{\Lambda^q \cap T^*V'}) \\
    &\simeq (\opb{q_V} \RR\pi_{\Lambda *} \, \muhom^\dbl(G,G') )|_{V'} \\
    &\to (\opb{q_V} \RR\pi_{\Lambda *}  \rsect_{\Lambda'_0} \, \muhom^\dbl(G,G'))|_{V'} .
  \end{split}
\end{equation*}
Let  
$i_V \colon V\times \{0\} \to M\times\R^2$ and $j \colon V' \to M \times\R^2$ be the
inclusions.  We apply the functor $\opb{i_V} \roim{j}$ to the above sequence and we
obtain a version of the morphism $b(F,G)$ of Definition~\ref{def:hom_Psi-vers-muhom}
for $\Der_{\Lambda, \shv}^\dbl(\cor_V)$:  
\begin{equation}\label{eq:hom_Psi-vers-muhomdbl}
\begin{split}
  b'(G,G') \colon  \opb{i_V} \roim{j}  &\rhom(G, G') \\
  &\to  \roim{\pi_\Lambda} \rsect_{\Lambda'_0} \, \muhom^\dbl(G,G') ,
\end{split}
\end{equation}
where $\Lambda'_0 = \SSi^\dbl(G')$.  The Proposition~\ref{prop:muhom=hompsi}
generalizes to this setting as follows.

\begin{proposition}
  \label{prop:muhom=hompsiEtendu}
  Let $G, G' \in \Der_{\Lambda, \shv}^\dbl(\cor_V)$.  Then the morphism $b'(G,G')$
  in~\eqref{eq:hom_Psi-vers-muhomdbl} is an isomorphism.  In particular, taking
  global sections gives
$$
  \varinjlim_{V'} \Hom(G|_{V'},G'|_{V'}) \isoto H^0(\Lambda'_0; \muhom^\dbl(G,G')),
$$
  where $V'$ runs over the open subsets of $M \times\R \times \rspos$ such that
  $(V \times \{0\}) \sqcup V'$ is open in $M \times\R \times \rpos$ and
  $\Lambda'_0 = \SSi^\dbl(G')$ is the open subset of $\Lambda \cap T^*V$ defined in
  Lemma~\ref{lem:SS_dbl}.
\end{proposition}
\begin{proof}
  Since the statement is local on $V$ we may as well assume that $G$ and $G'$ are
  decomposed as in~\eqref{eq:defDdbl} and it is enough to consider one summand in
  their decompositions.  By Lemma~\ref{lem:formules_dbl}-(iv-v) we can assume that
  $G = \Psi_W(\rsect_{V^B}(F))$ and $G' = \Psi_W(\rsect_{V^{B'}}(F'))$ for some open
  subset $W\subset V$, $F, F' \in \Derb_{[\Lambda]}(\cor_W)$ and $B,B' \subset A$.  By
  Lemma~\ref{lem:formules_dbl}-(iii) the restriction to $W$ of the left hand side
  of~\eqref{eq:hom_Psi-vers-muhomdbl} becomes
  \begin{equation}
    \label{eq:muhom=hompsiEtendu1}
    \opb{i_W}\roim{j} \rhom( G,G') 
 \simeq \rsect_{V^{B'}} (\opb{i_W}\roim{j} \rhom(G, \Psi_W(F'))) .
\end{equation}
We have  
$\muhom^\dbl(G,G') \simeq \muhom(\rsect_{V^B}(F),\rsect_{V^{B'}}(F'))|_\Lambda$ by
the construction of $\muhom^\dbl$ in Lemma~\ref{lem:muhom_dbl0}.  Let us set
$\Lambda' = \dot\SSi(F')$.  We have $\SSi^\dbl(G') = \Lambda' \cap T^*V^{B'}$ and the
restriction to $W$ of the right hand side of~\eqref{eq:hom_Psi-vers-muhomdbl} becomes
  \begin{align*}
    \RR\pi_{\Lambda *}  \rsect_{ \Lambda' \cap T^*V^{B'}} \,
    &(\muhom(\rsect_{V^B}(F),\rsect_{V^{B'}}(F'))|_\Lambda ) \\
    & \simeq \RR\pi_{\Lambda *}  \rsect_{ \Lambda' \cap T^*V^{B'}}\,
     ( \muhom(\rsect_{V^B}(F),F') |_\Lambda )      \\
    & \simeq \RR\pi_{\Lambda *}  \rsect_{T^*V^{B'}}\,
     ( \muhom(\rsect_{V^B}(F),F') |_\Lambda )      \\
    &  \simeq
      \rsect_{V^{B'}}\RR\pi_{\Lambda *}( \muhom(\rsect_{V^B}(F), F')|_\Lambda  ) ,
  \end{align*}
  where the first isomorphism follows from
  $\rsect_{ \Lambda' \cap T^*V^{B'}} \simeq \rsect_{ \Lambda' } \rsect_{T^*V^{B'}}$
  and the fact that $\rsect_U(H)$ only depends on $H|_U$ when $U$ is open (here
  $U = T^*V^{B'}$) and the second isomorphism follows from the inclusion
  $\supp(\muhom(H,F')|_\Lambda) \subset \dot\SSi(F') = \Lambda'$, whatever $H$.
  Since $\dot\SSi(F') \subset \{\tau>0\}$ we can apply
  Proposition~\ref{prop:muhom=hompsi} and we obtain that the right hand side
  of~\eqref{eq:hom_Psi-vers-muhomdbl} is isomorphic to
  $$
  \rsect_{V^{B'}} (\opb{i_W}\roim{j} \rhom(\Psi_W(\rsect_{V^B}(F)), \Psi_W(F'))).
  $$
  Comparing with~\eqref{eq:muhom=hompsiEtendu1} we obtain the result.

  The last assertion follows by applying the functor $H^0(V;-)$ to both sides
  of~\eqref{eq:hom_Psi-vers-muhomdbl} and using
  $H^0(V'; \rhom(G,G')) \simeq \Hom(G|_{V'}, G'|_{V'})$.
\end{proof}

  We have seen a notion of pure doubled sheaf
and we have defined $\kssfunc_{\Lambda_0}^\dbl$ in the paragraph
before~\eqref{eq:def_kssfuncLambda0dbl}.  In the same way that we deduced
Corollary~\ref{cor:kssfuncPsi} from Proposition~\ref{prop:muhom=hompsi}, we obtain
the following result, using Proposition~\ref{prop:muhom=hompsiEtendu}
and~\eqref{eq:hom_kssfuncLambda0dbl}:

\begin{corollary}\label{cor:kssfuncPsidbl}
  Let $\Lambda_0 \subset \Lambda \cap T^*V$ be an open subset.  Let
  $G, G' \in \Der_{\Lambda, \shv}^\dbl(\cor_V)$ be such that
  $\Lambda_0 \subset \SSi^\dbl(G)$ and $\SSi^\dbl(G') = \Lambda_0$.  We assume
  that $G$ and $G'$ are pure with the same shift.  Then
\begin{equation}\label{eq:Homdbl=Homkssfunc}
  \varinjlim_V \Hom(G|_V, G'|_V) \simeq
  \Hom(\kssfunc_{\Lambda_0}^\dbl(G), \kssfunc_{\Lambda_0}^\dbl(G')) ,
\end{equation}
  where $V$ runs over the open subsets of $U_\gammaof$ such that
  $(U \times \{0\}) \sqcup V$ is open in $M \times\R \times \rpos$.  In particular,
  if $\SSi^\dbl(G) = \SSi^\dbl(G') = \Lambda_0$ and
  $\kssfunc_{\Lambda_0}^\dbl(G) \simeq \kssfunc_{\Lambda_0}^\dbl(G')$, then there
  exists such an open subset $V$ such that $G|_V \simeq G'|_V$.
\end{corollary}
We remark that the hypothesis $\SSi^\dbl(G') = \Lambda_0$ implies that $\Lambda_0$ is
not any open subset of $\Lambda$ but of the form described in~(i) of
Lemma~\ref{lem:SS_dbl}.    In
Lemma~\ref{lem:descrptionfaiscdbl} below we describe the
isomorphism~\eqref{eq:Homdbl=Homkssfunc} locally when we have a decomposition of $G$
and $G'$ as in~\eqref{eq:defDdbl}.  Before that we show that the subsets
$\Lambda \cap T^*V^B$, $B\subset A$, have a local connectedness property near the
boundary.  For example the following situation is excluded.  Assume $M=\R$ and
$\Lambda$ is half of the conormal bundle of the cusp $\{(x,t)$; $x^3 = t^2 \}$, say
$\Lambda = \{(z^2,z^3; -\frac32 z\tau,\tau)$; $z\in \R$, $\tau>0\}$ (this example is
explained in Lemma~\ref{lem:sheafcusp}) and set $V = \mo]0,1\mc[\times \mo]-1,1[$.
Then $\Lambda \cap T^*V$ has two connected components but
$\Lambda \cap \pi_{M\times\R}^{-1}(\ol V)$ has only one component.  In fact $V$
cannot belong to an adapted family for $\Lambda$.  The next lemma says a bit more.

\begin{lemma}\label{lem:connex-LambdainterTV}
  Let $\Lambda$ and $\shv = \{V_a$; $a\in A\}$ be as in
  Definition~\ref{def:horiz_subset}.  Then, for any $B\subset A$,
  $\Lambda \cap T^*V^B$ has the following local connectedness property. For any
  $(x,t) \in M\times\R$ and any small enough neighborhood $W$ of $(x,t)$, denoting by
  $\Lambda \cap T^*W = \bigsqcup_{j\in J} \Lambda_j$ and
  $\Lambda_j \cap T^*V^B= \bigsqcup_{k\in K_j} \Lambda_j^k$ the decompositions into
  connected components we have: for any $j\in J$ there exists at most one $k\in K_j$
  such that $\ol{\Lambda_j^k} \cap T^*_{(x,t)}(M\times\R)$ is non empty.  In other
  words, there exists a smaller neighborhood $W'$ of $(x,t)$ such that the inclusion
  of $\Lambda_j \cap T^*V^B \cap T^*W'$ in $\Lambda_j \cap T^*V^B$ factorizes through
  a connected set.
\end{lemma}

\begin{remark}\label{rem:connex-LambdainterTV}
  A stronger statement would be that the subsets $\Lambda_j \cap T^*V^B$ in the lemma
  are connected (that is, $K_j$ is a singleton), but this would require a good choice
  of $W$.  The lemma only says that, when restricting to a smaller neighborhood of
  $T^*_{(x,t)}(M\times\R)$, at most one component of $\Lambda_j \cap T^*V^B$
  survives.
\end{remark}

\begin{proof}
  (i) We set $F_\Lambda = \roim{(\dot\pi_{M\times\R})}(\cor_\Lambda)$.  Let us prove
  that the natural morphism
  $u\colon (F_\Lambda)_{\ol{V^B}} \to \rsect_{V^B}(F_\Lambda)$ is an isomorphism, for
  any $B \subset A$.  By Theorem~\ref{thm:SSrhom} and the condition
  $\DD'(\cor_{V^B}) \simeq \cor_{\ol{V^B}}$, it is enough to see that
  $\SSi(\cor_{V^B})$ and $\SSi(F_\Lambda)$ do not meet outside the zero section.
  This follows from~\eqref{eq:def_horiz_subset} and the inclusion
  $\SSi(F_\Lambda) \subset \Lambda_+^a \hplus \Lambda_+$, where
  $\Lambda_+ = \Lambda \cup T^*_{M\times\R}(M\times\R)$, that we prove now.  Since
  this is a local problem, we can assume by Lemma~\ref{lem:simple_local_base} that
  there exists $F$ such that $\SSi(F) \subset \Lambda_+$ and
  $\muhom(F,F)|_{\dT^*(M\times\R)} \simeq\cor_\Lambda$ (see~\eqref{eq:carc_Fsimple}).
  Then the inclusion follows from the triangle~\eqref{eq:SatoDTmuhom2}, the
  triangular inequality for the microsupport and the bounds in
  Lemma~\ref{lem:im_inv_homprime} and Theorem~\ref{thm:SSrhom}.

  \sui(ii) Let $J$ be as in the lemma and let $J'$ be the set of half lines in
  $\Lambda \cap T^*_{(x,t)}(M\times\R)$.  We choose $W$ small enough so that the
  obvious map $J' \to J$ is injective.  Applying $H^0(W;-)$ to the isomorphism $u$
  of~(i) we obtain
  \begin{equation*}
  H^0(\dT^*W; \cor_{\Lambda \cap \pi_{M\times\R}^{-1}(\ol{V^B})}) 
  \isoto
  H^0(\dT^*W \cap \pi_{M\times\R}^{-1}(V^B); \cor_{\Lambda  }) .
  \end{equation*}
  This says that the connected components of
  $\Lambda \cap \pi_{M\times\R}^{-1}(\ol{V^B}) \cap \dT^*W$ and
  $\Lambda \cap \pi_{M\times\R}^{-1}(V^B) \cap \dT^*W$ are in bijection.  Now we
  assume that there exist $j\in J$ and $k,k' \in K_j$ such that $\ol{\Lambda_j^{k'}}$
  and $\ol{\Lambda_j^k}$ both meet $T^*_{(x,t)}(M\times\R)$.  Then
  $\ol{\Lambda_j^{k'}}$ and $\ol{\Lambda_j^k}$ both contain the same half line of
  $J'$. In particular they have a non empty intersection and must be the same
  connected component of $\Lambda \cap \pi_{M\times\R}^{-1}(\ol{V^B}) \cap \dT^*W$,
  that is, $k=k'$, as required.
\end{proof}

The next result says that the local decomposition~\eqref{eq:defDdbl} can be written
in a canonical way and describe the morphism~\eqref{eq:Homdbl=Homkssfunc} locally
when we have such a decomposition.

\begin{lemma}\label{lem:descrptionfaiscdbl}
  Let $V$ be an open subset of $M\times\R$ and let $(x,t) \in V$ be given.  Let
  $\{\lambda_i\}_{i\in I}$ be the set of half lines in
  $\Lambda \cap T^*_{(x,t)}(M\times\R)$. Let $W_0$ be a neighborhood of $(x,t)$ small
  enough so that the map $I \to \pi_0(\Lambda\cap T^*W_0)$ is injective; let
  $\Lambda_i$ be the connected component of $\Lambda \cap T^*W_0$ containing
  $\lambda_i$. We assume that, for each $i\in I$, there exists a simple sheaf
  $F_i \in \Der_{[\Lambda_i]}(\cor_{W_0})$.  Now let
  $G, G' \in \Der_{\Lambda, \shv}^\dbl(\cor_V)$ be pure objects with the same
  shift. We set $\Lambda_0 = \SSi^\dbl(G)$, $\Lambda'_0 = \SSi^\dbl(G')$.  Then we
  have:

  \sui{(i)} There exists an isomorphism, for some smaller neighborhood $W$ of
  $(x,t)$,
  $$
  G|_{W_\gammaof}
  \simeq \bigoplus_{i\in I} \rsect_{W_i \times \rspos}  (\Psi_W(F_i \ltens (E_i)_W)),
  $$
  where $W_i = W \cap \dot\pi_{M\times\R}(\Lambda_i \cap \Lambda_0)$ and
  $E_i \in \Der(\cor)$ is given by $E_i = (\muhom^\dbl(\Psi_W(F_i), G))_{p_i}$ for
  any $p_i\in \Lambda_i \cap \Lambda_0$ (and $E_i=0$ if $\Lambda_i \cap \Lambda_0$ is
  empty).

  \sui{(ii)} We assume that $\Lambda'_0 \subset \Lambda_0$ and we define $W'_i$,
  $E'_i$ like $W_i$, $E_i$ in~(i), choosing the same $p_i$ for $G$ and $G'$ when
  $\Lambda_i \cap \Lambda'_0 \not=\emptyset$.  For a given $u$ in
  $$
  \Hom(\kssfunc_{\Lambda'_0}^\dbl(G), \kssfunc_{\Lambda'_0}^\dbl(G'))
  \simeq
  H^0(\Lambda'_0; \muhom^\dbl(G,G'))
  $$
  we let $u_i \colon E_i \to E'_i$, $e \mapsto u_{p_i} \mucirc e$, be the
  morphism induced by the composition~\eqref{eq:comp_muhom} (we use the
  notation~\ref{not:mucomposition}).  Then, up to shrinking $W$, the inverse image of
  $u|_{T^*W} \in \Hom(\kssfunc_{T^*W \cap \Lambda'_0}^\dbl(G), \kssfunc_{T^*W \cap
    \Lambda'_0}^\dbl(G'))$ through~\eqref{eq:Homdbl=Homkssfunc} (replacing $V$ in the
  corollary by $W$) is represented by $\bigoplus_{i\in I} v_i$, where
  $$
  v_i \colon \rsect_{W_i \times \rspos} (\Psi_W(F_i \ltens (E_i)_W)) \to
  \rsect_{W'_i \times \rspos} (\Psi_W(F_i \ltens (E'_i)_W))
  $$
  is the composition of the morphism
  $\rsect_{W_i \times \rspos} (-) \to \rsect_{W'_i \times \rspos} (-)$ induced by the
  open inclusion $W'_i \subset W_i$ and the morphism
  $\Psi_W(F_i \ltens (E_i)_W) \to \Psi_W(F_i \ltens (E'_i)_W)$ induced by $u_i$.
\end{lemma}
\begin{proof}
  (i) We first take $W$ so that we have a decomposition~\eqref{eq:defDdbl}
  $G|_{W_\gammaof} \simeq \bigoplus_{i\in I'} \rsect_{V^{A_i} \times \rspos}
  (\Psi_W(F'_i))$ where $I' = \pi_0(\Lambda\cap T^*W)$ and
  $F'_i \in \Derb_{[\Lambda]}(\cor_W)$ ($I'$ contains $I$ but could be bigger).
  Shrinking $W$ we can forget the components in $I'\setminus I$ (maybe
  $\Lambda_i \cap T^*W$ will be no longer connected but we don't care).  Hence we can
  assume $I'=I$.

  Up to shrinking $W$ several times, Lemma~\ref{lem:simple_local} implies that, for
  each $i\in I$, $\kssfunc_{\Lambda_i}(F'_i)$ is isomorphic to
  $\kssfunc_{\Lambda_i}(F_i \ltens (E^0_i)_W)$, where
  $E^0_i = (\muhom(F_i,F'_i))_{p_i}$, and Corollary~\ref{cor:kssfuncPsi} implies
  $\Psi_W(F'_i) \simeq \Psi_W(F_i \ltens (E^0_i)_W)$.  We remark that
  \begin{align*}
  \muhom(F_i,F'_i)|_{\Xi_i}
  &\simeq \muhom^\dbl(\Psi_W(F_i), \Psi_W(F'_i))|_{\Xi_i} \\
  &\simeq \muhom^\dbl(\Psi_W(F_i), G)|_{\Xi_i}, 
  \end{align*}
  where $\Xi_i = \Lambda_i \cap \SSi^\dbl(G) \cap T^*W$. Hence $E^0_i \simeq
  E_i$. Finally
  $W \cap V^{A_i} = W \cap \dot\pi_{M\times\R}(\Lambda_i \cap \Lambda_0)$ by
  Lemma~\ref{lem:SS_dbl}~(i), proving the formula in~(i).
  
  \sui(ii) Using the decomposition in~(i) for $G$ and $G'$, we can assume that
  $G = \rsect_{W_i \times \rspos} (\Psi_W(F_i \ltens (E_i)_W))$,
  $G' = \rsect_{W'_i \times \rspos} (\Psi_W(F_i \ltens (E'_i)_W))$.  The hypothesis
  that $G,G'$ are pure with the same shifts says that the complexes $E_i$, $E'_i$ are
  concentrated in the same degree.  Hence
  $\muhom^\dbl(G,G')|_{\Lambda'_0 \cap \Lambda_i}$ is a constant sheaf concentrated
  in degree $0$ with stalks $\Hom(E_i, E'_i)$.  By
  Lemma~\ref{lem:connex-LambdainterTV} we may consider that
  $\Lambda'_0 \cap \Lambda_i$ is connected, up to shrinking $W$.  Then $u$ is
  determined by its germ at $p_i$.  By the construction of $v_i$ in the lemma, the
  morphism $v_i^\mu$ has the same germ as $u$ at $p_i$.  Hence $v_i$ represents $u$.
\end{proof}

\begin{corollary}
  \label{cor:decompo_dblsh}
  Let $G \in \Der_{\Lambda, \shv}^\dbl(\cor_V)$ and
  $\Lambda_0 = \SSi^\dbl(G) \subset T^*V$ (see Lemma~\ref{lem:SS_dbl}).  Let
  $\Lambda_0 = \Lambda_0^1 \sqcup \Lambda_0^2$ be a decomposition of $\Lambda_0$ into
  two open and closed subsets.  Then there exists an open subset $V'$ of $V_\gammaof$
  such that $(V \times \{0\}) \sqcup V'$ is open in $M \times\R \times \rpos$ and
  there exist $G_1 ,G_2 \in \Der_{\Lambda, \shv}^\dbl(\cor_{V'})$ such that
  $\SSi^\dbl(G_i) = \Lambda_0^i\cap T^*V'$, $i=1,2$, and
  $G|_{V'} \simeq G_1|_{V'} \oplus G_2|_{V'}$.
\end{corollary}
\begin{proof}
  (i) By Proposition~\ref{prop:muhom=hompsiEtendu} we have
  \begin{equation}
    \label{eq:decompo_dblsh1}
 \varinjlim_{V'} \Hom(G|_{V'},G|_{V'}) \isoto H^0(\Lambda_0; \muhom^\dbl(G,G)),
\end{equation}
where $V'$ runs over the open subsets of $M \times\R \times \rspos$ such that
$(V \times \{0\}) \sqcup V'$ is open in $M \times\R \times \rpos$. The identity
morphism of $G$ induces an element $\mathbf{1}_G$ of the left hand side
of~\eqref{eq:decompo_dblsh1}.  We let $\mathbf{1}^\mu_G$ be the corresponding section
of $\muhom^\dbl(G,G)$.  Since $\Lambda_0$ is split into two open and closed subsets
we have
  $$
  H^0(\Lambda_0; \muhom^\dbl(G,G)) \simeq
\bigoplus_{i = 1,2}  H^0(\Lambda_0^i; \muhom^\dbl(G,G))
  $$
  and we write $\mathbf{1}^\mu_G = e_1 + e_2$ according to this decomposition.  Since
  $\Lambda_0^1$ and $\Lambda_0^2$ are disjoint we have $e_1 e_2 = e_2 e_1 = 0$. Hence
  $e_1$ and $e_2$ are orthogonal idempotents.  By~\eqref{eq:decompo_dblsh1} again we
  deduce a decomposition $\mathbf{1}_G = f_1 + f_2$ where $f_1$ and $f_2$ are also
  orthogonal idempotents.

  \sui(ii) We can find $V'$ as in~\eqref{eq:decompo_dblsh1} and $f'_1$, $f'_2$ in
  $\Hom(G|_{V'}, G|_{V'})$ which represent $f_1$, $f_2$.  Up to shrinking $V'$ we can
  also assume that $f'_1$, $f'_2$ are orthogonal idempotents
  (using~\eqref{eq:decompo_dblsh1} again).  By~\cite[Prop.~3.2]{BN93} we deduce a
  corresponding decomposition $G|_{V'} \simeq G_1 \oplus G_2$ in $\Derb(\cor_{V'})$.
  
  \sui(iii) We extend $G_1, G_2$ arbitrarily to $V_\gammaof$.  It remains to check
  that $G_i \in \Der_{\Lambda, \shv}^\dbl(\cor_V)$ and
  $\SSi^\dbl(G_i) = \Lambda_0^i\cap T^*V'$.  We will prove that we have a local
  decomposition~\eqref{eq:defDdbl} of $G_i$ in a neighborhood $W$ of any given point
  $(x,t) \in V$. We first choose $W$ so that $W_\gammaof \subset V'$ and a
  decomposition~\eqref{eq:defDdbl} holds for $G|_{W_\gammaof}$ (and we use the
  corresponding notations $I$, $A_i$).    For $i\in I$ the set
  $\Lambda_i \cap T^* V^{A_i}$ may be non connected, but, by
  Lemma~\ref{lem:connex-LambdainterTV}, if we start with $W$ small enough, at most
  one component of $\Lambda_i \cap T^* V^{A_i}$ can meet $T^*W'$ for some smaller
  neighborhood $W'$.  Hence $\Lambda_i \cap T^* V^{A_i} \cap T^*W'$ is contained in
  either $\Lambda_0^1$ or $\Lambda_0^2$. We choose $W'$ so that this holds for all
  $i\in I$.

  We define $G'_1$ (resp. $G'_2$) to be the sum of the summands of $G|_{W'_\gammaof}$
  in~\eqref{eq:defDdbl} indexed by the $j\in I$ so that
  $\Lambda_j \cap T^* V^{A_j} \cap T^*W'$ is contained in $\Lambda_0^1$
  (resp. $\Lambda_0^2$).  Then $G'_i \in \Der_{\Lambda, \shv}^\dbl(\cor_V)$.  This
  gives another decomposition of $G|_{W'_\gammaof}$ and corresponding idempotents
  $f_1''$, $f_2''$. Since
  $\SSi(G'_i) \cap \Lambda = T^*W'_\gammaof \cap \Lambda_0^i$, the section of
  $\muhom^\dbl(G,G)|_{T^*W' \cap \Lambda}$ associated with $f''_i$
  by~\eqref{eq:decompo_dblsh1} must be $e_i|_{T^*W' \cap \Lambda}$. Up to restricting
  $W'$ once more, we deduce $f''_i = f'_i|_{W'_\gammaof}$ and then
  $G'_i \simeq G_i|_{W'\times \mo]0,\varepsilon[}$, which proves the result.
\end{proof}

\part{Quantization}
\label{part:quant}

The main result of this part is that, for any global object $\shf$ of the
Kashiwara-Schapira stack $\kss(\cor_{\Lambda})$, there exists
$F\in \Derb(\cor_{M\times\R})$ with $\dot\SSi(F) = \Lambda$ which represents
$\shf$, when $\Lambda$ is the conification of a compact exact Lagrangian
submanifold of $T^*M$ (see~\eqref{eq:conif_Lambdatilde} below).  This recovers a result
of Viterbo in~\cite{V19} who proves the existence of such a sheaf using Floer
theory (the proof of~\cite{V19} was sketched in~\cite{V11} in 2011).

\smallskip

We first consider a compact Legendrian submanifold of $J^1M$, or equivalently, a
closed conic Lagrangian submanifold $\Lambda$ of $T^*_{\tau >0}(M\times \R)$ such
that $\Lambda/\rspos$ is compact.  We apply the procedure sketched in the
introduction of Part~\ref{chap:conv_mic} and we prove that any object
$\shf \in \kss(\cor_{\Lambda})$, or $\shf \in \kss_{/[1]}(\cor_{\Lambda})$, is
represented by some $F\in \Derb(\cor_{M\times\R})$, or
$F\in \Der_{/[1]}(\cor_{M\times\R})$, such that
$\dot\SSi(F) = \Lambda \sqcup T_\varepsilon(\Lambda)$ for $\varepsilon>0$ small,
where $T_\varepsilon$ is the translation along the factor $\R$.  Then, assuming that
$\Lambda/\rspos$ has no Reeb chord we prove that there exists another representative
$F'$ such that $\dot\SSi(F') = \Lambda$.  We recall that ``having no Reeb chord''
means that $\rho_M \colon T^*_{\tau >0}(M\times\R) \to T^*M$,
$(x,t;\xi,\tau) \mapsto (x;\xi/\tau)$, induces an injection
$\Lambda/\rspos \hookrightarrow T^*M$.  The image
$\widetilde\Lambda = \rho_M(\Lambda)$ is then a compact exact Lagrangian submanifold
of $T^*M$ in the sense that $\alpha_M|_{\widetilde\Lambda}$ is exact.  The link
between $\Lambda$ and $\widetilde\Lambda$ is given by
\begin{equation}
  \label{eq:conif_Lambdatilde}
\Lambda = \{(x,t;\xi,\tau);\; \tau>0, \; (x;\xi/\tau) \in \widetilde\Lambda,
\; t = -f(x;\xi/\tau) \} ,
\end{equation}
where $f$ is a primitive of $\alpha_M|_{\widetilde\Lambda}$.

\medskip

Now we introduce some notations.  We first remark that, by Theorem~\ref{thm:GKS}, we
can as well move $\Lambda$ by any contact isotopy of $J^1M$ and assume from the
beginning that it satisfies some genericity hypotheses.    Hence we can assume that the map $\Lambda/\rspos \to M$ has finite
fibers and, by Lemma~\ref{lem:exist_adapted_family}, we can assume that there exists
an adapted family $\shv = \{V_a$; $a\in A\}$ in the sense of
Definition~\ref{def:horiz_subset}.  We let $\Lambda_b$, $b\in B_a$, be the family of
components of $\Lambda \cap T^*V_a$.  We set $B = \bigsqcup_{a\in A} B_a$.

We have introduced the subcategory $\Der_{\Lambda, \shv}^\dbl(\cor_{M\times\R})$ of
$\Der(\cor_{M\times\R\times\rspos})$ in Definition~\ref{def:cat_dbl_sheaves}.  For
$G \in \Der_{\Lambda, \shv}^\dbl(\cor_{M\times\R})$ we have defined $\SSi^\dbl(G)$
in Lemma~\ref{lem:SS_dbl}; it is an open subset of $\Lambda$ of the form
$\bigcup_{b\in B'} \Lambda_b$ for some $B'\subset B$.  For an open subset $\Lambda_0$
of $\Lambda$ we also have a functor (see~\eqref{eq:def_kssfuncLambda0dbl})
\begin{equation*}
  \kssfunc_{\Lambda_0}^\dbl \cl \{G \in \Der_{\Lambda, \shv}^\dbl(\cor_V); \;
  \Lambda_0 \subset \SSi^\dbl(G)\}  \to \kss(\cor_{\Lambda_0}) .
\end{equation*}

The same definitions make sense for the orbit category.  We define the subcategory
$\Der_{/[1],\Lambda, \shv}^\dbl(\cor_{M\times\R})$ of
$\Der_{/[1]}(\cor_{M\times\R\times\rspos})$ as in
Definition~\ref{def:cat_dbl_sheaves}, replacing everywhere $\Der(\cor_X)$ by
$\Der_{/[1]}(\cor_X)$.  We can also define $\SSi^\dbl(G)$ for
$G \in \Der_{/[1],\Lambda, \shv}^\dbl(\cor_{M\times\R})$ and a functor
\begin{equation}
  \label{eq:def_kssfuncorbdbl}
  \kssfunc_{/[1],\Lambda_0}^\dbl \cl
 \{G \in \Der_{/[1],\Lambda, \shv}^\dbl(\cor_V); \; \Lambda_0 \subset \SSi^\dbl(G)\} \to
  \kss_{/[1]}(\cor_{\Lambda_0}) .
\end{equation}

In Theorems~\ref{thm:quant_legendrian_dbl} and~\ref{thm:quant_legendrian_orb_dbl} we
see that the functors $\kssfunc_{\Lambda_0}^\dbl$ and
$\kssfunc_{/[1],\Lambda_0}^\dbl$ are essentially surjective.  For a given
$G \in \Der_{\Lambda, \shv}^\dbl(\cor_{M\times\R})$ the microsupport of
$G|_{M\times\R \times \{\varepsilon\}}$, for $\varepsilon>0$ small enough, is made of
two copies of $\Lambda$.  In Corollary~\ref{cor:exist_quant} we see that, if we have
no Reeb chords, we can translate one copy of $\Lambda$ vertically using a Hamiltonian
isotopy and obtain an object of $\Derb_{[\Lambda]}(\cor_{M\times\R})$ from a given
one in $\Der_{\Lambda, \shv}^\dbl(\cor_{M\times\R})$.  In~\S\ref{sec:restr_infty} we
see a relation between objects of $\Derb_{[\Lambda]}(\cor_{M\times\R})$, their
restrictions to $M\times\{t_0\}$, $t_0\gg0$, and their microlocalizations; in
particular we see that $F \mapsto F|_{M\times\{t_0\}}$ induces a fully faithful
functor from sheaves on $M\times\R$ with microsupport $\Lambda$ and vanishing on
$M\times\{t\}$, $t\ll0$, and locally constant sheaves on $M$.

\section{Quantization for the doubled Legendrian}
\label{sec:quant-dbl}

Let $\Lambda$ be a closed conic Lagrangian submanifold of $T^*_{\tau >0}(M\times \R)$
such that $\Lambda/\rspos$ is compact.  By Lemma~\ref{lem:exist_adapted_family}, we
can assume that there exists an adapted family $\shv = \{V_a$; $a\in A\}$ for
$\Lambda$ which is stable by intersection. We can also assume that
$\Lambda \cap T^*V_a$ has finitely many connected components, for each $a\in A$.  We
let $\Lambda_b$, $b\in B_a$, be the family of components of $\Lambda \cap T^*V_a$.
We set $B = \bigsqcup_{a\in A} B_a$ and we let $\sigma \colon B\to A$ be the obvious
map.  Hence $\Lambda_b$ is a component of $\Lambda \cap T^*V_{\sigma(b)}$.
  When we use Lemma~\ref{lem:exist_adapted_family}
we can also assume that the family $\{\Lambda_b\}_{b\in B}$ refines any given family.
Since we know that simple sheaves along $\Lambda$ locally exist
(Lemma~\ref{lem:simple_local_base}), we can assume that, for each $b\in B$, there
exist a neighborhood $V'_b$ of $\ol{V_{\sigma(b)}}$, a contractible component
$\Lambda'_b$ of $\Lambda \cap T^*V'_b$ and $F_b \in \Derb(\cor_{V'_b})$ such that
$\Lambda_b$ is a component of $\Lambda'_b \cap T^*V_{\sigma(b)}$,
$\dot\SSi(F_b) = \Lambda'_b$ and $F_b$ is simple.

\begin{theorem}\label{thm:quant_legendrian_dbl}
  In the above setting let $B'$ be a subset of $B$ and set
  $\Lambda_0 = \bigcup_{b\in B'} \Lambda_b$,
  $\Lambda'_0 = \bigcup_{b\in B'} \Lambda'_b$.  Then, for any pure object
  $\shf \in \kss(\cor_{\Lambda'_0})$ there exists
  $F \in \Der_{\Lambda, \shv}^\dbl(\cor_{M\times\R})$ such that
  $\SSi^\dbl(F) = \Lambda_0$ and
  $\kssfunc_{\Lambda_0}^\dbl(F) \simeq \shf|_{\Lambda_0}$, where
  $\kssfunc_{\Lambda_0}^\dbl$ is defined in~\eqref{eq:def_kssfuncLambda0dbl}.
\end{theorem}
\newcommand{\UUU}{O} \newcommand{\VVV}{{O'}}

\begin{proof}
  (i) We proceed by induction on $|B'|$.  Let $b\in B$.  Since $\Lambda'_b$ is
  contractible, the objects of $\kss(\cor_{\Lambda'_b})$ are of the form
  $\kssfunc_{\Lambda'_b}(E_{V'_b} \ltens F_b)$ for some $E \in \Derb(\cor)$.  So we
  write $\shf|_{\Lambda'_b} = \kssfunc_{\Lambda'_b}(E_{V'_b} \ltens F_b)$ and we set
  $$
  G = \rsect_{(V_{\sigma(b)} \times \rspos) \cap (V'_b)_\gammaof}
  (\Psi_{V'_b}(E_{V'_b} \ltens F_b)),
  $$
  extended by zero outside $(V'_b)_\gammaof$ (the formula defines $G$ on
  $(V'_b)_\gammaof$ with a support contained in
  $(V'_b)_\gammaof \cap (\ol{V_{\sigma(b)}} \times \rspos)$).  Let us prove that $G$
  belongs to $\Der_{\Lambda, \shv}^\dbl(\cor_{M\times\R})$.  We check
  Definition~\ref{def:cat_dbl_sheaves} around a point $(x,t) \in M\times \R$:

  If $(x,t) \not\in \ol{V_{\sigma(b)}}$, we choose $W$ such that
  $W \cap \ol{V_{\sigma(b)}} = \emptyset$ and we have $G|_{W_\gammaof} =0$
  so~\eqref{eq:defDdbl} is trivial.

  If $(x,t) \in V'_b$, we choose $W = V'_b$ and the defining formula for $G$
  satisfies~\eqref{eq:defDdbl}. We remark that the family $I$ in~\eqref{eq:defDdbl}
  consists of one element, say $I=\{i_0\}$, and $A_{i_0} = \{\sigma(b)\}$.

  \smallskip
  We have $\SSi^\dbl(G) = \Lambda'_b \cap T^*V_{\sigma(b)}$ and $\Lambda_b$ is a
  connected component of $\SSi^\dbl(G)$.  By Corollary~\ref{cor:decompo_dblsh} there
  exists $G_b, G' \in \Der_{\Lambda, \shv}^\dbl(\cor_{M\times\R})$ such that
  $\SSi^\dbl(G_b) = \Lambda_b$, $\SSi^\dbl(G') = \SSi^\dbl(G) \setminus \Lambda_b$
  and $G|_\UUU \simeq G_b|_\UUU \oplus G'|_\UUU$ for some open subset $\UUU$ of
  $M\times\R\times\rspos$ satisfying
  \begin{equation}
    \label{eq:condvalbordV}
  \text{$(M\times\R\times \{0\})\cup \UUU$ is open in $M\times\R\times\rpos$.}  
  \end{equation}
  By construction we have
  $\kssfunc_{\Lambda_b}^\dbl(G_b) \simeq \shf|_{\Lambda_{\sigma(b)}}$, which proves
  the case $B' = \{b\}$.

  \sui(ii) Now we write $B' = B'' \cup \{b\}$ and assume that the result holds for
  $B''$. We set $\Lambda_1 = \bigcup_{c\in B''} \Lambda_c$ and
  $\Lambda_{1b} = \Lambda_1 \cap \Lambda_b$.  By the induction hypothesis and by~(i)
  there exist $G_1, G_b \in \Der_{\Lambda, \shv}^\dbl(\cor_{M\times\R})$ with
  $\SSi^\dbl(G_1) = \Lambda_1$, $\SSi^\dbl(G_b) = \Lambda_b$ and isomorphisms
  $\varphi_1 \colon \kssfunc_{\Lambda_1}^\dbl(G_1) \isoto \shf|_{\Lambda_1}$,
  $\varphi_b \colon \kssfunc_{\Lambda_b}^\dbl(G_b) \isoto \shf|_{\Lambda_b}$.  We
  have assumed that the adapted family $\{V_a\}_{a\in A}$ is stable by
  intersection. Hence
  $$
  G'_1 = \rsect_{(V_{\sigma(b)} \times \rspos)} ( G_1)
$$
belongs to $\Der_{\Lambda, \shv}^\dbl(\cor_{M\times\R})$ (this can be checked
directly on the local decomposition\eqref{eq:defDdbl}).  We have
$\SSi^\dbl(G'_1) = \Lambda_1 \cap T^*V_{\sigma(b)}$ and $\Lambda_{1b}$ is open and
closed in $\SSi^\dbl(G'_1)$.   As in~(i) we use Corollary~\ref{cor:decompo_dblsh} to find
$G_{1b}\in \Der_{\Lambda, \shv}^\dbl(\cor_{M\times\R})$ such that
$\SSi^\dbl(G_{1b}) = \Lambda_{1b}$ and $G_{1b}|_\VVV$ is the summand of $G'_1|_\VVV$
corresponding to $\Lambda_{1b}$, for some  open subset $\VVV$
satisfying~\eqref{eq:condvalbordV}.

We have a natural morphism $g_1 \colon G_1|_\VVV \to G_{1b}|_\VVV$ given by the composition
of the natural morphism $G_1 \to G'_1$ and the projection to the summand
$G'_1|_\VVV \to G_{1b}|_\VVV$.  By construction we have an isomorphism
$\varphi_{1b} \colon \kssfunc_{\Lambda_{1b}}^\dbl(G_{1b}) \isoto
\shf|_{\Lambda_{1b}}$ compatible with $\varphi_1$ in the sense that
$\varphi_1|_{\Lambda_{1b}} = \varphi_{1b} \circ \kssfunc_{\Lambda_{1b}}^\dbl(g_1)$
(we use
$\kssfunc_{\Lambda_{1b}}^\dbl(G_1) =
\kssfunc_{\Lambda_1}^\dbl(G_1)|_{\Lambda_{1b}}$).

We define
$\tilde g_b = \varphi_{1b}^{-1} \circ (\varphi_b|_{\Lambda_{1b}}) \colon
\kssfunc_{\Lambda_{1b}}^\dbl(G_b) \isoto \kssfunc_{\Lambda_{1b}}^\dbl(G_{1b})$.
Since $\shf$ is pure, we can apply Corollary~\ref{cor:kssfuncPsidbl} and, up to
shrinking $\VVV$, we obtain $g_b\colon G_b|_\VVV \to G_{1b}|_\VVV$ which represents
$\tilde g_b$. Finally we define $G \in \Derb(\cor_{M\times\R\times\rspos})$ by the
distinguished triangle on $\VVV$
\begin{equation}
  \label{eq:quant_legendrian1}
  G \to G_1|_\VVV \oplus G_b|_\VVV  \to[(g_1, -g_b)] G_{1b}|_\VVV \to[+1] 
\end{equation}
and by extending $G$ by zero outside $\VVV$. We prove in~(iii) that $G$ represents
$\shf$ over $\Lambda_0 = \Lambda_1 \cup \Lambda_b$.

\sui(iii) We first have to check that $G$ belongs to
$\Der_{\Lambda, \shv}^\dbl(\cor_{M\times\R})$, which means that it can be decomposed
as in~\eqref{eq:defDdbl}, around any given point $(x,t) \in M\times \R$.    We use
Lemma~\ref{lem:descrptionfaiscdbl} to describe $G_*$, for $* = 1$, $*=b$ or $*=1b$,
over a small enough neighborhood $W$ of $(x,t)$.  We have a partition
$\Lambda \cap T^*W = \bigsqcup_{i\in I} \Xi_i$ into connected components and a simple
sheaf $F_i \in \Der_{[\Xi_i]}(\cor_W)$ for each $i\in I$.  We choose $p_i \in \Xi_i$
and set $E_i = (\hom(\kssfunc_{\Xi_i}(F_i), \shf))_{p_i}$ (this does not depend on
$p_i$).  Over $\SSi^\dbl(G_*)$ we have
  \begin{align*}
    \muhom^\dbl(\Psi_W(F_i), G_*)
    &\simeq \hom(\kssfunc_{\Xi_i}^\dbl(\Psi_W(F_i)),  \kssfunc_{\Xi_i}^\dbl(G_*)) \\
    &\simeq \hom(\kssfunc_{\Xi_i}^\dbl(\Psi_W(F_i)), \shf) \\
    &\simeq \hom(\kssfunc_{\Xi_i}(F_i), \shf) . 
  \end{align*}
  Hence, setting for short $H_i = \Psi_W(F_i \ltens (E_i)_W)$ and maybe shrinking
  $W$, we have by Lemma~\ref{lem:descrptionfaiscdbl}~(i)
  \begin{equation*}
 G_*|_{W_\gammaof}
 \simeq \bigoplus_{i\in I} \rsect_{W^*_i \times \rspos}  (H_i),
 \qquad * = 1,\, b,\, 1b ,
  \end{equation*}
  where
  $W^*_i = W \cap \dot\pi_{M\times\R}(\Xi_i \cap \SSi^\dbl(G_*)) = W \cap
  \dot\pi_{M\times\R}(\Xi_i \cap \Lambda_*)$.

  By Lemma~\ref{lem:descrptionfaiscdbl}~(ii) the morphisms $g_1$, $g_b$
  in~\eqref{eq:quant_legendrian1} are obtained by composing morphisms
  $\alpha_i^* \colon \rsect_{W^*_i \times \rspos} (-) \to \rsect_{W^{1b}_i \times
    \rspos} (-)$, $*=1,b$, induced by the open inclusions $W^{1b}_i \subset W^*_i$,
  and morphisms $\beta_i^* \colon H_i \to H_i$, induced by
  $\kssfunc_{\Xi_i}^\dbl(g_*)$.  For the $i$'s such that
  $\Xi_i \cap \Lambda_{1b} \not=\emptyset$, $\kssfunc_{\Xi_i}^\dbl(g_*)$ is an
  isomorphism, and so is $\beta_i^*$ (for the other $i$'s we have
  $W_i^{1b} = \emptyset$).  Setting $\beta_i^* = \id$ when $W_i^{1b} = \emptyset$ we
  obtain the commutative square
  \begin{equation*}
    \def\flecheun{(g_1,-g_b)}
  \begin{tikzcd}
    \bigoplus_{i\in I} (\rsect_{W^1_i \times \rspos}  (H_i) \oplus 
    \rsect_{W^b_i \times \rspos}  (H_i) ) \ar[r, "\flecheun"]  \ar[d, "\beta", "\wr"']
   & \bigoplus_{i\in I} \rsect_{W^{1b}_i \times \rspos}  (H_i) \ar[d, equal]  \\
      \bigoplus_{i\in I} (\rsect_{W^1_i \times \rspos}  (H_i) \oplus 
    \rsect_{W^b_i \times \rspos}  (H_i) ) \ar[r, "\alpha"] 
   & \bigoplus_{i\in I} \rsect_{W^{1b}_i \times \rspos}  (H_i) ,
\end{tikzcd}
\end{equation*}
where $\alpha = \bigoplus_{i\in I} (\alpha^1_i , -\alpha^b_i)$ and
$\beta = \bigoplus_{i\in I} \bpmat \beta^1_i & 0 \\ 0& \beta^b_i\epmat$.  Since
$\beta$ is an isomorphism, we obtain that~\eqref{eq:quant_legendrian1} is isomorphic
to the sum of the Mayer-Vietoris distinguished triangles
\begin{align*}
  \rsect_{(W^1_i \cup W^b_i)\times\rspos}(H_i)
  \to \rsect_{W^1_i \times\rspos}(&H_i) \oplus \rsect_{W^b_i \times\rspos}(H_i) \\
 & \to \rsect_{W^{1b}_i \times\rspos}(H_i) \to[+1].
\end{align*}
It follows that
$G|_{W_\gammaof} \simeq \bigoplus_{i\in I} \rsect_{(W^1_i \cup
  W^b_i)\times\rspos}(H_i)$ which is a decomposition like~\eqref{eq:defDdbl}.  This
proves that $G \in \Der_{\Lambda, \shv}^\dbl(\cor_{M\times\R})$ and also that
$\SSi^\dbl(G) = \Lambda_0$.

By construction we also have $\kssfunc_{\Lambda_0}^\dbl(G) \simeq \shf|_{\Lambda_0}$
since $g_1$, $g_b$ are defined to represent the morphisms gluing $\shf|_{\Lambda_1}$
and $\shf|_{\Lambda_b}$ into $\shf|_{\Lambda_0}$.
\end{proof}

\section{The triangulated orbit category case}
\label{sec:quant_orbcat}

We checked in Sections~\ref{sec:def_orb_cat} and~\ref{sec:micsup_orb_cat} that the
results we used in the category $\Derb(\cor_M)$ have analogs in the category
$\Orb(\cor_M)$.  In particular we have already defined a Kashiwara-Schapira stack
$\kss_{/[1]}(\cor_{\Lambda})$ in this situation.  However the proof of
Theorem~\ref{thm:quant_legendrian_dbl} does not work because we used the fact that
some $\muhom$ sheaf was concentrated in degree $0$ (we used
Corollary~\ref{cor:kssfuncPsidbl}), which makes no sense in the orbit category.  We
prove the result for the triangulated orbit category by gluing two sheaves obtained
on two open subsets by Theorem~\ref{thm:quant_legendrian_dbl}.  For this we first
remark in Lemma~\ref{lem:decomp_deuxouverts} below that we can decompose $\Lambda$ in
two open subsets with vanishing Maslov classes.

\begin{lemma}\label{lem:decomp_deuxouverts}
  Let $X$ be a manifold and let $c\in H^1(X;\Z_X)$.  Then there exists a map
  $f\colon X \to S^1$ of class $\Cinf$ such that $c=f^*(\delta)$, where
  $\delta \in H^1(S^1;\Z_{S^1})$ is the canonical class.  In particular, for any
  covering $S_1 = I_1\cup I_2$ by two open intervals, the restrictions
  $c|_{f^{-1}(I_i)} \in H^1(f^{-1}(I_i);\Z_{f^{-1}(I_i)})$ vanish, for $i=1,2$.
\end{lemma}
\begin{proof}
  Since $H^1(X;\Z_X) \to H^1(X;\R_X)$ is injective, it is enough to prove the
  result for the image of $c$ in $H^1(X;\R_X)$, which we represent by a $1$-form
  $\alpha$.  Let $r\cl X' \to X$ be the universal covering of $X$ and let
  $g\cl X' \to \R$ be a primitive of $r^*(\alpha)$.  Then, for any
  $x_1,x_2 \in X'$ such that $r(x_1)=r(x_2)$ we have
  $g(x_1)-g(x_2) = \langle c,\gamma \rangle$, where $\gamma$ is the loop at
  $g(x_1)$ determined by $x_1,x_2$. Hence $g(x_1)-g(x_2)$ is an integer and $g$
  descends to a map $f\cl X \to S^1$ which satisfies the conclusion of the
  lemma.
\end{proof}

For the next result the ring is $\cor = \Z/2\Z$.
\begin{theorem}\label{thm:quant_legendrian_orb_dbl}
  For any global object $\shf$ of $\kss_{/[1]}(\cor_{\Lambda})$ there exists an
  object $F \in \Der_{/[1],\Lambda,\shv}^\dbl(\cor_{M\times\R})$, for some finite
  family $\shv$ of open subsets of $M\times\R$ which is adapted to $\Lambda$, such
  that $\kssfunc_{/[1],\Lambda}^\dbl(F) \simeq \shf$.
\end{theorem}
\begin{proof}
  (i) By Proposition~\ref{prop:KSstackorb} $\kss_{/[1]}(\cor_{\Lambda})$ has a unique
  simple object, $\shf_0$, and $\shf$ is of the type
  $\shf \simeq \shf_0 \epstens_{\cor_\Lambda} L$ for some $L \in \loc(\cor_\Lambda)$.

  We let $\mu_1(\Lambda) \in H^1(\Lambda;\Z_\Lambda)$ be the sheaf obstruction class
  of $\Lambda$ (which coincides with the Maslov class).  We apply
  Lemma~\ref{lem:decomp_deuxouverts} to obtain $f\colon \Lambda \to S^1$ such that
  $\mu_1(\Lambda)$ is the pull-back of the fundamental class of $S^1$.  We choose a
  covering $S_1 = I_+\cup I_-$ by two open intervals and set
  $\Lambda_\pm = f^{-1}(I_\pm)$.  Then $\mu_1(\Lambda)|_{\Lambda_\pm} = 0$ and the
  categories $\kss(\cor_{\Lambda_\pm})$ have simple objects, say
  $\shf_{0,\pm}$. Their images in $\kss_{/[1]}(\cor_{\Lambda_\pm})$ are
  $\shf_0|_{\Lambda_\pm}$ because $\kss_{/[1]}(\cor_{\Lambda_\pm})$ has a unique
  simple object.  The intersection $I_+ \cap I_-$ is the union of two intervals, say
  $I_a$, $I_b$.  We set $\Lambda_a = f^{-1}(I_a)$, $\Lambda_b = f^{-1}(I_b)$ and
  $\shf_{0,a} = \shf_{0,+}|_{\Lambda_a}$, $\shf_{0,b} = \shf_{0,+}|_{\Lambda_b}$.  We
  remark that simple objects in $\kss(\cor_{\Xi})$, for $\Xi$ open in $\Lambda$, with
  coefficients $\cor = \Z/2\Z$, are unique up to shift and up to a unique
  isomorphism.  Hence we have the following canonical isomorphisms
  \begin{alignat*}{2}
    \varphi_+^a = \id &\colon \shf_{0,a} \isoto \shf_{0,+}|_{\Lambda_a},
    &\quad
    \varphi_+^b = \id &\colon \shf_{0,b} \isoto \shf_{0,+}|_{\Lambda_b} , \\
    \varphi_-^a  &\colon \shf_{0,a} \isoto \shf_{0,-}|_{\Lambda_a}[d_a],
    &\quad
    \varphi_-^b  &\colon \shf_{0,b} \isoto \shf_{0,-}|_{\Lambda_b}[d_b] ,
  \end{alignat*}
  where $d_a,d_b \in \Z$ are locally constant functions on $\Lambda_a$, $\Lambda_b$.
  We set $\shf_{\pm} = \shf_{0,\pm} \otimes L$, $\shf_a = \shf_{0,a} \otimes L$,
  $\shf_b = \shf_{0,b} \otimes L$ and $\Phi_\pm^* = \varphi_\pm^* \otimes \id_L$, for
  $*=a,b$.

  \sui(ii) Up to shrinking $\Lambda_0$ and $\Lambda_1$ we can find a finite family
  $\shv = \{V_a$; $a\in A\}$ of open subsets of $M\times\R$ which is adapted to
  $\Lambda$ such that $\Lambda_+$ and $\Lambda_-$ are of the form $\Lambda^B$ for
  some $B \subset A$ (see Lemma~\ref{lem:exist_adapted_family}).  We can also assume
  that the adapted family is stable by intersection.  Hence $\Lambda_a$ and
  $\Lambda_b$ are also of the form $\Lambda^B$.

  By Theorem~\ref{thm:quant_legendrian_dbl} there exist
  $F_\bullet \in \Der_{\Lambda,\shv}^\dbl(\cor_{M\times\R})$, for
  $\bullet = +,-,a,b$, such that $\SSi^\dbl(F_\bullet) = \Lambda_\bullet$ and
  $\kssfunc_{/[1],\Lambda_\bullet}^\dbl(F_\bullet) \simeq \shf_\bullet$. Moreover,
  by Proposition~\ref{prop:muhom=hompsiEtendu} we have morphisms, for $*=a,b$,
  $\Psi_+^* \colon F_*|_V \to F_+|_V$ and $\Psi_-^* \colon F_*|_V \to F_-|_V[d_*]$
  representing $\Phi_\pm^*$, where $V$ is some open subset of $M\times\R\times\rspos$
  satisfying~\eqref{eq:condvalbordV}.  In $\Orb(\cor_V)$ we have
  $F_- \simeq F_-[d_*]$ and we can define $F \in \Orb(\cor_V)$ by the distinguished
  triangle
$$
F_a \oplus F_b \to[\left( \begin{smallmatrix}
  \Psi_+^a & \Psi_+^b \\   \Psi_-^a & \Psi_-^b
\end{smallmatrix} \right)] F_+ \oplus F_- \to F \to[+1] .
$$
We extend $F$ by zero outside $V$. We can then check as in part~(iii) of the proof of
Theorem~\ref{thm:quant_legendrian_dbl} that
$F \in \Der_{/[1],\Lambda,\shv}^\dbl(\cor_{M\times\R})$ and that
$\kssfunc_{/[1],\Lambda}^\dbl(F) \simeq \shf$.
\end{proof}

\begin{remark}\label{rem:exist_quantorb}
  By Proposition~\ref{prop:KSstackorb} $\kss_{/[1]}(\cor_{\Lambda})$ has a simple
  object. Hence Theorem~\ref{thm:quant_legendrian_orb_dbl} gives the existence of a
  simple object in $\Der_{/[1],\Lambda,\shv}^\dbl(\cor_{M\times\R})$, say $F_0$.
    In the case
  where $\Lambda$ has no Reeb chord, we will see that $\loc(\cor_\Lambda)$ is
  equivalent to $\loc(\cor_M)$ and all objects of
  $\Der_{/[1],\Lambda,\shv}^\dbl(\cor_{M\times\R})$ are of the form
  $F_0 \epstens_{\cor_{M\times\R}} L$ for some $L \in \loc(\cor_{M\times\R})$.
  (However in the course of the proof of the previous theorem this is unknown, so we
  cannot deduce a representative for $\shf$ from one for $\shf_0$.)
\end{remark}

\section{Translation of the microsupport}
\label{sec:transl_micsup}

Now we assume that the Legendrian submanifold $\Lambda/\rspos$ of $J^1M$ has no
Reeb chord, that is, the map $T^*_{\tau >0}(M\times\R) \to T^*M$,
$(x,t;\xi,\tau) \mapsto (x;\xi/\tau)$, induces an embedding
$\Lambda/\rspos \hookrightarrow T^*M$ (see~\eqref{eq:conif_Lambdatilde}).

For $u\in\R$ we define the translation $T_u \cl M\times\R \to M\times\R$, $(x,t)
\mapsto (x,t+u)$. We denote by $T'_u \cl T^*(M\times\R) \to T^*(M\times\R)$,
$(x,t;\xi,\tau) \mapsto (x,t+u;\xi,\tau)$, the induced map on the cotangent
bundle.  We also introduce some notations, for $\Lambda \subset
T^*_{\tau >0}(M\times\R)$:
\begin{equation}\label{eq:def-Lambdaplus}
\begin{split}
\Lambda_u &= \Lambda \cup T'_u(\Lambda) \quad \subset
 \quad T^*_{\tau >0}(M\times\R), \quad \text{for $u>0$,}  \\
\Lambda^+ &= q_d\opb{q_\pi}(\Lambda) \sqcup r_d\opb{r_\pi}(\Lambda)
\quad \subset \quad T^*_{\tau >0}(M\times\R\times\rspos) ,
\end{split}
\end{equation}
where $q,r \colon M\times\R\times\rspos \to M\times\R$ are given by
$q(x,t,u) = (x,t)$, $r(x,t,u) = (x,t-u)$.    We remark that
$\Lambda^+$ is non-characteristic for the inclusions
$i_u\colon M\times\R\times\{u\} \to M\times\R\times\rspos$, $u>0$, and that
$\Lambda_u = (i_u)_d(\opb{(i_u)_\pi}(\Lambda^+))$.

\begin{lemma}\label{lem:isotopy_transl}
  There exists $\phi\cl \dT^*(M\times\R) \times \rspos \to \dT^*(M\times\R)$, a
  homogeneous Hamiltonian isotopy, such that $\phi_1 = \id$ and, using the
  notations~\eqref{eq:def_compo_ensembles} and~\eqref{eq:def-Lambdaplus}, we
  have $\Gamma_\phi \circ^a \Lambda_1 = \Lambda^+$.  In particular
  $\phi_u(\Lambda_1) = \Lambda_u$, for all $u>0$.
\end{lemma}
\begin{proof}
  (i) We set $I=\rspos$.  Since the map $(x,t;\xi,\tau) \mapsto (x;\xi/\tau)$
  induces an injection $\Lambda/\rspos \hookrightarrow T^*M$, the sets $\Lambda$
  and $T'_u(\Lambda)$ are disjoint for all $u>0$.  Considering all $u>0$ at once
  we define the following closed subsets of $\dT^*(M\times \R) \times I$:
$$
\Lambda^0 = \Lambda \times \rspos ,  \qquad
\Lambda^1 = \bigsqcup_{u>0} (T'_u(\Lambda) \times \{u\}) .
$$
Then $\Lambda^0$ and $\Lambda^1$ are disjoint and the projections
$\Lambda^i/\rspos \to I$ are proper for $i=0,1$.  Hence we can find a conic
neighborhood $\Omega$ of $\Lambda^1$ in $\dT^*(M\times \R) \times I$ such that
$\ol\Omega \cap \Lambda^0 = \emptyset$ and the projection $\ol\Omega/\rspos \to
I$ is proper, that is, $\ol\Omega \cap (\dT^*(M\times \R) \times \{u\})$ is
compact for all $u>0$.

\medskip\noindent
(ii) We choose a $C^\infty$-function $h\cl \dT^*(M\times\R) \times I \to \R$
such that,
\begin{itemize}
\item [(a)] $h_u \eqdot h|_{\dT^*(M\times\R) \times \{u\}}$ is homogeneous of
  degree $1$, for all $u\in I$,
\item [(b)] $h$ vanishes outside $\Omega$,
\item [(c)] there exists a neighborhood $\Omega'$ of $\Lambda^1$ such that
  $h(x,t;\xi,\tau) = -\tau$, for all $(x,t;\xi,\tau) \in \Omega'$.
\end{itemize}
By~(a), (b) and the compactness of
$(\ol\Omega \cap (\dT^*(M\times \R) \times \{u\})) / \rspos$, the Hamiltonian
flow of $h$, say $\phi$, is defined on $I$ (the initial time here is $t_0=1$).
Then $\phi_u$ is the identity map outside $\Omega$ for all $u\in I$ and
$\phi_u(x,t;\xi,\tau) = (x,t+u-1;\xi,\tau)$, for all
$((x,t;\xi,\tau),u) \in \Omega'$.  Since
$\Lambda^0 \subset (\dT^*(M\times\R) \setminus \Omega)$ and
$\Lambda^1 \subset \Omega'$, the lemma follows.
\end{proof}

\begin{corollary}
  \label{cor:exist_quant}
  Let $\Lambda$ be a closed conic Lagrangian submanifold of
  $T^*_{\tau >0}(M\times \R)$ coming from a compact exact submanifold of $T^*M$ as
  in~\eqref{eq:conif_Lambdatilde}.  Then, for any pure object   $\shf \in \kss(\cor_\Lambda)$ there
  exists $F \in \Derb_{[\Lambda]}(\cor_{M\times\R})$ such that
  $\kssfunc_\Lambda(F) \simeq \shf$ and $F|_{M \times \{t\}} \simeq 0$ for $t\ll0$.
  In the same way, for any $\shf \in \kss_{/[1]}(\cor_\Lambda)$ there exists
  $F \in \Der_{/[1],[\Lambda]}(\cor_{M\times\R})$ such that
  $\kssfunc_{/[1],\Lambda}(F) \simeq \shf$ and $F|_{M \times \{t\}} \simeq 0$ for
  $t\ll0$.
\end{corollary}
\begin{proof}
  (i) The proofs are the same for $\shf \in \kss(\cor_\Lambda)$ or
  $\shf \in \kss_{/[1]}(\cor_\Lambda)$. We first consider
  $\shf \in \kss(\cor_\Lambda)$ and then emphasize some points for the other case.

  By Theorem~\ref{thm:quant_legendrian_dbl} there exists
  $F_0 \in \Der_{\Lambda,\shv}^\dbl(\cor_{M\times\R})$, for some finite family $\shv$
  of open subsets of $M\times\R$ which is adapted to $\Lambda$, such that
  $\SSi^\dbl(F_0) = \Lambda$ and $\kssfunc_{\Lambda}^\dbl(F_0) \simeq \shf$.  By the
  definition of $\Der_{\Lambda,\shv}^\dbl(\cor_{M\times\R})$ any point
  $x\in M\times\R$ has a neighborhood $W_x$ such that a
  decomposition~\eqref{eq:defDdbl} holds for $F_0|_{W_{x,\gammaof}}$.  Since
  $\SSi^\dbl(F_0) = \Lambda$, we have in fact
  $F_0|_{W_{x,\gammaof}} \simeq \Psi_{W_x}(F_x)$ for some
  $F_x \in \Derb_{[\Lambda]}(\cor_{W_x})$ (in Lemma~\ref{lem:descrptionfaiscdbl}-(i)
  all $W_i$ coincide with $W_x$).  Setting
  $V = \bigcup_{x\in M\times\R} W_{x,\gammaof}$ we then have
  $\supp(F_0) \cap V \subset (\gammaf \star \dot\pi_{M\times \R}(\Lambda)) \cap V$
  and $\dot\SSi(F_0) \cap T^*V \subset \Lambda^+\cap T^*V$ by
  Lemmas~\ref{lem:suppPsiuL} and~\ref{lem:SSPsiF}.  In particular we can find a
  compact neighborhood $C$ of $\dot\pi_{M\times \R}(\Lambda)$ such that, for
  $u\leq 1$, the support of $F_0|_{V \cap (M\times\R\times\mo]0,u[)}$ is contained in
  $C \times \mo]0,u[$. For $u>0$ small enough we have $(C \times \mo]0,u[) \subset V$
  and we obtain $F_1 \in \Derb_{[\Lambda^+]}(\cor_{M\times\R\times\mo]0,u[})$ by
  extending $F_0|_{V \cap (M\times\R\times\mo]0,u[)}$ by zero.  For any
  $v\in \mo]0,u[$ we have $\dot\SSi(F_1|_{M\times\R\times\{v\}}) = \Lambda_v$ and it
  follows from~\eqref{eq:mLambda=mLambdadbl} that
  $\kssfunc_\Lambda(F_1|_{M\times\R\times\{v\}}) \simeq \shf$.
  
  \sui(ii) We set $F_2 = F_1|_{M\times\R\times\{\demi u\}}$.  Then
  $\dot\SSi(F_2) = \Lambda_{\demi u}$.  We consider the isotopy $\phi$ of
  Lemma~\ref{lem:isotopy_transl} and let
  $K_\phi \in \Derlb(\cor_{(M\times\R)^2\times \rspos})$ be the sheaf associated with
  $\phi$ by Theorem~\ref{thm:GKS}.  For given $0<u_1<u_2$ the restriction of $\phi$
  to $\dT^*(M\times\R\times\mo]u_1,u_2[)$ has compact support (after quotienting by
  the $\rspos$ action in the fibers) and
  $K_\phi |_{(M\times\R)^2\times \mo]u_1,u_2[}$ coincides with
  $\cor_{\Delta_{M\times\R} \times \mo]u_1,u_2[}$ outside a compact set, hence is
  bounded.  In particular, for any $v>0$, the sheaf
  $F_{2,v} = K_{\phi,v} \circ K_{\phi, \demi u}^{-1} \circ F_2$ belongs to
  $\Derb(\cor_{M\times\R})$ and satisfies
  $\dot\SSi(F_{2,v}) = (\phi_v \circ \phi_{\demi u}^{-1})(\Lambda_{\demi u}) =
  \Lambda_v$.  Moreover, for given $0<u_1<u_2$, $\phi_u$ is the identity map on some
  neighborhood $\Omega$ of $\Lambda$, for any $u \in \mo]u_1,u_2[$, and it follows
  that $K_{\phi,u} \circ -$ induces the identity functor on
  $\Der(\cor_{M\times\R}; \Omega)$. Hence
  $\kssfunc_\Lambda(F_{2,v}) \simeq \kssfunc_\Lambda(F_2) \simeq \shf$.

  \sui(iii) Since $\Lambda/\rspos$ is compact we can choose $A>0$ such that
  $\Lambda \subset T^*_{\tau >0}(M\times \mo]-A,A[)$.  For $v\geq 2A+1$ we have
  $T'_v(\Lambda) \subset T^*_{\tau >0}(M\times \mo]A+1,+\infty[)$.  We set
  $F_3 = F_{2,2A+1}|_{M\times \mo]-\infty, A+1[}$.  Then $\dot\SSi(F_3) = \Lambda$.
  We choose a diffeomorphism $f \colon \mo]-\infty, A+1\mc[ \isoto \R$ such that $f$
  is the identity map on $\mo]-\infty, A[$.  Then $F = \roim{(\id_M\times f)}(F_3)$
  satisfies $\dot\SSi(F) = \Lambda$ and $\kssfunc_\Lambda(F) \simeq \shf$.

  We set $U_A = M \times \mo]-\infty,-A\mc[$.    We remark that $\dot\SSi(F_{2,v}|_{U_A})$ is
  empty for all $v>0$, hence $F_{2,v}$ is locally constant on $U_A$.  If
  $F_{2,v}|_{U_A}$ does not vanish, then $\supp(F_{2,v})$ must contain $U_A$.  We have
  seen that $\supp(F_1)$ is contained in $C \times \mo]0,u[$ for some compact set
  $C$.  Hence $F_2$ and all $F_{2,v}$ have compact support.  It follows that
  $F_{2,v}|_{U_A}$ vanishes and so does $F|_{U_A}$.  This concludes the proof for the
  case $\shf \in \kss(\cor_\Lambda)$.

  \sui(iv) Now we assume $\shf \in \kss_{/[1]}(\cor_\Lambda)$.  We use
  Theorem~\ref{thm:quant_legendrian_orb_dbl} instead
  of~\ref{thm:quant_legendrian_dbl} in~(i).  We didn't state Theorem~\ref{thm:GKS}
  for the orbit category but there is nothing really new to say: as soon as the sheaf
  $K_{\phi}$ belongs to $\Derb(\cor_{(M\times\R)^2\times \mo]u_1,u_2[})$ (rather than
  the locally bounded category) we can take its image in
  $\Der_{/[1]}(\cor_{(M\times\R)^2\times \mo]u_1,u_2[})$. The relation
  $K_{\phi,u} \circ K_{\phi,u}^{-1} \simeq \cor_{\Delta_{M\times\R}}$ still holds in
  the orbit category and we deduce the same equivalence
  $K_{\phi,u} \circ - \colon \Der_{/[1],[A]}(\cor_{M\times\R}) \isoto
  \Der_{/[1],[\phi_u(A)]}(\cor_{M\times\R})$, for any $A \subset \dT^*(M\times\R)$.
  Now the step~(ii) works the same way.  The step~(iii) also, using
  Proposition~\ref{prop:SSorb-sectnulle} to see that $F_{2,v}$ is locally constant on
  $U_A$.
\end{proof}

\begin{remark}\label{rem:Lambdacompacte}   In the proof of Theorem~\ref{thm:quant_legendrian_dbl} the compactness
  of $\Lambda/\rspos$ was used since we glued doubled sheaves defined on open subsets
  of a covering of $\dot\pi_{M\times\R}(\Lambda)$ by a finite induction.  We could
  probably use a countable covering (with some care to take the limit) and the
  compactness of $\Lambda/\rspos$ was not essential.  However in
  Corollary~\ref{cor:exist_quant} this compactness is necessary to ensure that the
  sheaf given by Theorem~\ref{thm:quant_legendrian_dbl} can be modified to have the
  required microsupport over $M\times\R\times\mo]0,u[$ for some $u>0$.
\end{remark}

\begin{remark}\label{rem:exist_quantorb2}
  Using Remark~\ref{rem:exist_quantorb} and Corollary~\ref{cor:exist_quant} we see
  that, if $\Lambda$ comes from a compact exact submanifold of $T^*M$, then there
  exists a simple object $F \in \Der_{/[1],[\Lambda]}(\cor_{M\times\R})$ such that
  $F|_{M \times \{t\}} \simeq 0$ for $t\ll0$.
\end{remark}

\section{Restriction at infinity}
\label{sec:restr_infty}

As in~\S\ref{sec:transl_micsup} we assume that $\Lambda/\rspos$ has no Reeb
chord.  Since $\Lambda/\rspos$ is compact, we can choose $A>0$ such that
$\Lambda \subset T^*_{\tau >0}(M\times \mo]-A,A[)$.  Then, for any
$F \in \Derlb_{[\Lambda]}(\cor_{M\times\R})$, the restrictions
$F|_{M\times \mo]-\infty,-A[}$ and $F|_{M\times \mo]A,+\infty[}$ have locally
constant cohomology sheaves.

\begin{definition}\label{def:DerlbLambdaplus}
  For $F \in \Derlb_{[\Lambda]}(\cor_{M\times\R})$ we let
  $F_-, F_+ \in \Derlb(\cor_{M})$ be the restrictions at infinity
  $F_- = F|_{M\times \{-t\}}$, $F_+ = F|_{M\times \{t\}}$, for any
  $t\in [A,+\infty[$.  Then $F_-, F_+$ are indeed independent of
  $t \in [A,+\infty[$ and have locally constant cohomology sheaves.  We let
  $\Derlb_{[\Lambda],+}(\cor_{M\times\R})$ be the full subcategory of
  $\Derlb_{[\Lambda]}(\cor_{M\times\R})$ consisting of the $F$ such that
  $F_- \simeq 0$.

  Replacing  
  $\Derlb_{[\Lambda]}(-)$ by $\Der_{/[1],[\Lambda]}(-)$ we also define
  $F_-, F_+ \in \Der_{/[1]}(\cor_{M})$ for
  $F \in \Der_{/[1],[\Lambda]}(\cor_{M\times\R})$ and a similar category
  $\Der_{/[1],[\Lambda],+}(\cor_{M\times\R})$.
\end{definition}

For $F \in \Derlb_{[\Lambda],+}(\cor_{M\times\R})$ we have by definition
\begin{equation}\label{eq:F_restr-infini}
F|_{M\times ]A,+\infty[} \simeq F_+\etens \cor_{]A,+\infty[} ,
\quad  F|_{M\times ]-\infty,-A[} \simeq 0 \quad \text{ for $A\gg0$.}
\end{equation}

\begin{lemma}\label{lem:annulation_reim-pM}
  Let $F\in \Derlb(\cor_{M\times\R})$ (or $F\in\Der_{/[1]}(\cor_{M\times\R})$). We
  assume that there exists $A>0$ such that $\supp(F) \subset M\times [-A,A]$. We also
  assume either $\dot\SSi(F) \subset T^*_{\tau> 0}(M\times\R)$ or
  $\dot\SSi(F) \subset T^*_{\tau<0}(M\times\R)$.  Let $p_M \cl M\times\R \to M$ be
  the projection.  Then $\reim{p_M}F \simeq \roim{p_M}F \simeq 0$.
\end{lemma}
\begin{proof}
  By Proposition~\ref{prop:oim} (or~\ref{prop:im-dir-SSorb}), we have
  $\SSi(\reim{p_M}F) \subset T^*_MM$.  Hence $\reim{p_M}F$ is locally constant (for
  $F\in\Der_{/[1]}(\cor_{M\times\R})$, use Proposition~\ref{prop:SSorb-sectnulle})
  and it is enough to prove that $(\reim{p_M}F)_x \simeq 0$ for one $x\in M$.

  The base change formula gives
  $(\reim{p_M}F)_x \simeq \rsect(\R; F|_{\{x\}\times\R})$.  Now $F|_{\{x\}\times\R}$
  has a compact support and a microsupport in $T^*_{\tau\geq 0}\R$ (or
  $T^*_{\tau\leq 0}\R$) and the result follows from Corollary~\ref{cor:Morse} (or
  Corollary~\ref{cor:Morse-orb} for the case of $\Der_{/[1]}(\cor_{M\times\R})$: take
  $a,b$ in the corollary so that $]a,b[$ contains the support of $F$).
\end{proof}

\begin{lemma}\label{lem:projFsurM}
  Let $F \in \Derlb_{[\Lambda],+}(\cor_{M\times\R})$ (or
  $F \in \Der_{/[1],[\Lambda],+}(\cor_{M\times\R})$).  We have
  $\roim{p_M}F \simeq F_+$ and $\reim{p_M}F \simeq 0$.
\end{lemma}
\begin{proof}
  Let us choose $A>0$ so that~\eqref{eq:F_restr-infini} holds.  We set
  $G = F\tens \cor_{M\times ]-\infty, A+1[}$.  Then $\supp(G) \subset [-A, A+1]$ and
  $\SSi(G) \subset T^*_{\tau> 0}(M\times\R)$ by Theorem~\ref{thm:SSrhom}.  By
  Lemma~\ref{lem:annulation_reim-pM} we obtain $\roim{p_M}G \simeq 0$.
  
  By~\eqref{eq:F_restr-infini} we have the distinguished triangle
  $G \to F \to F_+ \etens \cor_{[A+1,+\infty[} \to[+1]$ and we deduce
  $\roim{p_M} F \simeq \roim{p_M} ( F_+ \etens \cor_{[A+1,+\infty[} ) \simeq F_+$ and
  $\reim{p_M} F \simeq \reim{p_M} ( F_+ \etens \cor_{[A+1,+\infty[} ) \simeq 0$.
\end{proof}

We recall the maps $q,r,T_u$ (see around~\eqref{eq:def-Lambdaplus}).  By
Lemma~\ref{lem:dtgrPsiq} we have a morphism $\opb{q}F \to \opb{r}F$, for any
$F\in \Dertpn(\cor_{M\times\R})$.  Restricting to $M\times\R \times\{u\}$ we obtain a
morphism $F \to \oim{T_u}F$.

\begin{lemma}\label{lem:Hom-isomorphes}
  For all $F,F'$ in $\Derlb_{[\Lambda],+}(\cor_{M\times\R})$ (or in
  $\Der_{/[1],[\Lambda],+}(\cor_{M\times\R})$) and any $u\geq 0$, the morphism
  $F' \to \oim{T_u}F'$ induces the isomorphism
  \begin{equation*}
\RHom(F,F')  \isoto \RHom(F, \oim{T_u}F')  .
\end{equation*}
Moreover, for any $u>0$, we have $\RHom(\oim{T_u}F, F') \simeq 0$.
\end{lemma}
\begin{proof}
  (i) We extend $q,r$ to $M\times\R \times\R$ (with the same formulas) and we define
  $p_2 \colon M\times\R \times\R \to \R$, $(x,t,u) \mapsto u$.  We introduce
  $G = \roim{p_2} \rhom(\opb{q}F, \epb{r}F')$.  For $u\in \R$, letting
  $i_u \colon M\times\R \times\{u\} \to M\times\R \times\R$ be the inclusion we have
  $q\circ i_u =\id$, $r\circ i_u = T_{-u}$, and hence, by
  Proposition~\ref{prop:formulaire}-(h-j),
  \begin{align*}
    \rsect_{\{u\}}(G)  &\simeq \rsect(M\times\R; \epb{i_u}\rhom(\opb{q}F, \epb{r}F')) \\
    &\simeq \rsect(M\times\R; \rhom(F, \epb{T_{-u}}F')) \\
    &\simeq \RHom(F, \oim{T_u}F') ,
  \end{align*}
  Using the microsupport bounds of
  \S\ref{sec:microsupport} or \S\ref{sec:micsup_orb_cat} we obtain
  \begin{align*}
    \SSi(G) \subset \{(u&;\upsilon);\; \exists (x,t;\xi,\tau) \in \SSi(F), \,
    \exists (x',t';\xi',\tau') \in \SSi(F'), \\
&    x=x',\, t-u=t',\, (-\xi,-\tau,0)+(\xi',\tau',-\tau') = (0,0,\upsilon)\}.
  \end{align*}
  Using $\dot\SSi(F) = \dot\SSi(F') = \Lambda$ and the fact that $\Lambda$ has no
  Reeb chord (hence $\Lambda \cap T'_u(\Lambda) = \emptyset$ for $u\not=0$), we
  deduce $\dot\SSi(G) \subset \{(0;\upsilon)$; $\upsilon<0\}$.

  \sui(ii) Using Corollary~\ref{cor:Morse} (or Corollary~\ref{cor:Morse-orb} for
  $\Der_{/[1]}(\cor_{M\times\R})$) we deduce $\rsect_{[a,b[}(\R;G) \simeq 0$ for any
  $a<b$ and $\rsect_{]a',b']}(\R;G) \simeq 0$ for any $0\leq a'<b'$.  In particular
  $\rsect_{\{u\}}(G) \isoto \rsect_{[u,u']}(\R;G) \isofrom \rsect_{\{u'\}}(G)$ for
  $0\leq u \leq u'$ and the first isomorphism follows.

  By the same argument $\RHom(\oim{T_u}F, F') \simeq \RHom(F, \oim{T_{-u}}F')$ is
  independent of $u>0$.  If we choose $u>2A$, where $A$
  satisfies~\eqref{eq:F_restr-infini} both for $F$ and $F'$, then $\oim{T_{-u}}F'$
  coincides with $\opb{p_M}F'_+$ over $\supp(F)$, where $p_M$ is the projection
  $M\times\R \to M$.  Hence
  $$
  \RHom(F, \oim{T_{-u}}F') \simeq \RHom(F, \opb{p_M}F'_+)
  \simeq \RHom(\reim{p_M} F, F'_+)[-1]
  $$
  vanishes by Lemma~\ref{lem:projFsurM}.
\end{proof}

\begin{theorem}\label{thm:restr-infini-ff}
Let $F,F' \in \Derlb_{[\Lambda],+}(\cor_{M\times\R})$. We let $F_+, F'_+ \in
\Derlb(\cor_{M})$ be their restrictions to $M\times\{t\}$, $t\gg0$, as in
Definition~\ref{def:DerlbLambdaplus}.  Then
\begin{equation}\label{eq:restr-infini-ff1}
\RHom(F,F')  \isoto \RHom(F_+,F'_+) .
\end{equation}
In particular the functor $\Derlb_{[\Lambda],+}(\cor_{M\times\R}) \to
\Derlb(\cor_{M})$ given by $F \mapsto F_+$ is fully faithful and we have:
$F\simeq F'$ if and only if $F_+ \simeq F'_+$.

The same statement holds with $\Derlb_{[\Lambda],+}(\cor_{M\times\R})$,
$\Derlb(\cor_{M})$ replaced by $\Der_{/[1],[\Lambda],+}(\cor_{M\times\R})$,
$\Der_{/[1]}(\cor_{M})$ (using $\RHom^\varepsilon$ defined
before~\ref{eq:Hom_RHomepsilon}).
\end{theorem}
\begin{proof}
  Let $p_M\cl M\times\R \to M$ be the projection.  Let us choose $A>0$ so
  that~\eqref{eq:F_restr-infini} holds for $F$ and $F'$ and let $u>2A$.  Then
\begin{align*}
\RHom(F,F') &\simeq  \RHom(F,\oim{T_u} F') \\
& \simeq \RHom(\opb{p_M}F_+,\oim{T_u} F')  \\
&\simeq \RHom(F_+, \roim{p_M} \oim{T_u} F') \\
&\simeq \RHom(F_+,F'_+) ,
\end{align*}
where the first isomorphism follows from Lemma~\ref{lem:Hom-isomorphes}, the second
one from $\supp(\oim{T_u} F') \subset M\times ]A,+\infty[$ and the last one from
Lemma~\ref{lem:projFsurM}.
\end{proof}

With Theorem~\ref{thm:restr-infini-ff} we can recover a classical result~\cite{LS91}
of Lalonde and Sikorav:
\begin{corollary}\label{cor:LambdasurjsurM}
  Let $\tilde \Lambda \subset T^*M$ be a compact exact Lagrangian submanifold. Then
  the map $\tilde \Lambda \to M$ is onto. In particular $M$ is compact.
\end{corollary}
\begin{proof}
  We let $\Lambda \subset \dT^*(M\times\R)$ be the conification of $\tilde \Lambda$
  as in~\eqref{eq:conif_Lambdatilde}.  By Remark~\ref{rem:exist_quantorb2} there
  exists a simple object $F \in \Der_{/[1],[\Lambda],+}(\cor_{M\times\R})$. In
  particular $F\not= 0$. By Theorem~\ref{thm:restr-infini-ff} (for the case of the
  orbit category) we have $\RHom(F_+,F_+) \not=0$, hence $F_+ \not=0$.  Let us assume
  that $\tilde \Lambda \to M$ is not onto. Then there exists an open subset $U$ of
  $M$ such that $\Lambda \cap \dT^*(U\times\R) = \emptyset$.  Hence
  $F|_{U \times \R}$ is locally constant, $F|_{U \times \{t\}} \simeq 0$ for $t\ll0$
  and $F|_{U \times \{t\}} \not\simeq 0$ for $t\gg0$, which is a contradiction.
\end{proof}

\begin{theorem}\label{thm:muhom=hom}
Let $F,F' \in \Derlb_{[\Lambda],+}(\cor_{M\times\R})$. Then we have an isomorphism
\begin{equation}\label{eq:muhom=hom}
\RHom(F,F')  \isoto \rsect(\Lambda; \mu hom(F,F'))  .
\end{equation}
Its composition with~\eqref{eq:restr-infini-ff1} gives a canonical isomorphism
\begin{equation}\label{eq:muhom=hom-infini1}
\RHom(F_+,F'_+) \simeq \rsect(\Lambda; \mu hom(F,F')). 
\end{equation}
As in Theorem~\ref{thm:restr-infini-ff} the same statement holds with
$\Derlb_{[\Lambda],+}(\cor_{M\times\R})$ replaced by
$\Der_{/[1],[\Lambda],+}(\cor_{M\times\R})$.
\end{theorem}
\begin{proof}
  (i) For an interval $I$ of $\rspos$ we set $N_I = M\times\R\times I$.
  Proposition~\ref{prop:muhom=hompsi} and the fact that $N_{]0,\varepsilon[}$ is open
  give the isomorphisms, for any $i\in \Z$,
\begin{align*}
  H^i&\rsect(\Lambda; \mu hom(F,F')) \\
                   &\simeq \varinjlim_{\varepsilon>0}
H^i\RHom(\Psi_{M\times \R}(F) |_{N_{]0,\varepsilon[}},
 \Psi_{M\times \R}(F')|_{N_{]0,\varepsilon[}} ))
                   \simeq \varinjlim_{\varepsilon>0} H^iA_\varepsilon ,
\end{align*}
where we set
$A_\varepsilon = \RHom((\Psi_{M\times \R}(F) )_{N_{]0,\varepsilon[}}, \Psi_{M\times
  \R}(F'))$.  It is enough to prove that $A_\varepsilon \simeq \RHom(F,F')$ for all
$\varepsilon>0$.  The triangle~\eqref{eq:dtgamqr2},
$\Psi_{M\times \R}(F') \to \opb{q}(F') \to \opb{r}(F') \to[+1]$, yields another
distinguished triangle $A_\varepsilon \to B_\varepsilon \to C_\varepsilon \to[+1]$
with
\begin{align*}
  B_\varepsilon &=  \RHom((\Psi_{M\times \R}(F) )_{N_{]0,\varepsilon[}}, \opb{q}(F')) \\
  & \simeq \RHom(\reim{q}((\Psi_{M\times \R}(F) )_{N_{]0,\varepsilon[}}), F')[-1], \\
  C_\varepsilon &=  \RHom((\Psi_{M\times \R}(F) )_{N_{]0,\varepsilon[}}, \opb{r}(F')) \\
  & \simeq \RHom(\reim{r}((\Psi_{M\times \R}(F) )_{N_{]0,\varepsilon[}}), F')[-1] .
\end{align*}
We check in~(ii) below that $B_\varepsilon \simeq \RHom(F,F')$ and in~(iii) that
$C_\varepsilon \simeq 0$, which proves the theorem.

\sui(ii) We compute $B_\varepsilon$.  The triangle~\eqref{eq:dtgamqr2} again gives
the distinguished triangle
\begin{equation}\label{eq:muhom=hom2}
  \begin{aligned}
\reim{q}((\Psi_{M\times \R}(F) )_{N_{]0,\varepsilon[}}) \to
\reim{q}&((\opb{q}F)_{N_{]0,\varepsilon[}})  \\
&\to \reim{q}((\opb{r}F)_{N_{]0,\varepsilon[}}) \to[+1] .
\end{aligned}
\end{equation}
In the following computations we consider $q$, $r$ as maps defined on $M\times\R^2$
(with the same formulas -- see after~\eqref{eq:def-Lambdaplus}). The projection
formula (Proposition~\ref{prop:formulaire}-(g)) gives
$\reim{q}((\opb{q}F)_{N_{]0,\varepsilon[}}) \simeq F[-1]$.  We have
$\dot\SSi(\opb{r}F) \subset \{ -\upsilon = \tau > 0\}$ and
$\SSi(\cor_{N_{]0,\varepsilon]}}) \subset \{\upsilon\leq 0\}$.  Hence
$\dot\SSi((\opb{r}F)_{N_{]0,\varepsilon]}}) \subset \{\upsilon< 0\}$ and
Lemma~\ref{lem:annulation_reim-pM} (used with $q\colon M\times\R^2 \to M\times\R$
instead of $p_M \colon M\times\R \to M$) gives
$\reim{q}((\opb{r}F)_{N_{]0,\varepsilon]}}) \simeq 0$.  Using the triangle
$\cor_{N_{\{\varepsilon\}}}[-1] \to \cor_{N_{]0,\varepsilon[}} \to
\cor_{N_{]0,\varepsilon]}} \to[+1]$ we deduce
$$
\reim{q}((\opb{r}F)_{N_{]0,\varepsilon[}})
\simeq \reim{q}((\opb{r}F)_{N_{\{\varepsilon\}}})[-1]
\simeq \oim{T_\varepsilon} F[-1].
$$
Lemma~\ref{lem:Hom-isomorphes} gives $\RHom(\oim{T_\varepsilon} F,F') \simeq 0$ for
any $\varepsilon>0$.  Applying $\RHom(-,F')$ to the triangle~\eqref{eq:muhom=hom2} we
deduce $B_\varepsilon \simeq \RHom(F,F')$, as claimed.

\sui(iii) Now we prove
$\reim{r}((\Psi_{M\times \R}(F) )_{N_{]0,\varepsilon[}}) \simeq 0$, which implies
$C_\varepsilon \simeq 0$. As in~(ii) we have the triangle
\begin{equation}\label{eq:muhom=hom3}
  \begin{aligned}
\reim{r}((\Psi_{M\times \R}(F) )_{N_{]0,\varepsilon[}}) \to
\reim{r}&((\opb{q}F)_{N_{]0,\varepsilon[}})  \\
&\to[\alpha] \reim{r}((\opb{r}F)_{N_{]0,\varepsilon[}}) \to[+1] 
\end{aligned}
\end{equation}
and an isomorphism $\reim{r}((\opb{r}F)_{N_{]0,\varepsilon[}}) \simeq F[-1]$.  The
microsupport bound $\dot\SSi((\opb{q}F)_{N_{[0,\varepsilon[}}) \subset \{\tau>0$,
$\upsilon\geq 0\}$ and Lemma~\ref{lem:annulation_reim-pM} again give
$\reim{r}((\opb{q}F)_{N_{[0,\varepsilon[}}) \simeq 0$.  We deduce
$$
\reim{r}((\opb{q}F)_{N_{]0,\varepsilon[}})
\simeq \reim{q}((\opb{r}F)_{N_{\{0\}}})[-1] \simeq F[-1] .
$$
Moreover the morphism $\alpha$ in~\eqref{eq:muhom=hom3} corresponds to $\id_F$
through these isomorphisms and we deduce
$\reim{r}((\Psi_{M\times \R}(F) )_{N_{]0,\varepsilon[}}) \simeq 0$.
\end{proof}

\begin{remark}\label{rem:comp_muhom=hom-infini}
We have recalled in~\eqref{eq:comp_muhom} that $\mu hom$ admits a composition
morphism (denoted by $\mucirc$ in Notation~\ref{not:mucomposition}) compatible
with the composition morphism for $\rhom$.  In particular the
isomorphism~\eqref{eq:muhom=hom} is compatible with the composition
morphisms $\circ$ and $\mucirc$.  Since~\eqref{eq:restr-infini-ff1} is
clearly compatible with $\circ$, we deduce that~\eqref{eq:muhom=hom-infini1}
also is compatible with $\circ$ and $\mucirc$.
\end{remark}

\part{Exact Lagrangian submanifolds in cotangent bundles}
\label{chap:ex_Lag}
In this part $M$ is a connected manifold and $\Lambda_0 \subset T^*M$ is a
compact exact Lagrangian submanifold.  We use the sheaves constructed in
Corollary~\ref{cor:exist_quant} and Theorems~\ref{thm:restr-infini-ff},
\ref{thm:muhom=hom} to recover some results on the topology of $\Lambda_0$,
namely that the projection $\Lambda_0 \to M$ is a homotopy equivalence and that
the first and second Maslov classes of $\Lambda_0$ vanish.

The fact that $\Lambda$ and $M$ have the same homology was proved by Fukaya,
Seidel, Smith in~\cite{FSS08} and also by Nadler in~\cite{N09} using the Fukaya
category of the cotangent bundle and assuming that the first Maslov class
vanishes.  The fact that the projection $\Lambda_0 \to M$ is a homotopy
equivalence was proved by Abouzaid in~\cite{A12}, also assuming the vanishing of
the Maslov class.  Then Kragh in~\cite{Kr13} proved that this vanishing holds.

Abouzaid and Kragh gave more precise results on the topology of
$\Lambda$. In~\cite{AK18} they proved that the map $\Lambda_0 \to M$ is a simple
homotopy equivalence and in~\cite{AK16} they proved a vanishing result for the
higher Maslov classes (very roughly the images of the Maslov classes by the map
$\mathrm{BO} \to \mathrm{BH}$ vanish, where $H$ is the group of homotopy
equivalences of the sphere -- this gives the vanishing of obstructions for the
existence of a sheaf in spectra -- see~\cite{JT17, Jin19}).

\smallskip

In the following we choose $f\cl \Lambda_0 \to \R$ such that
$df = \alpha_M|_{\Lambda_0}$ and define as in~\eqref{eq:conif_Lambdatilde}
\begin{equation*}
\Lambda = \{(x,t;\xi,\tau);\; \tau>0, \; (x;\xi/\tau) \in \Lambda_0,
\; t = -f(x;\xi/\tau) \} .  
\end{equation*}
We have $\Lambda/\rspos = \Lambda_0$ and we work with $\Lambda$ instead of
$\Lambda_0$.

\section{Fundamental groups}

We   let
$\pi_\Lambda \cl \Lambda \to M$ be the projection to the base and we denote by
$\pi_1(\pi_\Lambda) \cl \pi_1(\Lambda) \to \pi_1(M)$ the induced morphism of
fundamental groups.

\begin{proposition}\label{prop:morph-Poincare-inj}
The morphism $\pi_1(\pi_\Lambda) \cl \pi_1(\Lambda) \to \pi_1(M)$ is injective.
\end{proposition}
\begin{proof}
(i) We set $\cor=\Z/2\Z$ and $G= \pi_1(\Lambda)$.  We let $\rho \cl G \to
GL(\cor[G])$ be the regular representation of $G$. This means that $\cor[G]$ is
the vector space with basis $\{e_g\}_{g \in G}$ and the action of $G$ is given
by $g\cdot e_h = e_{gh}$, for all $g,h \in G$.  We let $\shl_\rho$ be the local
system on $\Lambda$ with stalks $\cor[G]$ corresponding to this representation
$\rho$.

\sui(ii) We recall some results of~\S\ref{sec:KSstack_orbit}.  The $\hom$ sheaf
in $\kss_{/[1]}(\cor_\Lambda)$ is induced by $\ol{\muhom^\varepsilon}$ through the
equivalence $h_\Lambda \colon \Oloc(\cor_\Lambda) \simeq \loc(\cor_\Lambda)$ of
Lemma~\ref{lem:equiv_loc-Oloc}.  By Proposition~\ref{prop:KSstackorb} there
exists a simple object $\shf_0 \in \kss_{/[1]}(\cor_\Lambda)$ and
$\hom(\shf_0,\cdot)$ gives an equivalence
$\kss_{/[1]}(\cor_\Lambda) \isoto \loc(\cor_\Lambda)$.

We let $\shf_\rho \in \kss_{/[1]}(\cor_\Lambda)$ be the object associated with
  $\shl_\rho$ by this equivalence.   By Corollary~\ref{cor:exist_quant}, there exist
$F_0, F_\rho \in \Der_{/[1],[\Lambda],+}(\cor_{M\times \R})$ such that
$\kssfunc_{/[1],\Lambda}(F_0) \simeq \shf_0$ and
$\kssfunc_{/[1],\Lambda}(F_\rho) \simeq \shf_\rho$.  Moreover  
$$
\shl_\rho \simeq \hom(\shf_0, \shf_\rho)
\simeq h_\Lambda(\muhom^\varepsilon(F_0,F_\rho)|_\Lambda) .
$$
We define $L_0, L_1 \in \Orb(\cor_M)$ by $L_0 = F_0|_{M\times \{t\}}$ and
$L_1 = F_\rho|_{M\times \{t\}}$ for $t\gg 0$.  We let $p\cl M\times\R \to M$ be the
projection and we set $F=F_0 \epstens_{\cor_{M\times \R}} \opb{p}L_1$ and
$F' = F_\rho \epstens_{\cor_{M\times \R}} \opb{p}L_0$.  Taking the tensor product
with a locally constant sheaf (like $L_0$, $L_1$) does not increase the microsupport
and we still have $F, F' \in \Der_{/[1],[\Lambda],+}(\cor_{M\times \R})$.  Then
$F|_{M\times \{t\}} \simeq L_0 \epstens_{\cor_M} L_1 \simeq F'|_{M\times \{t\}}$ for
$t\gg 0$ and Theorem~\ref{thm:restr-infini-ff} (for the orbit category) implies
\begin{equation}\label{eq:morph-Poincare-inj1}
  F_0 \epstens_{\cor_{M\times \R}} \opb{p}L_1
  \simeq F_\rho \epstens_{\cor_{M\times \R}} \opb{p}L_0 .
\end{equation}

\sui(iii)   Applying $\kssfunc_{/[1],\Lambda}$
to~\eqref{eq:morph-Poincare-inj1} we find
$\shf_0 \epstens_{\cor_\Lambda} \opb{\pi_\Lambda} L_1 \simeq \shf_\rho
\epstens_{\cor_\Lambda} \opb{\pi_\Lambda} L_0$ in $\kss_{/[1]}(\cor_\Lambda)$.  Using
the equivalences $\hom(\shf_0,\cdot)$ and $h_\Lambda$ recalled in~(ii) we obtain
$\opb{\pi_\Lambda} L'_1 \simeq \shl_\rho \tens \opb{\pi_\Lambda} L'_0$, where
$L'_i = h_M(L_i)$, for $i=0, 1$ (see Lemma~\ref{lem:equiv_loc-Oloc} -- $L'_i$ is the
locally constant sheaf on $M$ associated with the presheaf
$U \mapsto \Hom_{\Orb(\cor_U)}(\cor_U,L_i)$).

\sui(iv) We let $\rho'_0$ and $\rho'_1$ be the representations of $\pi_1(M)$
corresponding to the local systems $L'_0$ and $L'_1$.  They induce
representations of $G= \pi_1(\Lambda)$, say $\rho''_0$ and $\rho''_1$, through
the morphism $\pi_1(\pi_\Lambda)$. Then the result of~(iii) gives the
isomorphism of representations of $G$, $\rho''_1 \simeq \rho \tens \rho''_0$.
We restrict these representations to the subgroup $K =
\ker(\pi_1(\pi_\Lambda))$ of $G$. Then $\rho''_0|_K$ and $\rho''_1|_K$ are
trivial representations and we deduce that $\rho|_K$ also is trivial.  Since
$\rho$ is a faithful representation of $G$, this gives $K= \{1\}$, as required.
\end{proof}

We will see later (see Theorem~\ref{thm:quant_canon} and
Proposition~\ref{prop:equiv_Poinc}) that the sheaf $F_0$ introduced in part~(ii) of
the proof of Proposition~\ref{prop:morph-Poincare-inj} satisfies
$F_0|_{M\times \{t\}} \simeq \cor_{M}$ for $t\gg0$ (in general it is of rank one but
here $\cor = \Z/2\Z$ so it is constant).  If we already knew this,
\eqref{eq:morph-Poincare-inj1} would give
$F_0 \epstens_{\cor_{M\times \R}} \opb{p}L_1 \simeq F_\rho$ and we would have
directly $\opb{\pi_\Lambda} L'_1 \simeq \shl_\rho$, simplifying the end of the proof.

Let $r\cl M' \to M$ be a covering. The derivative of $r$ induces a covering
$r' \cl T^*M' \to T^*M$.  We let $\Lambda'_0$ be a connected component of
  $\opb{r'}(\Lambda)$.  Then
$\Lambda'_0 \to \Lambda$ is a covering and $\pi_1(\Lambda'_0)$ is a subgroup of
$\pi_1(\Lambda)$.  We have the commutative diagram
\begin{equation}\label{eq:diag_Poinc-groups}
\begin{tikzcd}
\pi_1(\Lambda'_0) \ar[r, hook]  \ar[d]
  & \pi_1(\Lambda)  \ar[d, hook, "{\pi_1(\pi_\Lambda)}"]  \\
\pi_1(M') \ar[r] & \pi_1(M) ,  
\end{tikzcd}
\end{equation}
where $\pi_1(\pi_\Lambda)$ is injective by
Proposition~\ref{prop:morph-Poincare-inj}.  This implies that the morphism
$\pi_1(\Lambda'_0) \to \pi_1(M')$ is injective.  In particular, if $M'$ is the
universal cover of $M$, then $\pi_1(\Lambda'_0)$ vanishes, that is, $\Lambda'_0$
is the universal cover of $\Lambda$.

We let $m_\Lambda \cl \pi_1(\Lambda) \to \Z$ be the group morphism induced by
the Maslov class $\mu^{sh}_1(\Lambda) \in H^1(\Lambda;\Z_\Lambda)$ introduced
in~\S\ref{sec:obstr_classes}. We remark that $m_\Lambda$ determines
$\mu^{sh}_1(\Lambda)$.

\begin{corollary}\label{cor:revet-Maslov}
  There exists a covering map $r\colon M' \to M$ and a closed conic connected
  Lagrangian submanifold $\Lambda' \subset \dT^*(M' \times \R)$ such that the
  derivative of $r$ and the projection $\Lambda' \to M'$ induce isomorphisms
  $\Lambda' \isoto \Lambda$ and  
  $\pi_1(\Lambda) \isofrom \pi_1(\Lambda') \isoto \pi_1(M')$:
\begin{equation}\label{eq:diag_revet-Maslov0}
\begin{tikzcd}
\Lambda'  \ar[r, "\sim"] \ar[d]  & \Lambda  \ar[d]  \\
 M' \ar[r,"r"]  & M . 
\end{tikzcd}
\end{equation}
Moreover the isomorphism $\Lambda' \isoto \Lambda$ identifies $\mu^{sh}_1(\Lambda')$
and $\mu^{sh}_1(\Lambda)$.
\end{corollary}
By Corollary~\ref{cor:LambdasurjsurM} $M$ and its covering $M'$ are compact.
\begin{proof}
  As noticed around the diagram~\eqref{eq:diag_Poinc-groups}, if
  $\tilde r \colon \tilde M \to M$ is the universal cover of $M$ and $\Lambda'_0$ a
  connected component of the pull-back of $\Lambda$ by $\tilde r$, then $\Lambda'_0$
  is the universal cover of $\Lambda$.  We can see that the action of
  $\pi_1(\Lambda)$ on $\tilde M$ (via the inclusion
  $\pi_1(\Lambda) \subset \pi_1(M)$) induces an action on $T^*\tilde M$ which
  preserves $\Lambda'_0$. The result follows by setting
  $M' = \tilde M / \pi_1(\Lambda)$ and $\Lambda' = \Lambda'_0 / \pi_1(\Lambda)$.
\end{proof}

\section{Vanishing of the Maslov class}

By   Corollary~\ref{cor:mush_et_mu} the vanishing of $\mu^{sh}_1(\Lambda)$
implies the vanishing of the usual Maslov class $\mu_1(\Lambda)$.  We recall that
$m_\Lambda \cl \pi_1(\Lambda) \to \Z$ is the group morphism induced by
$\mu^{sh}_1(\Lambda)$ and that $m_\Lambda$ determines $\mu^{sh}_1(\Lambda)$.
  In this section we prove
the vanishing of $\mu^{sh}_1(\Lambda)$ as follows.  Assuming
$\mu^{sh}_1(\Lambda)\not=0$ we consider the cyclic cover $M'$ of $M$ associated with
$\mu^{sh}_1(\Lambda)$. We let $\psi_n$ be the action of $\Z$ on $M'$ and we let
$\Lambda'$ be the pull-back of $\Lambda$. We construct
$G \in \Der_{[\Lambda']}(\cor_{M'\times\R})$ which is quasi-periodic in the sense
$\opb{\psi_n}(G) \simeq G[-n]$, for each $n\in \Z$.  We then check that
$G|_{M' \times \{t\}}$ is bounded for $t\gg0$, obtaining a contradiction.  Since
$\Lambda'$ is not compact, we cannot apply Corollary~\ref{cor:exist_quant}
immediately and we will use the $\Z$-action to construct $G$.

We first give the statement at the level of stacks.  In this section we take
$\cor=\Z/2\Z$.  By Corollary~\ref{cor:revet-Maslov}, up to replacing $M$ by some
covering (and without changing $\Lambda$), we can assume that
$\pi_1(\Lambda) \isoto \pi_1(M)$.  Hence we have
$\pi_\Lambda^* \colon H^1(M; \Z_M) \isoto H^1(\Lambda; \Z_\Lambda)$. By
Lemma~\ref{lem:decomp_deuxouverts} there exists a map $f\colon M \to S^1$ such that
$\mu^{sh}_1(\Lambda) = \pi_\Lambda^*f^*(\delta)$, where $\delta \in H^1(S^1;\Z_{S^1})$ is
the canonical class.   We set $M' = M \times_{S^1} \R$ and
$\Lambda' = \Lambda \times_{S^1} \R$:
\begin{equation}\label{eq:diag-Maslovzero}
\begin{tikzcd}
\Lambda'  \ar[r] \ar[d, "r'"] & M'  \ar[r, "f'"] \ar[d, "r"] & \R  \ar[d]  \\
\Lambda \ar[r, "\pi_\Lambda"] & M \ar[r, "f"]  & S^1 
\end{tikzcd}
\end{equation}
and we have $\mu^{sh}_1(\Lambda') = r'^*(\mu^{sh}_1(\Lambda)) = 0$.  By construction
$M'$ comes with a $\Z$-action, denoted $\psi_n \colon M' \to M'$ for $n\in \Z$, which
translates by $n$ in the fibers of $r$.  We abusively denote by $\psi_n$ the
diffeomorphisms induced by $\psi_n$ on $M' \times \R$, $M' \times \R \times \rspos$
or $\Lambda'$.

\begin{lemma}\label{lem:Maslovzerostack}
  There exists $\shf \in \kss(\cor_{\Lambda'})$ which is simple and satisfies
  $\opb{\psi_n}(\shf) \simeq \shf[-n]$, for each $n\in \Z$.
\end{lemma}
If $\Lambda'$ is connected (that is, if $\mu^{sh}_1(\Lambda) \not=0$),
Proposition~\ref{prop:KSstack=Dloc} then implies that all objects
$\shf' \in \kss(\cor_{\Lambda'})$ satisfy $\opb{\psi_n}(\shf') \simeq \shf'[-n]$.
\begin{proof}
  In~\S\ref{sec:obstr_classes} we have defined $\mu^{sh}_1(\Lambda)$ by a \v Cech
  cocycle as follows.  First we choose a covering $\{\Lambda_i\}_{i\in I}$ of
  $\Lambda$ by suitable small open subsets. We choose
  $F_i \in \Derb_{(\Lambda_i)}(\cor_{M\times\R})$, for each $i\in I$, which is simple
  along $\Lambda_i$.  We have seen that there exist isomorphisms
  $\kssfunc_{\Lambda_i}(F_i)|_{\Lambda_{ij}} \isoto
  \kssfunc_{\Lambda_j}(F_j)|_{\Lambda_{ij}}[d_{ij}]$, for some integers $d_{ij}$, and
  that the cochain $\{d_{ij}\}_{i,j \in I}$ is a cocyle whose cohomology class,
  $\mu^{sh}_1(\Lambda)$, is independent of choices.

  The relation $\mu^{sh}_1(\Lambda)= \pi_\Lambda^* f^*(\delta)$ gives a
  representative of $\mu^{sh}_1(\Lambda)$, as follows.  We cover $S^1$ by three arcs,
  say $A_0$, $A_1$, $A_2$, and represent $\delta$ by the cocycle
  $\{\delta_{ij}\}_{i,j=0,1,2}$, $\delta_{01} = \delta_{12} = 0$, $\delta_{20} = 1$
  (and $\delta_{ij} = - \delta_{ij}$).  Taking the $\Lambda_i$'s small enough, we can
  assume that there exists $\sigma \colon I \to \{0,1,2\}$ such that
  $f(\pi_\Lambda(\Lambda_i)) \subset A_{\sigma(i)}$ for each $i\in I$.  Then
  $d'_{ij} := \delta_{\sigma(i), \sigma(j)}$ defines a cocycle representing
  $\mu^{sh}_1(\Lambda)$. Hence $\{d_{ij}\}_{i,j \in I}$ and $\{d'_{ij}\}_{i,j \in I}$
  differ by a coboundary and, shifting the $F_i$'s by this coboundary, we obtain
  \begin{equation}\label{eq:Maslovzerostack1}
  \kssfunc_{\Lambda_i}(F_i)|_{\Lambda_{ij}} \isoto
  \kssfunc_{\Lambda_j}(F_j)|_{\Lambda_{ij}}[d'_{ij}].
\end{equation}
We write the pull-back of $A_i$ to $\R$ as $\bigsqcup_{n\in \Z} A_i^n$ in such a way
that $A_i^{n+1} = 1+ A_i^n$ and $A_1^0$ meets $A_0^0$ and $A_2^0$.  We numerate the
pull-backs of the $\Lambda_i$'s to $T^*(M'\times\R)$ accordingly and obtain a
covering $\{\Lambda_i^n\}_{i\in I, n\in \Z}$ of $\Lambda'$. The pull-back of $F_i$
yields $F_i^n \in \Derb_{(\Lambda_i^n)}(\cor_{M'\times\R})$, $n\in \Z$ (actually
$F_i^n = F_i^m$ but we see $F_i^n$ and $F_i^m$ in different categories), and we set
$\shf_i^n = \kssfunc_{\Lambda_i^n}(F_i^n)[n]$.  Then the
relation~\eqref{eq:Maslovzerostack1} gives
$\shf_i^n |_{\Lambda_i^n \cap \Lambda_j^m} \simeq \shf_j^m |_{\Lambda_i^n \cap
  \Lambda_j^m}$.  Since $\cor = \Z/2\Z$ and the $\shf_i^n$'s are simple, the
compatibility conditions on the triple intersections are trivial.  Hence the
$\shf_i^n$'s glue together in an object $\shf$ which satisfies
$\opb{\psi_n}(\shf) \simeq \shf[-n]$, for each $n\in \Z$.
\end{proof}

\begin{theorem}\label{thm:Maslovzero}
  We have $\mu_1(\Lambda) = 0$.
\end{theorem}
\begin{proof}
  Let $\shf \in \kss(\cor_{\Lambda'})$ be given by Lemma~\ref{lem:Maslovzerostack}.
  We first build in~(i) a doubled sheaf $F$ on $M'\times\R$ which represents
  $\shf$. Then in~(ii) we deduce from $F$ a usual sheaf, as in the proof of
  Corollary~\ref{cor:exist_quant}, and in~(iii) we prove that $\mu^{sh}_1(\Lambda)$
  vanishes.

  \sui(i) As in \S\ref{sec:quant-dbl} we choose an adapted family $\shv = \{V_a$;
  $a\in A\}$ for $\Lambda$ which is stable by intersection.  We can assume
  $f(V_a) \not= S^1$, for each $a\in A$, and hence $(r\times \id_\R)^{-1}(V_a)$
  decomposes as a disjoint union $\bigsqcup_{n\in \Z} V'_{a,n}$ with
  $V'_{a,n} \isoto V_a$, for each $n$.  In this way the pull-back of $\shv$ to
  $M'\times\R$ gives an adapted family $\shv' = \{V'_b$; $b\in B\}$ for $\Lambda'$
  which is stable by intersection and by the $\Z$-action.

  We set $B_0 = \{b\in B$; $V'_b \cap f'^{-1}([0,1]) \not= \emptyset\}$ and
  $U_0 = \bigcup_{b\in B_0} V'_b$, $\Lambda'_0 = \Lambda' \cap T^*U_0$,
  $U_n = \psi_n^{-1}(U_0)$, $\Lambda'_n = \psi_n^{-1}(\Lambda'_0)$.  By
  Theorem~\ref{thm:quant_legendrian_dbl} there exists
  $F_0 \in \Der_{\Lambda', \shv'}^\dbl(\cor_{M'\times\R})$ such that
  $\SSi^\dbl(F_0) = \Lambda'_0$ and
  $\kssfunc_{\Lambda'_0}^\dbl(F_0) \simeq \shf|_{\Lambda'_0}$.  By $\Z$-invariance
  $F_1 := \opb{\psi_1}F_0$ belongs to
  $\Der_{\Lambda', \shv'}^\dbl(\cor_{M'\times\R})$, $\SSi^\dbl(F_1) = \Lambda'_1$,
  and
  $\kssfunc_{\Lambda'_1}^\dbl(F_1) \simeq (\opb{\psi_1}\shf)|_{\Lambda'_1} \simeq
  \shf[-1]|_{\Lambda'_1}$.

  We set $U_n^+ = U_n \times \rspos$, $U_{01}^+ = U_0^+ \cap U_1^+$ and
  $\Lambda'_{01} = \Lambda'_0 \cap \Lambda'_1$.  Since $\shv'$ is stable by
  intersection, $\rsect_{U_{01}^+}F_0$ and $\rsect_{U_{01}^+}F_1$ belong to
  $\Der_{\Lambda', \shv'}^\dbl(\cor_{M'\times\R})$ and we have
  $\SSi^\dbl(\rsect_{U_{01}^+}F_0) = \SSi^\dbl(\rsect_{U_{01}^+}F_0) = \Lambda'_{01}$
  and
  $\kssfunc_{\Lambda'_{01}}^\dbl(\rsect_{U_{01}^+}F_0) \simeq
  \kssfunc_{\Lambda'_{01}}^\dbl(\rsect_{U_{01}^+}F_1)[1] \simeq
  \shf|_{\Lambda'_{01}}$.  By Corollary~\ref{cor:kssfuncPsidbl} there exists an open
  subset $V$ of $M'\times\R\times\rspos$ such that
  $(M'\times\R \times \{0\}) \sqcup V$ is open in $M' \times\R \times \rpos$ and an
  isomorphism
  $u\colon (\rsect_{U_{01}^+} F_1[1])|_V \isoto (\rsect_{U_{01}^+} F_0)|_V$.

  Now we can glue $F_0$ and $F_1$ and define
  $F^{01}\in \Der(\cor_{M'\times\R\times\rspos})$ by the triangle
  $F^{01} \to \rsect_{V\cap U_0^+}F_0 \oplus \rsect_{V\cap U_1^+}F_1[1]) \to[u']
  \rsect_{V\cap U_{01}^+} F_0 \to[+1]$, where $u' = (\id,-u)$.  Then
  $F^{01}|_{(U_i\times\rspos) \cap V} \simeq F_i[i]|_{(U_i\times\rspos) \cap V}$ for
  $i=0,1$ and hence $F^{01}$ belongs to
  $\Der_{\Lambda', \shv'}^\dbl(\cor_{M'\times\R})$ and represents
  $\shf|_{\Lambda'_0 \cup \Lambda'_1}$.  More generally, assuming that the $V'_b$'s are
  small enough so that $U_0 \cap U_2 = \emptyset$, we can glue all translates of
  $F_0$ at once and we define $F$ by the triangle
  $$
  F \to \bigoplus_{n\in \Z} \opb{\psi_n}(\rsect_{V\cap U_0^+}F_0)[n] \to[v]
  \bigoplus_{n\in \Z} \opb{\psi_n}(\rsect_{V\cap U_{01}^+} F_0)[n] \to[+1],
  $$
  where the restriction of $v$ to the summand $\opb{\psi_n}(\rsect_{V\cap U_0^+}F_0)$
  is $\opb{\psi_n}(\id) \oplus \opb{\psi_{n-1}}(-u)$.  Then $F$ belongs to
  $\Der_{\Lambda', \shv'}^\dbl(\cor_{M'\times\R})$ (but is only locally boun\-ded in
  cohomological degrees) and represents $\shf$.  Moreover, we have by construction
  $\opb{\psi_n}F \simeq F[-n]$, for all $n\in\Z$.

  \sui(ii) As in~(i) of the proof of Corollary~\ref{cor:exist_quant} we can see that
  there exists $u>0$ such that, truncating $F$ near
  $(\gammaf \star \dot\pi_{M'\times \R}(\Lambda'))$, we obtain
  $G \in \Derlb_{[\Lambda'^+]}(\cor_{M'\times\R\times\mo]0,u[})$, representing
  $\shf$, where
  $\Lambda'^+ = q_d\opb{q_\pi}(\Lambda') \sqcup r_d\opb{r_\pi}(\Lambda')$ as
  in~\eqref{eq:def-Lambdaplus} (here $\Lambda'$ is not compact but we can first
  consider the restriction of $F$ to a neighborhood of $\ol{U_0}$ and then get the
  result by $\Z$-equivariance).  We still have $\opb{\psi_n}G \simeq G[-n]$, for all
  $n\in\Z$.

  As in the proof of Corollary~\ref{cor:exist_quant} we deduce
  $G' \in \Derlb_{[\Lambda']}(\cor_{M'\times\R})$, representing $\shf$, and such that
  $\opb{\psi_n}G' \simeq G'[-n]$, for all $n\in\Z$.  (To define an isotopy $\phi$
  which translates $\Lambda'$ as in Lemma~\ref{lem:isotopy_transl}, we first apply
  the lemma to $\Lambda$ and pull back the isotopy to $\dT^*(M'\times\R)$.)

  \sui(iii) We have $G'|_{M' \times \{t\}} \simeq 0$ for $t\ll0$ and we define
  $G'_+ = G'|_{M' \times \{t\}}$ for $t\gg 0$ as in
  Definition~\ref{def:DerlbLambdaplus}.  Theorem~\ref{thm:restr-infini-ff} gives
  $\RHom(G',G') \isoto \RHom(G'_+,G'_+)$.  Since $G'\not\simeq 0$ it follows that
  $G'_+ \not\simeq 0$.  We also know that $G'_+$ has locally constant cohomology
  sheaves.

  Let us assume that $\mu^{sh}_1(\Lambda) \not=0$. Then $M'$ is connected and the
  stalk $(G'_+)_x$ is independent of $x\in M'$.  By Lemma~\ref{lem:stalk_simple_sh}
  $\dim(H^*(G'_+)_x)$ is finite.  In particular $G'_+$ is bounded in degrees.  On the
  other hand $\opb{\psi_n}G'_+ \simeq G'_+[-n]$, for all $n\in \Z$, and we have a
  contradiction.
\end{proof}

\section{Restriction at infinity}

We recall the notation $\Derb_{[\Lambda],+}(\cor_{M\times\R})$ of
Definition~\ref{def:DerlbLambdaplus} and $F_+ = F|_{M \times \{t_0\}}$ for
$t_0\gg 0$, for $F \in \Derb_{[\Lambda],+}(\cor_{M\times\R})$.

\begin{proposition}\label{prop:restr-inf-corps}
We assume that $\cor = \Z$ or $\cor$ is a finite field.  Let $F \in
\Derb_{[\Lambda],+}(\cor_{M\times\R})$.  We assume that $F$ is simple along
$\Lambda$.  Then $F_+$ is concentrated in one degree, say $i$, and $H^iF_+$ is a
local system with stalks isomorphic to $\cor$.
\end{proposition}
\begin{proof}
  (i) We first assume that $\cor$ is a finite field.  Let us prove that $F_+$ is
  concentrated in one degree.  Let $a\leq b$ be respectively the minimal and maximal
  integers $i$ such that $H^iF_+ \not\simeq 0$. By Lemma~\ref{lem:stalk_simple_sh}
  the local systems $H^iF_+$ are of finite rank.  Since $\cor$ is finite we can find
  a finite cover $r\cl M' \to M$ such that $\opb{r}( H^iF_+)$ is a constant sheaf,
  for $i=a,b$.  We set $F' = \opb{(r\times \id_\R)} F$ and
  $\Lambda' = \opb{d(r\times \id_\R)}(\Lambda)$. Then
  $\opb{r}( H^iF_+) \simeq H^iF'_+$, $F'$ is simple along $\Lambda'$ and we have
  $\mu hom(F',F') \simeq \cor_{\Lambda'}$. Since $\Lambda'/\rspos$ is compact,
  Theorem~\ref{thm:muhom=hom} gives
\begin{equation}\label{eq:restr-inf-corps1}
\RHom(F'_+,F'_+) \simeq \rsect(\Lambda'; \cor_{\Lambda'}).
\end{equation}
On the other hand the complex $G=\rhom(F'_+,F'_+)$ is concentrated in degrees
greater than $a-b$ and $H^{a-b}G \simeq \hom(H^aF'_+,H^bF'_+)$ is a non zero
constant sheaf.  Hence $H^{a-b}\RHom(F'_+,F'_+)$ is non zero.
By~\eqref{eq:restr-inf-corps1} we deduce that $H^{a-b}\rsect(\Lambda';
\cor_{\Lambda'})$ also is non zero, which implies $a-b\geq 0$.  Hence $a=b$ and
$F_+$ is concentrated in a single degree.

\medskip\noindent
(ii) Now we prove that $H^aF_+$ is of rank one, that is, $H^aF'_+ \simeq
\cor_{M'}$. There exists $d\geq 1$ such that $H^aF'_+ \simeq \cor^d_{M'}$. The
isomorphism~\eqref{eq:restr-inf-corps1} gives in degree $0$:
\begin{equation}\label{eq:restr-inf-corps2}
\Hom(\cor^d,\cor^d) \simeq H^0(\Lambda'; \cor_{\Lambda'}) .
\end{equation}
By Remark~\ref{rem:comp_muhom=hom-infini} this isomorphism is compatible with
the algebra structures of both terms.  Let $I$ be the set of connected
components of $\Lambda'$. We obtain $|I| = d^2$.  The natural decomposition
$H^0(\Lambda'; \cor_{\Lambda'}) \simeq \bigoplus_{i\in I} H^0(\Lambda'_i;
\cor_{\Lambda'_i})$ gives an expression of the unit as a sum of orthogonal
idempotents, $1 = \sum_{i\in I} e_i$, where $e_i$ is the projection
$$
e_i \cl H^0(\Lambda'; \cor_{\Lambda'}) 
 \to H^0(\Lambda'_i; \cor_{\Lambda'_i}), \qquad i\in I.
$$
We let $m_i \in \Hom(\cor^d,\cor^d) = \Mat_{d\times d}(\cor)$ be the image of
$e_i$ by~\eqref{eq:restr-inf-corps2}.  The relation $1 = \sum_{i\in I} e_i$
gives a decomposition of the identity matrix $I_d = \sum_{i\in I} m_i$ as a sum
of $|I|$ non-zero orthogonal projections, that is, $m_i^2 = m_i$ and $m_im_j=0$,
for $i\not= j$. We deduce that $|I| \leq d$, that is, $d^2 \leq d$. Hence $d=1$,
as claimed.

\medskip\noindent (iii) Now we assume that $\cor = \Z$.  By
Lemma~\ref{lem:stalk_simple_sh} the stalks of $F_+$ are of finite rank over
$\Z$.  We recall that any object of $\Derb(\Z)$ is the sum of its
cohomology. Hence, for $z= (x,t)$, $t\gg 0$, we can write
$(F_+)_z = \bigoplus_{i=a}^b M_i[-i]$, where the $M_i$'s are abelian groups of
finite rank.

We first prove that the $M_i$'s are free. If this is not the case, there exist
$i\in \Z$ and a prime $p$ such that $M_i$ has $p$-torsion.  We set
$G = F \ltens_\Z \Z/p\Z$.  Then $G$ is simple along $\Lambda$ and
$(G_+)_z \simeq (F_+)_z \ltens_\Z \Z/p\Z$.  We have seen that $(G_+)_z$ is
isomorphic to $\Z/p\Z[j]$ for some shift $j$.  On the other hand  
$H^{-1}(M_i \ltens_\Z \Z/p\Z)$ and $H^0(M_i \ltens_\Z \Z/p\Z)$ are both non zero,
since $M_i$ has $p$-torsion.  This gives a contradiction and proves that, for
each $i$, we have $M_i \simeq \Z^{d_i}$ for some $d_i$.

We set $G = F \ltens_\Z \Z/2\Z$.  We have again
$(G_+)_z \simeq (F_+)_z \tens_\Z \Z/2\Z \simeq \Z/2\Z[j]$.  Hence $d_i=0$ for
all $i\not= j$ and $d_j=1$, as claimed.
\end{proof}

\begin{corollary}\label{cor:cohomology_Lambda}
We assume that $\cor = \Z$ or $\cor$ is a finite field and that
$\kss(\cor_\Lambda)$ has at least one global simple object. Then the projection
$\Lambda \to M$ induces an isomorphism $\rsect(M;\cor_M) \isoto
\rsect(\Lambda;\cor_\Lambda)$.
\end{corollary}
\begin{proof}
  We choose a simple object $\shf \in \kss(\cor_\Lambda)$. By
  Corollary~\ref{cor:exist_quant} there exists
  $F \in \Derb_{[\Lambda],+}(\cor_{M\times\R})$ such that
  $\kssfunc_{\Lambda}(F) \simeq \shf$.  By
  Proposition~\ref{prop:restr-inf-corps} $F_+$ is concentrated in one degree,
  say $i$, and $H^iF_+$ is a local system with stalks isomorphic to $\cor$.
  Hence $\rhom(F_+,F_+) \simeq \cor_M$ and
  $\RHom(F_+,F_+) \simeq \rsect(M;\cor_M)$.  Since $F$ is simple we also have
  $\mu hom(F,F)|_\Lambda \simeq \cor_\Lambda$.  By Theorem~\ref{thm:muhom=hom}
  we deduce an isomorphism
\begin{equation}\label{eq:cohomology_Lambda1}
\rsect(M;\cor_M) \simeq \rsect(\Lambda;\cor_\Lambda) .
\end{equation}
By construction~\eqref{eq:cohomology_Lambda1} is given by taking the global
sections in the bottom morphism of the commutative diagram:
$$
\begin{tikzcd}
\cor_{M\times\R} \ar[rr, "a"] \ar[d, "b"]
 &&  \roim{\dot\pi}(\cor_\Lambda) \ar[d, "c", "\wr"']  \\
\rhom(F,F) \ar[r, "\sim"] & \roim{\pi} \mu hom(F,F) \ar[r]
& \roim{\dot\pi} (\mu hom(F,F)|_\Lambda) , 
\end{tikzcd}
$$
where
$\pi = \pi_{M\times\R}$,
$b$ and $c$ map the sections $1$ to the identity morphisms.  When taking
global sections, $b$ and $c$ induce isomorphisms and~$a$ induces the natural
morphism $\rsect(M;\cor_M) \to \rsect(\Lambda;\cor_\Lambda)$ given by the
projection of $\Lambda$ to the base $M$. The bottom horizontal arrow
induces~\eqref{eq:cohomology_Lambda1}.  This shows
that~\eqref{eq:cohomology_Lambda1} is indeed induced by the projection to the
base.
\end{proof}

\begin{remark}\label{rem:cohom_Z2Z}
  By Theorem~\ref{thm:Maslovzero} the first Maslov class of $\Lambda$ vanishes.
  Hence, when $\cor = \Z/2\Z$ the stack $\kss(\cor_\Lambda)$ has a global simple
  object and Corollary~\ref{cor:cohomology_Lambda} applies: the projection
  $\Lambda \to M$ induces an isomorphism
$$
\rsect(M;\Z/2\Z_M) \isoto \rsect(\Lambda;\Z/2\Z_\Lambda).
$$
\end{remark}

\section{Vanishing of the second obstruction class}

We have seen the class $\mu^{sh}_2(\Lambda) \in H^2(\Lambda; \cor^\times)$
in~\S\ref{sec:obstr_classes}. By Corollary~\ref{cor:mush_et_mu}, if $\cor=\Z$,
the vanishing of $\mu^{sh}_2(\Lambda)$ implies the vanishing of the usual
obstruction class $\mu^{gf}_2(\Lambda)$.  Here we prove that
$\mu^{sh}_2(\Lambda) \in H^2(\Lambda;\Z/2\Z_\Lambda)$ vanishes.  For this we
will use Corollary~\ref{cor:exist_quant} in the framework of twisted sheaves.
Let $c\in H^2(M;\Z/2\Z)$ be given and let $\check c = \{c_{ijk}\}$,
$i,j,k \in $, be a \v Cech cocycle representing $c$ with respect to a finite
covering $\{U_i\}_{i\in I}$ of $M$. We view $\Z/2\Z$ as the multiplicative group
$\{\pm 1\}$ and $c_{ijk} = \pm 1$, for all $i,j,k$.
\begin{definition}\label{def:twisted-sheaves}
A $\check c$-twisted sheaf $F$ on $M$ is the data of sheaves $F_i \in
\Mod(\cor_{U_i})$ and isomorphisms $\varphi_{ij} \cl F_j|_{U_{ij}} \isoto
F_i|_{U_{ij}}$ satisfying the condition
$$
\varphi_{ij} \circ \varphi_{jk} = c_{ijk} \, \varphi_{ik} .
$$
The $\check c$-twisted sheaves form an abelian category that we denote by
$\Mod(\cor^{\check c}_M)$. We denote by $\Derb(\cor^{\check c}_M)$ its derived
category.
\end{definition}
The prestack $U \mapsto \Mod(\cor^{\check c|_U}_U)$ is a stack which is locally
equivalent to the stack of sheaves.  The usual operations on sheaves extend to
twisted sheaves. In particular if $\check c$, $\check d$ are \v Cech cocycles on
$M$ and $F \in \Derb(\cor^{\check c}_M)$, $F' \in \Derb(\cor^{\check d}_M)$, we
have a tensor product $F \ltens F' \in \Der(\cor^{\check c + \check d}_M)$ and a
homomorphism sheaf $\rhom(F,F')\in \Der(\cor^{\check d - \check c}_M)$. If $f
\cl M \to N$ is a morphism of manifolds and $\check d$ is a \v Cech cocycle on
$N$ with values in $\{\pm 1\}$, we have inverse images $\opb{f}, \epb{f} \cl
\Derb(\cor^{\check d}_N) \to \Derb(\cor^{f^* \check d}_M)$ and direct images
$\roim{f}, \reim{f} \cl \Derb(\cor^{f^* \check d}_M) \to \Derb(\cor^{\check
  d}_N)$ with the usual adjunction properties.  The notion of microsupport also
generalizes to the twisted case (since this is a local notion and twisted
sheaves are locally equivalent to sheaves) with the same behaviour with respect
to the sheaves operations.

We can define a Kashiwara-Schapira stack $\kss(\cor^{\check c}_{\Lambda})$ and
formulate a version of Corollary~\ref{cor:exist_quant} in this framework: for
$\shf \in \kss(\cor^{\check c}_{\Lambda})$ there exists
$F \in \Derb(\cor^{\check c}_{M\times\R})$ such that $\dot\SSi(F) = \Lambda$,
$F|_{M\times \{t\}} \simeq 0$ for $t\ll 0$ and
$\kssfunc^{\check c}_{\Lambda}(F) \simeq \shf$.

\begin{proposition}\label{prop:vanishingSW}
  The class $\mu^{gf}_2(\Lambda) \in H^2(\Lambda;\Z/2\Z_\Lambda)$ is zero.
\end{proposition}
\begin{proof}
  (i) By Corollary~\ref{cor:cohomology_Lambda} and Remark~\ref{rem:cohom_Z2Z} we
  have an isomorphism $H^2(M;\Z/2\Z_M) \isoto H^2(\Lambda;\Z/2\Z_\Lambda)$.  We
  let $c\in H^2(M;\Z/2\Z_M)$ be the inverse image of $\mu^{sh}_2(\Lambda)$ by
  this isomorphism and we choose a \v Cech cocycle $\check c$ representing $c$.
  Then the twisted Kashiwara-Schapira stack $\kss(\Z^{\check c}_{\Lambda})$ has
  a simple global object and the twisted version of
  Corollary~\ref{cor:exist_quant} gives
  $F \in \Derb_{[\Lambda],+}(\Z^{\check c}_{M\times\R})$ which is simple along
  $\Lambda$.  By Proposition~\ref{prop:restr-inf-corps} we have
  $F_+ \simeq L[d]$ where $L \in \Mod(\Z^{\check c}_M)$ is a twisted locally
  constant sheaf with stalks isomorphic to $\Z$ and $d$ is some integer.

  \sui(ii) Now we prove that the existence of $L \in \Mod(\Z^{\check c}_M)$ as
  in~(i) implies that $\check c$ is a boundary, that is,
  $\mu^{sh}_2(\Lambda) = 0$.  The cocycle $\check c$ is associated with a
  covering $\{U_i\}_{i\in I}$ of $M$.  The object $L \in \Mod(\Z^{\check c}_M)$
  is given by sheaves $L_i \in \Mod(\Z_{U_i})$ and isomorphisms
  $\varphi_{ij} \cl L_j|_{U_{ij}} \isoto L_i|_{U_{ij}}$, for any $i,j\in I$,
  such that $\varphi_{ij} \circ \varphi_{jk} = c_{ijk} \, \varphi_{ik}$ for all
  $i,j,k \in I$.  We can assume that $U_i$ is contractible and that $U_{ij}$ is
  connected for any $i,j\in I$.  Since $L$ is locally constant, we can choose an
  isomorphism $\varphi_i \cl L|_{U_i} \simeq \Z_{U_i}$ for each $i\in I$. Then
  the composition $b_{ij} = \varphi_i \varphi_{ij} \varphi_j^{-1}$ is an
  isomorphism $\Z\isoto \Z$, that is, $b_{ij} = \pm 1$.  We let $\check b$ be
  the $1$-cochain defined by $\{b_{ij}\}_{i,j\in I}$.  Then the equality
  $\varphi_{ij} \circ \varphi_{jk} = c_{ijk} \, \varphi_{ik}$ says that
  $\check c$ is the boundary of $\check b$, as required.
\end{proof}

\section{Homotopy equivalence}

Here we recover a result of~\cite{A12} that the projection
$\pi_\Lambda \cl \Lambda \to M$ is a homotopy equivalence, that is, it induces
an isomorphism of all fundamental groups.  It is well know that this is
equivalent to the following: $\pi_1(\Lambda) \isoto \pi_1(M)$ and, for any local
system $L$ on $M$, $\rsect(M;L) \isoto \rsect(\Lambda; \opb{\pi_\Lambda}L)$.

We recall the subcategory $\Derlb_{[\Lambda],+}(\cor_{M\times\R})$ of
$\Derlb_{[\Lambda]}(\cor_{M\times\R})$ introduced in
Definition~\ref{def:DerlbLambdaplus}.  It consists of the $F$ such that
$F_- \simeq 0$, where we have set $F_\pm = F|_{M\times \{\pm t\}}$ for $t\gg0$.

\begin{theorem}\label{thm:quant_canon}
Let $\cor$ be a ring with finite global dimension.
\begin{itemize}
\item [(i)] There exists $F \in\Derb_{[\Lambda],+}(\cor_{M\times\R})$ such that
  $F_+ \simeq \cor_M$; it is a simple sheaf.
\item [(ii)] The functor $G \mapsto G_+$,
  $\Derb_{[\Lambda],+}(\cor_{M\times\R}) \to \Derb(\cor_M)$, induces an equivalence
  between $\Derb_{[\Lambda],+}(\cor_{M\times\R})$ and the full subcategory of
  $\Derb(\cor_M)$ formed by the locally constant sheaves.  In particular the object
  $F$ in~{\rm(i)} is unique up to a unique isomorphism.
\item [(iii)] The projection $\pi_\Lambda \cl \Lambda \to M$ induces an
  isomorphism in cohomology
  $\rsect(M;\cor_M) \isoto \rsect(\Lambda;\cor_\Lambda)$.
\item [(iv)] More generally we have
  $\rsect(M;L) \isoto \rsect(\Lambda; \opb{\pi_\Lambda}L)$ for any local system
  $L$ on $M$.
\end{itemize}
\end{theorem}
\begin{proof}
  (i) We first assume that $\cor=\Z$.  By Theorem~\ref{thm:Maslovzero} and
  Proposition~\ref{prop:vanishingSW} we know that $\mu_1(\Lambda) = 0$ and
  $\mu^{sh}_2(\Lambda)= 0$.  By Corollary~\ref{cor:exist_quant} there exists
  $F^0 \in \Derb_{[\Lambda],+}(\cor_{M\times\R})$ which is simple along $\Lambda$.
  By Proposition~\ref{prop:restr-inf-corps} we have $F^0_+ \simeq L[d]$ where
  $L \in \Mod(\cor_M)$ is locally constant with stalks isomorphic to $\Z$ and
  $d$ is some integer. Let $p\cl M\times \R \to M$ be the projection. Then
  $F^1 = F^0 \tens \opb{p}L^{\otimes-1}[-d]$ satisfies the required properties.
  For a general ring $\cor$ we set
  $F = F^1 \ltens_{\Z_{M\times\R}} \cor_{M\times\R}$.

  \sui(ii) By Theorem~\ref{thm:restr-infini-ff} the functor $G \mapsto G_+$ is
  fully faithful.  Let $L \in \Derb(\cor_M)$ be a locally constant sheaf.  Then
  $F \ltens \opb{p}L$ belongs to $\Derb_{[\Lambda],+}(\cor_{M\times\R})$ and we
  have $(F \ltens \opb{p}L)_+ \simeq F_+ \ltens L \simeq L$, which proves that
  $G \mapsto G_+$ is essentially surjective.
  
  \sui (iii) follows from Corollary~\ref{cor:cohomology_Lambda}.

\sui(iv) We apply Theorem~\ref{thm:muhom=hom} with $F$ given by~(i) and
$F' = F \ltens (L \etens \cor_\R)$.  We have in this case
$F'_+ \simeq F_+ \ltens L \simeq L$, hence $\rhom(F_+, F'_+) \simeq L$, and also
$\muhom(F,F') \simeq \muhom(F,F) \ltens \opb{\pi_\Lambda}L \simeq
\opb{\pi_\Lambda}L$. The theorem then gives
$\rsect(M;L) \simeq \rsect(\Lambda; \opb{\pi_\Lambda}L )$.
\end{proof}

Now we prove that $\pi_1(\Lambda) \to \pi_1(M)$ is an isomorphism.  It is
equivalent to show that the inverse image by $\pi_\Lambda$ induces an
equivalence of categories $\loc(\cor_M) \isoto \loc(\cor_\Lambda)$, for some
field $\cor$.

\begin{proposition}\label{prop:equiv_Poinc}
Let $\cor$ be a field. Let $\pi_\Lambda \cl \Lambda \to M$ be the
projection. Then the inverse image functor $\opb{\pi_\Lambda} \cl \loc(\cor_M)
\to \loc(\cor_\Lambda)$ is an equivalence of categories.
\end{proposition}
\begin{proof}
(i) We first prove that $\opb{\pi_\Lambda}$ is fully faithful. Let $F \in
\Derb_{[\Lambda],+}(\cor_{M\times\R})$ be the simple object given by
Theorem~\ref{thm:quant_canon}.  Since $F$ is simple we have $\mu
hom(F,F)|_\Lambda \simeq \cor_\Lambda$ and we deduce, for $L, L' \in
\loc(\cor_M)$,
\begin{equation}\label{eq:equiv_Poinc1}
\mu hom(F\tens \opb{p}L,F\tens \opb{p}L') \simeq 
\hom(\opb{\pi_\Lambda} L, \opb{\pi_\Lambda} L') ,
\end{equation}
where $p\cl M\times\R \to M$ is the projection. Using $(F\tens \opb{p}L)_+ \simeq L$
  and the
isomorphisms~\eqref{eq:muhom=hom-infini1} and~\eqref{eq:equiv_Poinc1}, we obtain
\begin{align*}
\Hom(L,L') & \simeq H^0(\Lambda; \mu hom(F\tens \opb{p}L,F\tens \opb{p}L')) \\
& \simeq  \Hom(\opb{\pi_\Lambda} L, \opb{\pi_\Lambda} L') ,
\end{align*}
which means that $\opb{\pi_\Lambda}$ is fully faithful.

\medskip\noindent (ii) We prove that $\opb{\pi_\Lambda}$ is essentially surjective.
Let $L_1 \in \loc(\cor_\Lambda)$ be given.  We view $\loc(\cor_\Lambda)$ as a
subcategory of $\Dloc(\cor_\Lambda)$ of objects concentrated in degree $0$, as
justified by Remark~\ref{rem:Dloc}.   We recall that the functor $\mu hom(F,\cdot)$
induces an equivalence $\kss(\cor_\Lambda) \isoto \Dloc(\cor_\Lambda)$ (see
Proposition~\ref{prop:KSstack=Dloc}, where the induced functor is denoted
$\ol{\mu hom}(F,\cdot)$).    Hence there exists $\shl_1\in \kss(\cor_\Lambda)$
such that $\ol{\mu hom}(F,\shl_1) \simeq L_1$.  By Corollary~\ref{cor:exist_quant}
there exists $F_1 \in \Derb(\cor_{M\times\R})$ such that
$\kssfunc_\Lambda(F_1) \simeq \shl_1$. Then we have an isomorphism in
$\Derb(\cor_\Lambda)$
\begin{equation}
  \label{eq:equiv_Poinc2}
  \mu hom(F,F_1)|_\Lambda \simeq L_1 .
\end{equation}
Indeed   it holds first in
$\Dloc(\cor_\Lambda)$. Then, applying the functors $H^i$ of Remark~\ref{rem:Dloc}
to~\eqref{eq:equiv_Poinc2}, we see that $\mu hom(F,F_1)|_\Lambda$ is concentrated in
degree $0$. Hence~\eqref{eq:equiv_Poinc2} is an isomorphism in $\loc(\cor_\Lambda)$
and then in $\Derb(\cor_\Lambda)$ because $\loc(\cor_\Lambda)$ is also a subcategory
of $\Derb(\cor_\Lambda)$.

We set $L = (F_1)_+ \in \Derb(\cor_M)$.  Then $\dot\SSi(L) = \emptyset$ and,
since $F_+ \simeq \cor_M$, we also have $L \simeq (F\tens \opb{p}L)_+$.  Hence
$(F_1)_+ \simeq (F\tens \opb{p}L)_+$ and Theorem~\ref{thm:restr-infini-ff} gives
$F_1 \simeq F\tens \opb{p}L$.  We deduce
$$
\mu hom(F,F_1)|_\Lambda \simeq \mu hom(F, F\tens \opb{p}L)|_\Lambda 
 \simeq \opb{\pi_\Lambda} L .
$$
Hence $L_1 \simeq \opb{\pi_\Lambda} L$.  Taking $H^0$ of both sides we have
$L_1 \simeq \opb{\pi_\Lambda} H^0L$.  Since $H^0L \in \loc(\cor_M)$, we have
$L_1 \in \opb{\pi_\Lambda}(\loc(\cor_M))$, as required.  (We could see in fact that
$L$ is concentrated in degree $0$.)
\end{proof}

\begin{corollary}\label{cor:equiv_homot}
The projection $\Lambda \to M$ is a homotopy equivalence.
\end{corollary}
\begin{proof}
  As already recalled this follows from Theorem~\ref{thm:quant_canon}-(iv) and
  Proposition~\ref{prop:equiv_Poinc}.
\end{proof}

\providecommand{\bysame}{\leavevmode\hbox to3em{\hrulefill}\thinspace}

\end{document}